\RequirePackage[l2tabu, orthodox]{nag} 

\documentclass[twoside]{um}  

\usepackage{blindtext} 
\usepackage{coffee4}    
\usepackage{amsmath}
\usepackage{amsthm}
\usepackage{mathtools}
\usepackage{bm}
\usepackage{todonotes}
\usepackage{dsfont}
\usepackage[capitalize]{cleveref}
\usepackage{autonum} 
\usepackage{tikz-cd} 
\usepackage[mathscr]{eucal}
\usepackage{float}
\usepackage{pdfpages}

\makeatletter
\long\def\@makecaption#1#2{%
	\vskip\abovecaptionskip
	\sbox\@tempboxa{{\bfseries#1.} #2}%
	\ifdim \wd\@tempboxa >\hsize
	{\bfseries#1.} #2\par
	\else
	\global \@minipagefalse
	\hb@xt@\hsize{\hfil\box\@tempboxa\hfil}%
	\fi
	\vskip\belowcaptionskip}
\makeatother

\DeclareMathOperator{\std}{std}
\DeclareMathOperator{\Av}{Av}
\DeclareMathOperator{\occ}{occ}
\DeclareMathOperator{\pocc}{\widetilde{occ}}
\DeclareMathOperator{\cocc}{c\text{-}occ}
\DeclareMathOperator{\pcocc}{\widetilde{c\text{-}occ}}
\DeclareMathOperator{\pat}{pat}
\DeclareMathOperator{\Perm}{\mathbf{Perm}}
\DeclareMathOperator{\Leb}{Leb}
\DeclareMathOperator{\indmax}{indmax}

\DeclareMathOperator{\size}{size}
\DeclareMathOperator{\dec}{dec}
\DeclareMathOperator{\CanTree}{CT} 
\DeclareMathOperator{\Pack}{PA} 
\DeclareMathOperator{\DF}{DF}
\DeclareMathOperator{\DT}{DT}
\DeclareMathOperator{\RP}{RP}
\DeclareMathOperator{\Sh}{Sh}
\DeclareMathOperator{\sgn}{sgn}

\DeclareMathOperator{\Lab}{Lab}
\DeclareMathOperator{\Occ}{Occ}

\DeclareMathOperator{\wcp}{WC}
\DeclareMathOperator{\fortree}{LFor}

\DeclareMathOperator{\labtree}{LTr}
\DeclareMathOperator{\cpbp}{CP}
\DeclareMathOperator{\bow}{OW}
\DeclareMathOperator{\bobp}{OP}

\DeclareMathOperator{\idf}{\mathds{1}}
\DeclareMathOperator{\Var}{Var}
\DeclareMathOperator{\Id}{Id}
\DeclareMathOperator{\mult}{mult}
\DeclareMathOperator{\GG}{G}
\DeclareMathOperator{\CL}{CL}
\DeclareMathOperator{\Pat}{Pat}
\DeclareMathOperator{\Cov}{Cov}
\DeclareMathOperator{\LRM}{LRmax}
\DeclareMathOperator{\LRm}{LRmin}
\DeclareMathOperator{\RLM}{RLmax}
\DeclareMathOperator{\RLm}{RLmin}
\DeclareMathOperator{\conv}{conv}

\newcommand{\cP}{\mathcal{P}}
\newcommand{\cS}{\mathcal{S}}
\newcommand{\myvec}[1]{\vec{#1}} %
\newcommand{\rv}[1]{\mathbf{#1}} %

\newcommand{\SG}{\mathcal{S}}
\newcommand{\cC}{\mathcal{C}}

\newcommand{\zz}{\bm{z}}
\newcommand{\asqnk}{\asq(n,k)}
\newcommand{\sq}{Sq}
\newcommand{\asq}{ASq}
\newcommand{\pu}{pos_U}
\newcommand{\ctu}{ct_U}
\newcommand{\pd}{pos_D}
\newcommand{\ctd}{ct_D}
\newcommand{\pr}{pos_R}
\newcommand{\ctr}{ct_R}
\newcommand{\pl}{pos_L}
\newcommand{\ctl}{ct_L}
\newcommand{\irr}{\mathcal{R}_{irr}}

\newcommand{\ext}{\text{ext}}

\newcommand{\avn}{A\!v_n}
\newcommand{\avnk}{ASq(\avn(321),k)}
\newcommand{\Inverse}{\Theta}
\newcommand{\Walks}{\mathfrak W}
\newcommand{\Coals}{\mathfrak C}
\newcommand{\Perms}{\mathcal S}
\newcommand{\Maps}{\mathfrak m}
\newcommand{\Steps}{A}

\newcommand{\exc}{\mathrm{exc}}
\newcommand{\eps}{\varepsilon}
\newcommand{\conti}[1]{{\bm{\mathscr #1}}}
\newcommand{\solution}{F}
\newcommand\indep{\perp\!\!\!\perp}

\def\N{\mathbb{N}}
\def\Z{\mathbb{Z}}
\def\E{\mathbb{E}}
\def\P{\mathbb{P}}
\def\R{\mathbb{R}}
\def\RR{\mathbb{R}}
\def\S{\mathcal{S}}
\def\Asi{A_{\sigma,i}}
\def\leqsi{\preccurlyeq_{\sigma,i}}
\def\Sr{\mathcal{S}^{\bullet}}
\def\Sri{\tilde{\mathcal{S}}^{\bullet}}

\def\Vint{V_{\text{int}}}
\def\Seq{\textsc{Seq}}
\def\DDD{\mathcal{D}} 
\def\Sall{\mathfrak{S}_{\text{all}}} 
\def\MMM{\mathcal{M}on} 
\def\SS{\mathcal{S}} 
\def\nonp{{\scriptscriptstyle{\mathrm{not}}{\oplus}}}
\def\PackedTree{P}
\def\GGG{\mathcal{G}} 
\def\V{\mathbb{V}}

\def\Uinf{\mathcal{U}_{\infty}^{\bullet}}
\def\lufT{\mathfrak{T}^{\bullet,\text{\tiny luf}}}
\def\lufTD{\mathfrak{T}^{\bullet,\text{\tiny luf}}_{\mathcal D}}
\def\lufPT{\mathfrak{P}^{\bullet,\text{\tiny luf}}}

\def\setTkt{\mathcal T_{k,\Omega}^{[t]}}

\def\ValGraph[#1]{\mathcal{O}v(#1)}

\makeatletter
\DeclareRobustCommand{\cev}[1]{%
	\mathpalette\do@cev{#1}%
}
\newcommand{\do@cev}[2]{%
	\fix@cev{#1}{+}%
	\reflectbox{$\m@th#1\vec{\reflectbox{$\fix@cev{#1}{-}\m@th#1#2\fix@cev{#1}{+}$}}$}%
	\fix@cev{#1}{-}%
}
\newcommand{\fix@cev}[2]{%
	\ifx#1\displaystyle
	\mkern#23mu
	\else
	\ifx#1\textstyle
	\mkern#23mu
	\else
	\ifx#1\scriptstyle
	\mkern#22mu
	\else
	\mkern#22mu
	\fi
	\fi
	\fi
}

\makeatother

\newcounter{indice}

\newcommand{\permutation}[1]{
	\setcounter{indice}{0};
	\foreach \i in {#1}
	\addtocounter{indice}{1};
	
	\addtocounter{indice}{1}
	\draw [help lines] (1,1) grid (\theindice,\theindice);
	
	\setcounter{indice}{1};
	
	\foreach \i in { #1 } {
		\draw (\theindice+.5,\i+.5) [fill] circle (.2);
		\addtocounter{indice}{1};
	}
	\addtocounter{indice}{-1};
	
}

\usepackage[toc,section=chapter,style=super]{glossaries}
\makenoidxglossaries
\defglsentryfmt{\textit{\glsgenentryfmt}}

\newglossaryentry{one_line}{%
	name=\ensuremath{\sigma(1)\sigma(2)\dots\sigma(n)},
	description={One-line notation for permutations}
}

\newglossaryentry{size_perm}{%
	name=\ensuremath{|\sigma|},
	description={The size of a permutation $\sigma$}
}

\newglossaryentry{perm_n}{%
	name=\ensuremath{\mathcal{S}_n},
	description={The set of permutations of $[n]=\{1,2,\dots,n\}$}
}

\newglossaryentry{perm}{%
	name=\ensuremath{\mathcal{S}},
	description={The set of permutations of finite size}
}

\newglossaryentry{standa}{%
	name=\ensuremath{\mathrm{std}(x_1\dots x_n)},
	description={The unique permutation that is in the same relative order as $x_1\dots x_n$}
}

\newglossaryentry{pattern_ind}{%
	name=\ensuremath{\pat_I(\sigma)},
	description={The permutation induced by $(\sigma(i))_{i\in I}$}
}

\newglossaryentry{class_n}{%
	name=\ensuremath{\mathrm{Av}_n(B)},
	description={The set of $B$-avoiding permutations of size $n$}
}

\newglossaryentry{class}{%
	name=\ensuremath{\mathrm{Av}(B)},
	description={The set of $B$-avoiding permutations of finite size}
}

\newglossaryentry{occ}{%
	name=\ensuremath{\mathrm{occ}(\pi,\sigma)},
	description={The number of occurrences of a pattern $\pi$ in a permutation $\sigma$}
}

\newglossaryentry{pocc}{%
	name=\ensuremath{\widetilde{\mathrm{occ}}(\pi,\sigma)},
	description={The proportion of occurrences of a pattern $\pi$ in a permutation $\sigma$}
}

\newglossaryentry{c_occ}{%
	name=\ensuremath{\mathrm{c\text{-}occ}(\pi,\sigma)},
	description={The number of consecutive occurrences of a pattern $\pi$ in a permutation $\sigma$}
}

\newglossaryentry{pc_occ}{%
	name=\ensuremath{\widetilde{\mathrm{c\text{-}occ}}(\pi,\sigma)},
	description={\fontsize{9.65}{6}\selectfont The proportion of consecutive occurrences of a pattern $\pi$ in a permutation $\sigma$}
}

\newglossaryentry{o_plus}{%
	name=\ensuremath{\oplus},
	description={The direct sum operation for permutations}
}

\newglossaryentry{o_minus}{%
	name=\ensuremath{\ominus},
	description={The skew sum operation for permutations}
}

\newglossaryentry{rec_max}{%
	name=\ensuremath{\LRM(\sigma)},
	description={The set of left-to-right maxima of a permutation $\sigma$ (the other records are denoted $\LRm(\sigma)$, $\RLM(\sigma)$, and $\RLm(\sigma)$)}
}

\newglossaryentry{root_perm_n}{%
	name=\ensuremath{\mathcal{S}^{\bullet}_n},
	description={The set of rooted permutations of size $n$}
}

\newglossaryentry{root_perm}{%
	name=\ensuremath{\mathcal{S}^{\bullet}},
	description={The set of rooted permutations of finite size}
}

\newglossaryentry{root_perm_inf}{%
	name=\ensuremath{\mathcal{S}_{\infty}^\bullet},
	description={The set of rooted permutations of infinite size}
}

\newglossaryentry{root_perm_fin_inf}{%
	name=\ensuremath{\tilde{\mathcal{S}}^{\bullet}},
	description={The set of rooted permutations of finite and infinite size}
}

\newglossaryentry{distance_perm_local}{%
	name=\ensuremath{d},
	description={The local distance on $\tilde{\mathcal{S}}^{\bullet}$}
}

\newglossaryentry{annealed_BS}{%
	name=\ensuremath{\xrightarrow{aBS}},
	description={Convergence in the annealed Benjamini--Schramm sense}
}

\newglossaryentry{quenched_BS}{%
	name=\ensuremath{\xrightarrow{qBS}},
	description={Convergence in the quenched Benjamini--Schramm sense}
}

\newglossaryentry{induce_perm_k}{%
	name=\ensuremath{\Perm_k(\mu)},
	description={The random permutation of size $k$ induced by a permuton $\mu$}
}

\newglossaryentry{space_perm}{%
	name=\ensuremath{\mathcal M},
	description={The set of permutons}
}

\newglossaryentry{distance_perm}{%
	name=\ensuremath{d_\square},
	description={The permuton distance on $\mathcal M$}
}

\newglossaryentry{ind_max}{%
	name=\ensuremath{\indmax(\sigma)},
	description={The index of the maximal value $|\sigma|$ of a permutation $\sigma$}
}

\newglossaryentry{left_right_perm}{%
	name=\ensuremath{\sigma_L \text{ }\mathrm{and}\text{ } \sigma_R},
	description={The left and right subsequences of $\sigma,$ before and after the maximal value}
}

\newglossaryentry{left_right_tree}{%
	name=\ensuremath{T^v_L \text{ }\mathrm{and}\text{ } T^v_R},
	description={The left and the right fringe subtrees of a binary tree $T$ hanging below the vertex $v$}
}

\newglossaryentry{binary_tree}{%
	name=\ensuremath{\mathbb{T}^b},
	description={The set of binary trees}
}

\newglossaryentry{rooted_plane_tree_n}{%
	name=\ensuremath{\mathbb{T}_n},
	description={The set of rooted plane trees with $n$ vertices}
}

\newglossaryentry{rooted_plane_tree}{%
	name=\ensuremath{\mathbb{T}},
	description={The set of finite rooted plane trees}
}

\newglossaryentry{sub_perm}{%
	name=\ensuremath{\theta[\nu^{(1)},\dots,\nu^{(d)}]},
	description={The \emph{substitution} of $\nu^{(1)},\dots,\nu^{(d)}$ in $\theta$}
}

\newglossaryentry{s_all}{%
	name=\ensuremath{\Sall},
	description={The set of all simple permutations}
}

\newglossaryentry{can_tree}{%
	name=\ensuremath{\CanTree(\nu)},
	description={The canonical tree associated with a permutation $\nu$}
}

\newglossaryentry{mon_perm}{%
	name=\ensuremath{\MMM},
	description={The set of all monotone (increasing or decreasing) permutations}
}

\newglossaryentry{simple_mon_perm}{%
	name=\ensuremath{\widehat{\Sall}:=\Sall \cup \MMM},
	description={The set of all monotone and simple permutations}
}

\newglossaryentry{set_can_trees}{%
	name=\ensuremath{\mathcal{T}},
	description={The of  canonical trees with decorations in $\widehat{\mathfrak S}$}
}

\newglossaryentry{set_can_trees_not}{%
	name=\ensuremath{\mathcal{T}_{\nonp}},
	description={The set of canonical trees with a root that is \emph{not} labelled $\oplus$}
}

\newglossaryentry{simple_sub}{%
	name=\ensuremath{\mathfrak{S}},
	description={The set of simple permutations of a given substitution-closed class}
}

\newglossaryentry{gadget}{%
	name=\ensuremath{\GGG(\mathfrak{S})},
	description={The set of all $\mathfrak{S}$-gadgets}
}

\newglossaryentry{gadget_and_mon}{%
	name=\ensuremath{\widehat{\GGG(\mathfrak{S})}},
	description={The set of all $\mathfrak{S}$-gadgets and decorations $\{\circledast_k, k \ge 2\}$}
}

\newglossaryentry{baxter}{%
	name=\ensuremath{\mathcal P},
	description={The set of Baxter permutations}
}

\newglossaryentry{tandem}{%
	name=\ensuremath{\mathcal W},
	description={The set of tandem walks}
}

\newglossaryentry{orientations}{%
	name=\ensuremath{\mathcal{O}},
	description={The set of bipolar orientations}
}

\newglossaryentry{tandem_walks}{%
	name=\ensuremath{\mathcal W},
	description={The set of tandem walks}
}

\newglossaryentry{coal_proc}{%
	name=\ensuremath{\mathscr{C}},
	description={The set of coalescent-walk processes}
}

\newglossaryentry{coal_proc_I}{%
	name=\ensuremath{\Coals(I)},
	description={The set of coalescent-walk processes on some interval $I$}
}

\newglossaryentry{walks_I}{%
	name=\ensuremath{\Walks(I)},
	description={The set of two-dimensional walks with time space $I$ (considered up to an additive constant)}
}

\newglossaryentry{pointed_plane_trees}{%
	name=\ensuremath{\mathfrak{T}^\bullet},
	description={The set of (possibly infinite) pointed plane trees}
}

\newglossaryentry{luf_pointed_plane_trees}{%
	name=\ensuremath{\lufT},
	description={The set of locally and upwards finite pointed trees}
}

\newglossaryentry{lufd_pointed_plane_trees}{%
	name=\ensuremath{\lufTD},
	description={The set of $\DDD$-decorated locally and upwards finite pointed trees}
}

\newglossaryentry{lufd_dist}{%
	name=\ensuremath{d_t},
	description={The local distance on the set $\lufTD$}
}

\newtheorem{thm}{Theorem}[section]
\newtheorem{cor}[thm]{Corollary}
\newtheorem{prop}[thm]{Proposition}
\newtheorem{lem}[thm]{Lemma}
\newtheorem{conj}[thm]{Conjecture}

\newtheorem{prbl}[thm]{Problem}

\theoremstyle{definition}
\newtheorem{defn}[thm]{Definition}

\newtheorem{obs}[thm]{Observation}
\newtheorem{ass}[thm]{Assumption}

\theoremstyle{remark}
\newtheorem{rem}[thm]{Remark}

\newtheorem{exmp}[thm]{Example}

\title{Random Permutations}  
\tagline{A geometric point of view}      
\author{Jacopo Borga}            
\authorID{???}                   
\supervisor{M.\ Bouvel and V.\ Féray}              
\department{Department of Mathematics}   
\faculty{Faculty of Science}       
\degree{Doctor of Philosophy}       
\doctype{dissertation}               
\degreedate{\monthyeardate\today}    

\graphicspath{{./images/}{./chap1/images/}{./chap2/images/}}   

\makeindex

\begin{document}

\frontmatter 
    \maketitle
    \begin{secondtitle}
{\fontsize{30}{60}\selectfont Random Permutations}\\[5mm]
\emph{\fontsize{18}{60}\selectfont A geometric point of view}\\[5cm]

{\Huge{\textbf{Jacopo Borga}}}
\end{secondtitle}

    \begin{copyrightenv}
\end{copyrightenv}
          
    \begin{dedication}
{\large{To the crossroad in front of my house}}\\[5mm]
You know, there is COVID-19
\end{dedication}

    \begin{acknowledgements}
	This manuscript gives an overview of my work during the last four years, that is the time I spent here in Zurich for my Ph.D. Reading this manuscript, one might think that my Ph.D. was only made of mathematics, but my experience was actually made of people. Fantastic people that have walked next to me in different ways during this wonderful journey.
	
	Je ne peux m'empêcher de commencer par Valentin et Mathilde, mes deux directeurs de thèse (officiellement) mais surtout mes deux guides professionnels et non professionnels. Merci beaucoup de m'avoir appris à expérimenter les mathématiques d'une manière professionnelle mais insouciante. Pour m'avoir soutenu au quotidien dans la réalisation de mes idées (pas toujours totalement sensées) et pour avoir fait de moi le mathématicien que je suis aujourd'hui. Merci beaucoup de m'avoir accueilli dans votre merveilleuse famille en même temps, je n'oublierai jamais les moments fantastiques passés avec vos trois merveilleux enfants.
	
	Thanks a lot to Nicolas Curien and Peter Winkler for agreeing to be readers for this dissertation and being so supportive. Your feedback is a great stimulus for my future mathematical career.
	
	Grazie mille a Benedetta, per essermi sempre stata a fianco, per aver riletto centinaia di pagine che ho scritto, per aver sopportato tutti i momenti in cui mi perdo a pensare alla matematica, ma soprattutto per avermi sempre amato.
	
	Un grandissimo ringraziamento spetta ovviamente ad Emilio, senza di lui gran parte di questo percorso non sarebbe stato possibile. Non dimenticherò mai il momento in cui mi hai suggerito di raggiungerti a Parigi, una scelta che mi ha cambiato la vita.
	
	Se oggi sono quello che sono lo devo soprattutto alla mia famiglia. Grazie mamma e grazie papà, potervi rendere orgogliosi sarà sempre un mio obiettivo per ringraziarvi di tutto quello che avete fatto per me. Grazie Tommy per essere un fratello speciale, nonostante la costante distanza siamo sempre riusciti a mantenere un rapporto stupendo, e saprò sempre che potrò contare su di te in ogni momento.
	
	\vspace{-0.4cm}
	
	\begin{figure}[H]
		\centering
		\includegraphics[scale=.87]{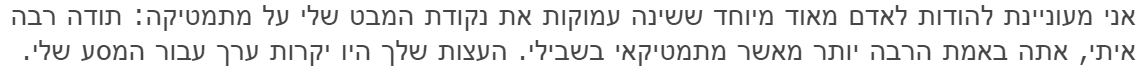}
	\end{figure}
	 

	\vspace{-0.6cm}

	Grazie mille Pegu per essere sempre la prima persona a venirmi a trovare ogni volta che torno in Italia o vado in un nuovo posto nel mondo. Sei una persona unica.
	
	Muito obrigado Raúl por partilhares comigo uma grande parte desta jornada.
	Fazer \break
	matemática contigo foi uma das melhores partes do doutoramento, mas especialmente pôr comboios em artigos só foi possível contigo. Obrigado, man! 
	Vejo-te em breve em São Francisco - fico mesmo feliz que estejas lá.
	
	Un grande ringraziamento ai miei due maestri Paolo Dai Pra e Giambattista Giacomin per avermi sempre supportato quando ero un giovane studente. Siete stati un esempio per me.
	
	A big thanks to all my collaborators, especially to Erik Slivken. I will never forget our first collaboration, how easy it was to start to work with you, and how much fun we had together. And thanks for being my \emph{American guide}, you played a big role in my final decision to go to California.
	
	Danke an alle meine Kollegen der UZH und ETH. Die vielen Erlebnisse mit euch haben Spass gemacht, insbesondere unsere Skiwochenenden und Trinkseminare. Ich bedanke mich auch bei allen Professoren der UZH für die hilfreichen Diskussionen über verschiedene mathematische Probleme. Insbesondere möchte ich mich bei Jean Bertoin und Askan Nikeghbali bedanken. 
	
	Grazie a tutti i \emph{Padovani} per accogliermi ogni volta che torno in Italia come se me ne fossi andato due giorni prima. E grazie al \emph{Borgo Ruga} per ricordarmi ogni volta che senza un po' di ignoranza nulla ha senso, siete grandi!
	
	Grazie a tutto lo staff di \emph{Da Pino} per avermi insegnato che anche nel mondo del lavoro si possono creare delle amicizie che rimarranno per sempre.
	
	Thanks a lot to all the people that supported me during this years, writing letters and spending good words for me. D'une manière particulière, merci à Gregory Miermont de m'avoir soutenu dans plusieurs de mes projets et d'avoir été si solidaire à chaque étape de ma carrière.
	
	Last but not the least, a big thanks goes to Stanford and in particular to Amir Dembo, Persi Diaconis and Sourav Chatterje for giving me the opportunity to join a so special place, but mainly, to make one of my dreams real. Thank you!
\end{acknowledgements}


\begin{abstract}
We look at \emph{geometric} limits of large random \emph{non-uniform} permutations. We mainly consider two theories for limits of permutations: \emph{permuton limits}, introduced by Hoppen, Kohayakawa,
Moreira, Rath, and Sampaio to define a notion of scaling limits for permutations; and \emph{Benjamini-Schramm limits}, introduced by the author to define a notion of local limits for permutations. 

The models of random permutations that we consider are mainly \emph{constrained models}, that is,  uniform permutations belonging to a given subset of the set of all permutations. We often identify this subset using pattern-avoidance, focusing on: permutations avoiding a pattern of length three, substitution-closed classes, (almost) square permutations, permutation families encoded by generating trees, and Baxter permutations.

We explore some universal phenomena for the models mentioned above. For Benjamini-Schramm limits we explore a \emph{concentration phenomenon} for the limiting objects. For permuton limits we deepen the study of some known universal permutons, called \emph{biased Brownian separable permutons}, and we introduce some new ones, called \emph{Baxter permuton} and \emph{skew Brownian permutons}.
In addition, for (almost) square permutations, we investigate the occurrence of a phase transition for the limiting permutons.

On the way, we establish various combinatorial results both for permutations and other related objects. Among others, we give a complete description of the \emph{feasible region for consecutive patterns} as the \emph{cycle polytope} of a specific graph; and we find new bijections relating Baxter permutations, bipolar orientations, walks in cones, and a new family of discrete objects called \emph{coalescent-walk processes}.
\end{abstract}\cleardoublepage
    {
    	\hypersetup{linkcolor=black}
    	\tableofcontents*\if@openright\cleardoublepage\else\clearpage\fi
    }
    

\pagestyle{umpage}
\floatpagestyle{umpage}
\mainmatter 
    \chapter{Introduction: A geometric perspective on random constrained permutations}
\chaptermark{Introduction: A geometric perspective}

\begin{adjustwidth}{8em}{0pt}
	\emph{In which we discuss the main topics of this thesis. We start by contextualizing our work in relation to the existing literature and prior studies. We then give a brief description of permuton and local convergence for permutations. Finally, we present an overview of the remaining chapters of this manuscript.}
\end{adjustwidth}

\bigskip

\bigskip

\bigskip

\noindent An old and fascinating, even though ill-posed, mathematical question that has attracted the attention of many mathematicians in the last century is the following one:

\vspace{0.4cm}

\begin{equation}
	\emph{What does a large random permutation look like?}
\end{equation}

\vspace{0.65cm}

\noindent For years researchers have tried to make sense of this question, particularly in the case of \emph{uniform} random permutations.
The most established approach has been to look at the convergence of relevant statistics\footnote{References and a more complete discussion can be found in \cref{sect:lim_stat}.}, such as number of cycles, number of inversions, length of the longest increasing subsequence and many others.

More recently, a geometric point of view has been explored\footnote{References and a more complete discussion can be found in \cref{sect:perm_lim_intro,sect:local_lim_intro}.}. Instead of looking at statistics on permutations, the question is to directly determine, from a global or local perspective, the limit of the permutation itself. 

\medskip

Let us now briefly explain with an example how these geometric notions of limits work. We start with the global point of view. Consider a large uniform random permutation $\bm\sigma$. One would expect that the points of the diagram of $\bm\sigma$ -- i.e.\ the points $(i,\bm \sigma(i))$ -- are spread out in a homogeneous way inside a square.  And indeed the diagram of a large uniform random permutation looks like this:

\begin{figure}[H]
	\centering
	\includegraphics[scale=.4]{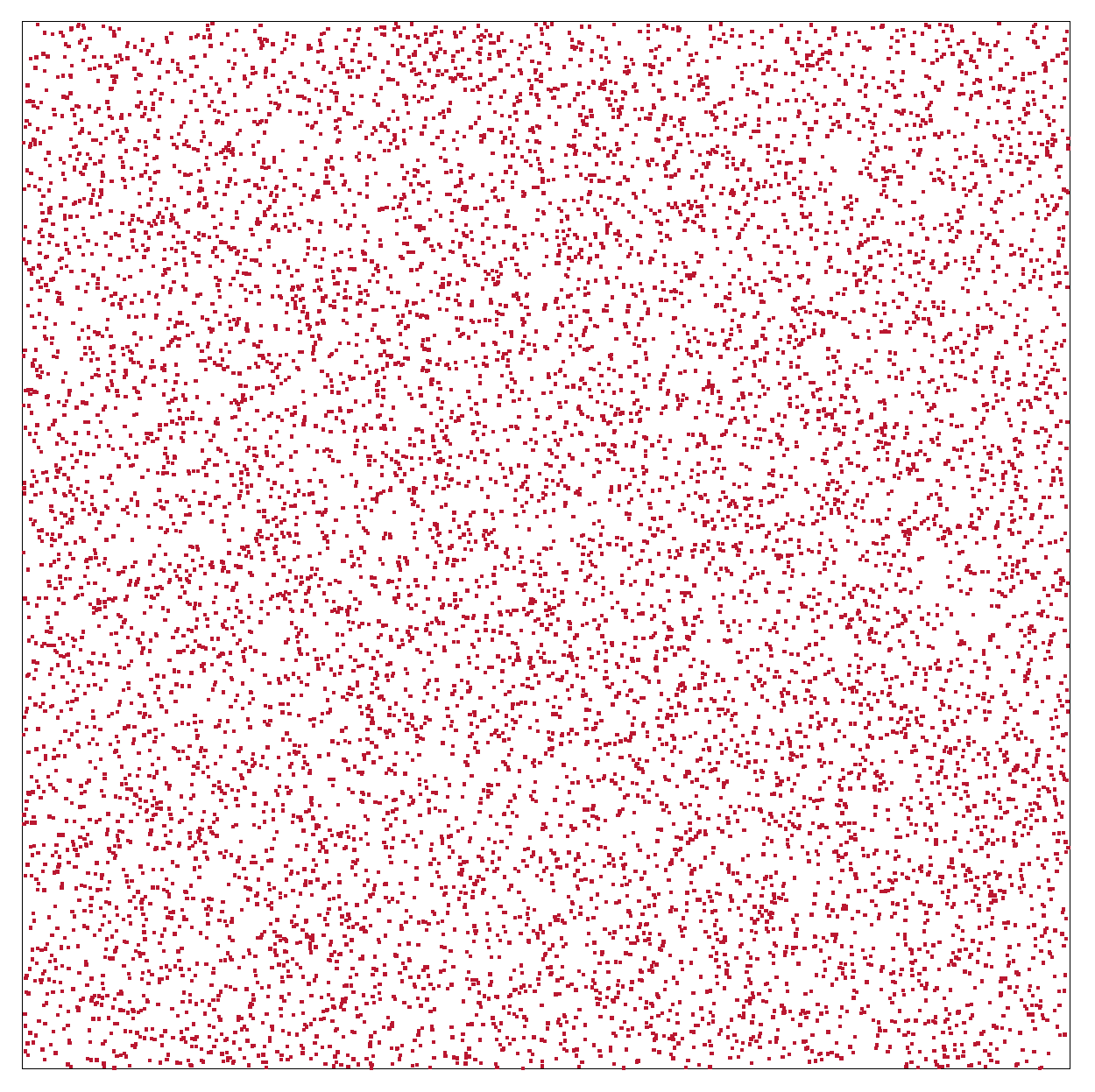}
\end{figure}

This intuition is made rigorous by showing that large uniform random permutations converge in the \emph{permuton sense} to the two-dimensional Lebesgue measure on the unit square, that is, the \emph{measure induced} by the points in the diagram of $\bm\sigma$ converges weakly to the two-dimensional Lebesgue measure on the unit square.

We now move to the local point of view. Consider again a large uniform random permutation $\bm\sigma$. We look at the behavior of the permutation around a distinguished point, called the \emph{root}, asking what are the relative positions of the points around this root. If the latter is chosen uniformly at random, one would expect that these neighboring points are still distributed as a uniform permutation. This intuition is made rigorous by showing that a uniform random  permutation \emph{locally converges}, or even better \emph{Benjamini-Schramm converges}, to the \emph{uniform random infinite rooted permutation}.

The theory of permutons goes back to the works of Hoppen, Kohayakawa,
Moreira, Rath and Sampaio~\cite{hoppen2013limits} and Presutti and
Stromquist~\cite{MR2732835}. In this thesis (\cref{chp:conv_theories}), we formally introduce the notion of \emph{local convergence} or \emph{Benjamini-Schramm convergence} for permutations. 

In our work, instead of looking at uniform random permutations, we focus on constrained models of random permutations, mainly pattern-avoiding permutations\footnote{The reader who is not familiar with the terminology of permutation patterns (such as pattern occurrences, pattern avoidance, \emph{etc.}) can find the necessary background in \cref{sect:not_perm_patt}.}. The goal is to understand how these constraints affect the limiting global and local shape of random permutations.

As we will see, the answers are much more interesting than in  the uniform case. Among various results, we discuss a universal phenomenon for local limits of random constrained permutations (\cref{chp:local_lim}), we  explore some phase transitions for the limiting shape of (almost) square permutations (\cref{chp:square}), and  we present some known and new universal random limits for constrained permutations, leading to the definition of the \emph{skew Brownian permuton} (\cref{chp:perm_lim}).

In order to prove these results, we explore several connections between models of constrained permutations and many other discrete objects, such as trees, walks, and maps (\cref{chp:models}). This provides an occasion to exhibit and discover some pretty combinatorial constructions related with pattern-avoiding permutations.

In \cref{fig:simulations_models}, the reader can see some simulations\footnote{In this manuscript, we present several simulations of large random permutations in various models of non-uniform random permutations. These simulations were obtained using a dozen of different algorithms developed and coded by the author during his Ph.D. We do not discuss the details of such algorithms in this manuscript; we just mention that most of them build on some bijections between permutations and trees or walks, the latter objects being simpler to simulate.} of the diagrams of large random permutations in various models of non-uniform random permutations. As it can be noticed, an incredibly rich variety of limiting behaviors emerges, all of which will be explored in this manuscript. 

\vspace{0.5cm}

\begin{figure}[htbp]
	\centering
	\includegraphics[scale=.28]{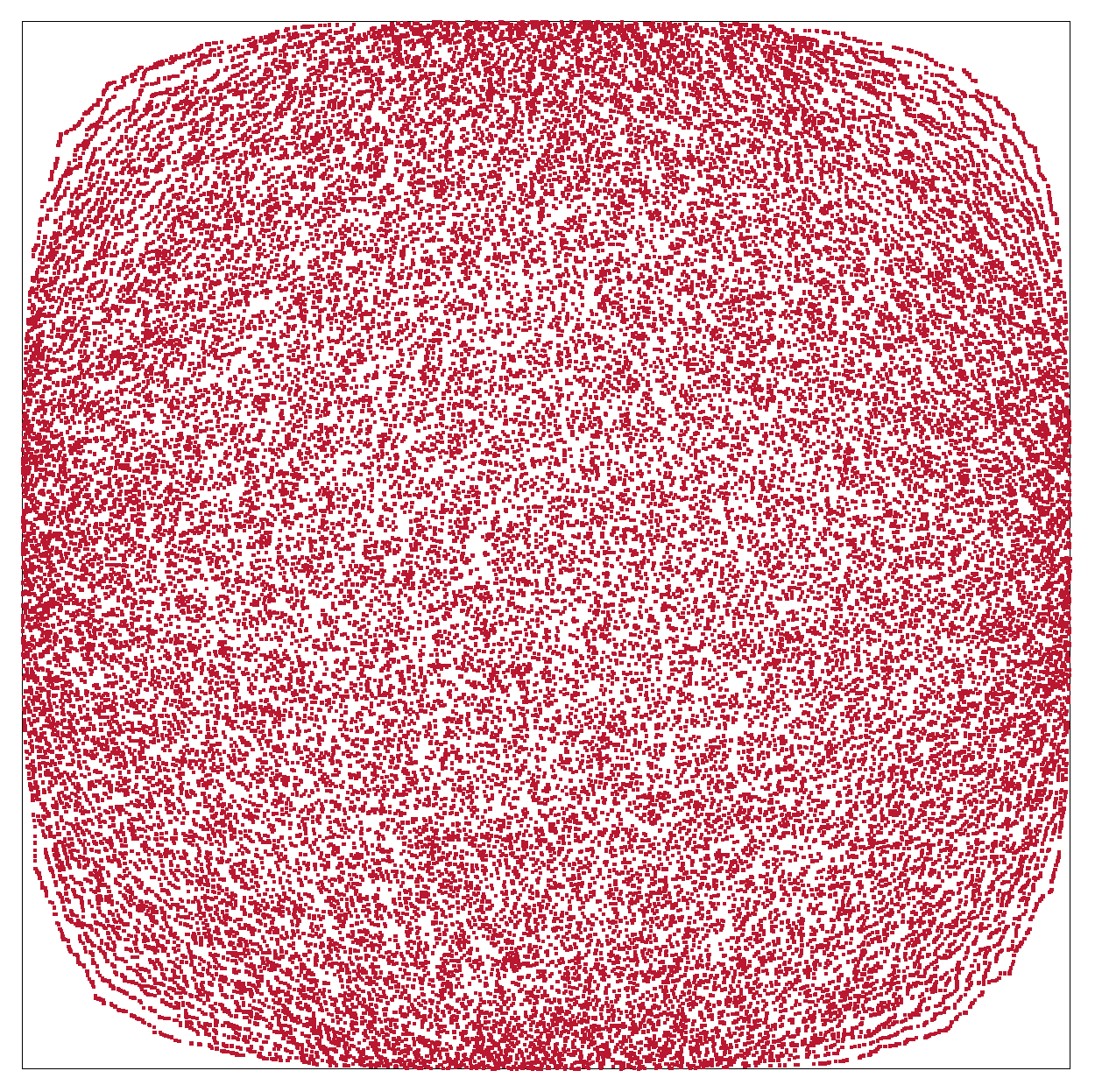}
	\includegraphics[scale=.28]{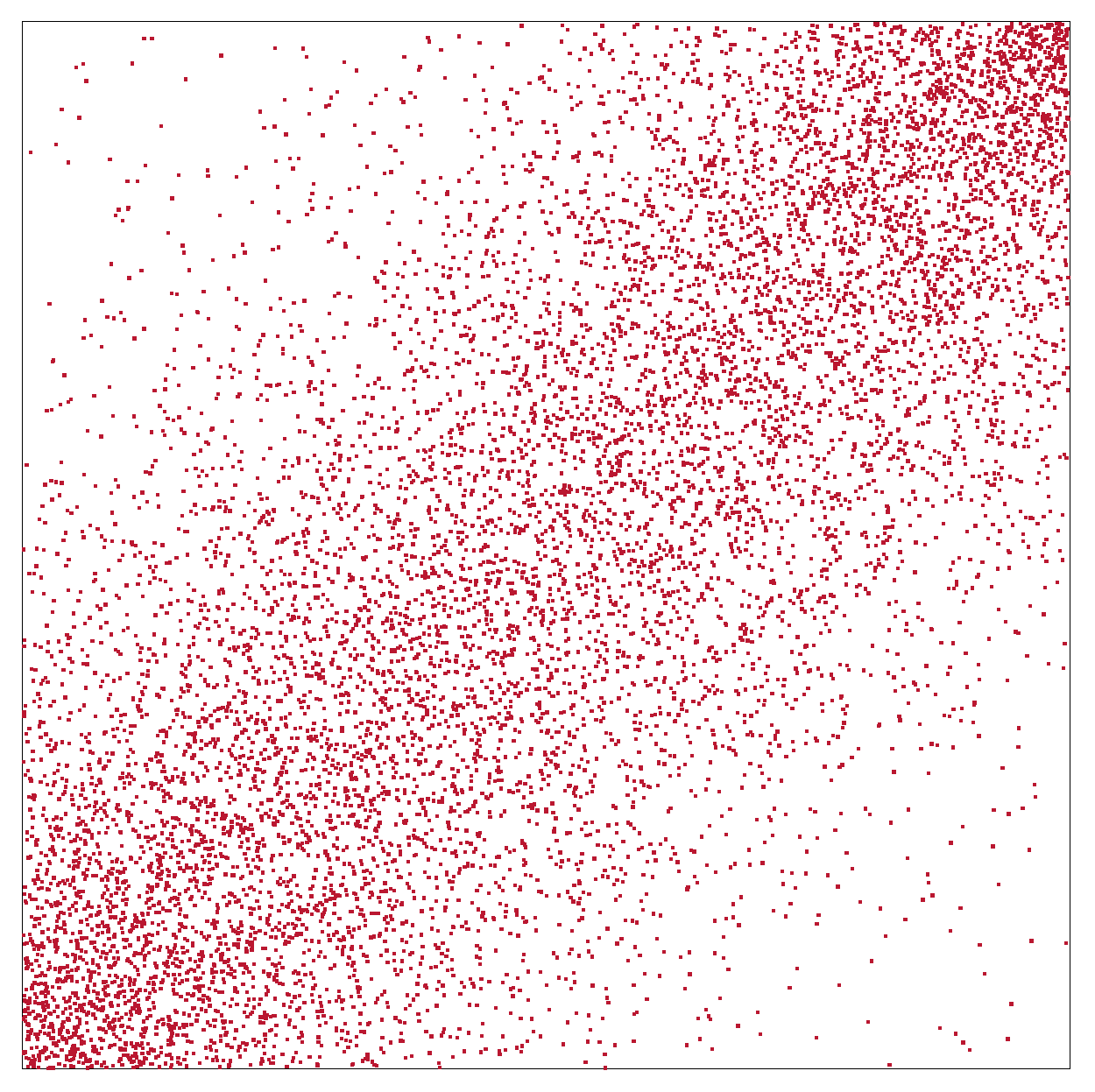}
	\includegraphics[scale=.28]{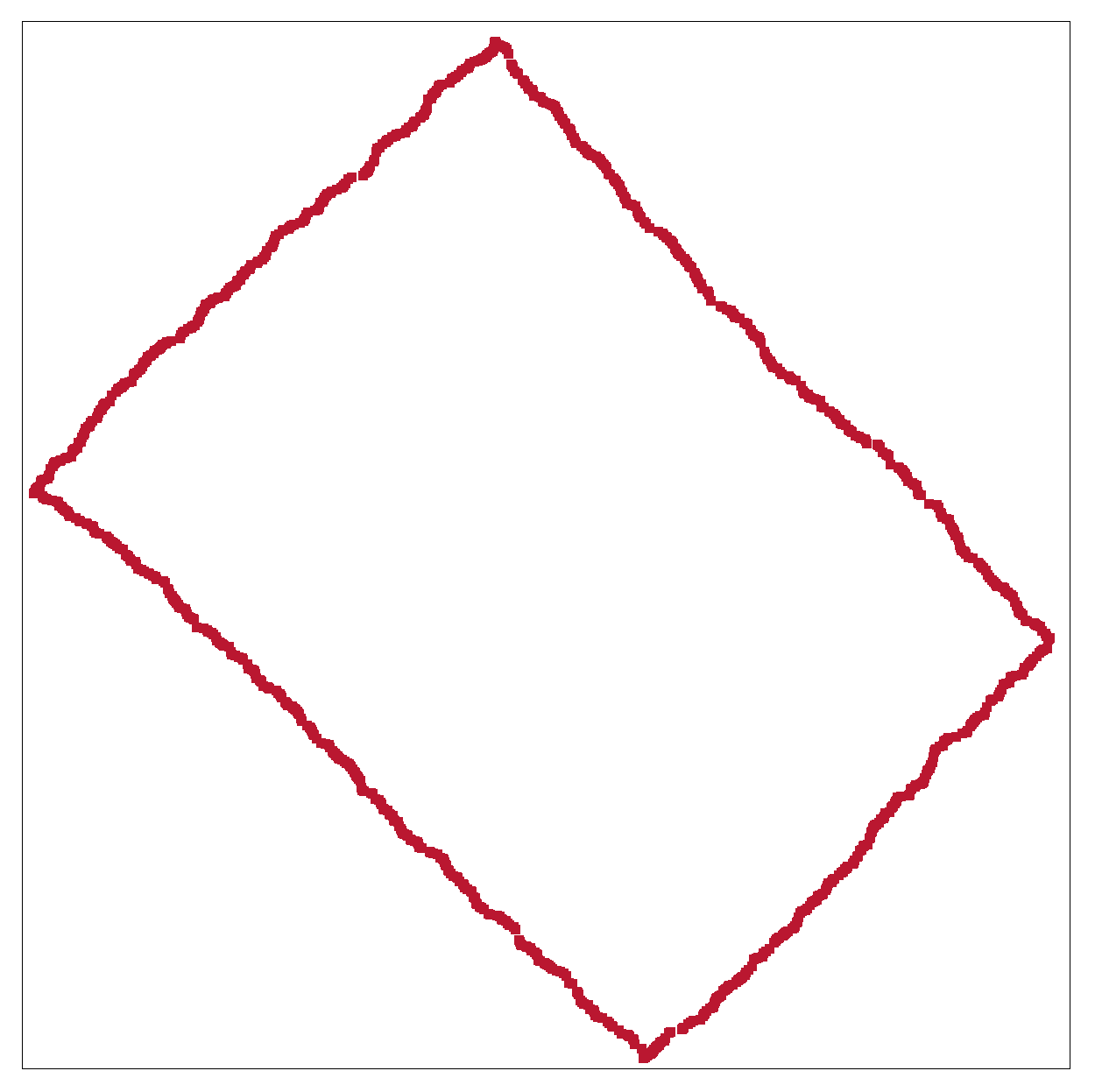}
	\includegraphics[scale=.28]{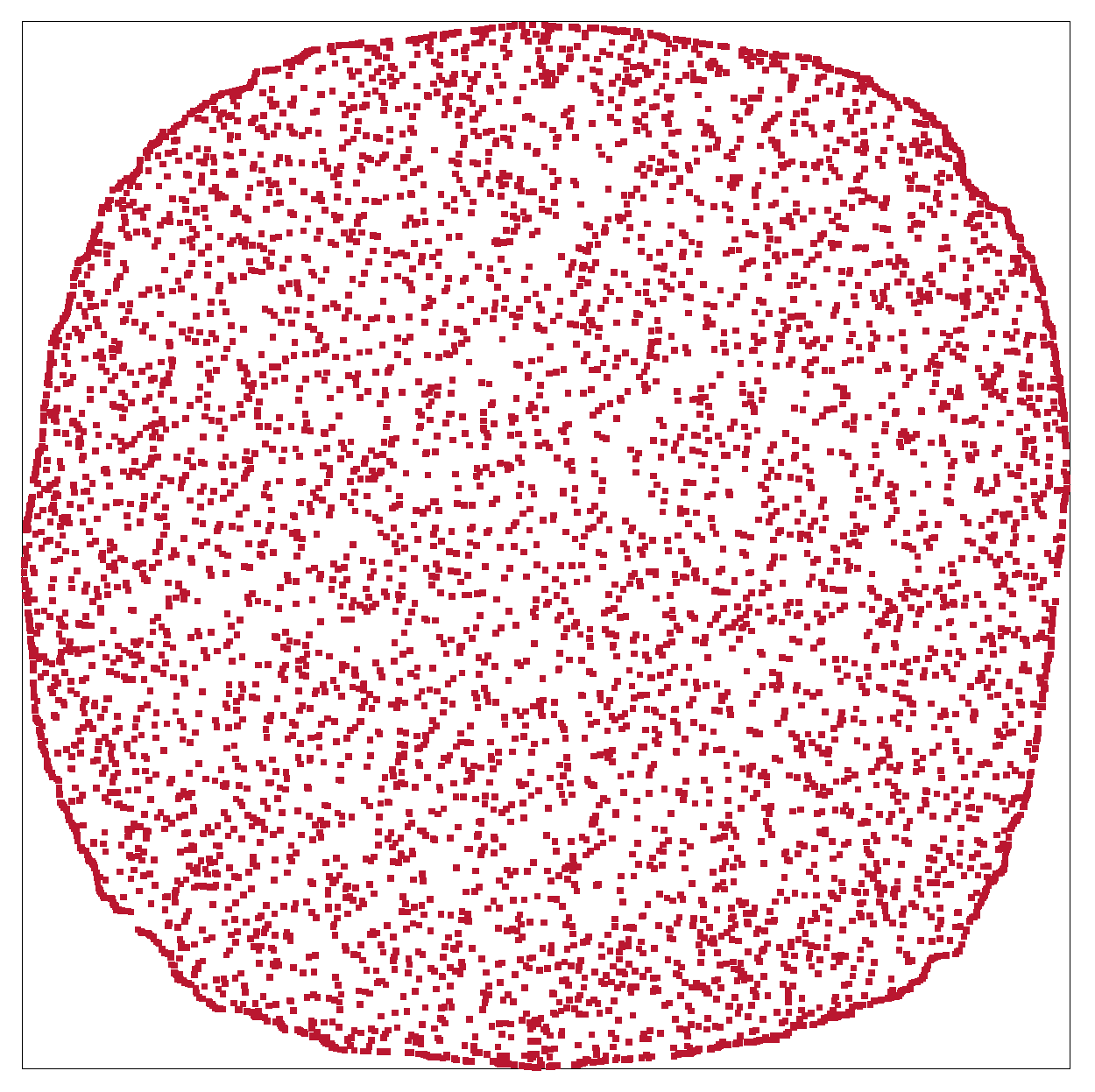}
	\includegraphics[scale=.28]{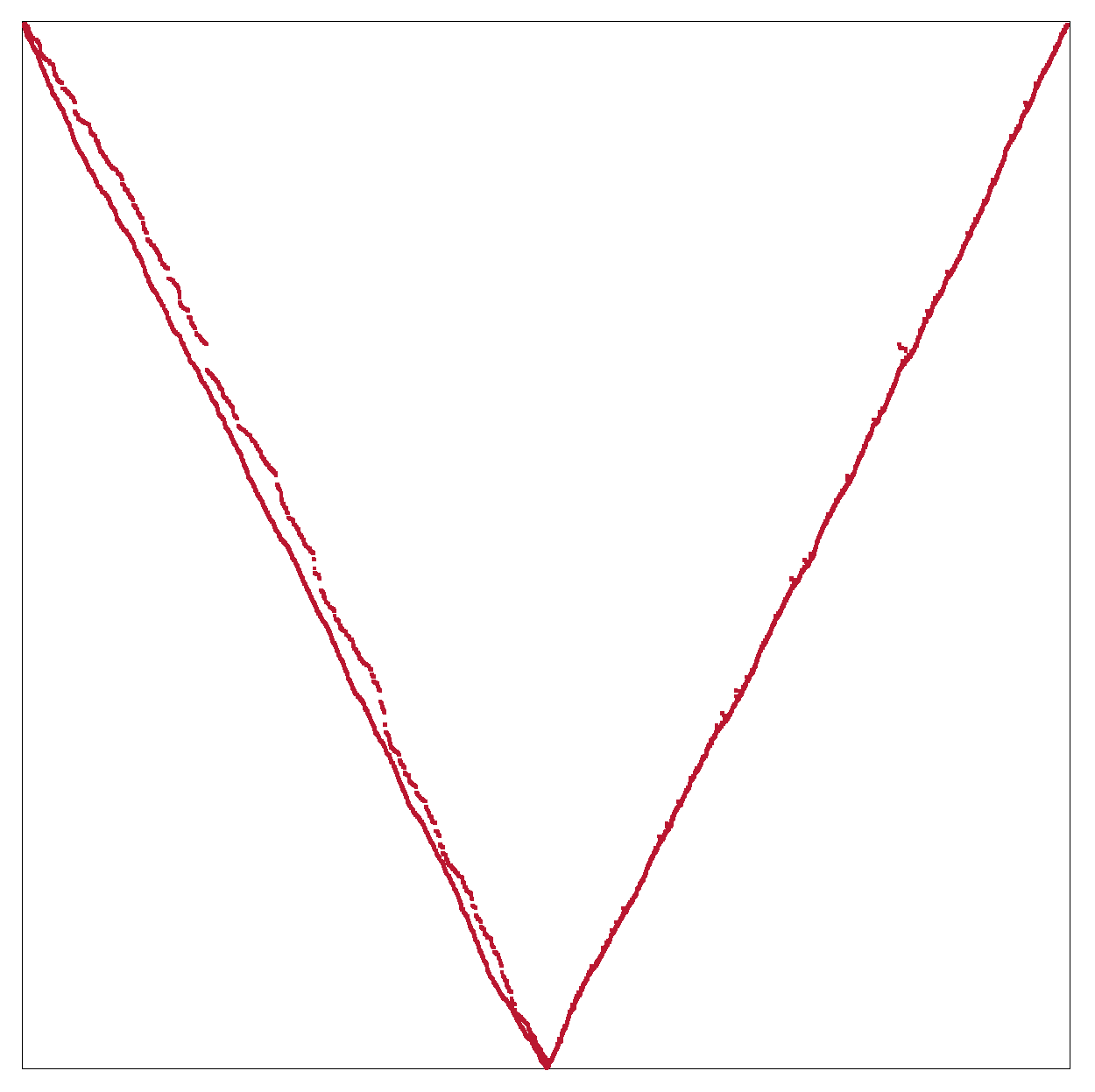}
	\includegraphics[scale=.28]{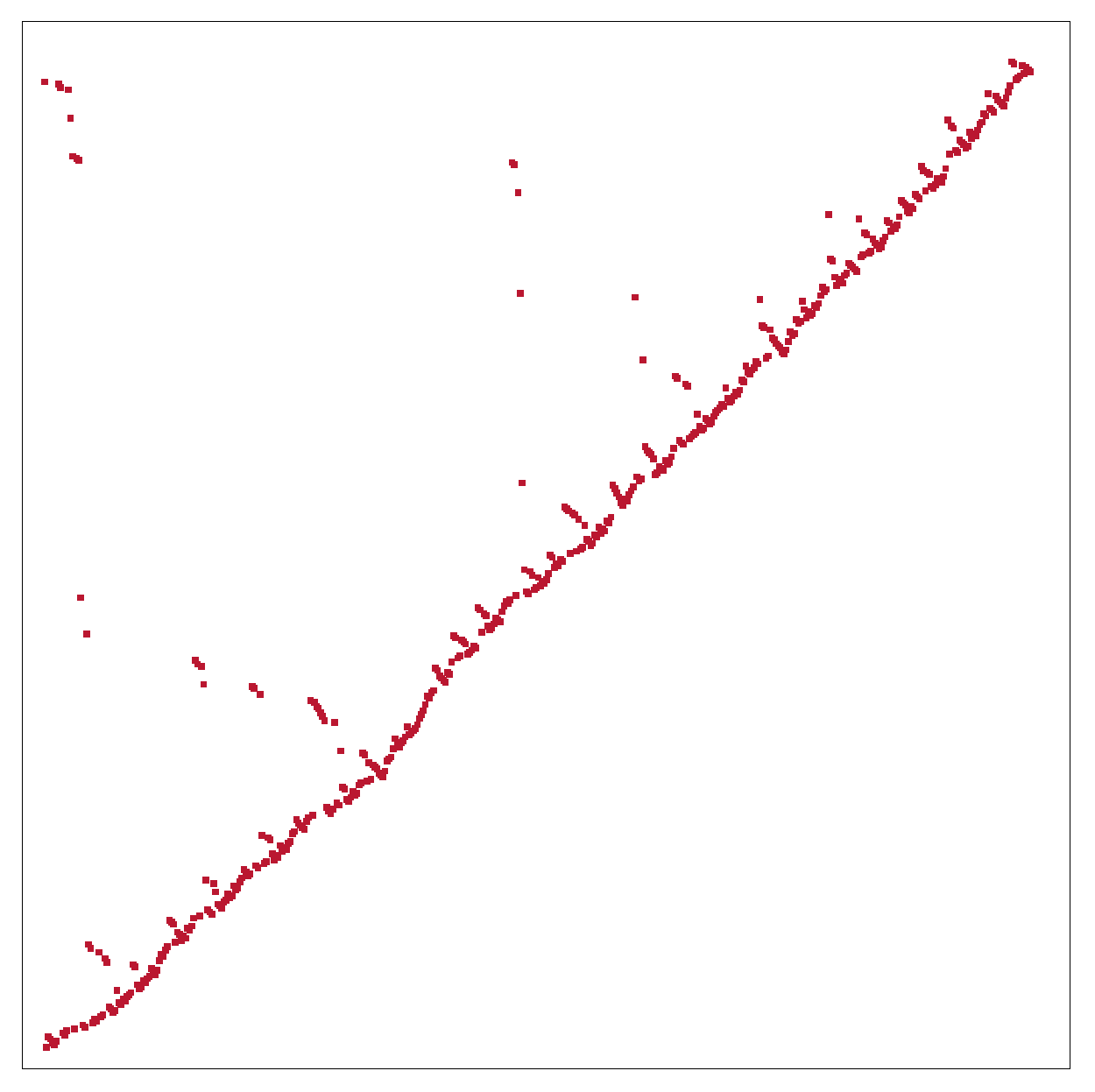}
	\includegraphics[scale=.28]{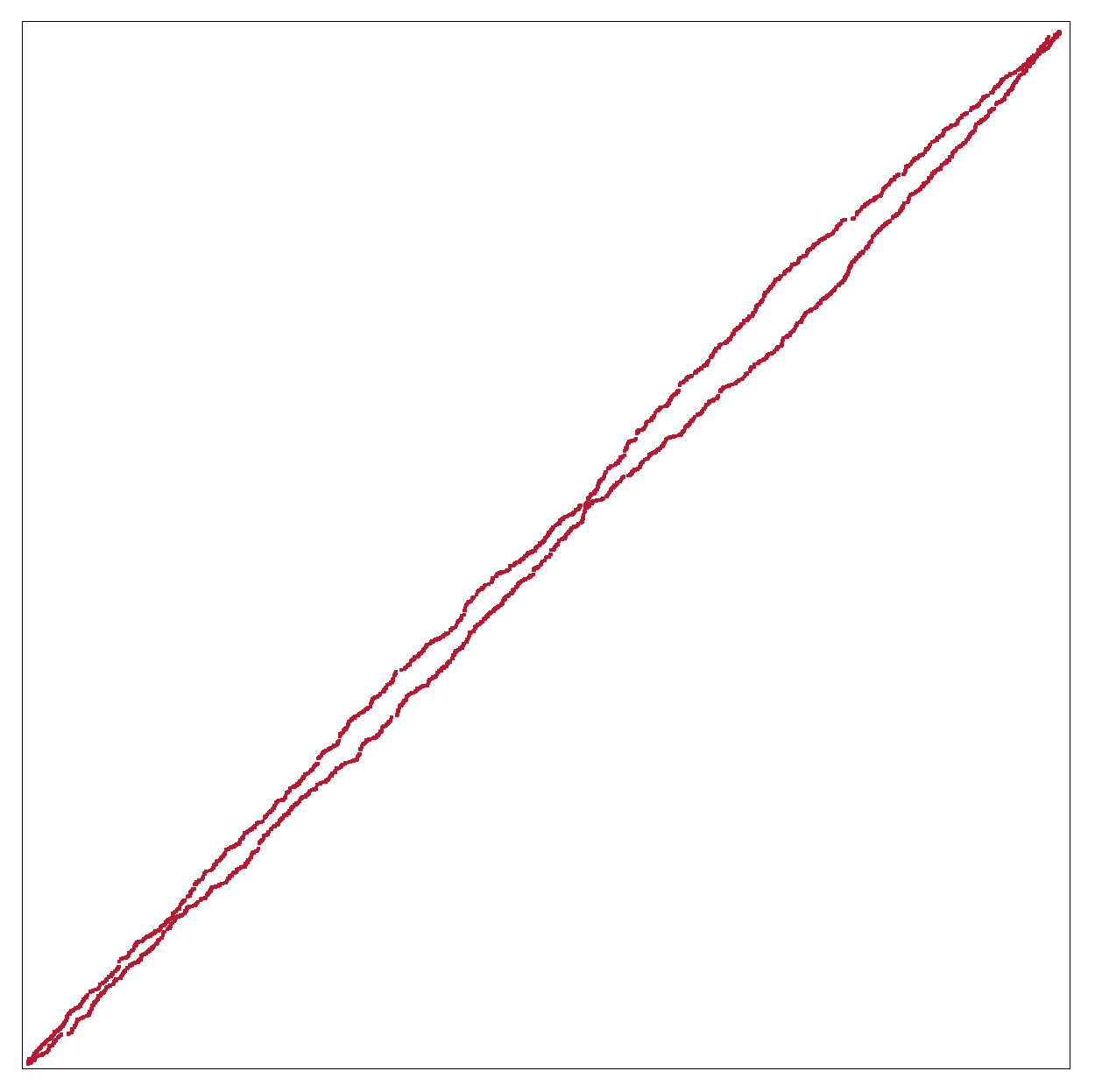}
	\includegraphics[scale=.28]{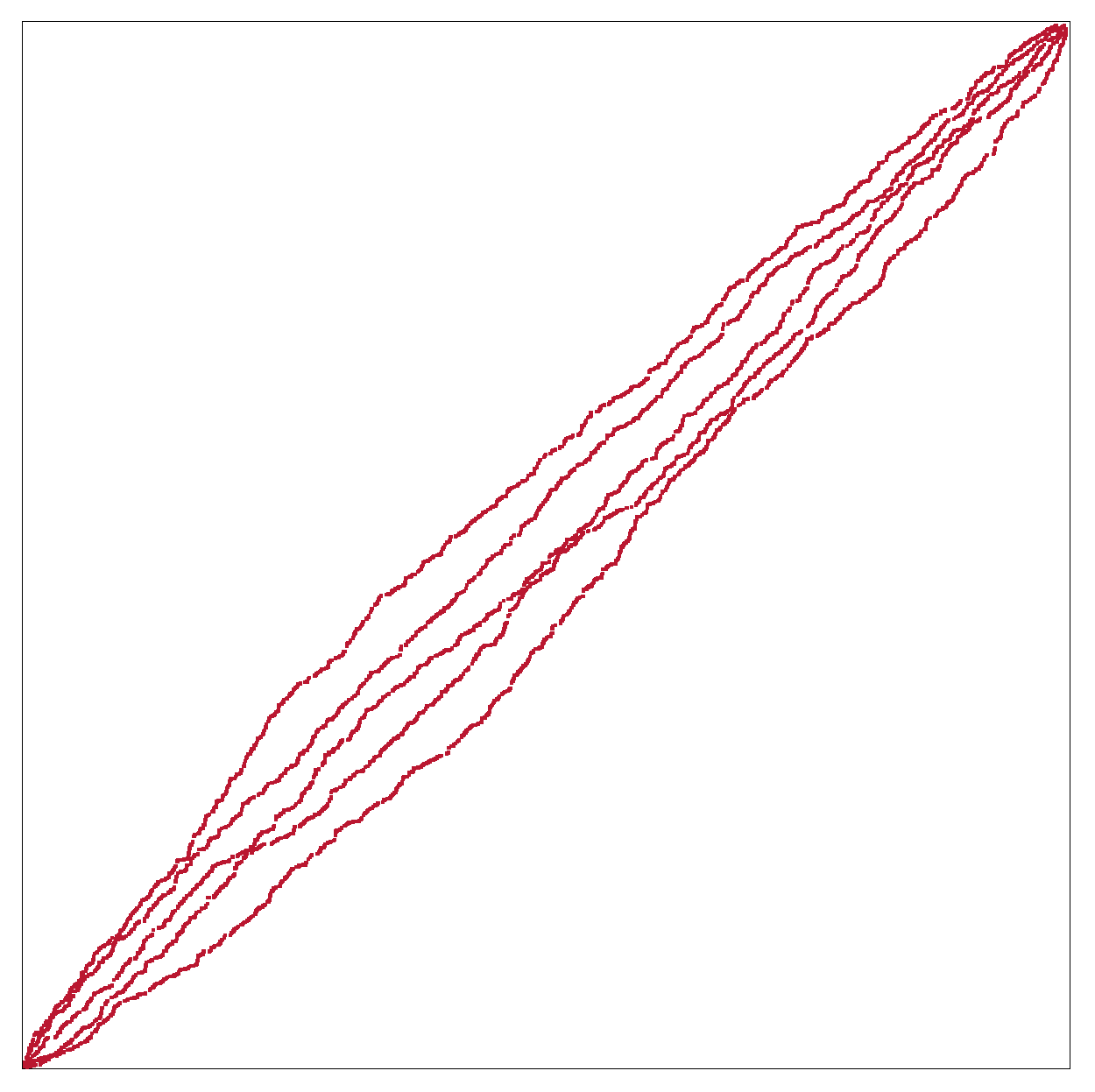}
	\includegraphics[scale=.07]{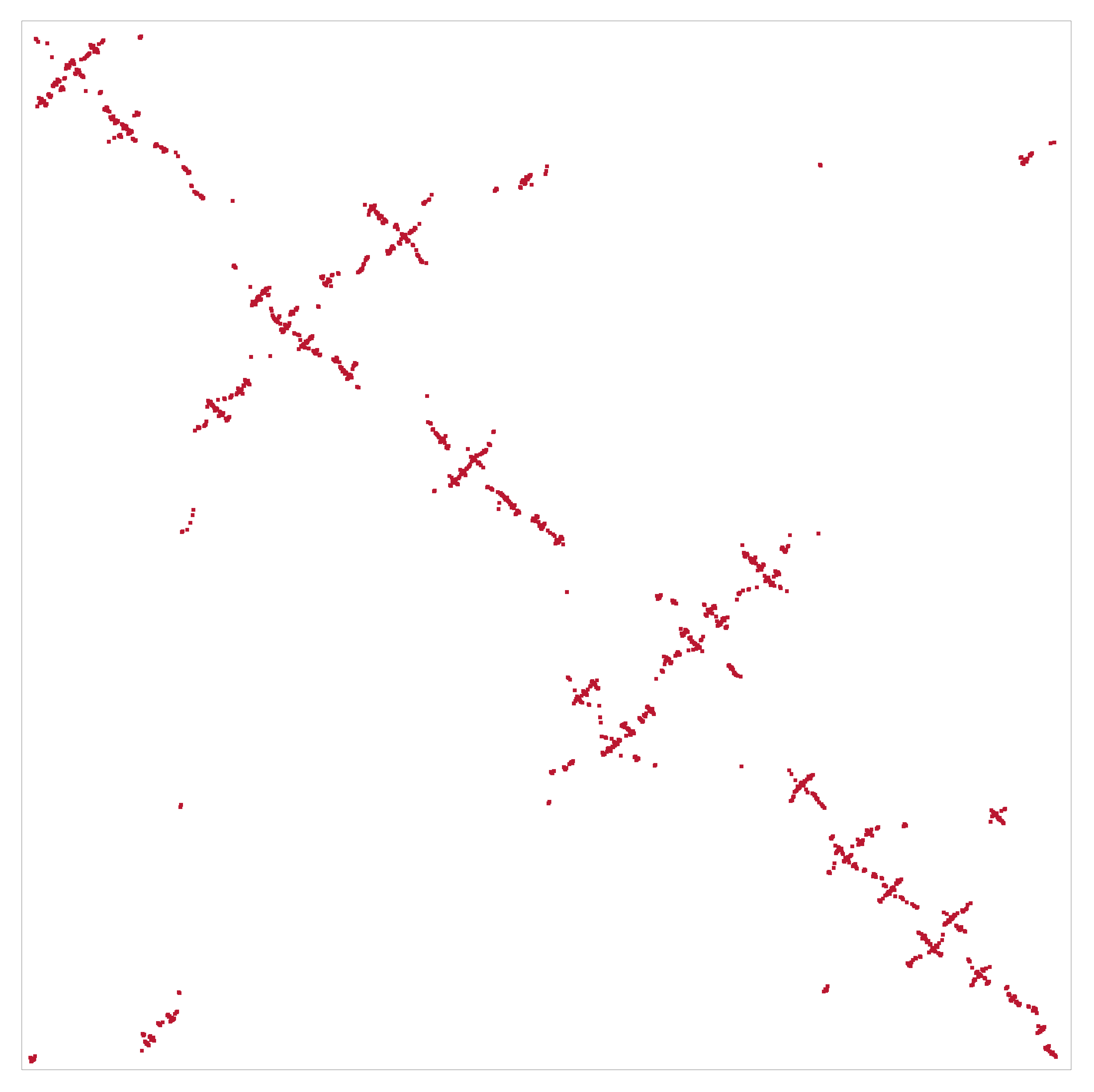}
	\includegraphics[scale=.055]{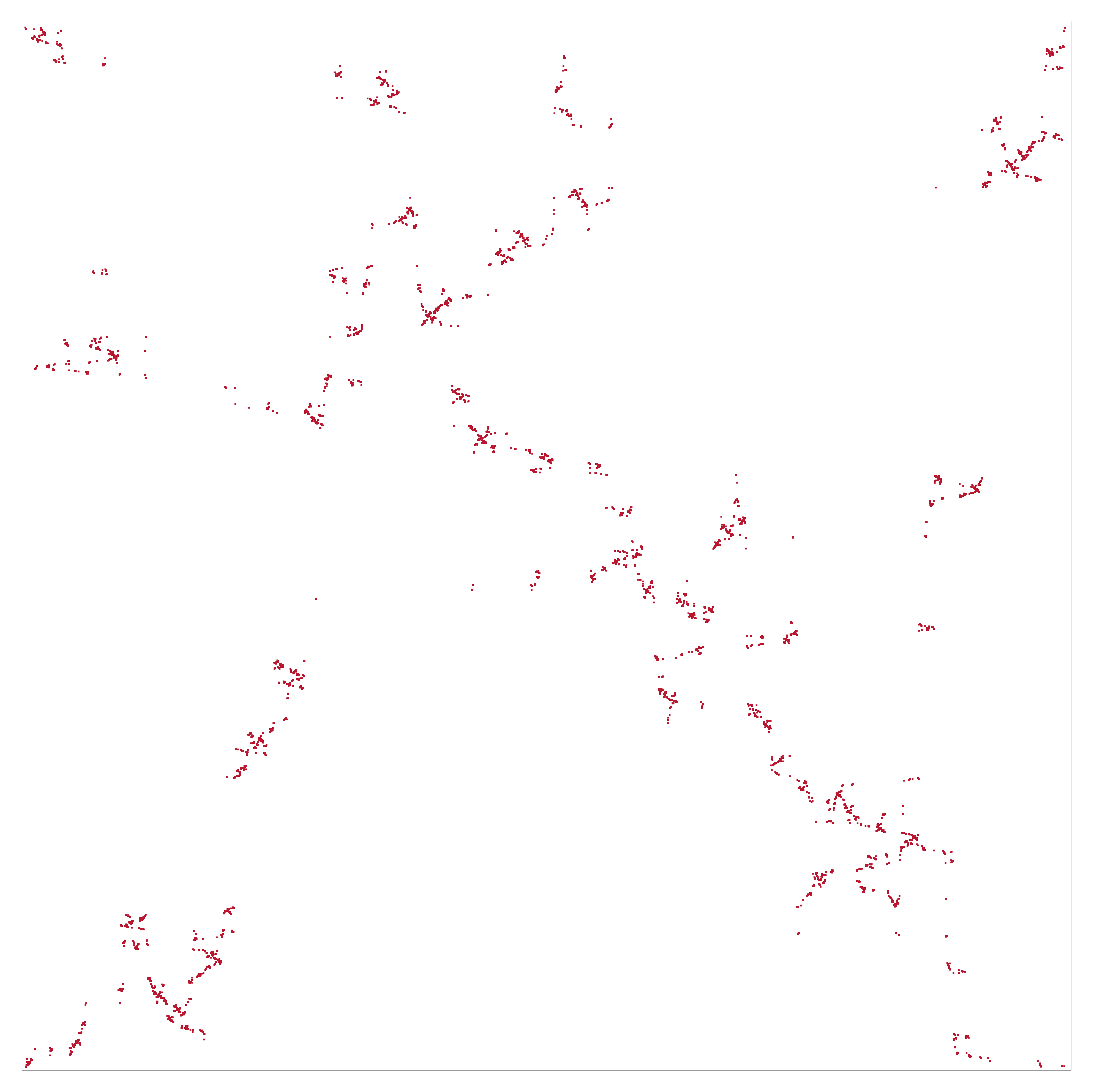}
	\includegraphics[scale=.07]{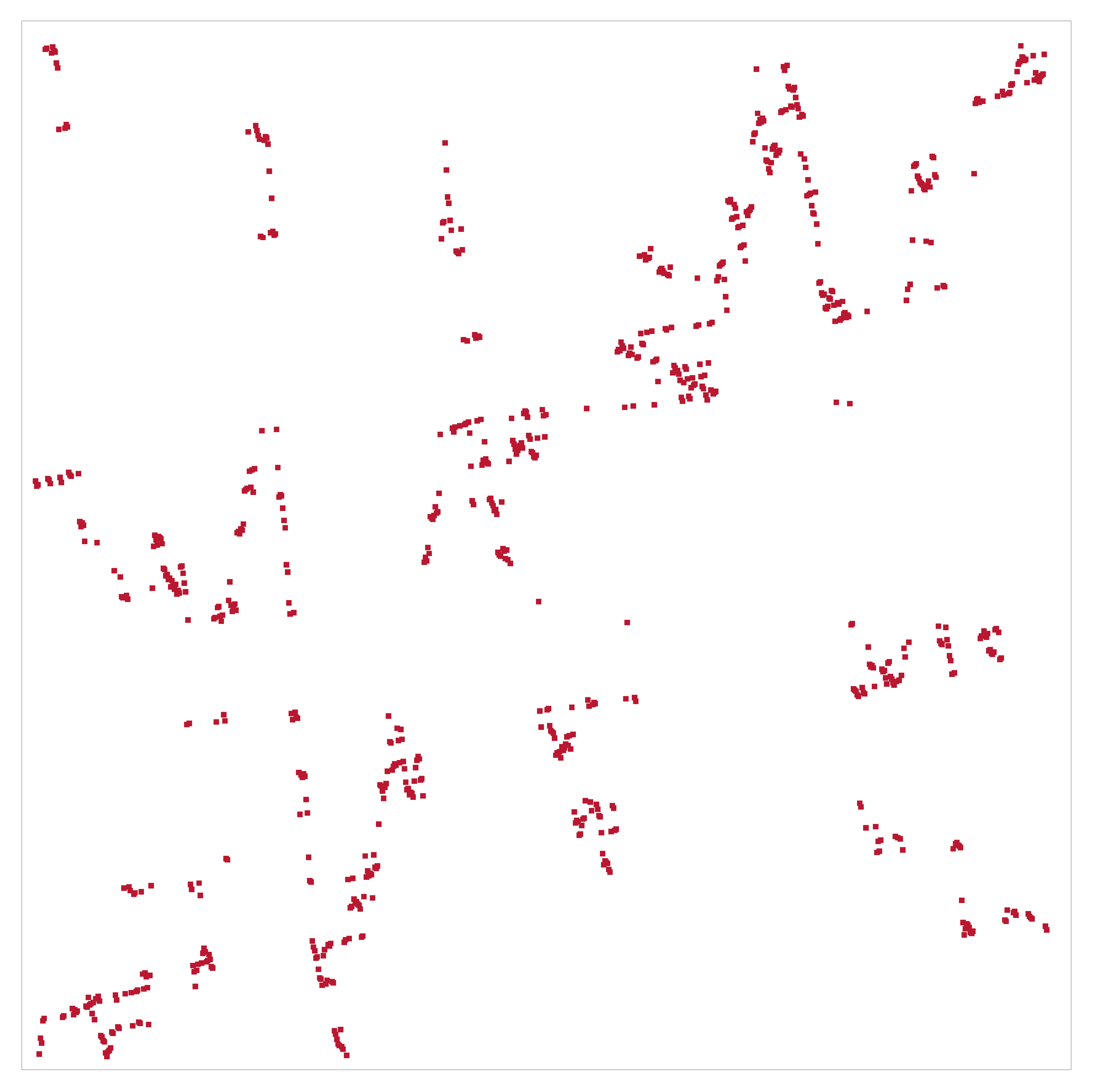}
	\includegraphics[scale=.28]{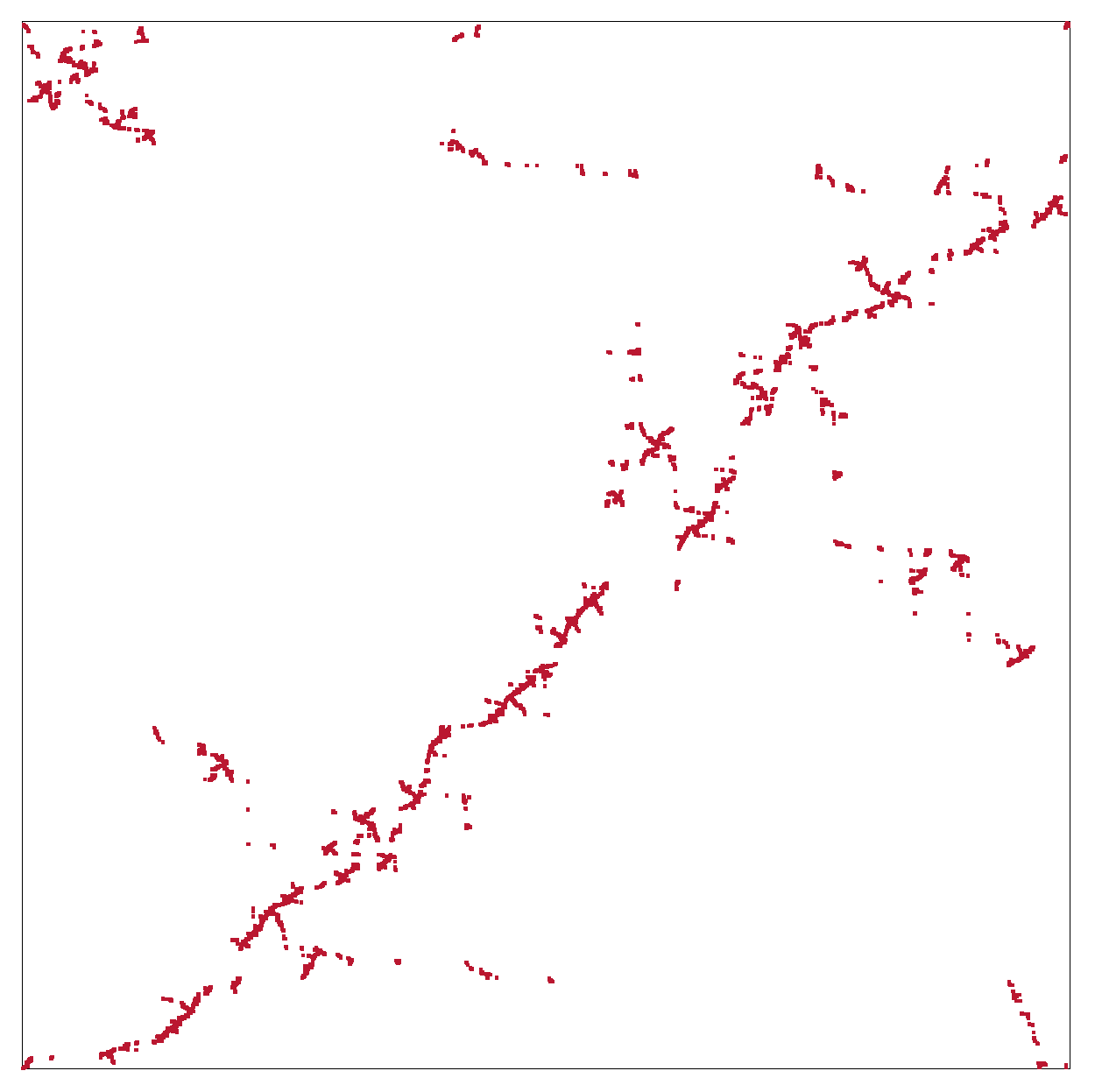}
	\caption{From left to right/top to bottom, large uniform permutations from the following models: Erd\"{o}s–Szekeres permutations, Mallows permutations, square permutations, almost square permutations (with a fifth of external points), $\{1342,2341\}$-avoiding permutations, 231-avoiding permutations, 321-avoiding permutations, 654321-avoiding permutations, separable permutations, Baxter permutations, semi-Baxter permutations, strong-Baxter permutations. All the models are formally defined in this thesis. \label{fig:simulations_models}}
\end{figure}

\section{Pattern-avoiding permutations}

One fundamental aspect of our work is that we do not focus on \emph{uniform} random permutations but we look at \emph{non-uniform} models. There are two main lines of research on non-uniform models of random permutations:

\begin{itemize}
	\item The first one looks at \emph{biased} models of random permutations. Typically, one considers a distribution on the set of permutations where every permutation $\pi$ has probability proportional to $q^{\text{stc}(\pi)}$, where $\text{stc}(\cdot)$ denotes a certain statistics on permutations and $q$ is a positive real-valued parameter. A remarkable model of biased permutations is the Mallows model (see the second picture in \cref{fig:simulations_models}): every permutation $\pi$ has probability proportional to $q^{\text{inv}(\pi)}$, where $\text{inv}(\pi)$ denotes the number of inversions of $\pi$, i.e.\ the number of pairs $i<j$ such that $\pi(i)>\pi(j)$.
		
	\item The second line of research looks at \emph{constrained} models of random permutations. This is the area where this thesis belongs. Typically, one considers the uniform distribution on a given subset of the set of all permutations. A classical way to identify this subset is to use pattern-avoidance, that is, by considering all the permutations avoiding one or multiple patterns (see the third and the last eight pictures\footnote{We remark that Baxter permutations, semi-Baxter permutations and strong-Baxter permutations are defined using \emph{generalized pattern-avoidance}, as we will see in \cref{sec:discrete}.} in \cref{fig:simulations_models}). We will also deal with other constrained models not defined via pattern-avoidance, such as Erd\"{o}s–Szekeres permutations and almost square permutations (see the first and fourth picture in \cref{fig:simulations_models}).
\end{itemize}

One of the big advantages of considering pattern-avoiding permutations is that many elegant and  useful combinatorial constructions are already available in the literature.

The study of patterns in permutations dates back a century, to the works of MacMahon \cite{MR2417935}, where he showed that the permutations that can be partitioned into two decreasing subsequences -- i.e.\ 123-avoiding permutations -- are counted by the Catalan numbers. Other significant early contributions on permutation patterns are due to Erd\"{o}s and Szekeres~\cite{MR1556929} and Schensted \cite{MR121305} on the longest increasing and decreasing subsequence of a permutation. All these works do not explicitly use the permutation patterns terminology.

The modern study of pattern-avoiding permutations started when Knuth published the first volume of \emph{The Art of Computer Programming}~\cite{MR0286317}. After that, around the 90's, the field expanded towards two main connected directions: to determine bijections between pattern-avoiding permutations and other well-studied discrete objects; to enumerate permutations in a given pattern-avoiding family.
For a detailed survey of the relevant literature we refer to \cite[Chapter 12]{MR3408702}, written by Vatter, or to the book of Kitaev~\cite{MR3012380}. 

In \cref{chp:models}, some of the nice combinatorial results established in the literature are presented and some new ones are proved. We will also explain how these results are extremely helpful in providing a convenient way to construct uniform pattern-avoiding permutations.

\section{Scaling limits}

Another fundamental aspect of our work is that we consider geometric notions of convergence for permutations. We discuss in this section the scaling limit approach and in \cref{sect:local_intro} the local limit approach.

\subsection{Scaling limits of discrete structures}

After a suitable rescaling, a simple random walk on $\mathbb{Z}$ converges to the 1-dimensional Brownian motion: probably this is the most famous scaling limit result in probability theory, a.k.a.\ Donsker’s theorem~\cite{MR40613}. In the last thirty years such scaling limits have been intensively studied for different random discrete structures. We mention some examples where this notion has been investigated, without aiming at giving a complete list.

One of the earliest scaling limit results for random discrete structures is due to Logan and Shepp \cite{MR1417317} and Ver\v{s}ik and Kerov \cite{MR0480398}: they investigated Plancherel distributed Young diagrams, proving convergence to the so-called Logan–Shepp–Kerov–Vershik curve $\Omega$.

In the framework of random trees, the systematic study of scaling limits has been initiated by Aldous with the pioneering series of articles about the \emph{Continuum random tree}~\cite{MR1085326,MR1166406,MR1207226}. After that, scaling limits of random planar maps have been thoroughly studied with motivations coming from
string theory and conformal field theory. Convergence results for many models of random planar maps were obtained, both as random metric spaces (with the celebrated results of Le Gall and Miermont~\cite{le2013uniqueness,miermont2013brownian} showing convergence to the Brownian map) and, more recently, as random Riemannian surfaces (with the remarkable achievements of Holden and Sun~\cite{holden2019convergence} showing convergence to the $\sqrt{8/3}$-Liouville quantum gravity).

In statistical mechanics, the study of scaling limits of discrete models is likewise a very active field of research.
For instance, the scaling limit of the discrete Gaussian free field on lattices, i.e., the continuum
Gaussian free field \cite{MR2322706}, has been proven to be the scaling limit of many other random structures, such as uniform domino tilings~\cite{MR1872739}.

Lastly, we mention a notion of scaling limits for dense graphs: \emph{Graphon convergence} was introduced in \cite{MR2455626} and has been a major topic of interest in graph combinatorics ever since – see
\cite{lovasz2012large} for a broad perspective on the field. Heuristically, graphon convergence corresponds to convergence of the rescaled adjacency matrix. As we will see, this notion of convergence shares some similarities with permuton convergence.

\subsection{Permuton limits}\label{sect:perm_lim_intro}

As already mentioned, a notion of scaling limit for permutations, called \emph{permuton}, was recently introduced  in \cite{hoppen2013limits}. Permutons are probability measures on the unit square with uniform marginals, i.e.\ with uniform projections on the axes. They represent the scaling limit of the diagrams of permutations as the size grows to infinity. Any permutation can be interpreted as a permuton through the procedure described in \cref{fig:permtoperm} (for a complete and formal introduction to permuton convergence we refer to \cref{sect:perm_conv}).

\begin{figure}[htbp]
	\begin{minipage}[c]{0.54\textwidth}
		\centering
		\includegraphics[scale=1.5]{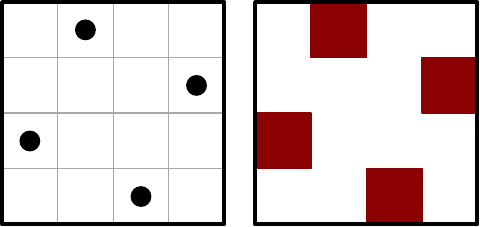}
	\end{minipage}
	\begin{minipage}[c]{0.453\textwidth}
		\caption{On the left-hand side the diagram of the permutation 2413. On the right-hand side the associated permuton: each subsquare in the diagram containing a dot has been endowed with Lebesgue measure of total mass $1/4$. In this way we obtain a probability measure on the unit square with uniform marginals, that is, a permuton.\label{fig:permtoperm}}
	\end{minipage}
\end{figure}

Convergence of permutons is defined as weak convergence of measures. Scaling limits of permutations have been studied in several works. We mention a few of them.
\begin{itemize}
	\item To the best of our knowledge, the first scaling limit result for permutations (without using the setting of permutons) was proved by Romik \cite{MR2266895}. He showed that random Erd\"{o}s–Szekeres permutations (see the first picture in \cref{fig:simulations_models}), i.e.\ uniform permutations of size $n^2$ with no monotone subsequences\footnote{We recall that, by the Erd\"{o}s-Szekeres theorem \cite{MR1556929}, any permutation of size $n^2$ contains at least one monotone subsequence of length $n$.} of length $n+1$, converge to a limiting shape described by an algebraic curve of degree 4. He also determined the internal density of the limiting measure: it is not uniform and the explicit expression is quite involved.
	
	\item With a statistical-mechanical approach, Starr \cite{starr2009thermodynamic} investigated the permuton limit of Mallows permutations (see the second picture in \cref{fig:simulations_models}). Treating the question as a mean-field problem, he was able to calculate the distribution of the permuton limit (after a suitable rescaling of the parameter $q$). This limit is a deterministic measure on the square with an explicit density. Again, such results do not make use of the permuton language. 
	
	\item Another remarkable model of non-uniform random permutations is given by \emph{random sorting networks}~\cite{MR2355610}. A sorting network is a shortest path from $12\dots n$ to $n\dots21$ in the Cayley graph of the symmetric group generated by nearest-neighbour transpositions. There is a natural permutation-valued process associated with a sorting
	network: the value of the process at time $m$ is given by the permutation obtained from $12\dots n$ applying the first $m$ transpositions in the sorting network.
 	Recently, Dauvergne~\cite{dauvergne2018archimedean} determined the permuton limit (without using this vocabulary) of uniform sorting network processes at time $\lfloor nt \rfloor$ showing convergence to the so-called \emph{Archimedean measure} $\mathfrak{A}\mathfrak{r}\mathfrak{c}\mathfrak{h}_t$, the latter being a family of permutons indexed by $t\in[0,1]$.
	
	\item The study of the scaling limits of a uniform random permutation avoiding a pattern of length three (see the sixth and seventh picture in \cref{fig:simulations_models}) was independently initiated by Miner and Pak \cite{miner2014shape} and  Madras and Pehlivan \cite{madras2016structure}. Afterwards, with a series of two articles, Hoffman, Rizzolo and Slivken \cite{hoffman2017pattern,hoffman2017pattern_2} expanded on these early results by exploring the connection of these uniform pattern-avoiding permutations with Brownian excursions. Recently, they further generalized these results to permutations avoiding any monotone pattern (see the eighth picture in \cref{fig:simulations_models}), showing convergence to the traceless Dyson Brownian bridge \cite{hoffman2019dyson}.  We highlight that the permuton limit of a uniform random permutation avoiding either a pattern of length three or a monotone pattern is quite trivial (namely, a diagonal of the square); all these works look at second order fluctuations of the points in the diagram.
	
	\item Dokos and Pak \cite{MR3238333} explored the expected limit shape of the so-called doubly alternating Baxter permutations. In their article they claimed that \emph{"it would be interesting to compute the limit shape of random Baxter permutations"}. In \cref{sect:intro_scaling_limit_results} we provide an answer to this open question.
	
	\item The first concrete and explicit example of convergence formulated in the permuton language was given by Kenyon, Kral,  Radin and Winkler \cite{kenyon2020permutations}. They studied scaling limits of random permutations in which a finite number of pattern densities has been fixed.
	
	\item Lastly, Bassino, Bouvel, F\'eray, Gerin and Pierrot \cite{bassino2018separable} showed that a sequence of uniform random separable permutations of size $n$ converges to the \emph{Brownian separable permuton} (see the ninth picture in \cref{fig:simulations_models}). This is the first instance of a \emph{random} limiting permuton.
	This new limiting object was later investigated by Maazoun \cite{maazoun17BrownianPermuton}. In a second work \cite{bassino2017universal}, the six authors showed that the Brownian permuton has a universality property: they considered uniform random permutations in proper substitution-closed classes and studied their limiting behavior in the sense of permutons, showing that the limit is an elementary one-parameter deformation of the Brownian separable permuton (we will come back to these results in \cref{sect:sub_close_cls}). Finally, in  \cite{bassino2019scaling} the same authors also investigated permuton limits for permutations in classes having a finite combinatorial specification for the substitution decomposition.
\end{itemize}

\section{Local limits}\label{sect:local_intro}

\subsection{Local limits of discrete structures}\label{sect:loc_lim_graph}

In parallel to scaling limits, there is a second notion of limits of discrete structures: the local limits. Informally, scaling limits look at the convergence of the objects from a global
point of view (after a proper rescaling of the distances between points of the objects), while local limits look at discrete objects in a neighborhood of a distinguished point (without rescaling distances). As done for scaling limits, we mention some examples where this notion has been studied, again without aiming at giving a complete overview.

Local limit results around the root of random trees were first implicitly proved by Otter \cite{otter1949multiplicative} and then explicitly by Kesten \cite{kesten1986subdiffusive}, and Aldous and Pitman \cite{aldous1998tree}. Janson \cite{janson2012simply} gave a unified treatment of the local limits of simply generated random trees as the number of vertices tends to infinity. Recently, Stufler \cite{stufler2016local} studied also the local limits for large Galton--Watson trees around a uniformly chosen vertex, building on previous results of Aldous \cite{aldous1991asymptotic} and Holmgren and Janson \cite{holmgren2017fringe}. 

Although implicit in many earlier works, the notion of local convergence around a random vertex (sometimes called \emph{weak local convergence}) for random graphs was formally introduced by Benjamini and Schramm \cite{benjamini2001recurrence} and Aldous and Steele \cite{aldous2004objective}. One would expect that a limiting object for this topology should "look the same" when regarded from any of its vertices (because of the uniform choice of the root). This property is made precise by the notion of \emph{unimodular} random rooted graph, and the weak local limit of any sequence of random graphs is indeed proven to be unimodular (see \cite{benjamini2001recurrence}). However, it is an open problem\footnote{The problem has been recently solved for random \emph{planar} graphs \cite{timar2019unimodular}.}, that goes under the name of \emph{"Sofic problem"}, to determine if every unimodular rooted random graph is the local limit of a sequence of finite random graphs rooted uniformly at random. We will expand upon this notion of unimodularity and the Sofic problem in \cref{sect:sofic}. 

Local convergence has been studied in the framework of random planar maps as well. For example, a well-known result is that the local limit for random  triangulations/quadrangulations is the uniform infinite planar triangulation/quadrangulation (UIPT/UIPQ) (see for instance \cite{angel2003uniform,krikun2005local,stephenson2018local}). Many other local limits of random planar maps are known, for instance it was shown recently by Budzinski and Louf~\cite{MR4199439} that the local limits of uniform random triangulations whose genus is proportional to the number of faces are the Planar Stochastic Hyperbolic Triangulations (PSHT), defined by Curien in \cite{MR3520011}.

\subsection{Local limits for permutations}\label{sect:local_lim_intro}

To the best of our knowledge, a local limit approach had not been investigated in the framework of permutations prior to this thesis. One of the goals of the present work is to fill this gap.

We give here a brief description of local convergence for permutations (a detailed presentation can be found in \cref{sect:local_theory}). In the context of local convergence, we look at permutations with a distinguished entry, called the \emph{root} (see \cref{rest}). We say that a pair $(\sigma,i)$ is a \emph{finite rooted permutation} if $\sigma$ is a permutation and $i$ an index of $\sigma$, i.e.\ an integer between 1 and the size of $\sigma$. 

To each rooted permutation $(\sigma,i)$, we associate a total order $\preccurlyeq_{\sigma,i}$ on a finite interval of integers containing 0, denoted by $A_{\sigma,i}.$
The total order $(A_{\sigma,i},\preccurlyeq_{\sigma,i})$
is simply obtained from the diagram of $\sigma$ as follows (see the left-hand side of \cref{rest}): We shift the indices of the $x$-axis in such a way that the column containing the root of the permutation has index zero (the new indices are displayed under the columns of the diagram). Then we set $j\preccurlyeq_{\sigma,i}k$ if the point in column $j$ is lower than the point in column $k.$

\begin{figure}[htbp]
	\centering
	\includegraphics[scale=.8]{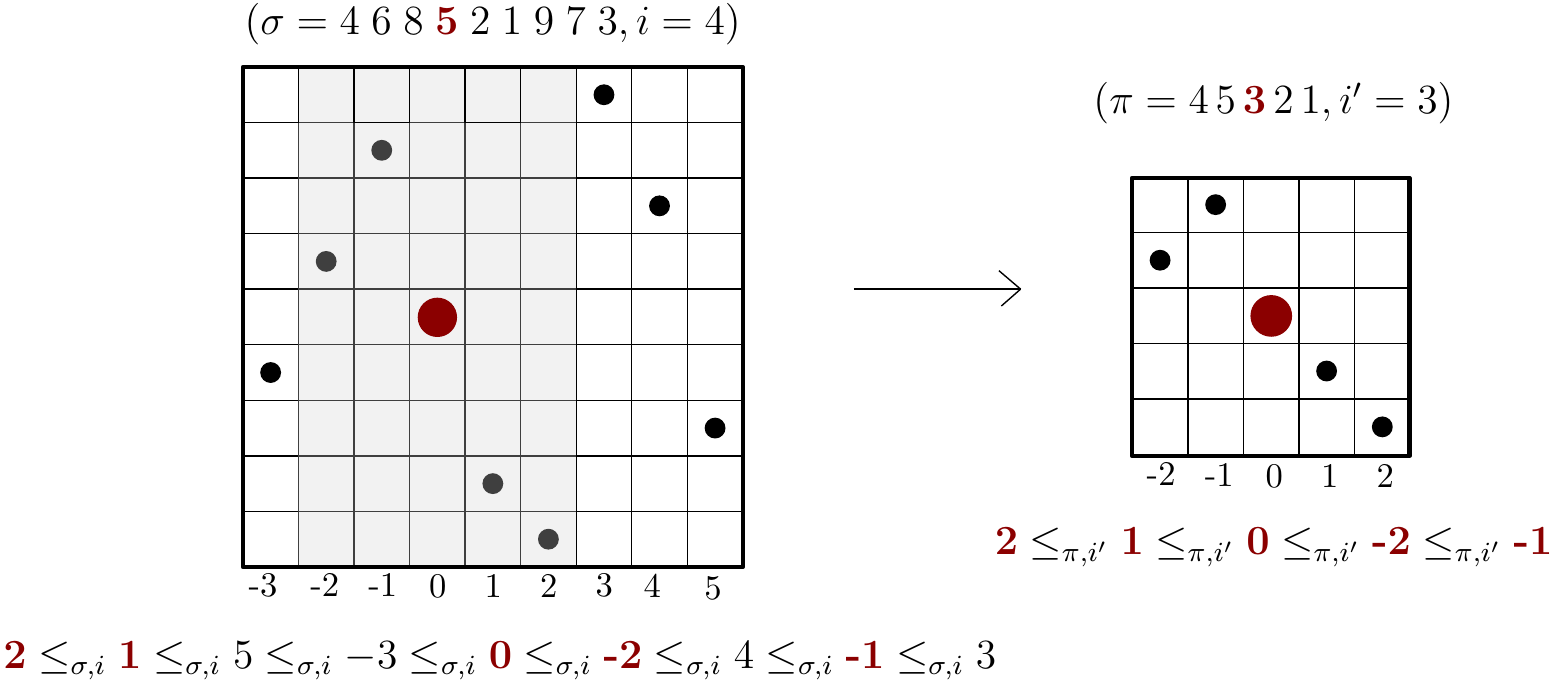}\\
	\caption{Two rooted permutations and the associated total orders. The one on the right is the $2$-restriction of the one on the left (see below for explanations). \label{rest}}
\end{figure}

Since this correspondence defines a bijection, we may identify every rooted permutation $(\sigma,i)$ with the total order $(A_{\sigma,i},\preccurlyeq_{\sigma,i})$.
In light of this identification, it is natural to call \emph{infinite rooted permutation} a pair $(A,\preccurlyeq)$ where $A$ is an infinite interval of integers containing 0 and $\preccurlyeq$ is a total order on $A$.

In order to define a notion of local convergence, we introduce a notion of neighborhood of
the root, called \emph{$h$-restriction}. It can be thought of as
(the diagram of) the pattern induced by a "vertical strip" of width $2h+1$ around the root of the permutation (see again \cref{rest}), or equivalently as
the restriction of the order $(A,\preccurlyeq)$ to $A\cap[-h,h].$ 

We then say that a sequence $(A_{n},\preccurlyeq_{n})_{n\in\Z_{>0}}$ of rooted permutations  \emph{converges locally} to a rooted permutation $(A,\preccurlyeq)$ if, for all $h\in\Z_{>0},$ the $h$-restrictions of the sequence $(A_{n},\preccurlyeq_{n})_{n\in\Z_{>0}}$ converge to the $h$-restriction of $(A,\preccurlyeq)$.

We extend this notion of local convergence for rooted permutations to (unrooted) permutations, rooting them at a uniformly chosen index of $\sigma$. With this procedure, a fixed permutation $\sigma$ naturally identifies a random variable $(\sigma,\bm{i})$ taking values in the set of finite rooted permutations.
In this way, the following notion of weak local convergence becomes natural. Consider a sequence of permutations $(\sigma_n)_{n\in\Z_{>0}}$, and denote by $\bm{i}_n$ a uniform index of $\sigma_n$. We say that $(\sigma_n)_{n\in\Z_{>0}}$ \emph{Benjamini--Schramm converges} to a random (possibly infinite) rooted permutation $(\bm{A},\bm{\preccurlyeq})$ if the sequence $(\sigma_n,\bm{i}_n)_{n\in\Z_{>0}}$ converges in distribution to $(\bm{A},\bm{\preccurlyeq})$ with respect to the local topology defined above. We point out that our choice of the terminology comes from the analogous notion of convergence for graphs (see \cite{benjamini2001recurrence}).

Note that the Benjamini--Schramm convergence has been introduced so far for sequences of deterministic permutations. We also extend this notion  to sequences of \emph{random} permutations $(\bm{\sigma}_n)_{n\in\Z_{>0}}.$ The presence of two sources of randomness, one for the choice of the permutation and one for the (uniform) choice of the root, leads to two non-equivalent possible definitions: the \emph{annealed} and the \emph{quenched}  version of the Benjamini--Schramm convergence. Intuitively, in the second definition the random permutation is frozen, whereas in the first one the random permutation and the choice of the root are treated on the same level.

\section{A connection with limits of permutation statistics}\label{sect:lim_stat}

So far we have seen two natural geometric ways of defining limits of permutations, i.e.\ permutons and Benjamini--Schramm limits. As mentioned at the beginning of this chapter, another classical approach is to look at limits of statistics on permutations. We now quickly explain this additional point of view and then we discuss how the two approaches are related.

\medskip

Let us start by considering the following example: Suppose we are given some samples of a random permutation of size $n$, it is then natural to wonder if this sample is uniform or not. Answering this question is not a straightforward task because we have $n!$ possible outputs. A way to do it is to measure the behavior of some relevant statistics on our samples. Say for convenience that we look at the statistics $\text{stc}(\cdot)$.
Once these data have been collected, one would like to compare them with the expected behavior for a uniform random permutation (see for instance \cite{MR3476619}). This leads to the following natural questions: Consider a uniform random permutation $\bm\sigma_n$ of size $n$.
\begin{itemize}
	\item Does $\text{stc}(\bm\sigma_n)$ satisfy a law of large numbers? What are the mean and the variance of $\text{stc}(\bm\sigma_n)$?
	\item Does $\text{stc}(\bm\sigma_n)$ satisfy a central limit theorem?
\end{itemize}

These kinds of questions have been intensively investigated -- with several other motivations in addition to those mentioned above, see for instance \cite[Chapter 6]{MR964069} -- both for uniform permutations (see for instance \cite{MR0131369, MR652736, MR1682248, MR2118603,MR3338847}) and more recently for non-uniform ones (see for instance \cite{MR3091722,MR3729641,hoffman2016fixed,MR4058419,he2020central}).
 
The notions of permutons and Benjamini--Schramm limits are also well-suited to answer the first two questions for many statistics $\text{stc}(\cdot)$, specifically  those which can be expressed in terms of patterns and consecutive patterns in permutations. Indeed, in \cref{chp:conv_theories}, we will see that permuton convergence (resp.\ Benjamini--Schramm convergence) for a sequence of permutations is equivalent to convergence of all proportions of patterns (resp.\ consecutive patterns). More precisely, we have the following characterizations (see Theorems \ref{thm:randompermutonthm} and \ref{strongbsconditions}): For any $n\in\Z_{>0}$, let $\bm\sigma_n$ be a random permutation of size $n$, then
\begin{itemize}
		\item $(\bm \sigma_n)_{n\in\Z_{>0}}$ converges in distribution w.r.t.\ the permuton topology $\iff$ $\big(\pocc(\pi,\bm\sigma_n)\big)_{\pi \in \S}$ converges in distribution w.r.t.\ the product topology;
		\item $(\bm \sigma_n)_{n\in\Z_{>0}}$ quenched Benjamini-Schramm converges $\iff$  $\left(\widetilde{\cocc}(\pi,\bm{\sigma}_n)\right)_{\pi\in\mathcal{S}}$ converges in distribution w.r.t.\ the product topology;
\end{itemize}
where $\pocc(\pi,\sigma)$ (resp.\ $\widetilde{\cocc}(\pi,\sigma)$) denotes the proportion of occurrences (resp.\ consecutive occurrences) of the patterns $\pi$ in $\sigma$. The two results were established in \cite{bassino2017universal} and \cite{borga2020localperm}, respectively.

\section{Overview of remaining chapters}

\textbf{\emph{Note.}} \emph{Precise statements of all the main probabilistic results proved in this thesis can be found in the first sections of \cref{chp:local_lim,chp:square,chp:perm_lim}, while the combinatorial ones are collected in \cref{chp:models} and \cref{cref:vediuyewvdb}.  \cref{sect:local_theory}, where local convergence is defined, contains original results of the author as well. A summarizing table of all the main theorems proved in this thesis is displayed at the end of this section.} 
	
\noindent\emph{We highlight that, in order to give a more pleasant and organic presentation of the various results obtained by the author during his Ph.D., we opt for a completely different organization of the material compared to the exposition in the various papers. Nevertheless, most of the material in the following chapters is taken from the articles\footnote{During his thesis, the author worked also on an independent project with Cavalli \cite{borga2021quenched} on a model for multi-player leagues. This work is not included in this manuscript.} \cite{borga2019almost,borga2020asymptotic,borga2020localperm,borga2019feasible,borga2020feasible,borga2020scaling,borga2020square,borga2020decorated}.}

\bigskip

\noindent \cref{chp:conv_theories}: \textbf{The theories of permutons
 \& local limits}.

\medskip

\noindent We properly introduce the theories of permutons and local limits of permutations briefly described above. We characterize these topologies in terms of convergence of proportions of patterns and consecutive patterns, as hinted at the end of the previous section. We are able to give a complete characterization of Benjamini-Schramm limits in terms of a \emph{shift-invariant} property, solving the \emph{Sofic problem} for permutations.  At the end of the chapter we also explore some connections between these two topologies and some feasible regions for patterns and consecutive patterns. In particular, we give a complete description of the feasible region for consecutive patterns as the cycle polytope of a specific graph, called \emph{overlap graph}.

\medskip

\noindent\emph{\cref{chp:conv_theories} builds on the works \cite{borga2020localperm,borga2019feasible,borga2020feasible}.}

\bigskip
	
\noindent \cref{chp:models}: \textbf{Models of random permutations \&
 combinatorial constructions}.

\medskip
 
\noindent This is a long journey in the beautiful combinatorics of (constrained) permutations. We present several remarkable constructions (some known/some new) relating permutations with other combinatorial structures. We start by connecting permutations avoiding a pattern of length three with  families of trees, and permutations in substitution-closed classes with some decorated trees. Then we show that square permutations can be encoded by simple walks. Lastly, permutations encoded by generating trees are connected with walks conditioned to stay positive, and Baxter permutations with multi-dimensional walks in cones, with a family of planar maps called \emph{bipolar orientations}, and with some new discrete structures called \emph{coalescent-walk processes}.

\medskip
 
\noindent\emph{\cref{chp:conv_theories} builds on the works \cite{borga2020asymptotic,borga2020localperm,borga2020scaling,borga2020square,borga2020decorated}.}

 \bigskip
	
\noindent \cref{chp:local_lim}: \textbf{Local limits: Concentration \&
 non-concentration phenomena}.

\medskip

\noindent We present many results establishing Benjamini--Schramm convergence for random constrained permutations. We show that, surprisingly, the limiting quenched local object is deterministic for many uniform pattern-avoiding permutations, i.e.\ a concentration phenomena occurs. This is not, however, the case for square permutations. We also present a central limit theorem for the proportion of consecutive patterns in permutations encoded by generating trees. We end this chapter by proving quenched local convergence for Baxter permutations and bipolar orientations.

\medskip

\noindent\emph{\cref{chp:local_lim} builds on the works \cite{borga2020asymptotic,borga2020localperm,borga2020scaling,borga2020square,borga2020decorated}.}

\bigskip

\noindent\cref{chp:square}: \textbf{Phase transition for square and almost
 square permutations, fluctuations,
 and generalizations to other models}.

\medskip

\noindent A \emph{record} in a permutation is an entry which is either larger or smaller than all entries either before or after it (there are therefore four types of records). Entries which are not records are called \emph{internal} points. \emph{Almost square permutations} are permutations with a fixed number of internal points. We explore permuton limits of uniform permutations in the families of almost square permutations with exactly $k$ internal points.
We first investigate the case when $k=0$; this is the class of \emph{square permutations}, i.e.\ permutations where every point is a record. The starting point for our results is a sampling procedure for asymptotically uniform square permutations presented in \cref{sect:sq_perm}. Building on that, we characterize the global behavior by showing that square permutations have a permuton limit described by a random rectangle. We also explore the fluctuations of this random rectangle, which can described through coupled Brownian motions.
We then identify the permuton limit of almost square permutations with $k$ internal points, both when $k$ is fixed and when $k$ tends to infinity along a negligible sequence with respect to the size of the permutation. Here the limit is again a random rectangle, but this time of a different nature: we show that a phase transition on the shape of the limiting rectangles occurs for different values of $k$.

We finally present a conjecture for the limiting shape of almost square permutations with $k$ internal points, when $k$ tends to infinity along a sequence of the same order as the size of the permutation. We further show how our methods can be used in other contexts.

\medskip

\noindent\emph{\cref{chp:square} builds on the works \cite{borga2019almost,borga2020square}.}

\bigskip
 
\noindent\cref{chp:perm_lim}: \textbf{Permuton limits: A path towards a
 new universality class}.

\medskip

\noindent One of the recurring themes in modern probability theory is the research of \emph{canonical} and \emph{universal} limiting objects, i.e.\ objects that are characterized by symmetries and that arise as limits of various random discrete structures. The most famous example is undoubtedly the Brownian motion, but in the recent years, many other canonical and universal limiting objects have been discovered, such as the Continuum random tree, the Brownian map, the Gaussian free field, the Liouville quantum gravity, and the SLE curves. These objects arise respectively as limits of discrete trees, discrete planar maps, discrete walks and more generally of discrete models coming from physics.

This chapter is devoted to universal limits for permutations. We give a new proof of the universality of the one-parameter deformation of the Brownian separable permuton, called \emph{biased Brownian separable permuton}, by showing that uniform random permutations in proper substitution-closed classes converge to this limiting object. Then, we show that Baxter permutations converge to a new limiting object called the \emph{Baxter permuton}, answering to the open question of Dokos and Pak \cite{MR3238333}. We finally conjecture that a two-parameter deformation of the Baxter permuton, called the \emph{skew Brownian permuton}, describes a new universal limiting object that includes as particular cases both the Baxter and the biased Brownian separable permuton. 

While proving permuton convergence for Baxter permutations, we also give a scaling limit result for the four trees characterizing bipolar orientations and their dual maps, answering Conjecture 4.4 of Kenyon, Miller, Sheffield, Wilson \cite{MR3945746} about convergence of bipolar orientations (jointly with their dual maps) to the $\sqrt{4/3}$-Liouville quantum gravity.
 
\medskip

\noindent\emph{\cref{chp:perm_lim} builds on the works \cite{borga2020scaling,borga2020decorated}.}

\bigskip

We display a summarizing table containing all the main theorems proved in this thesis. 

\begin{figure}[H]
	\begin{adjustwidth}{-1mm}{}
\begin{tabularx}{1.0\textwidth} { 
		| >{\centering\arraybackslash\hsize=.26\hsize}X 
		| >{\centering\arraybackslash\hsize=.21\hsize}X
		| >{\centering\arraybackslash\hsize=.14\hsize}X 
		| >{\centering\arraybackslash\hsize=.13\hsize}X
		| >{\centering\arraybackslash\hsize=.26\hsize}X | }
	\hline
	\textbf{Models/Topics} {\scriptsize(Section where introduced)} & \textbf{Discrete objects in bijection} & \textbf{Quenched B--S limits} & \textbf{Permuton limits} & \textbf{Other results} \\
	\hline
	 Perm.\ avoiding a pattern of length three (\cref{sect:perm_len_three})  & Binary trees/ Rooted plane trees &  Thms.\ \ref{thm_1} \& \ref{thm_2}  & (diagonal$^{\scriptscriptstyle{*}}$) {\scriptsize $^*$Consequence of  \cite{hoffman2017pattern}}& \hspace{2mm} Thm.\ \ref{thm:examples_ok} \newline {\scriptsize(CLT for consecutive patterns)} \\
	\hline
	Substitution-closed classes (\cref{sect:sub_close}) & Forests of decorated trees Thm. \ref{te:bijection} &  Thm.\ \ref{thm:local_intro} & {\fontsize{9.4}{6}\selectfont Thm.\ \ref{thm:scaling_intro}$^{\scriptscriptstyle{*}}$} \newline {\scriptsize $^*$New proof of} {\fontsize{5.2}{6}\selectfont \cite [Thm.\ 1.10]{bassino2017universal}}  &  / \\
	\hline
	(Almost) square permutations (\cref{sect:sq_perm})  & Simple walks Lemmas \ref{omega_size} \& \ref{square_is_rect}  &  Thm.\ \ref{thm:local_conv} \newline {\scriptsize(Only for square permutations)}   & Thm.\ \ref{thm:phase_trans} \newline \phantom E  \& \newline Conj.\ \ref{conj:phase_trans} &  Thm.\ \ref{thm:fluctuations}  {\scriptsize(Fluctuations for square permutations)} \newline Thms.\ \ref{approx_size} \& \ref{approx_size_2} \newline {\scriptsize (Asymptotic enumeration of almost square permutations)} \\
	\hline
	Perm.\ families encoded by generat.\ trees (\cref{sect:gen_tree-perm}) &  Walks conditioned to stay positive Prop.\ \ref{prop:bij_gen_tree} &   Corollary \ref{corl:main_thm} & / &  \hspace{2mm} Thm.\ \ref{thm:main_thm_CLT}\newline {\scriptsize(CLT for consecutive patterns)} \\
	\hline
	 Baxter permutations (\cref{sec:discrete})  & Bipolar orientations, Tandem walks, Coalescent-walk \phantom , \phantom , processes \newline Thm.\ \ref{thm:diagram_commutes} &  \hspace{1.1mm} Thm.\ \newline \phantom , \ref{thm:local} \newline {\scriptsize(Gives a joint quenched B--S limit for all the objects)}  & Thm.\ \ref{thm:joint_intro} &  \hspace{2mm} Thm.\ \ref{thm:discret_coal_conv_to_continuous} \newline {\scriptsize(Scaling limit of coalescent-walk processes)} \newline \phantom E Thm.\ \ref{thm:joint_scaling_limits} \newline {\fontsize{5.8}{6}\selectfont (Scaling limit of bipolar orientations jointly with other objects)} \\
	\hline
	321-avoiding perm.\ with few internal points (\cref{sect:321av})  & / &  /  & (diagonal) &  Thm.\ \ref{thm:fluctuations_321} {\tiny(Fluctuations)} \newline  \phantom I \phantom . Thm.\ \ref{thin red line} \newline  {\scriptsize(Asymptotic enumeration)} \\
	\hline
	Feasible regions (\cref{sect:feas}) & / &  /  & / &  \hspace{2mm} Thm.\ \ref{thm:main_res} \newline {\scriptsize(Description as polytope)} \\
	\hline
	\emph{Sofic problem} for perm.\ (\cref{sect:sofic}) & / &  /  & / &  Thm.\ \ref{ewoifhnoipewjfpoew} {\scriptsize(Characterization B--S limits)}\\
	\hline
	\multicolumn{5}{|>{\centering\arraybackslash\hsize=1.12\hsize}X|}{{\tiny The bijections between (i) permutations avoiding a pattern of length three and binary trees/rooted plane trees, (ii) Baxter permutations,  bipolar orientations and tandem walks, were already established in the literature. All the other results in this table are original results proved in this thesis.}}\\
	\hline
\end{tabularx}
\end{adjustwidth}
\end{figure}

\section{Notation and standard definitions}
\textbf{\emph{Note}}. \emph{Most of the notation introduced here and in the following chapters is summarized in a table at the end of this manuscript, page \pageref{not:notation}.}

\subsection{Permutations and patterns}\label{sect:not_perm_patt}
For any $n\in\Z_{>0},$ we denote the set of permutations of $[n]=\{1,2,\dots,n\}$ by  $\gls*{perm_n}$. We write permutations of $\mathcal{S}_n$ in one-line notation as $\sigma=\gls*{one_line}.$ For a permutation $\sigma\in\mathcal{S}_n$ the \emph{size} $n$ of $\sigma$ is denoted by $\gls*{size_perm}$. We let $\gls*{perm}\coloneqq\bigcup_{n\in\Z_{>0}}\mathcal{S}_n$ be the set of finite permutations\footnote{This convention is used for all combinatorial families studied in this manuscript, that is, if $\mathcal C$ is combinatorial family of objects of finite size, then  $\mathcal C_n$ denotes the combinatorial family of objects of size $n$.}. We write sequences of permutations in $\mathcal{S}$ as $(\sigma_n)_{n\in\Z_{>0}}.$ 

We often view a permutation $\sigma$ as a diagram, i.e.\ the set of points of the Cartesian plane with coordinates $(j,\sigma(j))$, for instance

\begin{equation}\label{eq:exemp_perm_dia}
\sigma=3562174=
\begin{array}{lcr}
\begin{tikzpicture}
\begin{scope}[scale=.3]
\permutation{3,5,6,2,1,7,4}
\draw[thick] (1,1) rectangle (8,8);
\end{scope}
\end{tikzpicture}
\end{array}.
\end{equation}

If $x_1\dots x_n$ is a sequence of distinct numbers, let $\gls*{standa}$ be the unique permutation $\pi$ in $\mathcal{S}_n$ that is in the same relative order as $x_1\dots x_n,$ i.e.\ $\pi(i)<\pi(j)$ if and only if $x_i<x_j.$
Given a permutation $\sigma\in\mathcal{S}_n$ and a subset of indices $I\subset[n]$, let $\gls*{pattern_ind}$ be the permutation induced by $(\sigma(i))_{i\in I},$ namely, $\pat_I(\sigma)\coloneqq\std\big((\sigma(i))_{i\in I}\big).$
For example, if $\sigma=87532461$ and $I=\{2,4,7\}$ then $\pat_{\{2,4,7\}}(87532461)=\std(736)=312$.

Given two permutations, $\sigma\in\mathcal{S}_n$ for some $n\in\Z_{>0}$ and $\pi\in\mathcal{S}_k$ for some $k\leq n,$ and a set of indices $I=\{i_1 < \ldots < i_k\}$, we say that $\sigma(i_1) \ldots \sigma(i_k)$ is an \emph{occurrence} of $\pi$ in $\sigma$ if $\pat_I(\sigma)=\pi$ (we also say that $\pi$ is a \emph{pattern} of $\sigma$). If the indices $i_1, \ldots ,i_k$ form an interval, then we say that $\sigma(i_1) \ldots \sigma(i_k)$ is a \emph{consecutive occurrence} of $\pi$ in $\sigma$ (we also say that $\pi$ is a \emph{consecutive pattern} of $\sigma$).
We denote intervals of integers using the standard notation\footnote{Whenever there might be confusion if an interval is an integer interval or a real interval, we explicitly state it.} $[n,m]=\{n,n+1,\dots,m\}$ for $n,m\in\Z_{>0}$ with $n\leq m$. 

\begin{exmp} 
	The permutation $\sigma=1532467$ contains $1423$ as a pattern but not as a consecutive pattern and $321$ as consecutive pattern. Indeed $\text{pat}_{\{1,2,3,5\}}(\sigma)=1423$ but no interval of indices of $\sigma$ induces the permutation $1423.$ Moreover, $\text{pat}_{[2,4]}(\sigma)=\text{pat}_{\{2,3,4\}}(\sigma)=321.$
\end{exmp}
We say that $\sigma$ \emph{avoids} $\pi$ if $\sigma$ does not contain $\pi$ as a pattern. We point out that the definition of $\pi$-avoiding permutations refers to patterns and not to consecutive patterns. Given a set of patterns $B\subset\mathcal{S},$ we say that $\sigma$ \emph{avoids} $B$ if $\sigma$ avoids $\pi,$ for all $\pi\in B$. We denote by $\gls*{class_n}$ the set of $B$-avoiding permutations of size $n$ and by $\Av(B)\coloneqq\bigcup_{n\in\Z_{>0}}\Av_n(B)$ the set of $B$-avoiding permutations of arbitrary finite size\footnote{For a set $B$ of classical patterns, the set $\gls*{class}$ is called \emph{permutation class}. In this manuscript we also consider permutations avoiding generalized patterns, like Baxter permutations; in this case we do not use the terminology \emph{class}.}.

In the following, let $\sigma$ be a permutation of size $n$ and $\pi$ a pattern of size $k\leq n$. We denote by $\gls*{occ}$ the number of occurrences of $\pi$ in $\sigma$, that is,
\begin{equation}
\occ(\pi,\sigma)\coloneqq\#\big\{I \subseteq [n] \big|\#I=k, \pat_I(\sigma)=\pi\big\},
\end{equation}
where $\#$ denotes the cardinality of a set.
Moreover, we denote by $\gls*{pocc}$ the proportion of occurrences of $\pi$ in $\sigma,$ that is,
\begin{equation}
\pocc(\pi,\sigma)\coloneqq\frac{\occ(\pi,\sigma)}{\binom{n}{k}}.
\end{equation}
Similarly, we denote by $\gls*{c_occ}$ the number of consecutive occurrences of $\pi$ in $\sigma$, that is,
\begin{equation}
\label{cocc}
\cocc(\pi,\sigma)\coloneqq\#\Big\{I\subseteq[n]\Big|\#I=k,\;I\text{ is an interval, } \text{pat}_I(\sigma)=\pi\Big\},
\end{equation}
and we denote by $\gls*{pc_occ}$ the proportion of consecutive occurrences of $\pi$ in $\sigma,$ that is,
\begin{equation}
\widetilde{\cocc}(\pi,\sigma)\coloneqq\frac{\cocc(\pi,\sigma)}{n}.
\end{equation}
\begin{rem}
	The natural choice for the denominator of the previous expression should be \break
	$n-k+1$ instead of $n,$ but we make this choice for later convenience (for example, in order to give a probabilistic interpretation of the quantity $\widetilde{\cocc}(\pi,\sigma)$). Moreover, for every fixed $k,$ this makes no difference in the asymptotics when $n$ tends to infinity. 
\end{rem}

We also introduce two classical operations on permutations. We denote by $\gls*{o_plus}$ the \emph{direct sum} of two permutations, i.e.\ for $\tau\in\S_m$ and $\sigma\in\mathcal{S}_n$,
\begin{equation}
	\tau\oplus\sigma=\tau(1)\dots\tau(m)(\sigma(1)+m)\dots(\sigma(n)+m)=
	\begin{minipage}[h]{0.12\linewidth}
		\includegraphics[scale=.67]{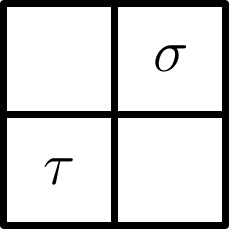}
	\end{minipage},
\end{equation}
and we denote by $\oplus_{\ell}\,\sigma$ the direct sum of $\ell$ copies of $\sigma$ (we remark that the operation $\oplus$ is associative). A similar definition holds for the \emph{skew sum} $\gls*{o_minus}$, 
\begin{equation} \label{eq:skew_sum}
\tau\ominus\sigma=(\tau(1)+n)\dots(\tau(m)+n)\sigma(1)\dots\sigma(n)=
	\begin{minipage}[h]{0.12\linewidth}
		\includegraphics[scale=.67]{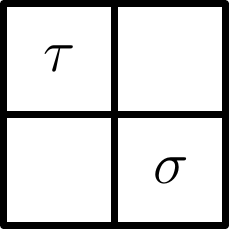}
		\end{minipage}.
\end{equation}
We say that a permutation is $\oplus$-indecomposable (resp.\ $\ominus$-indecomposable) if it cannot be written as the direct sum (resp.\ skew-sum) of two non-empty permutations.

Given a permutation $\sigma\in\mathcal{S}$ we recall that a left-to-right maximum of $\sigma$ is a point
$(i,\sigma(i))$ such that $\sigma(j)<\sigma(i)$ for every $j<i$. Similar definitions hold for right-to-left maxima, left-to-right minima, and right-to-left minima. These points are also called \emph{records}.
The sets of points $(i,\sigma(i))$ associated with the four types of records are denoted by $\gls*{rec_max}, \LRm(\sigma), \RLM(\sigma),$ and $\RLm(\sigma)$ (or simply $\LRM, \LRm, \RLM,$ and $\RLm$, when the permutation $\sigma$ considered is clear from the context).  These sets of points are not necessarily disjoint.  For any permutation $\sigma$, $(1,\sigma(1))$ is always in $\LRM\cap \LRm$.  Similar statements are true for $(n,\sigma(n))$, $(\sigma^{-1}(1),1)$ and $(\sigma^{-1}(n),n)$.  
The permutation may also have points in $\LRM \cap \RLm$ or $\LRm \cap \RLM$.  By simple counting arguments, the points in $\LRM\cap \RLm$ satisfy $i = \sigma(i)$ and the points in $\LRm \cap \RLM$ satisfy $i = n+1-\sigma(i).$

\subsection{Probabilistic notation}
Throughout the manuscript we denote random quantities using \textbf{bold} characters.
Given a probability measure $\mu$ we denote by $\E_{\mu}$ the expectation with respect to $\mu$ and given a random variable $\bm{X},$ we denote by $\mathcal{L}aw(\bm{X})$ its law. Given a sequence of random variables $(\bm{X}_n)_{n\in\Z_{>0}}$, we write $\bm{X}_n\xrightarrow{d}\bm{X}$ to denote  convergence in distribution, $\bm{X}_n\stackrel{P}{\rightarrow}\bm{X}$ to denote  convergence in probability, and $\bm{X}_n\xrightarrow{a.s.}\bm{X}$ to denote almost sure convergence. Given an event $A$ we denote by $A^c$ the complement event of $A.$

    \chapter[The theories of permutons \& local limits]{The theories of permutons\\
	 \& local limits}\label{chp:conv_theories}
\chaptermark{The theories of permutons \& local limits}

\begin{adjustwidth}{8em}{0pt}
	\emph{In which we introduce permuton and local convergence for permutations. For permuton convergence we follow  \cite{hoppen2013limits,bassino2017universal}. Local convergence has been introduced by the author of this manuscript in \cite{borga2020localperm}. This chapter has to be thought of as preparatory for the consecutive chapters. We skip most of the proofs, and we furnish for each result precise references.}
	
	\emph{In the end, we also discuss a relation between the notions of permuton/local limits and some feasible regions for patterns/consecutive patterns. In the case of consecutive patterns, we explore some connections with cycle polytopes.}
\end{adjustwidth}

\bigskip

\section{Permuton convergence}\label{sect:perm_conv}

We first recall in \cref{sect:determ_perm} the definition of (deterministic) permuton and some basic properties.
Then in \cref{sect:random_perm} we present criteria for determining convergence in distribution of random permutons.

\subsection{Deterministic permutons}\label{sect:determ_perm}

\begin{defn}
	A \emph{permuton} $\nu$ is a Borel probability measure on the unit square $[0,1]^2$ with uniform marginals, i.e.\ 
	$\nu( [0,1] \times [a,b] ) = \nu( [a,b] \times [0,1] ) = b-a$
	for all $0 \le a \le b\le 1$. 
\end{defn}

\begin{rem}
	Permutons were first considered by Presutti and
	Stromquist in \cite{MR2732835} under the name of \emph{normalized measures}. The theory of deterministic permutons (without using this terminology) was developed by Hoppen, Kohayakawa,
	Moreira, Rath and Sampaio~\cite{hoppen2013limits}. The terminology \emph{permuton}
	was given afterwards by analogy with
	\emph{graphons} by Glebov, Grzesik, Klimošová and Král~\cite{MR3279390}. Results of \cite{hoppen2013limits} were recently extended to random permutons by Bassino, Bouvel, F\'{e}ray, Gerin, Maazoun, and Pierrot~\cite{bassino2017universal}.
\end{rem}

Any permutation $\sigma$ of size $n \ge 1$ can be interpreted as the permuton $\mu_\sigma$ given by the sum of Lebesgue area measures
\begin{equation}
	\label{eq:defn_perm_perm}
	\mu_\sigma(A)= n \sum_{i=1}^n \Leb\big([(i-1)/n, i/n]\times[(\sigma(i)-1)/n,\sigma(i)/n]\cap A\big),
\end{equation}
for every Borel measurable set $A$ of $[0,1]^2$. Note that $\mu_\sigma$ is essentially the (normalized) diagram of $\sigma$,
where each dot has been replaced with a square of dimension $1/n \times 1/n$ carrying a mass $1/n$. Recall also the example given in \cref{fig:permtoperm} page~\pageref{fig:permtoperm}.

Let $\gls*{space_perm}$ be the set of permutons.
We endow $\mathcal M$ with the weak topology.
Then a sequence of permutons $(\mu_n)_{n\in\Z_{>0}}$ converges \emph{weakly} to $\mu$, simply denoted $\mu_n \to \mu$, if 
$$
\int_{[0,1]^2} f d\mu_n \rightarrow \int_{[0,1]^2} f d\mu,
$$
for every (bounded and) continuous function $f: [0,1]^2 \to \mathbb{R}$.
With this topology, $\mathcal M$ is compact and metrizable by a metric $\gls*{distance_perm}$  defined, for every pair of permutons $(\mu,\mu'),$ by
\begin{equation}\label{eq:perm_dist}
	d_{\square}(\mu,\mu')=\sup_{R\in\mathcal{R}}|\mu(R)-\mu'(R)|,
\end{equation}
where $\mathcal R$ denotes the set of (horizontal) rectangles contained in $[0,1]^2,$ i.e.\ sets of the form $(a,b)\times (c,d)$ with $a,b,c,d\in[0,1]$, $a<b$ and $c<d$.
In other terms (see Lemmas 2.5 and 5.3 in \cite{hoppen2013limits}):
$$
\mu_{n}\rightarrow  \mu \qquad  \Leftrightarrow \qquad d_\square(\mu_n,\mu)\rightarrow  0.
$$

Note that for $\sigma \in \S_n$ and $\pi \in \S_k$, $k\leq n$, we have
\begin{equation}\label{eq:weivfwebvfeuwobfoipe}
\pocc(\pi,\sigma)=\mathbb{P}\left(\pat_{{\bm I}_{n,k}}(\sigma)=\pi \right),
\end{equation}
where ${\bm I}_{n,k}$ is chosen uniformly at random among the $\binom{n}{k}$ subsets of $[n]$ with $k$ elements.

We define the {\em pattern density} $\pocc(\pi,\mu)$ of a pattern $\pi \in \S_k$ in a permuton $\mu$
by analogy with the formula above. We first define the permutation induced by $k$ points in the square $[0,1]^2$. Take a sequence of $k$ points $(X,Y)=((x_1,y_1),\dots, (x_k,y_k))$ in $[0,1]^2$ in general position, i.e.\ with distinct $x$ and $y$ coordinates. 
We denote by $\left((x_{(1)},y_{(1)}),\dots, (x_{(k)},y_{(k)})\right)$ the \emph{$x$-reordering} of $(X,Y)$,
i.e.\ the unique reordering of the sequence $((x_1,y_1),\dots, (x_k,y_k))$ such that
$x_{(1)}<\cdots<x_{(k)}$.
Then the values $(y_{(1)},\ldots,y_{(k)})$ are in the same
relative order as the values of a unique permutation of size $k$, that we call the \emph{permutation induced by} $(X,Y)$. An example is given in \cref{fig:schema_perm_ind}.

\begin{figure}[htbp]
	\centering
	\includegraphics[scale=0.83]{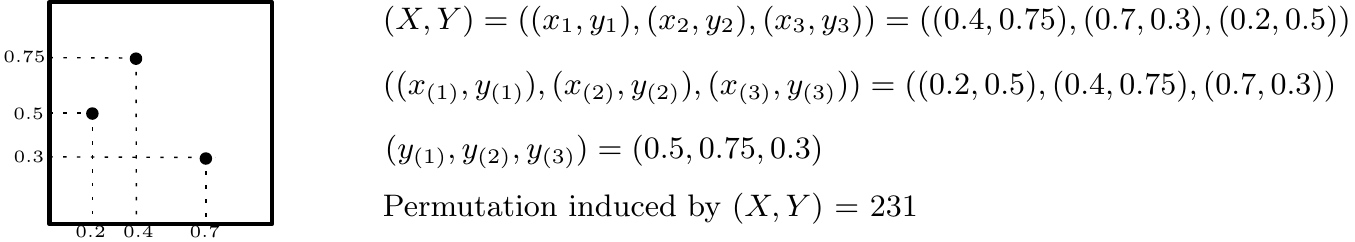}\\
	\caption{An example of permutation induced by 3 points in the square $[0,1]^2$. \label{fig:schema_perm_ind}}
\end{figure}

Let $\mu$ be a permuton and $((\bm X_i,\bm Y_i))_{i\in \Z_{>0}}$ be an i.i.d.\ sequence with distribution $ \mu$. We denote by $\gls*{induce_perm_k}$ the random permutation induced by $((\bm X_i,\bm Y_i))_{i \in [k]}$. 
Then we set for all  $\pi \in \S_k$,
\begin{equation}\label{eq:def_perm_den}
\pocc(\pi,\mu) \coloneqq \P( \Perm_k(\mu)=\pi ).
\end{equation}

The following concentration result shows that $\mu$ is close to $\Perm_k(\mu)$ in probability for $k$ large.

\begin{lem}[{\cite[Lemma 4.2]{hoppen2013limits}}]
	\label{lem:subpermapproxpermuton}
	There exists $K$ such that if $k>K$, for any permuton $\mu$,
	\[\P\left(
	d_{\square}(\mu_{\Perm_k(\mu)},\mu)
	\geq 16k^{-1/4}\right)
	\leq \frac12 e^{-\sqrt{k}}.
	\]
\end{lem}

\subsection{Random permutons and permuton convergence}\label{sect:random_perm}

We now consider random permutons $\bm \mu$. Note that the definition in \cref{eq:def_perm_den} naturally extend to random permutons in the following way:
\begin{equation}\label{eq:bfeweubfbfowe}
\pocc(\pi,\bm \mu) \coloneqq \P( \Perm_k(\bm \mu)=\pi |\bm \mu).
\end{equation}
We have the following version of  \cref{lem:subpermapproxpermuton} for random permutons that will be  useful in the following chapters (for instance, in the proof of \cref{thm:permuton} page \pageref{thm:permuton}).

\begin{lem}[{\cite[Lemma 2.3]{bassino2017universal}}]
	\label{lem:subpermapproxrandompermuton}
	There exists $K$ such that if $k>K$, for any random permuton $\bm \mu$,
	\[\P
	\left(
	d_{\square}(\mu_{\Perm_k(\bm \mu)},\bm \mu)
	\geq 16k^{-1/4}\right)
	\leq \frac12 e^{-\sqrt{k}}.
	\]
\end{lem}

\medskip

This result has an important consequence for the distribution of random permutons. 
\begin{prop}[{\cite[Proposition 2.4]{bassino2017universal}}]\label{Prop:CaracterisationLoiPermuton}
	Let $\bm \mu$, $\bm {\mu}'$ be two random permutons. If there exists $k_0$ such that $\P(\Perm_k(\bm \mu)=\pi)=\P(\Perm_k(\bm \mu')=\pi)$ for all $k\geq k_0$ and all $\pi\in\mathcal{S}_k$, 
	then $\bm \mu \stackrel{d}{=}\bm \mu'$.
\end{prop}

We now consider a sequence of random permutations $(\bm \sigma_n)_{n\in\Z_{>0}}$, with $\bm \sigma_n$ of size $n$.
Taking ${\bm I}_{n,k}$ independently from $\bm \sigma_n$, we have for every pattern $\pi$ of size $k$,
\begin{equation}\label{eq:E(occ)=P(pat)}
\E[\pocc(\pi,\bm \sigma_n)] 
\stackrel{\eqref{eq:weivfwebvfeuwobfoipe}}{=} \E\left[\mathbb{P}\left(\pat_{{\bm I}_{n,k}}(\bm \sigma_n)=\pi \middle| \bm{\sigma_n}\right)\right]
=\P\left(\pat_{{\bm I}_{n,k}}(\bm \sigma_n) = \pi\right).
\end{equation}
Similarly, for a random permuton $\bm{\mu}$, we have 
\begin{equation}
\E[\pocc(\pi,\bm \mu)]
\stackrel{\eqref{eq:bfeweubfbfowe}}{=}\P\left(\Perm_k(\bm \mu)=\pi\right).
\label{eq:E(occ)=P(perm)}
\end{equation}

The main theorem of this section (see \cref{thm:randompermutonthm}) is a consequence of the relations in the two equations above. It deals with convergence of a sequence of random permutations to a random permuton. 
It generalizes the result of \cite{hoppen2013limits} which states that \emph{deterministic} permuton convergence is characterized by convergence of pattern densities. 
More precisely, it shows that for a sequences of \emph{random} permutations, permuton convergence \emph{in distribution} is characterized by convergence in \emph{expectation} of pattern densities (see also the consecutive \cref{rem:exp_in_enough}),
or equivalently of the induced subpermutations of any (fixed) size.

\begin{thm}[{\cite[Theorem 2.5]{bassino2017universal}}]\label{thm:randompermutonthm}
	For any $n\in\Z_{>0}$, let $\bm\sigma_n$ be a random permutation of size $n$. 
	Moreover, for any fixed $k\in\Z_{>0}$, let ${\bm I}_{n,k}$ be a uniform random subset of $[n]$ with $k$ elements, independent of $\bm \sigma_n$.
	The following assertions are equivalent:
	\begin{enumerate}%
		\item [(a)] $(\mu_{\bm \sigma_n})_{n\in\Z_{>0}}$ converges in distribution w.r.t.\ the permuton topology to some random permuton $\bm \mu$. 
		\item [(b)] The random infinite vector $\big(\pocc(\pi,\bm\sigma_n)\big)_{\pi \in \S}$ converges in distribution w.r.t.\ the product topology to some random infinite vector $(\bm \Lambda_\pi)_{\pi \in \S}$. 
		\item [(c)] For every $\pi$ in $\S$, there exists $\Delta_\pi \geq 0$ such that 
		$\E[\pocc(\pi,\bm \sigma_n)] \xrightarrow{n\to\infty} \Delta_\pi.$
		\item [(d)] For every $k\in\Z_{>0}$, the sequence  $\big(\pat_{{\bm I}_{n,k}}(\bm\sigma_n)\big)_{n\in\Z_{>0}}$ of random permutations converges in distribution to some random permutation $\bm \rho_k$.
	\end{enumerate}
	Whenever these assertions are verified, we have $(\bm \Lambda_\pi)_{\pi\in\S} \stackrel d = (\pocc(\pi,\bm \mu))_{\pi\in\S}$, and for every $k\in\Z_{>0}$ and $\pi\in\S_k$,
	\[ \P(\bm \rho_{k} = \pi) = \Delta_\pi = \E[\bm \Lambda_\pi] = \E[\pocc(\pi,\bm{\mu})] = \P(\Perm_k(\bm\mu) = \pi). \] 

\end{thm}

\begin{rem}\label{rem:exp_in_enough}
	The fact that the convergence of $\E[\pocc(\pi,\bm \sigma_n)]$ for all $\pi\in\S$ is enough to prove permuton convergence of the sequence $(\mu_{\bm \sigma_n})_{n\in\Z_{>0}}$ might be surprising at a first read. 
	
	There are two possible explanations of why this holds. The first explanation, more probabilistic, is in analogy with the convergence of random functions. The latter is usually proven by establishing tightness and convergence of finite-dimensional marginals. Since $\mathcal{M}$ is compact, Prokhorov's theorem ensures that the space of probability distributions on $\mathcal M$ is compact (for convergences of measure, we refer to \cite{billingsley2013convergence}) therefore tightness of the sequence $(\mu_{\bm \sigma_n})_{n\in\Z_{>0}}$ is always guaranteed. Then the convergence of $\E[\pocc(\pi,\bm \sigma_n)]$ for all $\pi\in\S$ (or equivalently of $\pat_{{\bm I}_{n,k}}(\bm\sigma_n)$ for all $k\in\Z_{>0}$)  corresponds to the convergence in distribution of "finite-dimensional marginals".
	
	A second explanation, more algebraic, is the following: fix $k$ patterns $\pi_1,\dots,\pi_k$. Then there exists constants $(C_\rho)_{\rho\in\S}$ such that
	\begin{equation}
	\prod_{i=1}^k\occ(\pi_i, \sigma)=\sum_{\rho\in\S}C_\rho\occ(\rho,\sigma),\quad\text{for all}\quad \sigma\in\S.
	\end{equation}
	For a proof of this result see for instance \cite[Theorem 1.4]{penaguiao2020pattern}. The relation above implies that joint moments of $\pocc(\pi_i, \bm \sigma_n)$ are linear combinations of $(\E[\pocc(\rho,\bm \sigma_n)])_{\rho\in\S}$. Hence, if the expectations $\E[\pocc(\rho,\bm \sigma_n)]$ converge for all $\rho\in\S$, then the joint moments converge and we can recover convergence in distribution for $(\pocc(\pi_i,\bm \sigma_n))_{i\in[k]}$.
\end{rem}

\begin{rem}
	We point out that a characterization similar to the one in \cref{thm:randompermutonthm} was also discovered in the setting of graphons by Diaconis and Janson \cite[Theorem 3.1]{MR2463439}.
\end{rem}

We end this section with a discussion on the existence of random permutons
with prescribed induced subpermutations. 
\begin{defn}
	A family of random permutations $(\bm \rho_n)_{n\in\Z_{>0}}$ is \emph{consistent} if
	\begin{itemize}
		\item for every $n\geq 1$, $\bm \rho_n \in \S_n$,
		\item for every $n\geq k \geq 1$, if $\bm I_{n,k}$ is a uniform subset of $[n]$ of size $k$, independent of $\bm \rho_n$, then $\pat_{{\bm I}_{n,k}}(\bm \rho_n) \stackrel d = \bm \rho_k$.
	\end{itemize} 
	\label{Def:consistency}
\end{defn}
It turns out that consistent families of random permutations and random permutons are essentially equivalent.
\begin{prop}[{\cite[Proposition 2.9]{bassino2017universal}}]
	If $\bm{\mu}$ is a random permuton, then the family defined by $\bm \rho_k \stackrel d = \Perm_k(\bm{\mu})$ is consistent. Conversely, for every consistent family of random permutations $(\bm \rho_k)_{k\in\Z_{>0}}$, there exists a random permuton $\bm \mu$ whose distribution is uniquely determined, such that $\Perm_k(\bm{\mu}) \stackrel d = \bm\rho_k$. In that case, $\mu_{\bm \rho_n} \xrightarrow[n\to\infty]{d} \bm \mu$.
	\label{Prop:existence_permuton}
\end{prop}

\section{Local convergence}\label{sect:local_theory}

This section is dedicated to the theory of local convergence for permutations.

\subsection{The set of rooted permutations}
In the context of local convergence, we need to look at permutations with a distinguished entry, called the \emph{root}. In the first part of this section we formally introduce the notion of finite and infinite rooted permutation. Then we introduce a local distance and at the end we show that the space of (possibly infinite) rooted permutations is the natural space to study local limits of permutations. More precisely, we show that this space is compact and contains the space of finite rooted permutations as a dense subset.

For the reader convenience, we recall the following fundamental definition.
\begin{defn}
	A \emph{finite rooted permutation} is a pair $(\sigma,i),$ where $\sigma\in\mathcal{S}$ and $i\in[|\sigma|]$.
\end{defn}
We extend the notion of size of a permutation to rooted permutations in the natural way, i.e.\ $|(\sigma,i)|\coloneqq|\sigma|.$
We denote  with $\gls*{root_perm_n}$ the set of rooted permutations of size $n$ and with $\gls*{root_perm}\coloneqq\bigcup_{n\in\Z_{>0}}\mathcal{S}^{\bullet}_n$ the set of finite rooted permutations. We write sequences of finite rooted permutations in $\mathcal{S}^{\bullet}$ as $(\sigma_n,i_n)_{n\in\Z_{>0}}.$ 

To a rooted permutation $(\sigma,i),$ we associate the pair $(\Asi,\leqsi),$  where $\Asi\coloneqq[-i+1,|\sigma|-i]$ is a finite interval containing 0 and $\leqsi$ is a total order on $\Asi,$ defined for all $\ell,j\in \Asi$ by
\begin{equation}
\ell\leqsi j\qquad\text{if and only if}\qquad \sigma(i+\ell)\leq\sigma(i+j)\;.
\end{equation}  
Clearly this map is a bijection from the space of finite rooted permutations $\mathcal{S}^{\bullet}$ to the space of total orders on finite integer intervals containing zero. Consequently, we may identify every rooted permutation  $(\sigma,i)$ with the total order $(\Asi,\leqsi).$

We highlight that the total order $(\Asi,\leqsi)$ associated with a rooted permutation indicates how the elements of $\sigma$ at given positions compare to each other (recall for instance the example given in \cref{rest} page \pageref{rest}).

Thanks to the identification between rooted permutations and total orders, the following definition of an infinite rooted permutation is natural. 

\begin{defn}
	We call \textit{infinite rooted permutation} a pair $(A,\preccurlyeq)$ where $A$ is an infinite\footnote{Here, an infinite interval can be infinite just on one side.} interval of integers containing 0 and $\preccurlyeq$ is a total order on $A$. We denote the set of infinite rooted permutations by $\gls*{root_perm_inf}$.
\end{defn}

We highlight that infinite rooted permutations can be thought of as rooted at 0. We set 
$$\gls*{root_perm_fin_inf}\coloneqq\Sr\cup\mathcal{S}_\infty^{\bullet},$$
namely, the set of finite and infinite rooted permutations and we denote by $\Sr_{\leq n}\coloneqq\bigcup_{m\leq n}\Sr_m$ the set of rooted permutations with size at most $n.$ We write sequences of finite or infinite rooted permutations in $\Sri$ as $(A_n,\preccurlyeq_n)_{n\in\Z_{>0}}.$

We now introduce the following \textit{restriction function around the root} defined, for every $h\in\Z_{>0}$, as follow
\begin{align}
\label{rhfunct}
r_h \colon\quad &\tilde{\mathcal{S}}^{\bullet}\;\longrightarrow \qquad \;\mathcal{S}^\bullet\\
(A,&\preccurlyeq) \mapsto \big(A\cap[-h,h],\preccurlyeq\big)\;.
\end{align}
We can think of restriction functions as a notion of neighborhood around the root.
For finite rooted permutations we also have the equivalent description of the restriction functions $r_h$ in terms of consecutive patterns: if $(\sigma,i)\in\Sr$ then $r_h(\sigma,i)=(\text{pat}_{[a,b]}(\sigma),i-a+1)$ where $a=\max\{1,i-h\}$ and $b=\min\{|\sigma|,i+h\}.$
In the next example we give a graphical interpretation of the restriction function around the root.
\begin{exmp} We look at the rooted permutation $(\sigma,i)=(752934861,4).$ When we consider the restriction $r_h(\sigma,i),$ we draw in gray a vertical strip "around" the root of width $2h+1$ (or less if we are near the boundary of the diagram). An example is provided in \cref{examplerootedperm} where $r_2(\sigma,i)$ is computed. In particular $r_2(\sigma,i)=\big([-2,2],\preccurlyeq\big),$ with $-1\preccurlyeq1\preccurlyeq2\preccurlyeq-2\preccurlyeq0.$
\end{exmp}

\begin{figure}[htbp]
	\begin{minipage}[c]{0.55\textwidth}
		\begin{equation}
			\begin{array}{lcr}
				\begin{tikzpicture}
					\begin{scope}[scale=.3]
						\permutation{7,5,2,9,3,4,8,6,1}
						\fill[gray!60!white, opacity=0.4] (2,1) rectangle (7,10); 
						\draw (4+.5,9+.5) [UMRed, fill] circle (.3); 
						\draw[thick] (1,1) rectangle(10,10);
					\end{scope}
				\end{tikzpicture}
			\end{array}
			\stackrel{r_2}{\longrightarrow}
			\begin{array}{lcr}
				\begin{tikzpicture}
					\begin{scope}[scale=.3]
						\permutation{4,1,5,2,3}
						\draw (3+.5,5+.5) [UMRed, fill] circle (.3); 
						\draw[thick] (1,1) rectangle(6,6);
					\end{scope}
				\end{tikzpicture}
			\end{array}
		\end{equation}
	\end{minipage}
	\begin{minipage}[c]{0.44\textwidth}
		\caption{Diagram of the permutation $\sigma=752934861$ rooted at $i=4$ with the corresponding restriction $r_2(\sigma,i).$\label{examplerootedperm}}
	\end{minipage}
\end{figure}
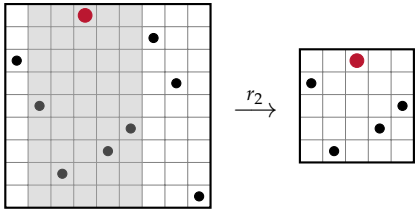 	

\begin{defn}
	We say that a family $(A_h,\preccurlyeq_h)_{h\in\Z_{>0}}$ of elements in $\Sr$ is \emph{consistent} if
	\begin{equation}
	\label{consprop}
	r_{h}(A_{h+1},\preccurlyeq_{h+1})=(A_h,\preccurlyeq_h),\quad \text{ for all }\quad h\in\Z_{>0}.
	\end{equation}
\end{defn}

\begin{obs}
	\label{consistentobs}
	For every (possibly infinite) rooted permutation $(A,\preccurlyeq),$ the corresponding family of restrictions $\big(r_h(A,\preccurlyeq)\big)_{h\in\Z_{>0}}$ is consistent.
\end{obs}

In the next example, we exhibit a concrete infinite rooted permutation, with some of its restrictions and the associated rooted permutations. 
\begin{exmp}
	We consider the following total order on $\mathbb{Z}:$
	\begin{equation}
	\label{extotord}
	-2\preccurlyeq-4\preccurlyeq-6\preccurlyeq\dots\preccurlyeq-1\preccurlyeq-3\preccurlyeq-5
	\preccurlyeq\dots\preccurlyeq0\preccurlyeq1\preccurlyeq2\preccurlyeq3\preccurlyeq\dots,
	\end{equation}
	that is, the standard order on the integers except that the order on negative numbers is reversed and negative even numbers are smaller than negative odd numbers. 
	Using the definition given in \cref{rhfunct}, we have for $h=4,$
	\begin{equation}
	r_4(\mathbb{Z},\preccurlyeq)=\big([-4,4],\preccurlyeq\big),\quad\text{with}\quad -2\preccurlyeq-4\preccurlyeq-1\preccurlyeq-3\preccurlyeq0\preccurlyeq1\preccurlyeq2\preccurlyeq3\preccurlyeq4,
	\end{equation}
	which represents the rooted permutation $(\pi^4,5)=(241356789,5)=\begin{array}{lcr}
		\begin{tikzpicture}
			\begin{scope}[scale=.2]
				\permutation{2,4,1,3,5,6,7,8,9}
				\draw (5+.5,5+.5) [UMRed, fill] circle (.3); 
				\draw[thick] (1,1) rectangle(10,10);
			\end{scope}
		\end{tikzpicture}
	\end{array}.$
	In particular, $r_4(\mathbb{Z},\preccurlyeq)=(A_{\pi^4,5},\preccurlyeq_{\pi^4,5}).$ 
	Moreover, for $h=3,$
	\begin{equation}
	r_3(\mathbb{Z},\preccurlyeq)=\big([-3,3],\preccurlyeq\big),\quad\text{with}\quad-2\preccurlyeq-1\preccurlyeq-3\preccurlyeq0\preccurlyeq1\preccurlyeq2\preccurlyeq3,
	\end{equation} 
	which represents the rooted permutation $(\pi^3,4)=(3124567,4)=\begin{array}{lcr}
		\begin{tikzpicture}
			\begin{scope}[scale=.2]
				\permutation{3,1,2,4,5,6,7}
				\draw (4+.5,4+.5) [UMRed, fill] circle (.3); 
				\draw[thick] (1,1) rectangle(8,8);
			\end{scope}
		\end{tikzpicture}
	\end{array}.$
	In particular, $r_3(\mathbb{Z},\preccurlyeq)=(A_{\pi^3,4},\preccurlyeq_{\pi^3,4}).$ We also note that $r_3(\pi^4,5)=(\pi^3,4)$.
\end{exmp}

\subsection{Local convergence for rooted permutations}
We are now ready to define a notion of \emph{local distance} on the set of (possibly infinite) rooted permutations $\tilde{\mathcal{S}}^{\bullet}$. Given two rooted permutations $(A_1,\preccurlyeq_1),(A_2,\preccurlyeq_2)\in\tilde{\mathcal{S}}^{\bullet},$ we define
\begin{equation}
\label{distance}
\gls*{distance_perm_local}\big((A_1,\preccurlyeq_1),(A_2,\preccurlyeq_2)\big)\coloneqq 2^{-\sup\big\{h\in\Z_{>0}\;:\;r_h(A_1,\preccurlyeq_1)=r_h(A_2,\preccurlyeq_2)\big\}},
\end{equation}
with the conventions that $\sup\emptyset=0,$ $\sup\Z_{>0}=+\infty$ and $2^{-\infty}=0.$ It is a basic exercise to check that $d$ is a distance. Actually, $d$ is ultra metric, and so all the open balls of radius $2^{-h},$ $h\in\Z_{>0},$ centered in $(A,\preccurlyeq)\in\Sri$ -- which we denote by $B\big((A,\preccurlyeq),2^{-h}\big)$ -- are closed and intersecting balls are contained in each other. We will say that the balls are \emph{clopen}, i.e.\ closed and open.

Once we have a distance, the following notion of convergence is natural. 
\begin{defn}
	\label{local_conv_def}
	We say that a sequence $(A_n,\preccurlyeq_n)_{n\in\Z_{>0}}$ of rooted permutations in $\Sri$ is \emph{locally convergent} to an element $(A,\preccurlyeq)\in\Sri,$ if it converges with respect to the local distance $d.$ In this case we write $(A_n,\preccurlyeq_n)\stackrel{loc}{\longrightarrow}(A,\preccurlyeq).$
\end{defn}

We now explain why the space $\tilde{\mathcal{S}}^{\bullet}$ is the right space to consider in order to study the limits of finite rooted permutations with respect to the local distance $d$ defined in \cref{distance}. More precisely, we will see that $\tilde{\mathcal{S}}^{\bullet}$ is compact and that $\Sr$ is dense in $\tilde{\mathcal{S}}^{\bullet}.$ 
The latter assertion is trivial since 
\begin{equation}
d\big((A,\preccurlyeq),r_h(A,\preccurlyeq)\big)\leq 2^{-h},\quad\text{for all}\quad (A,\preccurlyeq)\in \tilde{\mathcal{S}}^{\bullet},\text{ all } h\in\Z_{>0},
\label{keydistprop}
\end{equation}
and obviously $r_h(A,\preccurlyeq)\in\Sr.$
Before proving compactness we explore some basic but important properties of our local distance $d$. 
The next proposition gives a converse to \cref{consistentobs}.

\begin{prop}[{\cite[Proposition 2.12]{borga2020localperm}}]
	\label{consistprop}
	Given a consistent family $(A_h,\preccurlyeq_h)_{h\in\Z_{>0}}$ of elements in $\Sr$ there exists a unique (possibly infinite) rooted permutation $(A,\preccurlyeq)\in\Sri$ such that
	\begin{equation}
	r_{h}(A,\preccurlyeq)=(A_h,\preccurlyeq_h),\quad \text{ for all }\quad h\in\Z_{>0}.
	\end{equation}
	Moreover, $(A_h,\preccurlyeq_h)\stackrel{loc}{\longrightarrow}(A,\preccurlyeq).$
\end{prop}

We have the following consequence.

\begin{prop}[{\cite[Proposition 2.14]{borga2020localperm}}]
	\label{contrh}
	Given a sequence $(A_n,\preccurlyeq_n)_{n\in\Z_{>0}}$ of rooted permutations in $\Sri,$ the following assertions are equivalent:
	\begin{enumerate}[(a)]
		\item There exists $(A,\preccurlyeq)\in\Sri$ such that $(A_n,\preccurlyeq_n)\stackrel{loc}{\longrightarrow}(A,\preccurlyeq)$.
		\item There exists a family $(B_h,\ll_h)_{h\in\Z_{>0}}$ of finite rooted permutations such that
		\begin{equation}
		r_h(A_n,\preccurlyeq_n)\stackrel{loc}{\longrightarrow}(B_h,\ll_h),\quad\text{for all}\quad h\in\Z_{>0}.
		\end{equation}  
	\end{enumerate}
	In particular if one of the two conditions holds (and so both), then $(B_h,\ll_h)=r_h(A,\preccurlyeq),$ for all $h\in\Z_{>0}.$
\end{prop}

\begin{obs}
	\label{continuityrh}
	Note that for all $h\in\Z_{>0},$ the restriction functions $r_h$ are continuous. This is a simple consequence of implication $(a)\Rightarrow(b)$ in the previous proposition.
\end{obs}

Using the results above, we can prove the following.
\begin{thm}[{\cite[Theorem 2.16]{borga2020localperm}}]
	\label{compactpolish}
	The metric space $(\tilde{\mathcal{S}}^{\bullet},d)$ is compact.
\end{thm}

\subsection{Benjamini--Schramm convergence: the deterministic case}

In this section we want to define a notion of weak-local convergence for a deterministic sequence of finite (unrooted) permutations $(\sigma_n)_{n\in\Z_{>0}}.$ The case of a random sequence $(\bm{\sigma}_n)_{n\in\Z_{>0}}$ will be discussed in \cref{random case}. 
Some of the results contained in this section are just special cases of the those in \cref{wbsconv} about the annealed version of the Benjamini--Schramm convergence (B--S convergence, for short). However we choose to present the deterministic case separately with a double purpose: first, in this way, the definitions are clearer, secondly, we believe that the subsequent quenched version of the B--S convergence (see \cref{sbsconv}) will be more intuitive.

A possible way to define this weak-local convergence is to use the notion of local distance defined in \cref{distance}. Therefore, we need to construct a sequence of rooted permutations from the sequence $(\sigma_n)_{n\in\Z_{>0}}.$ A natural way to see a fixed permutation $\sigma$ as an element of $\mathcal{S}^\bullet$ is to choose a root uniformly at random among the indices of $\sigma.$

\begin{defn}
	\label{weakconv}
	Let $(\sigma_n)_{n\in\Z_{>0}}$ be a sequence in $\mathcal{S}$, $\bm{i}_n$ a uniform random index in $[|\sigma_n|]$, and $(\bm{A},\bm{\preccurlyeq})$ a random (possibly infinite) rooted permutation. We say that the sequence $(\sigma_n)_{n\in\Z_{>0}}$ \emph{Benjamini--Schramm converges} to $(\bm{A},\bm{\preccurlyeq})$, if the sequence $(\sigma_n,\bm{i}_n)_{n\in\Z_{>0}}$ converges in distribution to $(\bm{A},\bm{\preccurlyeq})$ with respect to the local distance $d$ defined in \cref{distance}.
	In this case we simply write $\sigma_n\xrightarrow{BS} (\bm{A},\bm{\preccurlyeq})$ instead of $(\sigma_n,\bm{i}_n)\xrightarrow{d}(\bm{A},\bm{\preccurlyeq}).$
\end{defn}

Sometimes, in order to simplify notation, we will denote the B--S limit of a sequence $(\sigma_n)_{n\in\Z_{>0}}$ simply by $\bm{\sigma}^{\infty}$ instead of $(\bm{A},\bm{\preccurlyeq}).$

Our main theorem in this section deals with the following interesting relation between the B--S convergence and the convergence of consecutive pattern densities.

\begin{thm}
	\label{bscharact}
	For any $n\in\Z_{>0},$ let $\sigma_n$ be a permutation of size $n.$ Then the following assertions are equivalent:
	\begin{enumerate}[(a)]
		\item There exists a random rooted infinite permutation $\bm{\sigma}^\infty$ such that $\sigma_n\xrightarrow{BS}\bm{\sigma}^\infty$.
		\item There exist non-negative real numbers $(\Delta_{\pi})_{\pi\in\mathcal{S}}$ such that $\widetilde{\cocc}(\pi,\sigma_n)\to\Delta_{\pi}$ for all $\pi\in\mathcal{S}$.
	\end{enumerate}
	Moreover, if one of the two conditions holds (and so both) we have the following relation between the limiting objects:
	\begin{equation}
	\mathbb{P}\big(r_h(\bm{\sigma}^\infty)=(\pi,h+1)\big)=\Delta_{\pi},\quad\text{for all}\quad h\in\Z_{>0},\quad\text{ all}\quad\pi\in\mathcal{S}_{2h+1}.
	\end{equation}
\end{thm}

We skip the proof of this theorem since it is a particular case of \cref{weakbsequivalence}. 

\begin{rem}
	We could have introduced local convergence for permutations also from another point of view, closer to the one developed for graphons (see \cite{lovasz2012large}): we could say that a sequence of permutations is locally convergent if for every pattern $\pi\in\mathcal{S}$ the sequences $\pcocc(\pi,\sigma_n)$ converge, and then characterize the limiting objects. \cref{bscharact} shows that the two approaches are equivalent.  
\end{rem}

\begin{rem}
	Recently, Pinsky \cite{pinsky2018infinite} studied limits of random permutations avoiding patterns of size three considering a different topology. His topology captures the local limit of the permutation diagram around a {\em corner}. The two topologies (that from \cite{pinsky2018infinite} and that from the present manuscript) are not comparable. We point out that there are two main advantages in our definition of local limit: first, convergence in our local topology is pleasantly equivalent to the convergence of consecutive pattern proportions; second, many natural models have an interesting limiting object for our topology (while converging to $+\infty$ in the topology of \cite{pinsky2018infinite}).
\end{rem}

\subsection{Benjamini--Schramm convergence: the random case}
\label{random case}
The goal of this section is to characterize the convergence in distribution of a sequence of \emph{random} permutations $(\bm{\sigma}_n)_{n\in\Z_{>0}}$ with respect to the local distance defined in \cref{distance}. We have two different natural choices for this definition, one stronger than the other (and not equivalent as shown in \cref{notequivalent}). The weaker definition, as already mentioned, is an analogue of the notion of B--S convergence for random graphs developed  in \cite{benjamini2001recurrence}.

\subsubsection{Annealed version of the Benjamini--Schramm convergence}
\label{wbsconv}
\begin{defn}[Annealed version of the B--S convergence]\label{weakweakconv}
	Given a sequence $(\bm{\sigma}_n)_{n\in\Z_{>0}}$ of random permutations in $\mathcal{S},$ let $\bm{i}_n$ be a uniform index of $\bm{\sigma}_n$ conditionally on $\bm{\sigma}_n$. We say that $(\bm{\sigma}_n)_{n\in\Z_{>0}}$ \emph{converges in the annealed Benjamini--Schramm sense} to a random variable $\bm{\sigma}^{\infty}$ with values in $\Sri$ if the sequence of random variables $(\bm{\sigma}_n,\bm{i}_n)_{n\in\Z_{>0}}$ converges in distribution to $\bm{\sigma}^{\infty}$ with respect to the local distance. In this case we write $\bm{\sigma}_n\gls*{annealed_BS}\bm{\sigma}^\infty$ instead of  $(\bm{\sigma}_n,\bm{i}_n)\xrightarrow{d}\bm{\sigma}^\infty.$
\end{defn}

Like in the deterministic case, the limiting object $\bm{\sigma}^{\infty}$ is a random rooted permutation.

\begin{rem}
	Even though the above definition is stated for sequences of random permutations of arbitrary size, we will often consider the case when, for all $n\in\Z_{>0},$ $|\bm{\sigma}_n|=n$ almost surely and the random choice of the root $\bm{i}_n$ is uniform in $[n]$ and independent of $\bm{\sigma}_n$.
\end{rem}

We give some characterizations of the  annealed version of the B--S convergence. 
\begin{thm}[{\cite[Theorem 2.24]{borga2020localperm}}]
	\label{weakbsequivalence}
	For any $n\in\Z_{>0},$ let $\bm{\sigma}_n$ be a random permutation of size $n$ and $\bm{i}_n$ be a uniform random index in $[n]$, independent of $\bm{\sigma}_n.$ Then the following assertions are equivalent:
	\begin{enumerate}[(a)]
		\item There exists a random rooted infinite permutation $\bm{\sigma}^\infty$ such that $\bm{\sigma}_n\xrightarrow{aBS}\bm{\sigma}^\infty$.
		\item For all $h\in\Z_{>0},$ there exist non-negative real numbers $(\Gamma^h_{\pi})_{\pi\in\mathcal{S}_{2h+1}}$ such that  $$\P\big(r_h(\bm{\sigma}_n,\bm{i}_n)=(\pi,h+1)\big)\xrightarrow[n\to\infty]{}\Gamma^h_{\pi},\quad\text{for all}\quad\pi\in\mathcal{S}_{2h+1}.$$
		\item There exist non-negative real numbers $(\Delta_{\pi})_{\pi\in\mathcal{S}}$ such that $\E[\widetilde{\cocc}(\pi,\bm{\sigma}_n)]\to\Delta_{\pi}$ for all $\pi\in\mathcal{S}$.
	\end{enumerate}
	Moreover, if one of the three conditions holds (and so all of them), for every fixed $h\in\Z_{>0},$ we have the following relations between the limiting objects,
	\begin{equation}\label{eq:fbweiuewbvfweuib}
	\label{limrel}
	\mathbb{P}\big(r_h(\bm{\sigma}^\infty)=(\pi,h+1)\big)=\Delta_{\pi}=\Gamma^h_{\pi},\quad\text{for all}\quad\pi\in\mathcal{S}_{2h+1}.
	\end{equation}
\end{thm}

Before proving the theorem we point out two important facts.

\begin{rem}\label{uyfvuoe3}
	The theorem ensures the existence of a random rooted infinite permutation $\bm{\sigma}^\infty$ but does not furnish any explicit construction of this object as a random total order on $\mathbb{Z}.$
\end{rem}

\begin{rem}
	\label{uyfvuoe2}
	For every fixed $h\in\Z_{>0},$ condition (b) in \cref{weakbsequivalence} only considers the probabilities $\P\big(r_h(\bm{\sigma}_n,\bm{i}_n)=(\pi,j)\big)$ for $\pi\in\mathcal{S}_{2h+1}$ and $j=h+1.$ We remark that for all the other cases, i.e.\ when $\pi\in\mathcal{S}_{2h+1}$ and $j\neq h+1,$ or $|\pi|<2h+1$ and $j\in[|\pi|]$, it is simple to show that
	$\P\big(r_h(\bm{\sigma}_n,\bm{i}_n)=(\pi,j)\big)\to 0.$
\end{rem}

\begin{proof}[Proof of \cref{weakbsequivalence}]
	$(a)\Rightarrow(b).$ For all $h\in\Z_{>0},$ the convergence in distribution of the sequence $\big(r_h(\bm{\sigma}_n,\bm{i}_n)\big)_{n\in\Z_{>0}},$ follows from the continuity of the functions $r_h$ (\cref{continuityrh}). Then (b) is a trivial consequence of the fact that $r_h(\bm{\sigma}_n,\bm{i}_n)$ takes its values in the finite set $\Sr_{\leq 2h+1}.$

	$(b)\Rightarrow(a).$ Thanks to \cref{compactpolish}, $(\Sri,d)$  is a compact (and so Polish) space. Applying Prokhorov's Theorem, in order to show that $(\bm{\sigma}_n,\bm{i}_n)\xrightarrow{d}\bm{\sigma}^{\infty}$ for some random infinite permutation $\bm{\sigma}^{\infty},$ it is enough to show that for every pair of convergent subsequences  $(\bm{\sigma}_{n_k},\bm{i}_{n_k})_{k\in\Z_{>0}}$ and $(\bm{\sigma}_{n_\ell},\bm{i}_{n_\ell})_{\ell\in\Z_{>0}}$ with limits $\bm{\sigma}^\infty_1$ and $\bm{\sigma}^\infty_2$ respectively, then
	$\bm{\sigma}^\infty_1\stackrel{d}{=}\bm{\sigma}^\infty_2.$

	Since $(\bm{\sigma}_{n_k},\bm{i}_{n_k})_{k\in\Z_{>0}}$ and $(\bm{\sigma}_{n_\ell},\bm{i}_{n_\ell})_{\ell\in\Z_{>0}}$ both satisfy $(b)$, the distributions of $\bm{\sigma}^\infty_1$ and $\bm{\sigma}^\infty_2$ must coincide on the collection of sets
	$$\Big\{B\big((A,\preccurlyeq),2^{-h}\big):h\in\Z_{>0},(A,\preccurlyeq)\in\Sri\Big\}.$$ 
	This is a separating class\footnote{For more explanations see \cite[Observation 2.23]{borga2020localperm}.} for the space $(\Sri,d)$ and so we can conclude that $\bm{\sigma}^\infty_1\stackrel{d}{=}\bm{\sigma}^\infty_2.$

	$(b)\Leftrightarrow(c).$
	We assume that $\pi\in\mathcal{S}_{2h+1}$ for some $h\in\Z_{>0}$. 
	Using the independence between $\bm{\sigma}_n$ and $\bm{i}_n$, we have
	\begin{equation}
	\label{relprobexp}
	\E\Big[\widetilde{\cocc}(\pi,\bm{\sigma}_n)\Big]=\E\Big[\P\left(r_h(\bm{\sigma}_n,\bm{i}_n)=(\pi,h+1)\middle| \bm{\sigma_n}\right)\Big]=\P\big(r_h(\bm{\sigma}_n,\bm{i}_n)=(\pi,h+1)\big).
	\end{equation}
	This proves that $(c)\Rightarrow(b)$. For $(b)\Rightarrow(c)$, we also need to prove that $\E\Big[\widetilde{\cocc}(\pi,\bm{\sigma}_n)\Big]$ converges when $\pi$ has even size. Details are given in the proof of \cite[Theorem 2.24]{borga2020localperm}.
\end{proof}

\subsubsection{Quenched version of the Benjamini--Schramm convergence}
\label{sbsconv}

We start by noting that rooting a permutation $\sigma$ uniformly at random, one naturally identifies a probability measure $\nu_{\sigma}$ on $\mathcal{S}^{\bullet},$ which is 
\begin{equation}
	\label{permprob}
	\text{for all Borel sets }A\subset\mathcal{S}^\bullet,\qquad\nu_{\sigma}(A)\coloneqq\mathbb{P}\big((\sigma,\bm{i})\in A\big),
\end{equation}
where $\bm{i}$ is a uniform index of $\sigma.$ Equivalently, $\nu_{\sigma}$ is the law of the random variable $(\sigma,\bm{i})$.

The quenched version of the B--S convergence is inspired by the following equivalent reformulation of \cref{weakconv}: given a sequence $(\sigma_n)_{n\in\Z_{>0}}$ of deterministic elements in $\mathcal{S},$ we can equivalently say that $(\sigma_n)_{n\in\Z_{>0}}$ B--S converges to a probability measure $\nu$ on $\tilde{\mathcal{S}}^\bullet$, if the sequence $(\nu_{\sigma_n})_{n\in\Z_{>0}}$  converges to $\nu$ with respect to the weak topology induced by the local distance $d.$ 
In analogy, we suggest the following definition for the random case.
\begin{defn}[Quenched version of the B--S convergence]
	\label{strongconv}
	Let $(\bm{\sigma}_n)_{n\in\Z_{>0}}$ be a sequence of random permutations in $\mathcal{S}$ and $\bm{\nu}^\infty$ a random measure  on $\Sri$.  We say that $(\bm{\sigma}_n)_{n\in\Z_{>0}}$ \emph{converges in the quenched Benjamini--Schramm sense} to $\bm{\nu}^\infty$ if the sequence of random measures $(\bm{\nu}_{\bm{\sigma}_n})_{n\in\Z_{>0}}$ converges in distribution to $\bm{\nu}^\infty$ with respect to the weak topology induced by the local distance. In this case we write $\bm{\sigma}_n\gls*{quenched_BS}\bm{\nu}^\infty$ instead of $\bm{\nu}_{\bm{\sigma}_n}\xrightarrow{d}\bm{\nu}^\infty.$
\end{defn}
Unlike the annealed version of the B--S convergence, the limiting object $\bm{\nu}^{\infty}$ is a \emph{random measure} on $\Sri.$
Note also that, given a random permutation $\bm{\sigma},$ the corresponding random measure $\bm{\nu}_{\bm{\sigma}}$ is the conditional law of the random variable $\big((\bm{\sigma},\bm{i})\big|\bm{\sigma}\big)$, where $\bm{i}$ is uniform on $[|\bm{\sigma}|]$ conditioning on $\bm{\sigma}$.

We are now ready to state a theorem that gives a characterization of the quenched version of the B--S convergence. The proof of this result is quite technical (because we are dealing with random measures) but it follows the same lines as the proof of \cref{weakbsequivalence}. For these reasons, we skip this proof in this manuscript.

\begin{thm}[{\cite[Theorem 2.32]{borga2020localperm}}]
	\label{strongbsconditions}
	For any $n\in\Z_{>0},$ let $\bm{\sigma}_n$ be a random permutation of size $n$ and $\bm{i}_n$ be a uniform random index in $[n]$, independent of $\bm{\sigma}_n.$ Then the following assertions are equivalent:
	\begin{enumerate}[(a)]
		\item There exists a random measure $\bm{\nu}^\infty$ on $\Sri$ such that   $\bm{\sigma}_n\stackrel{qBS}{\longrightarrow}\bm{\nu}^\infty$. 
		\item There exists a family of non-negative real random variables $(\bm{\Gamma}^h_{\pi})_{h\in\Z_{>0},\pi\in\mathcal{S}_{2h+1}}$ such that  $$\Big(\P\left(r_h(\bm{\sigma}_n,\bm{i}_n)=(\pi,h+1)\middle|\bm{\sigma}_n\right)\Big)_{h\in\Z_{>0},\pi\in\mathcal{S}_{2h+1}}\xrightarrow{d}(\bm{\Gamma}^h_{\pi})_{h\in\Z_{>0},\pi\in\mathcal{S}_{2h+1}},$$
		w.r.t. the product topology.
		\item There exists an infinite vector of non-negative real random variables $(\bm{\Lambda}_{\pi})_{\pi\in\mathcal{S}}$ such that $$\big(\widetilde{\cocc}(\pi,\bm{\sigma}_n)\big)_{\pi\in\mathcal{S}}\xrightarrow{d}(\bm{\Lambda}_{\pi})_{\pi\in\mathcal{S}}$$ w.r.t.\ the product topology.
	\end{enumerate}
\end{thm}

The obvious generalizations of Remarks \ref{uyfvuoe3} and  \ref{uyfvuoe2} stated after \cref{weakbsequivalence} hold also for \cref{strongbsconditions}.

\subsubsection{Relation between the annealed and quenched versions of the B--S convergence}

We recall that the \emph{intensity (measure)} $\E[\bm{\nu}]$ of a random measure $\bm{\nu}$ is defined as the expectation $$\E[\bm{\nu}](A)\coloneqq\E\left[\bm{\nu}(A)\right],\quad \text{for all measurable sets } A.$$ 
\begin{prop}[{\cite[Proposition 2.35]{borga2020localperm}}]
	\label{wsrel}
	For any $n\in\Z_{>0},$ let $\bm{\sigma}_n$ be a random permutation of size $n$, independent of $\bm{\sigma}_n.$ If $\bm{\sigma}_n\stackrel{qBS}{\longrightarrow}\bm{\nu}^\infty,$  for some random measure $\bm{\nu}^\infty$ on $\Sri,$ then  $\bm{\sigma}_n\xrightarrow{aBS}\bm{\sigma}^\infty,$
	where $\bm{\sigma}^\infty$ is the random rooted infinite permutation with law $\E[\bm{\nu}^\infty].$  
\end{prop}

We now show that in general the two versions of B--S convergence are not equivalent.

\begin{exmp}
	\label{notequivalent}
	For all $n\in\Z_{>0},$ let $\bm{\sigma}_n$ be the random permutation defined by
	$$\P(\bm{\sigma}_n=12\dots n)=\frac{1}{2}=\P(\bm{\sigma}_n=n\;n\text-1\dots 1),$$
	and $\tau_n$  be the deterministic permutation
	\begin{align}
	&\tau_n=135\dots n\;n\text-1\;n\text-3\dots 2,\quad\text{if } n\text{ is odd,}\\
	&\tau_n=135\dots n\text-1\;n\;n\text-2\dots 2,\quad\text{if } n\text{ is even.}
	\end{align} 
	We now show that the alternating sequence between $\bm{\sigma}_n$ and $\tau_n$ converges in the annealed B--S sense but does not converge in the quenched one.

	Indeed, 
	$
	\big(\widetilde{\cocc}(\pi,\bm{\sigma}_n)\big)_{\pi\in\mathcal{S}}\xrightarrow{d}(\bm{X}_{\pi})_{\pi\in\mathcal{S}},
	$
	where, taking a Bernoulli random variable $\bm{Y}$ with parameter $1/2,$ $\bm{X}_{\pi}=\bm{Y}$ if $\pi=12\dots k$ for some $k\in\Z_{>0},$ $\bm{X}_{\pi}=1-\bm{Y}$ if $\pi=k\;k\text{-}1\dots 1$  for some $k\in\Z_{>0}$ and $\bm{X}_{\pi}=0$ otherwise.
	
	On the other hand, 
	$
	\big(\widetilde{\cocc}(\pi,\tau_n)\big)_{\pi\in\mathcal{S}}\xrightarrow{d}(Z_{\pi})_{\pi\in\mathcal{S}},
	$
	where $Z_{\pi}=\frac{1}{2}$ if $\pi=12\dots k,$ or $\pi=k\;k\text{-}1\dots 1$ and $Z_{\pi}=0$ otherwise.
	
	Therefore, using condition (c) in \cref{strongbsconditions}, the alternating sequence between $\bm{\sigma}_n$ and $\tau_n$ does not converge in the quenched B--S sense. Moreover, since $\lim_{n\to\infty}\E\big[\widetilde{\cocc}(\pi,\bm{\sigma}_n)\big]$ is equal to $\lim_{n\to\infty}\E\big[\widetilde{\cocc}(\pi,\tau_n)\big]$ for all $\pi\in\mathcal{S}$, using condition (c) in \cref{weakbsequivalence} we conclude that the alternating sequence between $\bm{\sigma}_n$ and $\tau_n$ converges in the annealed B--S sense.
\end{exmp}

We now analyze the particular case when the limiting objects $(\bm{\Lambda}_{\pi})_{\pi\in\mathcal{S}}$ (or $(\bm{\Gamma}^h_{\pi})_{h\in\Z_{>0},\pi\in\mathcal{S}_{2h+1}}$) in \cref{strongbsconditions} are deterministic, i.e.\ when there is a concentration phenomenon (this will be the case of several results in \cref{chp:local_lim}). Before stating our result we need the following.
\begin{rem}
	When a random measure $\bm{\nu}$ on $\Sri$ is almost surely equal to a deterministic measure $\nu$ on $\Sri$, we will simply denote it with $\nu.$ In particular, if a sequence of random permutations $(\bm{\sigma}_n)_{n\in\Z_{>0}}$ converges in the quenched B--S sense to a deterministic measure $\nu$ (instead of random measure) on $\Sri$, we will simply write $\bm{\sigma}_n\stackrel{qBS}{\longrightarrow}\nu.$
\end{rem}

\begin{cor}[{\cite[Corollary 2.38]{borga2020localperm}}]
	\label{detstrongbsconditions}
	For any $n\in\Z_{>0},$ let $\bm{\sigma}_n$ be a random permutation of size $n$ and $\bm{i}_n$ be a uniform random index in $[n]$, independent of $\bm{\sigma}_n.$ Then the following assertions are equivalent:
	\begin{enumerate}[(a)]
		\item There exists a (deterministic) measure $\nu$ on $\Sri$ such that   $\bm{\sigma}_n\stackrel{qBS}{\longrightarrow}\nu.$
		\item There exists an infinite vector of non-negative real numbers $(\Lambda'_{\pi})_{\pi\in\mathcal{S}}$ such that for all $\pi\in\mathcal{S},$ $$\widetilde{\cocc}(\pi,\bm{\sigma}_n)\stackrel{P}{\to}\Lambda'_{\pi}.$$
	\end{enumerate}
\end{cor}
\begin{rem}
	Thanks to \cref{wsrel} note that if (a) holds then $\bm{\sigma}_n\xrightarrow{aBS}\bm{\sigma}^\infty,$
	where $\bm{\sigma}^\infty$ is the random rooted infinite permutation with law $\mathcal{L}aw(\bm{\sigma}^\infty)=\nu.$ 
\end{rem}

\begin{rem}
	\cref{detstrongbsconditions} shows that in order to prove convergence in the quenched B--S sense when the limiting object is deterministic, it is enough to prove pointwise convergence for the vector $\big(\widetilde{\cocc}(\pi,\bm{\sigma}_n)\big)_{\pi\in\mathcal{S}}$, instead of its joint convergence. This is not true when the limiting object is random. For a counterexample see \cite[Example 2.36]{borga2020localperm}.
\end{rem}

\subsection{A characterization of Benjamini-Schramm limits}\label{sect:sofic}

We characterize random limits for the B--S convergence introducing a "shift-invariant" property. 

For an order $(\Z,\preccurlyeq),$ its \emph{shift} $(\Z,\preccurlyeq')$ is defined by $i+1\preccurlyeq' j+1$ if and only if $i\preccurlyeq j.$ A random infinite rooted permutation is said to be \emph{shift-invariant} if it has the same distribution as its shift.

\begin{thm}
	\label{ewoifhnoipewjfpoew}
	The annealed B--S limit of a sequence of random permutations is shift-invariant.
	
	Moreover, every random shift-invariant infinite rooted permutation can be obtained as the B--S limit of a sequence of finite deterministic permutations. 
\end{thm}

The first statement of the theorem above follows from \cite[Proposition 2.44]{borga2020localperm}. The second part is a consequence of \cite[Theorem 2.45]{borga2020localperm} and \cite[Proposition 3.4]{borga2019feasible}.

In some sense, our "shift-invariant" property corresponds to the notion of unimodularity for random graphs discussed in \cref{sect:loc_lim_graph}. \cref{ewoifhnoipewjfpoew} solves the \emph{Sofic problem} for permutations.


\section[Feasible regions]{Feasible regions for patterns and consecutive patterns}\label{sect:feas}

A natural question, motivated by the characterizations of permuton and local convergence for permutations (Theorems \ref{thm:randompermutonthm} and \ref{strongbsconditions}), is the following: Given a finite family of patterns $\mathcal{A}\subseteq\SS$ and a vector $(\Delta_\pi)_{\pi\in\mathcal{A}}\in [0,1]^{\mathcal A}$, or $(\Gamma_\pi)_{\pi\in\mathcal{A}}\in [0,1]^{\mathcal A}$, does there exist a sequence of permutations $(\sigma^n)_{n\in\Z_{>0}}$ such that $|\sigma^n|\to\infty$ and
\begin{equation}
	\widetilde{\occ}(\pi,\sigma^n)\to\Delta_{\pi}, \quad \text{for all} \quad \pi\in\mathcal{A},
\end{equation}
or
$$\widetilde{\cocc}(\pi,\sigma^n)\to\Gamma_{\pi}, \quad \text{for all} \quad \pi\in\mathcal{A}\;?$$

We consider the classical pattern limiting sets, sometimes called the \textit{feasible region} for (classical) patterns, defined as
\begin{align}
	clP_k \coloneqq&\left\{\vec{v}\in [0,1]^{\SS_k} \big| \exists (\sigma^m)_{m\in\Z_{>0}} \in \SS^{\Z_{>0}} \text{ s.t. }|\sigma^m| \to \infty\text{ and } \pocc(\pi, \sigma^m ) \to \vec{v}_{\pi},\forall \pi\in\SS_k  \right\}\nonumber\\
	=&\left\{(\Delta_{\pi}(P))_{\pi\in\SS_k} \big| P\text{ is a permuton}  \right\},\label{eq:setofnotinter}
\end{align}
and the consecutive pattern limiting sets, called here the \textit{feasible region} for consecutive patterns,
\begin{align}
	P_k \coloneqq &\left\{\vec{v}\in [0,1]^{\SS_k} \big| \exists (\sigma^m)_{m\in\Z_{>0}} \in \SS^{\Z_{>0}} \text{ s.t. }|\sigma^m| \to
	\infty \text{ and }  \pcocc(\pi, \sigma^m ) \to \vec{v}_{\pi}, \forall \pi\in\SS_k \right\}\nonumber\\
	=&\left\{(\Gamma_{\pi}(\sigma^{\infty}))_{\pi\in\SS_k} \big| \sigma^{\infty}\text{ is a random infinite rooted \emph{shift-invariant} permutation}  \right\}. \label{eq:setofinter}
\end{align}
The equality in \cref{eq:setofinter} follows from \cref{ewoifhnoipewjfpoew}. The equality in \cref{eq:setofnotinter} follows from \cite[Theorem 1.6]{hoppen2013limits}.

The feasible region $clP_k$ for (classical) patterns was previously studied in some papers (see \cref{sect:feas_reg}), while the feasible region $P_k$ for consecutive patterns was introduced by the author of this manuscript together with Penaguiao \cite{borga2019feasible}. 
The main goal of this section is to analyze the feasible region $P_k$ for consecutive patterns, that turns out to be connected to specific graphs called \emph{overlap graphs} (see \cref{defn:ovgraph}) and their corresponding cycle polytopes (see \cref{defn:cycle_poly}).

\subsection{Feasible regions for classical patterns}
\label{sect:feas_reg}

A complete description of the feasible region $clP_k$ is (at the moment) out-of-reach, but several simpler questions have been investigated in the literature.

The feasible region $clP_k$ for some particular families of patterns instead of the whole $\SS_k$ was first studied in \cite{kenyon2020permutations}.  
More precisely, given a list of finite sets of permutations $(\mathcal{P}_1,\dots,\mathcal{P}_\ell)$, the authors considered the \emph{feasible region} for $(\mathcal{P}_1,\dots,\mathcal{P}_\ell)$, that is, the set
$$\left\{\vec{v}\in [0,1]^{\ell} \Bigg| \exists (\sigma^m)_{m\in\Z_{>0}} \in \SS^{\Z_{>0}} \text{ s.t. }|\sigma^m| \to \infty\text{ and } \sum_{\tau\in\mathcal{P}_i}\pocc(\tau,\sigma^m ) \to \vec{v}_i,\text{ for }i = 1, \dots , \ell  \right\} \, .$$

They first studied the simplest case when $\mathcal{P}_1=\{12\}$ and $\mathcal{P}_2=\{123,213\}$ showing that the corresponding feasible region for $(\mathcal{P}_1 ,\mathcal{P}_2)$ is the region of the square $[0,1]^2$ bounded from below by the parameterized curve $(2t-t^2,3t^2-2t^3)_{t\in[0,1]}$ and from above by the parameterized curve $(1-t^2,1-t^3)_{t\in[0,1]}$  (see \cite[Theorem 14]{kenyon2020permutations}).

They also proved in \cite[Theorem 15]{kenyon2020permutations} that if each $\mathcal{P}_i = \{\tau_i \}$ is a singleton, and if there is some value $p$ such that, for each $\tau_i$ its final element $\tau_i(|\tau_i|)$ is equal to $p$ , then the corresponding feasible region is convex. 
They remarked that one can construct examples where the feasible region is not strictly convex: e.g.\ in the case where $\mathcal{P}_1 = \{231,321\}$ and $\mathcal{P}_2 = \{123,213\}$.

They finally studied two additional examples: the feasible regions for $(\{12\},\{123\})$ (see \cite[Theorem 16]{kenyon2020permutations}) and for the patterns $(\{123\},\{321\})$ (see \cite[Section 10]{kenyon2020permutations}). In the first case, they showed that the feasible region is equal to the so-called “scalloped triangle” of
Razborov \cite{MR2433944,MR2371204} (this region also describes the space of limit densities for edges and triangles in graphs). 
For the second case, they showed that the feasible region is equal to the limit of densities of triangles versus the density of anti-triangles in graphs, see \cite{MR3200287,MR3572422}.

\medskip

The set $clP_k$ was also studied in \cite{MR3567538}, even though with a different goal.
There, it was shown that the region $clP_k$ contains an open ball $B$ with dimension $|I_k|$, where $I_k$ is the set of $\oplus$-indecomposable permutations of size at most $k$.
Specifically, for a specific ball $B\subseteq \R^{I_k}$, the authors constructed permutons $P_{\vec{x}}$ such that $\Delta_{\pi }(P_{\vec{x}}) = \vec{x}_{\pi}$, for each $\vec{x} \in B$ and each $\pi\in I_k$.

This work opened the problem of finding the maximal dimension of an open ball contained in $clP_k$, and placed a lower bound on it.
In \cite{vargas2014hopf} an upper bound for this maximal dimension was indirectly given as the number of so-called \textit{Lyndon permutations}\footnote{The precise definition of Lyndon permutations is not immediate. We just point out that they contain all $\oplus$-indecomposable permutations.} of size at most $k$, whose set is denoted by $\mathcal{L}_k$.
In this article, the author showed that for any permutation $\pi$ that is not a Lyndon permutation, $\pocc(\pi, \sigma ) $ can be expressed as a polynomial on the functions $\{\pocc(\tau, \sigma ) |\tau \in \mathcal{L}_k \}$ that does not depend on $\sigma$.
It follows that $clP_k$ sits inside an algebraic variety of dimension $|\mathcal{L}_k|$.
We expect that this bound is sharp since, from our computations, this is the case for small values of $k$.

\begin{conj}\label{conj:feas_reg}
	The feasible region $clP_k$ contains an open ball of dimension $|\mathcal{L}_k|$.
\end{conj}

\subsection{Feasible regions for consecutive patterns}\label{cref:vediuyewvdb}
Unlike in the case of classical patterns, we are able to obtain a full description of the feasible region $P_k$ as the cycle polytope of a specific graph, called the \emph{overlap graph} $\ValGraph[k]$. 

\begin{defn}\label{defn:ovgraph}
	The graph $\ValGraph[k]$ is a directed multigraph with labeled edges, where the vertices are elements of $\SS_{k-1}$ and for every $\pi\in\SS_{k}$ there is an edge labeled by $\pi$ from the pattern induced by the first $k-1$ indices of $\pi$ to the pattern induced by the last $k-1$ indices of $\pi$.
\end{defn}

The overlap graph $\ValGraph[4]$ is displayed in \cref{Overlap_graph_exemp}.

\begin{figure}[htbp]
	\begin{minipage}[c]{0.59\textwidth}
		\centering
		\includegraphics[scale=.45]{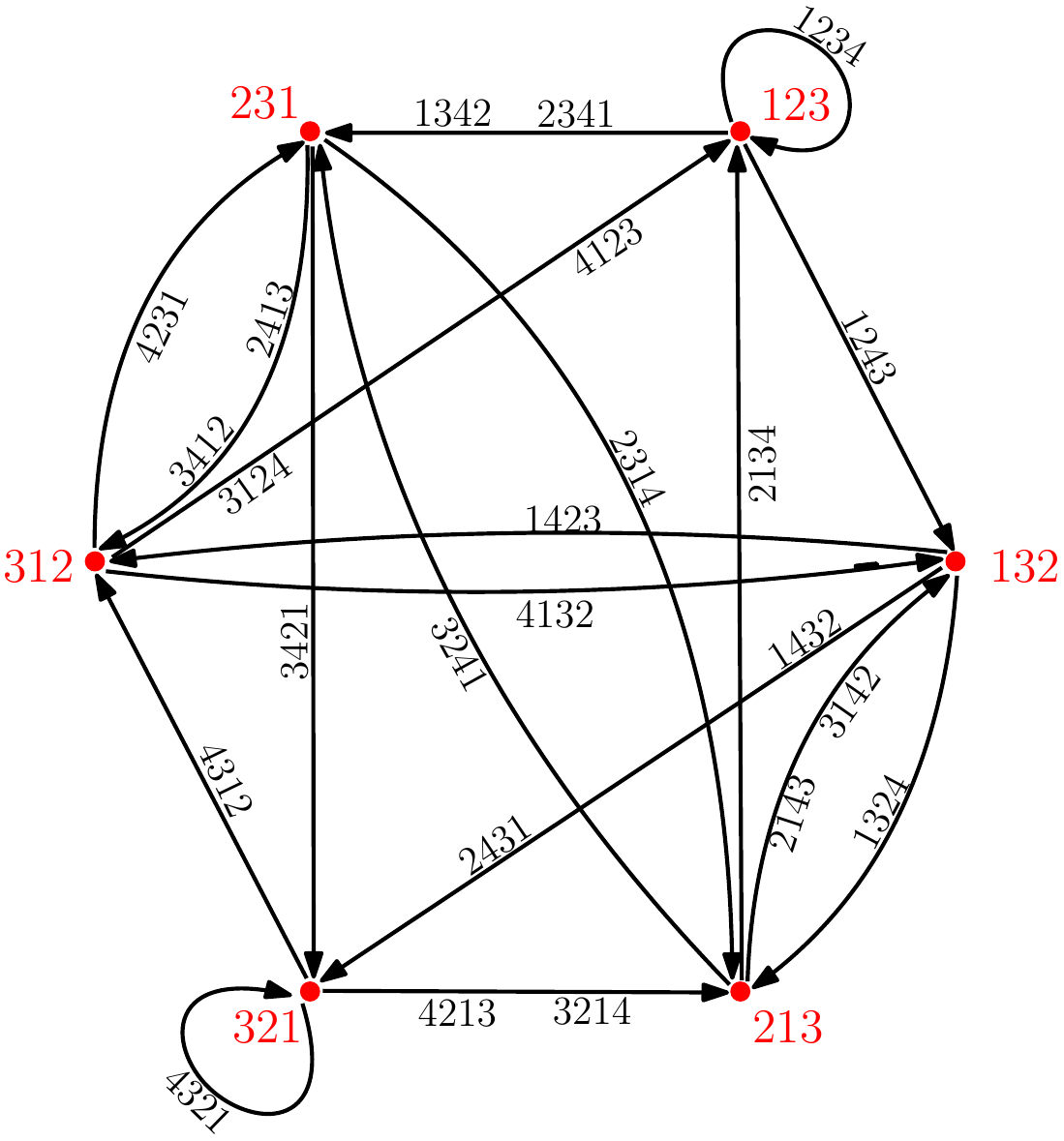}
	\end{minipage}
	\begin{minipage}[c]{0.4\textwidth}
		\caption{The overlap graph $\ValGraph[4]$. The six vertices are painted in red and the edges are drawn as labeled arrows. Note that in order to obtain a clearer picture we did not draw multiple edges, but we use multiple labels (for example the edge $231 \to 312$ is labeled with the permutations $3412$ and $2413$ and should be thought of as two distinct edges labeled with $3412$ and $2413$ respectively). \label{Overlap_graph_exemp}}
	\end{minipage}
\end{figure}

\begin{defn}\label{defn:cycle_poly}
	Let $G=(V,E)$ be a directed multigraph.
	For each simple cycle $\mathcal{C}$ in $G$, define $\vec{e}_{\mathcal{C}}\in \mathbb{R}^{E}$ by
	$$(\vec{e}_{\mathcal{C}})_e \coloneqq \frac{ \mathds{1}_{\{e\in\mathcal{C}\}}}{|\mathcal{C}|}, \quad \text{for all}\quad e\in E. $$
	We define the \emph{cycle polytope} of $G$ to be the polytope $P(G) \coloneqq \conv \{\vec{e}_{\mathcal{C}} | \, \mathcal{C} \text{ is a simple cycle of } G \}$.
\end{defn}

Our main result is the following.

\begin{thm}[{\cite[Theorem 1.6]{borga2019feasible}}]
	\label{thm:main_res}
	$P_k$ is the cycle polytope of the overlap graph $\ValGraph[k]$. Its dimension is $k! - (k-1)!$\,.
\end{thm}

An instance of the result stated in \cref{thm:main_res} is depicted in \cref{fig:P_3_and_rest2}.
\begin{figure}[htbp]
	\centering
	\includegraphics[scale=0.9]{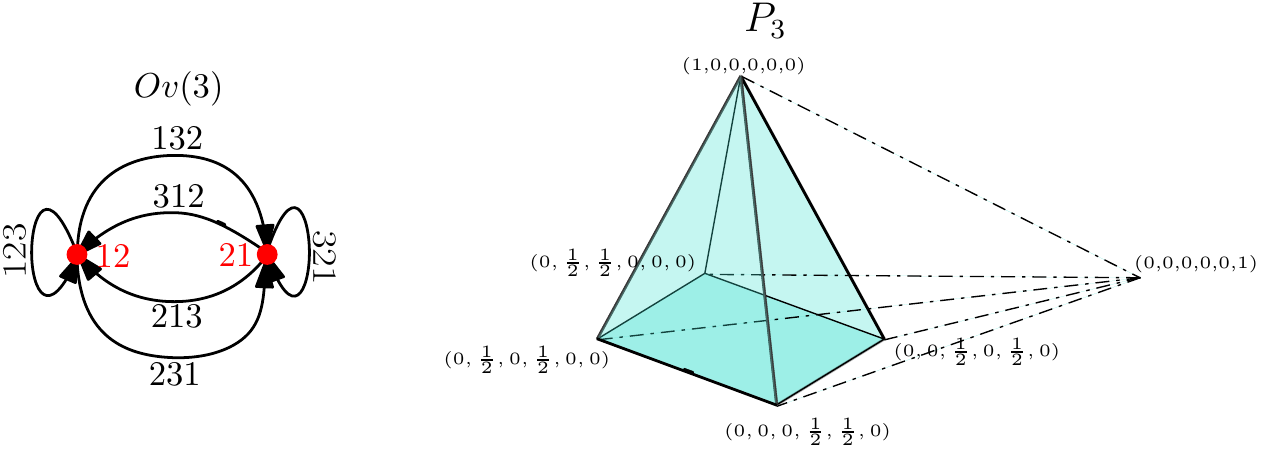}
	\caption{The overlap graph $\ValGraph[3]$ and the corresponding four-dimensional cycle polytope, that is the feasible region $P_3$. The coordinates of the vertices correspond to the patterns $(123,231,312,213,132,321)$ respectively. Note that the top vertex (resp.\ the right-most vertex) of the polytope corresponds to the loop indexed by $123$ (resp.\ $321$); the other four vertices correspond to the four cycles of length two in $\ValGraph[3]$. 
	We highlight in light-blue one of the six three-dimensional faces of $P_3$. This face is a pyramid with a square base. The polytope itself is a four-dimensional pyramid, whose base is the highlighted face. \label{fig:P_3_and_rest2}}
\end{figure}

In addition, we also determine the equations that describe the polytope $P_k$ (for a precise statement see \cite[Theorem 3.12]{borga2019feasible}). We finally mention that in \cite{borga2020feasible} we generalized the results in \cite{borga2019feasible} to pattern-avoiding permutations, i.e.\ we studied the regions
\begin{multline}\label{eq:weufyvewyufvewyf}
	P^{B}_k \coloneqq \{\vec{v}\in [0,1]^{\SS_k} \big| \exists (\sigma^m)_{m\in\Z_{>0}} \in \Av(B)^{\Z_{>0}} \text{ such that }\\
	|\sigma^m| \to \infty \text{ and }  \pcocc(\pi, \sigma^m ) \to \vec{v}_{\pi}, \forall \pi\in\SS_k \},
\end{multline}
where $B\subset\mathcal{S}$ is a fixed set of patterns. We proved the following.

\begin{thm}[{\cite[Theorem 1.1]{borga2020feasible}}]
	Fix $k\in\Z_{\geq 1}$ and a set of patterns $B\subset\mathcal{S}$ such that the family $\Av(B)$ is closed either for the $\oplus$ operation or $\ominus$ operation. The feasible region $P^{B}_k$ is closed and convex. Moreover,
	$$\dim(P^{B}_k)= |\Av_k(B)|-|\Av_{k-1}(B)|.$$
\end{thm}

We remark that the assumptions of the theorem above are satisfied (for instance) by all the families $\Av ( \tau )$ for any fixed pattern $\tau \in \mathcal S$ and various families of pattern-avoiding permutations avoiding multiple patterns. For instance, all substitution-closed classes (see \cref{sect:sub_close}).

\section{Open problems}\label{sect:op_probl_2}

\begin{itemize}
	\item In this chapter we introduced the notion of permuton and B--S limits for permutations. In the next chapters, we will mainly focus on determining convergence of models of random permutations. Another interesting direction of research would be to investigate properties of \emph{convergent} sequences.\newline 
	Some answers for permuton convergence are established: Mueller and Starr \cite{MR3055815}, building on the permuton convergence of Mallows permutations~\cite{starr2009thermodynamic}, determined a law of large numbers for the length of the longest increasing subsequence in the Mallows model; Mukherjee~\cite{MR3515570} connected short cycle count statistics to permuton
	convergence; and very recently Bassino, Bouvel, Drmota, F{\'e}ray, Gerin, Maazoun, and Pierrot~\cite{bassino2021increasing} showed that the length of the longest increasing subsequence in permutations converging to the Brownian separable permuton (introduced in \cref{sect:sub_close_cls}) has sublinear size (building on self-similarity properties of  the Brownian separable permuton). Nevertheless, a general theory (that is not model-dependent) for this type of questions would be desirable and we believe it will be investigated in future projects. \newline 
	On the other hand B-S convergence for permutations has not been investigated in this direction so far, but we expect several interesting results also here. Let us mention an analogy with B--S limits for graphs: recently, Salez~\cite{salez2021sparse} solved a long-standing open problem on the curvature of sparse expander graphs working directly at the level of Benjamini-Schramm limits. Other results were established in his Ph.D. thesis~\cite{salez:tel-00637130}. We believe that similar techniques can be used to establish results on permutations from their B-S limits.
	
	\item Local and global convergence for discrete structures are arguably the best-known notions of convergence, but at the same time other notions of convergence have been considered. For instance, the notion of semi-local convergence is used in graph theory and interpolates between the local and the global one. An example is the work of Curien and Le Gall \cite{curien2014brownian}, where they studied a notion of semi-local convergence for uniform triangulations, obtaining a new limiting object called the \emph{Brownian plane}. They also showed some interesting relations with the well-know Brownian map (the scaling limit) and the Uniform infinite planar triangulation (the local limit). We think it is possible to introduce a notion of \emph{semi-local convergence} also for random permutations that will be useful to investigate the relation between permutons (scaling limit) and infinite rooted permutations (local limits). 
	Moreover, as global convergence is in connection with permutation patterns and local convergence is in connection with consecutive permutation patterns, semi-local convergence would be characterized in terms of convergence of \emph{semi-consecutive patterns}, \emph{i.e.,} patterns determined by a set of indexes with some restrictions on the maximal distance between them. 
	Therefore, the main questions are
	\vspace{-0.14 cm}
		\begin{itemize}
			\item How can a notion of semi-local convergence for permutations be defined?
			\vspace{-0.13 cm}
			\item Can it be characterized in terms of semi-local patterns?
			\vspace{-0.13 cm}
			\item What is the relationship between the different types of convergence and the corresponding limiting objects?		
		\end{itemize}
	\vspace{-0.14 cm}
	We remark that Bevan~\cite{bevan2020independence} implicitly considered a notion of \emph{semi-local convergence} for permutations, but he discusses neither the potential limiting objects nor the relations with the other notions of convergence.
	
	\item An interesting open problem of combinatorial/geometric flavor is to prove \cref{conj:feas_reg}.
	
\end{itemize}

    \chapter{Models of random permutations \& combinatorial constructions}\label{chp:models}
\chaptermark{Models of random permutations}

\begin{adjustwidth}{8em}{0pt}
	\emph{In which we introduce several models of random constrained permutations. The goal is to collect all the combinatorial constructions that we need to prove later various probabilistic results. The structure of the following sections is quite similar: we start by describing a specific family of permutations, then we present or construct some bijections with other useful combinatorial objects (such as trees, maps, walks, etc.), and finally we provide a nice and convenient way to construct uniform permutations in the family.}
\end{adjustwidth}
 
 \bigskip
 
  \bigskip
  
  \bigskip

\noindent Here is a list of families of permutations considered in this chapter. In addition, we indicate where the various sections in this chapter will be used later.

\begin{itemize}
	\item Permutations avoiding a pattern of length three (\cref{sect:perm_len_three} $\Rightarrow$ \cref{sect:231proofs,321}).
	\item Substitution-closed classes (\cref{sect:sub_close} $\Rightarrow$  \cref{sec:local_lim,sect:subclosedperm}).
	\item Square and almost-square permutations (\cref{sect:sq_perm} $\Rightarrow$ \cref{const_lim_obj} and \cref{chp:square}).
	\item Permutation families encoded by generating trees (\cref{sect:gen_tree-perm} $\Rightarrow$ \cref{sect:CLTgentree}).
	\item Baxter permutations (\cref{sec:discrete} $\Rightarrow$  \cref{sec:local,sect:baxperm}).
\end{itemize}

\subsubsection*{Notation for trees}

We only consider rooted plane trees;
{\em plane} means that the children of any given vertex are endowed with a linear order.
Let $\gls*{rooted_plane_tree_n}$ denote the set of rooted plane trees with $n$ vertices and $\gls*{rooted_plane_tree}$ the set of all finite rooted plane trees.
The \emph{out-degree} $d^+_T(v)$ (or $d^+(v)$ when there is no ambiguity) of a vertex $v$ in a tree $T$ is the number of its children (sometimes called arity in other works). 
Note that it may be  different from the graph-degree:
the edge to the parent (if it exists) is not counted in the out-degree. 
We consider both finite and infinite trees. 
We say a tree is {\em locally finite} if all its vertices have finite degree. 
A vertex of $T$ is called a {\em leaf} if it has out-degree zero. The collection of non-leaves (also called \emph{internal vertices}) is denoted by $\Vint(T)$.
The \emph{fringe subtree} of a tree $T$ rooted at a vertex $v$ is the subtree of $T$ containing $v$ and all its descendants. 
We will also speak of {\em branch} attached to a vertex $v$ for a fringe subtree rooted at a child of $v$.

We recall the definition of the Ulam--Harris tree  $\mathcal{U}_{\infty}$. 
The vertex set of  $\mathcal{U}_{\infty}$ is given by the collection of all finite sequences of positive integers, 
and the offspring of a vertex $(i_1, \ldots, i_k)$, i.e.\ the list of children, is given by all sequences $(i_1, \ldots, i_k, j)$, $j\in\Z_{>0}$. The root of $\mathcal{U}_{\infty}$ is the unique sequence of length $0$.
Any plane tree can be encoded in a canonical way as a subtree of  $\mathcal{U}_{\infty}$.

We finally recall two classical ways of visiting vertices of a rooted plane tree that go under the name of \emph{depth-first traversals.} These traversals, originally developed for algorithms searching vertices in a tree, work as follows.
Let $T$ be a rooted plane tree with root $r$. 

\begin{defn}\label{defn:preorder}
	If $T$ consists only of $r$, then $r$ is the \emph{pre-order} traversal of $T$. Otherwise, suppose that $T_1,T_2,\dots,T_d$ are the fringe subtrees in $T$ respectively rooted at the children of $r$ from left to right. The pre-order traversal begins by visiting $r$. It continues by traversing $T_1$ in pre-order, then $T_2$ in pre-order,	and so on, until $T_d$ is traversed in pre-order.
\end{defn}

\begin{defn}\label{defn:postorder}
	If $T$ consists only of $r$, then $r$ is the \emph{post-order} traversal of $T$. Otherwise, suppose that $T_1,T_2,\dots,T_d$ are the fringe subtrees in $T$ respectively rooted at the children of $r$ from left to right. The post-order traversal begins by traversing $T_1$ in post-order, then $T_2$ in post-order, and so on, until $T_d$ is traversed in post-order. It ends by visiting $r.$
\end{defn}

We will also say that $T$ is labeled with the pre-order (resp.\ post-order) labeling starting from $k$, if we associate the label $i+k-1$ with the $i$-th visited vertex by the pre-order (resp.\ post-order) traversal. Moreover, if we simply say that $T$ is labeled with the pre-order (resp.\ post-order), we will assume that we are starting from $k=1$. Examples will be given later.

\section{Permutations avoiding a pattern of length three}\label{sect:perm_len_three}

In this first section we focus on the classes of $\rho$-avoiding permutations, when $|\rho|=3.$ These permutations are well-known to be enumerated by Catalan numbers and they have been intensively studied in the literature. In particular, several bijections with other combinatorial objects -- such as rooted plane trees and Dyck paths -- are known for these families of permutations (see for instance \cite{claesson2008classification}). Here, we only present the bijections we need for our work. Using trivial symmetries, it is enough to investigate 231-avoiding and 321-avoiding permutations.

\subsection{231-avoiding permutations}
\label{231bij}
We present in this section a well-known bijection between 231-avoiding permutations and binary trees (see for instance \cite{bona2012surprising}) which sends the size of the permutation to the number of vertices of the tree. 

A binary tree is a rooted plane tree where each vertex has at most two children; when a vertex has only one child, this can be either a left or right child\footnote{We highlight that for binary trees we are using a slightly generalized notion of "\emph{plane tree}" since we are distinguishing left and right children. We believe that this will not make any confusion. See also the left-hand side of \cref{maxontree} for an example of binary tree.}.
Given a vertex $v$ in the tree, we denote by $\gls*{left_right_tree}$ the left and the right fringe subtrees of $T$ hanging below the vertex $v.$ We simply write $T_L$ and $T_R$ if $v$ is the root. Finally we let $\gls*{binary_tree}$ be the set of all binary trees.

Given $\sigma\in\mathcal{S}_n,$ let $\gls*{ind_max}$ be the index of the maximal value $n,$ i.e.\ $\sigma(\indmax(\sigma))=n.$ If $\ell=\indmax(\sigma),$ we define $\gls*{left_right_perm}$ as $\sigma_L\coloneqq\sigma(1)\dots\sigma(\ell-1)$ and $\sigma_R\coloneqq\sigma(\ell+1)\dots\sigma(n),$ respectively the (possibly empty) left and right subsequences of $\sigma,$ before and after the maximal value. In particular we have $\sigma=\sigma_L\sigma(\ell)\sigma_R,$ where we point out that $\sigma_L\sigma(\ell)\sigma_R$ is not the composition of permutations but just the concatenation of $\sigma_L,\sigma(\ell)$ and $\sigma_R$ seen as words.

The bijection mentioned above is built on the following result (see for instance \cite{bona2010absence}). 
\begin{obs} 
	\label{permfact}
	Let $\sigma$ be a permutation and $\ell=\indmax(\sigma).$ The permutation $\sigma$ avoids $231$ if and only if $\sigma_L$ and $\sigma_R$ both avoid $231$ and furthermore $\sigma(i)<\sigma(j)$ whenever $i<\ell$ and $j>\ell.$
\end{obs}
Given a permutation $\sigma\in\Av(231)$ we build a binary tree $T=\varphi(\sigma)$ as follows. If $\sigma$ is empty then $T$ is the empty tree. Otherwise we add the root in $T,$ which corresponds to the maximal element of $\sigma$ and we split $\sigma$ in $\sigma_L$ and $\sigma_R;$ then the left subtree of $T$ is the tree corresponding to $\sigma_L,$ i.e.\ $T_L=\varphi(\sigma_L)$ and similarly, the right subtree of $T$ is the tree corresponding to $\text{std}(\sigma_R),$ i.e.\ $T_R=\varphi(\text{std}(\sigma_R)).$

Conversely, given a binary tree $T$ we construct the corresponding permutation $\sigma=\psi(T)$ in $\text{Av}(231)$ as follows (see \cref{bijtreeperm} below): if $T$ is empty then $\sigma$ is the empty permutation. Otherwise  we split $T$ in $T_L$ and $T_R$ and we set $n\coloneqq|T|$, the number of vertices of $T,$ $S_L\coloneqq\psi(T_L)$ and $S_R\coloneqq\psi(T_R).$ Then we define $\sigma$ as $\sigma=\sigma_Ln\sigma_R$ where $\sigma_L$ is simply $S_L$ and $\sigma_R$ is obtained by shifting all entries of $S_R$ by $|T_L|,$ i.e.\ $\sigma_R(j)=S_R(j)+|T_L|,$ for all $j\leq |T_R|.$

It is clear that $\varphi$ and $\psi$ are inverse of each other, hence providing the desired bijection.

\begin{figure}[htbp]
	\begin{minipage}[c]{0.7\textwidth}
		\centering
		\includegraphics[scale=.70]{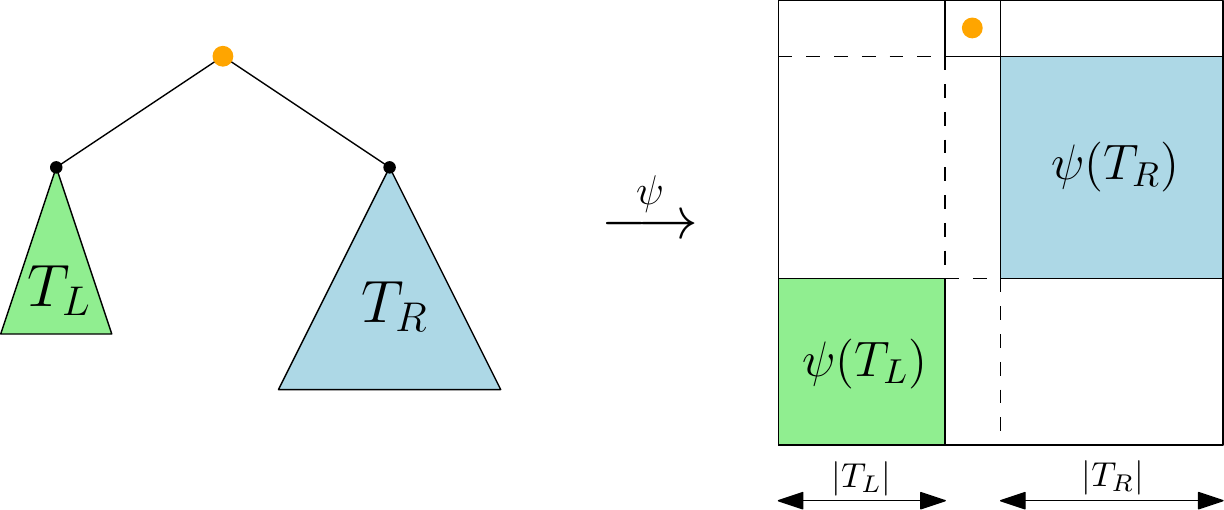}
	\end{minipage}
	\begin{minipage}[c]{0.29\textwidth}
		\caption{A binary tree and the corresponding 231-avoiding permutation as a diagram.\label{bijtreeperm}}
	\end{minipage}
\end{figure}

We denote by $T_{\sigma}$ the binary tree associated with $\sigma$ and, analogously, by $\sigma_T$ the permutation associated with the binary tree $T.$ We have the following consequence.

\begin{prop}\label{prop:unif_231_as_trees}
	Let $\bm{T}_n$ denote a uniform random binary tree with $n$ vertices. Then $\sigma_{\bm{T}_n}$ is a uniform 231-avoiding permutation of size $n$.
\end{prop} 

We end this section with some observations and remarks that will be useful in the sequel.

\begin{obs}
	\label{maxnode}
	By construction, the left-to-right maxima in a permutation $\sigma\in\Av(231)$ correspond to the vertices of the form $v=1\dots1$ (in the Ulam--Harris labeling). Similarly, the right-to-left maxima correspond to the vertices of the form $v=2\dots2.$ Obviously the maximum corresponds to the root of the tree. An example is given in \cref{maxontree}.
\end{obs}

\begin{figure}[htbp]
		\centering
		\includegraphics[scale=.70]{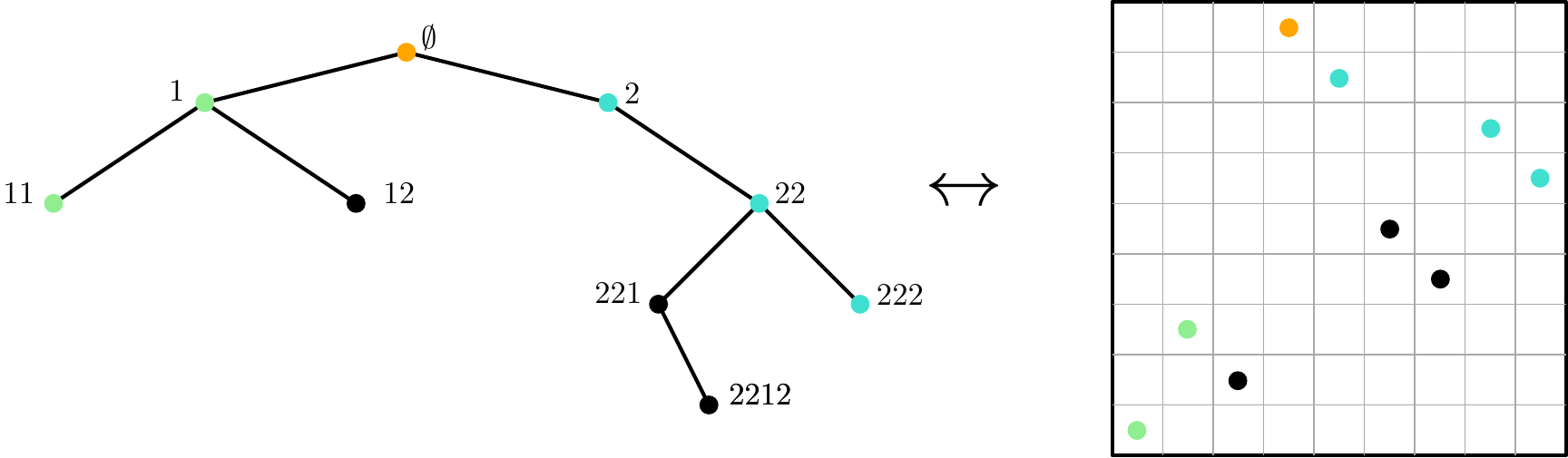}\\
		\caption{Correspondence between maxima of a permutation and vertices on the left and right branches in the associated tree. We highlight in green the left-to-right maxima, in light-blue the right-to-left maxima and in orange the maximum.
		}\label{maxontree}
\end{figure}

The notion of pre-order and post-order traversal (recall Definitions \ref{defn:preorder} and \ref{defn:postorder}) extend immediately to binary trees. Moreover, in the case of binary trees, the in-order traversal is also available. Let $T$ be a binary tree with root $r$. 
\begin{defn}
	If $T$ consists only of $r$, then $r$ is the \emph{in-order} traversal of $T$. Otherwise the in-order traversal begins by traversing $T_L$ in in-order, then visits $r,$ and concludes by traversing $T_R$ in in-order.
\end{defn}
The notion of in-order labeling is defined similarly to the pre-order labeling or post-order labeling.
We exhibit an example of the three different labelings in \cref{exemporder1}.

\begin{figure}[htbp]
	\centering
		\includegraphics[scale=.70]{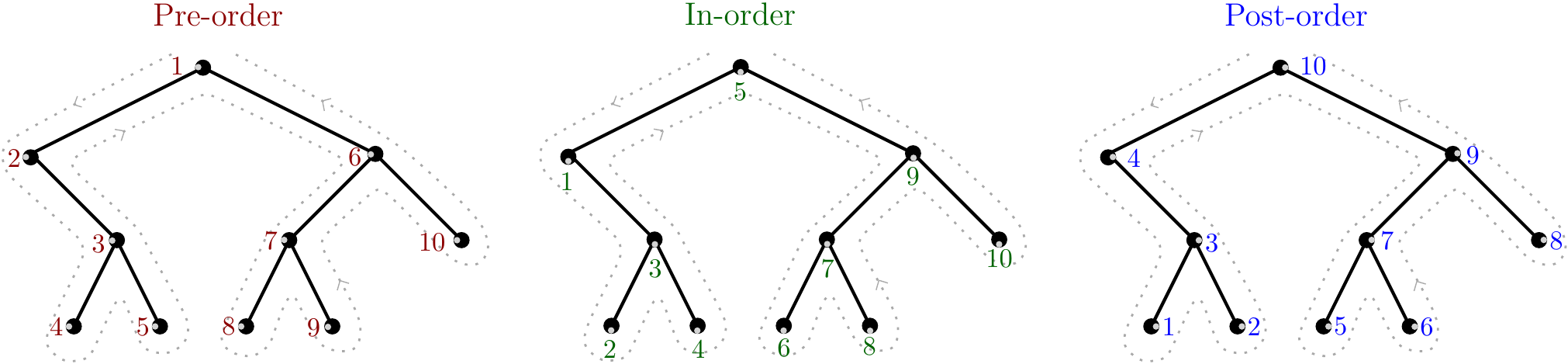}\\
		\caption{We label in red the left picture using the pre-order, in green the middle picture using the in-order and in blue the right picture using the post-order. The three labelings can be obtained following the gray dotted lines from left to right and putting a tag at each time we reach the small gray dots inside each vertex (that are on the left of the vertex for the pre-order, on the bottom of the vertex for the in-order and on the right of the vertex for the post-order).
		}\label{exemporder1}
\end{figure}

We make the following final remark that is not strictly necessary in this section but useful for comparison with the next one.

\begin{rem}
	It is not difficult to prove the following equivalent description of the bijection between trees and permutations based on the use of the tree traversals. Given a binary tree $T$ with $n$ vertices we reconstruct the corresponding permutation in $\text{Av}(231)$ setting $\sigma(i)$ to be equal to the label given by the post-order to the $i$-th vertex visited by the in-order. Namely, we are using the in-order labeling for the indices of the permutation and the post-order labeling for the values. An example is provided in \cref{Av(231)bij}.
\end{rem}

\begin{figure}[htbp]
		\centering
		\includegraphics[scale=.8]{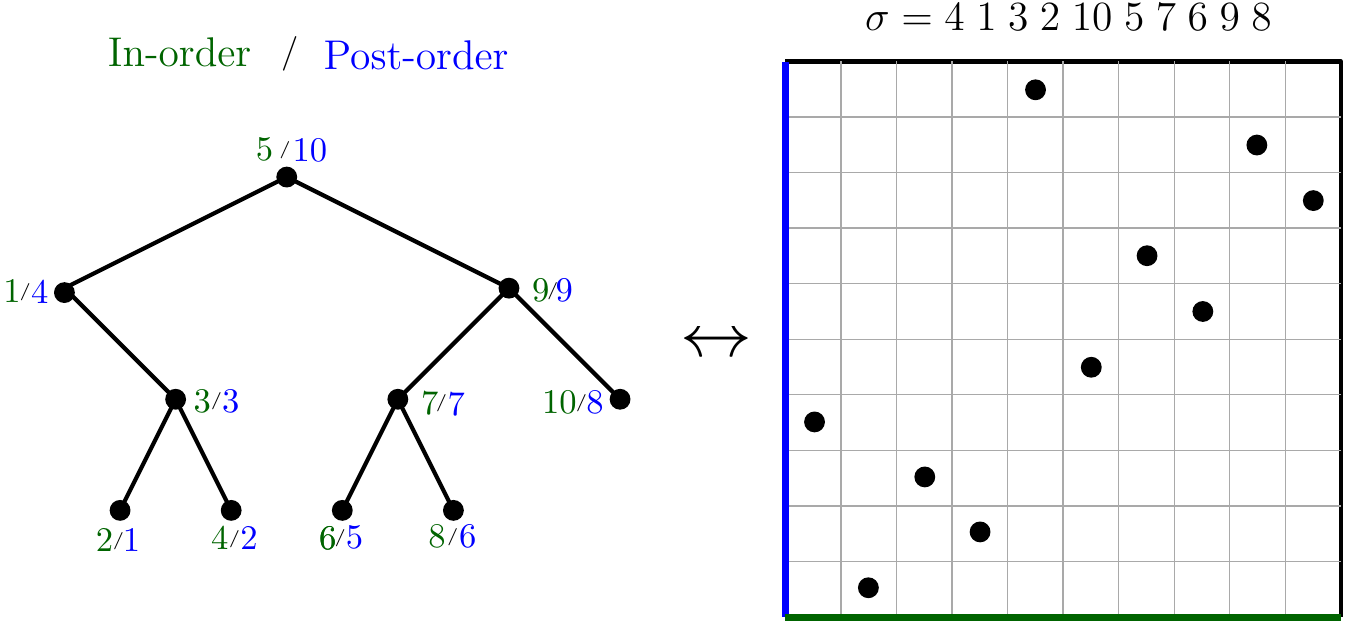}\\
		\caption{A binary tree and the corresponding 231-avoiding permutation.}\label{Av(231)bij}
\end{figure}

\subsection{321-avoiding permutations}
\label{premres_321}
It is well known that 321-avoiding permutations can be broken into two increasing subsequences, one weakly above the diagonal and one strictly below the diagonal (see for instance \cite[Section 4.2]{bona2016combinatorics}). Therefore, defining for every permutation $\sigma\in\Av_n(321)$ 
\begin{equation}
\begin{split}
&E^+=E^+(\sigma)\coloneqq\big\{i\in[n]|\sigma(i)\geq i\big\},\\
&E^-=E^-(\sigma)\coloneqq\big\{i\in[n]|\sigma(i)< i\big\},
\end{split}
\end{equation}
we have that $\text{pat}_{E^+}(\sigma)$ and $\text{pat}_{E^-}(\sigma)$ are increasing patterns.
Note that by convention the indices of fixed points are elements of $E^+.$

We now describe a slightly different, but equivalent (see \cref{rem_equiv}), version of a well-known bijection between 321-avoiding permutations of size $n$ and rooted plane trees with $n+1$ vertices (see for instance \cite{hoffman2016fixed}).

Recall that $\mathbb{T}_{n}$ denotes the set of rooted plane trees with $n$ vertices. Suppose $T\in\mathbb{T}_{n+1}$ has $k$ leaves and let $\ell(T)=(\ell_1,\dots,\ell_k)$ be the list of leaves of $T,$ listed in order of appearance in the pre-order traversal of $T,$ i.e.\ from left to right. Moreover let $s_i$ and $q_i$ be the labels given to the leaf $\ell_i$ respectively by the pre-order labeling of $T$ starting from $0$ and the post-order labeling of $T$ starting from 1. We set $S(T)\coloneqq\{s_i\}_{1\leq i\leq k}$ and $Q(T)\coloneqq\{q_i\}_{1\leq i\leq k}.$  An example is given in the left-hand side of \cref{321bij_1}.

We now construct the permutation $\sigma_{T}$ associated with the tree $T.$ We define $\sigma_T$ pointwise on $Q(T)=\{q_i\}_{1\leq i\leq k}$ by
$$\sigma_T(q_i)=s_i,\quad\text{for all}\quad1\leq i\leq k,$$ 
and then we extend $\sigma_T$ on $[n]\setminus Q(T)$ by assigning values of $[n]\setminus S(T)$ in increasing order. Clearly $\sigma_{T}$ is a $321$-avoiding permutation. An example is given in \cref{321bij_1}. 

We highlight (without proof, for more details see \cite{hoffman2016fixed}) that this construction implies $E^+(\sigma_T)=Q(T)$ and $\sigma(E^+(\sigma_T))=S(T).$ It can also be proved (see again \cite{hoffman2016fixed}) that this construction is invertible: For a permutation $\sigma\in\Av_n(321)$ there is a unique rooted plane tree $T_{\sigma}$ with $n+1$ vertices such that $Q(T_{\sigma})=E^+(\sigma)$ and $S(T_{\sigma})=\sigma(E^+(\sigma)).$

\begin{figure}[htbp]
		\centering
		\includegraphics[scale=0.73]{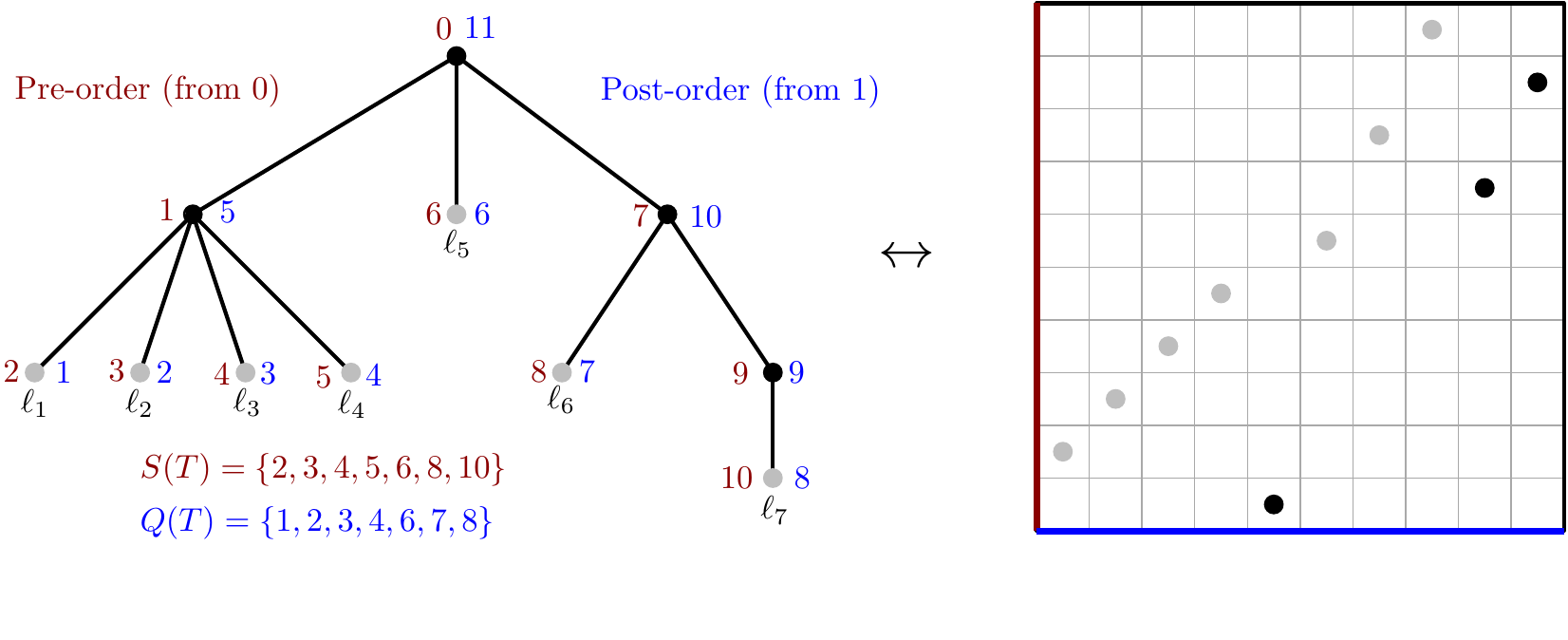}\\
		\caption{\textbf{Left}: A tree $T$ with 11 vertices and 7 leaves. We record in red on the left of each vertex the pre-order label and in blue on the right of each vertex the post-order label. \textbf{Right}: The associated 321-avoiding permutation of size 10. In the diagram, for every leaf in the tree, we draw a gray dot using the post-order label for the $x$-coordinate and the pre-order label for the $y$-coordinate. Then we complete the diagram with the unique increasing subsequence fitting in the empty rows and columns (shown with black dots on the picture).}\label{321bij_1}
\end{figure}

\begin{rem}
	\label{rem_equiv}
	At first sight, it seems that the construction presented here is different from that of \cite{hoffman2016fixed}; in \cite[Remark 5.8]{borga2020localperm} we explained that they are in fact the same.
\end{rem}

\begin{obs}
	\label{obs_treeperm}
	Let $\bm{\eta}\sim Geom(1/2)$, i.e.\ $\eta_k\coloneqq\P(\bm{\eta}=k)=2^{-k-1},$ for all $k\in\Z_{\geq 0}$. We recall that a uniform random rooted plane tree with $n$ vertices $\bm{T}_n$ has the same distribution as a Galton--Watson tree $\bm{T}^{\eta}$ with offspring distribution $\bm{\eta}$ conditioned on having $n$ vertices (see \cite[Example 9.1]{janson2012simply}), namely
	$\bm{T}_n\stackrel{(d)}{=}\big(\bm{T}^\eta\;\big|\;|\bm{T}^\eta|=n\big)$.
\end{obs}

Therefore, thanks to the bijection previously introduced, we have the following result.

\begin{prop}
	Let $\bm{\eta}\sim Geom(1/2)$ i.e.\ $\eta_k\coloneqq\P(\bm{\eta}=k)=2^{-k-1},$ for all $k\in\Z_{\geq 0}$, and set $\bm{T}^\eta_n\coloneqq\big(\bm{T}^\eta\;\big|\;|\bm{T}^\eta|=n\big)$. Then $\sigma_{\bm{T}^\eta_n}$ is a uniform 321-avoiding permutation of size $n$.
\end{prop}

\section{Substitution-closed classes}\label{sect:sub_close}

In the previous section we have seen two specific families of pattern-avoiding permutations encoded by trees. 
Here, we still consider pattern-avoiding permutations encoded by trees, but with two remarkable differences: First, the encoding trees have a richer structure, namely decorations on internal vertices; secondly, the number of permutation families treated is infinite.

\subsubsection*{Notation}

Let $\DDD$ be a set and $\size:\DDD \to \Z_{\geq 0}$ be a map
from $\DDD$ to the set of non-negative integers, associating to each object in $\DDD$ its size. 
We say $\DDD$ is an (unlabeled) combinatorial class,  if for any $n \in \Z_{\geq 0}$ the number $d_n$ of $n$-sized objects in $\DDD$ is finite. This allows to form the \emph{(ordinary) generating series}
$
\DDD(z) = \sum_{n \ge 0} d_n z^n.
$
Note that we try to use the same curvy letter $\mathcal{D}$ for the class and its generating series.
This should hopefully not lead to confusions.
Two combinatorial classes $\mathcal{D}_1, \mathcal{D}_2$ are considered \emph{isomorphic} if there is a size-preserving bijection between the two,
or equivalently if they have the same generating series.

Various standard operations are available for combinatorial classes. 
For example, whenever $\DDD$ has no objects of size $0$,
we can form the combinatorial class $\Seq(\DDD)$,
which is the collection of finite sequences of objects from $\DDD$. 
The size of such a sequence is defined to be the sum of sizes of its components. 
We may also consider the subclass $\Seq_{\ge 1}(\DDD) \subset \Seq(\DDD)$ of non-empty sequences. 
For a great introduction to combinatorial classes we refer the reader to the book \cite{flajolet2009analytic}.

\begin{defn} \label{defn:decorated_tree}
	Let $\DDD$ be a combinatorial class.
	A $\DDD$-\emph{decorated tree} is a rooted locally finite plane tree $T$,
	equipped with a function $\dec:\Vint(T) \to \DDD$ from the set of internal vertices of $T$ to $\DDD$
	such that the following holds:
	{\em for each $v$ in $\Vint(T)$, the out-degree of $v$ is exactly $\size(\dec(v))$.} 
\end{defn}

\subsection{Encoding permutations as forests of decorated trees}
\label{sec:bijections}

In this section, we first properly define substitution-closed classes and then we show that any substitution-closed class can be bijectively encoded as a forest of decorated trees. This goal is achieved in \cref{te:bijection}.

\subsubsection{Substitution decomposition and canonical trees}
\label{sec:DecoTrees}
We recall classical concepts related to permutations.
\begin{defn}
	Let $\theta=\theta(1)\cdots \theta(d)$ be a permutation of size $d$, and let $\nu^{(1)},\dots,\nu^{(d)}$ be $d$ other permutations. 
	The \emph{substitution} of $\nu^{(1)},\dots,\nu^{(d)}$ in $\theta$,
	denoted by $\gls*{sub_perm}$, 
	is the permutation of size $|\nu^{(1)}|+ \dots +|\nu^{(d)}|$ 
	obtained by replacing each $\theta(i)$ by a sequence of integers isomorphic to $\nu^{(i)}$ while keeping the relative order induced by $\theta$ between these subsequences.
\end{defn}
Examples of substitution (see \cref{fig:sum_and_skew} below) are conveniently presented representing permutations by their diagrams: 
the diagram of $\nu = \theta[\nu^{(1)},\dots,\nu^{(d)}]$ is obtained by inflating each point $\theta(i)$ of $\theta$ by a square containing the diagram of $\nu^{(i)}$. 
Note that each $\nu^{(i)}$ then corresponds to a \emph{block} of $\nu$, a block being defined as an interval of $[|\nu|]$ which is mapped to an interval by $\nu$. 

Throughout this manuscript, the increasing permutation $12 \ldots d$ will be denoted by $\oplus_d$,
or even $\oplus$ when its size $d$ can be recovered from the context:
this is the case in an inflation $\oplus[\nu^{(1)},\dots,\nu^{(d)}]$ where the size of $\oplus$ is the number $d$ of permutations inside the brackets.
Similarly, we denote the decreasing permutation $d \ldots 21$ by $\ominus_d$,
or $\ominus$ when there is no ambiguity.

\begin{figure}[htbp]
	\centering
		\includegraphics[width=7cm]{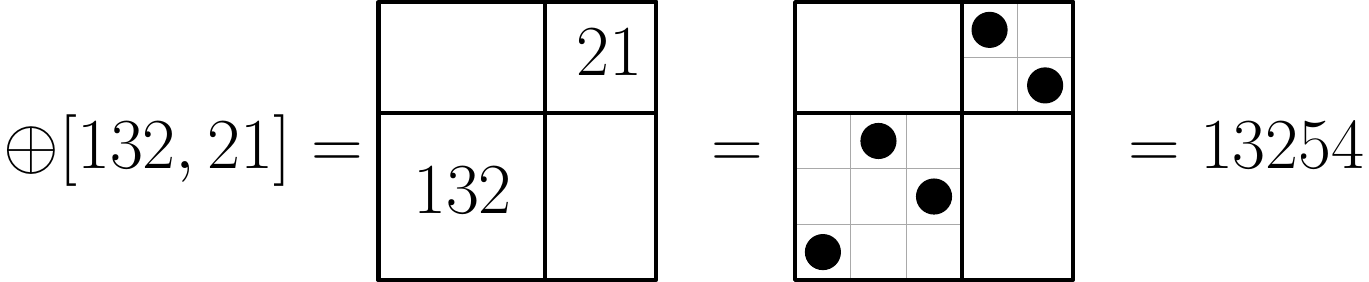} \qquad 
		\includegraphics[width=7cm]{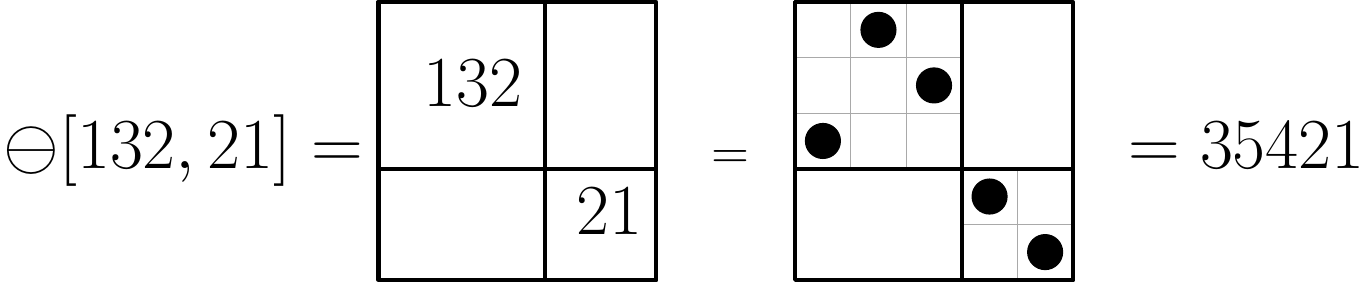}\medskip\\
		\includegraphics[width=95mm]{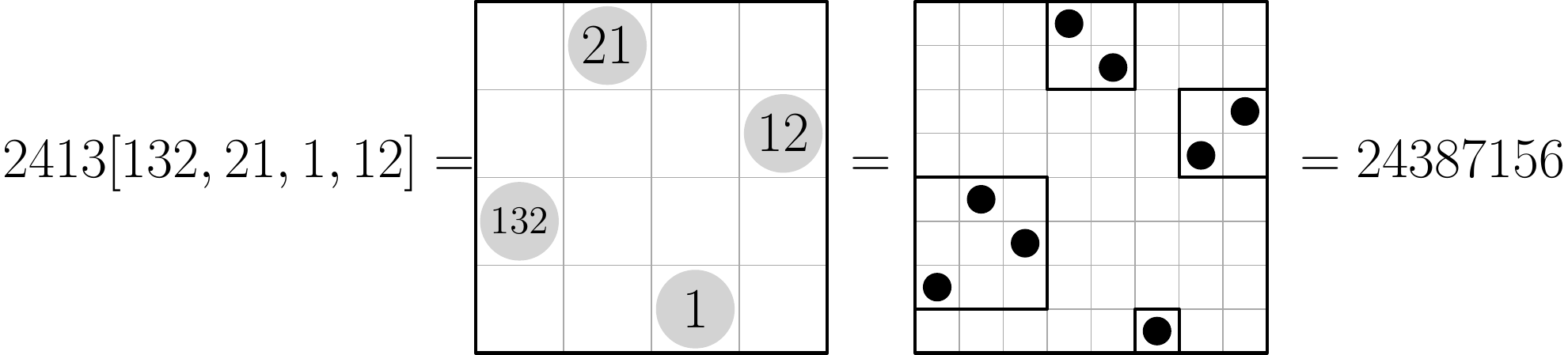}
	\caption{Substitution of permutations.\label{fig:sum_and_skew}}
\end{figure}

Permutations can be decomposed in a canonical way using recursively the substitution operation.
To explain this, we first need to define several notions of indecomposable objects.
\begin{defn}
	A permutation $\nu$ is \emph{$\oplus$-indecomposable} (resp.\ $\ominus$-indecomposable) 
	if it cannot be written as $\oplus[\nu^{(1)},\nu^{(2)}]$ (resp.\ $\ominus[\nu^{(1)},\nu^{(2)}]$).
	
	A permutation of size $n > 2$ is \emph{simple} if it contains no nontrivial block, i.e.\ if it does not map any \emph{nontrivial} interval 
	(i.e.\ a range in $[n]$ containing at least two and at most $n-1$ elements) onto an interval.
\end{defn}
For example, $451326$ is not simple as it maps the interval $[3,5]$ onto the interval $[1,3]$. 
The smallest simple permutations are $2413$ and $3142$ (there is no simple permutation of size $3$). 
We denote by $\gls*{s_all}$ the set of simple permutations.
\medskip

\begin{rem}
	Usually in the literature, the definition of a simple permutation requires $n \geq 2$ instead of $n > 2$, so that $12$ and $21$ are considered to be simple.
	However, for decomposition trees, $12$ and $21$ do not play the same role as the other simple permutations, that is why we do not consider them to be simple.
\end{rem}

\begin{thm}\label{Th:AlbertAtkinson}
	Every permutation $\nu$ of size $n\geq 2$ can be uniquely decomposed as either:
	\begin{itemize}
		\item $\alpha[\nu^{(1)},\dots,\nu^{(d)}]$, where $\alpha$ is a simple permutation (of size $d\geq 4$),
		\item $\oplus[\nu^{(1)},\dots,\nu^{(d)}]$, where $d\geq 2$ and $\nu^{(1)},\dots,\nu^{(d)}$ are $\oplus$-indecomposable,
		\item $\ominus[\nu^{(1)},\dots,\nu^{(d)}]$, where $d\geq 2$ and $\nu^{(1)},\dots,\nu^{(d)}$ are $\ominus$-indecomposable.
	\end{itemize}
\end{thm}

\begin{rem}
	The theorem above is essentially Proposition 2 in \cite{albert2005simple}, presented with a slightly different point of view.
	The decomposition according to \cref{Th:AlbertAtkinson} is obtained from the one of \cite[Proposition 2]{albert2005simple} 
	by merging maximal sequences of nested substitutions in $12$ (resp.\ $21$) into a substitution in $\oplus$ (resp.\ $\ominus$).
	For example, the second item above for $d=4$ corresponds to $12[\nu^{(1)}, 12[\nu^{(2)},12[\nu^{(3)},\nu^{(4)}]]]$ with the notation of \cite{albert2005simple}. 
	With this obvious rewriting, the statements of \cite[Proposition 2]{albert2005simple} and of \cref{Th:AlbertAtkinson} are trivially equivalent.  
\end{rem}

This decomposition theorem can be applied recursively
inside the permutations $\nu^{(i)}$ appearing in the items above, 
until we reach permutations of size $1$. 
Doing so, a permutation $\nu$ can be naturally encoded by a rooted decorated plane tree $\gls*{can_tree}$ as follows (the notation $\CanTree(\nu)$ stands for \emph{canonical tree}):\label{def:cantree} 

\begin{itemize}
	\item If $\nu=1$ is the unique permutation of size $1$,
	then $\CanTree(\nu)$ is reduced to a single leaf.
	\item If $\nu=\beta[\nu^{(1)},\dots,\nu^{(d)}]$, where $\beta$ is either a simple permutation
	or the increasing (resp.\ decreasing) permutation (denoted by $\oplus$, resp.\ $\ominus$),
	then $\CanTree(\nu)$ has a root of degree $d$
	decorated by $\beta$ and the subtrees attached to the root from left to right
	are $\CanTree(\nu^{(1)})$, \ldots, $\CanTree(\nu^{(d)})$.
\end{itemize}

From the theorem above, the decomposition $\nu=\beta[\nu^{(1)},\dots,\nu^{(d)}]$ exists
and is unique if $|\nu|~\ge~2$. Moreover, $\nu^{(1)},\dots,\nu^{(d)}$ have size smaller than $\nu$
so that this recursive procedure always terminates and its result is unambiguously defined.
The map $\CanTree$ is therefore well-defined.
An example of this construction is shown on \cref{fig:ExCanTree}.

\begin{figure}[htbp]
	\begin{minipage}[c]{0.67\textwidth}
		\centering
		\includegraphics[height=4.5cm]{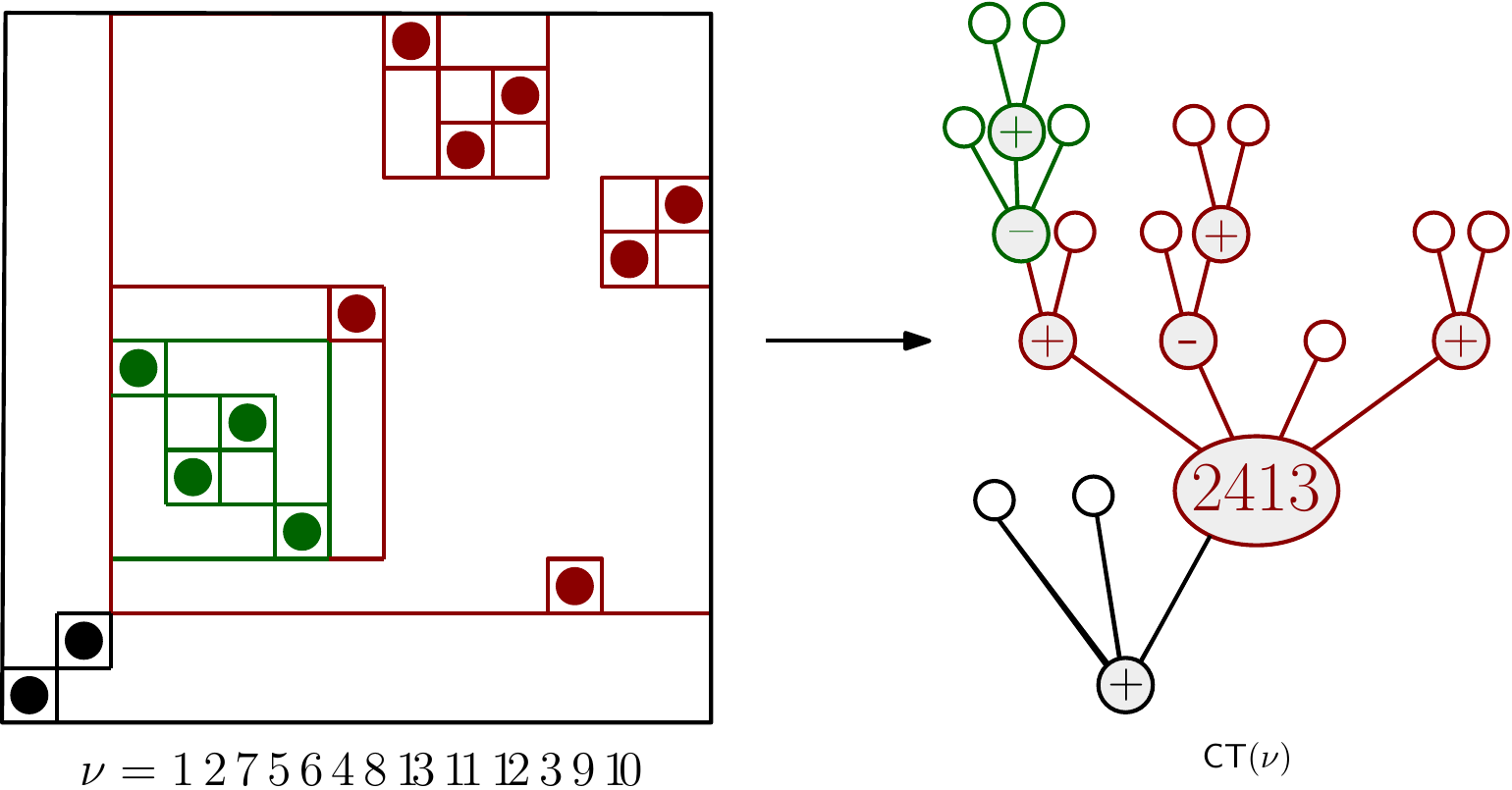}
	\end{minipage}
	\begin{minipage}[c]{0.32\textwidth}
		\caption{A permutation $\nu$ and its decomposition tree $\CanTree(\nu)$.
			To help the reader understand the construction, we have colored accordingly
			some blocks of $\nu$ and some subtrees of $\CanTree(\nu)$.\label{fig:ExCanTree}}
	\end{minipage}
\end{figure}

Since the decoration of each vertex record the permutation $\beta$
in which we substitute, it is clear that $\CanTree$ is injective. 
Moreover, its inverse (once restricted to $\CanTree(\mathcal{S})$) 
is immediate to describe, simply by performing the iterated substitutions recorded in the tree. 
We are just left with identifying the image set of $\CanTree$.
Recall that $\Sall$ denotes the set of all simple permutations, and let $\gls*{mon_perm}$ be the set of all monotone (increasing or decreasing) permutations of size at least $2$. Denote $\gls*{simple_mon_perm}$.
\begin{defn}\label{defintro:CanonicalTree}
	A \emph{canonical tree}  is an $\widehat{\Sall}$-decorated tree
	in which we cannot find two adjacent vertices both decorated
	with increasing permutations (i.e.\ with $\oplus$) or both decorated with decreasing permutations (i.e.\ with $\ominus$).
\end{defn}
Canonical trees are also known in the literature under several names:
decomposition trees, substitution trees,\ldots
We choose the term {\em canonical}
to be consistent with \cite{bassino2017universal}.
The following is an easy consequence of \cref{Th:AlbertAtkinson}.
\begin{prop}\label{prop:Can_tree}
	The map $\CanTree$ defines a size-preserving bijection from the set of
	all permutations to the set of all canonical trees,
	the size of a tree being its number of leaves.
\end{prop}

\begin{rem}
	\label{rk:CT-1}
	We note that the inverse map $\CanTree^{-1}$, which builds a permutation from a canonical tree performing nested substitutions, 
	can obviously be extended to all $\widehat{\Sall}$-decorated trees, regardless of whether they contain $\oplus - \oplus$ or $\ominus - \ominus$ edges. 
	However, $\CanTree^{-1}$ is no longer injective on this larger class of ``non-canonical'' decomposition trees.
\end{rem}

We will be interested in the restriction of $\CanTree$
to some permutation class $\mathcal C$.
The following condition ensures that its image $\CanTree(\mathcal C)$
has a nice description.

\begin{defn}%
	\label{def:substclosed}
	A permutation class $\mathcal{C}$ is \emph{substitution-closed}
	if  for every $\theta, \nu^{(1)},\dots,\nu^{(d)}$ in $\mathcal{C}$ it holds that
	$\theta[\nu^{(1)},\dots,\nu^{(d)}] \in \mathcal{C}$.
\end{defn}

Let $\mathcal{C}$ be a substitution-closed permutation class, and assume\footnote{Otherwise, 
	$\mathcal{C} \subseteq \{12\ldots k : k \geq 1 \}$ or 
	$\mathcal{C} \subseteq \{k\ldots 21 : k \geq 1 \}$ and these cases are trivial.} that $12,21 \in \mathcal{C}$. 
Denote by $\gls*{simple_sub}$ the set of simple permutations in $\mathcal{C}$ (note that $\mathfrak{S}$ depends on $\mathcal{C}$; since $\mathcal{C}$ will be always fixed, we prefer to not highlight this dependence in the notation).

\begin{prop}
	\label{prop:treesOfSubsClosedClasses} 
	The set of canonical trees encoding permutations of $\mathcal{C}$ 
	is the set of canonical trees with decorations in
	$\widehat{\mathfrak{S}}:=\mathfrak{S} \cup \MMM$. 
\end{prop}

If necessary, details can be found in~\cite[Lemma 11]{albert2005simple}.

\subsubsection{Packed decomposition trees}
\label{sect:Packed_decomposition_trees}

{\em 
	From now until the end of this section we assume that $\mathcal{C}$ is a \emph{proper} substitution-closed class, that is we exclude the case where $\mathcal{C}$ is the class of all permutations. To avoid trivial cases, we furthermore assume that $12, 21 \in \mathcal{C}$. We denote with  $\mathfrak{S}$ the set of simple permutations in $\mathcal{C}$.}

\medskip

The assumption that we are working in $\mathcal{C}$ rather than in the set of all permutations is however often tacit: 
for example, we simply refer to canonical trees instead of canonical trees with decorations in $\widehat{\mathfrak{S}}=\mathfrak{S} \cup \MMM$.  We let $\gls*{set_can_trees}$ denote the collection of  canonical trees with decorations in $\widehat{\mathfrak{S}}$, and $\gls*{set_can_trees_not}$ the subset of canonical trees with a root that is \emph{not} decorated $\oplus$.

\medskip

In this section we introduce a new family of trees called ``packed trees''
and describe a bijection between the collection $\mathcal{T}_{\nonp} \subset \mathcal{T}$ and packed trees.
Packed trees are decorated trees,
whose decorations are themselves trees.
Let us define these decorations, that we call {\em gadgets}.

\begin{defn}\label{def:S_plus_decoration}
	An $\mathfrak{S}$-gadget is an $\widehat{\mathfrak{S}}$-decorated tree of height at most $2$ such that: 
	\begin{itemize}
		\item The root is an internal vertex decorated by a simple permutation;
		\item The children of the root are either leaves or decorated by an increasing permutation.  
	\end{itemize}	
	The size of a gadget is its number of leaves. 
\end{defn}
We denote with $\gls*{gadget}$ the set of $\mathfrak{S}$-gadgets. An example of size 7 is shown on \cref{fig:S/+-decoration}.

\begin{figure}[htbp]
	\begin{minipage}[c]{0.5\textwidth}
		\centering
		\includegraphics[height=2.5cm]{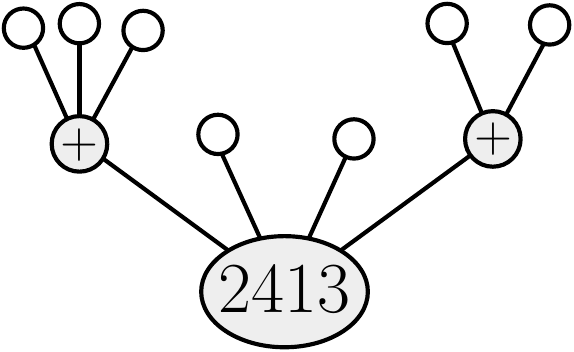}
	\end{minipage}
	\begin{minipage}[c]{0.5\textwidth}
		\caption{An $\mathfrak{S}$-gadget.\label{fig:S/+-decoration}}
	\end{minipage}
\end{figure}

Finally, let $\gls*{gadget_and_mon} = \GGG(\mathfrak{S})\cup\{\circledast_k, k \ge 2\}$,
where for each integer $k\ge 2$ the object $\circledast_k$ has size~$k$. 
To shorten notation,  $\widehat{\GGG(\mathfrak{S})}$ is sometimes denoted $\mathcal{Q}$ in the following. 

\begin{defn}
	An \emph{$\mathfrak{S}$-packed tree} is a $\widehat{\GGG(\mathfrak{S})}$-decorated tree,
	its size being its number of leaves.
\end{defn}

\begin{rem}
	We will often refer to $\mathfrak{S}$-packed trees simply as packed trees since in our analysis, the substitution-closed class $\mathcal{C}$ and 
	its set of simple permutations $\mathfrak{S}$ will be fixed.
\end{rem}

An example of packed tree is shown on the right-hand side of \cref{fig:packed_tree_bij}.

\begin{rem}
	Note that in \cref{fig:packed_tree_bij} the subscript $k$ is not reported in the vertices decorated by an element in $\{\circledast_k, k \ge 2\}.$ 
	Indeed, it can be easily recovered by counting the number of children of the vertex. 
\end{rem}

We now describe a bijection between canonical trees with a root that is not decorated with $\oplus$ 
and packed trees. Given a tree $T\in\mathcal{T}_{\nonp}$ the corresponding packed tree $\Pack(T)$ is obtained modifying $T$ as follows. 
\begin{itemize}
	\item For each internal vertex $v$ of $T$ decorated by a simple permutation, we build an $\mathfrak{S}$-gadget $G_v$
	whose internal vertices are $v$ and the $\oplus$-children of $v$, 
	the parent-child relation in $G_v$ and the left-to-right order between children are inherited from the ones in $T$, 
	and we add leaves so that the out-degree of each internal vertex is the same in $G_v$ as in $T$.
	Then, in $\Pack(T)$, we merge $v$ and the $\oplus$-children of $v$ into a single vertex decorated by $G_v$. 
	\item The remaining vertices of $T$, decorated by $\ominus_k$ or $\oplus_k$,
	are decorated by $\circledast_k$ instead.
\end{itemize}
An example is given on \cref{fig:packed_tree_bij}.

\begin{figure}[ht]
	\centering
		\includegraphics[height=5.8cm]{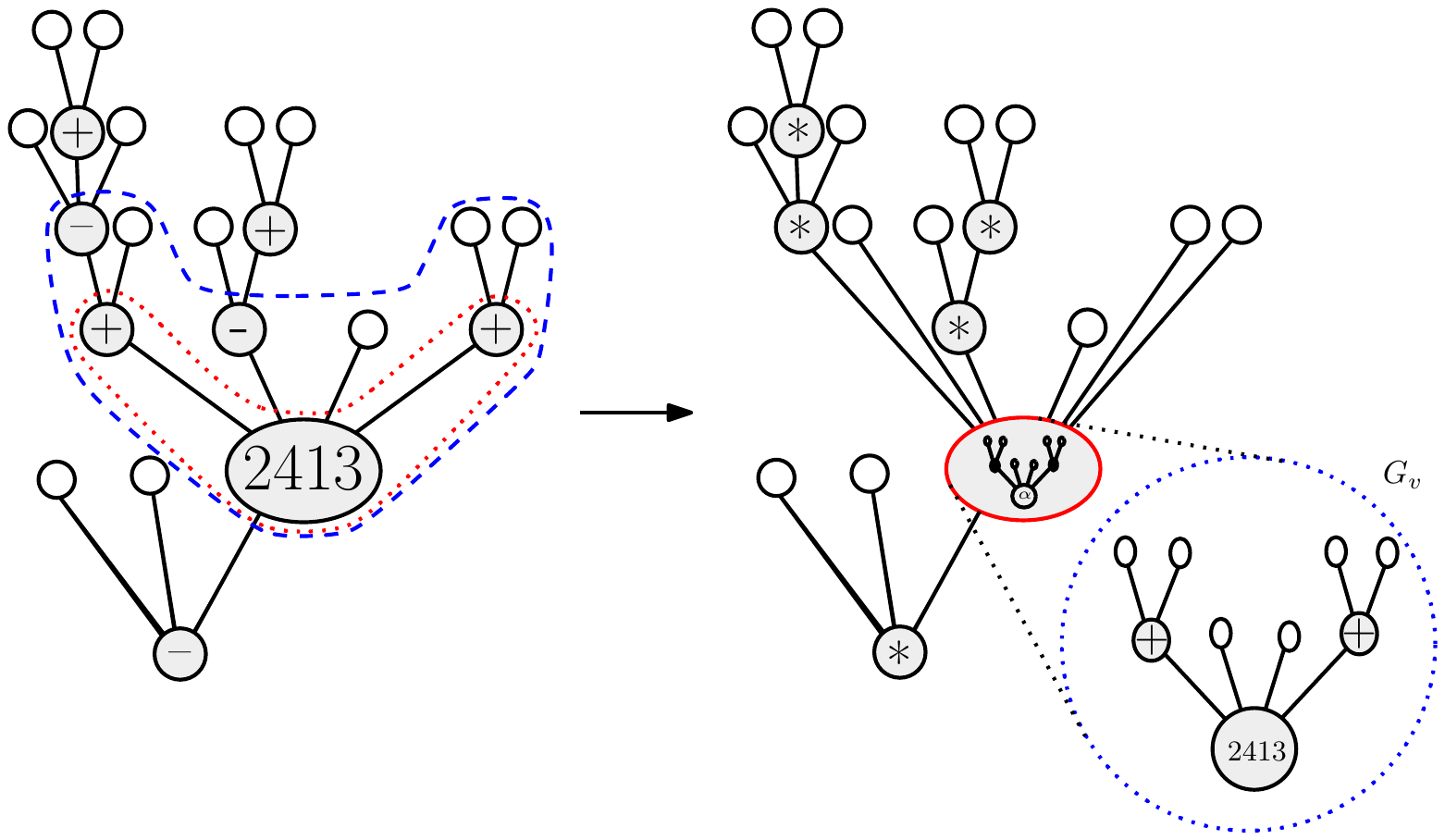}
	\caption{A canonical tree $T$ with the corresponding packed tree $\Pack(T)$. The red dotted line highlights the vertices in $T$ that are merged in the red vertex in $\Pack(T).$ The blue dashed line records the subtree in $T$ that determines the decoration $G_v$ in the tree $\Pack(T)$.}
	\label{fig:packed_tree_bij}
\end{figure}
\begin{prop}[{\cite[Proposition 2.16]{borga2020decorated}}]\label{prop:pack_tree}
	The map $\Pack$ defines a size-preserving bijection from the set $\mathcal{T}_{\nonp}$
	of canonical trees with a root that is not decorated $\oplus$ to the set of $\mathfrak{S}$-packed trees.
\end{prop}

\begin{rem}
	\label{rk:Leaves_Elements}
	If $T$ (or $P$) is a tree  with $n$ leaves,
	we can label its leaves with number from $1$ to $n$
	using a depth-first traversal of the tree from left to right.
	Then the $i$-th leaf
	of the canonical or packed tree associated with a permutation
	$\nu$ corresponds to the $i$-th element in the one-line notation
	of $\nu$.
	We will use this identification between leaves and elements of the permutations
	later.
\end{rem}

\subsubsection{Permutations as forests of decorated trees}
\label{sec:finalqwerty}

Summing up the results obtained in the previous sections (in particular in Propositions \ref{prop:Can_tree}, \ref{prop:treesOfSubsClosedClasses} and \ref{prop:pack_tree}),
we obtain a bijective encoding of $\oplus$-indecomposable permutations in $\mathcal{C}$:

\begin{lem}
	\label{le:bij_perm_tree}
	The map 
	\[
	\Pack\circ\CanTree: \mathcal{C}_{\nonp} \to \cP
	\]
	is a size-preserving bijection from the set $ \mathcal{C}_{\nonp}$
	of all $\oplus$-indecomposable permutations in $\mathcal{C}$
	to the set $\cP$ of all $\mathfrak{S}$-packed trees.
\end{lem}

By \cref{Th:AlbertAtkinson}, any $\oplus$-decomposable permutation corresponds uniquely to a sequence of at least two $\oplus$-indecomposable permutations.
Hence any permutation corresponds bijectively to a non-empty sequence of $\oplus$-indecomposable permutations. 
If we apply the bijection $\Pack\circ\CanTree$  to each we obtain a plane forest of packed trees. 
That is, it is an element of the collection $\Seq_{\ge 1}(\cP)$ of non-empty ordered sequences 
of $\widehat{\mathcal{G}(\mathfrak{S})}$-decorated trees. 
We define the size of such a forest to be the total number of leaves.
The function that maps a permutation of $\mathcal{C}$ to the corresponding forest of packed trees is denoted by $\DF$
($\DF$ stands for decorated forest).
Summing up:

\begin{thm}
	\label{te:bijection}
	The function 
	\begin{equation}
	\DF: \mathcal{C} \to \Seq_{\ge 1}(\cP)
	\end{equation}
	is a size-preserving bijection between the substitution-closed class $\mathcal{C}$ 
	and the set of forests of packed trees.
\end{thm}

\subsection{Reading patterns in trees}
\label{ssec:patterns_subtrees}
Let us consider a permutation $\nu$ in $\mathcal C_{\nonp}$ and the associated canonical and packed trees
$T=\CanTree(\nu)$ and $P=\Pack(T)$.
Let $I$ be a subset of $[n]$. 
Using \cref{rk:Leaves_Elements},
$I$ can be seen as a subset of the leaves of $T$ (or $P$).
The purpose of this section is to explain how to read out the pattern $\pi = \pat_I(\nu)$ on the trees $T$ or $P$. This will be a fundamental tool to prove later B--S and permuton convergence for uniform permutations in substitution-closed classes.

Let us first note that a pattern $\pi = \pi(1) \dots \pi(k)$ is entirely determined when we know,
for each $i_1<i_2$, whether $\pi(i_1) \pi(i_2)$ forms an \emph{inversion} 
(i.e.\ an occurrence of the pattern $21$) or a non-inversion (occurrence of $12$). 
Therefore, to read patterns on $T$ (or $P$),
we should explain how to determine,
for any two leaves $\ell_1$ and $\ell_2$ of $I$, whether
the corresponding elements of $\nu$ form an inversion or not
(in the sequel, we will simply say that $\ell_1$ and $\ell_2$ form an inversion,
and not refer anymore to the corresponding elements of $\nu$). 
\medskip

Looking at $T$, this is rather simple.
We consider the closest common ancestor of $\ell_1$ and $\ell_2$, call it $v$.
By definition, $\ell_1$ and $\ell_2$ are descendants of different children of $v$, 
say the $i_1$-th and $i_2$-th.
Then the following holds: $\ell_1$ and $\ell_2$ form an inversion in $\nu$
if and only if $i_1$ and $i_2$ form an inversion in the decoration $\beta$ of $v$.
\medskip

Let us now look at $P$. We consider the closest common ancestor $u\in P$ 
of $\ell_1$ and $\ell_2$ and as before,
we assume that $\ell_1$ and $\ell_2$ are descendants of the $i_1$-th and $i_2$-th children
of $u$. 
Note that, in the packing bijection, the vertex $u$
corresponds to $v$ (the common ancestor of $\ell_1$ and $\ell_2$ in $T$)
potentially merged with other vertices.

Consider first the case that $u$ is decorated by an $\mathfrak{S}$-gadget $G$.
Then $G$ contains the information of the decoration of all vertices merged into $u$,
including $v$.
Therefore, whether $\ell_1$ and $\ell_2$ form an inversion in $\nu$
can be determined by looking at the $i_1$-th and $i_2$-th leaves 
of the gadget $G$ (see \cref{ex:reconstruction_pattern} and \cref{fig:reconstruction_pattern}).

If on the contrary $u$ is not decorated by an $\mathfrak{S}$-gadget but by a $\circledast$, 
we need to determine whether $v$ is decorated with $\oplus$ (implying that $\ell_1$ and $\ell_2$ form a non-inversion) or $\ominus$ (resp., an inversion). 

Assume first that there is a closest ancestor $u'$ of $u$ 
that is decorated with an $\mathfrak{S}$-gadget. 
In this case, we claim that
$v$ is decorated by $\ominus$ if $d(u,u')$ is odd, and it is decorated by $\oplus$ if $d(u,u')$ is even.
Indeed, decorations $\oplus$ and $\ominus$ alternate,
and, by construction of the packing bijection,
the children of an $\mathfrak{S}$-gadget are decorated by $\ominus$.

It remains to analyse the case where $u$ is decorated by $\circledast$, 
as well as all vertices on the path from $u$ to the root $r$ of $P$. 
By construction, this implies that the root of $T$ is decorated by $\ominus$ (recall that $T\in\mathcal{T}_{\nonp}$). 
So, using again the alternation of $\oplus$ and $\ominus$ in $T$, 
the decoration of $v\in T$ is $\ominus$ if $d(u,r)$ is even, 
and $\oplus$ if $d(u,r)$ is odd.
\medskip

We note in particular that the pattern induced by a set $I$ of leaves in $P$
is determined by any fringe subtree containing all leaves of $I$
and {\em rooted at any vertex decorated with an $\mathfrak{S}$-gadget}.

\begin{exmp}\label{ex:reconstruction_pattern}
	Let  $\nu = 13 \; 12 \; 5\; 3\; 4\; 2\; 6\; 11\; 9\; 10\; 1\; 7\; 8$ be a permutation in $\mathcal C_{\nonp}$ with associated canonical and packed trees $T=\CanTree(\nu)$ and $P=\Pack(T)$ shown in Fig.~\ref{fig:reconstruction_pattern}. 
	We explain in the following example how to read out in $P$ the pattern induced by the leaves $\ell_1$, $\ell_2$ and $\ell_3$. 
	
	The closest common ancestor $u\in P$ of $\ell_1$ and $\ell_2$ is decorated with a $\circledast$, 
	which is at distance $1$ from its closest ancestor decorated with an $\mathfrak{S}$-gadget.
	We can conclude that the leaves $\ell_1$ and $\ell_2$ induce an inversion
	(the closest ancestor $v$ of $\ell_1$ and $\ell_2$ in $T$ carries a $\ominus$ decoration).
	
	Now consider $\ell_1$ and $\ell_3$. Their closest common ancestor $u'$ in $P$
	is decorated with an $\mathfrak{S}$-gadget. 
	Note $\ell_1$ and $\ell_3$ are descendants of the first and fifth
	children of this $\mathfrak{S}$-gadget; the corresponding leaves of the $\mathfrak{S}$-gadget
	have the vertex decorated by $2413$ as common ancestor and are attached to the branches
	corresponding to $2$ and $3$.
	We deduce that $\ell_1$ and $\ell_3$ do not form an inversion in $\nu$.
	Similarly, $\ell_2$ and $\ell_3$ do not form an inversion either in $\nu$.
	
	Putting all together, the pattern induced by $\ell_1$, $\ell_2$ and $\ell_3$ is $213$.
	Let us check that it is indeed the case, by reading this pattern on the permutation.
	These three leaves correspond to the 4th, 6th and 12th elements of the permutation respectively,
	which have values 3, 2 and 7. The induced pattern is indeed $213$.
\end{exmp}

\begin{figure}[htbp]
	\begin{minipage}[c]{0.64\textwidth}
		\centering
		\includegraphics[height=5.3cm]{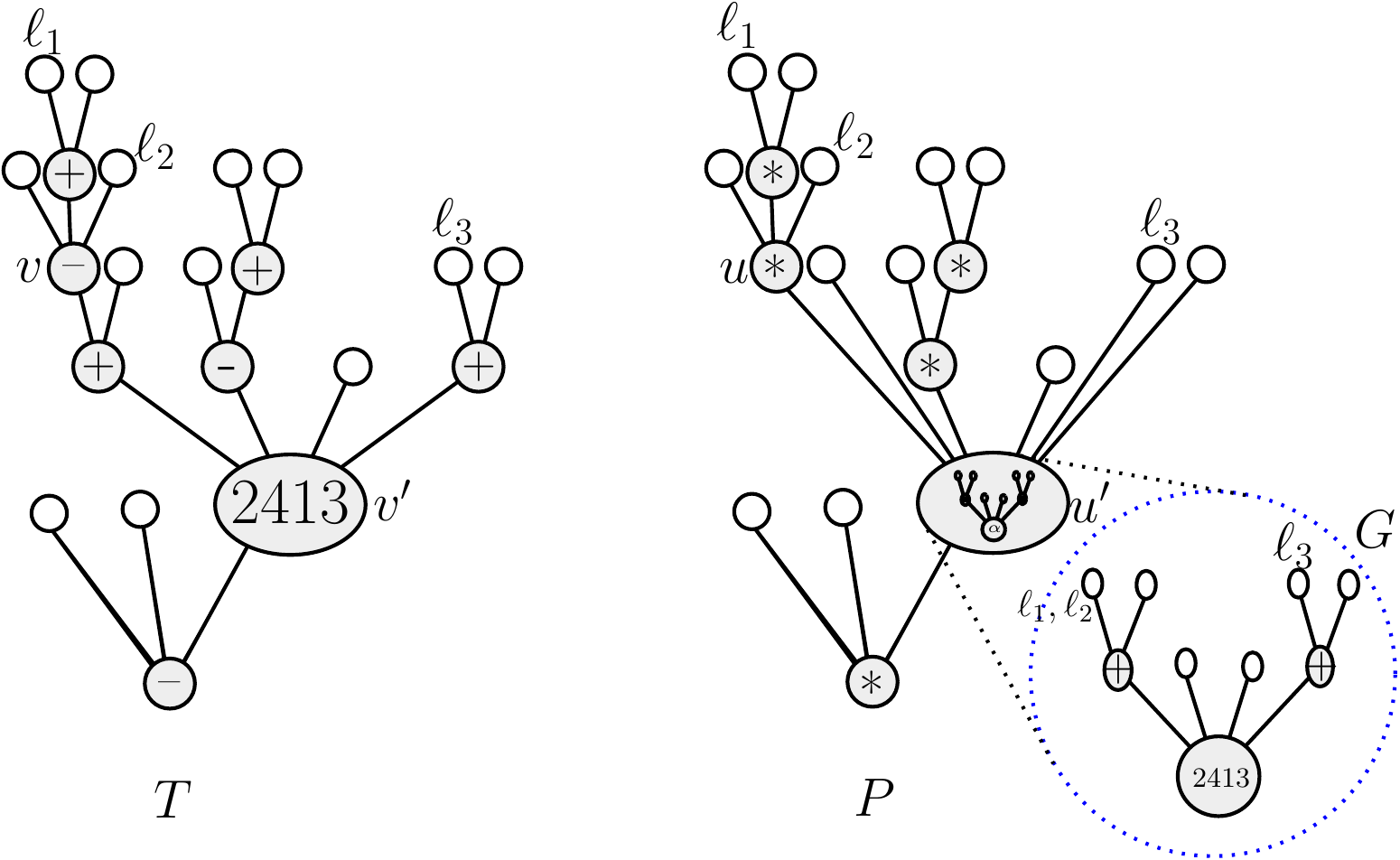}
	\end{minipage}
	\begin{minipage}[c]{0.35\textwidth}
		\caption{Reading patterns from trees -- see \cref{ex:reconstruction_pattern}.\label{fig:reconstruction_pattern}}
	\end{minipage}
\end{figure}

\subsection{Random permutations as conditioned Galton--Watson trees}
\label{sec:indecomposable_GW}

Theorem~\ref{te:bijection} allows us to see a uniform random  permutation $\bm{\nu}_n$ of size $n$ in the substitution-closed permutation class $\mathcal{C}$ as a uniform random forest of packed trees with $n$ leaves.\label{def:bmnu}
In \cref{subsec:Large_Indec_Component} we recall some results from \cite{borga2020decorated} to show that a giant component with size $n - O_p(1)$ emerges. This goal is achieved in Proposition~\ref{prop:giant_comp_perm}. 
This reduces the study of $\bm{\nu}_n$ to that of a uniform random packed tree with $n$ vertices. Finally, in \cref{subsec:PackedTrees_GW,subsec:PackedTrees_GW2}, we explain how to sample a uniform random packed tree with $n$ leaves
as a randomly decorated  Galton--Watson tree conditioned on having $n$ leaves.

\subsubsection{A giant $\oplus$-indecomposable component}

\label{subsec:Large_Indec_Component}

Let $\nu$ be a permutation in the proper substitution-closed class $\mathcal C$.
From \cref{Th:AlbertAtkinson}, we know that 
\begin{itemize}
	\item either $\nu$ is $\oplus$-indecomposable,
	\item or $\nu$ can be uniquely written as $\nu=\oplus[\nu^{(1)},\ldots,\nu^{(d)}]$, 
	where $d \ge 2$ and $\nu^{(1)},\ldots,\nu^{(d)} \in \mathcal{C}_{\nonp}$, the set of $\oplus$-indecomposable permutations of $\mathcal{C}$.
\end{itemize}
In the first case, we set $d=1$ and $\nu^{(1)}=\nu$ for convenience. Recall that \cref{le:bij_perm_tree} allows us to identify the classes $\mathcal{C}_{\nonp}$ and $\cP$. Using Gibbs partition results~\cite[Thm. 3.1]{stufler2016gibbs} one can obtain the following result:
\begin{prop}[{\cite[Proposition 3.2]{borga2020decorated}}]\label{prop:giant_comp_perm}
	Let $\bm \nu_n$ be a uniform permutation of size $n$ in $\mathcal C$ and define $\bm d$, $\bm{\nu^{(1)}}$, \ldots, $\bm{\nu^{(d)}}$ as above.
	Let $\bm{m}$ be the smallest  index such that $|\bm{\nu^{(m)}}|=\max(|\bm{\nu^{(1)}}|,\ldots, |\bm{\nu^{(d)}}|)$. Then $\bm{\nu^{(m)}}$ has size $n - O_p(1)$, and conditionally on its size,
	$\bm{\nu^{(m)}}$ is uniformly distributed among all $|\bm{\nu^{(m)}}|$-sized  $\oplus$-indecomposable permutations in $\mathcal{C}$.
\end{prop}

\begin{rem}
	We excluded the case of uniform unrestricted $n$-sized permutations. In this case, it is well-known that the permutation is with high probability $\oplus$-indecomposable. 
	This follows for example from~\cite[Thm 3.4]{PerfectSorting}.
\end{rem}

\subsubsection{From permutations to simply generated trees}
\label{subsec:PackedTrees_GW}

Proposition~\ref{prop:giant_comp_perm} and Lemma~\ref{le:bij_perm_tree} 
reduce the study of the proper substitution-closed class $\mathcal{C}$ to the study of the class $\cP$ of packed trees.
In this section, we explain how a random tree in $\cP$ can be seen as a random simply generated tree with random decorations. 

\medskip

We can describe a packed tree $\PackedTree$ as a pair $(T,\lambda_T)$
where $T$ is a rooted plane tree and $\lambda_T$ is a map
from the internal vertices of $T$ 
to the set $\mathcal{Q} = \widehat{\GGG(\mathfrak{S})}$ which records the decorations of the vertices. 
In order to sample a uniform packed tree with $n$ leaves, we first simulate a random rooted plane tree $\bm{T}_n$ 
and then a random decoration map $\bm{\lambda}_{\bm{T}_n}$ as follows.

Define the weight-sequence $\myvec{q} = (q_k)_{k \ge 0}$,
where, for $k \ge 2$, $q_k$ denotes the $k$-th coefficient of the generating series $\mathcal Q(z)=\widehat{\GGG(\mathfrak{S})}(z)$,
while we set $q_0=1$ and $q_1=0$.
We consider the simply generated tree $\bm{T}_n$ (with $n$ leaves) associated 
with weight-sequence $\myvec{q},$ i.e.\ by definition, $\bm{T}_n$ is a random rooted plane tree such that
\begin{equation}\label{eq:simply_gen_distrib}
\P(\bm{T}_n=T)=\frac{\prod_{v\in T} q_{d^+(v)}}{Z_n}=\frac{\prod_{v\in \Vint(T)} q_{d^+(v)}}{Z_n},
\end{equation}
for all rooted plane trees $T$ with $n$ leaves (we recall that $\Vint(T)$
denotes the set of internal vertices of $T$). 
Here, $Z_n$ is the \emph{partition function} given by 
$
Z_n=\sum_{T}\prod_{v\in T} q_{d^+(v)},
$
where the sum runs over all rooted plane trees with $n$ leaves.
For a general introduction on simply generated trees see \cite[Section 2.3]{janson2012simply}.

Then, given a rooted plane tree $T$, let $\bm{\lambda}_T$ be the random map such that for all internal vertices $v$ of~$T$,
\begin{align}\label{eq:decoration_distrib}
\P(\bm{\lambda}_T(v)=Q)=\frac{1}{q_{d^+_T(v)}}, \quad\text{for all $Q \in \mathcal{Q}$ with $|Q| = d^+_T(v)$},
\end{align}
independently of all other choices. Namely, the decoration of each internal vertex $v$ of $T$ gets drawn uniformly at random among all $d^+_T(v)$-sized decorations in $\mathcal{Q}$, independently of all the other decorations. 

\begin{lem}[{\cite[Lemma 3.4]{borga2020decorated}}]
	\label{lem:unif_packed_tree} \label{def:Pn}
	The random packed tree $\bm{P}_n = (\bm{T}_n,\bm{\lambda}_{\bm{T}_n})$ is uniform among all the packed trees with $n$ leaves.
\end{lem}

\subsubsection{Random packed trees as conditioned Galton--Watson trees}
\label{subsec:PackedTrees_GW2}
Building on \cref{lem:unif_packed_tree}, in what follows we explain how to sample a uniform packed tree with $n$ leaves
as a randomly decorated  Galton--Watson tree conditioned on having $n$ leaves.

\medskip

Let $\rho_q$ denote the radius of convergence of the generating series ${\mathcal Q}(z)$.
As observed in \cite[Section 3.1]{borga2020decorated}, it holds that $\rho_q > 0$.
As we shall see, this implies that $\bm{T}_n$ has the distribution of a Galton--Watson tree conditioned of having $n$
leaves, whose offspring distribution $\xi$ is defined below
(for similar discussion with fixed number of vertices,
see \cite[Section 4]{janson2012simply}).

The offspring distribution $\xi$ is given by 
\begin{align}
\label{eq:offspring_distribution_packed_tree}
\begin{cases}
\P(\xi = 0) = a\\
\P(\xi = 1) = 0\\
\P(\xi = k) = q_k t_0^{k-1} \text{ for } k  \geq 2.
\end{cases}
\end{align}
with $a, t_0>0$ constants that are defined as follows. 
If $\lim_{z \nearrow \rho_q} {\mathcal Q}'(z) \ge 1$, let $0<t_0\le \rho_q$ be the unique number with ${\mathcal Q}'(t_0) = 1$. 
If the limit is less than $1$, then set $t_0 = \rho_q$.
Finally set $a = 1 - \sum_{k \ge 2} q_k t_0^{k-1} >0$.

Note that the tilting in \cref{eq:offspring_distribution_packed_tree} previously appeared in \cite[Proposition 2]{pitman2015schroder} (see also the discussion above Corollary 1 in the same paper).
We note that $\xi$ is always aperiodic since $q_k>0$ for $k \ge 2$ (because of the $\circledast$ decorations).
Moreover, we have
\begin{align}
\E[\xi] = \mathcal Q'(t_0) \le 1,
\end{align}
so that the Galton--Watson tree $\bm T^\xi$ of offspring distribution $\xi$ is either subcritical or critical.
It is a simple exercise to check that $\bm T^\xi$, conditioned on having $n$ leaves,
has the same distribution as the simply generated tree $\bm{T}_n$ defined by \cref{eq:simply_gen_distrib}.

The following proposition determines if $\xi$ is either subcritical or critical
in terms of the generating functions of $\mathfrak{S}$ , that we conveniently denote by $\mathcal S(z)$. From Stanley-Wilf-Marcus-Tardös theorem \cite{marcus2004excluded}, it always has a positive radius of convergence
$\rho_{\mathcal S}>0$.
Below, we write $\cS'(\rho_\cS)$ for $\lim_{z \nearrow \rho_\cS} \cS'(z)$, noting that this limit may be infinite. 

\begin{prop}[{\cite[Proposition 3.5]{borga2020decorated}}]
	\label{prop: offspring_distr_charact}
	It holds that $\E[\xi] = 1$ if and only if
	\begin{align}
	\label{eq:type1}
	\cS'(\rho_\cS) \ge \frac{2}{(1 +\rho_\cS)^2} -1.
	\end{align}
	In this case, $t_0 = \kappa/(1+\kappa)$ for the unique number $0  < \kappa \le \rho_\cS$ with $\cS'(\kappa) = 2/(1+\kappa)^2 -1$, and
	\begin{align}
	\V[\xi] = \kappa (1+\kappa)^3 \cS''(\kappa) + 4 \kappa.
	\end{align}
\end{prop}

\section{Square and almost square permutations}\label{sect:sq_perm}

We now leave the combinatorics of trees, deeply used in the two previous sections. Here we encode random permutations with sequences of i.i.d.\ random variables, i.e.\ (unconstrained) random walks.

\medskip

The points of a permutation can be divided into two types, internal and external.  The external points are the records of the permutation, either maximum or minimum, from the left or from the right.  The internal points are the points that are not external.  Square permutations are permutations where every point is external.  Almost square permutations are permutations with some fixed number of internal points. See \cref{example_intro} for two examples. We use the notation $\sq(n)$ to denote the set of square permutations of size $n$ and $\asqnk$ to denote the set of almost square permutations of size $n+k$ with exactly $n$ external points and $k$ internal points.

\begin{figure}[htbp]
	\begin{minipage}[c]{0.55\textwidth}
		\centering
		\includegraphics[scale=0.7]{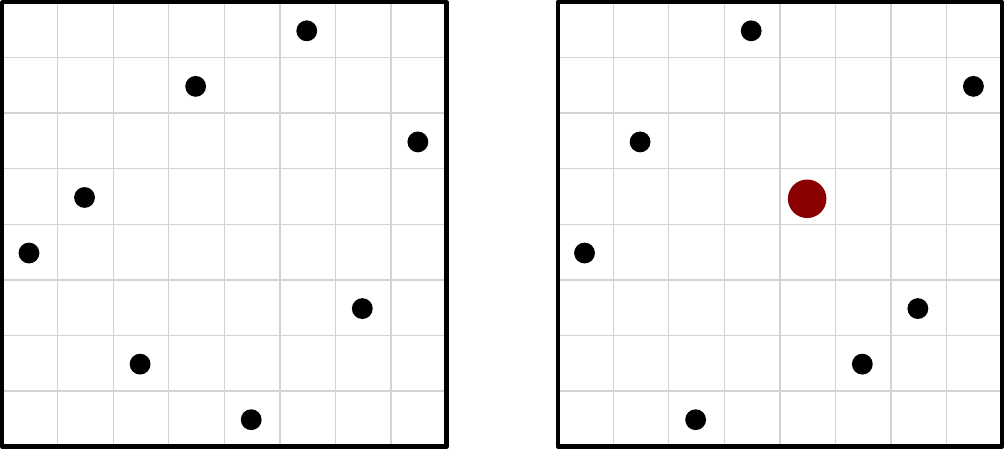}
	\end{minipage}
	\begin{minipage}[c]{0.44\textwidth}
		\caption{The diagram of two permutations. The permutation on the left is a square permutation of size 8. The permutation on the right is an almost square permutation with one internal point, highlighted in red.\label{example_intro}}		
	\end{minipage}
\end{figure}

In \cite{albert2011convex} square permutations were referred to as \emph{convex permutations} and were described by pattern-avoidance. In particular, square permutations are permutations that avoid the sixteen permutations of size $5$ that have one internal point.

The starting point for all our results on square and almost square permutations (see  \cref{const_lim_obj} and \cref{chp:square}) is the sampling procedure for uniform square permutations described in the following sections. 
Inspired by the approach in \cite{hoffman2019dyson}, we define a projection map from the set of square permutations to the set of anchored pairs of sequences of labels, i.e.\ triplets $(X,Y,z_0)\in\{U,D\}^n\times\{L,R\}^n\times [n]$. For every square permutation $\sigma,$ the labels of $(X,Y)$ are determined by the record types. The sequence $X$ records if a point is a maximum ($U$) or a minimum ($D$) and the sequence $Y$ records if a point is a left-to-right record ($L$) or a right-to-left record ($R$); the anchor $z_0$ is the value $\sigma^{-1}(1)$.  Section \ref{sect:perm_to_anchored_seq} gives a precise definition and examples. 

This projection map is not surjective, but in Section \ref{sect:inverse_projection} we show that we can identify subsets of anchored pairs of sequences (called \emph{regular}) and of square permutations where the projection map is a  bijection. We then construct a simple algorithm to produce square permutations from regular anchored pairs of sequences.  We show that asymptotically almost all square permutations can be constructed from regular anchored pairs of sequences, thus a permutation sampled uniformly from the set of regular anchored pairs of sequences will produce, asymptotically, a uniform square permutation.  

These regular anchored pairs of sequences are defined using a slight modification of the \emph{Petrov conditions}, i.e.\ technical conditions on the labels, found in \cite{hoffman2017pattern} and again in \cite{hoffman2019dyson} (a uniform pair of sequences satisfies these conditions outside a set of exponentially small probability). We say that an anchored pair of sequences is \emph{regular} if it satisfies these Petrov conditions, and the anchored point is neither too close to $1$ nor to $n$.

\subsection{Projections for square permutations}
\label{sect:perm_to_anchored_seq}
We begin this section with the following key definition.
\begin{defn}
	An \emph{anchored pair of sequences} is a triplet $(X,Y,z_0)$, where $X\in\{U,D\}^n$, \break
	$Y\in\{L,R\}^n$ and $z_0 \in [n].$ We say that the pair $(X,Y)$ is anchored at $z_0$.
\end{defn}
Given a square permutation $\sigma\in Sq(n),$ we associate to it an anchored pair of sequences $(X,Y,z_0)$ in the following way (\emph{cf.}\ Fig.~\ref{Square_perm_sampling_example}). First let $X_1 = X_n = D$ and $Y_1 = Y_n = L$. Then, for all $i\in \{2,\cdots,n-1\},$  we set
\begin{itemize}
	\item $X_i=\begin{cases}
		D\quad\text{if}\quad (i,\sigma(i))\in\LRm\cup\RLm,\\
		U\quad\text{if}\quad (i,\sigma(i))\in\LRM\cup\RLM.
	\end{cases}$
	\item $Y_i=\begin{cases}
	L\quad\text{if}\quad (\sigma^{-1}(i),i)\in\LRm\cup\LRM,\\
	R\quad\text{if}\quad (\sigma^{-1}(i),i)\in\RLm\cup\RLM.
	\end{cases}$
\end{itemize}
In the case that $(i,\sigma(i))\in\LRM\cap\RLm$ or $(i,\sigma(i))\in\LRm\cap\RLM$ we set $X_i=D$ and $Y_{\sigma(i)}=L.$
Finally, let $z_0 = \sigma^{-1}(1)$. Note that $X_{z_0}$ is always equal to $D$.
Intuitively, the sequence $X$ tracks if the points in the columns of the diagram of $\sigma$ are minima or a maxima and the sequence $Y$ tracks if the points in the rows are left or right records.

\begin{figure}[htbp]
	\begin{minipage}[c]{0.54\textwidth}
		\centering
		\includegraphics[scale=0.37]{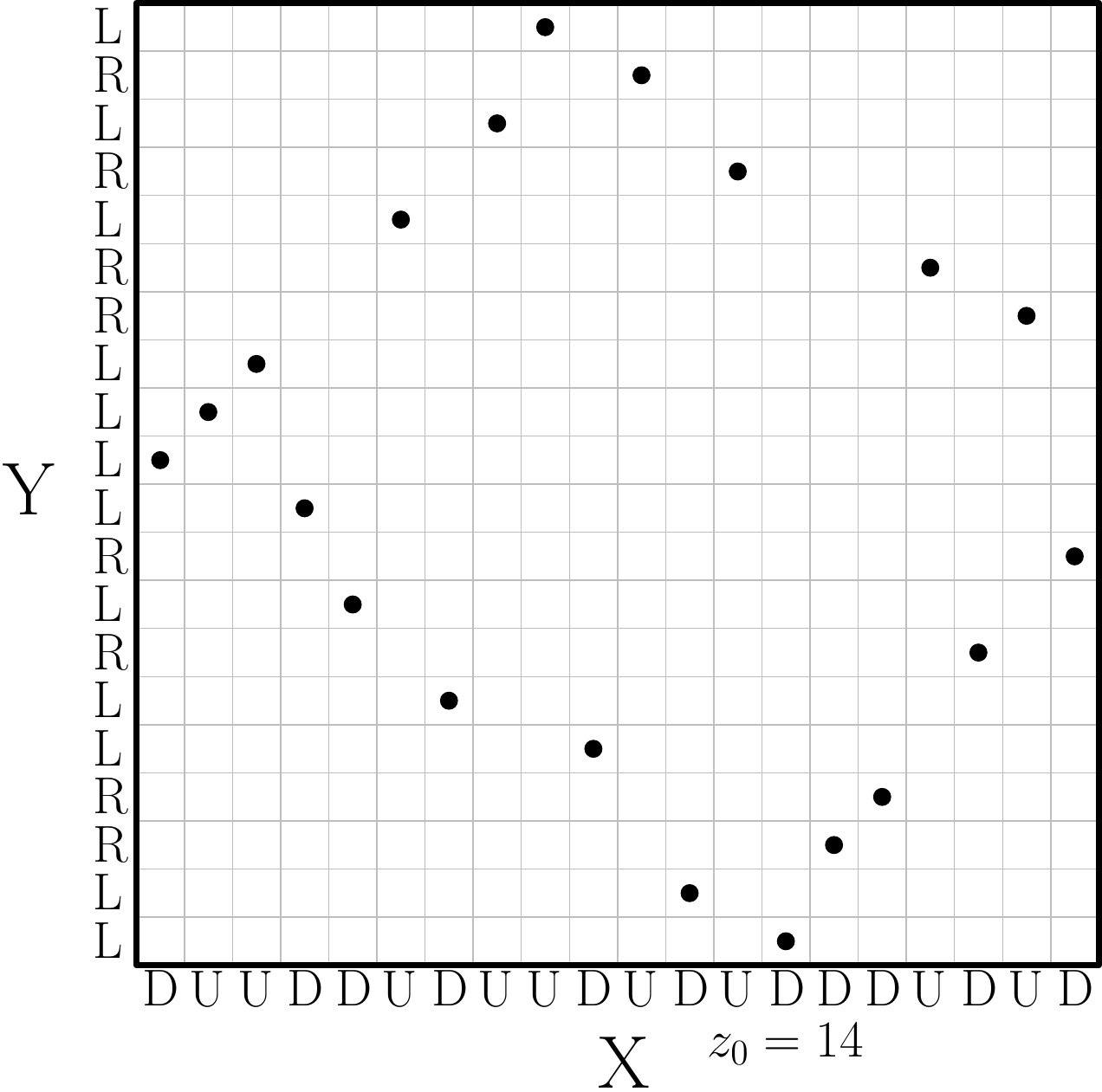}
	\end{minipage}
	\begin{minipage}[c]{0.45\textwidth}
		\caption{A square permutation $\sigma$ with the associated anchored pair of sequences $(X,Y,z_0).$  The sequence $X$ is reported under the diagram (read from left to right) of the permutation and the sequence $Y$ on the left (read from bottom to top).\label{Square_perm_sampling_example}}
	\end{minipage}
\end{figure}

We denote with $\phi$ the map that associates to every square permutation the corresponding anchored pair of sequences, therefore
$$\phi:Sq(n)\to\{U,D\}^n\times\{L,R\}^n\times[n].$$

We say that an anchored pair of sequences $(X,Y,z_0)$ of size $n$ is \emph{good} if $X_1 = X_n = X_{z_0}=D$ and $Y_1 = Y_n = L$. Note that $\phi(Sq(n))$ is contained in the set of good anchored pairs of size $n$.

\begin{rem}
	This projection map is also used in \cite{duchi2019square}. The author shows that $\phi$ is injective.  
\end{rem}

\subsection{Anchored pairs of sequences and Petrov conditions}\label{sect:inverse_projection}

From a good anchored pair we wish to construct a square permutation.  For most good anchored pairs this will be possible, though we do need to proceed with caution.

Let $\ctd(i)$ denote the number of $D$s in $X$ up to (and including) position $i$.  Similarly define $\ctu(i)$, $\ctl(i)$ and $\ctr(i)$ for the number of $U$s in $X$ and the number of $L$s or $R$s in $Y$, respectively.  Let $\pd(i)$ denote the index of the $i$th $D$ in $X$ with $\pd(i) = n$ if there are fewer than $i$ indices labeled with $D$ in $X$.  Similarly define $\pu(i)$, $\pl(i)$ and $\pr(i)$ for the location of the indices of the other labels.

\begin{defn}\label{defn:petrov} We say that the label $D$ in $X$ satisfies the \emph{Petrov conditions} if the following are true:
	
	\begin{enumerate}
		\item $|\ctd(i) - \ctd(j) - \frac12(i-j) | < n^{.4}$, for all $|i-j| < n^{.6}$;
		\item $|\ctd(i) - \ctd(j) - \frac12(i-j) |	 < \frac{1}{2}|i-j|^{.6}$, for all $|i-j|> n^{.3}$;
		\item $|\pd(i) - \pd(j) - 2(i-j)| < n^{.4}$, for all $|i-j| < n^{.6}$ and $i,j \leq \ctd(n)$;
		\item $|\pd(i) - \pd(j) - 2(i-j)| < 2|i-j|^{.6}$,	 for all $|i-j|> n^{.3}$ and $i,j \leq \ctd(n)$.
	\end{enumerate}
	
	In particular, in the above inequalities we allow $j$ to be equal to 0 (defining $\ctd(0)\coloneqq 0$ and $\pd(j)\coloneqq 0$) obtaining that
	
	\begin{enumerate}
		\item[5.] $|\ctd(i) - \frac i 2 | < n^{.6}$, for all $i\leq n$;
		\item[6.] $|\pd(i)- 2i| < 2n^{.6}$, for all $i\leq  \ctd(n)$.
	\end{enumerate}	
\end{defn}

A similar definition holds for the label $U$ in $X$ and the labels $L$ and $R$ in $Y$ for the functions $\ctu,\ctl,\ctr,$ and $\pu, \pl, \pr$.  We say that the Petrov conditions hold for the pair of sequences $X$ and $Y$ if the Petrov conditions hold for all the labels of $X$ and $Y$. 
We state a technical result useful for later results.
\begin{lem}[{\cite[Lemma 3.2]{borga2020square}}]\label{petrov_often}
	If $X \in \{U,D\}^n$ satisfies the Petrov conditions then, for all $i\leq n$, 
	$$i - \pd(\ctd(i)) \leq n^{.3}.$$
\end{lem}

We let $\Omega_{n}$ denote the space of good anchored pairs of sequences such that both $X$ and $Y$ satisfy the Petrov conditions and $\delta_n \leq z_0 \leq n -\delta_n$ for some\footnote{In \cite{borga2020square} a regular anchored pair of sequences $(X,Y,z_0)$ satisfies $n^{.9} \leq z_0 \leq n -n^{.9}$ instead of $\delta_n \leq z_0 \leq n -\delta_n$. One can check that all the statements of \cite{borga2020square} are also true for this slightly more general definition.} $\delta_n=o(n),\delta_n\geq n^{.9}$. We will refer to these as \emph{regular} anchored pairs of sequences.

\begin{lem}[{\cite[Lemma 3.4]{borga2020square}}]\label{omega_size}
	Let $\bm X$, $\bm Y$ and $\bm z_0$ be chosen independently and uniformly at random from $\{U,D\}^n$, $\{L,R\}^n$ and $\{1, \cdots , n\}$ respectively.  Then $\P((X,Y,z_0) \in \Omega_n)=1-o(1).$
\end{lem}

\subsection{From regular anchored pairs of sequences to square permutations}\label{sect:rho_def}
We wish to define a map $\rho: \Omega_n \to Sq(n)$ by constructing a unique matching between the labels of $X$ and the labels of $Y$.  The matching will depend on the parameter $z_0$.  Once this map is properly defined we show that the composition $\phi\circ \rho$ acts as the identity on $\Omega_n$ (Lemma \ref{lem:map_welldef}).  

This construction may not be well defined on every good anchored pair of sequences in $\{U,D\}^n\times \{L,R\}^n \times [n]$, as we will see in \cref{fail1}, but it will be well-defined if we restrict to $\Omega_n$.
First define the following values (whose role will be clearer later) with respect to $(X,Y,z_0) \in \Omega_n$:
\begin{itemize}
	\item $z_1 = \pl(\ctd(z_0)),$
	\item $z_2 = \pu( \ctl(n) - \ctd(z_0)),$
	\item $z_3 = \pr( \ctd(n) - \ctd(z_0) ).$ 	
\end{itemize} 

The following lemma states a regularity property satisfied by the points $z_1,z_2$ and $z_3$.

\begin{lem}[{\cite[Lemma 3.5]{borga2020square}}]\label{okz}
	Let $(X,Y,z_0) \in \Omega_n.$ Then 
	$$\max(|z_1-z_0|,|z_2-z_3|, |n-z_0-z_2|) < 10n^{.6}.$$
\end{lem}

Define the following index sets:
\begin{itemize}
	\item $I_1 = \{1,\cdots, \ctd(z_0)\}$;
	\item $I_2 = \{1,\cdots,\ctu(z_2)\}$;
	\item $I_3 = \{\ctd(z_0) +1,\cdots, \ctd(n)\}$;
	\item $I_4 = \{\ctu(z_2) +1, \cdots, \ctu(n)\}$.	
\end{itemize}
Using these index sets, we create four sets of points:
\begin{itemize}
	\item $\Lambda_1 = \{(\pd(i),\pl(\ctd(z_0) + 1 -i))\}_{ i \in I_1}$;
	\item $\Lambda_2 = \{(\pu(i),\pl(\ctd(z_0) +i))\}_{ i \in I_2}$;
	\item $\Lambda_3 = \{(\pd(i),\pr(i-\ctd(z_0)))\}_{ i \in I_3}$; 
	\item $\Lambda_4 = \{(\pu(i),\pr(n-\ctd(z_0)+1-i))\}_{ i \in I_4}$. 
\end{itemize}

\begin{figure}[htbp]
	\begin{minipage}[c]{0.52\textwidth}
		\centering
		\includegraphics[scale=.37]{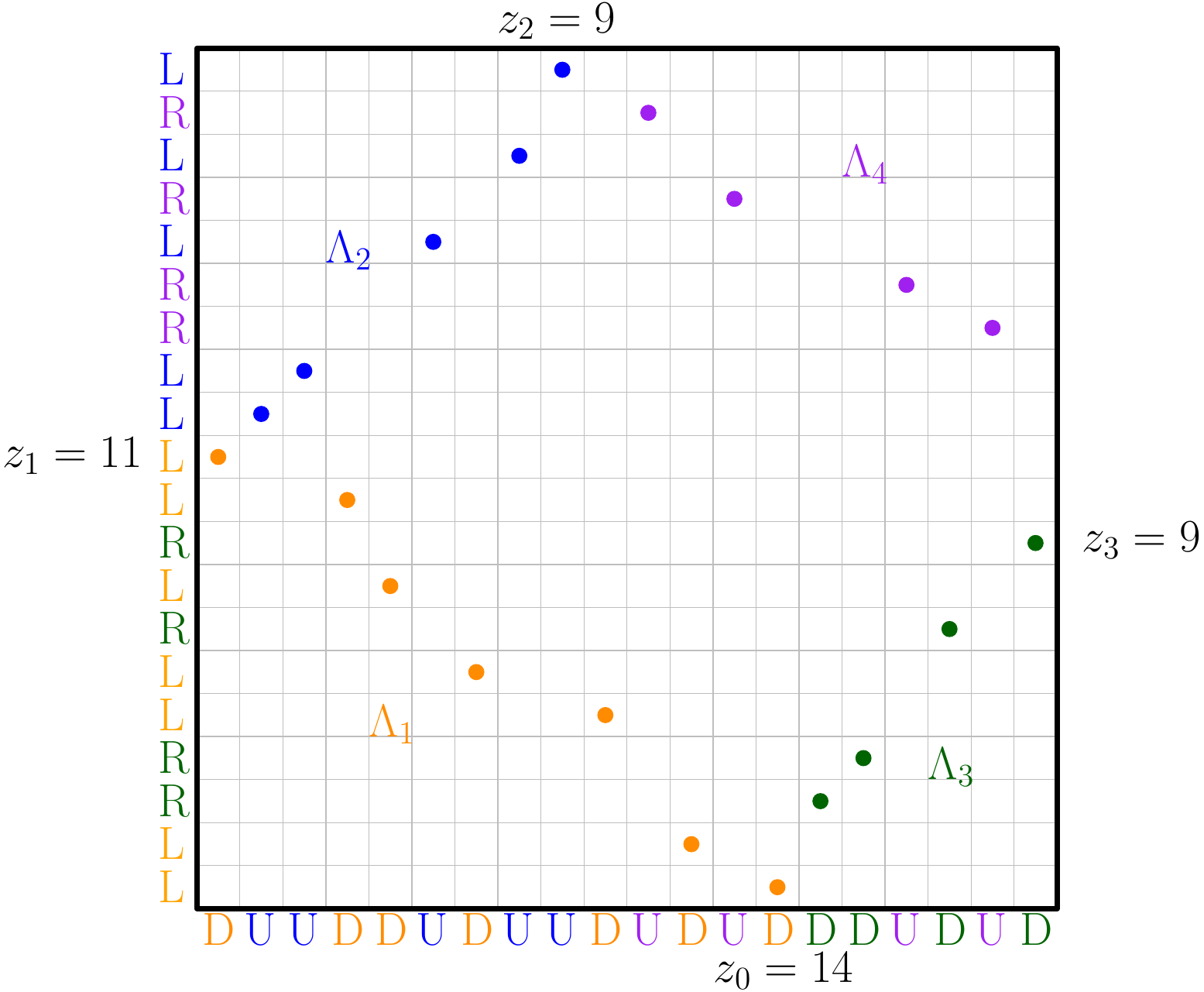}
	\end{minipage}
	\begin{minipage}[c]{0.47\textwidth}
		\caption{An example of square permutation obtained from a regular anchored pair of sequences $(X,Y,z_0)$. We color in orange the labels in the sequences $X$ and $Y$ corresponding to indices of $I_1$ and we also color in orange the points of $\Lambda_1.$ Similarly, we color in blue the labels and the points corresponding to $I_2$ and $\Lambda_2,$ in green the ones corresponding to $I_3$ and $\Lambda_3$, and in purple the ones corresponding to $I_4$ and $\Lambda_4$.\label{reconstruction}}
	\end{minipage}
\end{figure}

First a few observations about each of the sequences $\Lambda_i$ (\emph{cf.}\ Fig.~\ref{reconstruction}):  
\begin{itemize}
	\item The first sequence, $\Lambda_1,$ is obtained by matching the first $\ctd(z_0)$ $D$s in $X$ with the first $\ctd(z_0)$ $L$s in $Y$ to create a decreasing sequence starting from $(1,z_1)$ and ending at $(z_0,1)$ (note that thanks to Petrov conditions, for $n$ large enough, $\ctd(z_0)$ is smaller than the total number of $L$s in $Y$ and so the operations are well-defined\footnote{We point out that the set $\Lambda_1$ is in full generality well-defined for all $n\geq1$, but in the case that $\ctd(z_0)>\ctl(n)$ then $|\Lambda_1|<|I_1|$ (because by definition  $\pl(i)=n$ for all  $i\geq\ctl(n)$). Similar observations will hold also for $\Lambda_2$, $\Lambda_3,$ and $\Lambda_4$.});
	\item the second sequence, $\Lambda_2,$ is obtained by matching the remaining $\ctl(n)-\ctd(z_0)$ $L$s in $Y$ with the first $\ctl(n)-\ctd(z_0)$ $U$s in $X$ (using Petrov conditions it follows that for $n$ large enough $\ctl(n) - \ctd(z_0)<\ctu(n)$) to create an increasing sequence starting from  $(1,z_1)$ and ending at $(z_2,n)$;
	\item the third sequence, $\Lambda_3,$ is obtained by matching the remaining $\ctd(n)-\ctd(z_0)$ $D$s in $X$ with the first $\ctd(n)-\ctd(z_0)$ $R$s in $Y$ to create an increasing sequence starting from $(z_2,n)$ and ending at $(n,z_3)$ (note that thanks to Petrov conditions, for $n$ large enough, $\ctd(n)-\ctd(z_0)$ is smaller than the total number of $R$s in $Y$);
	\item the fourth sequence, $\Lambda_4,$ is obtained by matching the remaining $\ctu(n)-(\ctl(n)-\ctd(z_0))$ $U$s in $X$ with the remaining $\ctr(n)-(\ctd(n)-\ctd(z_0))$ $R$s in $Y$ 
	to create a decreasing sequence between $(z_2,n)$ and $(n,z_3)$ (these two boundary points are not in $\Lambda_4$ by definition).
\end{itemize} 
\medskip
We can conclude that, for $n$ large enough, the indices corresponding to each $D$ and $U$ in $X$ and each $L$ and $R$ in $Y$ are used exactly once. Therefore by this construction each index of $X$ is paired with a unique index in $Y$ so the resulting collection of points must correspond to the points of a permutation $\sigma:[n]\to[n].$ 
We will see (in Lemma \ref{lem:map_welldef} below) that in fact $\sigma$ is in $Sq(n)$ and so that the assignment $\rho((X,Y,z_0))\coloneqq\sigma$ define a map from $\Omega_n$ to $Sq(n)$ when $n$ is large enough.

Note that for some good anchored sequences $(X,Y,z_0)$ that are not in $\Omega_n$ it is possible that the construction of the permutation $\sigma$ might fail (see Fig.~\ref{fail1}).

\begin{figure}[htbp]
	\begin{minipage}[c]{0.52\textwidth}
		\centering
		\includegraphics[scale=.35]{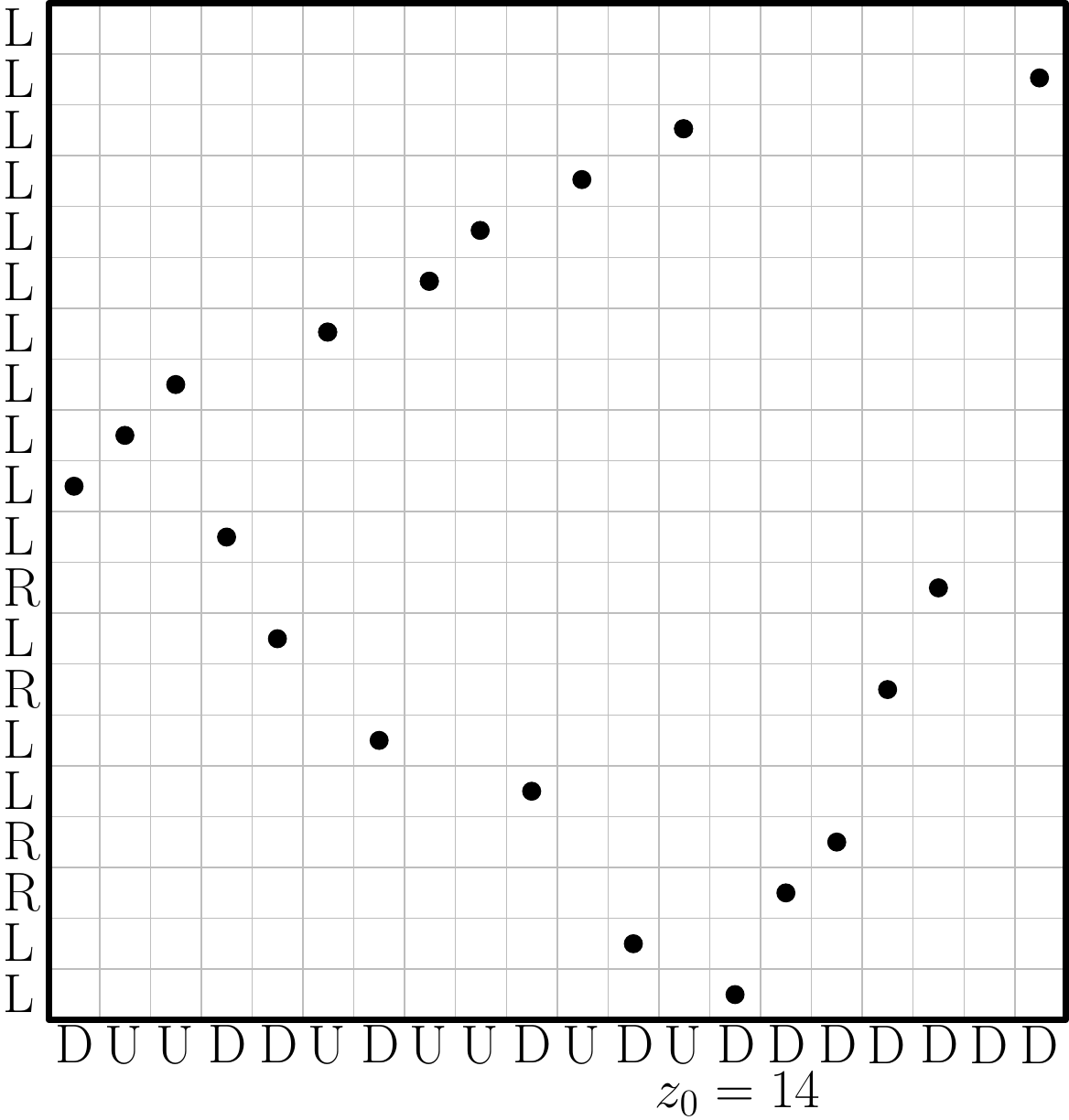}
	\end{minipage}
	\begin{minipage}[c]{0.47\textwidth}
		\caption{An example where the construction of $\sigma$ from a good anchored pair of sequences fails.  The last $L$ does not have a corresponding $U$ with which to match. The problem is that the sequences $X$ (resp.\ $Y$) contains too many $D$s (resp.\ too many $L$s) and so it does not satisfy the Petrov conditions.\label{fail1}}
	\end{minipage}
\end{figure}

The Petrov conditions force the following conditions on the sequences $\Lambda_i,$ for $i=1,2,3,4$. 
\begin{lem}[{\cite[Lemma 3.6]{borga2020square}}]\label{guiding light}
	If $(X,Y,z_0)\in \Omega_n$, then 
	\begin{itemize}
		\item $|s+t -z_0| < 10n^{.6}$ for $(s,t) \in \Lambda_1$,	
		\item $ |t-s - z_0| < 10n^{.6}$ for $(s,t)  \in \Lambda_2$,
		\item $ |s-t - z_0| < 10n^{.6}$ for $(s,t)  \in \Lambda_3$,
		\item $ |2n- s-t-z_0| < 10n^{.6}$ for $(s,t)  \in \Lambda_4$.
	\end{itemize}
\end{lem}

In words, the conditions above impose that the points in the sequences $\Lambda_i,$ for $i=1,2,3,4$ are close to a rectangle rotated of 45 degrees and with bottom vertex at $(z_0,1)$ (see \cref{fig:min_max_perm3}). We highlight that the statement of \cref{guiding light} gives a deterministic result.

\begin{figure}[htbp]
	\begin{minipage}[c]{0.54\textwidth}
		\centering
		\includegraphics[scale=.3]{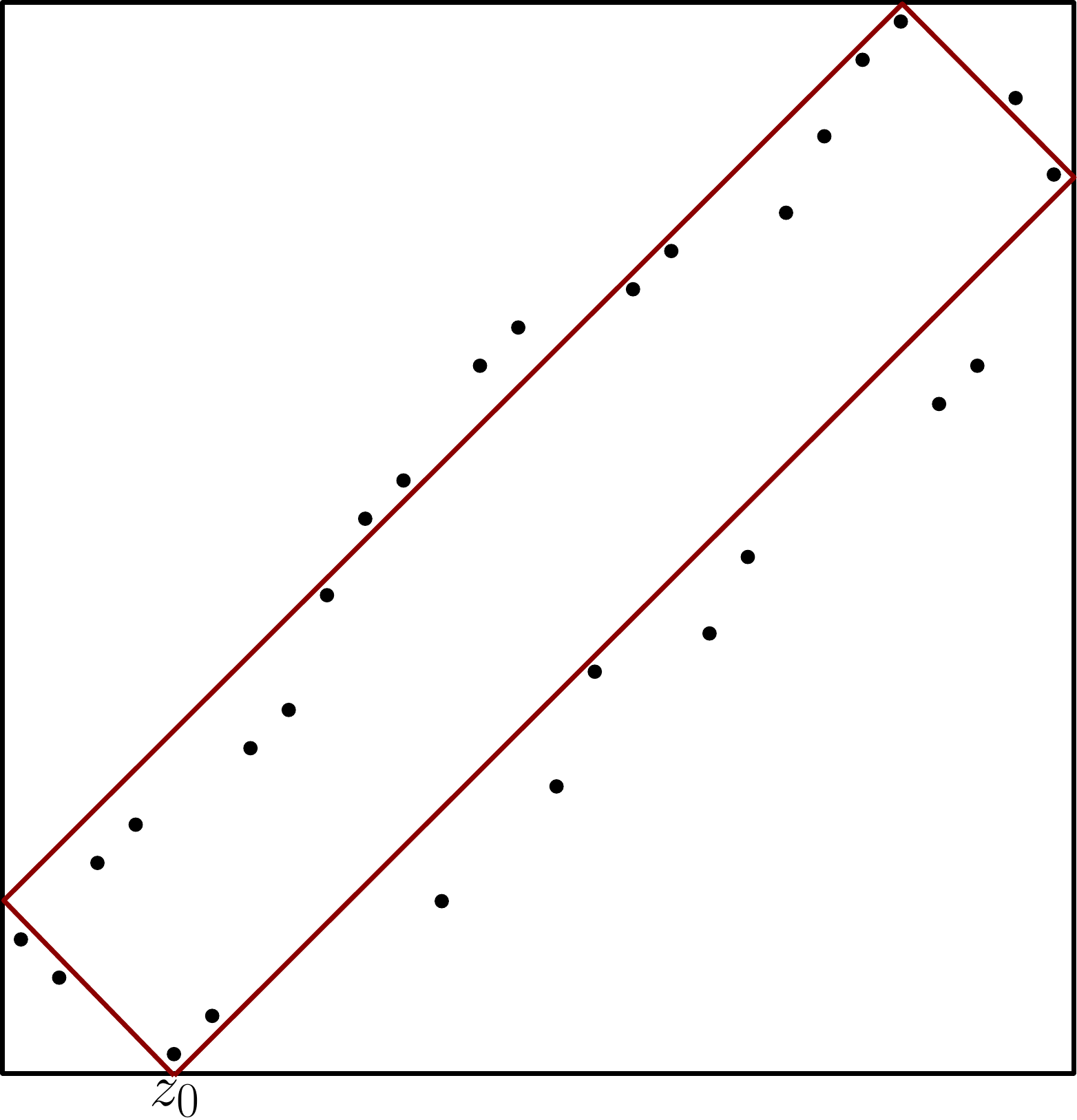}
	\end{minipage}
	\begin{minipage}[c]{0.45\textwidth}
		\caption{A square permutation obtained from an anchored pair of sequences $(X,Y,z_0)\in \Omega_{28}$. In red a rectangle rotated of 45 degrees and with bottom vertex at $(z_0,1)$. \label{fig:min_max_perm3}}
	\end{minipage}
\end{figure}

\begin{lem}[{\cite[Lemma 3.7]{borga2020square}}]
	\label{lem:map_welldef}
	For $n$ large enough the following holds. Let  $(X,Y,z_0) \in \Omega_n$ and  let $\sigma=\rho(X,Y,z_0)$ as above.  Then $\sigma$ is in $Sq(n)$ with $\LRm = \Lambda_1$, $\LRM = \Lambda_2 \cup \{(1,z_1)\}$, $\RLm = \Lambda_3 \cup \{(z_0,1)\}$ and $\RLM = \Lambda_4\cup \{(z_2,n),(n,z_3)\}.$  Moreover, $\rho$ is injective with $\phi\circ\rho$ acting as the identity on $\Omega_n$.   
\end{lem}

We conclude this section with the following key result.

\begin{lem}\label{square_is_rect}
	With probability $1-o(1)$ a uniform random square permutation $\bm{\sigma}_n$ of size $n$ is in $\rho(\Omega_n)$.	
\end{lem}

\begin{proof}
	
	The map $\rho$ is injective from $\Omega_n$ into $Sq(n)$ and thus $\P(\bm{\sigma}_n \in \rho(\Omega_n)) = \frac{|\Omega_n|}{|Sq(n)|}$.
	
	In \cite[Theorem 5.1]{mansour2006grid} it was shown that $|Sq(n)|=2(n+2)4^{n-3}-4(2n-5){ 2n-6 \choose n-3 }$.
	Noting that the negative term satisfies $(2n-5){ 2n-6 \choose n-3 } = o(2(n+2)4^{n-3})$, we obtain that 
	\begin{equation}
	\label{eq:est_sqn}
	|Sq(n)|=2(n+2)4^{n-3}(1-o(1)).
	\end{equation}
	The number of good anchored pairs of sequences of size $n$ is $2^{n-2}(2\cdot 2^{n-2}+(n-2)\cdot 2^{n-3})= 2(n+2)4^{n-3}$. Therefore, using Lemma \ref{omega_size} the size of $\Omega_n$ satisfies 
	$$|\Omega_n|=2(n+2)4^{n-3}(1-o(1)).$$
	Thus, as $n$ tends to infinity, $\P(\bm{\sigma}_n\in \rho(\Omega_n)) \to 1.$  
\end{proof} 

Note that combing Lemmas \ref{square_is_rect} and \ref{omega_size} we get the following remarkable conclusion: Let $\bm X^n$ and $\bm Y^n$ be uniform random walks of size $n$ on $\{U,D\}$ and $\{L,R\}$ respectively. Let also  $\bm z^n_0$ be a uniform number in $\{1, \cdots , n\}$. Then $\rho(\bm X,\bm Y,\bm z^n_0)$ is asymptotically a uniform square permutation of size $n$.

\section{Permutation families encoded by generating trees}\label{sect:gen_tree-perm}

In the previous section we used (unconstrained) random walks to obtain asymptotically uniform random square permutations. We now present a way to encode some other families of permutations with walks conditioned to stay positive. The latter have several remarkable combinatorial/probabilistic properties that will be  used in \cref{sect:CTL_increments}.

\subsection{Definitions and examples}
Since we are interested in the study of permutations, we restrict the definition of generating tree to these specific objects despite they can be defined for other combinatorial families (see for instance \cite{MR1717162,MR3537914,MR3992288}). We need the following preliminary construction.

\begin{defn}
	Given a permutation $\sigma\in\mathcal{S}_n$ and an integer $m\in [n+1],$ we denote by $\sigma^{*m}$ the permutation obtained from $\sigma$ by \emph{appending a new final value} equal to $m$ and shifting by $+1$ all the other values larger than or equal to $m.$ Equivalently,
	$$\sigma^{*m}\coloneqq\text{std}(\sigma(1),\dots,\sigma(n),m-1/2).$$
\end{defn}

We say that a family of permutations $\mathcal{C}$ \emph{grows on the right} if for every permutation $\sigma\in\mathcal{S}$ such that $\sigma^{*m}\in\mathcal{C}$ for some $m\in [|\sigma|+1]$, we have $\sigma\in\mathcal{C}.$ From now until the end of this section we only consider families of permutations that grow on the right, without explicitly saying it.

\begin{defn}
	\label{def:gen_tree}
	The \emph{generating tree} of a family of permutations $\mathcal{C}$ is the infinite rooted tree whose vertices are the permutations of $\mathcal{C}$ (each appearing exactly once in the tree) and such that permutations of size $n$ are at level $n$ (the level of a vertex being the distance from the root plus one).
	The children of some permutation $\sigma\in\mathcal{C}$ correspond to the permutations obtained by appending a new final value to $\sigma$.
\end{defn}

\begin{rem}
	Another classical way of defining generating trees for families of permutations is by appending a new maximal value on the top of the diagram of $\sigma$. We will never consider this case in this manuscript.
\end{rem}

We now give an example of a generating tree for the family of permutations avoiding the patterns 1423 and 4123 inspired by a result of Kremer \cite{kremer2000permutations, kremer2003postscript}\footnote{Note that the various results in \cite{kremer2000permutations, kremer2003postscript} are stated for families of permutations that grows on the top and not on the right. However, these results  can be translated into permutations growing on the right by using the diagonal symmetry of the diagram of pattern-avoiding permutations. For instance, we can recover the result for $\{1423,4123\}$-avoiding permutations growing on the right from $\{1342,2341\}$-avoiding permutations growing on the top.}. This family will appear several times along this section (see Examples \ref{exmp:1423_4123_av_part0}, \ref{exmp:1423_4123_av} and \ref{exmp:1423_4123_av_part2}).

\begin{exmp}
	\label{exmp:1423_4123_av_part0}
	The first levels of the generating tree for $\{1423,4123\}$-avoiding permutations are drawn in \cref{gen_tree_1423_4123}. Note that every child of a permutation is obtained by appending a new final value. For instance, the permutation $321$ on the left-hand side of the third level is obtained from the permutation $21$ by appending a new final value equal to $1.$
\end{exmp}

\begin{defn}
	\label{def:gen_tree_part2}
	Let $S$ be a set and assume there exists an $S$-valued statistics on the permutations of $\mathcal{C},$ whose value determines the number of children in the generating tree and the values of the statistics for the children. Then we give labels to the objects of $\mathcal{C}$ which indicate the value of the statistics. The associated \emph{succession rule} is then given by the label $\lambda$ of the root and, for any label $k,$ the labels $e_1(k),\dots,e_{h(k)}(k)$ of the $h(k)$ children of an object labeled by $k.$ In full generality, the associated succession rule is represented by the formula
	\begin{equation}
	\begin{cases} 
	\text{Root label}: (\lambda) \\
	(k)\to (e_1{\scriptstyle(k)}),\dots,(e_{h(k)}{\scriptstyle(k)}).
	\end{cases}
	\label{eq:ass_succ_rul}
	\end{equation}
	The set of labels appearing in a generating tree is denoted by $\mathcal{L}$ and for all $k\in\mathcal{L},$ the multiset of labels of the children $\{e_1(k),\dots,e_{h(k)}(k)\}$ is denoted by $\CL(k),$ where $\CL$ stands for \emph{``children labels"}.
\end{defn}

\begin{rem} \label{rem:multi-dim}
	In our examples, we only consider generating trees with labels taking values in some subset of $\mathbb{Z}_{>0}.$ However, other types of generating trees have been considered and studied, for instance trees with $(\mathbb{Z}_{>0})^2$-valued labels (an example is the generating tree for the well-known Baxter permutations \cite{bousquet2003four}) or (integer tuples)-valued labels (see \cite{marinov2003generating}). 
\end{rem}

Before continuing Example \ref{exmp:1423_4123_av_part0}, we introduce some terminology.

\begin{defn}
	Given a family of permutations $\mathcal{C},$ a permutation $\sigma\in\mathcal{C}_n$ and an integer $m\in [n+1],$ we say that $m$ is an \emph{active site} of $\sigma$ if $\sigma^{*m}\in\mathcal{C}.$ We denote by $\text{AS}(\sigma)$ the set of active sites of $\sigma$ and by $|\text{AS}(\sigma)|$ its cardinality.
	If $\text{AS}(\sigma)=\{i_1,\dots,i_k\},$ with $i_1<\dots<i_k,$ we call $i_1$ the \emph{bottom active site} and $i_k$ the \emph{top active site}.
\end{defn}

\begin{exmp}\label{exmp:1423_4123_av}
	Given a permutation $\sigma\in\Av_n(1423,4123),$ we have the following properties for the set $\text{AS}(\sigma)$ (for a proof see \cite[Theorem 8]{kremer2000permutations}):
	\begin{itemize}
		\item 1,2 and $n+1$ are always active sites, i.e.\  $\{1,2,n+1 \}\subseteq\text{AS}(\sigma)$.
		\item If $\text{AS}(\sigma)=\{i_1=1,i_2=2,i_3,\dots,i_{k-1},i_k=n+1\}$ then
		\begin{itemize}
			\item $\text{AS}(\sigma^{*1})=\{1,2,3,i_3+1,i_4+1,\dots,i_{k-1}+1,n+2\}$;
			\item $\text{AS}(\sigma^{*(n+1)})=\{1,2,i_3,\dots,i_{k-1},n+1,n+2\}$;
			\item $\text{AS}(\sigma^{*i_j})=\{1,2,i_3,\dots,i_{j},n+2\},$ for all $2\leq j< k.$
		\end{itemize}
	\end{itemize}
	
	Note that, as a consequence, $|\text{AS}(\sigma^{*1})|=|\text{AS}(\sigma^{*(n+1)})|=|\text{AS}(\sigma)|+1,$
	while,
	$|\text{AS}(\sigma^{*i_j})|=j+1,$ for $j\in[2,k-1]$.
	In  particular, the statistics $|\text{AS}(\cdot)|$ fulfills the conditions of \cref{def:gen_tree_part2} for  the family of $\{1423,4123\}$-avoiding permutations, and the associated succession rule is
	\begin{equation}
	\begin{cases} 
	\text{Root label}: (2) \\
	(k)\to (k+1),(3),(4)\dots,(k),(k+1). 
	\end{cases}
	\label{eq:succ_rule_1423_4123}
	\end{equation}
	For convenience, we choose to write the succession rule and to draw the children of a permutation (see \cref{gen_tree_1423_4123} below) in increasing order of the appended final value. 
	
	\begin{figure}[htbp]
		\centering
			\includegraphics[scale=0.56]{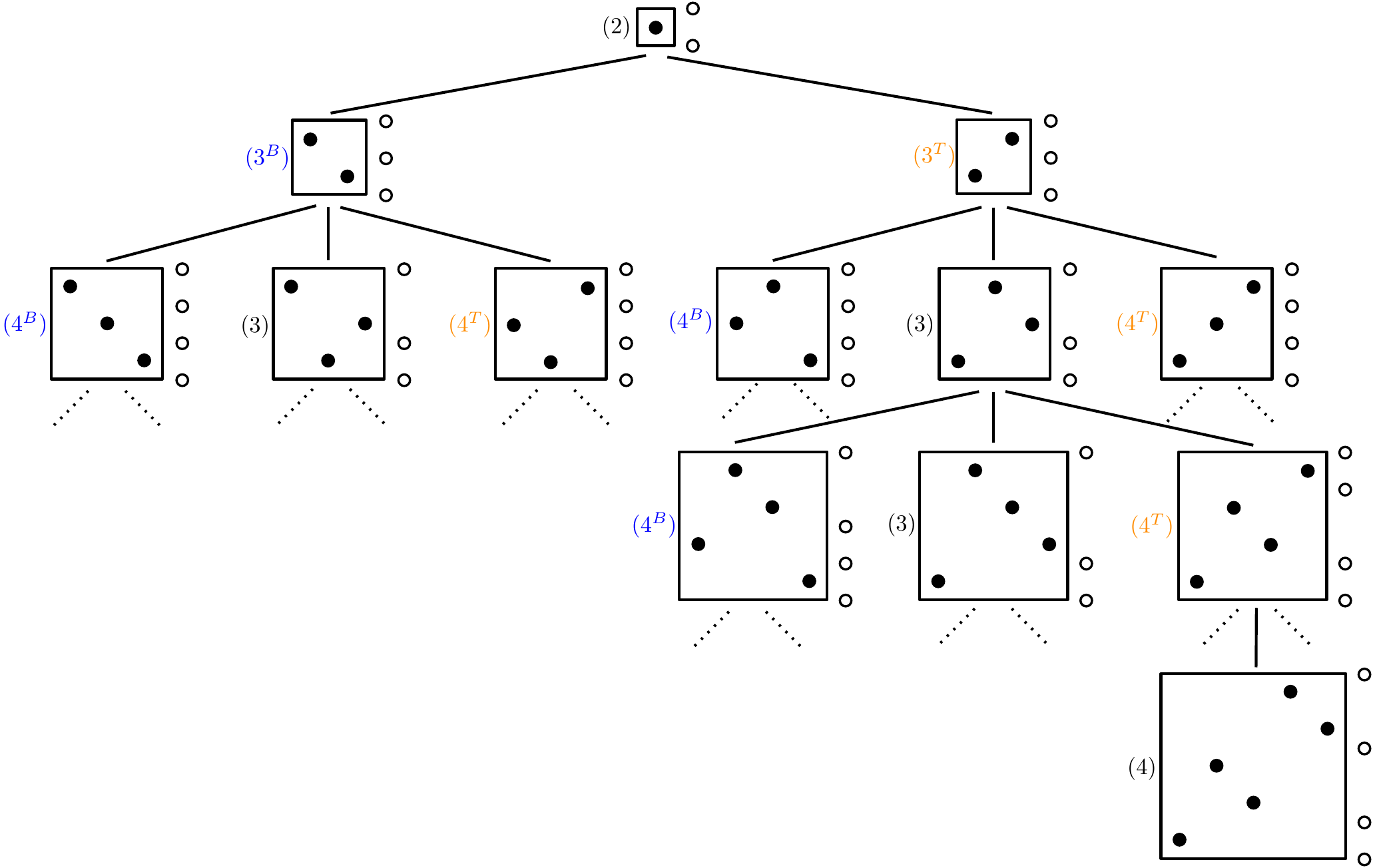}\\
			\caption{The generating tree for $\{1423,4123\}$-avoiding permutations. We completely draw only the first three levels of the tree; for the fourth level, we only draw the children of the permutation 132, and for the fifth level only one of the children of the permutation 1324. For each permutation, we highlight the active sites with some circles on the right-hand side of the diagram. Moreover, on the left-hand side of each diagram we report the corresponding label given by the statistics $|\text{AS}(\cdot)|$. The superscripts $T$ and $B$ and the colors that appear in some labels will be clarified in Example \ref{exmp:1423_4123_av_part2}. \label{gen_tree_1423_4123}}
	\end{figure}
\end{exmp}

\begin{rem}
	In our examples, we only consider generating trees whose succession rule is determined (in the sense of \cref{def:gen_tree_part2}) by the statistics $|\text{AS}(\cdot)|$. However,  in other examples, we need to consider other statistics. For instance, the succession rule for the generating tree for Baxter permutations is determined by the numbers of left-to-right and right-to-left maxima (see \cite{bousquet2003four}).
\end{rem}

\subsection{A bijection between permutations encoded by a generating tree and colored walks}
\label{sect:bij_color_walk_perm}

We start by pointing out that we are not assuming that the children labels $e_1(k),\dots,e_{h(k)}(k)$ appearing in the succession rule in \cref{eq:ass_succ_rul} are distinct (notice for instance that the label $(k+1)$ is produced twice from the label $(k)$ in \cref{eq:succ_rule_1423_4123}). For each pair of labels $(k,k')\in\mathcal{L}^2$, we denote by $\mult_{k}(k')$ the number of indices $i$ such that $ e_{i}(k)=k'.$

There is a natural bijection between permutations in a family encoded by a generating tree and the set of paths in the tree starting at the root. We simply associate to each path the permutation appearing in its endpoint. We may further identify each path in the tree with the list of labels appearing on this path, but this requires to distinguish the potential repeated labels in the succession rule. More precisely, for each label $k\in\mathcal{L}$ and for each element $k'\in\CL(k)$ such that $\mult_{k}(k')>1,$  we distinguish the repeated labels by painting them  with different colors (say colors $\{1,\dots,\mult_{k}(k')\}$). Then the \emph{colored succession rule} is represented as 
\begin{equation}
\begin{cases} 
\label{eq:colored_succ_rule}
\text{Root label}: (\lambda) \\
(k)\to (e^{c}_1{\scriptstyle(k)}),\dots,(e^{c}_{h(k)}{\scriptstyle(k)}), 
\end{cases}
\end{equation}	
where the exponents $c$ recall that the labels might be colored. We highlight that the colors and the values of the children labels depend only on the value of the parent label (and not on the color).

In this way, every permutation of size $n$ is bijectively encoded by a sequence of $n$ (colored) labels $(k_1=\lambda,k_2^{c_2},\dots,k_n^{c_n})$, where $k_i$ records the value of the label of the vertex at level $i$ and $c_i$ its color.  Note that every pair of two consecutive labels, $(k_i^{c_i},k_{i+1}^{c_{i+1}})$ for $1\leq i\leq n-1,$ is \emph{consistent with the colored succession rule}, i.e.\ there exist $j=j(i)\in [1,h(k_i)]$ such that $k_{i+1}^{c_{i+1}}=e_{j}^{c}(k_i)$. In order to simplify the notation, we will often write a sequence of colored labels as $(k_i^{c})_i$ instead of $(k_i^{c_i})_i$.

\medskip

We denote by $\GG$ this map between consistent sequences of colored labels and permutations in the generating tree.

\begin{prop}\label{prop:bij_gen_tree}
	The map $\GG$ is a bijection between consistent sequences of colored labels and permutations in the generating tree.
\end{prop}

\begin{exmp}
	\label{exmp:1423_4123_av_part2}
	We continue Examples \ref{exmp:1423_4123_av_part0} and \ref{exmp:1423_4123_av}. In the succession rule given in \cref{eq:succ_rule_1423_4123} the label $(k+1)$ is produced twice. We distinguish the two occurrences of $(k+1)$ by painting the first one in blue (and we will write $\textcolor{blue}{{(k+1)}^{B}}$) and the second one in tangerine\footnote{The choice for the the colors is made with the aim of remembering that a Blue label activates the Bottom active site and a Tangerine label activates the Top active site.} (and we will write $\textcolor{orange}{{(k+1)}^{T}}$). The colored succession rule is 
	\begin{equation}\label{eq:biwfbwrubfwifnwepi}
	\begin{cases} 
	\text{Root label}: (2) \\
	(k)\to \textcolor{blue}{(k+1)^{B}},(3),(4)\dots,(k),\textcolor{orange}{(k+1)^{T}}. 
	\end{cases}
	\end{equation}
	
	With this coloring of the labels, the permutation 13254 (the only one displayed in the fifth level of the generating tree in \cref{gen_tree_1423_4123}) is encoded by the consistent sequence of (colored) labels $(2,\textcolor{orange}{3^{T}},3,\textcolor{orange}{4^{T}},4),$ i.e.\ $\GG(2,\textcolor{orange}{3^{T}},3,\textcolor{orange}{4^{T}},4)=13254.$
\end{exmp}

\begin{rem}
	We will see in \cref{sect:gen_tree_CLT} that, for many families of pattern-avoiding permutations $\mathcal C$ encoded by a generating tree, the walk associated with a uniform random permutation in $\mathcal C$ is a one-dimensional random walk conditioned to stay positive with independent colors.
\end{rem}

\section{Baxter permutations and related combinatorial objects}
\label{sec:discrete}

We conclude this chapter with Baxter permutations. Here we will see how to encode these permutations with higher dimensional discrete structures, such as two-dimensional walks in cones and a family of decorated planar maps.

\medskip

Baxter permutations were introduced by Glen Baxter in 1964 \cite{MR0184217} to study fixed points of commuting functions. A permutation $\sigma$ is Baxter if it is not possible to find $i < j < k$ such that $\sigma(j+1) < \sigma(i) < \sigma(k) < \sigma(j)$ or $\sigma(j) < \sigma(k) < \sigma(i) < \sigma(j+1)$. Baxter permutations are well-studied from a combinatorial point of view by the \textit{permutation patterns} community (see for instance \cite{MR0250516,MR491652,MR555815,bousquet2003four,MR3882946}). They are a particular example of family of permutations avoiding \textit{vincular patterns} (see \cite{MR2901166} for more details).
We denote by $\gls*{baxter}$ the set of Baxter permutations. 

\medskip

In the next sections we introduce several other combinatorial objects connected with Baxter permutations that will be useful for later purposes.

\subsection{Plane bipolar orientations, tandem walks, and coalescent-walk processes}\label{sect:coal_walk_intro}

\textit{Plane bipolar orientations}, or \emph{bipolar orientations} for short, are planar maps (i.e.\ connected graphs properly embedded in the plane up to continuous deformation) equipped with an acyclic orientation of the edges with exactly one source $s$ (i.e.\ a vertex with only outgoing edges) and one sink $s'$ (i.e.\ a vertex with only incoming edges), both on the outer face. We denote by $\mathcal{O}$ the set of bipolar orientations. The size of a bipolar orientation $m$ is its number of edges and it is denoted by $|m|$.

Every bipolar orientation can be plotted in the plane with every edge oriented from bottom to top (this is a consequence for instance of \cite[Proposition 1]{MR2734180}; see the black map on the left-hand side of \cref{fig:bip_orient} for an example). We think of the outer face as split in two: the \textit{left outer face}, and the \textit{right outer face}. The orientation of the edges around each vertex/face is constrained: we sum up these local constraints, setting some vocabulary, on the right-hand side of \cref{fig:bip_orient}. We call in-degree/out-degree of a vertex the number of incoming/outgoing edges around this vertex. We call left degree (resp.\ right degree) of an inner face the number of left (resp.\ right) edges around that face.

\begin{figure}[htbp]
	\centering
	\includegraphics[scale=0.75]{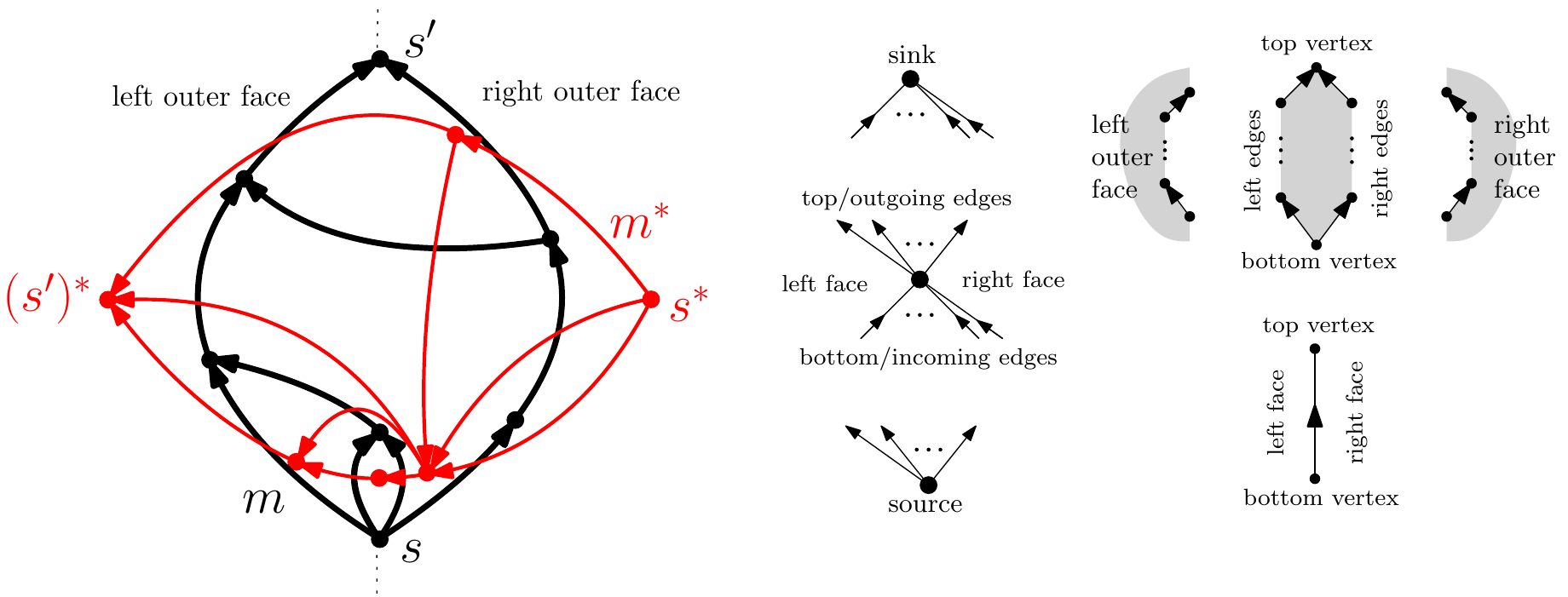}\\
	\caption{On the left-hand side, in black, a bipolar orientation $m$ of size 10 drawn with every edge oriented from bottom to top. In red, its dual map $m^*$, drawn with every edge oriented from right to left. On the right-hand side, the behavior of the orientation around each vertex/face/edge. Note for instance that in the clockwise ordering around each vertex different from the source and the sink there are top/outgoing edges, a right face, bottom/incoming edges, and a left face.\label{fig:bip_orient}}
\end{figure}

The \emph{dual map} $m^*$ of a bipolar orientation $m$ (called the \textit{primal}) is obtained by putting a vertex in each face of $m$, and an edge between two faces separated by an edge in $m$, oriented from the right face to the left face. The primal right outer face becomes the dual source, and the primal left outer face becomes the dual sink. Then $m^*$ is also a bipolar orientation (see the left-hand side of \cref{fig:bip_orient}).
The map $m^{**}$ is just $m$ with the orientation reversed, and $m^{****}=m$.

We now define a notion at the heart of many future combinatorial constructions.
Let $m$ be a bipolar orientation. Disconnecting every incoming edge but the rightmost one at every vertex turns the map $m$ into a plane tree $T(m)$ rooted at the source, which we call the \textit{down-right tree} of the map (see the left-hand side of \cref{fig:bip_orient _with_trees} for an example). The tree $T(m)$ contains every edge of $m$, and the clockwise contour exploration\footnote{Recall that the \emph{exploration} of a tree $T$ is the visit of its vertices (or its edges) starting from the root and following the contour of the tree in the clockwise order.} of $T(m)$ identifies an ordering of the edges of $m$. We denote by $e_1,\ldots,e_{|m|}$ the edges of $m$ in this order (see again \cref{fig:bip_orient _with_trees}). The tree $T(m^{**})$ can be obtained similarly from $m$ by disconnecting every outgoing edge but the leftmost, and is rooted at the sink. The following remarkable facts hold:
\textit{The contour exploration of $T(m^{**})$ visits edges of $m$ in the order $e_{|m|},\ldots,e_1$. Moreover, one can draw $T(m)$ and $T(m^{**})$ in the plane, one next to the other, in such a way that the interface between the two trees traces a path, called \emph{interface path}\label{def:int_path}\footnote{Note that the interface path coincides with the clockwise contour exploration of $T(m)$, an example is given in the first two pictures of \cref{fig:bip_orient _with_trees}.}, from the source to the sink visiting edges $e_1,\ldots, e_{|m|}$ in this order (see the middle picture of \cref{fig:bip_orient _with_trees} for an example).}

A bijection between bipolar orientations and a specific family of two-dimensional walks in the non-negative quadrant was introduced by Kenyon, Miller, Sheffield and Wilson~\cite{MR3945746}.

\begin{defn}\label{defn:KMSW}
	Let $n\geq 1$, $m \in \mathcal O_n$. We define $\bow(m)=(X_t,Y_t)_{1\leq t\leq n} \in (\Z_{\geq 0}^2)^n$ as follows: for $1\leq t\leq n$, $X_t$ is the height in the tree $T(m)$ of the bottom vertex of $e_t$ (i.e.\ its distance in $T(m)$ from the source $s$), and $Y_t$ is the height in the tree $T(m^{**})$ of the top vertex of $e_t$ (i.e.\ its distance in $T(m^{**})$ from the sink $s'$).
\end{defn}

An example is given in the right-hand side of \cref{fig:bip_orient _with_trees}.

\begin{figure}[htbp]
	\centering
	\includegraphics[scale=0.95]{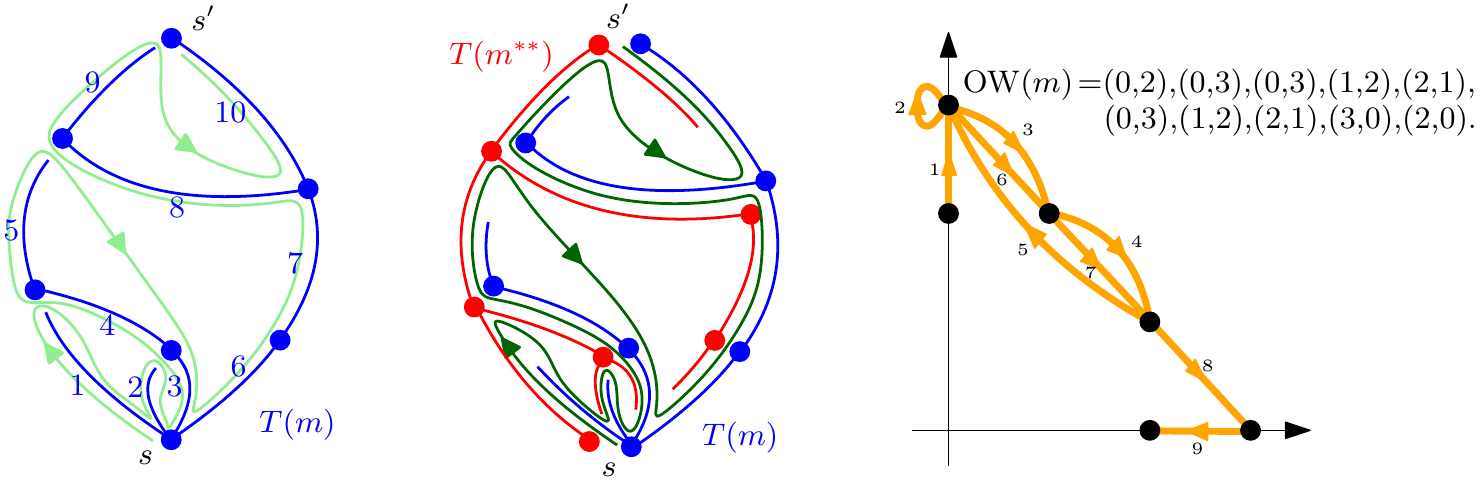}\\
	\caption{On the left-hand side the tree $T(m)$ built by disconnecting the bipolar orientation $m$ from \cref{fig:bip_orient} with the edges ordered according to the exploration process (in light green). In the middle, the two trees $T(m)$ and $T(m^{**})$ with the \emph{interface path} tracking the interface between the two trees (in dark green). On the right-hand side, the two-dimensional walk $\bow(m)$ defined in  \cref{defn:KMSW}. \label{fig:bip_orient _with_trees}}
\end{figure}

\begin{thm}[Theorem 1 of \cite{MR3945746}] \label{thm:KMSW}
	The mapping $\bow$ is a size-preserving bijection between $\mathcal O$ and the set $\mathcal W$ of walks in the non-negative quadrant $\Z_{\geq 0}^2$ starting on the $y$-axis, ending on the $x$-axis, and with increments in 
	\begin{equation}
		\label{eq:admis_steps}
		\Steps = \{(+1,-1)\} \cup \{(-i,j), i\in \Z_{\geq 0}, j\in \Z_{\geq 0}\}.
	\end{equation}
\end{thm}

We call $\gls*{tandem}$ the set of \emph{tandem walks}, as done in \cite{bousquet2019plane}. 

\medskip

We now introduce a second bijection, fundamental for our results, between bipolar orientations and Baxter permutations, discovered by Bonichon, Bousquet-Mélou and Fusy~\cite{MR2734180}.

\begin{defn}\label{defn:bobp}
	Let $n\geq 1, m\in \mathcal O_n$. Recall that every edge of the map $m$ corresponds to its dual edge in the dual map $m^*$. Let $\bobp(m)$ be the only permutation $\pi$ such that for every $1\leq i \leq n$, the $i$-th edge to be visited in the exploration of $T(m)$ corresponds to the $\pi(i)$-th edge to be visited in the exploration of $T(m^*)$.
\end{defn}

An example is given in \cref{fig:bip_orient_and_perm}.

\begin{figure}[htbp]
	\begin{minipage}[c]{0.59\textwidth}
		\centering
		\includegraphics[scale=0.4]{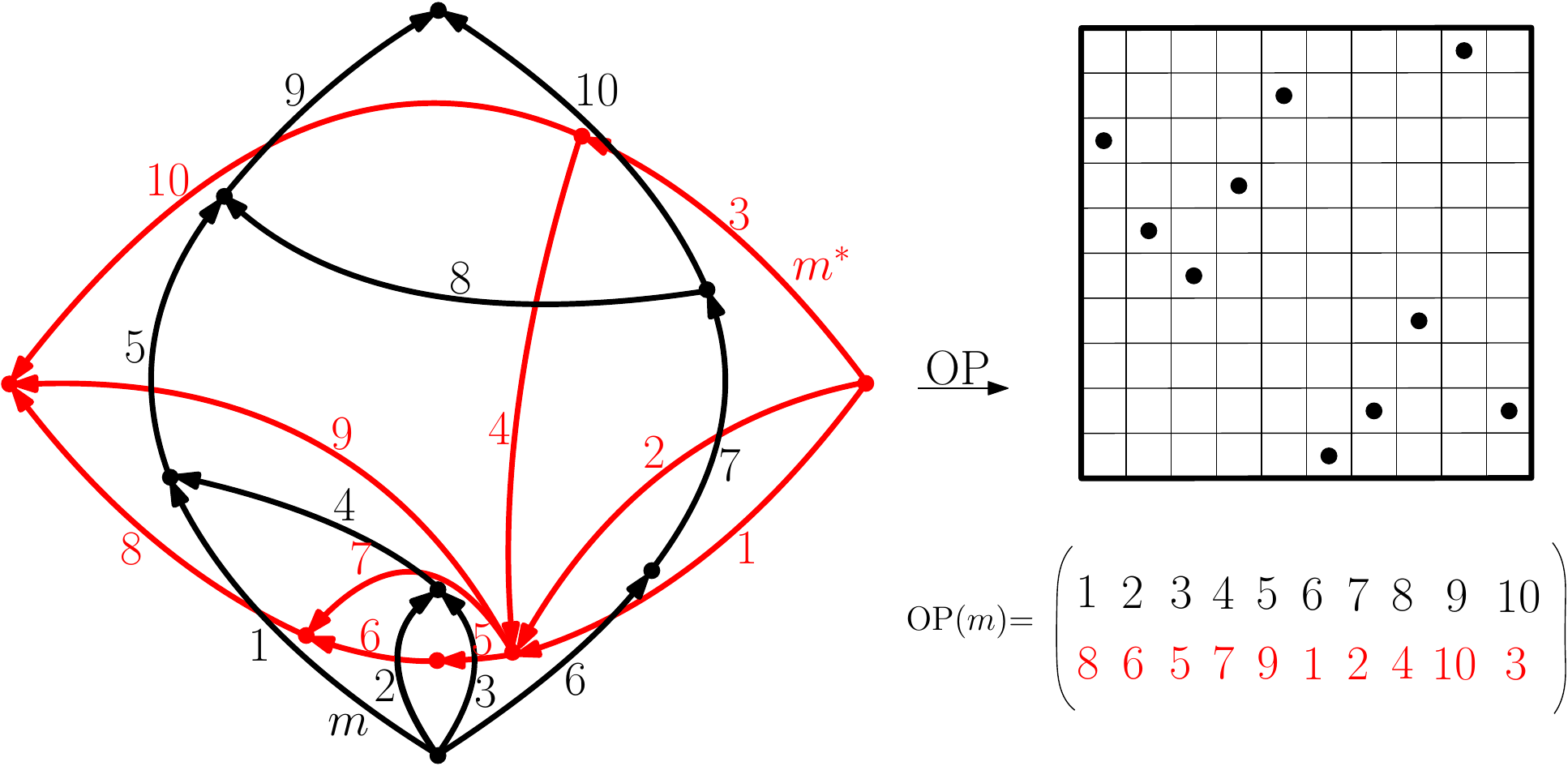}
	\end{minipage}
	\begin{minipage}[c]{0.40\textwidth}
		\caption{A schema explaining the mapping $\bobp$. On the left-hand side, the bipolar orientation $m$ and its dual $m^*$, from \cref{fig:bip_orient}. We plot in black the labeling of the edges of $m$ obtained in \cref{fig:bip_orient _with_trees} and in red the labeling of the edges of $m^*$ obtained by the same procedure. On the right-hand side, the permutation $\bobp(m)$ (together with its diagram) obtained by pairing the labels of the corresponding primal and dual edges between $m$ and $m^*$.  \label{fig:bip_orient_and_perm}}
	\end{minipage}
\end{figure}

\begin{thm}[Theorem 2 of \cite{MR2734180}]\label{thm:bobp}
	The mapping $\bobp$ is a size-preserving bijection between the set $\mathcal O$ of bipolar orientations and the set $\mathcal P$ of Baxter permutations.
\end{thm}

The definition given in \cref{defn:bobp} is a simple reformulation of the bijection presented in \cite{MR2734180}, for more details see \cite[Section 2.3]{borga2020scaling}.
We have the following additional properties of the mapping $\bobp$.

\begin{thm}[\cite{MR2734180}, Theorems 2 and 3, and Propositions 1 and 4]\label{thm:rotation}
	One can draw $m$ on the diagram of $\bobp(m)$ in such a way that every edge of $m$ passes through the corresponding point of $\bobp(m)$. Moreover, we have the following symmetry properties: 
	\begin{itemize}
		\item Denoting by $\sigma^*$ the permutation obtained by rotating the diagram of $\sigma\in \mathfrak S_n$ clockwise by angle $\pi/2$, we have $\bobp(m^*) = \bobp(m)^*$.
		\item We have $\bobp(m^{-1}) = \bobp(m)^{-1}$, where $m^{-1}$ is the symmetric image of $m$ along the vertical axis.
	\end{itemize}
\end{thm}

So far we have considered three families of objects: Baxter permutations ($\mathcal{P}$), tandem walks ($\mathcal W$), and bipolar orientations ($\mathcal{O}$). We saw that they are linked by the mappings $\bow$ and $\bobp$. 

To investigate local and scaling limits of Baxter permutations, it is natural to first prove local and scaling limits results for tandem walks and then try to transfer these convergences to permutations through the mapping $\bobp \circ \bow^{-1}$. However, the definition of this composite mapping makes it not very tractable, and our first combinatorial result is a new reformulation of it.

Consider a tandem walk $W = (X,Y) \in \mathcal W_n$ and the corresponding Baxter permutation $\sigma = \bobp \circ \bow^{-1} (W)$. 
We introduce the \textit{coalescent-walk process} driven by $W$. It is a family of discrete walks $Z = \{Z^{(i)}\}_{1\leq i \leq n}$, where $Z^{(i)}=Z^{(i)}_t$ has time indices  $t\in\{i,\ldots,n\}$ and it is informally defined as follows: $Z^{(i)}$ starts at $0$ at time $i$, takes the same steps as $Y$ when it is non-negative, takes the same steps as $-X$ when it is negative unless such a step would force $Z^{(i)}$ to become non-negative. If the latter case happens at time $j$, then $Z^{(i)}$ is forced to coalesce with $Z^{(j)}$ at time $j+1$. For a precise definition we refer the reader to \cref{sec:discrete_coal}. An illustration of a coalescent-walk process is given on the left-hand side of \cref{fig:Coal_process_exemp}.

\begin{figure}[htbp]
	\begin{minipage}[c]{0.69\textwidth}
		\centering
		\includegraphics[scale=0.56]{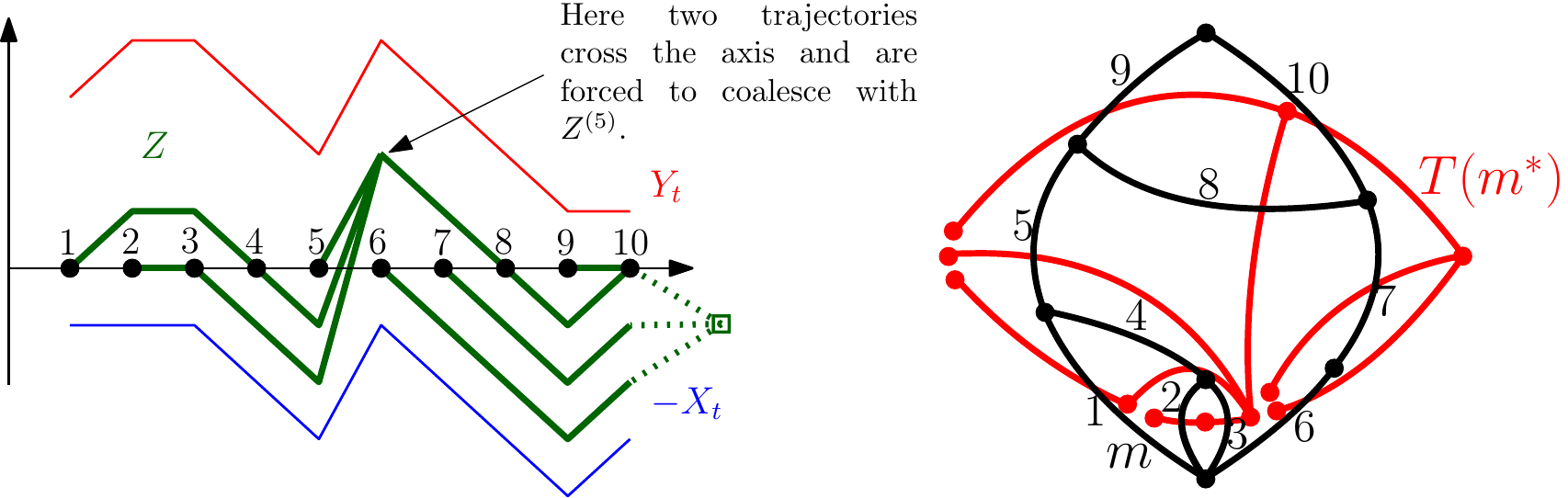}
	\end{minipage}
	\begin{minipage}[c]{0.30\textwidth}
		\caption{The coalescent-walk process $Z = \wcp(W)$ associated with the walk $W=(X,Y) = \bow(m)$. The walk $Y$ is plotted in red and $-X$ is plotted in blue. On the right-hand side, the map $m$ together with the tree $T(m^*)$ drawn in red. \label{fig:Coal_process_exemp}}
	\end{minipage}
\end{figure}

We denote by $\gls*{coal_proc}$ the set of coalescent-walk processes obtained in this way from tandem walks in $\mathcal W$, and we define $\wcp:\mathcal W \to \mathscr{C}$ to be the mapping that associates a tandem walk $W$ with the corresponding coalescent-walk process $Z$.

In a coalescent-walk process, trajectories do not cross but coalesce, hence the name. From a coalescent-walk process one can construct a permutation of the integers: If $Z\in \mathscr{C}_n$, we denote $\cpbp(Z)$ the only permutation $\pi \in \Perms_n$ such that for $i,j\in [n]$ with $i<j$, $\pi(i)<\pi(j)$ if and only if $Z^{(i)}_j<0$. We will see that $\cpbp:\mathscr{C}_n \to \mathcal P_n$, hence $\cpbp(Z)$ is a Baxter permutation. The reader can check that in the case of \cref{fig:Coal_process_exemp} we have $\cpbp(Z) = 8\,6\,5\,7\,9\,1\,2\,4\,10\,3$, which corresponds to $\bobp \circ \bow^{-1}(W)$ (see \cref{fig:bip_orient _with_trees,fig:bip_orient_and_perm}) witnessing an instance of our main combinatorial result.

\begin{thm}\label{thm:diagram_commutes}
	For all $n\in \Z_{>0}$, the following diagram of bijections commutes
	\begin{equation}
		\label{eq:comm_diagram}
		\begin{tikzcd}
			\mathcal{W}_n \arrow{r}{\wcp}  & \mathscr{C}_n \arrow{d}{\cpbp} \\
			\mathcal{O}_n \arrow{u}{\bow} \arrow{r}{\bobp}& \mathcal{P}_n
		\end{tikzcd} \;.
	\end{equation}
\end{thm}

Note that the mappings involved in the diagram are denoted using two letters that refer to the domain and co-domain. 
The key-step for the proof of \cref{thm:diagram_commutes} is the following fact, formally stated in \cref{prop:eq_trees}: \emph{Given a bipolar orientation $m$, the ``branching structure" of the trajectories of the coalescent-walk process $\wcp\circ\bow(m)$ is equal to the tree $T(m^*)$}. The reader is invited to verify it in \cref{fig:Coal_process_exemp}.

\medskip

The rest of this section is organized as follows: 
In \cref{sec:discrete_coal} we properly define coalescent-walk processes and the mappings $\cpbp$ and $\wcp$. The proof of \cref{thm:diagram_commutes} is outlined in \cref{sect:equiv_bij}. Finally, in \cref{sect:anti-invo} we generalize some constructions to the four trees characterizing a bipolar map
and its dual.

\subsection{Discrete coalescent-walk processes}
\label{sec:discrete_coal}

This section is devoted to the definition of coalescent-walk processes and our specific model of coalescent-walk processes obtained from tandem walks by the mapping $\wcp$. We then define the permutation and forest naturally associated with a coalescent-walk process.

\begin{defn}
	\label{def:discrete_coal_process}
	Let $I$ be a (finite or infinite) interval of $\Z$. We call \emph{coalescent-walk process} on $I$ a family $\{(Z^{(t)}_s)_{s\geq t, s\in I}\}_{t\in I}$ of one-dimensional walks such that:
	\begin{itemize}
		\item for every $t\in I$, $Z^{(t)}_t=0$;
		\item for $t'\geq t \in I$, if $Z^{(t)}_k\geq Z^{(t')}_k$ (resp.\ $Z^{(t)}_k\leq Z^{(t')}_k$) then $Z^{(t)}_{k'}\geq Z^{(t')}_{k'}$ (resp.\ $Z^{(t)}_{k'}\leq Z^{(t')}_{k'}$) for every $k'\geq k$.
	\end{itemize}
\end{defn}

Note that, as a consequence, if there is a time $k$ such that $Z^{(t)}_k= Z^{(t')}_k$, then $Z^{(t)}_{k'}=Z^{(t')}_{k'}$ for every $k'\geq k$.
In this case, we say that $Z^{(t)}$ and $Z^{(t')}$ are \emph{coalescing} and we call \emph{coalescent point} of $Z^{(t)}$ and $Z^{(t')}$ the point $(\ell,Z^{(t)}_\ell)$ such that  $\ell=\min\{k\geq \max\{t,t'\}|Z^{(t)}_{k}=Z^{(t')}_{k}\}$.

We denote by $\gls*{coal_proc_I}$ the set of coalescent-walk processes on some interval $I$.

\subsubsection{The coalescent-walk process associated with a two-dimensional walk}
\label{sect: bij_walk_coal}
We introduce formally the coalescent-walk processes driven by some specific two-dimensional walks that include tandem walks. Let $I$ be a (finite or infinite) interval of $\Z$. 
We denote by $\Walks_\Steps(I)$ the set of walks (considered up to an additive constant) indexed by $I$, and that take their increments in $\Steps$, defined in \cref{eq:admis_steps} page~\pageref{eq:admis_steps}.
\begin{defn}\label{defn:distrib_incr_coal}
	Let $W\in\Walks_\Steps(I)$ and write $W_t = (X_t,Y_t)$ for $t\in I$. The \emph{coalescent-walk process associated with} $W$
	is the family of walks $\wcp(W) = \{Z^{(t)}\}_{t\in I}$, defined for $t\in I$ by $Z^{(t)}_t=0,$ and for all $\ell\geq t$ such that $\ell+1 \in I$,
	
	\begin{itemize}
		\item if $W_{\ell+1}-W_{\ell}=(1,-1)$ then $Z^{(t)}_{\ell+1}-Z^{(t)}_{\ell}=-1$;
		\item if $W_{\ell+1}-W_{\ell}=(-i,j)$, for some $i,j\geq 0$, then
		\begin{equation}
			Z^{(t)}_{\ell+1}-Z^{(t)}_{\ell}=
			\begin{cases}
				j, &\quad\text{if}\quad Z^{(t)}_{\ell}\geq0,\\
				i,&\quad\text{if}\quad Z^{(t)}_{\ell}<0\text{ and }Z^{(t)}_{\ell}<-i,\\
				j-Z^{(t)}_{\ell},&\quad\text{if}\quad Z^{(t)}_{\ell}<0\text{ and }Z^{(t)}_{\ell}\geq -i.
			\end{cases}
		\end{equation} 
	\end{itemize}
\end{defn}
Note that this definition is invariant by addition of a constant to $W$. We check easily that $\wcp(W)$ is a coalescent-walk process meaning that $\wcp$ is a mapping $\Walks_\Steps(I) \to \Coals(I)$.
Recall that
$\mathscr{C}_n= \wcp(\mathcal W_n)$ and $\mathscr{C}=\cup_{n\in \Z_{\geq 0}}\mathscr{C}_n.$ For two examples, we refer the reader to the left-hand side of \cref{fig:Coal_process_exemp} (for the case of tandem walks) and to \cref{example_discrete_coal} (for the case of non tandem walks) . We finally suggest to the reader to compare the formal \cref{defn:distrib_incr_coal} with the more intuitive definition given in \cref{sect:coal_walk_intro}.

\begin{figure}[htbp]
	\centering
	\includegraphics[scale=.7]{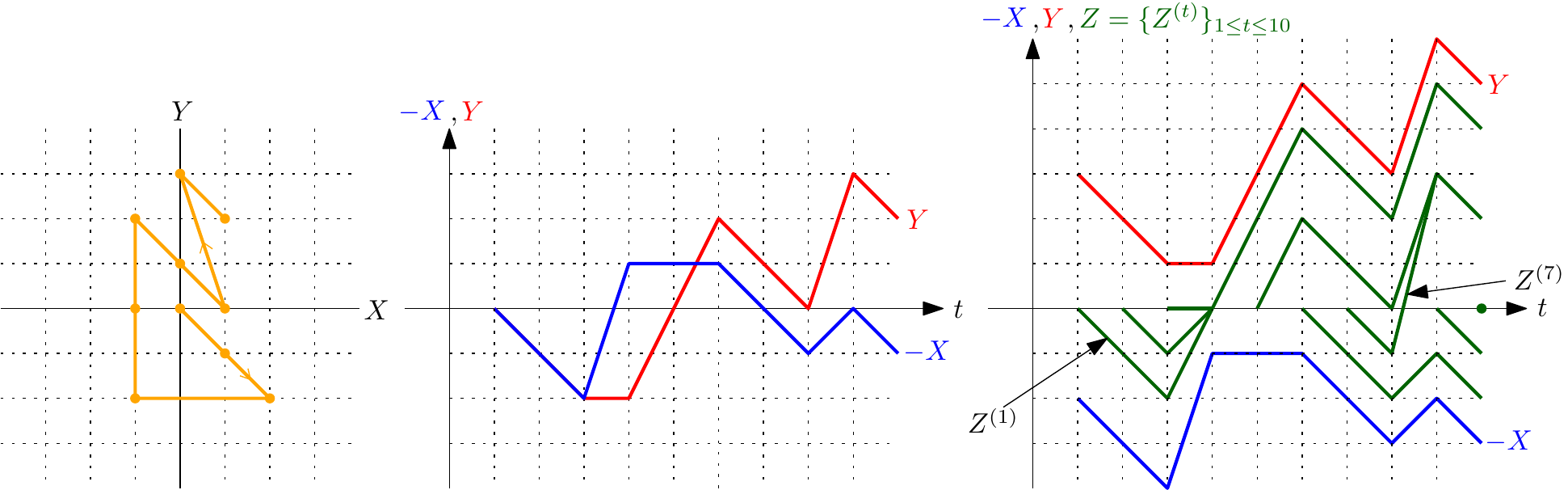}\\
	\caption{Construction of the coalescent-walk process associated with the orange walk $W=(W_t)_{1\leq t\leq 10}$ on the left-hand side. In the middle diagram the two walks $Y$ (in red) and $-X$ (in blue) are plotted. Finally, on the right-hand side the two walks are shifted (one towards the top and one to the bottom) and the ten walks of the coalescent-walk process are plotted in green. \label{example_discrete_coal}}
\end{figure}

\begin{obs}
	The $y$-coordinates of the coalescent points of a coalescent-walk process obtained in this fashion are non-negative. 
\end{obs}

\subsubsection{The permutation associated with a coalescent-walk process}\label{sect:from_coal_to_perm}

Given a coalescent-walk process $Z = \{Z^{(t)}\}_{t\in I} \in \Coals(I)$ defined on a (finite or infinite) interval $I$, we can define a binary relation $\leq_Z$ on $I$ as follows:

\begin{equation}\label{eq:coal_to_perm}
	\begin{cases}i\leq_Z i,\\
		i\leq_Z j,&\text{ if }i<j\text{ and }\ Z^{(i)}_j<0,\\
		j\leq_Z i,&\text{ if }i<j\text{ and }\ Z^{(i)}_j\geq0.\end{cases}
\end{equation}

\begin{prop}[{\cite[Proposition 2.9]{borga2020scaling}}]
	\label{prop:tot_ord}
	$\leq_Z$ is a total order on $I$. 
\end{prop}

This definition allows to associate a permutation with a coalescent-walk process on $[n]$.

\begin{defn}Fix $n\in\Z_{\geq 0}$. Let $Z = \{Z^{(t)}\}_{i\in [n]} \in \Coals([n])$ be a coalescent-walk process on $[n]$. Denote $\cpbp(Z)$ the unique permutation $\sigma \in \Perms_n$ such that for all $1\leq i, j\leq n$, $$\sigma(i)\leq\sigma(j) \iff i\leq_Z j.$$
\end{defn}

We will furnish an example that clarifies the definition above in \cref{exmp:coal_tree} where it will be also clear that this definition is equivalent to the one given in \cref{sect:coal_walk_intro} (see the discussion above \cref{thm:diagram_commutes}).
We have that pattern extraction in the permutation $\cpbp(Z)$ depends only on a finite number of trajectories (see \cref{prop:patterns} below), a key step towards proving permuton convergence for uniform Baxter permutations.

\begin{prop}
	\label{prop:patterns}
	Let $\sigma$ be a permutation obtained from a coalescent-walk process $Z=\{Z^{(t)}\}_{1\leq t\leq n}$ via the mapping $\cpbp$. Let $I=\{i_1<\dots<i_k\}\subseteq[n]$. Then $\pat_I(\sigma)=\pi$ if and only if the following condition holds: for all $1\leq \ell< s \leq k,$
	$$Z^{(i_\ell)}_{i_s} \geq 0   \iff  \pi(s)<\pi(\ell).$$
\end{prop}
This proposition is immediate once one notes that
\begin{equation}
	Z^{(i_\ell)}_{i_s} \geq 0  \iff i_s\leq_Z i_\ell \iff \sigma(i_s)\leq\sigma(i_\ell)\iff  \pi(s)<\pi(\ell).
\end{equation}

\subsubsection{The coalescent forest of a coalescent-walk process}\label{sect:from_coal_to_trees} 
As previously mentioned, given a coalescent-walk process on $[n]$, the plane drawing of the family of the trajectories $\{Z^{(t)}\}_{t\in I}$ identifies a natural tree structure, more precisely, a \emph{$\Z$-planted plane forest}, as per the following definition.
\begin{defn}
	A \emph{$\Z$-planted plane  tree} is a rooted plane tree such that the root has an additional parent-less edge which is equipped with a number in $\Z$ called its (root-)index.
	
	A \emph{$\Z$-planted plane forest} is an ordered sequence of $\Z$-planted plane trees $(T_1,\ldots T_\ell)$ such that the (root-)indices are weakly increasing along the sequence of trees. A $\Z$-planted plane forest admits an exploration process, which is the concatenation of the exploration processes of all the trees, following the order of the sequence.
\end{defn}

An example of a $\Z$-planted plane forest is given in the middle of \cref{coal_tree} (each tree is drawn with the root on the right; trees are ordered from bottom to top; the root-indices are indicated on the right of each tree).

We give here a formal definition of the $\Z$-planted plane forest corresponding to a coalescent-walk process. For a more informal description, we suggest to look at \cref{coal_tree} and at the description given in \cref{exmp:coal_tree}.

\begin{defn}\label{defn:corr_trees}
	Let $Z$ be a coalescent-walk process on a finite interval $I$. Its \emph{forest}, denoted $\fortree(Z)$ for ``labeled forest", is a $\Z$-planted plane forest with additional edge labels in $I$, defined as follows:
	\begin{itemize}
		\item the edge-set is $I$, vertices are identified with their parent edge, and the edge $i\in I$ is understood as bearing the label $i$.
		\item For any pair of edges $(i,j)$ with $i<j$, $i$ is a descendant of $j$ if $(j,0)$ is the coalescent point of $Z^{(i)}$ and $Z^{(j)}$.
		\item Children of a given parent are ordered by $\leq_Z$.
		\item The different trees of the forests are ordered such that their root-edges are in increasing $\leq_Z$-order.
		\item The index of the tree whose root-edge has label $i$ is the value $Z^{(i)}_{\max I}$.
	\end{itemize}
\end{defn}

\begin{figure}[htbp]
	\centering
	\includegraphics[scale=0.85]{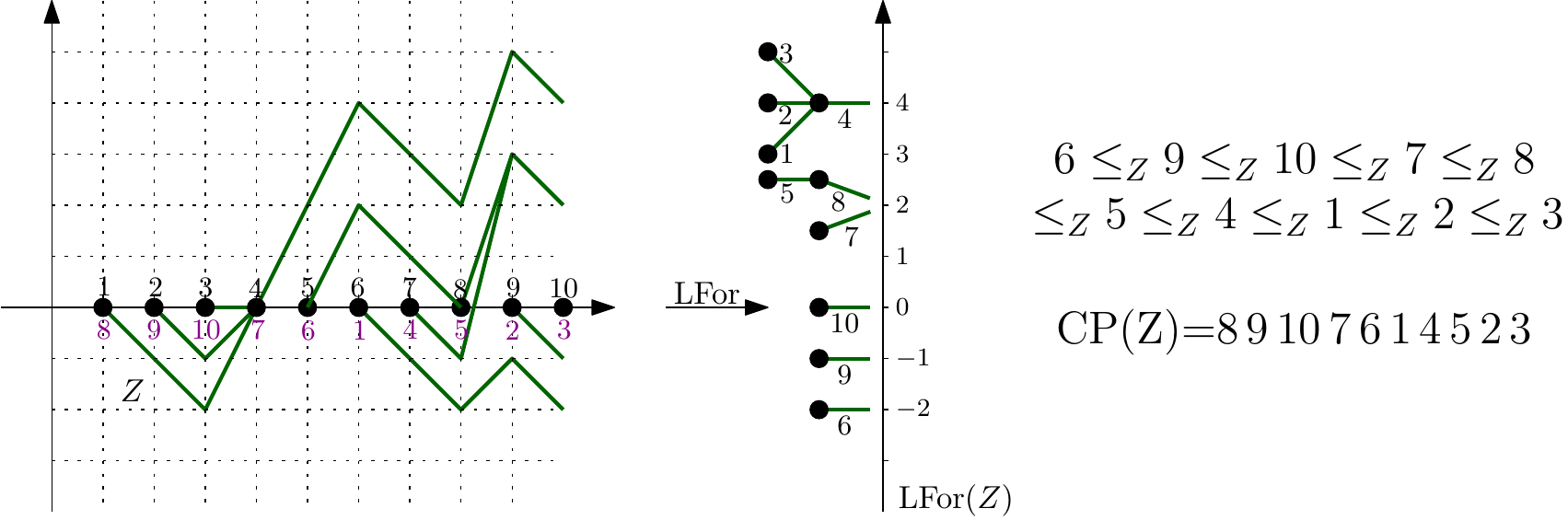}\\
	\caption{In the middle of the picture, the forest $\fortree(Z)$ corresponding to the coalescent-walk process represented on the left that was obtained in \cref{example_discrete_coal}. How this forest is constructed is explained in \cref{exmp:coal_tree}. On the right-hand side we also draw the associated total order $\leq_Z$ and the associated permutation $\cpbp(Z)$.  \label{coal_tree}}
\end{figure}

 An immediate consequence of the properties of a coalescent-walk process is the following.

\begin{prop}\label{prop:fortree_cpbp}
	$\fortree(Z)$ is a $\Z$-planted plane forest, equipped with a labeling of its edges by the values of $I$.
	Moreover the total order $\leq_Z$ on $I$ coincides with the total order given by the exploration process of the forest $\fortree(Z)$.
\end{prop}

\begin{rem}\label{rk:cpbp_through_fortree}
	In the case where $I=[n]$ for some $n\in\Z_{\geq 0}$, the permutation $\pi = \cpbp(Z)$ is readily obtained from $\fortree(Z)$: for $1\leq i \leq n$, $\pi(i)$ is the position in the exploration of $\fortree(Z)$ of the edge with label $i$.
\end{rem}

\begin{exmp}\label{exmp:coal_tree}
	\cref{coal_tree} shows the forest of trees $\fortree(Z)$ corresponding to the coalescent-walk process $Z=\{Z^{(t)}\}_{t\in[10]}$ plotted on the left-hand side (where the forest is plotted from bottom to top). It can be obtained by marking with ten dots the points $\{(t,Z^{(t)}_t=0)\}_{t\in[10]}$. The edge structure of the trees in $\fortree(Z)$ is given by the green lines starting at each dot, and interrupted at the next dot. The lines that go to the end uninterrupted (for example this is the case of the line starting at the fourth dot), correspond to the root-edges of the different trees. The indices of these root-edges are determined by the final heights of the corresponding interrupted lines.   
	The plane structure of $\fortree(Z)$ is inherited from the drawing of $Z$ in the plane. 
	
	We determine the order $\leq_Z$ by considering the exploration process of the forest:  $6\leq_Z 9\leq_Z 10\leq_Z 7 \leq_Z 8 \leq_Z 5 \leq_Z 4 \leq_Z 1 \leq_Z 2 \leq_Z 3$. As a result, $\cpbp(Z) =8\,9\,10\,7\,6\,1\,4\,5\,2\,3$. 
	
	Equivalently, we can pull back (on the points $(t,Z^{(t)}_t=0)$ of the coalescent-walk process) the position of the edges in the exploration process (these positions are written in purple), and then $\cpbp(Z)$ is obtained by reading these numbers from left to right. 	
	
\end{exmp}

\subsection{From walks to Baxter permutations via coalescent-walk processes}
\label{sect:equiv_bij}
We are now in position to explain how to prove the main result of this section, that is \cref{thm:diagram_commutes}.
We are going to show that $\bobp = \cpbp\circ\wcp \circ \bow.$
The key ingredient is to show that the dual tree $T(m^*)$ of a bipolar orientation can be recovered from its encoding two-dimensional walk by building the associated coalescent-walk process $Z$ and looking at the coalescent forest $\fortree(Z)$.
More precisely, let $W = (W_t)_{t\in [n]} = \bow (m)$ be the walk encoding a given bipolar orientation $m$, and $Z = \wcp(W)$ be the corresponding coalescent-walk process. Then the following admitted result holds.

\begin{prop}[{\cite[Proposition 2.17]{borga2020scaling}}]\label{prop:eq_trees} 
	The following are equal:
	\begin{itemize}
		\item the dual tree $T(m^*)$ with edges labeled according to the order given by the exploration of $T(m)$;
		\item the tree obtained by attaching all the edge-labeled trees of $\fortree(Z)$ to a common root.
	\end{itemize}
\end{prop}
\cref{thm:diagram_commutes} then follows immediately, by construction of $\bobp(m)$ from $T(m^*)$ and $T(m)$ (\cref{thm:bobp}) and of $\cpbp(Z)$ from $\fortree(Z)$ (see \cref{rk:cpbp_through_fortree}). \cref{prop:eq_trees} is illustrated in an example in \cref{fig:bip_orient_with_coal}.

\begin{figure}[htbp]
	\begin{minipage}[c]{0.65\textwidth}
		\centering
		\includegraphics[scale=0.55]{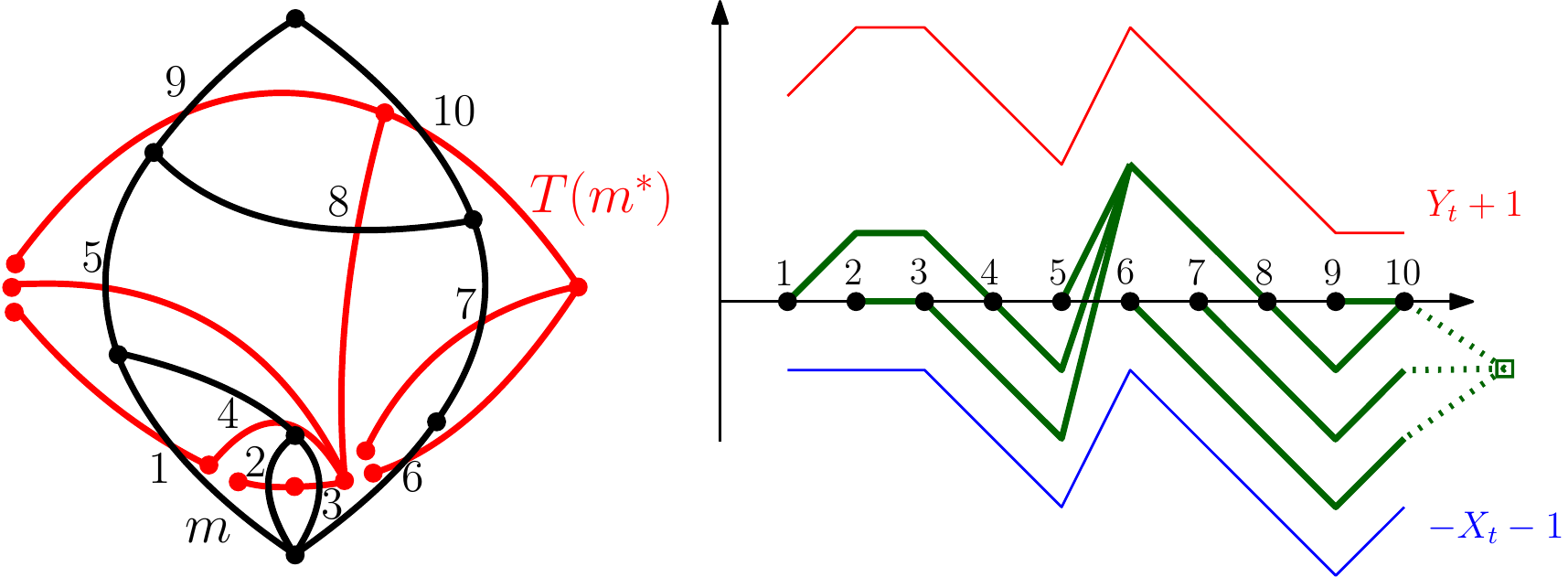}
	\end{minipage}
	\begin{minipage}[c]{0.34\textwidth}
		\caption{In the left-hand side the map $m$ from \cref{fig:bip_orient_and_perm} with the dual tree $T(m^*)$ in red with edges labeled according to the order given by the exploration of $T(m)$. In the right-hand side the associated coalescent-walk process $Z=\wcp\circ\bow(m)$. Note that the red tree (with its labeling) and the green tree (with its labeling) are equal.
			\label{fig:bip_orient_with_coal}}
	\end{minipage}
\end{figure}

An interesting corollary of \cref{prop:eq_trees}, fundamental for later purposes (see the proof of \cref{thm:joint_scaling_limits} page \pageref{proof_of_thm}), is the following result. Given a coalescent-walk process $Z$, we introduce the discrete local time process $L_Z = \left(L_Z^{(i)}(j)\right)$, $1\leq i \leq j \leq n$, defined by
\begin{equation}
	\label{eq:local_time_process}
	L_Z^{(i)}(j)=\#\left\{k\in [i,j]\middle|Z^{(i)}_k=0\right\}.
\end{equation}

\begin{cor}\label{cor:local_time}
	Let $(m,W,Z,\sigma)$ be objects of size $n$ in $\mathcal{O}\times\mathcal{W}\times\mathscr{C}\times\mathcal{P}$ connected by the commutative diagram in \cref{eq:comm_diagram}. 
	Then the height process\footnote{We recall that the height process of a tree $T$ is the sequence of integers obtained by recording for each visited vertex (following the exploration of $T$) its distance to the root.} $(X^*_i)_{i\in [n]}$ of $T(m^*)$ is equal to
	$\left(L_Z^{(\sigma^{-1}(i))}(n)\right)_{i\in [n]}$. In other words, $$X^*_{\sigma(i)}=L_Z^{(i)}(n)-1,\quad{i\in [n]}.$$
\end{cor}

\subsection{The four trees of bipolar orientations}
\label{sect:anti-invo}

We now want to investigate the relations between the four trees $T(m)$, $T(m^*)$, $T(m^{**})$, $T(m^{***})$ of a bipolar orientation $m$ (and its dual maps) and some corresponding coalescent-walk processes. The results presented in this section will be also useful for \cref{sec:final}.

Let $m$ be a bipolar orientation,  $W=\bow(m)$ be the corresponding tandem walk, and $Z=\wcp\circ\bow(m)=\wcp(W)$ be the corresponding coalescent-walk process. We set

\begin{itemize}
	\item $LT(m^*)$ to be the tree $T(m^*)$ with edges labeled according to the exploration of $T(m)$.
	\item $\labtree(Z)$ to be the tree obtained by attaching all the trees of $\fortree(Z)$ to a common root.
\end{itemize}

We saw in \cref{prop:eq_trees} that $LT(m^*)=\labtree(Z)$.
We now want to recover from the walk $W$ (and a new associated coalescent-walk process) the tree $LT(m^{***})$, which is defined as the tree $T(m^{***})$ with edges labeled according to the exploration of $T(m^{**})$.

\begin{defn}
	Fix $n\in\Z_{>0}.$
	Given a one dimensional walk $X=(X_t)_{t\in [n]}$ we denote by $\cev{X}$ the time-reversed walk $(X_{n+1-t})_{t\in [n]}$.
	Given a two-dimensional walk $W=(X,Y)=(X_t,Y_t)_{t\in [n]}$, we denote by $\cev{W}$ the walk $(\cev{Y},\cev{X})$ where time is reversed and coordinates are swapped.
\end{defn}

\begin{prop}[{\cite[Proposition 2.24]{borga2020scaling}}]\label{prop:rev_coal_prop}
	Let $m$ be a bipolar orientation and  $W=\bow(m)$ be the corresponding walk. Consider the walk $\cev{W}$ and the corresponding coalescent-walk process $\cev{Z}\coloneqq\wcp(\cev{W})$.
	Then 
	$\bow (m^{**})=\cev{W}$ and $LT(m^{***})= \labtree(\cev{Z})$.
\end{prop}

An example of the coalescent-walk process $\cev{Z}=\wcp(\cev{W})$ is given on the bottom-left picture of \cref{fig:anti-invo} in the case of the bipolar orientation $m$ considered in \cref{fig:bip_orient}.

\medskip

We now know that given a bipolar orientation $m$ and the corresponding walk $W=\bow(m)$, we can read the trees $LT(m^*)$ and $LT(m^{***})$ in the coalescent-walk processes $Z=\wcp(W)$ and $\cev{Z}=\wcp(\cev{W})$ respectively.
Obviously, considering the bipolar map $m^*$ and the corresponding walk $W^*=\bow(m^*)$, we can read the trees $LT(m^{**})$ and $LT(m^{****})=LT(m)$ in the coalescent-walk processes $Z^*=\wcp(W^*)$ and $\cev{Z}^*=\wcp(\cev{W}^*)$ respectively. 

Actually, we can determine the walk $W^*=\bow(m^*)$ directly from the coalescent-walk processes $Z$ and $\cev{Z}$, as explained in \cref{prop:anti-inv} below. 
We recall that the discrete local time process $L_Z = \left(L_Z^{(i)}(j)\right)$, $1\leq i \leq j \leq n$,  was defined by
$L_Z^{(i)}(j)=\#\left\{k\in [i,j]\middle|Z^{(i)}_k=0\right\}.$ We also recall (see \cref{thm:rotation}) that $\sigma^*$ denotes the permutation obtained by rotating the diagram of a permutation $\sigma$ clockwise by angle $\pi/2$.

\begin{prop}[{\cite[Proposition 2.25]{borga2020scaling}}]\label{prop:anti-inv}
	Let $m$ be a bipolar orientation of size $n$. Set $W^*=(X^*,Y^*)=\bow(m^*)$, $\sigma=\bobp(m)$, $W=\bow(m)$, $Z=\wcp(W)$ and $\cev{Z}=\wcp(\cev{W})$. Then
	\begin{equation}
	(X^*_{i})_{i\in[n]}=\left(L_Z^{(\sigma^{-1}(i))}(n)-1\right)_{i\in [n]}\quad\text{and}\quad
	(Y^*_{i})_{i\in[n]}=\left(L_{\cev{Z}}^{(\sigma^{*}(i))}(n)-1\right)_{i\in [n]}.
	\end{equation} 
\end{prop}

The coalescent-walk processes $Z_1$, $\cev{Z}_1$, $Z_2$ and $\cev{Z}_2$ corresponding to the edge-labeled trees $LT(m^{*})$, $LT(m^{***})$, $LT(m^{**})$ and $LT(m)$, for our running example, are plotted in \cref{fig:anti-invo}. 

\begin{figure}[H]
	\includegraphics[scale=.45]{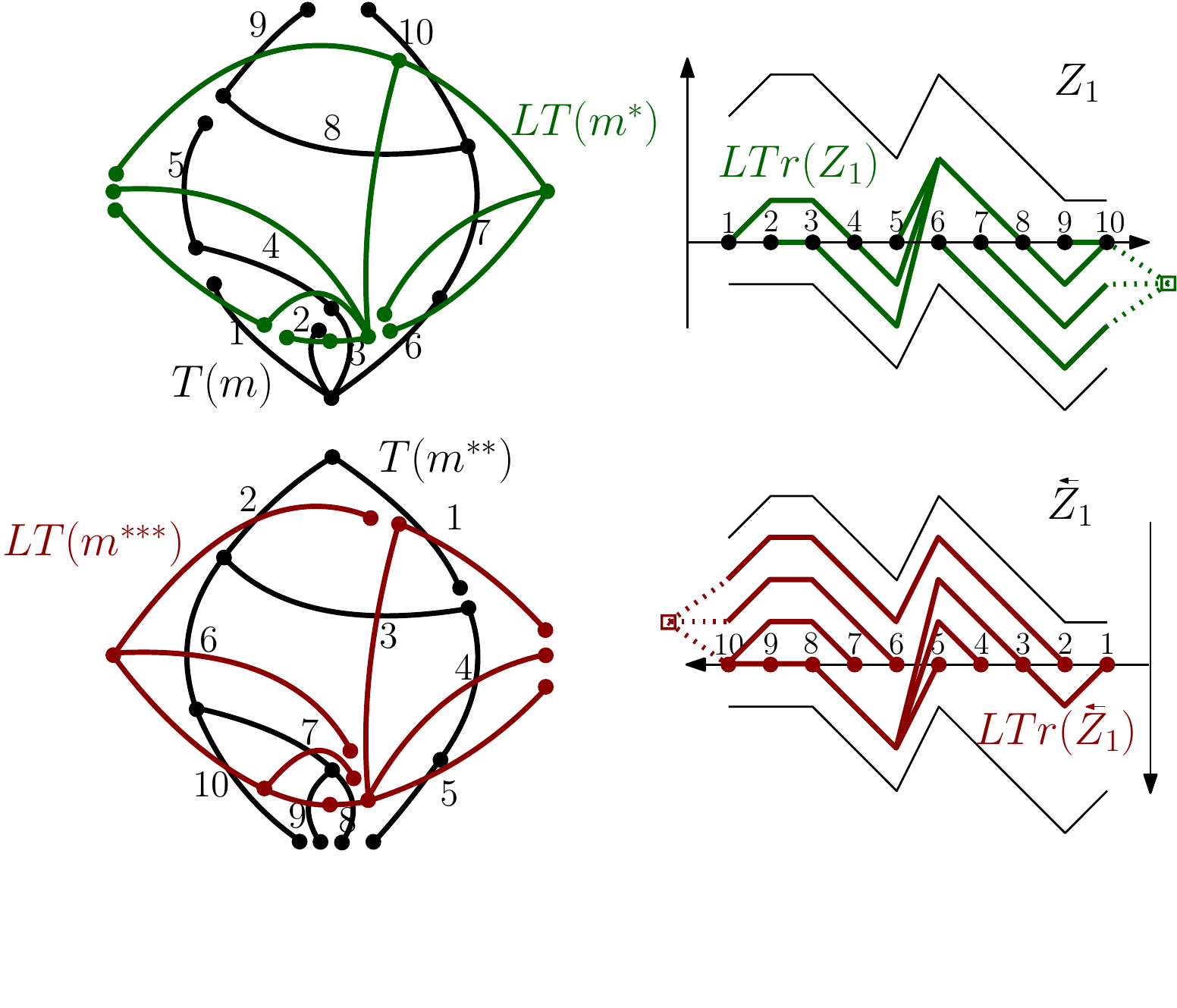}
	\includegraphics[scale=.45]{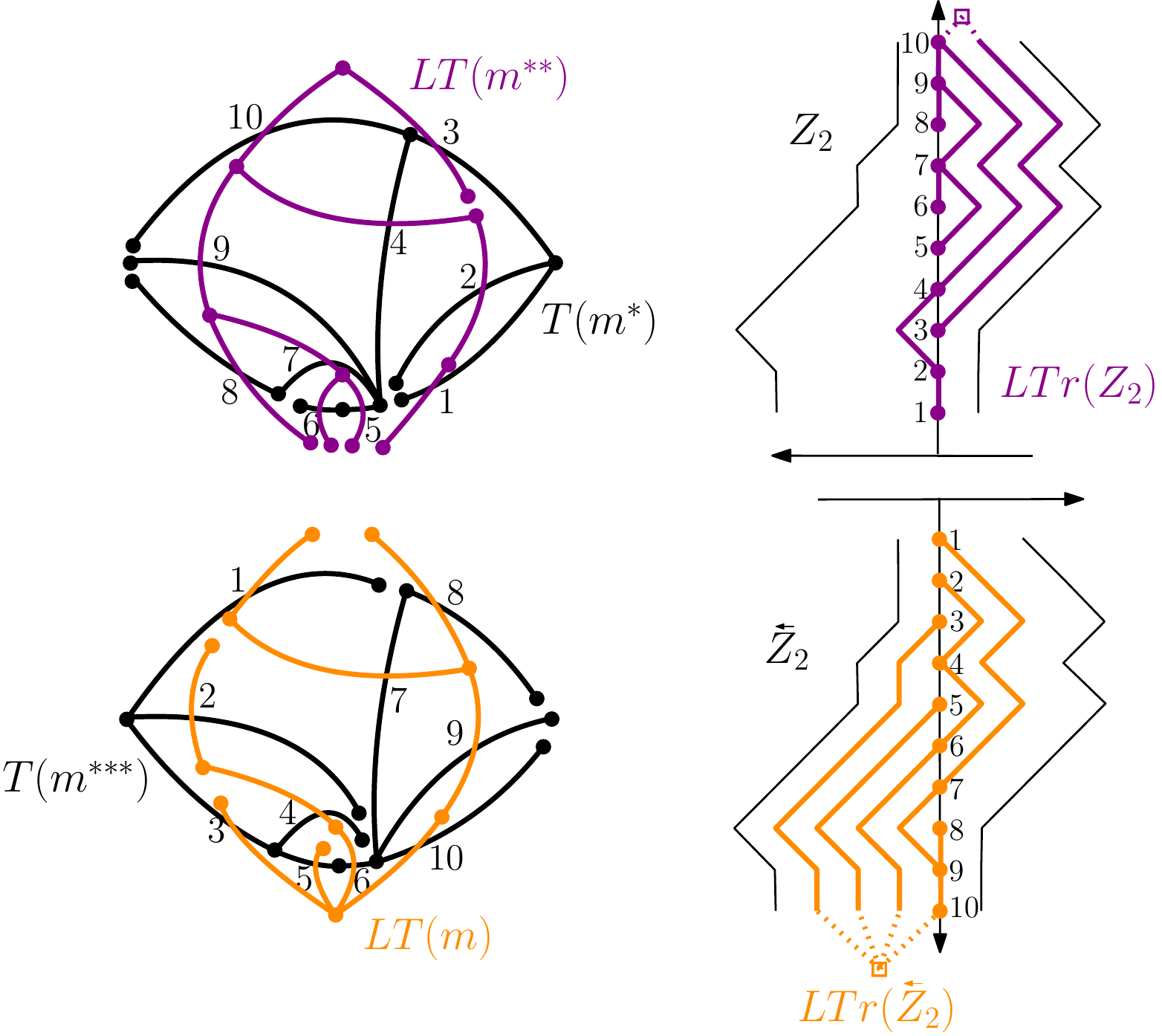}
	\caption{\label{fig:anti-invo}On the left-hand side the coalescent-walk processes $Z_1$ and $\cev{Z}_1$ and the corresponding edge-labeled trees $LT(m^{*})$ and $LT(m^{***})$. On the right-hand side the coalescent-walk processes $Z_2$ and $\cev{Z}_2$, and the corresponding edge-labeled trees $LT(m^{**})$ and $LT(m)$. We oriented the coalescent-walk processes in such a way that the comparison between trees is convenient.}
\end{figure}

    \chapter{Local limits: Concentration \& non-concentration phenomena}\label{chp:local_lim}
\chaptermark{Local limits}
  
\begin{adjustwidth}{8em}{0pt}
	\emph{In which we prove local convergence for various models of random permutations. We show that in many cases a concentration phenomenon occurs, namely the limiting proportion of consecutive patterns is deterministic. On the other hand, we show that this phenomenon is not valid for square permutations. This leads to the open question of characterizing which models show a concentration phenomenon.}
\end{adjustwidth}

\bigskip

\bigskip

\section{Overview of the main results}

We saw in \cref{chp:conv_theories} that local convergence for permutations is equivalent to convergence of proportions of consecutive patterns. In particular we have seen that a sequence of random permutations quenched B--S converges if the corresponding proportions of consecutive patterns jointly converge in distribution (see \cref{strongbsconditions} page \pageref{strongbsconditions}). In the next sections we consider many families of permutations $\mathcal C$, mainly pattern-avoiding ones, and we show that if  $\bm\sigma_n$ denotes a uniform permutation of size $n$ in $\mathcal C$, there exist deterministic numbers $(\gamma_{\pi,\mathcal C})_{\pi\in\mathcal{S}}$ in $[0,1]$ such that
\begin{equation}
	\pcocc(\pi,\bm\sigma_n) \stackrel{P}{\longrightarrow}  \gamma_{\pi,\mathcal C}, \quad\text{for all}\quad \pi\in\mathcal{S}.
\end{equation}
This kind of convergence results will be referred to as \emph{concentration phenomena}, namely the \emph{random} sequence $\pcocc(\pi,\bm\sigma_n)$ converges to a \emph{deterministic} limiting point $\gamma_{\pi,\mathcal C}$.

These results could let us suspect that every uniform permutation in a pattern-avoiding family shows a concentration phenomenon for the proportion of consecutive patterns. In \cref{thm:local_conv} we will see that this is not the case. Indeed, if $\bm\sigma_n$ is a uniform square permutation\footnote{Recall that square permutations are permutations that avoid the sixteen patterns of size $5$ with an internal point, therefore they are an example of pattern-avoiding family.} of size $n$, then
\begin{equation}
	\pcocc(\pi,\bm\sigma_n) \stackrel{d}{\longrightarrow}  \bm\Gamma_{\pi,\mathcal C}, \quad\text{for all}\quad \pi\in\mathcal{S},
\end{equation}
where $\bm\Gamma_{\pi,\mathcal C}$ are (non-trivial) random variables in $[0,1]$.

These results lead to the following open question: Let $\mathcal C$ be a pattern-avoiding family.
\begin{align}
	\text{\emph{Under which conditions on $\mathcal C$ does a uniform permutation in $\mathcal C$ show a}}\\
	\text{\emph{concentration phenomenon for the proportion of consecutive patterns?}}
\end{align}
We point out that a similar question was raised for (classical) patterns by Janson \cite[Remark 1.1]{janson2018patterns}. He notices that, in some pattern-avoiding families, we have concentration for the (classical) pattern occurrences around their mean, in others not. Also in this case the author did not furnish an answer.

\medskip

In the following sections we present several local limits results. For convenience, according to the model we are studying, we state our results using one of the various equivalent characterizations given in \cref{strongbsconditions} or \cref{detstrongbsconditions}.
Moreover, in \cref{thm:main_thm_CLT} we furnish a stronger result, proving a central limit theorem for proportions of consecutive patterns of permutations encoded by generating trees.

Most of the stated results prove existence of a random infinite rooted permutation $\bm{\sigma}^\infty_{\mathcal C}$ without explicitly constructing this limiting object. Nevertheless, in most cases, we provide an explicit construction of $\bm{\sigma}^\infty_{\mathcal C}$ at the end of the various sections.

We end the chapter proving BS-convergence for Baxter permutations but also for some discrete objects which are not permutations, specifically bipolar orientations, tandem walks, and coalescent-walk processes.

\subsection{A concentration phenomenon for $\rho$-avoiding permutations with $|\rho|=3$}\label{sect:loc_three_av}

Here we focus on the classes of $\rho$-avoiding permutations, for $|\rho|=3$.
Our two main results are the two following theorems.
\begin{thm}
	\label{thm_1}
	For any $n\in\Z_{>0}$, let $\bm{\sigma}_n$ be a uniform random $231$-avoiding permutation of size $n$. The following convergence in probability holds,
	\begin{equation}
	\label{primaeq}
	\widetilde{\cocc}(\pi,\bm{\sigma}_n)\stackrel{P}{\to}\frac{2^{\#\emph{LRmax}(\pi)+\#\emph{RLmax}(\pi)}}{2^{2|\pi|}},\quad\text{for all}\quad\pi\in\emph{Av}(231).
	\end{equation}
\end{thm}

\begin{thm}
	\label{thm_2}
	For any $n\in\Z_{>0}$, let $\bm{\sigma}_n$ be a uniform random $321$-avoiding permutation of size $n$. The following convergence in probability holds,
	\begin{equation}
	\label{shsshgea}
	\widetilde{\cocc}(\pi,\bm{\sigma}_n)\stackrel{P}{\to}\begin{cases}
	\frac{|\pi|+1}{2^{|\pi|}} &\quad\text{if }\pi=12\dots|\pi|,\\
	\frac{1}{2^{|\pi|}} &\quad\text{if }\cocc(21,\pi^{-1})=1, \\
	0 &\quad\text{otherwise,} \\ 
	\end{cases}\quad\text{for all$\quad\pi\in\emph{Av}(321)$.}
	\end{equation}
\end{thm}

\begin{obs}
	By symmetry, these two results cover all cases of $\rho$-avoiding permutations with $|\rho|=3.$ Indeed, every permutation of size three is in the orbit of either $231$ or $321$ by applying reverse (symmetry of the diagram w.r.t the vertical axis) and complementation (symmetry of the diagram w.r.t the horizontal axis). Beware that inverse (symmetry of the diagram w.r.t the principal diagonal) cannot be used since it does not preserve consecutive pattern occurrences.
\end{obs}

Here are some interesting remarks about Theorems \ref{thm_1} and \ref{thm_2}.
\begin{itemize}
	\item A first important fact is the concentration phenomenon.
	\item The second important fact is the different behavior of the two models of Theorem \ref{thm_1} and Theorem \ref{thm_2}: the first limiting density
	has full support on the space of 231-avoiding permutations, whereas the second gives positive measure only to 321-avoiding permutations whose inverse have at most one descent. 
	Indeed, despite having the same enumeration sequence, $\text{Av}(231)$ and $\text{Av}(321)$ are often considered as behaving really differently. Our results give new evidence of this belief.
	\item A further remarkable fact are the explicit expressions for the limiting proportion of consecutive occurrences. As we will see in some future sections, it is not always doable to find such closed formulas for the limiting proportions. For instance in \cref{thm:examples_ok} we will extend the results stated in Theorems \ref{thm_1} and \ref{thm_2} proving a central limit theorem for $\widetilde{\cocc}(\pi,\bm{\sigma}_n)$. Nevertheless, we will only obtain implicit expressions for the limiting mean and variance. 
	\item Explicit constructions of the limiting random total orders on $\Z$ for the B--S convergence are provided in Sections \ref{explcon} and \ref{explcon2}. 
\end{itemize}

Although the two theorems have very similar statements and in both models we use the bijections between $\rho$-avoiding permutations and rooted plane trees introduced in \cref{231bij,premres_321}, the two proofs involve different techniques.
\begin{itemize}
	\item For the proof of Theorem \ref{thm_1} we use the Second moment method. We study the  asymptotic behavior of the first and the second moments of $\widetilde{\cocc}(\pi,\bm{\sigma}_n)$ applying a technique introduced by Janson \cite{janson2003wiener,janson2017patterns}. Instead of studying uniform trees with $n$ vertices, we focus on specific families of binary Galton--Watson trees (which have some nice independence properties). Then we recover results for the first family of trees using singularity analysis for generating functions.
	\item For the proof of Theorem \ref{thm_2} we use a probabilistic approach building on local limits for Galton--Watson trees pointed at a uniform vertex. The bijection between trees and 321-avoiding permutations used strongly depends on the position of the leaves. We therefore study the contour functions of some specific Galton--Watson trees in order to extract information about the positions of the leaves in the neighborhood of a uniform vertex. 
\end{itemize}

\subsection{A concentration phenomenon for substitution-closed classes}
Our decorated tree approach introduced in \cref{sect:sub_close} allows us to obtain local limit results for uniform random permutations in substitution-closed classes. Recall the notation from \cref{prop: offspring_distr_charact} page \pageref{prop: offspring_distr_charact}.

\begin{thm}
	\label{thm:local_intro}
	Let $\mathcal{C}$ be a proper substitution-closed permutation class and assume that
	\begin{equation}
	\label{eq:S_TypeI}
	\cS'(\rho_\cS) \ge \frac{2}{(1 +\rho_\cS)^2} -1.
	\end{equation}
	For each $n \in \Z_{>0}$, consider a uniform random permutation $\bm{\sigma}_n$ of size $n$
	in $\mathcal{C}$.
	Then, for each pattern $\pi \in \mathcal C$,
	there exists $\gamma_{\pi,\mathcal C}$ in $[0,1]$ such that
	\begin{equation}
	\pcocc(\pi,\bm{\sigma}_n) \stackrel{P}{\longrightarrow}  \gamma_{\pi,\mathcal C}.
	\end{equation}
\end{thm}
Note that the theorem reveals a concentration phenomenon for all substitution-closed class
satisfying hypothesis \eqref{eq:S_TypeI}, highlighting an instance of a universal phenomenon.
The constants $\gamma_{\pi,\mathcal C}$ can be constructed from local limits of 
conditioned Galton--Watson trees around a random leaf, see \cref{sec:local_lim}
and in particular \cref{rk:gammas}.
They depend both on the pattern $\pi$ and on the class $\mathcal C$, but unlike the previous section, it does not seem simple to find a closed formula for $\gamma_{\pi,\mathcal C}$.

\subsection{A non-concentration phenomenon for square permutations}
\label{const_lim_obj}

We now investigate a model of random permutations where there is no concentration for the quenched B--S limit.
In contrast with the previous two sections, here we start directly by introducing the candidate limiting objects for the annealed and quenched B--S convergence of square permutations. Therefore we have to define a random infinite rooted permutation and a random measure on $\Sri$. After stating our main result, we will also give an intuitive explanation on the construction of these limiting objects.

We start by defining the random infinite rooted permutation as a random total order $\bm{\preccurlyeq}_{\infty}$ on $\Z.$ 
We consider the set of integer numbers $\Z,$ and a labeling $\mathcal{L}\in\{+,-\}^{\Z}$ of all integers with "$+$" or "$-$".
We set $\mathcal{L}^+\coloneqq\{x\in\Z:x\text{ has label }"+"\}$ and $\mathcal{L}^-\coloneqq\{x\in\Z:x\text{ has label }"-"\}.$

Then we define four total orders $\preccurlyeq_j^\mathcal{L}$ on $\Z$, for $j\in\{1,2,3,4\}$,  saying that for all $x,y\in\Z,$  
\begin{equation} 
	\begin{cases} 
	x\preccurlyeq_1^{\mathcal{L}} y \quad\text{ if }\quad (x<y \text{ and } x,y\in \mathcal{L}^-) \text{ or } (x<y \text{ and } x,y\in \mathcal{L}^+) \text{ or } (x\in \mathcal{L}^- \text{ and } y\in \mathcal{L}^+), \\
	x\preccurlyeq_2^{\mathcal{L}} y \quad\text{ if }\quad (x>y \text{ and } x,y\in \mathcal{L}^-) \text{ or } (x<y \text{ and } x,y\in \mathcal{L}^+) \text{ or } (x\in \mathcal{L}^- \text{ and } y\in \mathcal{L}^+), \\
	x\preccurlyeq_3^{\mathcal{L}} y \quad\text{ if }\quad (x<y \text{ and } x,y\in \mathcal{L}^-) \text{ or } (x>y \text{ and } x,y\in \mathcal{L}^+) \text{ or } (x\in \mathcal{L}^- \text{ and } y\in \mathcal{L}^+), \\
	x\preccurlyeq_4^{\mathcal{L}} y \quad\text{ if }\quad (x>y \text{ and } x,y\in \mathcal{L}^-) \text{ or } (x>y \text{ and } x,y\in \mathcal{L}^+) \text{ or } (x\in \mathcal{L}^- \text{ and } y\in \mathcal{L}^+).
	\end{cases}
\end{equation}
A more intuitive presentation of the four constructions is given in \cref{constr_order}.

\begin{figure}[htbp]
	\centering
		\includegraphics[scale=0.75]{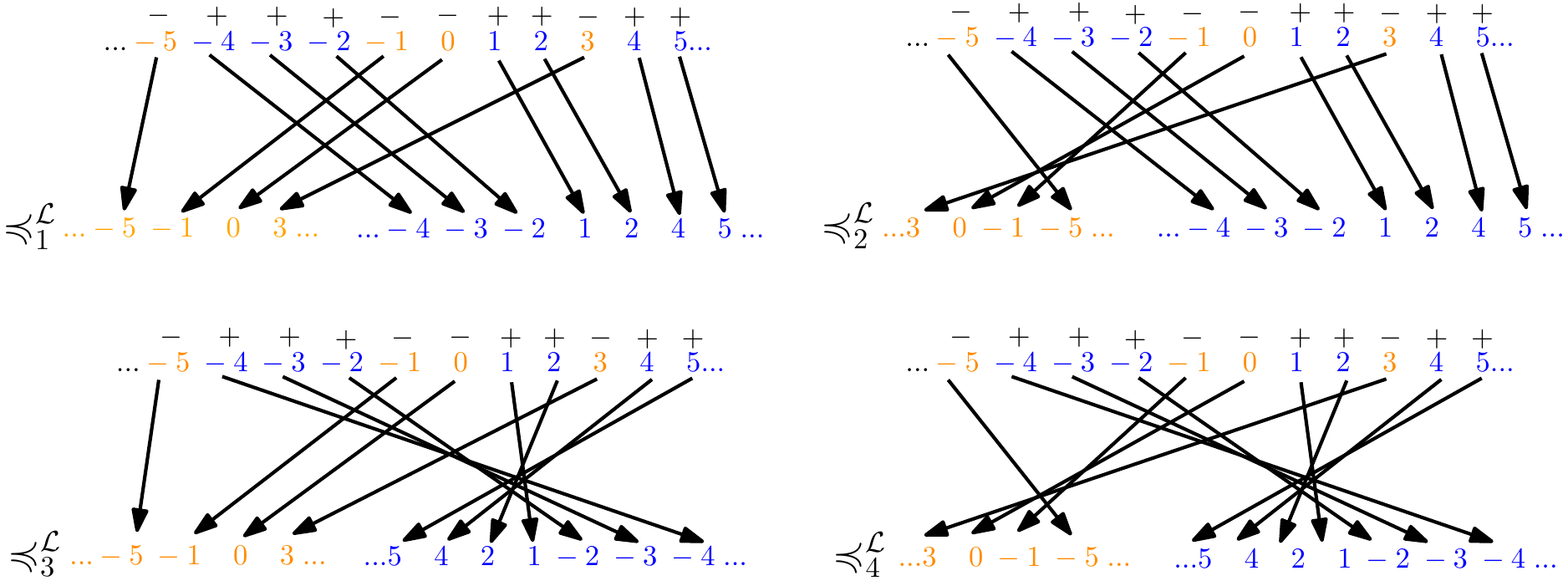}\\
		\caption{An example of the four total orders $(\Z,\preccurlyeq_j^\mathcal{L})$ for $j\in\{1,2,3,4\}$. For each of the four cases, on the top line, we see the standard total order on $\Z$ with the integers labeled by "$-$" signs (painted in orange) and "+" signs (painted in blue). Then, in the bottom line of each of the four cases, we move the "$-$"-labeled numbers at the beginning of the new total order and the  "$+$"-labeled  numbers at the end. Moreover, for $\preccurlyeq_1^{\mathcal{L}} $ we keep the relative order among integers with the same label, for $\preccurlyeq_2^{\mathcal{L}}$  we reverse the order on the "$-$"-labeled numbers, for $\preccurlyeq_3^{\mathcal{L}}$ we reverse the order on the "$+$"-labeled numbers and for $\preccurlyeq_4^{\mathcal{L}}$ we reverse the order on both "$-$"-labeled and "$+$"-labeled numbers. For each case, reading the bottom line gives the total order $\preccurlyeq_j^\mathcal{L}$ on $\Z.$ }\label{constr_order}
\end{figure}
The random total order $\bm{\preccurlyeq}_{\infty}$ on $\Z$ is defined as follows. We choose a Bernoulli labeling $\bm{\mathcal{L}}$ of $\Z,$ namely, 
$\P(x \text{ has label }"+")=1/2=\P(x \text{ has label }"-")$ for all $x\in\Z,$
independently for different values of $x.$
Finally, we set
\begin{equation}
	\label{eq:limiting_ord}
	\bm{\preccurlyeq}_{\infty}\quad\stackrel{d}{=}\quad {\preccurlyeq}_{\bm K}^{\bm{\mathcal{L}}},
\end{equation}
where $\bm K$ is a uniform random variable in $\{1,2,3,4\}$ independent of the random $\{+,-\}$-labeling.

The random total order $(\Z,\bm{\preccurlyeq}_{\infty})$ is the candidate limit for the annealed B--S convergence of square permutations. We now introduce, as a random measure on $\Sri$, the candidate limit for the quenched convergence.
We start by defining the following function from $[0,1]^2$ to $\{1,2,3,4\}$,
\begin{equation}
	\label{eq:def_map_J}
	J(u,v)=\begin{cases}
		1, &\text{if }  u<1/2\text{ and } u\leq v\leq 1-u, \\ 
		2, &\text{if }  v< \min\{u,1-u\},  \\
		3, &\text{if }  v> \max\{u,1-u\},\\
		4, &\text{if }  u\geq1/2\text{ and } 1-u\leq v\leq u.
	\end{cases}
\end{equation}
We consider two independent uniform random variables $\bm U,\bm V$ on the interval $[0,1]$ and we define the random probability measure $\bm{\nu}_{\infty}$ on $\Sri$ as
\begin{equation}
	\label{eq:def_of_quaenched_lim}
	\bm{\nu}_{\infty}=\mathcal{L}aw\big((\Z,\bm{\preccurlyeq}^{\bm{\mathcal{L}}}_{J(\bm U,\bm V) })\big|\bm U\big).
\end{equation}

\begin{thm}[{\cite[Theorem 6.4]{borga2020square}}]
	\label{thm:local_conv}
	Let $\bm{\sigma}_n$ be a uniform random square permutation of size $n$. Then $$\bm{\sigma}_n\stackrel{qBS}{\longrightarrow}\bm{\nu}_{\infty}\quad\text{and}\quad\bm{\sigma}_n\stackrel{aBS}{\longrightarrow}(\Z,\bm{\preccurlyeq}_{\infty}).$$
\end{thm}

We highlight the remarkable fact that $\bm{\nu}_{\infty}$ is a (non-trivial) random measure. Therefore, the class of square permutations does \emph{not} present a concentration phenomena.

We do not include the proof of this theorem in this manuscript but we try to explain the intuition behind the various constructions.
We saw in \cref{guiding light} page \pageref{guiding light} that the shape of a typical square permutation $\sigma$ is close to a rectangle rotated of 45 degrees and with bottom vertex at $(z_0=\sigma^{-1}(1),1)$ (see the two red rectangles in \cref{local_example}).

In order to prove the (quenched and annealed) B--S convergence for a sequence of uniform square permutations $(\bm{\sigma}_n)_n,$ we must understand, for any fixed $h\in\N,$ the behavior of the pattern induced by an $h$-restriction of $\bm{\sigma}_n$ around a uniform index $\bm{i}_n,$ denoted by $r_h(\bm{\sigma}_n,\bm i_n)$. Therefore, we fix an integer $h\in\N.$ The pattern $r_h(\bm{\sigma}_n,\bm i_n)$ can have four "different shapes", according to the relative position of $\bm z_0=\bm \sigma_n^{-1}(1),$ $\bm z_2=\bm \sigma_n^{-1}(n)$ and $\bm i_n.$ In particular (see also Fig.~\ref{local_example}), when $\bm i_n$ is far enough from $\bm z_0$ and $\bm z_2$ (and this will happen with high probability):
\begin{itemize}
	\item if $\bm z_0<\bm i_n<\bm z_2$ then $r_h(\bm{\sigma}_n,\bm i_n)$ is composed by two increasing sequences, one on top of the other;
	\item if $\bm i_n<\min\{\bm z_0,\bm z_2\}$ then $r_h(\bm{\sigma}_n,\bm i_n)$ is composed by two sequences, an increasing one on top of a decreasing one, that is, it has a "$\textless$"-shape;
	\item if $\bm i_n>\max\{\bm z_0,\bm z_2\}$ then $r_h(\bm{\sigma}_n,\bm i_n)$  is composed by two sequences, a decreasing one on top of an increasing one, that is, it has a "$\textgreater$"-shape;
	\item if $\bm z_0<\bm i_n<\bm z_2$ then $r_h(\bm{\sigma}_n,\bm i_n)$ is composed by two decreasing sequences, one on top of the other.
\end{itemize}

\begin{figure}[htbp]
	\centering
		\includegraphics[scale=0.8]{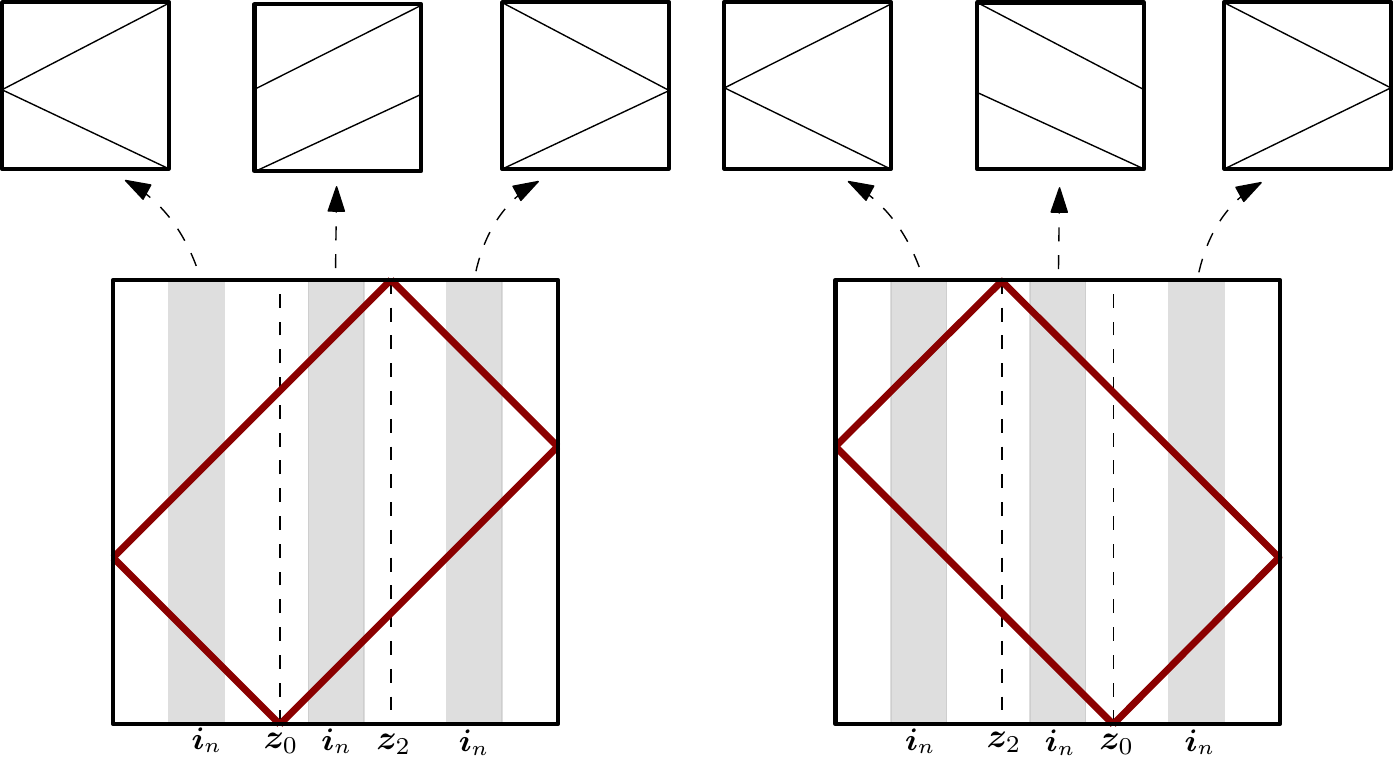}\\
		\caption{In red the approximate shapes of two large square permutations. In the first case $\bm z_0<\bm z_2$ and in the second case $\bm z_2<\bm z_0.$ The top row shows that the possible shapes of a pattern induced by a vertical strip (six of them are highlighted in gray) around an index $\bm i_n$ is determined by the relative position of $\bm z_0,\bm z_2$ and $\bm i_n$.}\label{local_example}
\end{figure}

Note that in the second and third case the fact that the two sequences are "disjoint", that is, that one is above the other, holds for every square permutation. On the other hand, the same result in the first and fourth case is only true with high probability.  

The quenched and annealed limiting objects that we introduced at the beginning of this section are constructed keeping in mind these four possible cases. Specifically, for the quenched limiting object $\bm{\nu}_{\infty}$, the uniform random variables $\bm U$ and $\bm V$ involved in the definition have to be thought of as the limits of the points $\bm z_0$ and  $\bm{i}_n$ respectively (after rescaling by a factor $n$).

Finally, we explain the choice of the Bernoulli labeling $\bm{\mathcal{L}}$ used in the construction of the four random total orders $\preccurlyeq_j^{\bm{\mathcal{L}}}$. It is enough to note that every point of the permutation, contained in a chosen vertical strip around $\bm{i}_n$, is in the top or bottom sequence with equiprobability and independently of the other points (this is a consequence of the construction presented in Section \ref{sect:inverse_projection}).

\subsection{Permutation families encoded by generating trees}\label{sect:gen_tree_CLT}

In the previous sections we have seen several law of large numbers for consecutive patterns in various models of random permutations.
We now present a central limit theorem (CLT) for the proportion of consecutive occurrences of a fixed pattern. Our results applies to uniform permutations encoded by a generating tree with one-dimensional labels. This CLT implies local convergence for such permutations.

\subsubsection{Three assumptions for our central limit theorem}

We first state three assumptions for a family of permutations $\mathcal{C}$ (growing on the right) encoded by a generating tree.

\begin{ass}
	\label{ass1}
	There exists a $\Z$-valued statistics that determines (in the sense of \cref{def:gen_tree_part2}) the succession rule of the generating tree\footnote{These generating trees are usually called \emph{one-label generating trees}, and we used \emph{one-dimensional-labeled generating trees} earlier.}. We further require that the labels appearing in the generating tree are elements of $\mathbb{Z}_{\geq\beta}$ for some $\beta\in\mathbb{Z}$, except for the root, where the label is equal to $\lambda=\beta-1$. 
\end{ass}

\begin{rem}
	The assumption above might seem not completely mathematically 
	meaningful\footnote{For instance, taking any natural example of two-labels generating tree, it is possible to reinterpret it as a one-label generating tree, simply conjugating the $\Z^2$-valued statistics with a bijection between $\Z^2$ and $\Z$.}. Despite this, it is well established in the literature (see for instance \cite{bousquet2003four,MR2376115}) to refer to one-label generating trees, two-labels generating trees, etc.	
	The key point is that we want the statistics describing the generating tree to be \emph{natural} in a sense made precise in \cref{ass2}.
\end{rem}

In order to state the second assumption, we need to introduce some notation.
Let $(\alpha_y)_{y\in\Z_{\leq 1}}$ be a probability distribution on $\Z_{\leq 1}$ and $(c_y)_{y\in\Z_{\leq 1}}$ be a sequence of non-negative integers such that $c_y\geq 1$ if and only if $\alpha_y>0$. Denote by ${\bm{Y}}^{\bm c}$ the corresponding colored $\mathbb{Z}_{\leq 1}$-valued random variable having value-distribution equal to $(\alpha_y)_{y\in\Z_{\leq 1}}$ (i.e.\ $\P({\bm{Y}}=y)=\alpha_y$ for all $y\in\Z_{\leq 1}$) and color-distribution defined as follows: conditioning on $\{{\bm{Y}}=y\}$ then $\bm{c}$ is a uniform random element of $[c_y]$.

Consider now a sequence $({\bm{Y}}^{\bm c_j}_j)_{j\geq 1}$ of i.i.d.\ copies of ${\bm{Y}}^{\bm c}$.
We define the corresponding random colored walk as $({\bm{X}}_i^{\bm c_{i-1}})_{i\geq 1}$, where
\begin{equation}\label{eq:definition_walk}
	{\bm{X}}_i\coloneqq (\beta-1)+\sum_{j=1}^{i-1}{\bm{Y}}_j,\quad\text{for all}\quad i\in\Z_{\geq 1}.
\end{equation}
Note that for all $i\in \Z_{\geq 1}$, ${\bm{X}}_i$ has color $\bm c_{i-1}$. In order to simplify the notation we will often simply write $({\bm{X}}_i^{\overleftarrow{\bm c}})_{i\geq 1}$ for the walk $({\bm{X}}_i^{\bm c_{i-1}})_{i\geq 1}$. 

We denote by $K^c$ the set of all finite sequences of \emph{colored} labels starting at $\lambda$ and consistent with the colored succession rule in \cref{eq:colored_succ_rule} page \pageref{eq:colored_succ_rule}. Moreover, $K^c_n$ is the set of sequences in $K^c$ of length $n.$ 
Finally, recall the bijection $\GG:K^c\to\mathcal{C}$, introduce before \cref{prop:bij_gen_tree} page \pageref{prop:bij_gen_tree}.

\begin{ass}
	\label{ass2}
	There exists a centered probability distribution $(\alpha_y)_{y\in\Z_{\leq 1}}$ on $\Z_{\leq 1}$ with finite variance and a sequence of non-negative integers $(c_y)_{y\in\Z_{\leq 1}}$ such that the corresponding one-dimensional random colored walk $({\bm{X}}^{\overleftarrow{\bm c}}_i)_{i\geq 1}$ (defined above) satisfies
	\begin{equation}\label{eq:rewriting_chain}
		\GG^{-1}(\bm\sigma_n)\stackrel{d}{=}\left({\bm{X}}_{i}^{\overleftarrow{\bm c}}\middle|({\bm{X}}_{j})_{j\in[2,n]}\geq \beta,{\bm{X}}_{n+1}=\beta\right)_{i\in[n]},
	\end{equation}
	where $\bm\sigma_n$ is a uniform permutation in $\mathcal C$ of size $n$. 
\end{ass}

\begin{rem}
	Note that, as a consequence of \cref{ass2}, we have that the children labels satisfies $\CL(k)\subseteq(-\infty,k+1]$, for all $k\in\mathcal{L}$. This is not as restrictive as it might seem. Indeed, many $\mathbb{Z}$-valued generating trees are constructed using the statistics that tracks the number of active sites and, in many cases, adding a new element to a permutation increases the value of this statistics at most by one.
\end{rem}	

\begin{rem}
	The assumption that the distribution $(\alpha_y)_{y\in\Z_{\leq 1}}$ on $\Z_{\leq 1}$ has finite variance is also not too restrictive. Indeed, to the best of our knowledge, we are not aware of a class of permutations coded by a one-dimensional walk with
	infinite variance.
\end{rem}

\begin{rem}\label{rem:ass_is_true}
	A strategy to check \cref{ass2} in many cases, determining a possible distribution $(\alpha_y)_{y\in\Z_{\leq 1}}$ and a sequence of non-negative integers $(c_y)_{y\in\Z_{\leq 1}}$, is furnished in \cite[Section 1.6.2]{borga2020asymptotic}.
\end{rem}

For the third assumption we need to introduce the following definitions.

\begin{defn}
	\label{defn:jumpslabel}
	Given a sequence $(k_i^{c_i})_{i\in[n]}\in K^c_n$ of (colored) labels in a generating tree, the corresponding sequence of \emph{colored jumps} is the sequence $(y_i^{c_{i+1}})_{i\in[n-1]}$ where $y_i=k_{i+1}-k_i$. Note that the label $y_i$ inherits the same color as the label $k_{i+1}$ (if it is colored).
\end{defn}

\begin{defn}\label{defn:rebfvriwbfowenf}
	A sequence of $h$ colored jumps $(y^{c_1}_1,\dots,y^{c_h}_h)$ is \emph{consistent} if it is equal to a factor of a sequence of jumps corresponding to an element of $K^c$.
\end{defn}

We denote by $\mathcal{Y}^c_h$ the set of consistent sequences of $h$ colored jumps. 

\begin{ass} \label{ass3} 
	For all $h \geq 1,$ there exists a function $\Pat:\mathcal{Y}^c_h \to\mathcal{S}_h$ and a constant $c(h )\geq 0$ such that if a sequence of labels $(k_i^{c_i})_{i\in[n]}\in K^c_n$ satisfies for some $m\in[n+1-h ]$ the condition 
	$$k_i>c(h ),\quad\text{for all}\quad i\in[m,m+h -1],$$ 
	then  $$\pat_{[m,m+h -1]}\big(\GG((k_i^{c_i})_{i\in[n]})\big)=\Pat\big((y_i^{c_{i+1}})_{i\in[m-1,m+h -2]}\big),$$
	where $(y_i^{c_{i+1}})_{i\in[n-1]}$ is the sequence of (colored) jumps associated with $(k_i^{c_{i}})_{i\in[n]}\in K^c_n$.
\end{ass}
In words, the consecutive pattern induced by the interval $[m,m+h -1]$ in the permutation $\GG((k_i^{c_i})_{i\in[n]})$ is uniquely determined through the function $\Pat$ by the factor of colored jumps $(y_i^{c_{i+1}})_{i\in[m-1,m+h -2]}$ as long as the labels are large enough (see for instance \cite[Example 3.8]{borga2020asymptotic} to understand why this condition is necessary).

\subsubsection{The statement of the CLT}

We can now state our main result.

\begin{thm}
	\label{thm:main_thm_CLT}
	Let $\mathcal{C}$ be a family of permutations encoded by a generating tree that satisfies Assumptions \ref{ass1}, \ref{ass2} and \ref{ass3}. Let $({\bm{Y}}^{\bm c_i}_i)_{i\geq1}$ be a sequence of i.i.d.\ random variables distributed as ${\bm{Y}}^{\bm c}.$ Let $\bm{\sigma}_n$ be a uniform random permutation in $\mathcal{C}_n,$ for all $n\in\Z_{>0}.$
	Then, for all $\pi\in\mathcal{S},$ we have the following central limit theorem,
	\begin{equation}
		\label{wfowrbfv3}
		\frac{\cocc(\pi,\bm{\sigma}_n)-n\mu_{\pi}}{\sqrt n}\stackrel{d}{\longrightarrow}\bm{\mathcal N}(0,\gamma_{\pi}^2),
	\end{equation}
	where $\mu_{\pi}=\P\left(\Pat\left({\bm{Y}}_1^{\bm c},\dots,{\bm{Y}}^{\bm c}_{|\pi|}\right)=\pi\right)$ and $\gamma_{\pi}^2=\kappa^2-\frac{\rho^2}{\sigma^2}$ with
	\begin{align}
		&\sigma^2=\Var\left({\bm{Y}}_1\right),\\
		&\rho=\E\left[\mathds{1}_{\left\{\Pat\left({\bm{Y}}_1^{\bm c},\dots,{\bm{Y}}^{\bm c}_{|\pi|}\right)=\pi\right\}}\cdot \sum_{j=1}^{|\pi|}\bm{Y}_j\right],\\
		&\kappa^2=2\nu+\mu_\pi-(2|\pi|-1)\cdot\mu_\pi^2,\quad\text{for}\\
		&\nu=\sum_{s=2}^{|\pi|}\sum_{\substack{
				(y^{c}_i)_{i\in[|\pi|]},(\ell^{d}_i)_{i\in[|\pi|]}\in\Pat^{-1}(\pi)\\
				\text{s.t. } ({y}^{c}_s,\dots,{y}^{c}_{|\pi|})=({\ell}^{c}_1,\dots,{\ell}^{d}_{|\pi|-s+1})}}
		\P\left(
		({\bm{Y}}^{\bm c}_{i})_{i\in[|\pi|+s-1]}=({y}^c_1,\dots,{y}^c_{|\pi|},{\ell}^d_{|\pi|-s+2},\dots,{\ell}^d_{|\pi|})\right).
	\end{align}
\end{thm}

The proof of this theorem can be found in Section \ref{sect:main_thm_proof}. A consequence of the above result is the following local limit result. 
\begin{cor}\label{corl:main_thm}
	Let $\mathcal{C}$ be a family of permutations encoded by a generating tree that satisfies Assumptions \ref{ass1}, \ref{ass2} and \ref{ass3}. Let $\bm{\sigma}_n$ be a uniform random permutation in $\mathcal{C}_n,$ for all $n\in\Z_{>0}.$ There exists a random infinite rooted permutation $\bm{\sigma}^\infty_{\mathcal{C}}$ such that  
	$$\bm{\sigma}_n\xrightarrow{qBS}\mathcal{L}aw(\bm{\sigma}^\infty_{\mathcal{C}})
	\quad\text{and}\quad
	(\bm{\sigma}_n,\bm{i}_n)\xrightarrow{aBS}\bm{\sigma}^\infty_{\mathcal{C}}.$$	
\end{cor}

In \cite[Section 3]{borga2020asymptotic} we collected several families of permutations that satisfy Assumptions \ref{ass1}, \ref{ass2} and \ref{ass3}, proving \cref{thm:examples_ok} below.

\begin{thm}\label{thm:examples_ok}
	Let $\mathcal C$ be one of the following families of permutations\footnote{We refer the reader to \cite[Section 3.4]{borga2020asymptotic} for the definition of the last two families of permutations in the list, which avoid generalized patterns and were introduced in \cite{MR2376115}.}:
	\begin{align}
		\Av(123),\quad
		\Av(132),\quad
		\Av(1423,&4123),\quad
		\Av(1234,2134),\quad
		\Av(1324,3124),\\
		\Av(2314,3214),\quad
		\Av(2413,4213),\quad
		\Av&(3412,4312),\quad
		\Av(213,\bar{2}\underbracket[.5pt][1pt]{31}),\quad
		\Av(213,\bar{2}^{o}\underbracket[.5pt][1pt]{31}).
	\end{align}
	For all $n\in\Z_{>0},$ let $\bm{\sigma}_n$ be a uniform random permutation in $\mathcal{C}$ of size $n$. 
	Then, for all patterns $\pi\in\mathcal S$, we have the following central limit theorem,
	\begin{equation}
		\frac{\cocc(\pi,\bm{\sigma}_n)-n\mu_{\pi}}{\sqrt n}\stackrel{d}{\longrightarrow}\bm{\mathcal N}(0,\gamma_{\pi}^2),
	\end{equation}
	where $\mu_{\pi}$ and $\gamma_{\pi}^2$ are described in \cref{thm:main_thm_CLT}.
\end{thm}

\begin{rem}
	Recall that a law of large numbers for the families $\Av(123)$ and $\Av(132)$ has been established in \cref{sect:loc_three_av}. Nevertheless the central limit theorem in \cref{thm:examples_ok} is a new result also for these two families. Recall that Theorems \ref{thm_1} and \ref{thm_2} furnish explicit expressions for $\mu_{\pi}$ in the specific case of these two families of permutations.
\end{rem}

\subsection{Baxter permutations and related objects}

So far we mainly focused on B--S limits of random permutations. Here, we also look at some different discrete objects: we consider B--S limits of the four families in the commutative diagram in \cref{eq:comm_diagram} page \pageref{eq:comm_diagram}, that is, Baxter permutations, bipolar orientations, tandem walks and coalescent-walk processes, whose sets are denoted $\mathcal P,\mathcal O, \mathcal W, \mathscr{C}$, respectively. 

B--S convergence for permutations has been extensively presented in this manuscript. In order to properly define the B--S convergence for the other three families, we need to present the spaces of infinite objects and the respective local topologies. This is done in \cref{sect:inf_loc_obj,sect:local topologies}, but we give a quick summary here.
\begin{itemize}
	\item $\widetilde \Walks_\bullet$ is the space of two-dimensional walks indexed by a finite or infinite interval of $\Z$ containing zero, with value $(0,0)$ at time $0$, local convergence being finite-dimensional convergence.
	\item $\widetilde \Coals_\bullet$ is the space of coalescent-walk processes indexed by a finite or infinite interval of $\Z$ containing zero, local convergence being finite-dimensional convergence.
	\item The space $\widetilde \Maps_\bullet$ of infinite rooted maps is equipped with the local topology derived from the local convergence for graphs of Benjamini and Schramm. See for instance \cite{curienrandom} for an introduction.
\end{itemize}
In the first two items, the index $0$ has to be understood as the root of the infinite object, and comparison between a rooted finite object and an infinite one is done after applying the appropriate shift.

For every $n\in \Z_{>0}$, let $\bm W_n$, $\bm Z_n$, $\bm \sigma_n$, and $\bm m_n$ denote uniform objects of size $n$ in $\mathcal W_n$, $\mathscr{C}_n$, $\mathcal P_n$, and $\mathcal O_n$, respectively, related by the four bijections in the commutative diagram in \cref{eq:comm_diagram} page \pageref{eq:comm_diagram}.
We define the candidate local random limiting objects. Let $\nu$ denote the probability distribution on $\Z^2$ given by:
\begin{equation}\label{eq:step_distribution_walk}
\nu = \frac 12 \delta_{(+1,-1)} + \sum_{i,j\geq 0} 2^{-i-j-3}\delta_{(-i,j)},\quad \text{where $\delta$ denotes the Dirac measure},
\end{equation}
and let $\overline{\bm W} = (\overline{\bm X},\overline{\bm Y}) = (\overline{\bm W}_t)_{t\in \Z}$ be a two-sided random two-dimensional walk with step distribution $\nu$, having value $(0,0)$ at time 0. Note that $\overline {\bm W}$ is not confined to the non-negative quadrant.

A formal definition of the other limiting objects requires an extension of the mappings in \cref{eq:comm_diagram} to infinite-volume objects\footnote{The terminology \emph{finite/infinite-volume} refers to the fact that the objects are defined in a compact/non-compact set.
	For instance a Brownian motion with time space $\mathbb R$ is a infinite-volume object and a Brownian excursion with time space $[0,1]$ is a finite-volume object.}
which was done in \cref{sec:discrete_coal} for the mapping $\wcp$, and will be done in Sections \ref{sect:infinite_bij} and \ref{sect:inf_loc_obj} for the other three mappings.
Admitting we know such extensions, let $\overline{\bm Z} = \wcp(\overline{\bm W})$ be the corresponding infinite coalescent-walk process, $\overline{\bm \sigma} = \cpbp(\overline{\bm Z})$ the corresponding infinite permutation on $\Z$, and $\overline{\bm m} = \bow^{-1}(\overline{\bm W})$ the corresponding infinite map.

\begin{thm}
	[Quenched Benjamini--Schramm convergence]
	\label{thm:local}
	Consider the sigma-algebra 
	$\mathfrak{B}_n$ generated by $\bm W_n$ (or equivalently by $\bm Z_n$, $\bm \sigma_n$, or $\bm m_n$).
	Let $\bm i_n$ be an independently chosen uniform index of $[n]$.
	We have the following convergence in probability in the space of probability measures on $\widetilde \Walks_\bullet\times\widetilde \Coals_\bullet\times\widetilde \Perms_\bullet\times\widetilde \Maps_\bullet$,
	\begin{equation}
		\mathcal{L}aw\Big(\big((\bm W_n, \bm i_n), (\bm Z_n, \bm i_n), (\bm \sigma_n, \bm i_n), (\bm m_n, \bm i_n)\big)\Big| \mathfrak{B}_n\Big)\xrightarrow[n\to\infty]{P}\mathcal{L}aw\left(\overline{\bm W},\overline{\bm Z},\overline{\bm \sigma},\overline{\bm m}\right),\label{eq:quenched_conv_walks}
	\end{equation}
	where we recall that $\mathcal{L}aw(\cdot)$ denotes the law of a random variable.
\end{thm}
We recall that the mapping $\bow^{-1}$ naturally endows the map $\bm m_n$ with an edge labeling and the root $\bm i_n$ of $\bm m_n$ is chosen according to this labeling. An immediate corollary, which follows by averaging, is the simpler \textit{annealed} statement.
\begin{cor}[Annealed Benjamini--Schramm convergence]
	We have the following convergence in distribution in the space $\widetilde \Walks_\bullet\times\widetilde \Coals_\bullet\times\widetilde \Perms_\bullet\times\widetilde \Maps_\bullet$,	\begin{equation}\label{eq:annealed_conv_walks}
	((\bm W_n, \bm i_n), (\bm Z_n, \bm i_n), (\bm \sigma_n, \bm i_n), (\bm m_n, \bm i_n)) \xrightarrow[n\to\infty]{d} (\overline{\bm W},\overline{\bm Z},\overline{\bm \sigma},\overline{\bm m}).
	\end{equation}
\end{cor}

\begin{cor}
	Let $\bm{\sigma}_n$ be a uniform Baxter permutation of size $n$.
	We have the following convergence in probability w.r.t.\ the product topology on $[0,1]^\Perms$
	\begin{equation}\label{eq:cocc_conv}
	\left(\widetilde{\cocc}(\pi,\bm{\sigma}_n)\right)_{\pi\in\Perms}\stackrel{P}{\to}\left(\widetilde{\cocc}(\pi,\overline{\bm \sigma})\right)_{\pi\in\Perms}.
	\end{equation}
\end{cor}

We collect a few comments on these results.
\begin{itemize}
	\item \cref{eq:quenched_conv_walks,eq:cocc_conv} witness a concentration phenomenon.
	\item The fact that the four convergences are joint follows from the fact that the extensions of the mappings in \cref{eq:comm_diagram} to infinite-volume objects are a.s.\ continuous.
	\item The annealed Benjamini-Schramm convergence for  bipolar orientations to the so-called \textit{Uniform Infinite Bipolar Map} $\overline{\bm m}$ was already proven in \cite[Prop. 3.10]{gwynne2017mating}; the quenched version is however a new result. 
\end{itemize}

\section{231-avoiding permutations}\label{sect:231proofs}

The goal of this section is to prove a law of large numbers for consecutive patterns of uniform random $231$-avoiding permutations, i.e.\ \cref{thm_1}.
Using \cref{detstrongbsconditions}, a consequence of \cref{thm_1} is the following result.
\begin{cor}
	\label{231corol}
	Let $\bm{\sigma}_n$ be a uniform random permutation in $\Av_n(231)$ for all $n\in\Z_{>0}.$ There exists a random infinite rooted permutation $\bm{\sigma}^\infty_{231}$ such that   $\bm{\sigma}_n\stackrel{qBS}{\longrightarrow}\mathcal{L}aw(\bm{\sigma}^\infty_{231}).$
\end{cor}	

The rest of this section is structured as follows:
\begin{itemize}
	\item In \cref{behavtree} we reduce the study of $\E\big[\widetilde{\cocc}(\pi,\bm{\sigma}_n)\big]$ to the study of a similar expectation defined in terms of a specific binary Galton--Watson tree. This builds on the bijection between binary trees and 231-avoiding permutations introduced in \cref{231bij} and a technique due to Janson \cite{janson2017patterns}.
	\item In Section \ref{weakres} we study the behavior of $\E\big[\widetilde{\cocc}(\pi,\bm{\sigma}_n)\big]$ (see in particular Proposition \ref{weakprop}) using the results from the previous section. Then we prove Theorem \ref{thm_1}: with similar techniques we study the second moment $\E\big[\widetilde{\cocc}(\pi,\bm{\sigma}_n)^2\big]$ and we conclude using the Second moment method.
	\item In Section \ref{explcon} we give an explicit construction of the limiting object $\bm{\sigma}^\infty_{231}$.
\end{itemize}

\subsection*{Notation}
We introduce some more notation used in this section and \cref{321}. When convenient, we extend in the trivial way the various notation introduced for permutations to arbitrary sequences of distinct numbers. For example, we extend the notion of $\cocc(\pi,\sigma)$ for two arbitrary sequences of distinct numbers $x_1\dots x_n$ and $y_1\dots y_k$  as
\begin{equation}
\cocc(y_1\dots y_k,x_1\dots x_n)\coloneqq\cocc\big(\text{std}(y_1\dots y_k),\text{std}(x_1\dots x_n)\big). 
\end{equation}  
We will write $\text{pat}_{b(k)}(\sigma)$ for $\text{pat}_{[1,k]}(\sigma)$ namely, for the pattern occurring in the first $k$ positions of $\sigma.$ Similarly we will write $\text{pat}_{e(k)}(\sigma)$ for $\text{pat}_{[|\sigma|-k+1,|\sigma|]}(\sigma)$ namely, for the pattern occurring in the last $k$ positions of $\sigma.$ Note that $b(k)$ and $e(k)$ stand for \emph{beginning} and \emph{end}. Moreover, if either $k=0$ or $k>|\sigma|$ we set $\text{pat}_{b(k)}(\sigma)\coloneqq\emptyset$ and $\text{pat}_{e(k)}(\sigma)\coloneqq\emptyset$.
To simplify notation, sometimes, we will simply write $\text{pat}_{b}(\sigma)=\pi$ or $\text{pat}_{e}(\sigma)=\pi,$ instead of $\text{pat}_{b(|\pi|)}(\sigma)=\pi$ or $\text{pat}_{e(|\pi|)}(\sigma)=\pi.$

\medskip

We recall that a Laurent polynomial over $\mathbb{R}$ is a linear combination of positive and negative powers of the variable with coefficients in $\mathbb{R},$ \emph{i.e.}, of the form
\begin{equation}
P(x)=a_{-m}x^{-m}+\dots+ a_{-1}x^{-1}+a_0+a_1x+\dots+a_n x^n,
\end{equation}
where $m,n\in\Z_{\geq0}$ and $a_i\in\mathbb{R}$ for all $-m\leq i\leq n.$
We  denote by $O_{LP}(x^{\alpha}),$ $\alpha\in\Z,$ an arbitrary Laurent polynomial in $x$ of valuation at least $\alpha,$ i.e.\ of the form $a_{\alpha}x^{\alpha}+\dots+a_n x^n.$ Note that one can think at these symbols as $O(\cdot)$ symbols around $x=0$. However, we remark that this is \emph{not} the classical definition of $O(\cdot)$ since we are adding the additional hypothesis that the elements of $O_{LP}(\cdot)$ are (Laurent) polynomials and not general functions.

\medskip

We need some further notation for trees. Given  a tree $T\in\mathbb{T},$ for each vertex $v\in T,$ the height of $v$ is denoted by $h(v).$ We denote by $d$ the classical distance on trees, i.e.\ for every $u,v\in T,$ $d(u,v)$ is equal to the number of edges in the unique path between $u$ and $v.$ We also denote by $|T|$ the number of vertices of $T$ and by $\mathcal{L}(T)$ the set of leaves of $T$. Additionally, $a(v)$ denotes the parent of $v$ in $T$ and we write $a^i(v),$ $i\geq 0,$ for the $i$-th ancestor of $v$ in $T,$ i.e.\ $a^i(v)=\underbrace{a\circ a\circ\dots\circ a}_{i\text{-times}}(v)$ (with the convention that if the root $r=a^k(v)$ then $a^i(v)=r$ for all $i\geq k$).
Finally, we denote by $c(u,v)$ the youngest common ancestor of $u$ and $v$ and   by $f(T,v)$ the fringe subtree of $T$ rooted at $v$.

\subsection{From uniform 231-avoiding permutations to binary Galton--Watson trees}
\label{behavtree}

Recalling the bijection $T\mapsto\sigma_T$ between binary trees and 231-avoiding permutations introduced in \cref{231bij} page \pageref{231bij} we define 
\begin{equation}
	\cocc(\pi,T)\coloneqq\cocc(\pi,\sigma_T),\quad\text{for all}\quad\pi\in\mathcal{S},
\end{equation}
and similarly
$$\text{pat}_{I}(T)\coloneqq\text{pat}_I(\sigma_T),\quad\text{ for all finite intervals}\quad I\subset\Z_{>0}.$$

Thanks to \cref{prop:unif_231_as_trees}, instead of studying $\E\big[\widetilde{\cocc}(\pi,\bm{\sigma}_n)\big]$ we can equivalently study the expectation $\E\big[\widetilde{\cocc}(\pi,\bm{T}_n)\big]$, where $\bm{T}_n$ denotes a uniform random binary tree with $n$ vertices. A  method for studying the behavior of the latter expectation is to analyze the expectation $\E\big[\widetilde{\cocc}(\pi,\bm{T}_{\delta})\big]$ for a binary Galton--Watson tree $\bm{T}_\delta$ defined as follows.

For $\delta\in(0,1),$ we set $p=\frac{1-\delta}{2},$ and we consider a binary Galton--Watson tree $\bm{T}_\delta$ with offspring distribution $\eta$ defined by:
\begin{equation}
\label{bintree}
\begin{split}
&\eta_0=\P(\text{The root has 0 child})=(1-p)^2,\\
&\eta_L=\P(\text{The root has only one left child})=p(1-p),\\
&\eta_R=\P(\text{The root has only one right child})=p(1-p),\\
&\eta_2=\P(\text{The root has 2 children})=p^2.
\end{split}
\end{equation}
\begin{rem}
	Note that with this particular offspring distribution, having a left child is independent from having a right child.
\end{rem}
We emphasize that $\bm{T}_\delta$ is an \emph{unconditioned} binary Galton--Watson tree.

\medskip

We now state two lemmas due to Janson. The first one is a classical result that rewrites the expectation of a function of $\bm{T}_\delta$ in terms of uniform random binary trees with $n$ vertices $\bm{T}_n,$ using the known fact that $\mathcal{L}aw(\bm{T}_n)=\mathcal{L}aw\big(\bm{T}_{\delta}\big||\bm{T}_{\delta}|=n\big).$ 
\begin{lem}[{\cite[Lemmas 5.2-5.4]{janson2017patterns}}]
	Let $F$ be a functional from the set of binary trees $\mathbb{T}^b$ to $\mathbb{R}$ such that $|F(T)|\leq C\cdot|T|^c$ for some constants $C$ and $c$ (this guarantees that all expectations and sums below converge). Then
	$$\E\big[F(\bm{T}_\delta)\big]=\frac{1+\delta}{1-\delta}\sum_{n=1}^{+\infty}\E\big[F(\bm{T}_n)\big]\cdot C_n\cdot\Big(\frac{1-\delta^2}{4}\Big)^{n},$$
	where $C_n=\frac{1}{n+1}\binom{2n}{n}$ is the $n$-th Catalan number.
\end{lem}

Applying singularity analysis (see \cite[Theorem VI.3]{flajolet2009analytic}) one obtains:

\begin{lem}
	\label{svantelemma}
	If $\E\big[F(\bm{T}_\delta)\big]=a\cdot\delta^{-m}+O_{LP}(\delta^{-(m-1)}),$ for $\delta\to 0,$ where $m\geq 1$ and $a\neq 0,$ then 
	$$\E\big[F(\bm{T}_n)\big]\sim\frac{\Gamma(1/2)}{\Gamma(m/2)}\cdot a\cdot  n^{\frac{m+1}{2}},\quad\text{as}\quad n\to\infty.$$
\end{lem} 

\begin{rem}
	We recall that since by definition the elements of $O_{LP}(\delta^{-(m+1)})$ are Laurent polynomials (and not general functions) then they are $\Delta$–analytic in $z=\frac{1-\delta^2}{4}$(for a precise definition see \cite[Definition VI.1]{flajolet2009analytic}) and so the standard hypothesis to apply singularity analysis is fulfilled. 
\end{rem}

Lemma \ref{svantelemma} shows us why it is enough to study the behavior of $\E[\widetilde{\cocc}(\pi,\bm{T}_\delta)]$ in order to derive information about the behavior of $\E[\widetilde{\cocc}(\pi,\bm{T}_n)].$ This is the goal of the next section.

\subsection{Annealed and quenched Benjamini--Schramm convergence}
\label{weakres}

For all $k>0,$ we define the following probability distribution on $\Av_k(231),$ 
$$P_{231}(\pi)\coloneqq\frac{2^{\#\text{LRmax}(\pi)+\#\text{RLmax}(\pi)}}{2^{2|\pi|}},\quad\text{for all}\quad \pi\in\Av_k(231).$$

\begin{rem}
	The fact that the equation above defines a probability distribution on $\Av_k(231)$ is a consequence of \cref{weakprop}, noting that in \cref{eq:fvhuweyhiufvwe}, summing over all $\pi\in\Av_k(231)$ the left-hand side, we obtain $\frac{n-k+1}{n}$ that tends to 1. 
	This fact can be also proved by elementary manipulations of the corresponding generating function, as shown in \cite[Appendix A]{borga2020localperm}.
\end{rem}

We are going to first prove the following weaker version of Theorem \ref{thm_1}.	
\begin{prop}
	\label{weakprop}
	Let $\bm{\sigma}_n$ be a uniform random permutation in $\Av_n(231)$ for all $n\in\Z_{>0}.$ It holds that
	\begin{equation}\label{eq:fvhuweyhiufvwe}
		\E\big[\widetilde{\cocc}(\pi,\bm{\sigma}_n)\big]\to P_{231}(\pi),\quad\text{for all}\quad\pi\in\emph{Av}(231).
	\end{equation}
\end{prop}

Observation \ref{permfact} page \pageref{permfact} gives a recursive formula for $\cocc(\pi,\sigma)$ that can be translated in term of trees. Recall that $\indmax(\sigma)\coloneqq\sigma^{-1}(n)$ for all $\sigma\in\mathcal{S}_n$. 
\begin{lem}[{\cite[Lemma 4.16]{borga2020localperm}}]
	\label{ricors}
	Let $\pi\in\Av_k(231)$ with $k\geq 1$ and  set $m=\indmax(\pi)$. Then, for every binary tree $T,$ denoting $\ell=\indmax(\sigma_{T}),$
	\begin{equation}
		\label{recccoctrees}
		\cocc(\pi,T)=\cocc(\pi,T_L)+\cocc(\pi,T_R)+\mathds{1}_{\big\{\emph{pat}_{[\ell-m+1,\ell+k-m]}(T)=\pi\big\}}.
	\end{equation}  
\end{lem}
We now focus on the behavior of $\E\big[\widetilde{\cocc}(\pi,\bm{T}_\delta)\big]$ for $\pi\in\Av(231)$. 
In order to simplify notation we set $\bm{T}\coloneqq\bm{T}_\delta.$ Thanks to Lemma \ref{ricors}, we know that, for all $\pi\in\Av(231),$
\begin{equation}
\label{twostar}
\cocc(\pi,\bm{T})=\cocc(\pi,\bm{T}_L)+\cocc(\pi,\bm{T}_R)+\mathds{1}_{\big\{\text{pat}_{\bm{J}}(\bm{T})=\pi\big\}},
\end{equation}
where $\bm{J}=[\bm{\ell}-m+1,\bm{\ell}+k-m]$, $\bm{\ell}=\text{indmax}(\bm{\sigma}_{\bm{T}})$ and $m=\text{indmax}(\pi).$ 
Taking the expectation in \cref{twostar} we obtain,
$$\E\big[\cocc(\pi,\bm{T})\big]=\E\big[\cocc(\pi,\bm{T}_L)\big]+\E\big[\cocc(\pi,\bm{T}_R)\big]+\P\big(\text{pat}_{\bm{J}}(\bm{T})=\pi\big).$$
Since $\bm{T}_L$ is an independent copy of $\bm{T}$ with probability $p$ and empty with probability $1-p,$ and the same holds for $\bm{T}_R,$ we have,
\begin{equation}
\label{step1}
\E\big[\cocc(\pi,\bm{T})\big]=\frac{\P\big(\text{pat}_{\bm{J}}(\bm{T})=\pi\big)}{1-2p}=\delta^{-1}\cdot\P\big(\text{pat}_{\bm{J}}(\bm{T})=\pi\big),
\end{equation}
where in the last equality we used that $p=\frac{1-\delta}{2}.$
We now focus on the term $\P\big(\text{pat}_{\bm{J}}(\bm{T})=\pi\big).$ With some simple computations it is possible to prove the following formula.

\begin{lem}[{\cite[Lemma 4.17]{borga2020localperm}}]
	\label{onecross}
	Let $\pi\in\emph{Av}(231).$ Using notation as above and decomposing $\pi$ in $\pi=\pi_L\pi(m)\pi_R,$
	\begin{equation}
	\label{oneredcross}
	\P\big(\emph{pat}_{\bm{J}}(\bm{T})=\pi\big)=
	\begin{cases}
	p^2\cdot\P\big(\emph{pat}_e(\bm{T})=\pi_L\big)\cdot\P\big(\emph{pat}_b(\bm{T})=\pi_R\big), &\quad\text{if }\pi_L\neq\emptyset,\;\pi_R\neq\emptyset,\\
	p\cdot\P\big(\emph{pat}_e(\bm{T})=\pi_L\big), &\quad\text{if }\pi_L\neq\emptyset,\;\pi_R=\emptyset,\\
	p\cdot\P\big(\emph{pat}_b(\bm{T})=\pi_R\big), &\quad\text{if }\pi_L=\emptyset,\;\pi_R\neq\emptyset,\\
	1, &\quad\text{if }\pi=1.
	\end{cases}
	\end{equation}
\end{lem}

In view of Lemma \ref{onecross}, we now focus on $\P\big(\text{pat}_e(\bm{T})=\pi)$ -- the analysis for $\P\big(\text{pat}_b(\bm{T})=\pi)$ following by symmetry. We want to rewrite the event $\big\{\text{pat}_e(\bm{T})=\pi\big\}$ conditioning on the position of the maximum among the last $|\pi|$ values of $\sigma_{\bm{T}}.$ Using Observation \ref{maxnode}, we know that this maximum is reached at an element of $\sigma_{\bm{T}}$ corresponding to a vertex of $\bm{T}$ of the form
$$v=\underbrace{2\dots2}_{n\text{-times}}\eqqcolon v_{2^n},\quad\text{for some}\quad n\in\Z_{\geq 0},$$
with the convention $v_{2^0}\coloneqq\emptyset.$ For an example, see \cref{lastmax}.

\begin{figure}[htbp]
	\begin{minipage}[c]{0.65\textwidth}
		\centering
		\includegraphics[scale=.56]{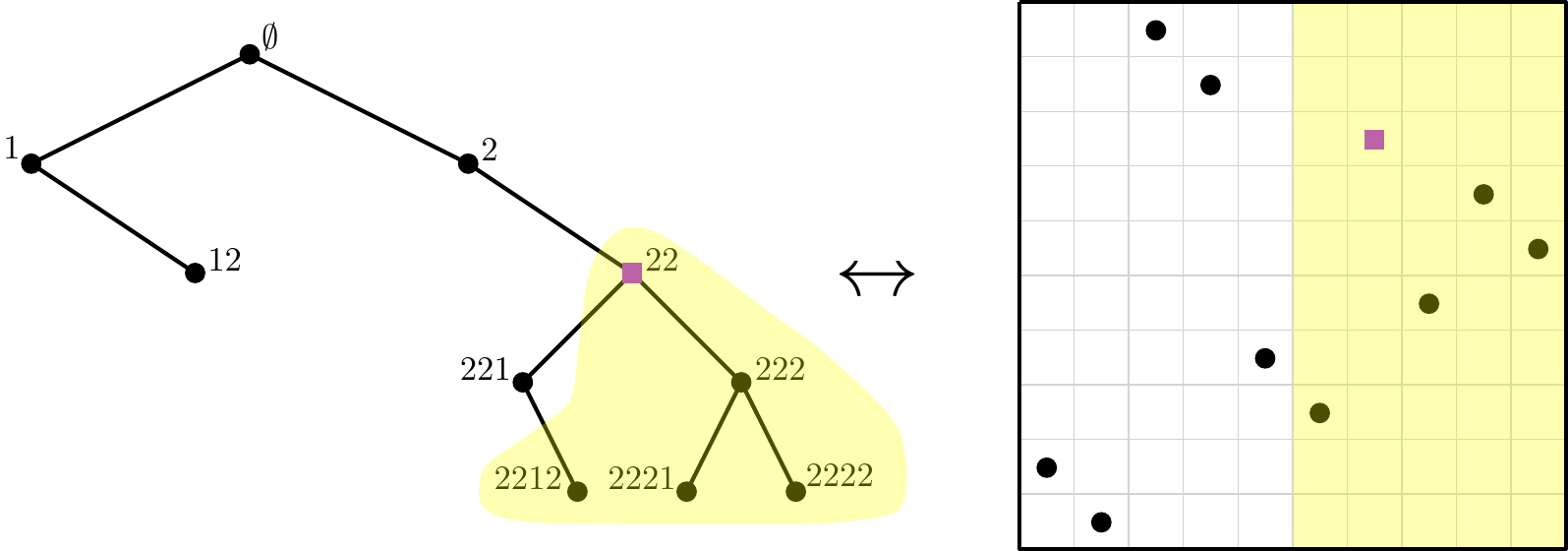}
	\end{minipage}
	\begin{minipage}[c]{0.34\textwidth}
		\caption{We marked with a purple square the maximum among the last 5 values (highlighted in yellow) of the permutation $\sigma=2\;1\;10\;9\;4\;3\;8\;5\;7\;6.$ \label{lastmax}}
	\end{minipage}
\end{figure}

Therefore, defining the following events, for all $n\in\Z_{\geq0},$ all $\pi\in\Av(231),$
\begin{equation}
M^n_{\pi}\coloneqq\big\{v_{2^n}\text{ is the vertex in } \bm{T} \text{ corresponding to the}
\text{ maximum among the last }|\pi|\text{ values of }\sigma_{\bm{T}}\big\},
\end{equation}
we have
$
\P\big(\text{pat}_e(\bm{T})=\pi)=\sum_{n=0}^\infty\P\big(\text{pat}_e(\bm{T})=\pi,M^n_{\pi}).
$
We also introduce the events, for all $n\in\Z_{\geq0},$
$$R^n\coloneqq\big\{\text{The vertex }v_{2^n}\text{ is in }\bm{T}\big\}.$$
Note that 
$
\sum_{n=0}^\infty\mathbb{P}(R^n)=\sum_{n=0}^\infty p^n=\frac{1}{1-p}.
$
The next lemma computes recursively $\P\big(\text{pat}_e(\bm{T})=\pi\big)$ and $\P\big(\text{pat}_b(\bm{T})=\pi\big)$.
\begin{lem}
	\label{twoandthreecross}
	Let $\pi\in\emph{Av}(231).$ Using notation as above,
	\begin{equation}
	\label{twocross}
	\P\big(\emph{pat}_e(\bm{T})=\pi\big)=
	\begin{cases}
	\frac{p^2}{1-p}\cdot\P\big(\emph{pat}_e(\bm{T})=\pi_L\big)\cdot\P\big(\bm{T}=T_{\pi_R}\big), &\quad\text{if }\pi_L\neq\emptyset,\;\pi_R\neq\emptyset,\\
	p\cdot\P\big(\emph{pat}_e(\bm{T})=\pi_L\big), &\quad\text{if }\pi_L\neq\emptyset,\;\pi_R=\emptyset,\\
	\frac{p}{1-p}\cdot\P\big(\bm{T}=T_{\pi_R}\big), &\quad\text{if }\pi_L=\emptyset,\;\pi_R\neq\emptyset,\\
	1, &\quad\text{if }\pi=1,
	\end{cases}
	\end{equation}
	and
	\begin{equation}
	\label{threecross}
	\P\big(\emph{pat}_b(\bm{T})=\pi\big)=
	\begin{cases}
	\frac{p^2}{1-p}\cdot\P\big(\bm{T}=T_{\pi_L}\big)\cdot\P\big(\emph{pat}_b(\bm{T})=\pi_R\big), &\quad\text{if }\pi_L\neq\emptyset,\;\pi_R\neq\emptyset,\\
	\frac{p}{1-p}\cdot\P\big(\bm{T}=T_{\pi_L}\big), &\quad\text{if }\pi_L\neq\emptyset,\;\pi_R=\emptyset,\\
	p\cdot\P\big(\emph{pat}_b(\bm{T})=\pi_R\big), &\quad\text{if }\pi_L=\emptyset,\;\pi_R\neq\emptyset,\\
	1, &\quad\text{if }\pi=1.
	\end{cases}
	\end{equation}
\end{lem}

\begin{proof}
	We only investigate $\P\big(\text{pat}_e(\bm{T})=\pi\big)$ assuming that $\pi_L\neq\emptyset,\;\pi_R\neq\emptyset$. The details for all the remaining cases can be found in the proof of \cite[Lemma 4.18]{borga2020localperm}.
	
	As we have seen above $\P\big(\text{pat}_e(\bm{T})=\pi)=\sum_{n=0}^\infty\P\big(\text{pat}_e(\bm{T})=\pi,M^n_{\pi})$. Conditioning on $v_{2^n}$ being the vertex in $\bm{T}$ corresponding to the maximum among the last $|\pi|$ values of $\sigma_{\bm{T}},$ then $\text{pat}_{e(|\pi|)}(\bm{T})=\text{pat}_{e(|\pi|)}(\bm{T}^{v_{2^n}}),$ and so 
	\begin{equation}
		\P\big(\text{pat}_e(\bm{T})=\pi)
		=\sum_{n\in\Z_{\geq0}}\P\big(\text{pat}_e(\bm{T}^{v_{2^n}}_L)=\pi_L,\bm{T}^{v_{2^n}}_R=T_{\pi_R},M^n_{\pi}).
	\end{equation}
	Since the event $\big\{\text{pat}_e(\bm{T}^{v_{2^n}}_L)=\pi_L\big\}\cap\big\{\bm{T}^{v_{2^n}}_R=T_{\pi_R}\big\}$ is contained both in $M^n_{\pi}$ and in $R^n,$ then
	\begin{equation}
			\P\big(\text{pat}_e(\bm{T})=\pi)=\sum_{n\in\Z_{\geq0}}\P\big(\text{pat}_e(\bm{T}^{v_{2^n}}_L)=\pi_L,\bm{T}^{v_{2^n}}_R=T_{\pi_R},R^n).
	\end{equation}
	Using the independence between $\bm{T}^{v_{2^n}}_L$ and $\bm{T}^{v_{2^n}}_R$ conditionally on $R^n$ and continuing the sequence of equalities,
	$$\P\big(\text{pat}_e(\bm{T})=\pi)=\sum_{n\in\Z_{\geq0}}\P\big(\text{pat}_e(\bm{T}^{v_{2^n}}_L)=\pi_L|R^n\big)\cdot\P\big(\bm{T}^{v_{2^n}}_R=T_{\pi_R}|R^n\big)\cdot\P\big(R^n).$$
	Now, noting that conditionally on $R^n,$ $\bm{T}^{v_{2^n}}_L$ is an independent copy of $\bm{T}$   with probability $p$ and empty with probability $1-p$ and the same obviously holds for $\bm{T}^{v_{2^n}}_R,$ we can rewrite the last term as
	\begin{equation}
			\begin{split}
				\P\big(\text{pat}_e(\bm{T})=\pi)&=\sum_{n\in\Z_{\geq0}}p^2\cdot\P\big(\text{pat}_e(\bm{T})=\pi_L\big)\cdot\P\big(\bm{T}=T_{\pi_R}\big)\cdot\P\big(R^n)\\
				&=p^2\cdot\P\big(\text{pat}_e(\bm{T})=\pi_L\big)\cdot\P\big(\bm{T}=T_{\pi_R}\big)\cdot\sum_{n\in\Z_{\geq0}}\P\big(R^n)\\
				&=\frac{p^2}{1-p}\cdot\P\big(\text{pat}_e(\bm{T})=\pi_L\big)\cdot\P\big(\bm{T}=T_{\pi_R}\big),
			\end{split}
	\end{equation}
	where in the last equality we used that $\sum_{n=0}^\infty\mathbb{P}(R^n)=\sum_{n=0}^\infty p^n=\frac{1}{1-p}$.
\end{proof}

We now continue the analysis of $\P\big(\text{pat}_{\bm{J}}(\bm{T})=\pi\big).$ In order to do that, we need a formula for $\P(\bm{T}=T),$ for a given tree $T,$ because such probabilities appear in \cref{twocross,threecross}.

\begin{obs}
	\label{obs:bintreecomp}
	In a binary tree with $n$ vertices, every vertex has two potential children. Out of these $2n$ potential children $n-1$ exist and $n+1$ do not exist. Hence
	\begin{equation}
	\label{bintreecomp}
	\P\big(\bm{T}=T\big)=p^{|T|-1}\cdot(1-p)^{|T|+1}.
	\end{equation}	
\end{obs}
Using together Lemmas \ref{onecross} and \ref{twoandthreecross} and the above observation we have an explicit recursion to compute the probability $\P\big(\text{pat}_{\bm{J}}(\bm{T})=\pi\big).$ We show an example of the recursion obtained for an explicit pattern $\pi$.

\begin{exmp}
	\label{bigexemp}
	Let $\pi$ be the following 231-avoiding permutation,  
	$$\pi=4\;1\;3\;2\;6\;5\;7\;10\;8\;9\;11\;12\;16\;13\;15\;14=
	\begin{array}{lcr}
	\begin{tikzpicture}
	\begin{scope}[scale=.25]
	\permutation{4,1,3,2,6,5,7,10,8,9,11,12,16,13,15,14}
	\draw (1+.5,4+.5) [green, fill] circle (.21);  
	\draw (5+.5,6+.5) [green, fill] circle (.21); 
	\draw (7+.5,7+.5) [green, fill] circle (.21); 
	\draw (8+.5,10+.5) [green, fill] circle (.21); 
	\draw (11+.5,11+.5) [green, fill] circle (.21); 
	\draw (12+.5,12+.5) [green, fill] circle (.21);
	\draw (13+.5,16+.5) [orange, fill] circle (.21);
	\draw (15+.5,15+.5) [blue, fill] circle (.21); 
	\draw (16+.5,14+.5) [blue, fill] circle (.21); 
	\draw[thick] (1,1) rectangle (17,17);
	\end{scope}
	\end{tikzpicture}
	\end{array},$$
	where we draw in green the left-to-right maxima, in blue the right-to-left maxima, and in orange the maximum.
	We now compute $\P\big(\text{pat}_{\bm{J}}(\bm{T})=\pi\big)$ using Lemmas \ref{onecross} and \ref{twoandthreecross}. First of all, we recursively split our permutation around its maximum as shown in \cref{decomp_tree}.
	
	\begin{figure}[htbp]
		\centering
			\includegraphics[scale=.60]{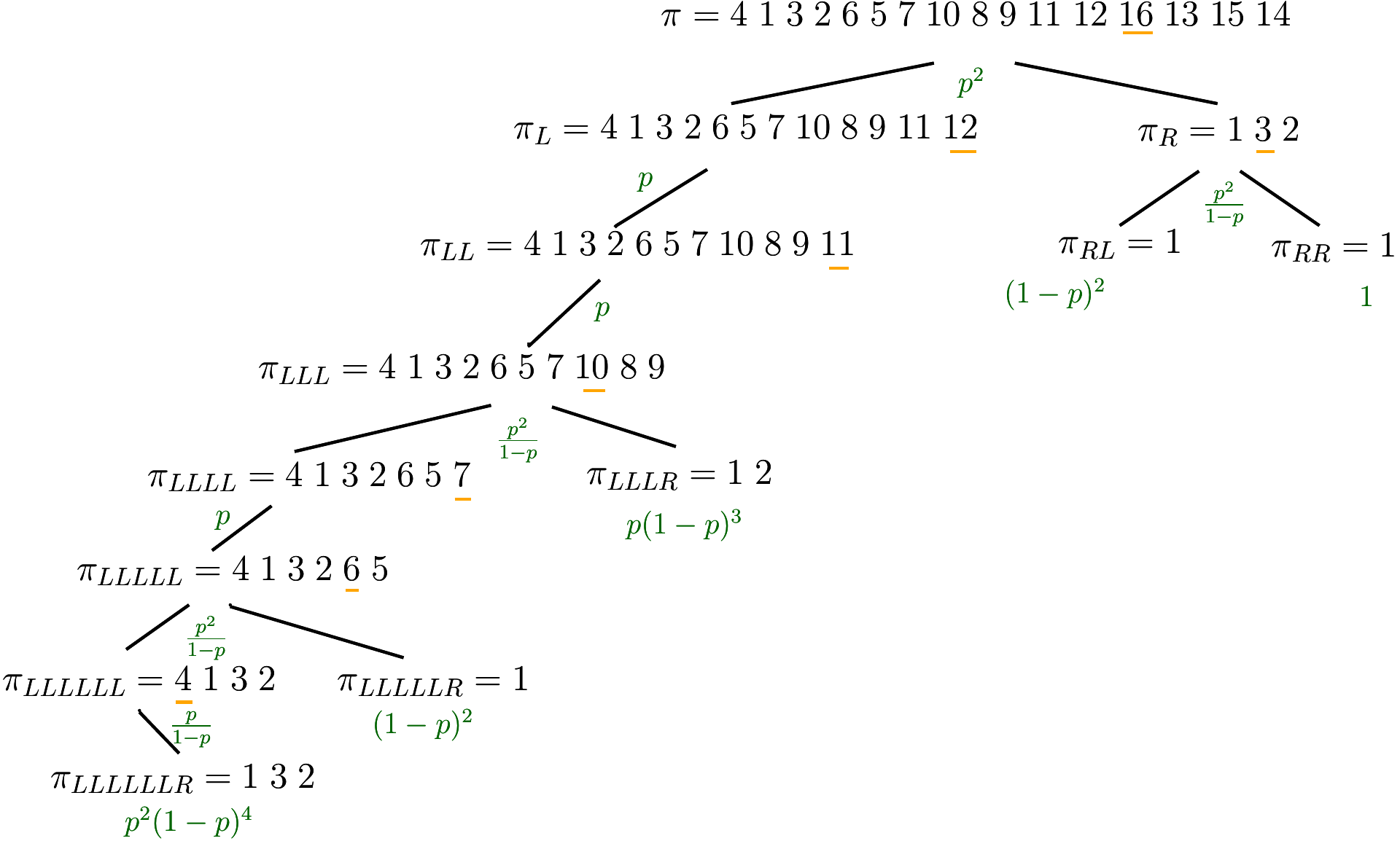}\\
			\caption{We draw the recursive decompositions in the left and right part (w.r.t.\ the position of the maximum underlined in orange) for the permutation $\pi=4\;1\;3\;2\;6\;5\;7\;10\;8\;9\;11\;12\;16\;13\;15\;14.$ Moreover we write in green the factors $p^\alpha\cdot(1-p)^{\beta}$ that we are adding at each step (coming from Lemmas \ref{onecross} and \ref{twoandthreecross}).}\label{decomp_tree}
	\end{figure}
	
	Using \cref{oneredcross} with the decomposition of $\pi$ in $\pi_L\pi(m)\pi_R$ (shown at the root of the tree in \cref{decomp_tree}),  we have
	\begin{equation}
	\P\big(\text{pat}_{\bm{J}}(\bm{T})=\pi\big)= p^2\cdot\P\big(\text{pat}_e(\bm{T})=\pi_L\big)\cdot\P\big(\text{pat}_b(\bm{T})=\pi_R\big).
	\end{equation}
	We continue decomposing $\pi_L$ and $\pi_R$ around their maxima (which correspond to the left and right children of the root in the tree in \cref{decomp_tree}). Using \cref{twocross} for the left part, we obtain 
	\begin{equation}
	\P\big(\text{pat}_e(\bm{T})=\pi_L\big)=p\cdot\P\big(\text{pat}_e(\bm{T})=\pi_{LL}\big),
	\end{equation}
	and using \cref{threecross} for the right part, we obtain
	\begin{equation}
	\P\big(\text{pat}_b(\bm{T})=\pi_R\big)=\frac{p^2}{1-p}\cdot\underbrace{\P\big(\bm{T}=T_{\pi_{RL}}\big)}_{\stackrel{(\ref{bintreecomp})}{=}(1-p)^2}\cdot\underbrace{\P\big(\text{pat}_b(\bm{T})=1\big)}_{1}.
	\end{equation}
	Therefore, summing up the last three equations and then proceeding similarly through the left subtree in \cref{decomp_tree}, we deduce that $\P\big(\text{pat}_{\bm{J}}(\bm{T})=\pi\big)$ is the product of all the green factors in \cref{decomp_tree}, that is
	\begin{equation}
	\P\big(\text{pat}_{\bm{J}}(\bm{T})=\pi\big)=p^{15}(1-p)^{7}=\Big(\frac{1-\delta}{2}\Big)^{15}\cdot\Big(\frac{1+\delta}{2}\Big)^{7}=\Big(\frac{1}{2}\Big)^{22}+O_{LP}(\delta),
	\end{equation}
	and this concludes \cref{bigexemp}.
\end{exmp}

We now proceed with the analysis of the general case. Using the recursion obtained by combining Lemmas \ref{onecross} and \ref{twoandthreecross} and \cref{obs:bintreecomp}, we immediately realize that
\begin{equation}
\label{vov0thg}
\P\big(\text{pat}_{\bm{J}}(\bm{T})=\pi\big)=p^{\alpha}\cdot(1-p)^{\beta},\quad\text{for some}\quad\alpha,\beta\in\Z_{\geq0}.
\end{equation}
Note that $\alpha\geq 0$ since in Lemmas \ref{onecross} and \ref{twoandthreecross} and Observation \ref{obs:bintreecomp} the $p^*$ factors always appear with nonnegative exponent. Moreover $\beta\geq 0$ since in \cref{twocross,threecross}, each time a $(1-p)^{-1}$ factor appears, there is also a $\P\big(\bm{T}=T\big)$ factor that contains a $(1-p)^{\gamma}$ factor with $\gamma\geq2.$
Since $p=\frac{1-\delta}{2},$ it follows that
$$\P\big(\text{pat}_{\bm{J}}(\bm{T})=\pi\big)=\Big(\frac{1}{2}-\frac{\delta}{2}\Big)^{\alpha}\cdot\Big(\frac{1}{2}+\frac{\delta}{2}\Big)^{\beta}=\Big(\frac{1}{2}\Big)^{\alpha+\beta}+O_{LP}(\delta),\quad\text{for some}\quad\alpha,\beta\in\Z_{\geq0}.$$
The exact value $\alpha+\beta$ is determined in the following proposition where a careful analysis on how many factors $p$ and $(1-p)$ appear in $\P\big(\text{pat}_{\bm{J}}(\bm{T})=\pi\big)$ was carried out (using the same ideas as in the example above).

\begin{prop}[{\cite[Proposition 4.22]{borga2020localperm}}]
	\label{patprob231}
	Let $\pi\in\emph{Av}(231)$ and $\bm{T}=\bm{T}_{\delta}$ be a Galton--Watson tree defined as above, then
	\begin{equation}
	\label{probpat231}
	\P\big(\emph{pat}_{\bm{J}}(\bm{T})=\pi\big)=\frac{2^{\#\emph{LRmax}(\pi)+\#\emph{RLmax}(\pi)}}{2^{2|\pi|}}+O_{LP}(\delta)=P_{231}(\pi)+O_{LP}(\delta).
	\end{equation}
\end{prop}

We can now prove Proposition \ref{weakprop}.

\begin{proof}[Proof of Proposition \ref{weakprop}]
	Summing up all the results and recalling that $\bm{T}=\bm{T}_{\delta}$, we obtain,
	\begin{equation}
	\begin{split}
	\label{coccT}
	\E\big[\cocc(\pi,\bm{T}_{\delta})\big]&\stackrel{(\ref{step1})}{=}\delta^{-1}\P\big(\text{pat}_{\bm{J}}(\bm{T}_{\delta})=\pi\big)\stackrel{(\ref{probpat231})}{=}\delta^{-1}\big(P_{231}(\pi)+O_{LP}(\delta)\big)\\&\;=\delta^{-1}\cdot P_{231}(\pi)+O_{LP}(1).
	\end{split}
	\end{equation}
	
	Applying Lemma \ref{svantelemma} and using the bijection between $231$-avoiding permutations and binary trees, we conclude that for $n\to\infty,$
	\begin{equation}
	\E\big[\cocc(\pi,\bm{\sigma}_n)\big]\sim P_{231}(\pi)\cdot n,\quad\text{for all}\quad \pi\in\Av(231). 
	\end{equation}
	Dividing by $n$ yields Proposition \ref{weakprop}.
\end{proof}

It remains to prove Theorem \ref{thm_1}. We are going to use the \emph{Second moment method}.
As before we start with a result regarding trees and then we will transfer it to permutations. The proof of the following result uses similar techniques as in the previous results, therefore we skip the details.
\begin{prop}[{\cite[Proposition 4.24]{borga2020localperm}}]
	\label{figrfigpqofhfpuh}
	Using notation as before and setting $\bm{T}=\bm{T}_\delta,$ we have,  
	$$\E\big[\cocc(\pi,\bm{T})^2\big]=\frac{P_{231}(\pi)^2}{2}\cdot\delta^{-3}+O_{LP}(\delta^{-2}),\quad\text{for all}\quad\pi\in\emph{Av}(231).$$
\end{prop}
	
We complete the proof of our main theorem.
\begin{proof}[Proof of Theorem \ref{thm_1}]
	Applying again Lemma \ref{svantelemma} and using our bijection between $231$-avoiding permutations and binary trees, we conclude that for $n\to\infty,$
	\begin{equation}
	\label{final1}
	\E\big[\cocc(\pi,\bm{\sigma}_n)^2\big]\sim P_{231}(\pi)^2\cdot n^2, \quad\text{for all}\quad \pi\in\Av(231).
	\end{equation}
	This, with Proposition \ref{weakprop}, implies that
	\begin{equation}
	\label{final2}
	\text{Var}\big(\widetilde{\cocc}(\pi,\bm{\sigma}_n)\big)\to 0, \quad\text{for all}\quad \pi\in\Av(231).
	\end{equation}
	We can apply the Second moment method and deduce that
	\begin{equation}
	\widetilde{\cocc}(\pi,\bm{\sigma}_n)\stackrel{P}{\rightarrow}P_{231}(\pi),\quad\text{for all}\quad\pi\in\Av(231).
	\end{equation}
	This ends the proof.	
\end{proof}

\subsection{Construction of the limiting object}
\label{explcon}

We now exhibit an explicit construction of the limiting object $\bm{\sigma}^{\infty}_{231}$ as a random total order $\bm{\preccurlyeq}_{231}$ on $\mathbb{Z}.$ 

\begin{rem}
	The intuition behind the construction that we are going to present comes from the local limit for uniform binary trees (for more details, see \cite{stufler2016local}). When the size of a uniform binary tree tends to infinity, looking around a uniform distinguished vertex, we see an infinite upward spine. Each vertex in this spine is the left or the right child of the previous one with probability $1/2.$ Moreover, attached to this infinite spine, there are some independent copies of binary Galton--Watson trees.
	Using this idea and the bijection between binary trees and $231$-avoiding permutations we are going to construct the limiting random total order.
	
	Note however that this intuition is not formally needed in the following since we present the construction of the limiting object using the permutation point of view. 
\end{rem}

We have to introduce some notation. We define two operations from $\Sr_{k}\times\mathcal{S}_{\ell}$ to $\Sr_{k+\ell+1},$ for $k>0$ and $\ell\geq0$ ($\mathcal{S}_{0}$ is the set containing the empty permutation): let $(\sigma,i)\in\Sr_{k}$ be a rooted permutation and $\pi\in\mathcal{S}_{\ell}$ be another (unrooted) permutation, we set
\begin{equation}
\begin{split}
&(\sigma,i)*^R \pi\coloneqq\big(\sigma(1)\dots\sigma(k)(\ell+k+1)(\pi(1)+k)\dots(\pi(\ell)+k)\,,\,i\big),\\
&(\sigma,i)*^L \pi\coloneqq\big(\pi(1)\dots\pi(\ell)(\ell+k+1)(\sigma(1)+\ell)\dots(\sigma(k)+\ell)\,,\,\ell+i+1\big).
\end{split}
\end{equation}

In words, from a graphical point of view, the diagram of $(\sigma,i)*^R \pi$ (resp.\ $(\sigma,i)*^L \pi$) is obtained starting from the diagram of the rooted permutation $(\sigma,i),$ adding on the top-right (resp.\ bottom-left) the diagram of $\pi$ and adding a new maximal element between the two diagrams. We give an example below.
\begin{exmp}
	Let $(\sigma,i)=(132,3)$ and $\pi=21$ then
	$$(\sigma,i)*^R \pi=
	\begin{array}{lcr}
	\includegraphics[scale=0.6]{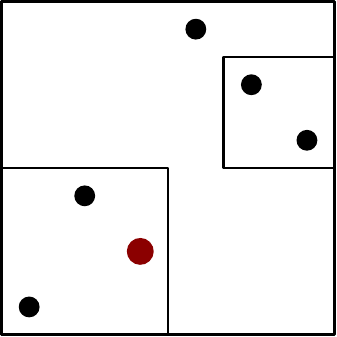}\\
	\end{array}=(132621,3)\quad\text{and}\quad(\sigma,i)*^L \pi=
	\begin{array}{lcr}
	\includegraphics[scale=0.6]{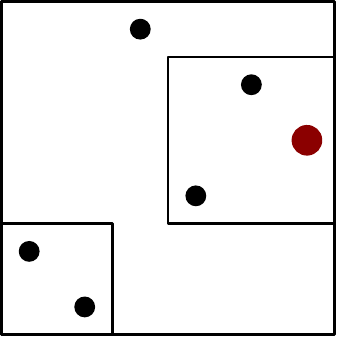}\\
	\end{array}=(216354,6).$$	
\end{exmp}

\begin{obs}
	\label{leftrightrem}
	Note that, given a rooted permutation $(\sigma,i),$ for all permutations $\pi$ (possibly empty), the rooted permutation $(\sigma,i)*^R \pi$ (resp.\ $(\sigma,i)*^L \pi$) contains at least one more element on the right (resp.\ left) of the root than $(\sigma,i)$.
\end{obs}

We now consider two families of random variables: let $\{\bm{S}^i\}_{i\geq 0}$ be i.i.d.\ random variables with values in $\Av(231)\cup\{\emptyset\}$ and Boltzman distribution equal to 
$$\P(\bm{S}^i=\pi)=\frac{1}{2}\Big(\frac{1}{4}\Big)^{|\pi|},\quad\text{for all}\quad\pi\in\Av(231)\quad\text{and}\quad\P(\bm{S}^i=\emptyset)=\frac{1}{2}.$$
This is a probability distribution since the generating function for the class $\Av(231)$ is given by $A(z)=\frac{1-\sqrt{1-4z}}{2z}$ (see  for instance \cite[Appendix A]{borga2020localperm}).  Moreover, let $\{\bm{X}_i\}_{i\geq 1}$ be i.i.d.\ random variables (also independent of $\{\bm{S}^i\}_{i\geq 0}$) with values in the set $\{R,L\}$ such that
$$\P(\bm{X}_i=\,R)=\frac{1}{2}=\P(\bm{X}_i=\,L).$$

\begin{obs}
	Recalling that $\bm{T}_\delta$ is the binary Galton--Watson tree with offspring distribution $\eta(\delta)$ defined by \cref{bintree}, we note that, for $\delta=0,$ $\sigma_{\bm{T}_{0}}\stackrel{(d)}{=}(\bm{S}^i|\bm{S}^i\neq\emptyset).$
\end{obs}

We are now ready to construct our random order $\bm{\preccurlyeq}_{231}$ on $\mathbb{Z}.$ Set $\tilde{\bm{S}}^0_{\bullet}\coloneqq(\tilde{\bm{S}}^0,\bm{m}),$ where $\tilde{\bm{S}}^0$ is the random variable $\bm{S}^0$ conditioned to be non empty, and $\bm{m}$ is the (random) index of the maximum of $\tilde{\bm{S}}^0.$ Then we set,
$$\bm{\sigma}_n^{\bullet}\coloneqq((\dots((\tilde{\bm{S}}^0_{\bullet}*^{\bm{X}_1}\bm{S}^1)*^{\bm{X}_2}\bm{S}^2)...\bm{S}^{n-1})*^{\bm{X}_n}\bm{S}^n),\quad\text{for all}\quad n\in\Z_{>0}.$$
Note that, for every fixed $h\in\Z_{>0},$ the sequence $\big\{r_h(\bm{\sigma}_n^{\bullet})\big\}_{n\in\Z_{>0}}$ is a.s.\ stationary (hence a.s.\ convergent). Let $\bm{\tau}_h$ be the limit of this sequence. The family $\big\{\bm{\tau}_h\big\}_{h\in\Z_{>0}}$ is a.s.\ consistent. Therefore, applying \cref{consistprop} page \pageref{consistprop}, the family $\big\{\bm{\tau}_h\big\}_{h\in\Z_{>0}}$ determines a.s.\ a unique random total order $(\Z,\bm{\preccurlyeq}_{231})$ such that, for all $h\in\Z_{>0},$ $$r_h(\Z,\bm{\preccurlyeq}_{231})=\bm{\tau}_h=\lim_{n\to\infty}r_h(\bm{\sigma}_n^{\bullet}),\quad\text{a.s.}$$
\begin{prop}
	Let $(\Z,\bm{\preccurlyeq}_{231})$ be the random total order defined above and $\bm{\sigma}^\infty_{231}$ be the limiting object defined in Corollary \ref{231corol}. Then
	$(\Z,\bm{\preccurlyeq}_{231})\stackrel{(d)}{=}\bm{\sigma}^\infty_{231}$.
\end{prop}

The (technical) proof of the proposition above can be found in \cite[Proposition 4.29]{borga2020localperm}.

\section{321-avoiding permutations}
\label{321}

The goal of this section is to prove -- using completely different techniques compared to the ones used in the previous section -- a law of large numbers for consecutive patterns of uniform random $321$-avoiding permutations, i.e.\ \cref{thm_2}. As before, we have the following consequence of \cref{thm_2}.

\begin{cor}
	\label{321corol}
	Let $\bm{\sigma}_n$ be a uniform random permutation in $\Av_n(321)$ for all $n\in\Z_{>0}.$ There exists a random infinite rooted permutation $\bm{\sigma}^\infty_{321}$ such that   $\bm{\sigma}_n\stackrel{qBS}{\longrightarrow}\mathcal{L}aw(\bm{\sigma}^\infty_{321}).$
\end{cor}	
An explicit construction of $\bm{\sigma}^\infty_{321}$ is given in \cref{explcon2}.

\subsection{Existence of the separating line and asymptotic independence of points}
\label{mainres_321}

We start by introducing, for all $k>0,$ the following probability distribution on $\Av_k(321)$,
\begin{equation}
	\label{321distr}
	P_{321}(\pi)\coloneqq
	\begin{cases}
		\frac{|\pi|+1}{2^{|\pi|}} &\quad\text{if }\pi=12\dots|\pi|,\\
		\frac{1}{2^{|\pi|}} &\quad\text{if }\cocc(21,\pi^{-1})=1, \\
		0 &\quad\text{otherwise.} \\ 
	\end{cases}
\end{equation}

\begin{rem}
	\label{rem_sep_line}
	We note that the second condition $\cocc(21,\pi^{-1})=1$ (i.e.\ $\pi^{-1}$ has exactly one descent, or equivalently, $\pi$ has exactly one \emph{inverse descent}) is equivalent to saying that we can split the permutation $\pi$ in exactly two increasing subsequences such that all the values of one subsequence are smaller than the values of the other one. Formally, if $\pi\in\Av_k(321)$ and $\cocc(21,\pi^{-1})=1$ there exists a unique bipartition of $[k]$ as $[k]=L(\pi)\sqcup U(\pi)$ such that for all $(i,j)\in L(\pi)\times U(\pi)$, $\pi(i)<\pi(j)$ and $\text{pat}_{L(\pi)}(\pi)$ and $\text{pat}_{U(\pi)}(\pi)$ are increasing.
	
	Given the diagram of a permutation in $\Av_k(321)$ with exactly one inverse descent, we call \emph{separating line} the horizontal line in the diagram between these two increasing subsequences. An example is given in \cref{321ex} in the case of the permutation $\pi=14526738,$ where we highlight in orange the separating line between the two increasing subsequences.  
\end{rem}

\begin{figure}[htbp]
	\begin{minipage}[c]{0.5\textwidth}
		\centering
		\includegraphics[scale=.70]{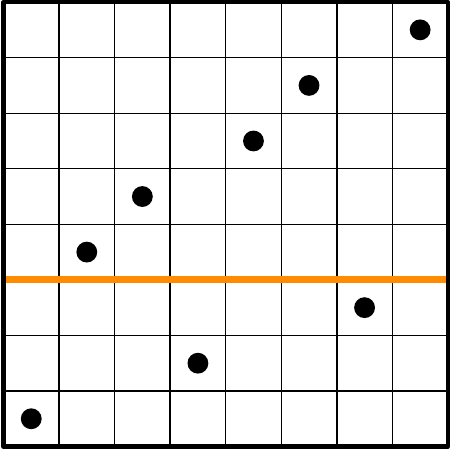}
	\end{minipage}
	\begin{minipage}[c]{0.49\textwidth}
		\caption{Diagram of the permutation $14526738$ with the separating line highlighted in orange.\label{321ex}}
	\end{minipage}
\end{figure} 

\begin{rem}
	\cref{321distr} defines a probability distribution on $\Av_k(321)$. This is a consequence of Theorem \ref{thm_2}: in \cref{shsshgea} page \pageref{shsshgea}, summing over all $\pi\in\Av_k(321)$ the left-hand side, we obtain $\frac{n-k+1}{n}$ which trivially tends to 1. It is however easy to give a direct explanation. Indeed the number of permutations in $\Av_k(321)$ with exactly one inverse descent and separating line at height $\ell\in\{1,\dots,k-1\}$ is $\binom{k}{\ell}-1$ (where the binomial coefficient $\binom{k}{\ell}$ counts the possible choices of points that are below the separating line and the term $-1$ stands for the identity permutation which has no inverse descent). Therefore the number of permutations in $\Av_k(321)$ with exactly one inverse descent is
	$$\sum_{\ell=1}^{k-1}\Bigg(\binom{k}{\ell}-1\Bigg)=2^k-(k+1),$$ 
	which entails $\sum_{\pi\in\Av_k(321)}P_{321}(\pi)=1.$
\end{rem}

We now state (without proof) two propositions that are the key results for the proof of Theorem \ref{thm_2}. We need the following.

\begin{defn}
%
	Given a permutation $\sigma\in\Av_n(321),$ we define for $i,k\in[n],$
	$$E^+_{i,k}(\sigma)=\big\{x\in[-k,k]:x+i\in E^+(\sigma)\big\},$$
	where we recall that $E^+(\sigma)\coloneqq\big\{i\in[n]|\sigma(i)\geq i\big\}$.
\end{defn}
%

\begin{prop}[{\cite[Proposition 5.20]{borga2020localperm}}]\label{prop1}
	For all $n\in\mathbb{Z}_{>0},$ let $\bm{\sigma}_n$ be a uniform random permutation in $\Av_n(321).$ Fix $k\in\mathbb{Z}_{>0}$ and let $\bm{i}_n$ be uniform in $[n]$ and independent of $\bm{\sigma}_n.$ Then, for every subset $A\subseteq[-k,k],$ as $n\to\infty,$
	$$\P\left(E^+_{\bm{i}_n,k}(\bm{\sigma}_n)=A \middle| \bm{\sigma}_n\right)\stackrel{P}{\to}\Big(\frac{1}{2}\Big)^{2k+1}.$$
\end{prop}

Informally, conditionally on $\bm{\sigma}_n,$ every index of $r_k(\bm{\sigma}_n,\bm{i}_n)$ is asymptotically equiprobably in $E^+$ or $E^-.$ Moreover, these events are asymptotically independent.

The proof of the proposition above follows from a combination of two results due to Stufler \cite{stufler2016local} and Janson \cite{janson2012simply} that characterize the local limit for Galton--Watson trees pointed at a uniform vertex (for more details see \cite[Proposition 5.9]{borga2020localperm}) and using the bijection between 321-avoiding permutations and rooted plane trees introduce in \cref{premres_321}. In particular, we were able to analyze the shape of the contour function of the local limiting infinite tree showing that it is a Dyck path with balanced up and down steps (see \cite[Lemma 5.18]{borga2020localperm}).

\medskip

Before stating our second proposition, we introduce some more notation (see Fig.~\ref{min_max_perm} below). Given a permutation $\sigma\in\text{Av}(321)$ we define, for $i\in[|\sigma|],$ $k\in\mathbb{Z}_{>0},$
\begin{equation}
\begin{split}
&m_{i,k}^+(\sigma)=\min\big\{E^+(\sigma)\cap[i-k,i+k]\big\},\\
&M_{i,k}^-(\sigma)=\max\big\{E^-(\sigma)\cap[i-k,i+k]\big\},
\end{split}
\end{equation}
with the conventions that $\min{\emptyset}=+\infty$ and $\max{\emptyset}=-\infty.$

\begin{figure}[htbp]
	\begin{minipage}[c]{0.5\textwidth}
		\centering
		\includegraphics[scale=0.35]{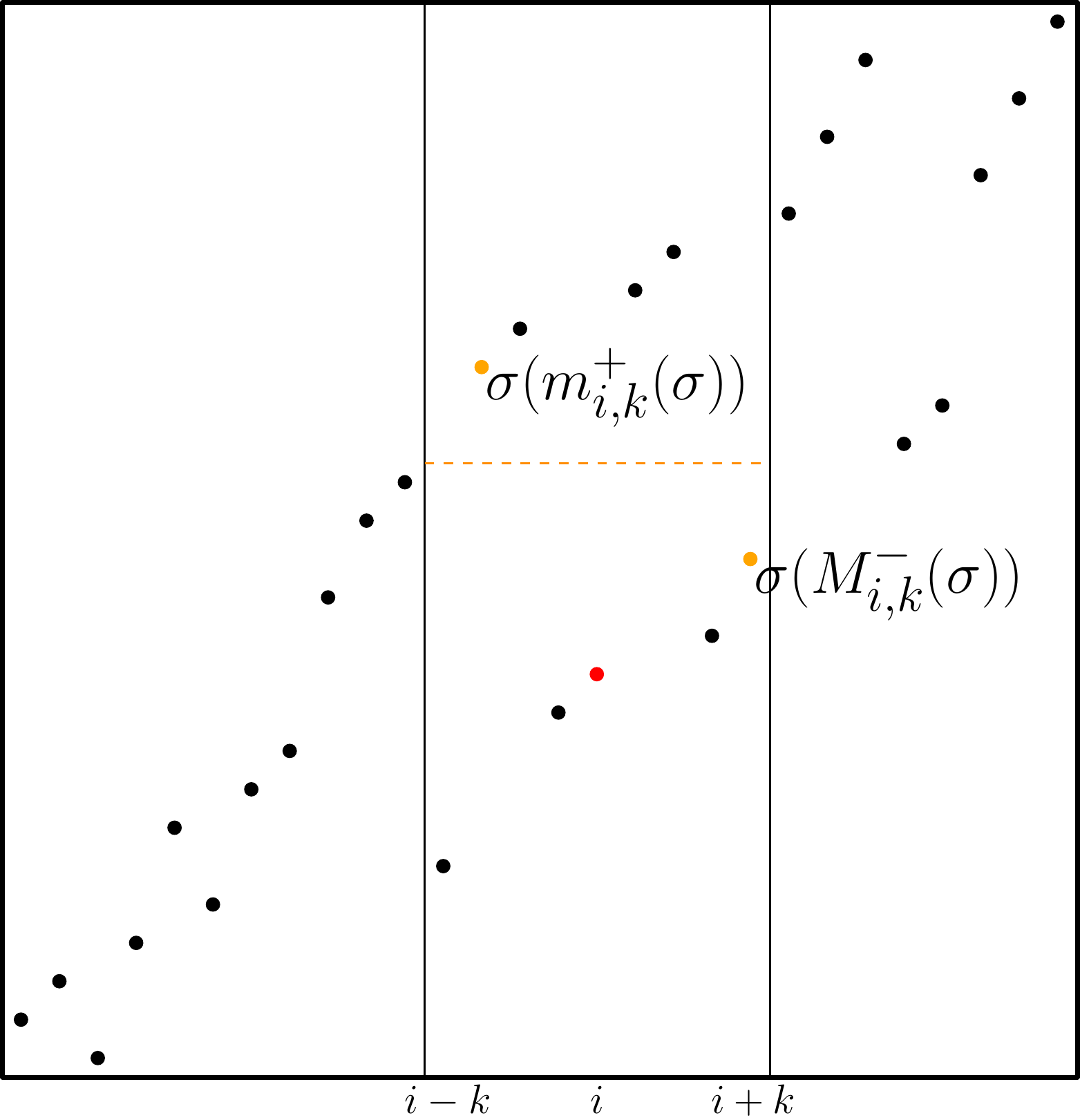}
	\end{minipage}
	\begin{minipage}[c]{0.49\textwidth}
		\caption{A 321-avoiding permutation. The orange dots identify the values $\sigma(m_{i,k}^+(\sigma))$ and $\sigma(M_{i,k}^-(\sigma))$ inside the vertical strip centered in $i$ of width $2k+1.$ Moreover, the dashed orange line identifies the separating line.\label{min_max_perm}}
	\end{minipage}
\end{figure} 

\noindent We also define, for a random 321-avoiding permutation $\bm{\sigma}$ and a uniform index $\bm{i}\in[|\sigma|],$ the following events, for all $k\in\mathbb{Z}_{>0},$
\begin{equation}
\begin{split}
S_k(\bm{\sigma},\bm{i})=&\big\{\bm{\sigma}(m_{\bm{i},k}^+(\bm{\sigma}))>\bm{\bm{\sigma}}(M_{\bm{i},k}^-(\bm{\sigma})),m_{\bm{i},k}^+(\bm{\sigma})\neq +\infty,M_{\bm{i},k}^-(\bm{\sigma})\neq -\infty\big\}\\
&\cup\big\{m_{\bm{i},k}^+(\bm{\sigma})= +\infty\big\}\cup\big\{M_{\bm{i},k}^-(\bm{\sigma})= -\infty\big\}.
\end{split}
\end{equation}
This is the event that the random rooted permutation $r_k(\bm{\sigma},\bm{i})$ splits into two (possibly empty) increasing subsequences separated by a separating line (in \cref{min_max_perm} this separating line is dashed in orange).

\begin{prop}[{\cite[Proposition 5.21]{borga2020localperm}}]\label{prop2}
	For all $n\in\mathbb{Z}_{>0},$ let $\bm{\sigma}_n$ be a uniform random permutation in $\Av_n(321).$ Fix $k\in\mathbb{Z}_{>0}$ and let $\bm{i}_n$ be uniform in $[n]$ and independent of $\bm{\sigma}_n.$ Then,  as $n\to\infty,$
	$$\P\left(S_k(\bm{\sigma}_n,\bm{i}_n) \middle| \bm{\sigma}_n\right)\stackrel{P}{\longrightarrow}1.$$
\end{prop}

Informally, conditionally on $\bm{\sigma}_n,$ asymptotically with high probability there exists a separating line in $r_k(\bm{\sigma}_n,\bm{i}_n)$.
We give a quick idea of the proof: by a result from \cite{hoffman2017pattern}, the points of a large uniform 321-avoiding permutation of size $n$ are at distance of order $\sqrt n$ from the diagonal $x=y.$ Therefore, in a window of constant width, the points above the diagonal are all higher than the points below the diagonal.

\begin{rem}
	A main difficulty in the proofs of Propositions \ref{prop1} and \ref{prop2} is that most of the results that we used were available only in an annealed version and so we had to extend them to the quenched setting.
\end{rem}

With these two propositions in our hands, we can prove Theorem \ref{thm_2}.
\begin{proof}[Proof of Theorem \ref{thm_2}]
	In order to prove \cref{primaeq} page \pageref{primaeq}, we can check the equivalent condition (d) of \cref{detstrongbsconditions} page \pageref{detstrongbsconditions}, i.e.\ we have to prove that, for all $h\in\Z_{>0},$ for all $\pi\in\text{Av}_{2h+1}(321),$
	\begin{equation}
	\label{checkthat}
	\P\left(r_h(\bm{\sigma}_n,\bm{i}_n)=(\pi,h+1) \middle| \bm{\sigma}_n\right)\xrightarrow{P}P_{321}(\pi).
	\end{equation}
	We fix $h\in\Z_{>0}$ and we distinguish three different cases for $\pi\in\text{Av}_{2h+1}(321)$:
	
	\underline{$\cocc(21,\pi^{-1})=1:$} We recall (see Remark \ref{rem_sep_line}) that, in this case, $\pi$ splits in a unique way into two increasing subsequences whose indices are denoted by $L(\pi)$ and $U(\pi).$ Note that, in general, these two sets are different from $E^-(\pi)$ and $E^+(\pi)$ (specifically when the permutation has some fixed points at the beginning, for example as in the case of Fig.~\ref{321ex} page~\pageref{321ex}). We set $U^*(\pi)\coloneqq\{j\in[-h,h]:j+h\in U(\pi)\},$ \emph{i.e.} $U^*(\pi)$ denotes the shift of the set $U(\pi)$ inside the interval $[-h,h].$ 
	Conditioning on $S_h(\bm{\sigma}_n,\bm{i}_n),$ namely assuming there is a separating line, $r_h(\bm{\sigma}_n,\bm{i}_n)$ is the union of two increasing subsequences with $U^*(r_h(\bm{\sigma}_n,\bm{i}_n))=E^+_{\bm{i}_n,h}(\bm{\sigma}_n).$ Then, conditioning on $S_h(\bm{\sigma}_n,\bm{i}_n)$, $r_h(\bm{\sigma}_n,\bm{i}_n)=(\pi,h+1)$ if and only if $,E^+_{\bm{i}_n,h}(\bm{\sigma}_n)=U^*(\pi)$. In particular, using Proposition \ref{prop2}, asymptotically we have that,
	$$\P\left(r_h(\bm{\sigma}_n,\bm{i}_n)=(\pi,h+1) \middle| \bm{\sigma}_n\right)=\P\left(S_h(\bm{\sigma}_n,\bm{i}_n),E^+_{\bm{i}_n,h}(\bm{\sigma}_n)=U^*(\pi) \middle| \bm{\sigma}_n\right)+o(1).$$
	By Proposition \ref{prop1} and \ref{prop2}, we have that
	$$\P\left(S_h(\bm{\sigma}_n,\bm{i}_n),E^+_{\bm{i}_n,h}(\bm{\sigma}_n)=U^*(\pi) \middle| \bm{\sigma}_n\right)\stackrel{P}{\to}\frac{1}{2^{2h+1}}=P_{321}(\pi),$$
	and so \cref{checkthat} holds.
	
	\underline{$\pi=12\dots|\pi|$:} In this case $\pi$ does not split in a unique way into two increasing subsequences. We have exacly $2h+2$ different choices for this splitting, namely setting either $U(\pi)=\emptyset$ or $U(\pi)=[k,h],$ for some $k\in[-h,h].$ Therefore, using Proposition \ref{prop2}, asymptotically we obtain that
	$$\P\left(r_h(\bm{\sigma}_n,\bm{i}_n)=(\pi,h+1) \middle| \bm{\sigma}_n\right)=\P\left(S_h(\bm{\sigma}_n,\bm{i}_n),E^+_{\bm{i}_n,h}(\bm{\sigma}_n)\in\{\emptyset\}\cup\bigcup_{k\in[-h,h]}\{[k,h]\} \middle| \bm{\sigma}_n\right)+o(1).$$
	Again by Proposition \ref{prop1} and \ref{prop2}, we have that
	$$\P\left(S_h(\bm{\sigma}_n,\bm{i}_n),E^+_{\bm{i}_n,k}(\bm{\sigma}_n)\in\{\emptyset\}\cup\bigcup_{k\in[-h,h]}\{[k,h]\} \middle| \bm{\sigma}_n\right)\stackrel{P}{\to}\frac{2h+2}{2^{2h+1}}=P_{321}(\pi),$$
	and so \cref{checkthat} holds.
	
	\underline{$\cocc(21,\pi^{-1})>1:$} In this case $\pi$ does not present a separating line. Since by Proposition \ref{prop2} we know that asymptotically almost surely there exists a separating line in $r_h(\bm{\sigma}_n,\bm{i}_n)$ we can immediately conclude that,
	\begin{equation}
	\P\left(r_h(\bm{\sigma}_n,\bm{i}_n)=(\pi,h+1) \middle| \bm{\sigma}_n\right)\leq\P\left(S_h(\bm{\sigma}_n,\bm{i}_n)^c \middle| \bm{\sigma}_n\right)\stackrel{P}{\to}0=P_{321}(\pi).
	\end{equation}
	This ends the proof.
\end{proof}

\subsection{Construction of the limiting object}
\label{explcon2}

We now exhibit an explicit construction of the limiting object $\bm{\sigma}_{321}^{\infty}$ as a random order $\bm{\preccurlyeq}_{321}$ on $\Z$ that is reminiscent of the one presented for square permutations. 

We consider the set of integer numbers $\Z,$ and we label, uniformly and independently, each integer with a "$+$" or a "$-$": namely, for all $x\in\Z,$ 
$$\P(x \text{ has label }"+")=\frac{1}{2}=\P(x \text{ has label }"-").$$
We set $A^+\coloneqq\{x\in\Z:x\text{ has label }"+"\}$ and $A^-\coloneqq\{x\in\Z:x\text{ has label }"-"\}.$
Then we define a random total order $\bm{\preccurlyeq}_{321}$ on $\Z$ saying that, for all $x,y\in\Z,$ $x\bm{\preccurlyeq}_{321}y$ if either
$x<y$ and $x,y\in A^-,$ or $x<y$ and $x,y\in A^+,$ or $x\in A^-$ and $y\in A^+.$ An intuitive description is given in \cref{constr_order321}.

\begin{prop}[{\cite[Proposition 5.27]{borga2020localperm}}]
	Let $(\Z,\bm{\preccurlyeq}_{321})$ be the random total order defined above and $\bm{\sigma}^\infty_{321}$ be the limiting object defined in Corollary \ref{321corol}. Then
	$$(\Z,\bm{\preccurlyeq}_{321})\stackrel{(d)}{=}\bm{\sigma}^\infty_{321}.$$
\end{prop}

\begin{figure}[htbp]
\centering
		\includegraphics[scale=1]{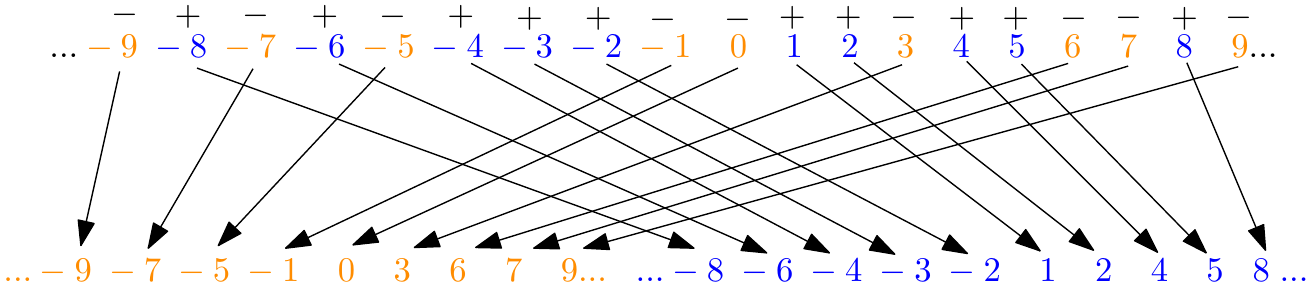}\\
		\caption{A sample of the random total order $(\Z,\bm{\preccurlyeq}_{321}).$ On the top, the standard total order on $\Z$ with the integers labeled by random "-" and "+" signs. We paint in orange the integers labeled by "-" and in blue the integers labeled by "+." Then, in the bottom part of the picture, we move the orange numbers at the beginning of the new random total order and the blue numbers at the end. Reading the bottom line from left to right gives the random total order $\bm{\preccurlyeq}_{321}$ on $\Z.$ }\label{constr_order321}
\end{figure}

\section{Substitution-closed classes}
\label{sec:local_lim}

In this section we investigate the local limits of uniform permutations in a fixed substitution-closed class $\mathcal{C}$, proving \cref{thm:local_intro}. We proceed as follows:
\begin{itemize}
	\item In \cref{sect:loca_top_trees} we introduce a local topology for decorated trees with a distinguished leaf.
	\item In \cref{sec:limiting_local_trees} we prove quenched local convergence  for random packed trees around a uniform leaf (when the number of leaves tends to infinity).
	\item In \cref{sec:continuity_bij} we extend to infinite objects the bijection $\DT:=\Pack\circ\CanTree$ between $\oplus$-indecomposable permutations in $\mathcal C$ and finite packed trees (recall \cref{le:bij_perm_tree} page \pageref{le:bij_perm_tree}). We show that this new map is a.s.\ continuous for the local topologies.
	\item Finally, in \cref{sect:local_lim_sub}, we complete the proof of \cref{thm:local_intro} building on the results proved in the previous two sections.
\end{itemize}

\subsection{Local topology for decorated trees}\label{sect:loca_top_trees}

We introduce a local topology for decorated trees with a distinguished {\em leaf}
(called pointed trees in the sequel).
This is a straight forward adaptation
of that for trees with a distinguished {\em vertex}
introduced by Stufler in \cite{stufler2019local}.

Following the presentation in \cite[Section 6.3.1]{stufler2016limits}, 
we start by defining an infinite pointed plane tree $\Uinf$ (see \cref{fig:Stufler_tree} below).
This infinite tree is meant to be a pointed analogue of Ulam--Harris tree, so that
pointed trees will be seen as subsets of it.
To construct $\Uinf$, we take a spine $(u_i)_{i\geq 0}$ that grows
downwards, that is, such that $u_i$ is the parent of $u_{i-1}$ for all $i\geq 1$. Any vertex $u_i,$ with $i\geq 1,$ has
an infinite number of children to the left and to the right  of its distinguished offspring $u_{i-1}$. 
The former are ordered from right to left and denoted by $(v^{i}_{L,j})_{j\geq1}$, the latter are ordered from left to right and denoted by $(v^{i}_{R,j})_{j\geq1}$. 
Each of these vertices not belonging to the spine $(u_i)_{i\geq 0}$ is the root of a copy of the Ulam--Harris tree $\mathcal{U}_{\infty}$. 
We always think of $\Uinf$ as a tree with distinguished leaf $u_0.$

\begin{figure}[htbp]
	\begin{minipage}[c]{0.6\textwidth}
		\centering
		\includegraphics[height=4.5cm]{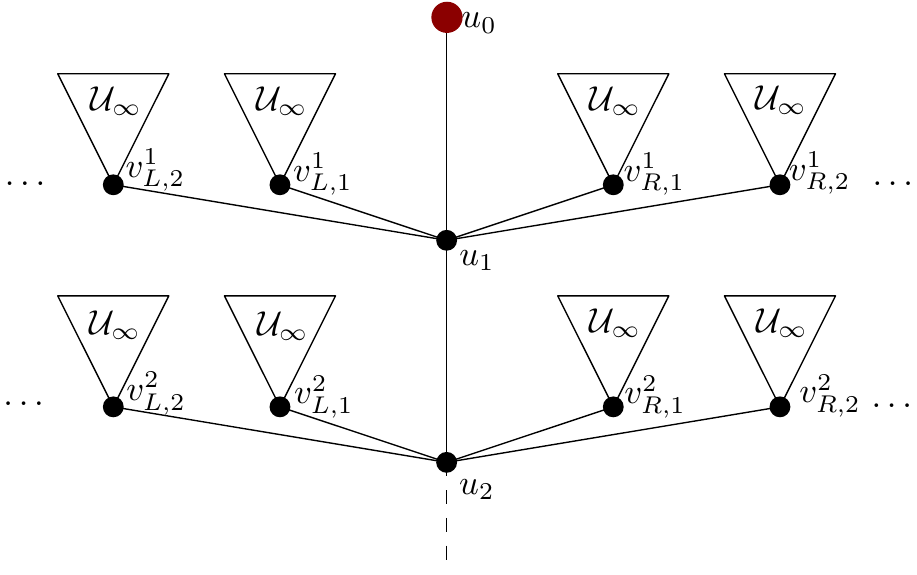}
	\end{minipage}
	\begin{minipage}[c]{0.39\textwidth}
		\caption{A schema of the infinite plane tree $\Uinf$.\label{fig:Stufler_tree}}
	\end{minipage}
\end{figure}

\begin{defn}
	A \emph{(possibly infinite) pointed plane tree} $T^{\bullet}$ is a subset of $\Uinf$ such that
	\begin{itemize}
		\item $u_0\in T^{\bullet}.$
		\item if $u_p\in T^{\bullet}$ then $u_i\in T^{\bullet},$ for all $0\leq i \leq p.$
		\item if $v^{i}_{L,q}\in T^{\bullet}$ (resp.\ $v^{i}_{R,q}\in T^{\bullet}$) then $u^{i}\in T^{\bullet}$ and $v^{i}_{L,j}\in T^{\bullet}$ (resp.\ $v^{i}_{R,j}\in T^{\bullet}$) for all $1\leq j \leq q.$
		\item Any maximal subset of $T^{\bullet}$ contained in 
		one of the Ulam--Harris trees $\mathcal{U}_{\infty}$  of $\Uinf$ is a (finite or infinite) plane tree.
	\end{itemize}
	We denote by $\gls*{pointed_plane_trees}$ the space of (possibly infinite) pointed plane trees.
\end{defn}
We say that a pointed tree $T^{\bullet}$ in $\mathfrak{T}^\bullet$ is {\em locally and upwards finite} if 
every vertex has finite degree and the intersection of $T^\bullet$ with any of the Ulam--Harris 
trees $\mathcal{U}_{\infty}$  of $\Uinf$ is finite.
The set of locally and upwards finite pointed trees will be denoted by $\gls*{luf_pointed_plane_trees}$.

Any finite plane tree $T$ together with a distinguished leaf $v_0$
may be interpreted in a canonical way as a pointed plane tree $T^\bullet$, such that $v_0$ is mapped to $u_0$.
In particular, the backward spine $u_0,u_1,\cdots$ of the associated pointed plane tree $T^\bullet$ is finite and ends at the root of $T$.
\medskip

Next, we need to extend this notion to decorated trees.
Let $\DDD$ be a combinatorial class.
We define a $\DDD$-decorated locally and upwards finite pointed tree
as a tree $T^\bullet$ in $\lufT$, endowed with a decoration function
$\dec: \Vint({T}^\bullet) \to \DDD$, such that, for each $v$ in $\Vint(T^{\bullet})$, the out-degree of $v$ is exactly $\size(\dec(v))$.
We denote such a tree with the pair $(T^\bullet,\dec)$ and the space of such trees as $\gls*{lufd_pointed_plane_trees}$. 
As above, a decorated tree with a distinguished leaf can be identified with an element of this set.

Given a $\DDD$-decorated pointed tree $(T^{\bullet},\lambda_{T^{\bullet}})\in\lufTD$, we denote by $f_h^{\bullet}(T^{\bullet},\lambda_{T^{\bullet}})$ the $\DDD$-decorated pointed tree $(f_h^{\bullet}(T^{\bullet}),f_h^{\bullet}(\lambda_{T^{\bullet}})),$ where $f_h^{\bullet}(T^{\bullet})$ is the pointed fringe subtree rooted at $u_h$ with distinguished leaf $u_0$ (if $u_h$ is not well-defined because $u_h\notin T^\bullet$, we set $f_h^{\bullet}(T^\bullet)=f_m^{\bullet}(T^\bullet),$ where $m$ is the maximal index such that $u_m\in T^\bullet$) 
and $f_h^{\bullet}(\lambda_{T^{\bullet}})$ is $\lambda_{T^{\bullet}}$ restricted to the domain ${\Vint(f_h^{\bullet}(T^{\bullet}))}$.
We note that, for any given $h$, the image set $f_h^{\bullet}(\lufTD)$ is countable.

We endow the set $\lufTD$ with the \emph{local distance} $\gls*{lufd_dist}$ defined for all $(T^{\bullet}_1,\lambda_{T^{\bullet}_1}),(T^{\bullet}_2,\lambda_{T^{\bullet}_2})\in\lufTD$ by
\begin{equation}
\label{eq:local_dist_for_trees}
d_t\big((T^{\bullet}_1,\lambda_{T^{\bullet}_1}),(T^{\bullet}_2,\lambda_{T^{\bullet}_2})\big)=2^{-\sup\{h\geq0\;:\;f^\bullet_h(T^{\bullet}_1,\lambda_{T^{\bullet}_1})=f^\bullet_h(T^{\bullet}_2,\lambda_{T^{\bullet}_2})\}},
\end{equation}
with the classical conventions that $\sup\emptyset=0,$ $\sup\Z_{>0}=+\infty$ and $2^{-\infty}=0$.

\begin{rem}
	The distance defined in \cref{eq:local_dist_for_trees} can be trivially restricted also to the space of non-decorated pointed trees. We point out that this distance is not equivalent to the distance considered in \cite[Section 6.3.1]{stufler2016limits} for the space of non-decorated pointed trees. For instance, if $S_n$, $1\leq n\leq \infty$ is a star where the root has out-degree $n$ and its children all have out-degree 0, then the sequence $(S_n^{\bullet})_{n\geq 1}$, pointed at the root, does not converge for our metric (and has no convergent subsequences). This implies that our space is not compact. On the contrary, the space of pointed trees endowed with the distance defined by Stufler in \cite{stufler2016limits} is compact. 
	
	We also note (without proof since we do not need this result) that in the subspace of locally finite pointed trees the two distances are topologically equivalent. A proof of this result would be a simple adaptation of \cite[Lemma 6.2]{janson2012simply}.
\end{rem}

\begin{prop}
	The space $(\lufTD,d_t)$ is a Polish space.
\end{prop}
\begin{proof}
	The separability is trivial since $\uplus_{h \ge 1} f_h^{\bullet}(\lufTD)$ is a countable dense set.
	The completeness follows from the fact that the space $\lufTD$ is a closed subspace of a countable product of discrete sets (which is complete) via the map $(T^{\bullet},\lambda_{T^{\bullet}})\to\big(f_h^\bullet(T^{\bullet},\lambda_{T^{\bullet}})\big)_{h\geq 1}.$
\end{proof}

We end this section defining two versions 
of the local convergence (similar to those defined for permutations) for random decorated trees 
with a uniform random distinguished leaf.
In both definition,
$(\bm{T}_n,\bm{\lambda}_n)_{n\in\mathbb{N}}$ is a sequence of random (finite) $\DDD$-decorated trees
and $\bm \ell_n$ is a uniform random leaf of $(\bm{T}_n,\bm{\lambda}_n)$.

\begin{defn}[Annealed Benjamini--Schramm convergence]
	We say that $(\bm{T}_n,\bm{\lambda}_n)_{n\in\mathbb{N}}$ converges in the annealed B--S sense to a random variable $(\bm{T}^{\bullet}_\infty,\bm{\lambda}_\infty)$
	with values in $\lufTD$ if the sequence of random pointed $\DDD$-decorated trees $((\bm{T}^{\bullet}_n,\bm{\lambda}_n),\bm \ell_n)_{n\in\mathbb{N}}$ converges in distribution to $(\bm{T}^{\bullet}_\infty,\bm{\lambda}_\infty)$
	with respect to the local distance defined in \cref{eq:local_dist_for_trees}.
\end{defn}

\begin{defn}[Quenched Benjamini--Schramm convergence]
	We say that $(\bm{T}_n,\bm{\lambda}_n)_{n\in\mathbb{N}}$ converges in the quenched B--S sense to a random measure $\bm{\mu}^{\bullet}_{\infty}$
	on $\lufTD$ if the sequence of conditional probability distributions  $\mathcal{L}aw\big(((\bm{T}_n,\bm{\lambda}_n),\bm \ell_n)\big|(\bm{T}_n,\bm{\lambda}_n)\big)_{n\in\mathbb{N}}$ converges in distribution to $\bm{\mu}^{\bullet}_\infty$
	with respect to the weak topology induced by the local distance defined in \cref{eq:local_dist_for_trees}.
\end{defn}
Again, the quenched version is stronger than the annealed one.

\begin{rem}
	It would also be natural, and closer to the usual notion of B--S convergence
	in the literature,
	to distinguish a uniform random vertex $\bm v_n$
	rather than a uniform random leaf $\bm \ell_n$ as above.
	The leaf version is however what we need here for our application to permutations.
\end{rem}

\subsection{Local convergence around a uniform leaf for random packed trees conditioned on the number of leaves}
\label{sec:limiting_local_trees}

We now fix a substitution-closed class $\mathcal C$ and we work under the following assumption on $\mathcal{C}$.
\begin{ass}
	\label{ass:type1}
	Consider the random variable $\xi$ defined\footnote{Recall that this definition is uniquely determined by the choice of the class $\mathcal{C}$.} in \cref{eq:offspring_distribution_packed_tree}.
	We assume that $$\E[\xi]=1.$$
\end{ass}
See also \cref{prop: offspring_distr_charact} page \pageref{prop: offspring_distr_charact} for an explicit characterization of this assumption. We also recall that $\mathfrak S$ denotes the set of simple permutations in $\mathcal{C}$.

\medskip

We begin by constructing the limiting random pointed packed tree $\bm{P}^\bullet_\infty=(\bm{T}^\bullet_\infty,\bm{\lambda}_\infty)$. 
This tree will be the limit of the sequence of uniform packed trees $(\bm{T}_n,\bm{\lambda}_{\bm{T}_n})$ considered in \cref{lem:unif_packed_tree} pointed at a random leaf.

We recall that $\bm{T}$ denotes a $\xi$-Galton--Watson tree.
Additionally, we recall that the \emph{size-biased} version of $\xi$, denoted $\hat{\xi}$, is the random variable with distribution $\P(\hat{\xi} = i) = i\cdot  \P(\xi = i)$.

We define the random tree $\bm{T}^\bullet_\infty$ in the space $\lufT$ as follows. Let $u_0$ be the distinguished leaf. For each $i\geq 1$, we let $u_i$ receive offspring according to an independent copy of $\hat{\xi}$. 
The vertex $u_{i-1}$ gets identified with a child of $u_i$ chosen uniformly at random. The rest of the children of $u_i$ become roots of independent copies of the Galton--Watson tree $\bm{T}$. 

Conditionally on $\bm{T}^\bullet_\infty,$ the random decoration $\bm{\lambda}_{\infty}(v)$ of each internal vertex $v$ of $\bm{T}^\bullet_\infty$ gets drawn uniformly at random among all $d^+_{\bm{T}^\bullet_\infty}(v)$-sized decorations in $\widehat{\GGG(\mathfrak{S})}$ independently of all the other decorations ($\widehat{\GGG(\mathfrak{S})}$ was introduced after \cref{def:S_plus_decoration}). This construction yields a random infinite locally and upwards finite pointed packed tree.

We refer to the sequence of (decorated) vertices $(u_i)_{i \geq 0}$ as the \emph{infinite spine} of $\bm{P}^\bullet_\infty=(\bm{T}^\bullet_\infty,\bm{\lambda}_\infty)$.

\bigskip

To simplify notation, we denote the space $\mathfrak{T}^{\bullet,\text{\tiny luf}}_{\widehat{\GGG(\mathfrak{S})}}$ of (possibly infinite) locally and upwards finite pointed packed trees as $\lufPT$.

\begin{prop}
	\label{prop:quenched_conv_tree}
	Let $\bm{P}_n=(\bm{T}_n,\bm{\lambda}_{\bm{T}_n})$ be the random packed tree considered in \cref{lem:unif_packed_tree} and  $\bm{P}^\bullet_\infty=(\bm{T}^\bullet_\infty,\bm{\lambda}_{\infty})$ be the random pointed packed tree constructed above. It holds that
	\begin{equation}
	\label{eq:random_meas_conv}
	\mathcal{L}aw\big((\bm{P}_n,\bm{\ell}_n)|\bm{P}_n\big)\stackrel{P}{\longrightarrow}\mathcal{L}aw(\bm{P}^\bullet_\infty),
	\end{equation}
	where $\bm{\ell}_n$ is a uniform leaf of $\bm{P}_n$ chosen independently of $\bm{P}_n.$
	
	In other words, $\bm{P}_n$ converges in the quenched B--S sense 
	to the deterministic measure $\mathcal{L}aw(\bm{P}^\bullet_\infty)$
	and  therefore in the annealed B--S sense to the random tree $\bm{P}^\bullet_\infty.$  
\end{prop}

Note that the $\mathcal{L}aw(\bm{P}^\bullet_\infty)$ is a measure on $\lufPT$.
Since the limiting object in quenched B--S convergence is in general a random measure on $\lufPT$,
it should be interpreted as a constant random variable, 
a.s.\ equal to the measure $\mathcal{L}aw(\bm{P}^\bullet_\infty)$.

\begin{proof}[Proof of \cref{prop:quenched_conv_tree}]
	The sequence $\mathcal{L}aw\big((\bm{P}_n,\bm{\ell}_n)|\bm{P}_n\big)_{n\in\Z_{>0}}$ is a sequence of random probability measures on the Polish space $(\lufPT,d_t).$ 
	The set of closed and open balls
	\begin{equation}
	\label{convergece_determ}
	\mathcal{B}=\Big\{B\big((T^\bullet,\lambda_{T^\bullet}),2^{-h}\big):h\in\Z_{>0},(T^\bullet,\lambda_{T^\bullet})\in\lufPT\Big\}
	\end{equation}
	is a convergence-determining class for the space $(\lufPT,d_t)$, 
	that is, for every probability measure $\mu$ and every sequence of probability measures $(\mu_n)_{n\in\Z_{>0}}$ on $\lufPT$, 
	the convergence $\mu_n(B)\to\mu(B)$ for all $B\in\mathcal{B}$ implies $\mu_n\to\mu$ w.r.t.\ the weak-topology. This is a trivial consequence of the monotone class theorem and the fact that the intersection of two balls in $\lufPT$ is either empty or one of them.
	
	Therefore, using \cite[Theorem 4.11]{kallenberg2017random}, the convergence in \cref{eq:random_meas_conv} is equivalent
	to the following convergence, for all $k\in\Z_{>0}$ and for all vectors of balls $(B_i)_{1\leq i\leq k}\in\mathcal{B}^k$:
	\begin{equation}
	\Big(\mathcal{L}aw\big((\bm{P}_n,\bm{\ell}_n)|\bm{P}_n\big)(B_i)\Big)_{1\leq i\leq k}\stackrel{d}{\longrightarrow}\Big(\mathcal{L}aw(\bm{P}^\bullet_\infty)(B_i)\Big)_{1\leq i\leq k}.
	\end{equation}
	Since the limiting vector in the above equation is deterministic, the above convergence in distribution is equivalent to the convergence in probability. Finally, standard properties of the convergence in probability imply that
	it is enough to show the component-wise convergence, that is, for all $B\in\mathcal{B}$, 
	\begin{equation}
	\label{eq:goaloftheproof1}
	\mathcal{L}aw\big((\bm{P}_n,\bm{\ell}_n)|\bm{P}_n\big)(B)\stackrel{P}{\longrightarrow}\mathcal{L}aw(\bm{P}^\bullet_\infty)(B).
	\end{equation}
	
	Fix a ball $B=B\big((T^\bullet,\lambda_{T^\bullet}),2^{-h}\big)\in\mathcal{B}$ and note that \cref{eq:goaloftheproof1} 
	(which we need to prove) rewrites as
	\begin{equation}
	\label{eq:goal1oftheproof}
	\P\big(f_h^{\bullet}(\bm{P}_n,\bm{\ell}_n)=f_h^{\bullet}(T^\bullet,\lambda_{T^\bullet})\big|\bm{P}_n\big)\stackrel{P}{\longrightarrow}\P\big(f^\bullet_h(\bm{P}^\bullet_\infty)=f^\bullet_h(T^\bullet,\lambda_{T^\bullet})\big).
	\end{equation}
	Note that the left-hand side is a function of $\bm{P}_n$, and hence, a random variable; the right-hand side is a number.
	W.l.o.g. we can assume that $f_h^{\bullet}(T^\bullet,\lambda_{T^\bullet})=(T^\bullet,\lambda_{T^\bullet})$. Denoting $\mathfrak{L}(\bm{P}_n)$ the set of leaves of $\bm{P}_n$,
	the left-hand side writes
	\begin{equation}
	\begin{split}
	\P\big(f_h^{\bullet}(\bm{P}_n,\bm{\ell}_n)=(T^\bullet,\lambda_{T^\bullet})\big|\bm{P}_n\big)&=\frac{\big|\big\{\ell\in \mathfrak{L}(\bm{P}_n):f_h^{\bullet}(\bm{P}_n,\ell)=(T^\bullet,\lambda_{T^\bullet})\big\}\big|}{n}\\
	&=\frac{1}{n}\sum_{\ell\in \mathfrak{L}(\bm{P}_n)}\mathds{1}_{\{f_h^{\bullet}(\bm{P}_n,\ell)=(T^\bullet,\lambda_{T^\bullet})\}}\\
	&=\frac{1}{n}\sum_{\ell\in \mathfrak{L}(\bm{T}_n)}\mathds{1}_{\{f_h^{\bullet}(\bm{T}_n,\ell)=T^\bullet\}}\mathds{1}_{\{\bm{\lambda}_{f_h^{\bullet}(\bm{T}_n,\ell)}=\lambda_{T^\bullet}\}}.
	\end{split}
	\end{equation}
	For a vertex $v$ of $\bm{T}_n$, we denote by $f(\bm{T}_n,v)$ the fringe subtree rooted at $v$
	and by $f(\bm{\lambda}_{(\bm{T}_n,v)})$ the map $\lambda_{|_{\Vint(f(\bm{T}_n,v))}}$.
	Let also $T$ be the unpointed version of $T^{\bullet}$.
	Note that a leaf $\ell\in \mathfrak{L}(\bm{T}_n)$ satisfies $f_h^{\bullet}(\bm{T}_n,\ell)=T^\bullet$
	if and only if its $h$-th ancestor $v$ satisfies $f(\bm{T}_n,v)=T$.
	Additionally, to any $v$ with $f(\bm{T}_n,v)=T$ corresponds exactly one leaf $\ell$ with $f_h^{\bullet}(\bm{T}_n,\ell)=T^\bullet$
	(which is determined by the pointing).
	Therefore we can rewrite the last term of the above equation as
	\begin{equation}
	\frac{1}{n}\sum_{v\in \bm{T}_n}\mathds{1}_{\{f(\bm{T}_n,v)=T\}}\mathds{1}_{\{f(\bm{\lambda}_{(\bm{T}_n,v)})=\lambda_{T}\}}.
	\end{equation}
	By \cite[Remark 1.9]{stufler2019offspring}, we have that 
	$$\frac{1}{n}\sum_{v\in \bm{T}_n}\mathds{1}_{\{f(\bm{T}_n,v)=T\}}\stackrel{P}{\longrightarrow}\P\big(f^\bullet_h(\bm{T}^\bullet_\infty)=f^\bullet_h(T^\bullet)\big).$$
	
	Note that all fringe subtrees of $\bm{T}_n$ that are equal to $T$ are necessarily disjoint and that, conditioning on $f(\bm{T}_n,v)=T,$ then $f(\bm{\lambda}_{(\bm{T}_n,v)})=\lambda_{T}$ with some probability $p,$ independently from the rest 
	(specifically $p = \prod_{u \in T} q_{d^+_T(u)}^{-1}$). Therefore, we can conclude using Chernoff concentration bounds that 
	\begin{equation}
	\frac{1}{n}\sum_{v\in \bm{T}_n}\mathds{1}_{\{f(\bm{T}_n,v)=T\}}\mathds{1}_{\{f(\bm{\lambda}_{(\bm{T}_n,v)})=\lambda_{T}\}}
	\stackrel{P}{\longrightarrow}p\cdot\P\big(f^\bullet_h(\bm{T}^\bullet_\infty)=f^\bullet_h(T^\bullet)\big)=\P\big(f^\bullet_h(\bm{P}^\bullet_\infty)=f^\bullet_h(T^\bullet,\lambda_{T^\bullet})\big),
	\end{equation}
	where the last equality follows from the construction of the map $\bm{\lambda}_{\infty}.$   
\end{proof}
\subsection{Continuous extension of the bijection between packed trees and  permutations to infinite objects}
\label{sec:continuity_bij}

In this section we consider a substitution-closed class $\mathcal{C}$ different from the class of separable permutations. 
The latter case is considered separately in \cite{borga2020decorated} and it will be not considered in this manuscript\footnote{The issue is that every packed tree obtained from a separable permutation contains no $\mathfrak S$-gadgets (because $\mathfrak S=\emptyset$ for separable permutations) and these decorations play a fundamental role in our constructions, as we shall see below. For all these reasons, we needed an \emph{ad hoc} construction for separable permutations.}.
We recall that $\DT:=\Pack\circ\CanTree$ is the bijection presented in \cref{le:bij_perm_tree} page \pageref{le:bij_perm_tree} between $\oplus$-indecomposable permutations of $\mathcal{C}$ and finite packed trees. 

The goal of this section is to extend the bijection $\DT^{-1}$ as a function $\RP$ (for \emph{rooted permutations}) from the metric space of (possibly infinite) locally and upwards finite pointed packed trees $(\lufPT,d_t)$ to the metric space of (possibly infinite) rooted permutations $(\tilde{\SG}_{\bullet},d)$.

First, we need to deal with the introduction of a root in permutations (resp.\ a pointed leaf in trees) on finite objects. 
This is very simple, and we extend $\DT^{-1}$ as a function $\RP$ 
from finite pointed packed trees to finite rooted permutations as follows. 
We recall (see \cref{rk:Leaves_Elements} page \pageref{rk:Leaves_Elements})
that the $i$-th leaf $\ell$ of a packed tree $P=\DT(\nu)$ corresponds to
the $i$-th element of the permutation $\nu$.
Therefore the following definition is natural:
\begin{equation}
\label{eq:G_extension_1}
\RP(\PackedTree,\ell):=(\DT^{-1}(\PackedTree),i).
\end{equation} 
Given an infinite pointed packed tree $\PackedTree^{\bullet}$ with infinitely many $\mathfrak{S}$-gadget decorations on its infinite spine, we consider the sequence of pointed subtrees 
$$\big(f^{\bullet}_{s(h)}(\PackedTree^{\bullet})\big)_{h\in\mathbb{N}}$$
consisting of all restrictions for $s(h)\in\mathbb{N}$ such that $f^{\bullet}_{s(h)}(\PackedTree^{\bullet})$ has root decorated by an $\mathfrak{S}$-gadget.

\begin{lem}[{\cite[Lemma 6.14]{borga2020decorated}}]
	\label{lem:limexistence}
	Let $\PackedTree^{\bullet}$ be an infinite pointed packed tree with infinitely many $\mathfrak{S}$-gadget decorations on its infinite spine.
	Then the (deterministic) sequence of rooted permutations $\big(\RP(f^{\bullet}_{s(h)}(\PackedTree^{\bullet}))\big)_{h\in\mathbb{N}}$ is locally convergent as $h$ tends to $+\infty$.
\end{lem}

This lemma allows to define, for an infinite pointed packed tree $\PackedTree^{\bullet}$ having infinitely many $\mathfrak{S}$-gadget decorations on its infinite spine,
\begin{equation}
\label{eq:G_extension_2}
\RP(\PackedTree^{\bullet}):=\lim_{h\to\infty}\RP(f^{\bullet}_{s(h)}(\PackedTree^{\bullet})).
\end{equation}

We now recover some results on the continuity of the function $\RP$ with respect to the local topologies. Note that $\RP$ is defined only for finite pointed packed trees and infinite pointed packed tree with infinitely many $\mathfrak{S}$-gadget decorations on the infinite spine. This will not be an issue: indeed, we will use the map on a sequence of random pointed packed trees that converges to $\bm{P}^\bullet_\infty$ and this limiting pointed packed tree has a.s.\ infinitely many $\mathfrak{S}$-gadget decorations on the infinite spine.
We set
\begin{equation}
\begin{split}
C_{\RP}:=\big\{\PackedTree^{\bullet}\in&\lufPT:\forall k>0,\;\exists h(k)>0\text{ s.t. }f^{\bullet}_{h(k)}(\PackedTree^{\bullet})\text{ contains at least } k \text{ leaves before and}\\
&\text{$k$ leaves after the distinguished leaf, and has a root decorated with an $\mathfrak{S}$-gadget}\big\},
\end{split}
\end{equation}
where we say that a leaf $\ell_1$ is before (resp.\ after)
a leaf $\ell_2$ if the pre-order label of $\ell_1$ is smaller (resp.\ greater) than
the pre-order label of $\ell_2$.

\begin{prop}[{\cite[Proposition 6.17]{borga2020decorated}}]
	\label{prop:continuity_RP}
	$\RP:(\lufPT,d_t)\to(\tilde{\SG}_{\bullet},d)$ is continuous on $C_{\RP}$.
\end{prop}

We state a final preliminary result for the proof of \cref{thm:local_intro} in the non-separable case.
\begin{prop}[{\cite[Proposition 6.18]{borga2020decorated}}]
	\label{prop:prob_discontinuity_RP}
	We have $\P(\bm{P}^\bullet_\infty \in C_{\RP})=1$.
\end{prop}

\subsection{Local limits of uniform permutations in substitution-closed classes} \label{sect:local_lim_sub}

We now prove a quenched B--S convergence result
for uniform random permutations in a proper substitution-closed class $\mathcal C$ different from the one of separable permutations, so that the set $\mathfrak{S}$ of simple permutations in $\cC$ is non-empty.
As we shall see at the end of the section, this implies our result stated in \cref{thm:local_intro}. Recall \cref{prop: offspring_distr_charact} page \pageref{prop: offspring_distr_charact}.

\begin{thm}
	\label{thm:quenched_BS_cv}
	For all $n\in\Z_{>0}$, let $\bm{\sigma}_n$ be a uniform permutation of size $n$ in a proper substitution-closed class $\mathcal{C}$ different from the one of separable permutations. If 
	\begin{align}
	\cS'(\rho_\cS) \ge \frac{2}{(1 +\rho_\cS)^2} -1,
	\end{align}
	then
	\begin{align}
	\bm{\sigma}_n\stackrel{qBS}{\longrightarrow}\mathcal{L}aw\big(\RP(\bm{P}^\bullet_\infty)\big)\quad\text{and}\quad \bm{\sigma}_n\stackrel{aBS}{\longrightarrow}\RP(\bm{P}^\bullet_\infty).
	\label{eq:local_lim_nonsep}
	\end{align}
\end{thm}

Like after \cref{prop:quenched_conv_tree}, we want to emphasize the nature of the limiting objects above. 
The limit $\mathcal{L}aw\big(\RP(\bm{P}^\bullet_\infty)\big)$  is a measure on $\Sri$.
Since the limiting object for the quenched B--S convergence is in general a random measure on $\Sri$,
it should be interpreted as a constant random variable, a.s.\
equal to the measure $\mathcal{L}aw\big(\RP(\bm{P}^\bullet_\infty)\big)$.
\begin{proof}
	We only need to prove the quenched convergence statement,
	the annealed versions being a simple consequence of the quenched one.
	Moreover, thanks to \cref{prop:giant_comp_perm},
	it is sufficient to prove the statement for a uniform $\oplus$-indecomposable
	permutation $\bm\sigma_n$.

	Let $\bm{P}_n=(\bm{T}_n,\bm{\lambda}_{\bm{T}_n})$ be the random packed tree considered in \cref{lem:unif_packed_tree} page \pageref{lem:unif_packed_tree}. Consider a uniform random leaf $\bm{\ell}_{n}$ in $\bm{P}_{n}$
	and a uniform random element $\bm i_{n}$ in $\bm{\sigma}_n$.
	We have the following equality in distribution:
	\begin{equation}
	\big(\bm{\sigma}_n,\bm i_n\big) \stackrel{d}{=}\RP(\bm{P}_n,\bm{\ell}_n).
	\label{eq:randomNonPlus_Tree}
	\end{equation}
	We analyse the right-hand side conditionally on $\bm{P}_{n}$.
	By \cref{prop:quenched_conv_tree}, we know that
	\begin{equation}
	\mathcal{L}aw\big((\bm{P}_n,\bm{\ell}_n)|\bm{P}_n\big)\stackrel{P}{\longrightarrow}\mathcal{L}aw(\bm{P}^\bullet_\infty).
	\end{equation}
	Moreover, by Propositions \ref{prop:continuity_RP} and \ref{prop:prob_discontinuity_RP}, $\RP$ is almost surely continuous at $\bm{P}^\bullet_\infty$. 
	Therefore, using a combination of the results stated in \cite[Theorem 4.11, Lemma 4.12]{kallenberg2017random}\footnote{\label{footnote:KAL}The specific result that we need is a generalization of the \emph{mapping theorem} for random measures: Let $(\bm{\mu}_n)_{n\in\Z_{>0}}$ be a sequence of random measures on a space $E$ that converges in distribution to a random measure $\bm{\mu}$ on $E.$ Let $F$ be a function from $E$ to a second space $H$ such that the set $D_F$ of discontinuity points of $F$ has measure $\bm{\mu}(D_F)=0$ a.s.. Then the sequence of pushforward random measures $(\bm{\mu}_n\circ F^{-1})_{n\in\Z_{>0}}$ converges in distribution to the pushforward random measure $\bm{\mu}\circ F^{-1}.$},
	\begin{equation}
	\mathcal{L}aw\big(\RP(\bm{P}_{n},\bm{\ell}_{n})|\bm{P}_{n}\big)\stackrel{P}{\longrightarrow}\mathcal{L}aw\big(\RP(\bm{P}^\bullet_\infty)\big).
	\end{equation}
	Note that the result described in footnote~\ref{footnote:KAL} gives convergence \emph{in distribution}; 
	the limit being a deterministic measure, convergence in probability follows. 
	Comparing with \cref{eq:randomNonPlus_Tree}, we have that
	\[\mathcal{L}aw\big( (\bm{\sigma}_n,\bm i_n) | \bm{\sigma}_n \big) 
	\stackrel{P}{\longrightarrow}\mathcal{L}aw\big(\RP(\bm{P}^\bullet_\infty)\big),\]
	which is the quenched convergence in \cref{eq:local_lim_nonsep}.	
\end{proof}

\begin{proof}
	[Proof of \cref{thm:local_intro}]
	We consider only the case that $\mathcal C$ is not the class of separable permutations (for the proof in the latter case see \cite[Section 6.6]{borga2020decorated}). 
	With the assumption of \cref{thm:local_intro}, we just proved (\cref{thm:quenched_BS_cv})
	that a uniform permutation $\bm \sigma_n$ in $\mathcal C$
	converges in the quenched B--S sense to some deterministic measure $\mathcal{L}aw\big(\bm \sigma_\infty)$.
	The quenched B--S convergence imply the (joint) convergence
	of the random variables $\widetilde{\cocc}(\pi,\bm{\sigma}_n)$ to some random variables $\bm\Lambda_\pi$, for all $\pi\in\mathcal S$.
	Additionally, since the quenched B--S limit is a deterministic measure,
	the random variable $\bm\Lambda_\pi$ are deterministic as well (\cref{detstrongbsconditions} page \pageref{detstrongbsconditions}),
	that is, they are numbers $\gamma_{\pi,\mathcal C}$ in $[0,1]$.
	This concludes the proof.
\end{proof}
\begin{rem}
	\label{rk:gammas}
	Concretely $\gamma_{\pi,\mathcal C}$ is the probability 
	that the restriction of the random order $\RP(\bm{P}^\bullet_\infty)$ 
	on a fixed integer interval of size $|\pi|$
	(e.g. $[0,|\pi|-1]$) is equal to $\pi$ (after the identification between permutations and total order on intervals).
	Computing this number involves a sum over countably many configurations of $\bm{P}^\bullet_\infty$
	and so it is not immediate, even for simple classes $\mathcal C$ and short patterns $\pi$.
\end{rem}

\section{Permutations encoded by generating trees}\label{sect:CLTgentree}

The goal of this section is to prove the central limit theorem stated in \cref{thm:main_thm_CLT} page \pageref{thm:main_thm_CLT}. First, in \cref{sect:CTL_increments} we collect all the results on conditioned random walks that we need for the proof of \cref{thm:main_thm_CLT}.
Then, in \cref{sect:main_thm_proof} we give the proof of this theorem.

\subsection{Results on conditioned random walks}\label{sect:CTL_increments}

We recall that $({\bm{X}}^{\overleftarrow{\bm c}}_{i})_{i\in\Z_{>0}}=({\bm{X}}^{\bm c_{i-1}}_{i})_{i\in\Z_{>0}}$ denotes the random colored walk defined in \cref{eq:definition_walk}. We assume that the distribution of its increments $(\alpha_y)_{y\in\Z_{\leq 1}}$ is centered with finite variance and that ${\bm X}_1=-1$ (i.e.\ $\beta=0$ in \cref{eq:definition_walk}). We also recall that we denote the colored increments of $({\bm{X}}^{\bm c_{i-1}}_{i})_{i\in\Z_{>0}}$ by ${\bm{Y}}^{\bm c_i}_{i}$ in such a way that $\bm{Y}_i={\bm{X}}_{i+1}-{\bm{X}}_{i},$ for all $i\in\Z_{>0}.$

For simplicity we also assume that the period\footnote{Setting $y'=\max\{y\in\Z_{\leq 1}:\alpha_{y}> 0\}$, the \emph{period} of $(\alpha_y)_{y\in\Z_{\leq 1}}$ is the greatest common divisor of the set $\{y'-y:y\in{\Z_{< y'}},\alpha_{y}> 0\}$.} of the distribution $(\alpha_y)_{y\in\Z_{\leq 1}}$ is 1, omitting the minor (and standard) modifications in the general case. For instance, if the period is $h$, when using local limits results (as in \cref{eq:locallimthm}), one has to replace the expression  
\begin{equation}
\P\left({\bm{X}}_{n}=\ell\right)=\frac{1}{\sqrt{2 \pi \sigma^2 n}}\left(e^{-\tfrac{\ell^2}{2n\sigma^2}}+o(1)\right),\quad\text{for all $\ell\in\Z$,}
\end{equation}
with
\begin{equation}
\P\left({\bm{X}}_{n}=\ell\right)=\frac{h}{\sqrt{2 \pi \sigma^2 n}}\left(e^{-\tfrac{\ell^2}{2n\sigma^2}}+o(1)\right),\quad\text{for all $\ell\in h\Z-1$}.
\end{equation}

\medskip

We state our first result on conditioned random walks, which will be helpful in the proof of \cref{thm:main_thm_CLT} in connection with \cref{ass3}.

\begin{prop}
	\label{prop:prop_labels_big}
	Let $c>0$ be a constant. Let $(a_n)_{n\in\Z_{\geq 0}}$ be a sequence of integers such that $a_n<n/2$ for all $n\in\Z_{\geq 0}$ and $a_n\to \infty$. Then, as $n$ tends to infinity,
	\begin{equation}\label{eq:goal_proof}
	\P\left({\bm{X}}_{i}>c \text{ for all } i\in[a_n,n-a_n]\Big|({\bm{X}}_{j})_{j\in[2,n]}\geq 0,{\bm{X}}_{n+1}=0\right)\to 1.
	\end{equation}
\end{prop}

\begin{proof} 
	Note that, using a union bound, we have
	\begin{multline}\label{eq:step1}
	\P\left({\bm{X}}_{i}>c\text{ for all } i\in[a_n,n-a_n]\Big|({\bm{X}}_{j})_{j\in[2,n]}\geq 0,{\bm{X}}_{n+1}=0\right)\\
	\geq 1-\sum_{i\in[a_n,n-a_n]}\P\left({\bm{X}}_{i}\leq c\Big|({\bm{X}}_{j})_{j\in[2,n]}\geq 0,{\bm{X}}_{n+1}=0\right).
	\end{multline}
	Fix now $i\in[a_n,\lfloor n/2\rfloor]$.
	Denoting with $\hat{\bm{X}}$ the time-reversed walk of ${\bm{X}}$ starting at 0 at time 1, we have that
	\begin{multline}
	\P\left({\bm{X}}_{i}\leq c\Big|({\bm{X}}_{j})_{j\in[2,n]}\geq 0,{\bm{X}}_{n+1}=0\right)=
	\sum_{y=0}^c\frac{\P\left({\bm{X}}_{i}=y,{\bm{X}}_{n+1}=0,({\bm{X}}_{j})_{j\in[2,n]}\geq 0\right)}{\P\left({\bm{X}}_{n+1}=0,({\bm{X}}_{j})_{j\in[2,n]}\geq 0\right)}\\
	=\sum_{y=0}^c\frac{\P\left({\bm{X}}_{i}=y,({\bm{X}}_{j})_{j\in[2,i]}\geq 0\right)\P\left(\hat{\bm{X}}_{n+1-i}=y,(\hat{\bm{X}}_{j})_{j\in[n+1-i]}\geq 0\right)}{\P\left({\bm{X}}_{n+1}=0,({\bm{X}}_{j})_{j\in[2,n]}\geq 0\right)}.
	\end{multline}
	Using a local limit theorem for random walks conditioned to stay positive (see \cite[Theorem 3]{MR2449127}) we have that for constants $C_1(y)$, $C_2(y)$, and $C_3$,
	\begin{align}
	&\P\left({\bm{X}}_{i}=y,({\bm{X}}_{j})_{j\in[2,i]}\geq 0\right)\sim C_1(y)\cdot i^{-3/2},\\
	&\P\left(\hat{\bm{X}}_{n+1-i}=y,(\hat{\bm{X}}_{j})_{j\in[n+1-i]}\geq 0\right)\sim C_2(y)\cdot(n-i)^{-3/2},\\
	&\P\left({\bm{X}}_{n+1}=0,({\bm{X}}_{j})_{j\in[2,n]}\geq 0\right)\sim C_3\cdot n^{-3/2}.
	\end{align}
	Since $i\in[a_n,\lfloor n/2\rfloor]$, $a_n\to\infty,$ and $y\in[0,c]\cap \Z$, we obtain that
	\begin{equation}
	\P\left({\bm{X}}_{i}\leq c\Big|({\bm{X}}_{j})_{j\in[2,n]}\geq 0,{\bm{X}}_{n+1}=0\right)=O(i^{-3/2}),
	\end{equation}
	uniformly for all $i\in[a_n,\lfloor n/2\rfloor]$.
	This bound together with \cref{eq:step1} leads to
	\begin{equation}
	\P\left({\bm{X}}_{i}>c\text{ for all } i\in[a_n,\lfloor n/2\rfloor]\Big|({\bm{X}}_{j})_{j\in[2,n]}\geq 0,{\bm{X}}_{n+1}=0\right)\geq 1-O((a_n)^{-1/2}).
	\end{equation}
	Using the same techniques, we get the same bound for $i\in[\lfloor n/2\rfloor,n-a_n]$, ending the proof.
\end{proof}

We now prove the following CLT for occurrences of patterns of consecutive colored increments in random walks conditioned to stay positive.

\begin{prop}\label{lem:lemma3}
	Fix $h\in\Z_{>0}$. Let ${\mathcal A}$ be a (possibly infinite) family of sequences of $h$ colored increments for the walk $({\bm{X}}^{\overleftarrow{\bm c}}_i)_{i\geq 1}$.  Then
	\begin{equation}
	\left(\frac{\sum_{j\in[1,n-h+1]}\mathds{1}_{\left\{({\bm{Y}}^{\bm c_i}_i)_{i\in[j,j+h-1]}\in{\mathcal A}\right\}}-n\cdot \mu_{{\mathcal A}}}{\sqrt n}\Bigg|({\bm{X}}_i)_{i\in[2,n]}\geq 0,{\bm{X}}_{n+1}=0\right)\stackrel{d}{\longrightarrow}\bm{\mathcal{N}}(0,\gamma_{{\mathcal A}}^2),
	\end{equation}
	where $\mu_{{\mathcal A}}=\P\big(({\bm{Y}}^{\bm c_1}_1,\dots,{\bm{Y}}^{\bm c_h}_{h})\in {\mathcal A}\big)$ and $\gamma_{{\mathcal A}}^2=\kappa^2-\frac{\rho^2}{\sigma^2}$ with
	\begin{align}
	&\sigma^2=\Var({\bm{Y}}_1),\\
	&\rho=\E\left[\mathds{1}_{\left\{({\bm{Y}}^{\bm c_i}_{i})_{i\in[h]}\in{\mathcal A}\right\}}\cdot\left({\bm{X}}_{h+1}+1\right)\right],\\
	&\kappa^2=2\nu+\mu_{{\mathcal A}}-(2h-1)\mu_{{\mathcal A}}^2,\quad\text{for}\\
	&\nu=\sum_{s=2}^h\sum_{\substack{
			({y}^c_i)_{i\in[h]},({\ell}^d_i)_{i\in[h]}\in{\mathcal A}\text{ s.t.}\\
			({y}^c_s,\dots,{y}^c_h)=({\ell}^d_1,\dots,{\ell}^d_{h-s+1})}}
	\P\left(({\bm{Y}}^{\bm c}_{i})_{i\in[h+s-1]}=({y}^c_1,\dots,{y}^c_h,{\ell}^d_{h-s+2},\dots,{\ell}^d_h)\right).
	\end{align}
\end{prop}

The proof of the result above is inspired by the proof of \cite[Lemma 7.1]{MR3432572} which is, in turn, based on a method introduced by Le Cam \cite{MR105735} and Holst \cite{MR628875}. Since our conditions are different, we include a complete proof of our result.

In the proof we denote by $\bm O(1)$ an unspecified sequence of random variables $(\bm \varepsilon_n)_{n\in\Z_{>0}}$ that are bounded by a constant (i.e.\ there exists a constant $C>0$ such that $\bm \varepsilon_n\leq C$ a.s.\ for all $n\in\Z_{>0}$).

\begin{proof}[Proof of \cref{lem:lemma3}]
	Recall that $\hat{\bm X}$ denotes the time-reversed walk of $\bm X$ starting at 0 at time 1. Denote by $\hat{\mathcal A}$ the family obtained from ${\mathcal A}$ by reversing the time and changing the sign of each sequence, i.e.\ if $(y^{c_i}_i)_{i\in [h]}\in\mathcal{A}$ then $(-y^{c_{h-i+1}}_{h-i+1})_{i\in [h]}\in\hat{\mathcal{A}}$. Note that the increments $(\hat{\bm{Y}}_{i})_{i\geq 1}$ of the walk $\hat{\bm{X}}$ are supported on $\Z_{\geq -1}$ and so using the cycle lemma (see for instance \cite[Lemma 6.1]{MR2245368}) we can cyclically shift the increments of the walk in order to
	ensure that $(\hat{\bm{X}}_i)_{i\in[n]}\geq 0$. We therefore obtain that
	\begin{multline}
	\left(\sum_{j\in[1,n-h+1]}\mathds{1}_{\left\{(\hat{\bm{Y}}^{\bm c_i}_i)_{i\in[j,j+h-1]}\in\hat{\mathcal A}\right\}}\Bigg|(\hat{\bm{X}}_i)_{i\in[n]}\geq 0,\hat{\bm{X}}_{n+1}=-1\right)\\
	\stackrel{d}{=}\left(\sum_{j\in[1,n-h+1]}\mathds{1}_{\left\{(\hat{\bm{Y}}^{\bm c_i}_i)_{i\in[j,j+h-1]}\in\hat{\mathcal A}\right\}}\Bigg|\hat{\bm{X}}_{n+1}=-1\right)+\bm{O}(1),
	\end{multline}
	where the error term is due to the fact that the cyclic shift can only affect $\bm{O}(1)$ of the consecutive occurrences (because of possible boundary problems). More precisely, the $\bm{O}(1)$ term above is a.s.\ at most $2h$.
	As a consequence,
	\begin{multline}
	\left(\sum_{j\in[1,n-h+1]}\mathds{1}_{\left\{({\bm{Y}}^{\bm c_i}_i)_{i\in[j,j+h-1]}\in{\mathcal A}\right\}}\Bigg|({\bm{X}}_i)_{i\in[2,n]}\geq 0,{\bm{X}}_{n+1}=0\right)\\
	\stackrel{d}{=}\left(\sum_{j\in[1,n-h+1]}\mathds{1}_{\left\{({\bm{Y}}^{\bm c_i}_i)_{i\in[j,j+h-1]}\in{\mathcal A}\right\}}\Bigg|{\bm{X}}_{n+1}=0\right)+\bm{O}(1).
	\end{multline}
	In words, the above equality implies that we can forget the conditional event $\{({\bm{X}}_i)_{i\in[2,n]}\geq 0\}$ in our analysis.
	We set 
	\begin{equation}
	g({\bm{Y}}^{\bm c}_j,\dots, {\bm{Y}}^{\bm c}_{j+h-1})\coloneqq\mathds{1}_{\left\{({\bm{Y}}^{\bm c}_{j+i})_{i\in[0,h-1]}\in{\mathcal A}\right\}}.
	\end{equation}	
	Define the centered sum
	\begin{equation}
	\bm{S}_n\coloneqq\sum_{j=1}^n\left(g({\bm{Y}}^{\bm c}_j,\dots, {\bm{Y}}^{\bm c}_{j+h-1})-\mu_{{\mathcal A}}\right),
	\end{equation}
	with the convention that $g({\bm{Y}}^{\bm c}_j,\dots, {\bm{Y}}^{\bm c}_{j+h-1})=0$ if $j>n-h+1$.
	
	We fix $\alpha$ with $0<\alpha<1$ and a sequence $n'=n'(n)$ with $\frac{n'}{n}\to\alpha$, for instance $n'=\lfloor \alpha n\rfloor$.
	By the central limit theorem for $h$-dependent variables (see \cite{MR26771,MR60175}), applied to the random vectors 
	$$\left(g({\bm{Y}}^{\bm c}_j,\dots, {\bm{Y}}^{\bm c}_{j+h-1})-\mu_{{\mathcal A}},{\bm{Y}}_j\right),$$ we have the following unconditioned result
	\begin{equation}\label{eq:clt_uncond}
	\left(\frac{\bm{S}_{n'}}{\sqrt{n}},\frac{{\bm{X}}_{n'}}{\sqrt n}\right)\xrightarrow{d}\bm{\mathcal{N}}\left(0,\alpha\left(\begin{array}{cc}
	\kappa^2 & \rho\\
	\rho & \sigma^2
	\end{array}\right)\right),
	\end{equation}
	where
	\begin{align}
	&\sigma^2=\Var({\bm{Y}}_1),\\
	&\kappa^2=\Var\left(g({\bm{Y}}^{\bm c}_1,\dots, {\bm{Y}}^{\bm c}_{h})\right)+2\sum_{s=2}^h\Cov\left(g({\bm{Y}}^{\bm c}_1,\dots, {\bm{Y}}^{\bm c}_{h}),g({\bm{Y}}^{\bm c}_s,\dots, {\bm{Y}}^{\bm c}_{s+h-1})\right),\\
	&\rho=\sum_{s=1}^h\Cov\left(g({\bm{Y}}^{\bm c}_1,\dots, {\bm{Y}}^{\bm c}_{h}), {\bm{Y}}_{s}\right)=\Cov\left(g({\bm{Y}}^{\bm c}_1,\dots, {\bm{Y}}^{\bm c}_{h}), {\bm{X}}_{h+1}+1\right).
	\end{align}
	We first compute $\kappa^2$. Note that
	\begin{equation}
	\Var\left(g({\bm{Y}}^{\bm c}_1,\dots, {\bm{Y}}^{\bm c}_{h})\right)=\E\left[\left(\mathds{1}_{\left\{({\bm{Y}}^{\bm c}_{i})_{i\in[h]}\in{\mathcal A}\right\}}\right)^2\right]-\mu_{{\mathcal A}}^2=\mu_{{\mathcal A}}-\mu_{{\mathcal A}}^2.
	\end{equation}
	Now fix $2\leq s\leq h,$
	\begin{equation}
	\Cov\left(g({\bm{Y}}^{\bm c}_1,\dots, {\bm{Y}}^{\bm c}_{h}),g({\bm{Y}}^{\bm c}_s,\dots, {\bm{Y}}^{\bm c}_{s+h-1})\right)=\E\left[\left(\mathds{1}_{\left\{({\bm{Y}}^{\bm c}_{i})_{i\in[h]}\in{\mathcal A}\right\}}\right)\left(\mathds{1}_{\left\{({\bm{Y}}^{\bm c}_{s+i-1})_{i\in[h]}\in{\mathcal A}\right\}}\right)\right]-\mu_{{\mathcal A}}^2.
	\end{equation}
	Noting that, for $({y}^{c_i}_i)_{i\in[h]},({\ell}^{d_i}_i)_{i\in[h]}\in{\mathcal{A}}$, then $\mathds{1}_{\left\{({\bm{Y}}^{\bm c}_{i})_{i\in[h]}=({y}^c_i)_{i\in[h]},({\bm{Y}}^{\bm c}_{s+i-1})_{i\in[h]}=({\ell}^d_i)_{i\in[h]}\right\}}\neq 0$ only if $({y}^c_s,\dots,{y}^c_h)=({\ell}^d_1,\dots,{\ell}^d_{h-s+1})$, we have
	\begin{multline}
	\Cov\left(g({\bm{Y}}^{\bm c}_1,\dots, {\bm{Y}}^{\bm c}_{h}),g({\bm{Y}}^{\bm c}_s,\dots, {\bm{Y}}^{\bm c}_{s+h-1})\right)\\
	=\E\left[\sum_{\substack{
			({y}^c_i)_{i\in[h]}, ({\ell}^d_i)_{i\in[h]}\in{\mathcal A}\text{ s.t.}\\
			({y}^c_s,\dots,{y}^c_h)=({\ell}^d_1,\dots,{\ell}^d_{h-s+1})}}
	\mathds{1}_{\left\{({\bm{Y}}^{\bm c}_{i})_{i\in[h+s-1]}=({y}^c_1,\dots,{y}^c_h,{\ell}^d_{h-s+2},\dots,{\ell}^d_h)\right\}}\right]-\mu_{{\mathcal A}}^2,
	\end{multline}
	and we obtain that $\kappa^2=2\nu+\mu_{{\mathcal{A}}}-(2h-1)\mu_{{\mathcal{A}}}^2$, where $\nu$ is defined in the statement of the theorem.
	
	We finally compute $\rho$. Note that since $\E[{\bm{X}}_{h+1}+1]=0$ (see \cref{eq:definition_walk} and recall that $\beta=0$) then
	\begin{equation}
	\Cov\left(g({\bm{Y}}^{\bm c}_1,\dots, {\bm{Y}}^{\bm c}_{h}), {\bm{X}}_{h+1}+1\right)=\E\left[\mathds{1}_{\left\{({\bm{Y}}^{\bm c}_{i})_{i\in[h]}\in{\mathcal A}\right\}}\cdot\left({\bm{X}}_{h+1}+1\right)\right],
	\end{equation}
	which corresponds to the expression in the statement of the proposition.
	
	We now define for convenience
	\begin{equation}
	\bar{\bm{S}}_n\coloneqq\bm{S}_n-\frac{\rho}{\sigma^2}{\bm{X}}_{n}.
	\end{equation}
	Since we know from \cref{eq:clt_uncond} that
	\begin{equation}
	\left(\frac{\bm{S}_{n'}}{\sqrt{n}},\frac{{\bm{X}}_{n'}}{\sqrt n}\right)\xrightarrow{d}\bm{\mathcal{N}}\left(0,\alpha\left(\begin{array}{cc}
	\kappa^2 & \rho\\
	\rho & \sigma^2
	\end{array}\right)\right),
	\end{equation}
	then with some basic computations we obtain that
	\begin{equation}
	\left(\frac{\bar{\bm{S}}_{n'}}{\sqrt{n}},\frac{{\bm{X}}_{n'}}{\sqrt n}\right)\xrightarrow{d}\bm{\mathcal{N}}\left(0,\alpha\left(\begin{array}{cc}
	\kappa^2-\frac{\rho^2}{\sigma^2} & 0\\
	0 & \sigma^2
	\end{array}\right)\right).
	\end{equation}
	In other words, $\frac{\bar{\bm{S}}_{n'}}{\sqrt{n}}$ and $\frac{{\bm{X}}_{n'}}{\sqrt n}$ are jointly asymptotically normal with independent limits $\bm W=\bm{\mathcal{N}}\left(0,\alpha\cdot\gamma_{{\mathcal A}}^2\right)$, where $\gamma_{{\mathcal A}}^2=\kappa^2-\frac{\rho^2}{\sigma^2}$, and $\bm R=\bm{\mathcal{N}}(0,\alpha\cdot \sigma^2)$.
	
	Next, let $f$ be any bounded continuous function on $\mathbb R$. Then, using a local limit theorem for random walks (see for instance \cite[Theorem VII.1]{MR0388499}) we have that, uniformly for all $\ell\in\Z$,
	\begin{equation}\label{eq:locallimthm}
	\P\left({\bm{X}}_{n}=\ell\right)=\frac{1}{\sqrt{2 \pi \sigma^2 n}}\left(e^{-\tfrac{\ell^2}{2n\sigma^2}}+o(1)\right),
	\end{equation}
	and so 	
	\begin{equation}
	\begin{split}\label{eq:jowefbuiowefbn}
	\E\bigg[f\left(\frac{\bar{\bm{S}}_{n'}}{\sqrt{n}}\right)&\bigg|{\bm{X}}_{n+1}=0\bigg]\\
	&=\frac{\E\left[\sum_jf\left(\frac{\bar{\bm{S}}_{n'}}{\sqrt{n}}\right)\mathds{1}_{\{{\bm{X}}_{n'}=j\}}\mathds{1}_{\{{\bm{X}}_{n+1}-{\bm{X}}_{n'}=-j\}}\right]}{\P\left({\bm{X}}_{n+1}=0\right)}\\
	&=\frac{\sum_j\E\left[f\left(\frac{\bar{\bm{S}}_{n'}}{\sqrt{n}}\right)\mathds{1}_{\{{\bm{X}}_{n'}=j\}}\right]\P\left({\bm{X}}_{n-n'+1}=-j\right)}{\P\left({\bm{X}}_{n+1}=0\right)}\\
	&=\sum_j\E\left[f\left(\frac{\bar{\bm{S}}_{n'}}{\sqrt{n}}\right)\mathds{1}_{\{{\bm{X}}_{n'}=j\}}\right]\sqrt{\frac{n+1}{n-n'+1}}\left(e^{-\tfrac{j^2}{2(n-n')\sigma^2}}+o(1)\right)\\
	&=\sqrt{\frac{n+1}{n-n'+1}}\cdot\E\left[f\left(\frac{\bar{\bm{S}}_{n'}}{\sqrt{n}}\right)e^{-\tfrac{{\bm{X}}_{n'}^2}{2(n-n')\sigma^2}}\right]+o(1)\\
	&\to(1-\alpha)^{-1/2}\cdot\E\left[f(\bm W)e^{-\tfrac{\bm R^2}{2(1-\alpha)\sigma^2}}\right]=\E[f(\bm W)](1-\alpha)^{-1/2}\E\left[e^{-\tfrac{\bm R^2}{2(1-\alpha)\sigma^2}}\right],
	\end{split}
	\end{equation}
	where we used that $\frac{n'}{n}\to\alpha$, $\frac{\bar{\bm{S}}_{n'}}{\sqrt{n}}\xrightarrow{d}\bm W$, $\frac{{\bm{X}}_{n'}}{\sqrt n}\xrightarrow{d}\bm R$, and that $\bm W$ and $\bm R$ are independent. Note that $$(1-\alpha)^{-1/2}\cdot\E\left[e^{-\tfrac{\bm R^2}{2(1-\alpha)\sigma^2}}\right]=1,$$ 
	this easily follows taking $f=1$ in \cref{eq:jowefbuiowefbn} (or from standard computations with Gaussian variables). The last two equations prove that 
	\begin{equation}\label{eq:reachedoneobj}
	\left(\frac{\bar{\bm{S}}_{n'}}{\sqrt{n}}\bigg|{\bm{X}}_{n+1}=0\right)\xrightarrow{d}\bm W=\bm{\mathcal{N}}\left(0,\alpha\cdot \gamma_{{\mathcal A}}^2\right).
	\end{equation}	
	
	We now look at
	\begin{equation}
	\bm{S}_{n}-\bar{\bm{S}}_{n'}
	=\sum_{j=n'+1}^{n}\left(g({\bm{Y}}^{\bm c}_j,\dots, {\bm{Y}}^{\bm c}_{j+h-1})-\mu_{{\mathcal A}}\right)+\frac{\rho}{\sigma^2}{\bm{X}}_{n'}.
	\end{equation}
	Recalling that $\hat{\bm{X}}$ denotes the time-reversed walk of ${\bm{X}}$ starting at 0 at time 1 (and $\hat{\bm Y}^{\bm c}_{j}$ the corresponding reversed steps), we have from the above equation that
	\begin{multline}\label{eq:bfirybfuiwbf0ow}
	\left(\bm{S}_{n}-\bar{\bm{S}}_{n'}\Big|{\bm{X}}_{n+1}=0\right)\\
	\stackrel{d}{=}\left(\sum_{j=1}^{n-n'}\left({g}(\hat{\bm{Y}}^{\bm c}_j,\dots, \hat{\bm Y}^{\bm c}_{j+h-1})-\mu_{{\mathcal A}}\right)+\frac{\rho}{\sigma^2}\hat{\bm{X}}_{n-n'+1}\Bigg|\hat{\bm{X}}_{n+1}=-1\right),
	\end{multline}
	where
	\begin{equation}
	{g}(\hat{\bm{Y}}^{\bm c}_j,\dots, \hat{\bm Y}^{\bm c}_{j+h-1})\coloneqq\mathds{1}_{\left\{(\hat{\bm{Y}}^{\bm c}_{j+i})_{i\in[0,h-1]}\in{\hat{\mathcal A}}\right\}}.
	\end{equation}	
	
	Note that, using \cref{eq:bfirybfuiwbf0ow} and similar arguments to the ones used for proving \cref{eq:reachedoneobj}, we obtain that
	\begin{equation}\label{eq:fiugwebhdndpwemd}
	\left(\frac{\bm{S}_{n}-\bar{\bm{S}}_{n'}}{\sqrt{n}}\bigg|{\bm{X}}_{n+1}=0\right)\xrightarrow{d}\bm{\mathcal{N}}\left(0,(1-\alpha)\gamma_{\hat{\mathcal A}}^2\right),
	\end{equation}
	for some constant $\gamma_{\hat{\mathcal A}}^2$.
	
	Letting $\alpha\to 1$, we obtain that
	\begin{align}
	\bm{\mathcal{N}}\left(0,\alpha\cdot \gamma_{{\mathcal A}}^2\right)\xrightarrow{d}\bm{\mathcal{N}}\left(0, \gamma_{{\mathcal A}}^2\right),\\
	\bm{\mathcal{N}}\left(0,(1-\alpha)\gamma_{\hat{\mathcal A}}^2\right)\xrightarrow{P}0.
	\end{align}
	Therefore from \cref{eq:reachedoneobj,eq:fiugwebhdndpwemd} and using \cite[Theorem 3.1]{billingsley2013convergence} we obtain that
	\begin{equation}
	\left(\frac{\bm{S}_{n}}{\sqrt{n}}\bigg|{\bm{X}}_{n+1}=0\right)\xrightarrow{d}\bm{\mathcal{N}}\left(0,\gamma_{{\mathcal A}}^2\right),
	\end{equation}
	concluding the proof.
\end{proof}

\subsection{A central limit theorem for consecutive patterns }\label{sect:main_thm_proof}

We can now prove our main result.

\begin{proof} [Proof of Theorem \ref{thm:main_thm_CLT}]
	We fix $h\in\Z_{>0}$ and a pattern $\pi\in\mathcal{S}_{h}.$ Let $\bm{\sigma}_n$ be a uniform permutation of size $n$ in $\mathcal C$. From \cref{ass2} (w.l.o.g.\ we can assume $\beta=0$) and \cref{ass3} we have that, conditioning on $\left\{({\bm{X}}_{i})>c(h)\text{ for all } i\in[j+1,j+h]\right\}$,
	\begin{equation}
	\pat_{[j+1,j+h]}(\bm{\sigma}_n)=\Pat\big(({\bm{Y}}_i^{\bm c})_{i\in[j,j+h-1]}\big|({\bm{X}}_i)_{i\in[2,n]}\geq 0,{\bm{X}}_{n+1}=0\big).
	\end{equation}
	Therefore, for any sequence of integers $(a_n)_{n\in\Z_{>0}}$, conditioning on 
	$$\left\{(\bm{X}_{i})>c(h) \text{ for all } i\in[a_n,n-a_n]\right\},$$ 
	we can rewrite $\cocc(\pi,\bm{\sigma}_n)$ as
	\begin{equation}\label{eq:rewriting_coc}
	\left(\sum_{j\in[1,n-h+1]}\mathds{1}_{\left\{({\bm{Y}}_i^{\bm c})_{i\in[j,j+h-1]}\in\Pat^{-1}(\pi)\right\}}\Bigg|({\bm{X}}_i)_{i\in[2,n]}\geq 0,{\bm{X}}_{n+1}=0\right)+O(a_n).
	\end{equation}
	Now, using \cref{lem:lemma3}, we have that
	\begin{equation}
	\left(\frac{\sum_{j\in[1,n-h+1]}\mathds{1}_{\left\{({\bm{Y}}^{\bm c}_i)_{i\in[j,j+h-1]}\in\Pat^{-1}(\pi)\right\}}-n\cdot \mu_{\pi}}{\sqrt n}\Bigg|({\bm{X}}_i)_{i\in[2,n]}\geq 0,{\bm{X}}_{n+1}=0\right)\stackrel{d}{\longrightarrow}\bm{\mathcal{N}}(0,\gamma_{\pi}^2),
	\end{equation}
	where the expressions of $\mu_{\pi}$ and $\gamma_{\pi}^2$ given in the statement of \cref{thm:main_thm_CLT} follow from the statement of \cref{lem:lemma3}.
	Since thanks to Proposition \ref{prop:prop_labels_big}, we can choose $a_n=o(\sqrt n)$ and such that
	\begin{equation}
	\P\left({\bm{X}}_{i}>c\text{ for all } i\in[a_n,n-a_n]\Big|({\bm{X}}_{j})_{j\in[2,n]}\geq 0,{\bm{X}}_{n+1}=0\right)\to 1,
	\end{equation}
	then we can conclude that
	\begin{equation}
	\frac{\cocc(\pi,\bm{\sigma}_n)-n\cdot \mu_{\pi}}{\sqrt n}\stackrel{d}{\longrightarrow}\bm{\mathcal{N}}(0,\gamma_{\pi}^2).
	\end{equation}
	This ends the proof.
\end{proof}

\section{Baxter permutations and related objects}
\label{sec:local}

The goal of this section is to prove \cref{thm:local} page \pageref{thm:local}, a joint quenched local convergence result for the four families of objects put in correspondence by the mappings in the commutative diagram in \cref{eq:comm_diagram} page \pageref{eq:comm_diagram}, that is, Baxter permutations $(\mathcal P)$, bipolar orientations $(\mathcal O)$, tandem walks $(\mathcal W)$, and coalescent-walk processes $(\mathscr{C})$.

As in the previous section, the consecutive patterns of a Baxter permutation are encoded by the local increments of the corresponding walk. Nevertheless, studying these local increments is much more
difficult. Indeed, in this section, we must deal with two-dimensional random walks in cones instead of a one-dimensional random walks conditioned to stay positive.

\medskip

We start with some extra properties on the Kenyon-Miller-Sheffield-Wilson bijection.

\subsection{The Kenyon-Miller-Sheffield-Wilson bijection}
\label{sect:KMSW}

Recall that the definition of the mapping $\bow : \mathcal O \to \mathcal W$ was given in \cref{defn:KMSW} page \pageref{defn:KMSW}.
Recall also that the set of tandem walks $\mathcal{W}$ is defined as the set of two-dimensional walks in the non-negative quadrant, starting at $(0,h)$ and ending at $(k,0)$ for some $h\geq 0, k\geq 0$, with increments in 
\begin{equation}
	\label{eq:admis_steps2}
	\Steps = \{(+1,-1)\} \cup \{(-i,j), i\in \Z_{\geq 0}, j\in \Z_{\geq 0}\}.
\end{equation}

An equivalent way of understanding $\bow$ is as follows.

\begin{rem}\label{rem:height_process}
	Let $m\in \mathcal O$ and $\bow(m) = ((X_t,Y_t))_{1\leq t \leq n}$. The walk $(0,X_1+1,\ldots,X_{|m|}+1)$ is the height process\footnote{We recall that the height process of a tree $T$ is the sequence of integers obtained by recording for each visited vertex (following the exploration of $T$) its distance to the root.} of the tree $T(m)$. The walk $(0,Y_{|m|}+1, Y_{|m|-1}+1,\ldots Y_1+1)$ is the height process of the tree $T(m^{**})$.
\end{rem}

We explain some properties of the bijection $\bow$. Let $m\in \mathcal O$ and $\bow(m) = ((X_t,Y_t))_{1\leq t \leq n}$. Suppose that the left outer face of $m$ has $h+1$ edges and the right outer face of $m$ has $k+1$ edges, for some $h,k\geq 0$. Then the walk $(X_t,Y_t)_{t\in[|m|]}$ starts at $(0,h)$, ends at $(k,0)$, and stays by definition in the non-negative quadrant $\mathbb{Z}_{\geq 0}^2$.

We give an interpretation to the increments of the walk, i.e.\ the values of $(X_{t+1},Y_{t+1})-(X_t,Y_t)$. 
We say that two edges of a tree are consecutive if one is the parent of the other.
The interface path of the map $m$, defined above \cref{defn:KMSW} page \pageref{def:int_path}, has two different behaviors: 
\begin{itemize}
	\item either it is following two edges $e_t$ and $e_{t+1}$ that are consecutive, both in $T(m)$ and $T(m^{**})$, in which case the increment is $(+1,-1)$;
	\item or it is first following $e_t$, then it is traversing a face of $m$, and finally is following $e_{t+1}$, in which case the increment is $(-i,+j)$ with $i,j\in\Z_{\geq 0}$, and the traversed face has left degree $i+1$ and right degree $j+1$.
\end{itemize}

\begin{exmp}\label{exemp:walk}
	Consider the map $m$ in Fig.~\ref{fig:bip_orient _with_trees2}. 
	The corresponding walk $\bow(m)$ plotted on the right-hand side of \cref{fig:bip_orient _with_trees2} is:
	\begin{equation}
		\begin{split}
			&W_{1}=(0,2),W_{2}=(0,3),W_{3}=(0,3),W_{4}=(1,2),W_{5}=(2,1),\\
			&W_{6}=(0,3),W_{7}=(1,2),W_{8}=(2,1),W_{9}=(3,0),W_{10}=(2,0).
		\end{split}
	\end{equation}
	Note, for instance, that $W_6 - W_5=(-2,2)$, indeed between the edges 5 and 6 the interface path is traversing a face with $3$ edges on the left boundary and $3$ edges on the right boundary.
	On the other hand $W_9-W_8 = W_8 - W_7 = W_7 - W_6 = (+1,-1)$. Indeed, in these cases, the interface path is following consecutive edges.
\end{exmp}

\begin{figure}[htbp]
	\begin{minipage}[c]{0.70\textwidth}
		\centering
		\includegraphics[scale=.95]{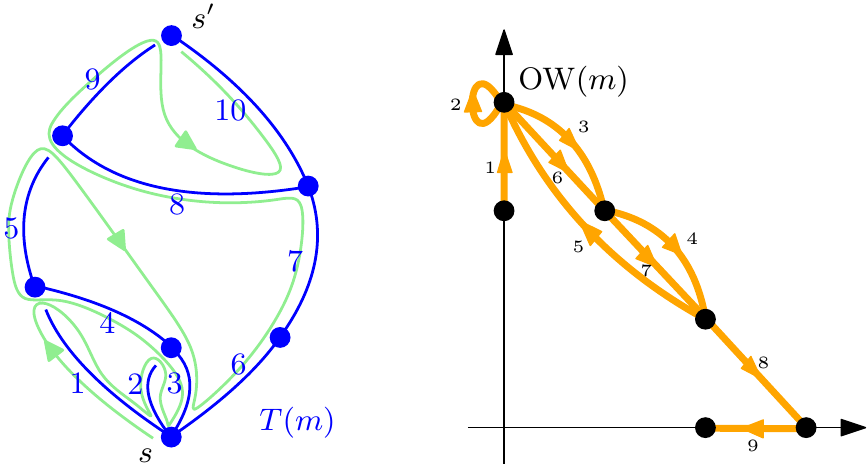}
	\end{minipage}
	\begin{minipage}[c]{0.29\textwidth}
		\caption{The bipolar orientation $m$ and the corresponding two-dimensional walk $\bow(m)$ already considered in \cref{fig:bip_orient _with_trees}. \label{fig:bip_orient _with_trees2}}
	\end{minipage}
\end{figure}

We finish this section by describing the inverse bijection  $\bow^{-1}$. We actually construct a mapping $\Inverse$ on a larger space of walks, whose restriction to $\mathcal W$ is the inverse of $\bow$.

Let $I$ be an interval (finite or infinite) of $\Z$. Let $\gls*{walks_I}$ be the set of two-dimensional walks with time space $I$ considered up to an additive constant. More precisely $\Walks(I)$ is the quotient $(\Z^2)^I / \sim$, where $w\sim w'$ if and only if there exists $x\in \Z^2$ such that $w(i) = w'(i)+x$ for all $i\in I$. We usually take an explicit representative of elements of $\Walks(I)$, chosen according to the context. For instance, if $0\in I$, we often take the representative that verifies $w(0)=(0,0)$, called "pinned at zero".
Let $\Walks_\Steps(I)\subset \Walks(I)$ be the restriction to two-dimensional walks with increments in $\Steps$ (see \cref{eq:admis_steps2}). For every $n\geq 1$, $\mathcal W_n$ is naturally embedded in $\Walks_\Steps([n])$, with an explicit representative.

Let $I=[j,k]$ be a finite integer interval. We shall define  $\Inverse$ on every walk in $ \Walks_\Steps(I)$ by induction on the size of $I$, and denote by $\Maps(I)$ the image of $\Walks_A(I)$ by $\Inverse$. An element $m\in\Maps(I)$ is a bipolar orientation, together with a subset of edges  labeled by $j,\ldots,k$ (which identifies a subinterval of the interface path started on the left boundary and ended on the right boundary of $m$). The edges labeled $j,\ldots,k$ are called \textit{explored edges}, and the edge labeled $k$ is called \textit{active}. The other edges, called \textit{unexplored}, are either below $j$ on the left boundary, or above $k$ on the right boundary. Bipolar orientations of size $n\geq 1$ are the elements of $\Maps([n])$ with no unexplored edges. The elements of $\Maps(I)$ are called \textit{marked bipolar orientations} in \cite{bousquet2019plane}.

The base case for our induction is $I=\{j\}$. In this case, an element $W$ of $\Walks_A(\{j\})$ is mapped to a single edge with label $j$. If $W \in \Walks_A([j,k+1])$ with $k\geq j$, then denote by $m' = \Inverse(W|_{[j,k]})$ and

\begin{enumerate}
	\item if $W_{k} - W_{k-1} = (1,-1)$, then $\Inverse(W)$ is obtained from $m'$, by giving label $k+1$ to the edge immediately above the edge of label $k$. If no such edge exists, a new edge is added on top of the sink with label $k+1$.
	\item If $W_{k} - W_{k-1} = (-i,j)$, then $\Inverse(W)$ is obtained from $m'$ by adding a face of left-degree $i+1$ and right-degree $j+1$. Its left boundary is glued to the right boundary of $m'$, starting with identifying the top-left edge of the new face with $e_k$, and continuing with edges below.
	The bottom-right edge of the new face is given label $k+1$, hence is now active. All other edges that were not present in $m'$ are unexplored.
\end{enumerate}

An example of this construction (inductively building $\Inverse(W)$ with $W\in \mathcal W$) is given in \cref{fig:bip_orient_from_walk}. For an example of application of the mapping $\Inverse$ to a walk that is not a tandem walk see \cref{fig:bip_orient_from_not_tandem_walk}. In this case, the final map still has unexplored edges.

\begin{figure}[htbp]
	\centering
	\includegraphics[scale=0.72]{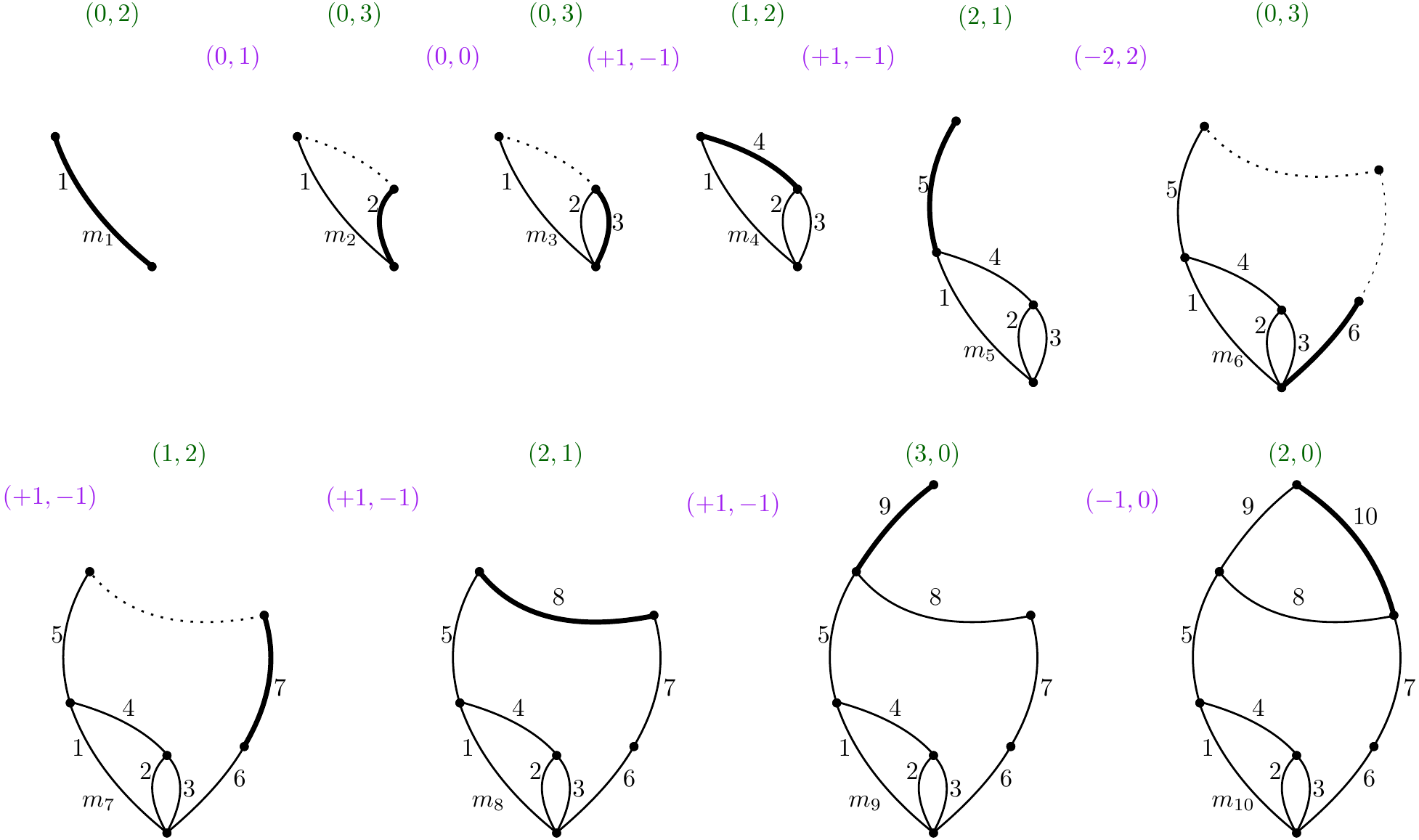}\\
	\caption{The sequence of bipolar orientations $m_k=\Inverse(W|_{[1,k]})$ determined by the walk $W$ considered in \cref{exemp:walk}. Note that $m_{10}$ is exactly the map $m$ in \cref{fig:bip_orient _with_trees2}. For each map $m_k$, we indicate on top of it (in green) the value $W_k$. Between two maps $m_k$ and $m_{k+1}$, we report (in purple) the corresponding increment $W_{k+1}-W_k$. For every map $m_k$, we draw the explored edges with full lines, the unexplored edges with dotted lines, and we additionally highlight the active edge $e_k$ in bold.   \label{fig:bip_orient_from_walk}}
\end{figure}

\begin{figure}[htbp]
	\centering
	\includegraphics[scale=0.75]{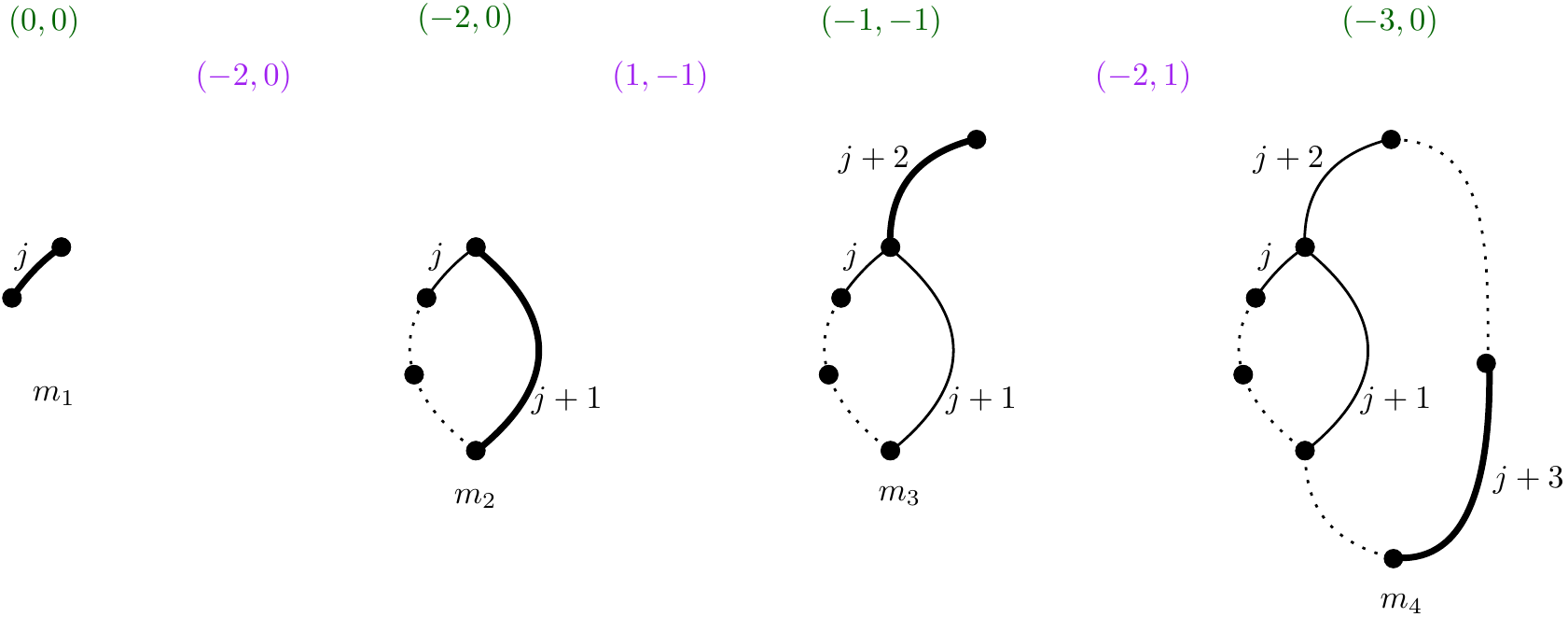}\\
	\caption{The sequence of bipolar orientations $m_k=\Inverse(W|_{[j,j+k-1]})$ determined by the walk $W_j=(0,0),W_{j+1}=(-2,0),W_{j+2}=(-1,-1),W_{j+3}=(-3,0)$, that is an element of the set $\Walks_\Steps([j,j+3])$. This walk is not a tandem walk. We used the same notation as in \cref{fig:bip_orient_from_walk}.  \label{fig:bip_orient_from_not_tandem_walk}}
\end{figure}

\begin{prop}[Theorem 1 of \cite{MR3945746}] \label{prop:inverse}
	For every finite interval $I$, the mapping $\Inverse: \Walks_A(I)\to \Maps(I)$ is a bijection. Moreover, for every $n\geq 1$, $\Theta(\mathcal W_n) = \mathcal O_n$ and $\bow^{-1}: \mathcal W_n \to \mathcal O_n$ coincides with $\Inverse$. 
	Finally, if $W \in \Walks_A(I)$ and $[j,k] \subset I$, then the map $\Inverse(W|_{[j,k]})$ is the submap obtained from $\Inverse(W)$ by keeping
	\begin{enumerate}
		\item edges with label $j,\ldots,k$ (explored edges);
		\item faces that have explored edges on both their left and right boundary (explored faces);
		\item other edges incident to explored faces (unexplored edges).
	\end{enumerate}
\end{prop}

\subsection{Mappings of the commutative diagram in the infinite-volume}\label{sect:infinite_bij}

Let $I$ be a (finite or infinite) interval of $\Z$. Recall that $\Walks_\Steps(I)$ is the subset of two-dimensional walks with time space $I$ and with increments in $\Steps$. Recall also that $\Coals(I)$ is the set of coalescent-walk processes on $I$, and that the mapping $\wcp : \Walks_\Steps(I) \to \Coals(I)$ sends walks to coalescent-walk processes, both in the finite and infinite-volume case.

In this section we extend the mappings $\cpbp$ and $\Inverse$, defined earlier, to infinite-volume objects. 

\medskip

We start with the mapping $\cpbp$. We denote by $\Perms(I)$ the set of total orders on $I$, which we call permutations on $I$. This terminology makes sense because for $n\geq 1$, the set $\Perms_n$ of permutations of size $n$ is readily identified with $\Perms([n])$, by the mapping $\sigma \mapsto \preccurlyeq_\sigma$, where 
\begin{equation}
	i \preccurlyeq_\sigma j\qquad\text{if and only if}\qquad \sigma(i)\leq\sigma(j)\;.
\end{equation} 

\medskip

Then we can extend $\cpbp: \Coals(I) \to \Perms(I)$ to the case when $I$ is infinite. For $Z = \{Z^{(t)}\}_{t\in I} \in \Coals(I)$, we set $\cpbp(Z)$ to be the total order $\leq_Z$  on $I$ defined, as in \cref{sect:from_coal_to_perm}, by
\begin{equation}
	\begin{cases}i\leq_Z i,\\
		i\leq_Z j,&\text{ if }i<j\text{ and }\ Z^{(i)}_j<0,\\
		j\leq_Z i,&\text{ if }i<j\text{ and }\ Z^{(i)}_j\geq0.\end{cases}
\end{equation}
This is consistent, through the stated identification, with our previous definition of $\cpbp$ when $I = [n]$.

\medskip

Finally, we also extend $\Inverse$ to the infinite-volume case. We need first to clarify what is our definition of \emph{infinite planar maps}. From now on,  we use the following equivalent definition of (finite) planar map: a planar map is a finite collection of finite polygons (the inner faces), glued
along some pairs of edges, so that the resulting surface has the topology of the disc, i.e.\ is
simply connected and has one boundary.

\begin{defn}\label{defn:Inf_map}
	An \emph{infinite oriented quasi-map} is an infinite collection of finite polygons with oriented edges, glued along some of their edges in such a way that the orientation is preserved. 
	In the case the graph corresponding to an infinite oriented quasi-map is locally finite (i.e.\ every vertex has finite degree) then we say that it is an \emph{infinite oriented map}. 
	
	We call \emph{boundary} of an infinite oriented map, the collection of edges of the finite polygons that are not glued with any other edge.
	An infinite oriented map $m$ is 
	\begin{itemize}
		\item \textit{simply connected} if for every finite submap $f\subset m$ there exists a finite submap $f'\subset m$ which is a planar map (i.e.\ is simply connected) with $f\subset f'\subset m$;
		\item \textit{boundaryless} if the boundary of $m$ is empty.
	\end{itemize}
	When these two conditions are verified, we say that $m$ is an \textit{infinite map of the plane}\footnote{We point out that when the map is locally finite, it is well-defined as a topological manifold, and the combinatorial notions of simple connectivity and boundarylessness defined above are equivalent to the topological ones. By the classification of two-dimensional surfaces, $m$ is an infinite map of the plane if and only if it is homeomorphic to the plane.}.
\end{defn}

Let $I$ be an infinite interval and $w\in \Walks_\Steps(I)$. Using the definition above, we can view a finite bipolar orientation as a finite collection of finite inner faces, together with an adjacency relation on the oriented edges of these faces. This allows us to construct $\Inverse(w)$ as a projective limit, as follows: from \cref{prop:inverse}, if $J$ and $J'$ are finite intervals and $J'\subset J \subset I$ then $\Inverse(w|_{J'})$ is a submap of $\Inverse(w|_J)$ defined in a unique way. This means that the face set of $\Inverse(w|_{J'})$ is included in that of $\Inverse(w|_J)$, and that two faces of $\Inverse(w|_{J'})$ are adjacent if and only if they are adjacent through the same edge in $\Inverse(w|_J)$.
Then, taking $J_n$ a growing sequence of finite intervals such that $I = \cup_n J_n$, we set $\Inverse(w) := \mathrlap {\  \uparrow} \bigcup_{n\geq 0} \Inverse(w|_{J_n})$ to be an infinite collection of finite polygons, together with a gluing relation between the edges of these faces, i.e.\ $\Inverse(w)$ is an infinite oriented quasi-map.

\subsection{Random infinite limiting objects}\label{sect:inf_loc_obj}
We define here what will turn out to be our local limiting objects. Recall that $\nu$ denotes the following probability distribution on $A$ (defined in \cref{eq:admis_steps} page~\pageref{eq:admis_steps}):
\begin{equation}\label{eq:walk_distrib}
	\nu = \frac 12 \delta_{(+1,-1)} + \sum_{i,j\geq 0} 2^{-i-j-3}\delta_{(-i,j)},\quad\text{where $\delta$ denotes the Dirac measure,}
\end{equation}
and recall that $\overline{\bm W} = (\overline{\bm X},\overline{\bm Y}) = (\overline{\bm W}_t)_{t\in \Z}$ is a bi-directional random two-dimensional walk with step distribution $\nu$, having value $(0,0)$ at time zero. The interest of introducing the probability measure $\nu$ resides in the following way of obtaining uniform elements of $\mathcal W_n$ as conditioned random walks. For all $n\geq 1$, let $\Walks_{A,\exc}^n \subset \Walks_A([0,n-1])$ be the set of two-dimensional walks $W=(W_t)_{0\leq t\leq n-1}$ in the non-negative quadrant, starting and ending at $(0,0)$, with increments in $A$. Notice that for $n\geq 1$, the mapping $\mathcal \Walks_{A,\exc}^{n+2} \to \mathcal W_{n}$ removing the first and the last step, i.e.\ $W \mapsto (W_{t})_{1\leq t \leq n}$,
is a bijection. A simple calculation then gives the following (obtained also in \cite[Remark 2]{MR3945746}):

\begin{prop}\label{prop:unif_law}
	Conditioning on $\{(\overline{\bm W}_t)_{0\leq t\leq {n+1}}  \in \mathcal \Walks_{A,\exc}^{n+2}\}$, the law of  $(\overline{\bm W}_t)_{0\leq t \leq n+1}$ is the uniform distribution on $\Walks_{A,\exc}^{n+2}$, and the law of $(\overline{\bm W}_t)_{1\leq t \leq n}$ is the uniform distribution on $\mathcal W_n$.
\end{prop} 

Let us now consider $\overline {\bm m} = \Inverse(\overline {\bm W})$ and describe some of its properties. Recall that it is an infinite quasi-map, i.e.\ a countable union of finite polygons, glued along edges. 

\begin{prop}[{\cite[Proposition 3.5]{borga2020scaling}}]\label{prop:inf_map_property}
	Almost surely, $\overline{\bm m}$ is an infinite map of the plane. 
	In particular, it is locally finite.
\end{prop}

The proposition above was proved using the following remarkable property of the trajectories of the coalescent-walk process $\overline{\bm Z} = \wcp(\overline{\bm W})$. This property will be also of great help when we will look at the scaling limit of $\overline{\bm Z}$ in \cref{sect:scal_coal_proc}.

\begin{prop}[{\cite[Proposition 3.3]{borga2020scaling}}]\label{prop:trajectories_are_rw}
	For every $t\in \Z$, $\overline{\bm Z}^{(t)}$ has the distribution of a random walk with the same step distribution as $\overline{\bm Y}$ (which is the same as that of $-\overline{\bm X}$).
\end{prop}

\begin{rem}
	Recall that the increments of a walk of a coalescent-walk process are not always equal to the increments of the corresponding walk (see the last case considered in \cref{defn:distrib_incr_coal}). The statement of \cref{prop:trajectories_are_rw} is a sort of "miracle" of the geometric distribution.
\end{rem}

\subsection{Local topologies}\label{sect:local topologies}

The following constructions and definitions are reminiscent of the ones given in \cref{sect:local_theory}, where we introduced the local topology for permutations.

We define a (finite or infinite) \emph{rooted walk} (resp.\ \emph{rooted coalescent-walk process}) as a walk (resp.\ coalescent-walk process) on a (finite or infinite) interval of $\Z$ containing $0$. More formally, we define the following sets (with the corresponding notions of size):
\begin{align}
	\widetilde \Walks_\bullet &\coloneqq \bigsqcup_{I \ni 0} \Walks(I),&&\text{where}\quad |W| \coloneqq |I|\text{ if }W\in \Walks(I),\\
	\widetilde \Coals_\bullet &\coloneqq \bigsqcup_{I \ni 0} \Coals(I),&&\text{where}\quad |Z| \coloneqq |I|\text{ if }Z\in \Coals(I).
\end{align}
Of course, $0$ has to be understood as the root of any object in one of these classes. For $n\in \Z_{>0} \cup \{\infty\}$, $\Walks_\bullet^n$ is the subset of objects in $\widetilde \Walks_\bullet$ of size $n$, i.e.\ $\Walks_\bullet^n=\bigcup_{I \ni 0, |I|=n} \Walks(I)$, and $\Walks_\bullet$ denotes the set of finite-size objects. We also define analogs for $\widetilde \Coals_\bullet$.

We justify the terminology. A rooted object of size $n$ can be also understood as an unrooted object of size $n$ together with an index in $[n]$ which identifies the root through the following identifications\footnote{Note that the natural identification for walks would be $(W,i) \longmapsto (W_{i+t} - W_i)_{t\in [-i+1, n-i]}$, but since we are considering walks up to an additive constant then the identification $(W,i) \longmapsto (W_{i+t})_{t\in [-i+1, n-i]}$ is equivalent.}: 
\begin{align}
	\Walks^n \times [n] \longrightarrow \Walks^n_\bullet,
	&\quad(W,i) \longmapsto (W_{i+t})_{t\in [-i+1, n-i]}\in \Walks([-i+1, n-i]),\\
	\Coals^n \times [n] \longrightarrow \Coals^n_\bullet,
	&\quad((Z^{(s)}_t)_{s,t\in [n]},i) \longmapsto (Z^{(i+s)}_{i+t})_{s,t\in [-i+1, n-i]}\in \Coals([-i+1, n-i]).
\end{align}

We may now define \emph{restriction functions}: for $h\geq 1$, $I$ an interval of $\Z$ containing 0, and $\square\in  \Walks(I)$ or $\Coals(I)$, we define
\[r_h(\square) =\square |_{I \cap [-h,h]}.\]
So, for all $h\geq 1$, $r_h$ is a well-defined function $\widetilde \Walks_\bullet \to \Walks_\bullet$, $\widetilde \Coals_\bullet \to \Coals_\bullet$. Finally local distances on either $\widetilde \Walks_\bullet$ and $\widetilde \Coals_\bullet$ are defined by a common formula:
\begin{equation}\label{eq:local_distance}
	d\big(\square_1,\square_2\big)=2^{-\sup\big\{h\geq 1\;:\;r_h(\square_1)=r_h(\square_2)\big\}},
\end{equation}
with the conventions that $\sup\emptyset=0,$ $\sup\Z_{>0}=+\infty$ and $2^{-\infty}=0.$

For the case of planar maps, the theory is slightly different. For a finite interval $I$, we denote by $\Maps(I)$ the set of planar oriented maps with edges labeled by the integers in the interval $I$. We point out that we do not put any restriction on the possible choices for the labeling.
We also define the set of finite rooted maps $\Maps_\bullet = \bigsqcup_{I\ni 0,|I|<\infty} \Maps(I)$, where the edge labeled by $0$ is called root.

A finite rooted map is obtained by rooting (i.e.\ distinguishing an edge) a finite unrooted map, by the identification $\Maps([n]) \times [n] \to \Maps_\bullet^n$ which sends $(m,i)$ to the map obtained from $m$ by shifting all labels by $-i$. We also define a distance.
Let $B_h(m)$ be the ball of radius $h$ in $m$, that is the submap of $m$ consisting of the faces of $m$ that contain a vertex at distance less than $h$ from the tail of the root-edge. We set $d(m,m') =2^{-\sup\big\{h\geq 1\;:\;B_h(m) = B_h(m')\big\}},$ with the same conventions as before. We do not describe the set of possible limits, but simply take $\widetilde \Maps_{\bullet}$ to be the completion of the metric space $(\Maps_\bullet, d)$. In particular, it is easily seen that $\widetilde \Maps_{\bullet}$ contains all the infinite rooted planar maps.

\begin{prop} 
	\label{prop:continuty_of_cpbp}
	The mappings $\wcp: \widetilde \Walks_\bullet \to	\widetilde \Coals_\bullet$ and $\cpbp: \widetilde \Coals_\bullet \to 
	\widetilde \Perms_\bullet$
	are $1$-Lipschitz.
\end{prop}
\begin{proof}
	This is immediate since by definition, $r_h(\wcp(W)) = \wcp(r_h(W))$ and $r_h(\cpbp(Z)) = \cpbp(r_h(Z))$.
\end{proof}
\begin{prop} 
	\label{prop:continuty_of_Inverse}
	The mapping $\Inverse: \widetilde \Walks_\bullet \to \widetilde \Maps_\bullet$ is almost surely continuous at $\overline{\bm W}$.
\end{prop}

\begin{proof}
	Assume we have a realization $\overline{W}$ of $\overline{\bm W}$ such that $\overline{ m} = \Inverse(\overline{ W})$ is a rooted infinite oriented planar map, in particular is locally finite (this holds for almost all realizations thanks to \cref{prop:inf_map_property}). Let $h > 0$. The ball $B_h(\overline{m})$ is a finite subset of $\overline{m}$. Since $\overline{m} = \bigcup_n \Inverse(\overline{W}|_{[-n,n]})$, there must be $n$ such that $\Inverse(\overline{W}|_{[-n,n]}) \supset B_h(\overline{m})$. As a result, for all $W'\in\widetilde \Walks_\bullet$ such that $d(W',\overline{W})<2^{-n}$, then $B_h(\Inverse(W'))= B_h(\Inverse(\overline{ W}))$. This shows a.s.\ continuity.
\end{proof}

\subsection{Local convergence for tandem walks, coalescent-walk processes, Baxter permutations and bipolar orientations}\label{sect:hbkfweibfoiwe}

We turn to the proof of \cref{thm:local} page \pageref{thm:local}. We will prove local convergence for walks and then transfer to the other objects by continuity of the mappings $\wcp,\cpbp,\Inverse$. Let us recall that thanks to \cref{prop:unif_law}, $\bm W_n$ is distributed like $\overline{\bm W}|_{[n]}$ under a suitable conditioning. This will allow us to prove the following lemmas.

\begin{lem}\label{lem:quenched_conv_walks1}
	Fix $h\in\Z_{>0}$ and $W \in \Walks([-h,h])\subset \Walks_\bullet^{2h+1}$.
	Fix $0<\varepsilon<1$. Then, uniformly for all $i$ such that $\lfloor n\varepsilon \rfloor +h < i< \lfloor (1-\varepsilon)n \rfloor-h$,
	\begin{equation}\label{eq:statement_uenched_conv_walks1}
		\mathbb{P}\left(
		r_{h}({\bm W}_n,i)=W
		\right)
		\to\mathbb{P}\left(
		r_{h}(\overline{\bm W})=W \right).
	\end{equation}
\end{lem}

\begin{lem}\label{lem:quenched_conv_walks2}
	Fix $h\in\Z_{>0}$ and $W \in \Walks([-h,h])\subset \Walks_\bullet^{2h+1}$.
	Fix $0<\varepsilon<1$. Then, uniformly for all $i,j$ such that $\lfloor n\varepsilon \rfloor +h < i,j< \lfloor (1-\varepsilon)n \rfloor-h$ and $|i-j|>2h$,
	\begin{equation}\label{eq:statement_uenched_conv_walks2}
		\mathbb{P}\left(
		r_{h}({\bm W}_n,i)=r_{h}({\bm W}_n,j)=W
		\right)
		\to\mathbb{P}\left(
		r_{h}(\overline{\bm W})=W \right)^2.
	\end{equation}
\end{lem}

We just prove the second lemma, the proof of the first one is similar and simpler. The proof of \cref{lem:quenched_conv_walks2} contains the key idea for establishing local convergence for tandem walks, that is to use some absolute continuity results between a conditioned walk away of its starting and ending points and an unconditioned one. These results were established in \cite{MR3342657,duraj2015invariance,bousquet2019plane} and they are collected in \cref{sec:appendix} since they will be also used later in \cref{sect:baxperm}.

For the proof of \cref{lem:quenched_conv_walks2} we need the following observation.
In what follows, if $W = (X,Y)$ is a two-dimensional walk, then $\inf W = (\inf X, \inf Y)$.

\begin{obs}\label{obs:infofwalks}
	Let $x=(x_i)_{i\in[0,n]}=(\sum_{j=1}^i y_j)_{i\in[n]}$ be a one-dimensional deterministic walk starting at zero of size $n$, i.e.\ $y_j\in\Z$ for all $j\in [n]$. Let $h$ be much smaller than $n$ and consider a second deterministic one-dimensional walk $x'=(x'_i)_{i\in[0,h]}=(\sum_{j=1}^i y'_j)_{i\in[0,h]}$. Fix also $k,\ell$ such that $0\leq k<\ell\leq n$ and consider the walk $x''=(x''_i)_{i\in[0,n+2h]}$ obtained by inserting two copies of the walk $x'$ in the walk $x$ at time $k$ and $\ell$. That is, for all $i\in[0,n+2h]$,
	\begin{equation}
		x''_i= \sum_{j=1}^k y_j\cdot\mathds{1}_{j\leq i}+\sum_{j=1}^h y'_j\cdot\mathds{1}_{j+k\leq i}+\sum_{j=k+1}^\ell y_j\cdot\mathds{1}_{j+h\leq i}+\sum_{j=1}^h y'_j\cdot\mathds{1}_{j+\ell+h\leq i}+\sum_{j=\ell+1}^n y_j\cdot\mathds{1}_{j+2h\leq i}.
	\end{equation}
	Then 
	$$\inf_{i\in[0,n+2h]}\{x''_i\}=\inf_{i\in[0,n]}\{x_i\}+\Delta,$$
	where $\Delta=\Delta(x,k,\ell,x')\in\R^2$ and it is bounded by twice the total variation of $x'$.
\end{obs}

\begin{proof}[Proof of \cref{lem:quenched_conv_walks2}]
	Set $E\coloneqq\left\{r_{h}(\overline{\bm W}|_{[n]},i)=r_{h}(\overline{\bm W}|_{[n]},j)=W\right\}$. By \cref{prop:unif_law}, the left-hand side of \cref{eq:statement_uenched_conv_walks2} can be rewritten as $\mathbb{P}\left(E \middle|(\overline{\bm W}_t)_{0\leq t\leq {n+1}}  \in \Walks_{A,\exc}^{n+2} \right)$. Using \cref{lem:AbsCont}, we have that
	\begin{multline}\label{eq:first_rewriting}
		\mathbb{P}\left(E \middle|(\overline{\bm W}_t)_{0\leq t\leq {n+1}}  \in \Walks_{A,\exc}^{n+2} \right)
		=\E\left[\mathds{1}_{\widetilde{E}}\cdot 
		\alpha_{n+2,\lfloor n\varepsilon \rfloor}^{0,0}\left(\inf_{0\leq k \leq n+2-2\lfloor n\varepsilon \rfloor} \overline{\bm W}_k\;,\;\overline{\bm W}_{n+2-2\lfloor n\varepsilon \rfloor}\right)
		\right],
	\end{multline} 
	where  $\alpha_{n+2,\lfloor n\varepsilon \rfloor}^{0,0}(a,b)$ is a function defined in \cref{eq:alpha_tilting_function} page \pageref{eq:alpha_tilting_function} and 
	$$\widetilde{E}\coloneqq\left\{r_{h}(\overline{\bm W}|_{[n]},i-\lfloor n\varepsilon \rfloor)=r_{h}(\overline{\bm W}|_{[n]},j-\lfloor n\varepsilon \rfloor)=W\right\}.$$
	From \cref{obs:infofwalks}, conditioning on $\widetilde{E}$, we have that
	\begin{multline}\label{eq:equality_walks}
		\left(\inf_{0\leq k \leq n+2-2\lfloor n\varepsilon \rfloor} \overline{\bm W}_k\;,\;
		\overline{\bm W}_{n+2-2\lfloor n\varepsilon \rfloor}
		\right)\\
		=\left(\inf_{0\leq k \leq n+2-2\lfloor n\varepsilon \rfloor-2(2h+1)} \overline{\bm S}_k+  \bm\Delta\;,\;\overline{\bm S}_{n+2-2\lfloor n\varepsilon \rfloor-2(2h+1)}+2\delta\right),
	\end{multline}
	where $\overline{\bm S}=(\overline{\bm S}_t)_{t\in \Z}$ is the walk obtained from $(\overline{\bm W}_t)_{t\in \Z}$ removing the $2h+1$ steps around $i-\lfloor n\varepsilon \rfloor$ and $j-\lfloor n\varepsilon \rfloor$ , $\delta = W_h - W_{-h}$ and $\bm \Delta$ is a deterministic function of $(\overline{\bm S}_t)_{t\in \Z}$, $i$, $j$ and $W$, bounded by twice the total variation of $W$.
	Using the relation in \cref{eq:equality_walks} we can rewrite the right-hand side of \cref{eq:first_rewriting} as
	\begin{equation}
		\mathbb{P}(\widetilde{E})
		\cdot
		\E\left[\alpha_{n+2,\lfloor n\varepsilon \rfloor}^{0,0}\left(\inf_{0\leq k \leq n-2\lfloor n\varepsilon \rfloor-2(2h+1)} \overline{\bm S}_k+ \bm \Delta\;,\;\overline{\bm S}_{n-2\lfloor n\varepsilon \rfloor-2(2h+1)}+2\delta\right)
		\right],
	\end{equation}
	where we used the independence between $\widetilde{E}$ and the right-hand side of \cref{eq:equality_walks}. Note that $\P(\widetilde{E})=
	\mathbb{P}\left(r_{h}(\overline{\bm W})=W \right)^2$ since  $|i-j|>2h$ by assumption.
	We now show that 
	the second factor of the equation above converges to 1 uniformly for all $i$,$j$.
	Set for simplicity of notation 
	\begin{equation}
		f(\overline{\bm S})=\left(\inf_{0\leq k \leq n-2\lfloor n\varepsilon \rfloor-2(2h+1)} \overline{\bm S}_k+ \bm \Delta\;,\;\overline{\bm S}_{n-2\lfloor n\varepsilon \rfloor-2(2h+1)}+2\delta\right).
	\end{equation}
	By \cref{lem:LLT} we have
	$$\sup_{i,j}\Bigg|\E\left[\alpha_{n+2,\lfloor n\varepsilon \rfloor}^{0,0}\left(f(\overline{\bm S})\right)\right]-\E\left[\alpha_\eps\left(g(\overline{\bm S})\right)\right]\Bigg|\to0,$$
	where $\alpha_\eps(\cdot)$ is defined in \cref{eq:fuction_alpha} and
	$$g(\overline{\bm S})=\left(\frac 1{\sqrt { n}}\left(\inf_{0\leq k \leq n-2\lfloor n\varepsilon \rfloor-2(2h+1)} \overline{\bm S}_k\right) + \frac {\bm\Delta}{\sqrt{ n}}\;,\;\frac 1{\sqrt { n}}\overline{\bm S}_{n-2\lfloor n\varepsilon \rfloor-2(2h+1)} + \frac {2\delta}{\sqrt{ n}}\right).$$
	Therefore, in order to conclude, it is enough to show that $\E\left[\alpha_\eps\left(g(\overline{\bm S})\right)\right]\to1$. 
	
	We have that $\E\left[\alpha_\eps\left(g(\overline{\bm S})\right)\right]\to\E\left[\alpha_\eps
	\left(g( {\conti W})\right)\right]$,
	where
	$\conti W = (\conti X,\conti Y)$ is a standard two-dimensional Brownian motion of correlation $-1/2$. This follows from the fact that $\bm \Delta$ is bounded, that $\alpha_\eps$ is a continuous and bounded function (see \cref{lem:LLT}), and that 
	$$\left(\frac{1}{\sqrt n}\overline{\bm S}_{\lfloor nt\rfloor}\right)_{t\in[0,1]}\stackrel{d}{\longrightarrow}\left(\bm{\conti W}_t\right)_{t\in[0,1]}.$$
	The latter claim is a consequence of Donsker's theorem and the basic computation $\Var(\nu) = \begin{psmallmatrix}2 &-1 \\ -1 &2 \end{psmallmatrix}$. In addition, we have that $\E\left[\alpha_\eps
	\left(g( {\conti W})\right)\right]=1$ by \cref{prop:brown_ex} (used with $h=1$), and so we can conclude the proof. 
\end{proof}

We can now prove the main result of this section, i.e.\ the joint quenched local limit result presented in the introduction.

\begin{proof}[Proof of \cref{thm:local}]
	We start by proving that
	\begin{equation}\label{eq:firststepintheproof}
		\mathcal{L}aw\Big((\bm W_n, \bm i_n)\Big|\bm W_n\Big)\stackrel{P}{\longrightarrow}\mathcal{L}aw\left(\overline{\bm W}\right).
	\end{equation}
	For that it is enough (recall \cref{strongbsconditions} and \cref{detstrongbsconditions} for a similar argument in the case of permutations) to show that, for any $h\geq 1$ and fixed finite rooted walk $W\in \Walks([-h,h]) \subset \Walks_\bullet$,
	\begin{equation}\label{eq:goaloftheproof3}
		\mathbb{P}\left(r_{h}(\bm W_n,\bm i_n)=W
		\mid
		\bm W_n\right)
		\stackrel{P}{\longrightarrow}\mathbb{P}(r_{h}(\overline{\bm W})=W).
	\end{equation}
	Note that 
	\begin{align}
		\mathbb{P}\left(r_{h}(\bm W_n,\bm i_n)=W
		\mid
		\bm W_n\right)
		&=\frac{\#\left\{j\in[n] : 
			r_h(\bm W_n,j)=W)\right\}}{n}
		=\frac 1 n \sum_{j\in [n]}\mathds{1}_{\left\{r_{h}(\bm W_n,j)=W)\right\}}.
	\end{align}
	We use the second moment method to prove that this sum converges in probability to $\mathbb{P}(r_{h}(\overline{\bm W})=W)$.
	We first compute the first moment:
	\begin{equation}
		\E\left[\frac 1 n \sum_{j\in [n]}\mathds{1}_{\left\{r_{h}({\bm W_n},j)=W)\right\}}\right]=\frac 1 n \sum_{j\in [n]}\mathbb{P}\left(r_{h}({\bm W_n},j)=W\right)
		\to\mathbb{P}\left(r_{h}(\overline{\bm W})=W\right),
	\end{equation}
	where for the limit we used \cref{lem:quenched_conv_walks1}.
	We now compute the second moment:
	\begin{multline}
		\E\left[\left(\frac 1 n \sum_{j\in [n]}\mathds{1}_{\left\{r_{h}(({\overline{\bm W}}_t)_{t\in [n]},j)=W)\right\}}\right)^2\;\right]\\
		=\frac{1}{n^2}\sum_{i,j\in [n], {|i-j|>2h}}\P\left(r_{h}({\bm W_n},i)=r_{h}({\bm W_n},j)=W\right) + O(1/n).
	\end{multline}
	This converges to $\mathbb{P}(r_{h}(\overline{\bm W})=W)^2$ by \cref{lem:quenched_conv_walks2}.
	The computations of the first and second moment, together with Chebyshev's inequality lead to the proof of \cref{eq:goaloftheproof3} and so to the quenched convergence of walks.
	
	Now to extend the result to other objects, we will use continuity of the mappings $\wcp,\cpbp,\Inverse$ (see \cref{prop:continuty_of_cpbp} and \cref{prop:continuty_of_Inverse}). 
	Using a combination of the results stated in \cite[Theorem 4.11, Lemma 4.12]{kallenberg2017random}, see also \cref{footnote:KAL} page \pageref{footnote:KAL}, we have that \cref{eq:firststepintheproof} implies the following convergence
	\begin{equation}
		\mathcal{L}aw\Big(\big((\bm W_n, \bm i_n), (\wcp(\bm W_n), \bm i_n), (\cpbp(\bm W_n), \bm i_n), (\Inverse(\bm W_n), \bm i_n)\big)\Big| \mathfrak{B}_n\Big)\stackrel{P}{\longrightarrow}\mathcal{L}aw\left(\overline{\bm W},\overline{\bm Z},\overline{\bm \sigma},\overline{\bm m}\right),
	\end{equation}
	where we recall that $\mathfrak{B}_n\coloneqq\sigma(\bm W_n)=\sigma(\bm Z_n)=\sigma(\bm \sigma_n)=\sigma(\bm m_n).$ This ends the proof.
\end{proof}

\section{Open problems}

\begin{itemize}
	
	\item We have seen concentration phenomena for consecutive patterns in many pattern-avoiding permutations and we have also seen that this is not the case for square permutations. A natural open question is the following one: Under what conditions a random pattern-avoiding permutation shows a concentration phenomenon for the proportion of consecutive patterns?
	
	Recall also that Janson \cite[Remark 1.1]{janson2018patterns} raised a similar question for classical patterns. 
	
	\item We proved that uniform 321-avoiding permutations converge in the quenched B--S sense (see \cref{thm_2} and \cref{321corol}). Recently,  Hoffman, Rizzolo and Slivken \cite{hoffman2019dyson} showed that uniform permutations avoiding any monotone pattern converge (after a proper rescaling) to the traceless Dyson Brownian bridge. Building on this new result, we believe that it is possible to generalize the proof of \cref{thm_2} for all uniform $\rho$-avoiding permutation when $\rho$ is an increasing or decreasing permutation of size greater than three.
	
	\item The central limit theorem for consecutive patterns of permutations encoded by generating trees should hold under weaker conditions. In particular, we believe that \cref{ass1} in \cref{thm:main_thm_CLT}, i.e.\ that the generating tree has one-dimensional labels, can be relaxed. Our assumptions reduce the problem on permutations to the study of one-dimensional random walks conditioned to be positive. Relaxing \cref{ass1} to multi-dimensional labels, leads to the study of \emph{multi-dimensional random walks conditioned to stay in a cone}. We believe that the work of Denisov and Wachtel~\cite{MR3342657} furnishes some useful results to investigate this problem, specifically to prove an analogue of the results proved in \cref{sect:CTL_increments} for random walks in cones. Recall also that in \cref{sect:hbkfweibfoiwe} we proved a law of large numbers for consecutive increments of tandem walks. The goal would be to establish a central limit theorem for consecutive increments of a larger class of random walks in cones.
	
	A motivation for studying this more general framework is given by that fact that several families of permutations have been enumerated using multi-dimensional generating trees, for instance Baxter permutations (see \cite{bousquet2003four}). In \cref{sec:local}, we
	used as a key tool to study consecutive patterns in Baxter permutations the Kenyon-Miller-Sheffield-Wilson bijection between Baxter permutations and tandem walks, since the latter walks are easier to study than the walks given by the generating tree approach. Nevertheless, the generating tree approach would give a more standard tool to encode permutations with walks in cones that can be used in different models, such as strong-Baxter and semi-Baxter permutations (see \cite{MR3882946}). We will discuss this point with more details in \cref{sect:univ}.
	
	\item Another open problem related with the central limit theorem for consecutive patterns of permutations encoded by generating trees is the following one: It is not evident to us that the expressions of the variances $\gamma_{\pi}^2$ in \cref{thm:main_thm_CLT} satisfy $\gamma_{\pi}^2>0$. It would be interesting to investigate this problem. Moreover, given a specific family of permutation $\mathcal{C}$, it would be interesting to explicitly compute the expressions for $\mu_\pi$, for all $\pi\in\mathcal{C}$, as done  in Theorems \ref{thm_1} and \ref{thm_2} for instance for the families $\Av(123)$ and $\Av(132)$.
\end{itemize}
    \chapter{Phase transition for square and almost square permutations, fluctuations, \& generalizations to other models}\label{chp:square}
\chaptermark{Square and almost square permutations}

\begin{adjustwidth}{8em}{0pt}
	\emph{In which we investigate the occurrence of a phase transition on the limiting permuton for square and almost square permutations when the number of internal points increases. For square permutations the limiting shape is a random rectangle: we further study the fluctuations of the dots of the diagram around the four edges of the rectangle. Finally, we extend our techniques to study 321-avoiding permutations with a fixed number of internal points.}
\end{adjustwidth}

\bigskip

\bigskip

\bigskip

\bigskip

\bigskip

\begin{figure}[h]
	\centering
	\includegraphics[scale=0.286]{Square_2_2000}
	\includegraphics[scale=0.286]{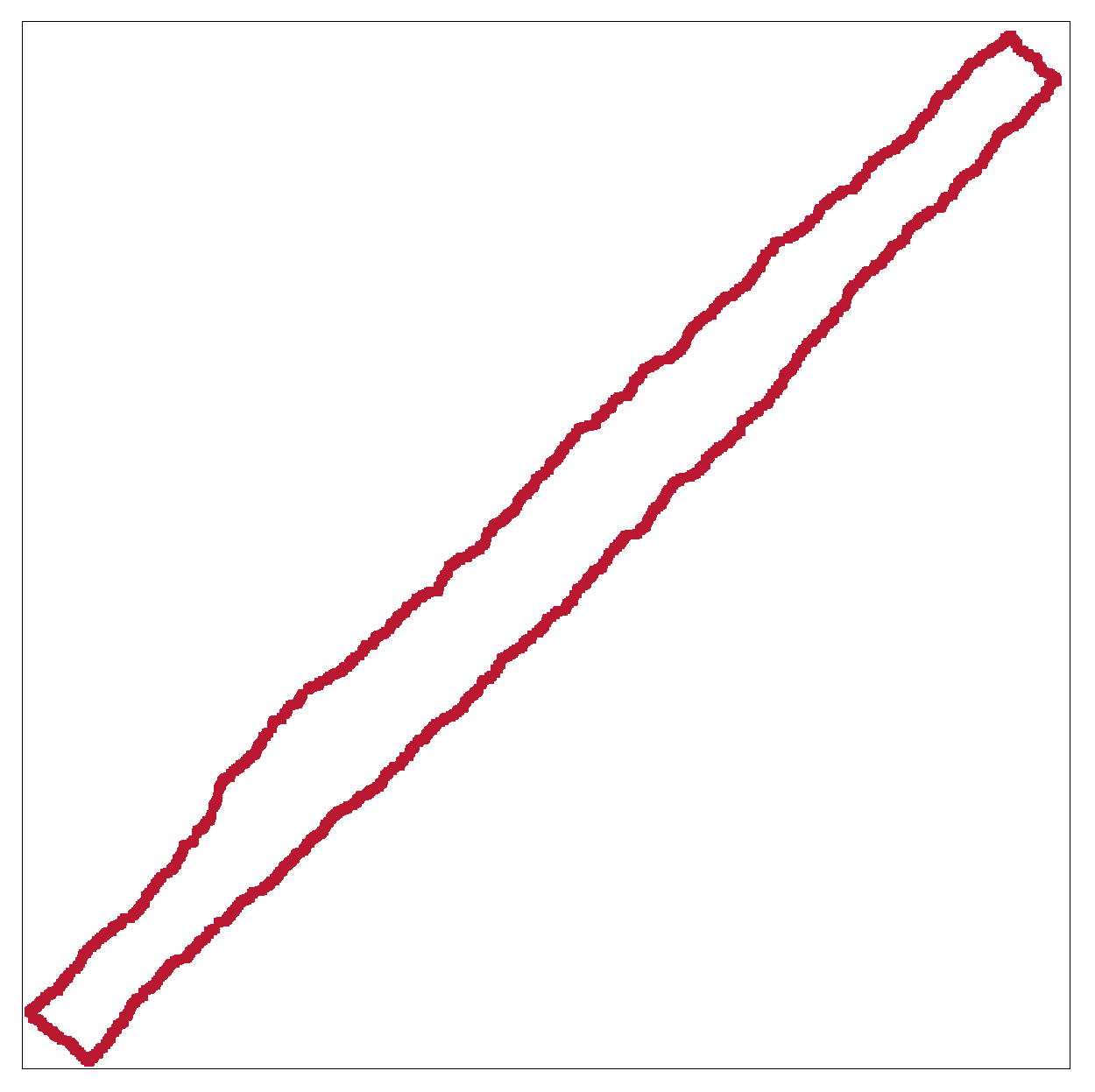}
	\includegraphics[scale=0.286]{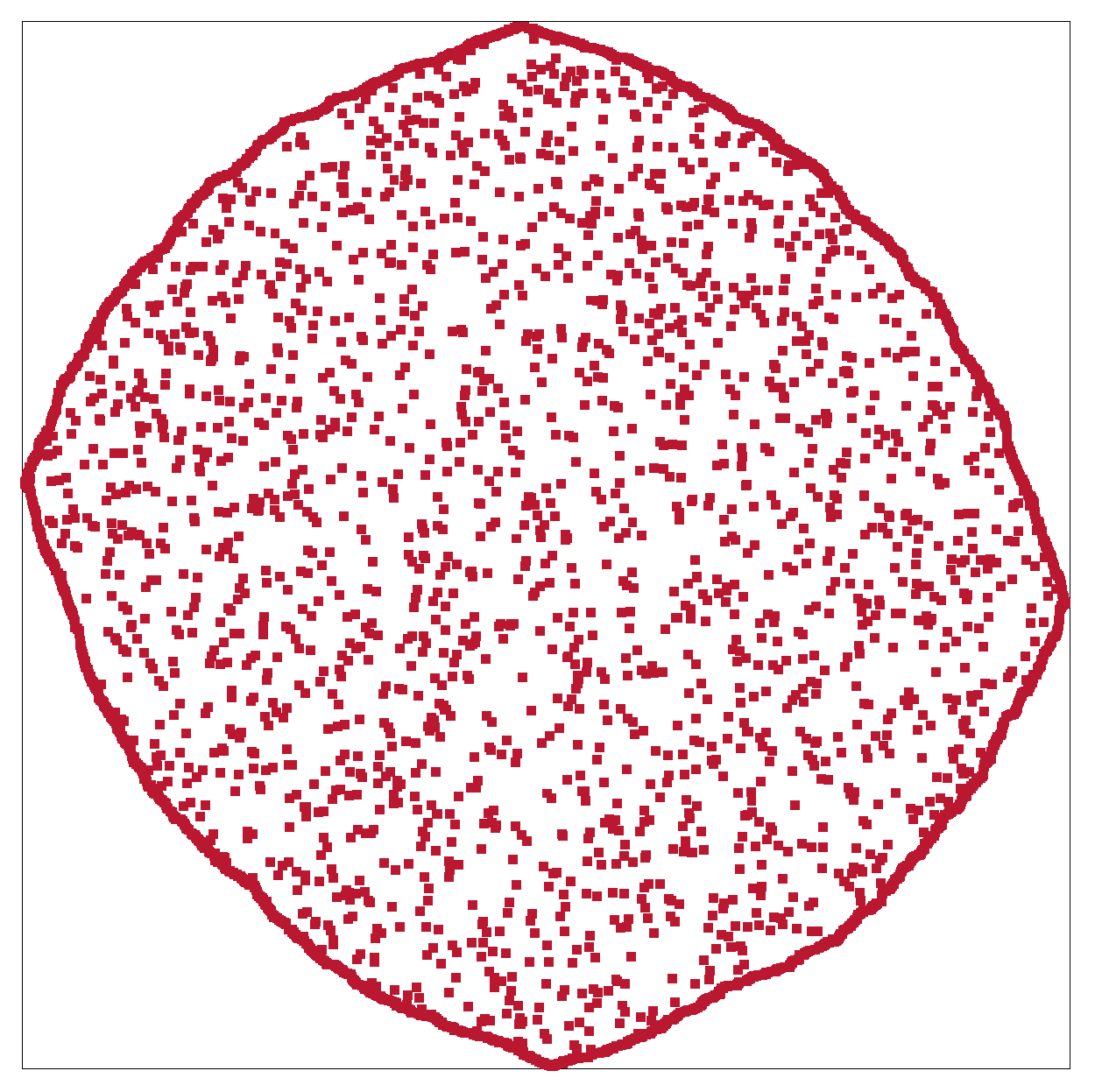}
	\includegraphics[scale=0.286]{circ_1_5_2}
	\caption{\textbf{Left}: Two large uniform square permutations. \textbf{Right}: Two large uniform almost square permutations; the first one has half of the points that are external, the second one a fifth. \label{fig:phase_trans}}
\end{figure}

\newpage

\section{Overview of the main results}
\subsection{The phase transition}

We first investigate the limiting permutons for uniform square or almost square permutations. In order to state our result we first need some definitions.

\begin{defn}\label{defn:square_perm}
	For $z\in(0,1)$, let $\mu^{z}$ denote the (deterministic) permuton given by the Lebesgue measure of total mass 1 supported on a rectangle in $[0,1]^2$ with corners at $(z,0), (0,z),(1-z,1)$ and $(1,1-z)$. 
\end{defn}

For two examples see the left-hand side of \cref{fig:circ_perm}. Fix now $k\in\Z_{\geq 0}$.  Let $\zz^{(k)}$ denote a random variable supported on $(0,1)$ with density
$$f_{\zz^{(k)}}(t) = (2k+1){2k \choose k} (t(1-t))^k,$$
that is, $\zz^{(k)}$ is a Beta random variable with parameters $(k+1,k+1)$. Note that for $k=0$, $\zz^{(0)}$ is a uniform random variable in $(0,1)$. We will consider the random permutons $\mu^{\zz^{(k)}}$ for all $k\in\Z_{\geq 0}$.

\medskip

We introduce a second family of permutons. For convenience, we suggest to look at the first picture in the right-hand side of \cref{fig:circ_perm} despite the definition covers also the cases of the other two pictures.

\begin{defn}\label{defn:circ_perm}
	Given $\alpha\in \R_{>0}$, we denote by $\nu^{\alpha}$ the permuton corresponding to the red shape in the first picture in the right-hand side of \cref{fig:circ_perm} where:
	\begin{itemize}
		\item The bottom-left curve (the one from $(0,1/2)$ to $(1/2,0)$) of the boundary of the red shape is parametrized by 
		\begin{equation}
			\Gamma(t)\coloneqq\left(\frac{\gamma^t-1}{2(\gamma-1)},\frac{{\gamma}^{1-t}-1}{2(\gamma-1)}\right),\qquad t\in[0,1],
		\end{equation}	
		where $\gamma=\gamma(\alpha)$ solves $\frac{1}{1+\alpha}=\frac{1}{\gamma-1}\log(\gamma)$. The remaining three curves of the boundary of the red shape are determined by trivial symmetries.
		\item There is Lebesgue measure of total mass $\frac{\alpha}{\alpha+1}$ in the interior of the red shape.
		\item There is the unique measure $\rho$ supported on $\Gamma$ of total mass $\frac{1}{4(\alpha+1)}$ satisfying
		\begin{equation}
			\rho(\Gamma([0,\log_\gamma(2s(\gamma-1)+1]))=\frac s 2-\frac{\alpha}{(\alpha+1)}A_s,\quad\text{for all}\quad s\in[0,1/2],
		\end{equation}
		where $A_s$ denotes the area of the points $(x,y)$ above the curve $\Gamma$ and satisfying $x\leq t,y\leq1/2$. The remaining three boundary measures are again determined by trivial symmetries.
	\end{itemize}
\end{defn}
\begin{figure}[htbp]
	\centering
	\includegraphics[scale=0.55]{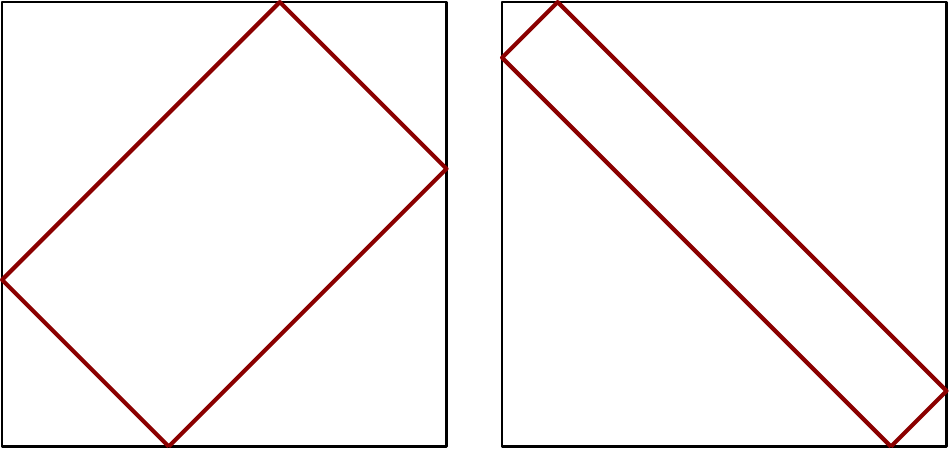}
	\hspace{1cm}
	\includegraphics[scale=.22]{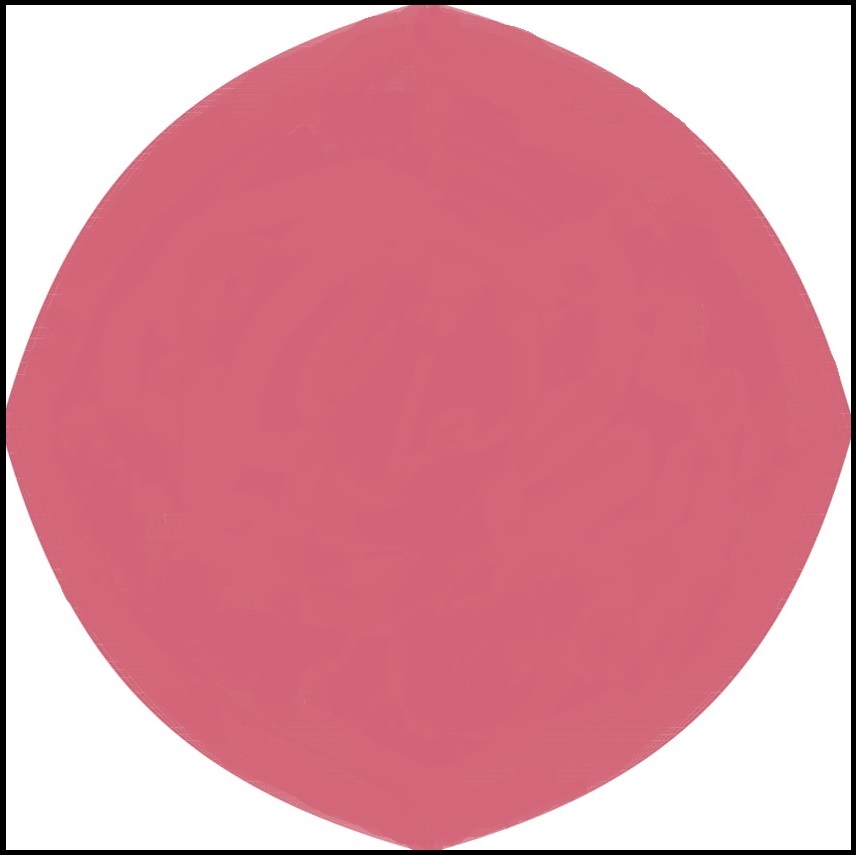}
	\includegraphics[scale=.22]{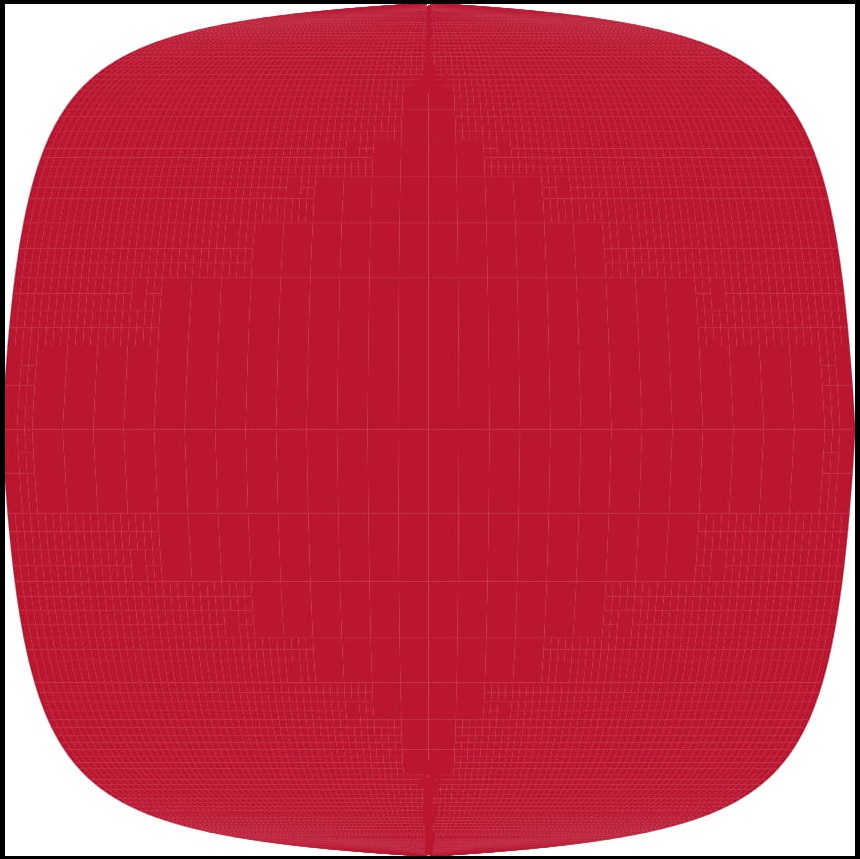}
	\includegraphics[scale=.22]{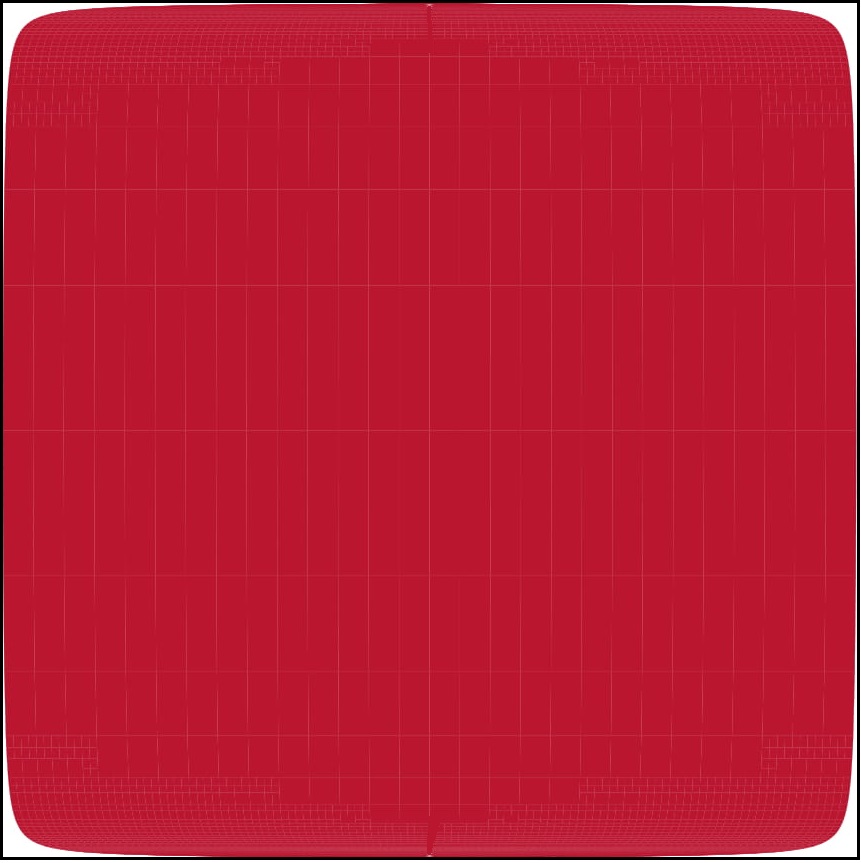}
	\caption{\textbf{Left}: The permuton $\mu^{z}$ defined in \cref{defn:square_perm} for $z=3/8$ and $z=7/8$. \textbf{Right}: The permuton $\nu^{\alpha}$ defined in  \cref{defn:circ_perm} for three values of $\alpha$ (increasing from left to right). Note that the color of the permuton is darker to recall that the internal total mass is larger.\label{fig:circ_perm}}
\end{figure}

Recall that $\asqnk$ denotes the set of almost square permutations of size $n+k$ with exactly $n$ external points and $k$ internal points.

\newpage

\begin{thm}\label{thm:phase_trans}	
	Let $\bm\sigma_n^k$ be a uniform permutation in $\asqnk$.
	\begin{enumerate}
		\item[(a)] Let $k=0$, i.e.\ $\bm\sigma_n^0$ is a uniform square permutation. Then as $n\to \infty,$
		$\mu_{\bm{\sigma}_n^0} \stackrel{d}{\longrightarrow} \mu^{\zz^{(0)}}.$
		\item[(b)] Let $k\in\Z_{> 0}$ be fixed (i.e.\ independent of $n$). Then as $n\to \infty,$
		$\mu_{\bm{\sigma}_n^k} \stackrel{d}{\longrightarrow} \mu^{\zz^{(k)}}.$
		\item[(c)] Let $k=k(n)$ and $n$ both tend to infinity with $k/n\to 0$. Then as $n\to \infty,$
		$\mu_{\bm{\sigma}_n^k} \stackrel{d}{\longrightarrow} \mu^{1/2}.$
	\end{enumerate}
\end{thm}
\begin{conj}\label{conj:phase_trans}
	Let $\bm\sigma_n^k$ be a uniform permutation in $\asqnk$.
	\begin{enumerate}
		\item[(d)] Let $k$ and $n$ both tend to infinity with $k/n\to \alpha,$ for $\alpha\in \R_{>0}$. Then as $n\to \infty,$
		$\mu_{\bm{\sigma}_n^k} \stackrel{d}{\longrightarrow} \nu^{\alpha}.$
		\item[(e)] Let $k=k(n)$ and $n$ both tend to infinity with $k/n\to \infty$. Then as $n\to \infty,$
		$\mu_{\bm{\sigma}_n^k} \stackrel{d}{\longrightarrow} \Leb_{[0,1]^2}.$
	\end{enumerate}
\end{conj}

\begin{rem}
	Item (a) in \cref{thm:phase_trans} could be included in Item (b), but we prefer to keep it separate for two main reasons: the case of square permutations is particularly relevant; and splitting the two items will help the following discussions.
\end{rem}

\begin{rem}
	In simple words, \cref{thm:phase_trans} says that:
	$$\emph{"Square permutations are typically rectangular, and almost square permutations are typically square"}$$
	This rephrasing inspired also the title of the works \cite{borga2020square,borga2019almost}.
\end{rem}

Note that \cref{thm:phase_trans} together with \cref{conj:phase_trans} describe a phase transition in the limiting shape of (almost) square permutations that interpolates between the random rectangle $\mu^{\zz^{(0)}}$ and the deterministic two-dimensional Lebesgue measure on the unit square.

\medskip

In the rest of this section we present some other results, more precisely:

\begin{itemize}
	\item In \cref{sect:fluct_square} we present (without giving a proof) some fluctuation results for the dots of the diagram of a uniform square permutation established in \cite{borga2020square}.
	\item In \cref{sect:asym_enum} we discuss the asymptotic enumeration of almost square permutations established in \cite{borga2019almost}. This is a fundamental step in the proof of Items (b) and (c) of \cref{thm:phase_trans}.
	\item In \cref{sect:321-fluct} we discuss how one can apply the techniques developed for (almost) square permutations in order to establish the asymptotic enumeration for permutations avoiding the pattern $321$ with $k$ additional internal points and to investigate some additional fluctuation results.
\end{itemize}

The proof of \cref{thm:phase_trans}, that is the heart of the works \cite{borga2020square,borga2019almost}, and some explanations about \cref{conj:phase_trans} are then given in \cref{sect:perm_conv_square}. Lastly, \cref{sect:321av} gives the proof of the results on 321-avoiding permutations with internal points.

\subsection{Fluctuations for square permutations}\label{sect:fluct_square}

We saw in Item (a) of \cref{thm:phase_trans} that the permuton limit of a sequence of uniform square permutations is a random rectangle. Here we investigate the fluctuations of the dots of the diagram of a uniform square permutation around the four edges of this rectangle. 

\bigskip

Let $\sigma_n\in Sq(n)$ and let $z_0=z_0(n)=\sigma_n^{-1}(1).$  For convenience, we assume that $z_0>\frac{n}{2}+10n^{.6}.$
We focus on the following three families of points of $\sigma_n$:
\begin{itemize}
	\item $DR=DR(n)=\RLm(\sigma_n),$ in green in Fig.~\ref{Squarepermsimulation};
	\item $DL=DL(n)=\{(i,\sigma_n(i))\in\LRm(\sigma_n):\sigma_n(i)\leq n-z_0+1\},$ in red in Fig.~\ref{Squarepermsimulation};
	\item $UR=UR(n)=\{(i,\sigma_n(i))\in\RLM(\sigma_n):i\geq z_0\},$ in blue in Fig.~\ref{Squarepermsimulation}.
\end{itemize}
Note that $(z_0,1)\in DR\cap DL$ and $(n,\sigma_n(n))\in DR\cap UR.$
For each set of points, we perform a particular rotation so that each of the following lines become the new $x$-axis:
\begin{itemize}
	\item $r^{DR}: y=x+(1-z_0),$  in green in Fig.~\ref{Squarepermsimulation};
	\item $r^{DL}: y=-x+(z_0+1),$  in red in Fig.~\ref{Squarepermsimulation};
	\item $r^{UR}: y=-x+(2n-z_0+1),$  in blue in Fig.~\ref{Squarepermsimulation}.
\end{itemize} 
More precisely, as shown in Fig.~\ref{Squarepermsimulation}, we apply a clockwise rotation of 45 degrees to the first family of points, a clockwise rotation of 135 degrees to the second family, and counter-clockwise rotation of 45 degrees to the third family. Note that the first two sequences of points (obtained form $DR$ and $DL$) starts at height zero. In order to have the same for the third sequence, we translate the $y$-coordinate of all the points in the third family by the distance of the first point from the line $r^{UR}$. We denote the three new families of points as $\mathcal{P}^{DR}$, $\mathcal{P}^{DL}$ and $\mathcal{P}^{UR}.$

\begin{figure}[htbp]
	\begin{minipage}[c]{0.55\textwidth}
		\centering
		\includegraphics[scale=0.29]{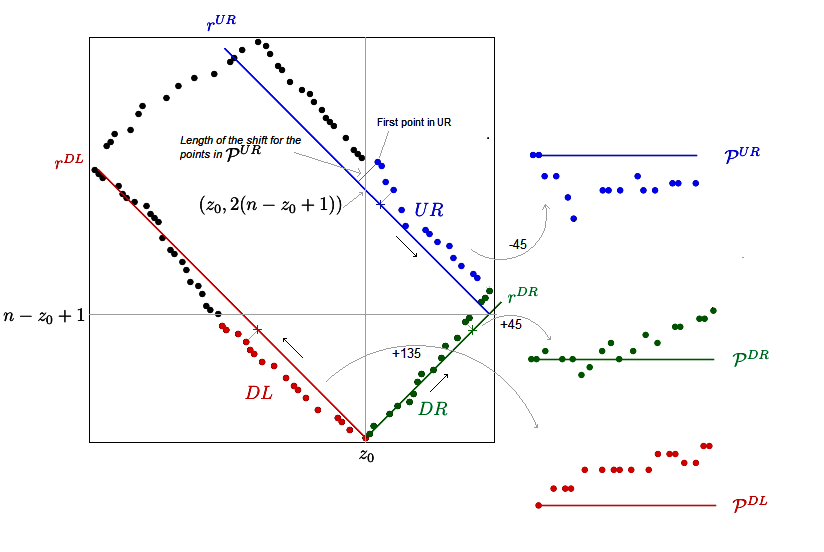}
	\end{minipage}
	\begin{minipage}[c]{0.45\textwidth}
		\caption{A square permutation $\sigma$ with the three families $DR,DL,UR$ highlighted. The dots of $DR,DL,UR$ are colored in the diagram of $\sigma$ in green, red and blue respectively. Similarly we paint the lines $r^{DR}$, $r^{DL}$ and $r^{UR}$ in green, red and blue. On the right, we draw the diagrams of the points in $\mathcal{P}^{UR}$, $\mathcal{P}^{DR}$ and $\mathcal{P}^{DL}$ obtained rotating the families of points $UR$, $DR$ and $DL$ by the indicated angle (with the additional translation for the points in $\mathcal{P}^{UR}$).\label{Squarepermsimulation}}
	\end{minipage}
\end{figure}

Given a family of points $\mathcal{P}=\{(x_i,y_i)\}_{i=0}^m$, with $x_0\leq x_1\leq x_2\leq\dots\leq x_m$, we denote by $F^{\mathcal{P}}(t),$ for $t\in[0,1],$ the linear interpolation among the points $\{(\frac{x_i}{x_m},\frac{y_i}{\sqrt{m}})\}_{i=0}^m.$
Let $\mathcal{C}([0,1],\mathbb{R})$ denote the space of continuous functions from $[0,1]$ to $\mathbb{R}$ endowed with the uniform distance. 

\begin{thm}
	\label{thm:fluctuations}
	Let $\bm{\sigma}_n$ be a uniform random square permutation of size $n$, and let $\bm{B}_1(t),\bm{B}_2(t),\bm{B}_3(t)$, and $\bm{B}_4(t)$ be four independent standard Brownian motions on the interval $[0,1].$ Fix a sequence of integers $(t_n)_n$ such that $\frac{n}{2}+10n^{.6}<t_n\leq n-n^{.9}.$ Conditioning on $\bm{z_0}=t_n,$ we have the following convergence in distribution in the space $\mathcal{C}([0,1],\mathbb{R})^3$:
	\begin{equation}
		\big(\bm{F}^{\mathcal{P}^{DR(n)}}(t),\bm{F}^{\mathcal{P}^{DL(n)}}(t),\bm{F}^{\mathcal{P}^{UR(n)}}(t)\big)_{t\in[0,1]}\stackrel{d}{\longrightarrow}\big(\bm{B}_1(t)+\bm{B}_2(t),\bm{B}_3(t)+\bm{B}_1(t),\bm{B}_4(t)+\bm{B}_2(t)\big)_{t\in[0,1]}.
	\end{equation} 
\end{thm}

The (rather technical) proof of \cref{thm:fluctuations} can be found in \cite[Section 5]{borga2020square} and is not included in this manuscript.

\begin{rem}
	Note that Theorem $\ref{thm:fluctuations}$ not only describes the scaling limit of the families of points $\mathcal{P}^{DR}$, $\mathcal{P}^{DL}$ and $\mathcal{P}^{UR}$, but also describes the dependency relations among them. Indeed, the limit $\lim_{n\to\infty}\bm{F}^{\mathcal{P}^{DL(n)}}(t)=\bm{B}_3(t)+\bm{B}_1(t)$ is independent of the limit $\lim_{n\to\infty}\bm{F}^{\mathcal{P}^{UR(n)}}(t)=\bm{B}_4(t)+\bm{B}_2(t).$ On the other hand, the limit $\lim_{n\to\infty}\bm{F}^{\mathcal{P}^{DR(n)}}(t)=\bm{B}_1(t)+\bm{B}_2(t)$ is correlated with both the limits $\lim_{n\to\infty}\bm{F}^{\mathcal{P}^{DL(n)}}(t)=\bm{B}_3(t)+\bm{B}_1(t)$ and $\lim_{n\to\infty}\bm{F}^{\mathcal{P}^{UR(n)}}(t)=\bm{B}_4(t)+\bm{B}_2(t)$. Moreover, this dependency is completely explicit.
\end{rem}

\begin{rem}
	We chose to study only the family of points $DL,$ $DR$ and $UR$ in order to simplify as much as we can the notation. Nevertheless, the result stated in Theorem \ref{thm:fluctuations} can be generalized to every possible choice of "a vertical and horizontal strip" in the diagram. In particular, in our case, the vertical strip is the one between indexes $z_0$ and $n$ and the horizontal strip is the one between values $1$ and $n-z_0+1$.
\end{rem}

\subsection{Asymptotic enumeration of almost square permutations}\label{sect:asym_enum}
A fundamental step in order to study permuton limit of almost square permutations is the fact that we are able to obtain an asymptotic enumeration of these families.
Almost square permutations were studied from an enumerative point of view in \cite{disanto2011permutations}. It was shown that, for fixed
$k > 0$, the generating function with respect to the size for $ASq(n, k)$ is algebraic of degree 2,
and this generating function was explicitly computed for $k = 1, 2, 3$. However,
computations become intractable for $k > 3$.

Using our probabilistic results on square permutations, in particular the fact that the limiting shape is a random rectangle (see below for further explanations), we were able in \cite{borga2019almost} to give the asymptotic enumeration of almost square permutations.  
\begin{thm}[{\cite[Theorem 1.1]{borga2019almost}}]\label{approx_size}
	For $k=o(\sqrt n)$, as $n\to \infty,$	
	\begin{equation}\label{approx_size_eq}
		|\asqnk| \sim \frac{k!2^{k+1}n^{2k+1}4^{n-3}}{(2k+1)!}\sim \frac{k!2^{k}n^{2k}}{(2k+1)!}|\sq(n)|.	
	\end{equation}	
\end{thm}
When $k$ grows at least as fast as $\sqrt n$ the above result fails. Nevertheless, when $k=o(n)$, we can still obtain the following weaker asymptotic expansion that determines the behavior of the exponential growth.
\begin{thm}[{\cite[Theorem 1.2]{borga2019almost}}]\label{approx_size_2}
	For $k=o(n)$, as $n\to \infty,$
	\begin{equation}
		\log\left(|\asq(n,k)|\right)=\log\left(\frac{k!}{(2k+1)!}2^{k+1}n^{2k+1}4^{n-3}\right)+o(k).
	\end{equation}
\end{thm}

We now explain how we obtained the results above.
For $\sigma\in \sq(n)$, let $\asq(\sigma,k)$ denote the set of permutations in $\asqnk$ with external points having the same relative position as $\sigma$.  For a collection of permutations $\mathcal{Q} \in \sq(n)$, let $\asq(\mathcal{Q},k) = \bigcup_{\sigma\in \mathcal{Q}} \asq(\sigma,k),$ i.e.\ the set of permutations in $\asqnk$ whose exterior lies in $\mathcal{Q}$.
Obviously,
\begin{equation}
	|\asqnk| = |\asq(\sq(n),k)| = |\asq(\rho(\Omega_n),k)|+ |\asq(\sq(n)\setminus \rho(\Omega_n),k)|,
\end{equation}
where the set $\Omega_n$ was defined in \cref{sect:inverse_projection} and the map $\rho$ in \cref{sect:rho_def}.
It is enough to investigate $|\asq(\rho(\Omega_n),k)|$, indeed it is simple to show that $|\asq(\sq(n)\setminus \rho(\Omega_n),k)|$ is negligible (see \cref{lem:irrbound} below). In order to determine the asymptotic enumerations of $\asq(\rho(\Omega_n),k)$, we used the understanding of the geometric structure of a typical square permutation in $\rho(\Omega_n)$ explained in \cref{guiding light} page \pageref{guiding light}. Specifically, thanks to the precise description of the typical shape of a large square permutation, we were able to find precise bounds on the different possible ways of adding internal points -- heuristically, these possible ways are proportional to the area of the limiting rectangle (for further details see \cref{cor:asympt} and the explanations given below). 

\medskip

In order to formalize these ideas we need some more notation.
Given a permutation $\sigma\in\sq(n)$, we say that $\sigma$ has \emph{regular projections} if the Petrov conditions (see \cref{defn:petrov} page \pageref{defn:petrov}) hold for the corresponding pair of sequences $X$ and $Y$.
For each $z_0 \in [n]$, let $\mathcal{R}(z_0)$ denote the subset of $\sq(n)$ consisting of permutations anchored at $z_0$ and having regular projections. 

In what follows, given a sequence $k=k_n=o(n)$ (resp.\ $k=k_n=o(\sqrt{n})$), we fix now some sequence $\delta_n=\delta_n(k)$ such that
\begin{equation}\label{eq:cond_dn}
	\delta_n=o(n),\quad \delta_n\geq n^{.9},\quad\text{and}\quad k=o(\delta_n)\quad(\text{resp.\ } k=o(\sqrt{\delta_n})).
\end{equation} 
This is always possible taking for instance $\delta_n=\max\{\sqrt{nk},n^{.9}\}$ (resp.\ $\delta_n=\max\{\sqrt{n}k,n^{.9}\}$).

We finally denote by $\irr$ the permutations of $\sq(n)$ that either do not have regular projections or such that the anchor $z_0$ belongs to $[n] \backslash (\delta_n,n-\delta_n)$. Note that 
\begin{equation}\label{eq:equall_bet_sets}
	\sq(n)=\irr\cup \rho(\Omega_n)\qquad\text{and}\qquad\rho(\Omega_n)=\bigcup_{z_0\in(\delta_n , n -\delta_n)}\mathcal{R}(z_0).
\end{equation}

\begin{lem}[{\cite[Corollary 5.5]{borga2019almost}}]\label{cor:asympt}
	For $k=o(\sqrt n)$ and $s\in(0,1)$,
	\begin{equation} 
		\sum_{z_0 \in (\delta_n,ns)}|\asq(\mathcal{R}(z_0),k)|\sim \frac{1}{k!}2^{k+1}n^{2k+1}4^{n-3} \int_0^s (t(1-t))^kdt.
	\end{equation}
\end{lem}

We give an informal explanation of the various factors appearing in the right-hand side of the expression above:

\begin{itemize}
	\item $2\cdot 4^{n-3}$ is the asymptotic cardinality of $\mathcal{R}(z_0)$.
	\item $2^kn^{2k}(t(1-t))^k$ comes from the following fact: The typical area of the rectangle approximating a permutation $\sigma$ in $\mathcal{R}(z_0)$ is equal to $2\cdot z_0\cdot (n-z_0)$. This area "counts" the possible ways for inserting an internal point in $\sigma$. Therefore the possible ways for inserting $k$ internal points in $\sigma$ is approximately $(2\cdot z_0 (n-z_0))^k$. Substituting $z_0=nt$ we get the factor $2^kn^{2k}(t(1-t))^k$.
	\item $k!$ counts the possible orders in which it is possible to insert $k$ internal points in a square permutation.
	\item The remaining factor $n$ comes from the sum-integral approximation.
\end{itemize}

\begin{lem}[{\cite[Eq. (13)]{borga2019almost}}]\label{lem:irrbound}
	For $k=k_n=o(n)$ or $k=k_n=o(\sqrt{n})$ and the corresponding sequence $\delta_n$, it holds that
	\begin{equation} 
		|\asq(\irr,k)|\leq \frac{1}{k!}2\delta_n4^n2^{2k}n^{2k}.
	\end{equation}
\end{lem}
These results led to the desired asymptotic enumeration, and will also give the description of the typical shape of a large almost square permutation, as we will see in \cref{sect:almost_square_fixed}.

\subsection{321-avoiding permutations with internal points}\label{sect:321-fluct}

The approach used to study (almost) square permutations has a wide range of possible applications. For example, in Section \ref{sect:321av}, we apply our techniques to establish the asymptotic enumeration for permutations avoiding the pattern $321$ with $k$ additional internal points and to investigate some additional fluctuation results. We recall that the points of a $321$-avoiding permutation can be partitioned into two increasing subsequences, one weakly above the diagonal and one strictly below the diagonal (see \cref{premres_321}). Therefore $321$-avoiding permutations are particular instances of square permutations.

Let $\avnk$ denote the set of permutations avoiding the pattern $321$ with $n$ external points and $k$ additional internal points or, equivalently, the subset of permutations $\sigma$ in $\asqnk$ where the pattern induced by the external points of $\sigma$ is in $\avn(321)$. For fixed $k$, we showed in \cite[Theorem 2.2]{borga2019almost} that the generating function of these permutations is algebraic of degree 2, and more precisely rational in the Catalan generating series.  Explicit expressions are derived for $k=1,2,3$. For $k>3$, the computations to determine the generating function become intractable. Nevertheless, using our new probabilistic approach, we are able to compute the first order approximation of the enumeration.  

Let $c_n$ denote the $n$-th Catalan number, $c_n=\frac{1}{n+1}{2n\choose n}$, so that $|\avn(321)| = c_n$.    

\begin{thm}\label{thin red line}
	Fix $k>0$.  Then as $n\to\infty,$
	$$|\avnk|\sim \frac{(2n)^{3k/2}}{k!}\cdot c_n\cdot  \mathbb{E}\left[\left(\int_0^1\bm e_t dt\right)^k\right],$$
	where $\bm e_t $ denotes the standard Brownian excursion on the interval $[0,1]$.	
\end{thm}

The proof of \cref{thin red line} is given in \cref{sect:321av_1}. The evaluation on the right-hand side of the $k$-th moment of the Brownian excursion area is derived in Section 2 of \cite{janson2007brownian} where the author shows that
$$\mathbb{E}\left[\left(\int_0^1\bm e_t dt\right)^k\right]=(36\sqrt{2})^{-k}\frac{2\sqrt{\pi}}{\Gamma((3k-1)/2)}\xi_k,$$ where $\xi_k$ satisfies the recurrence
\begin{equation}\label{svante constant}
	\xi_r = \frac{12r}{6r-1}\frac{\Gamma(3r+1/2)}{\Gamma(r+1/2)}- \sum_{j=1}^{r-1}{r \choose j} \frac{\Gamma(3j+1/2)}{\Gamma(j+1/2)}\xi_{r-j}, \qquad r\geq 1.
\end{equation}

The final result of this chapter is a generalization of Theorem 1.2 in \cite{hoffman2017pattern} where the authors proved that the points of a uniform random 321-avoiding permutation concentrate on the diagonal and the fluctuations of these points converge in distribution to a Brownian excursion. We are able to extend this result to uniform random permutations in $\avnk$ for any fixed $k\in\Z_{>0}$.
We define for a permutation $\tau^k_n\in \avnk$ (with the convention $\tau^k_n(0)=0$) and $t\in[0,1]$,
$$F_{\tau^k_n}(t) \coloneqq \frac{1}{\sqrt{2(n+k)}}\big |\tau^k_n(s(t)) - s(t) \big|,$$
where $s(t)=\max\left\{m\leq \lfloor (n+k)t \rfloor|\tau^k_n(m)\text{ is an external point}\right\}$. In words, the function $F_{\tau^k_n}(t)$ is interpolating only the external points of $\tau^k_n$ (once reflected above the diagonal), forgetting the internal ones. We also introduce the following notion of biased Brownian excursion.

\begin{defn} \label{def:kbiasedex}
	Let $k\in\Z_{>0}$. The \emph{$k$-biased Brownian excursion} $(\bm{e}^k_t)_{t\in[0,1]}$ is a random variable in the space of continuous functions $\mathcal C([0,1],\mathbb{R})$ with the following law: for every continuous bounded functional $G:\mathcal C([0,1],\mathbb{R})\to\R,$
	$$\E\left[G\left(\bm{e}^k_t\right)\right]=\mathbb{E}\left[\left(\int_0^1\bm e_t dt\right)^k\right]^{-1}\E\left[G(\bm{e}_t)\cdot\left( \int_0^1 \bm{e}_t dt\right) ^k\right],$$
	where $\bm{e}_t$ is the standard Brownian excursion on [0,1].
\end{defn}

\begin{thm}\label{thm:fluctuations_321}
	Fix $k>0$. Let $\bm{\tau}^k_n$ be a uniform random permutation in $\avnk$.  Then
	$$\left(F_{\bm{\tau}^k_n}(t)\right)_{t\in [0,1]} \stackrel{d}{\longrightarrow} \left(\bm{e}^k_t\right)_{t\in [0,1]},$$
	where the convergence holds in the space of right-continuous functions $D([0,1],\mathbb{R})$.
\end{thm}

The proof of \cref{thm:fluctuations_321} can be found in \cref{sect:321_fluctu}.

\section{Permuton convergence}\label{sect:perm_conv_square}

\subsection{Square permutations}\label{sect:permuton_square}

In this section we consider random square permutations proving Item (a) of \cref{thm:phase_trans}.
There are a few main steps in establishing convergence in distribution in the permuton topology for uniform elements of $Sq(n).$  We already showed in \cref{square_is_rect} page \pageref{square_is_rect} that it suffices to consider only permutations $\sigma_n$ in $\rho(\Omega_n)$ (see the definition in \cref{sect:inverse_projection}).  Then we show in Lemma \ref{permuton_bounds} a uniform bound for the distance between the permuton $\mu_{\sigma_n}$ and $\mu^{z_n}$ with $z_n = \sigma_n^{-1}(1)/n$.  Finally we prove the main result.   

First we clarify some aspects of our candidate limiting permuton $\mu^{\zz^{(0)}}$. Let $z$ be a point in $[0,1]$.  Let $L_1$ and $L_4$ denote the line segments with slope $-1$ connecting $(0,z)$ to $(z,0)$ and $(1-z,1)$ to $(1,1-z)$, respectively.  Similarly let $L_2$ and $L_3$ denote the line segments with slope $1$ connecting $(0,z)$ to $(1-z,1)$ and $(z,0)$ to $(1,1-z)$, respectively.  The union of $L_1$, $L_2$, $L_3$ and $L_4$ forms a rectangle in $[0,1]^2.$  

For each of the line segments $L_i$ ($i=1,2,3$, or $4$) we will denote by $\mu^z_i$ a particular rescaled one dimensional Lebesgue measure: Let $\nu$ be the Lebesgue measure on $[0,1]$.  Let $S$ be a Borel measurable set on $[0,1]^2$.  For each $i$, let $S_i = S\cap L_i$.  Finally let $\pi_x(S_i)$ be the projection of $S_i$ onto the $x$-axis and $\pi_y(S_i)$ the projection onto the $y$-axis.  As each line has slope $1$ or $-1$, the measures of the projections satisfy $\nu(\pi_x(S_i)) = \nu(\pi_y(S_i)).$  For each $i=1,2,3,4$, define $\mu^z_i(S) := \frac{1}{2} \nu( \pi_x( S_i ) ) = \frac12 \nu( \pi_y(S_i)).$ Note that $\mu^z = \mu^z_1+\mu^z_2+\mu^z_3+\mu^z_4.$

\begin{lem}\label{permuton_bounds}
	For every $\sigma_n \in \rho(\Omega_n)$, let $z_n =\sigma_n^{-1}(1)/n$. Then for $n$ large enough
	$$\sup_{\sigma_n \in \rho(\Omega_n)} d_{\square}(\mu_{\sigma_n},\mu^{z_n}) < 400n^{-.4},$$
	where $d_{\square}$ denotes the distance on the set of permutons defined in \cref{eq:perm_dist} page \pageref{eq:perm_dist}.
\end{lem}

\begin{proof}
	Fix $\sigma_n \in \rho(\Omega_n)$ and $R= (a,b)\times(c,d) \subset [0,1]^2$.  The permutation $\sigma_n$ consists of four disjoint sets of points $\Lambda_i$ for $i=1,2,3$, and $4$ (see the discussion before \cref{guiding light} page \pageref{guiding light}).  For each of these sets of points we define the measure $\lambda_i$ on $[0,1]^2$ as 
	$$\lambda_i := \frac{1}{n}\sum_{(i,j) \in \Lambda_i} \Leb\left([(i-1)/n,i/n)\times[(j-1)/n,j/n)\right).$$
	
	Noting that $\mu_{\sigma_n}=\lambda_1+\lambda_2+\lambda_3+\lambda_4$ we have the bound $|\mu_{\sigma_n}(R) - \mu^{z_n}(R)| \leq \sum_{i=1}^4 |\lambda_i(R)- \mu^{z_n}_i(R)|.$
	We will show explicitly that $|\lambda_1(R) - \mu^{z_n}_1(R)| < 100n^{-.4}$.  Similar arguments show the same bound for $i=2,3,4.$    
	
	Recall that $L_1$ is the line connecting $(0,z_n)$ and $(z_n,0)$.  Let $\ell$ denote the line segment given by $R\cap L_1$ (that we assume non-empty) with end points at $(x_1,y_1)$ and $(x_2,y_2)$ where $x_1\leq  x_2$ and $y_1 \geq y_2.$  These endpoints satisfy $x_1+y_1 = x_2 + y_2 = z_n$.  By this construction we have $\mu^{z_n}_1(R) = \frac{1}{2}(x_2-x_1).$ 
	
	Let $(s_1,t_1)$ be the leftmost point in $\Lambda_1 \cap nR$ and $(s_2,t_2)$ the rightmost point.  The total number of points in $\Lambda_1\cap nR$ is given by $\ctd(s_2) - \ctd(s_1) +1$. Therefore $\lambda_1(R)=\tfrac{1}{n}(\ctd(s_2) - \ctd(s_1) +1)\pm\frac{\varepsilon}{n},$ with $\varepsilon\leq2$ (the error term comes form the first and last area measures) and so $$|\lambda_1(R)-\mu^{z_n}_1(R)|<\left|\frac{1}{n}(\ctd(s_1)-1) - \frac{1}{2} x_1 \right|+\left|\frac{1}{n}\ctd(s_2) - \frac{1}{2} x_2 \right|+\frac{\varepsilon}{n}.$$
	By \cref{guiding light} page \pageref{guiding light}, the points of $\Lambda_1$ must lie between the lines $nL_1^-$ and $nL_1^+$ given by the equations  $x+y = nz_n \pm 20n^{.6}$ (\emph{cf.}\ Fig.~\ref{schema_for_proof}).
	
	Suppose $L_1$ exits $(a,b)\times(c,d)$ from the top so that $y_1 = d$.  Then points in $\Lambda_1$ with first coordinate in the interval $[nx_1-40n^{.6}, nx_1-20n^{.6}]$ must lie above the line $y = nd$.  Similarly points in $\Lambda_1$ in the interval $[nx_1+20n^{.6},nx_1 + 40n^{.6}]$ must lie below the line $y=nd$.  By \cref{petrov_often} page \pageref{petrov_often} there is at least one point in $\Lambda_1$ with $x$-coordinate in each of these intervals.  Thus the leftmost point $(s_1,t_1)$ must have $s_1$ in the interval $[nx_1 - 40n^{.6}, nx_1 + 40n^{.6}].$  This combined with the Petrov conditions shows that 
		$$\left|\ctd(s_1) - \frac{1}{2}nx_1\right| < \left|\ctd(s_1) - \frac{1}{2}s_1\right| + \left|\frac{1}{2}(s_1 -  nx_1)\right|< n^{.6} + 40n^{.6}$$
	and thus $\left |\frac{1}{n}(\ctd(s_1)-1) - \frac{1}{2} x_1 \right | < 50n^{-.4}$.

	\begin{figure}[htbp]
		\begin{minipage}[c]{0.59\textwidth}
			\centering
			\includegraphics[scale=.6]{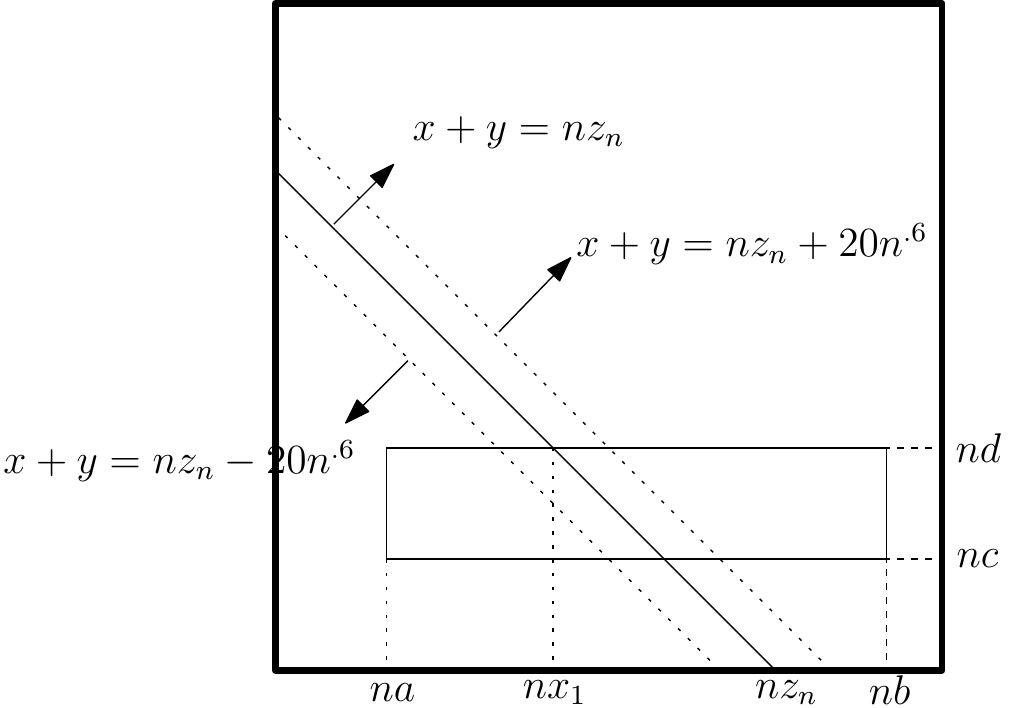}
		\end{minipage}
		\begin{minipage}[c]{0.4\textwidth}
			\caption{A diagram for the proof of Lemma \ref{permuton_bounds}.\label{schema_for_proof}} 
		\end{minipage}
	\end{figure}

	If $L_1$ exits $nR$ on the left so that $x_1 = a$, then a similar argument leads to the same conclusion. Likewise we can show that $\left |\frac{1}{n}(\ctd(s_2)) - \frac{1}{2} x_2 \right | < 50n^{-.4}$ and thus 
	$\left |\lambda_1 (R) -  \mu_1^{z_n}(R)\right| < 100n^{-.4}.$ 
	Similarly, for each $i=2,3,4$, $|\lambda_i(R) -  \mu^{z_n}_i(R)| < 100n^{-.4}$ and thus
	$\left |\mu_{\sigma_n}(R) - \mu^{z_n}(R)\right| < 400n^{-.4}.$
	This bound is uniform over all $\sigma_n \in \rho(\Omega_n)$ and $R\in [0,1]^2$ and so concludes the proof.  
\end{proof}

We now prove that a uniform square permutation $\bm{\sigma}_n$ converges to the permuton $\mu^{\zz}$, for $\zz=\zz^{(0)}$ uniformly random on $(0,1)$, that is Item (a) in \cref{thm:phase_trans}.

\begin{proof}[Proof of Item (a) in \cref{thm:phase_trans}]
	By \cref{square_is_rect} page \pageref{square_is_rect} it suffices to only consider permutations chosen uniformly from $\rho(\Omega_n)$ when showing the distributional limit of $\mu_{\bm{\sigma}_n}.$  
	
	Let ${\zz}_n = \bm{\sigma}_n^{-1}(1)/n$.  Since $\bm{\sigma}_n$ is uniform in $\rho(\Omega_n)$ then $\zz_n$ is uniform in $(n^{-.1},1-n^{-.1})$ and so converges in distribution to $\zz$. The map $z\to\mu^z$ is continuous as a function from $(0,1)$ to $\mathcal{M}$, and thus $\mu^{\zz_n}$ converges in distribution to $\mu^{\zz}.$
	By Lemma \ref{permuton_bounds}, we also have that $d_{\square}(\mu_{\bm{\sigma}_n},\mu^{\zz_n})$ converges almost surely to zero. Therefore, combining these results,  we can conclude that $\mu_{\bm{\sigma}_n}$ converges in distribution to $\mu^{\zz}.$
\end{proof}

\subsection{Almost square permutations with few internal points}\label{sect:almost_square_fixed}

We turn to the proof of Item (b) in \cref{thm:phase_trans}.
Before doing that we need two technical lemmas. The first one follows using arguments similar to the ones used for \cref{permuton_bounds}.

\begin{lem}[{\cite[Lemma 6.2]{borga2019almost}}]\label{lem:nochanges2}
	Let $\sigma$ be a permutation of size $n$ and $\sigma'$ be a permutation obtained from $\sigma$ by adding a point (not necessarily internal) to the diagram of $\sigma$. Then
	$d_{\square}(\mu_{\sigma},\mu_{\sigma'})\leq\frac 6 n.$
\end{lem}

In \cref{permuton_bounds} we showed that for $\sigma_n\in \rho(\Omega_n)\stackrel{\eqref{eq:equall_bet_sets}}{=}Sq(n)\setminus\irr$ the permutons $\mu_{\sigma_n}$ and $\mu^{z_n}$ with $z_n = \sigma_n^{-1}(1)/n$ have distance $d_{\square}(\mu_{\sigma_n},\mu^{z_n})$ that tends to zero as $n$ tends to infinity, uniformly over all choices of $\sigma_n$. We prove here that the same result holds for permutations in $ASq(Sq(n)\setminus\irr,k)$ whenever $k=o(n)$.

\begin{lem} \label{lem:nochanges}
	Let $k=o(n)$. For every $\sigma_n \in ASq(Sq(n)\setminus\irr,k)$, let $z_n =\sigma_n^{-1}(1)/n$. The following limit holds
	\begin{equation}
		\sup_{\sigma_n\in ASq(Sq(n)\setminus\irr,k)}d_{\square}(\mu_{\sigma_n},\mu^{z_n})\to 0,
	\end{equation}
	where $d_{\square}$ denotes the distance on the set of permutons defined in \cref{eq:perm_dist} page \pageref{eq:perm_dist}.
\end{lem}

\begin{proof} We denote by $\ext(\sigma_n)$ the permutation obtained from $\sigma_n$ removing all its internal points.
	We have the following bound for every $\sigma_n\in ASq(Sq(n)\setminus\irr,k)$
	$$d_{\square}(\mu_{\sigma_n},\mu^{z_n})\leq d_{\square}(\mu_{\sigma_n},\mu_{\ext(\sigma_n)})+d_{\square}(\mu_{\ext(\sigma_n)},\mu^{z_n})$$
	that translates into
	$$\sup_{\sigma_n\in ASq(Sq(n)\setminus\irr,k)}d_{\square}(\mu_{\sigma_n},\mu^{z_n})\leq \sup_{\sigma_n\in ASq(Sq(n)\setminus\irr,k)}d_{\square}(\mu_{\sigma_n},\mu_{\ext(\sigma_n)})+\sup_{\sigma_n\in Sq(n)\setminus\irr}d_{\square}(\mu_{\sigma_n},\mu^{z_n}).$$
	The second term in the right-hand side of the above equation tends to zero thanks to the aforementioned \cref{permuton_bounds}. The first term tends to zero because the addition of $k=o(n)$ internal points cannot modify the permuton limit of a sequence of permutations (see \cref{lem:nochanges2}).
\end{proof}

We can now prove the main result of this section, that is, if $k>0$ is fixed and $\bm\sigma_n$ is uniform in $\asqnk$, then $\mu_{\bm{\sigma}_n} \stackrel{d}{\longrightarrow} \mu^{\zz^{(k)}}$ as $n\to \infty$.

\begin{proof}[Proof of Item (b) in \cref{thm:phase_trans}]
	With the asymptotic formula for the cardinality of $\asqnk$ (obtained in Theorem \ref{approx_size}) and \cref{cor:asympt}  we can determine the distribution of the value of $\bm \sigma^{-1}_n(1)$, for a uniform permutation $\bm{\sigma}_n$ in $\asqnk$ when $k$ is fixed. Specifically,
	\begin{multline}\label{coral}
		\P(\bm{\sigma}_n^{-1}(1) \leq ns )=\frac{\#\{\sigma\in\asqnk:{\sigma}^{-1}(1) \leq ns \}}{|\asqnk|}\\
		\sim \frac{ \frac{1}{k!}2^{k+1}n^{2k+1}4^{n-3}\int_0^s (t(1-t))^kdt}{\frac{k!}{(2k+1)!}2^{k+1}n^{2k+1}4^{n-3}} =  (2k+1){2k \choose k}\int_{0}^s(t(1-t))^kdt,
	\end{multline}
	where we used the fact that the contribution to the numerator of permutations in $\irr$ is negligible w.r.t. the cardinality of $\asqnk$ (see \cref{lem:irrbound}). Therefore $\zz_n=\frac{\bm{\sigma}_n^{-1}(1)}{n}\stackrel{d}{\longrightarrow}\zz^{(k)}.$ 
	
	The map $z\to\mu^z$ is continuous as a function from $(0,1)$ to $\mathcal{M}$, and thus $\mu^{\zz_n}$ converges in distribution to $\mu^{\zz^{(k)}}.$ By Lemma \ref{lem:nochanges}, and again the fact that $|\irr|$ is negligible w.r.t. $|\asqnk|$, we also have that $d_{\square}(\mu_{\bm{\sigma}_n},\mu^{\zz_n})$ converges almost surely to zero. Therefore, combining these results,  we can conclude that $\mu_{\bm{\sigma}_n}$ converges in distribution to $\mu^{\zz^{(k)}}.$
\end{proof}

The proof of Item (c) in \cref{thm:phase_trans} follows -- in the same manner as Item (b) -- from the following result, that is a consequence of \cref{approx_size_2}.

\begin{lem}[{\cite[Lemma 6.4]{borga2019almost}}]\label{lemm:conc_result}
	Let $k=o(n)$ and assume that $k\to\infty$. Then for $\bm{\sigma}_n$ uniform in $\asqnk$ it holds 
	$$\frac{\bm{\sigma}_n^{-1}(1)}{n}\stackrel{d}{\longrightarrow}1/2.$$
\end{lem}

\subsection{Almost square permutations with many internal points}\label{sect:strat_conj}

We now give an informal and partial explanation on why \cref{conj:phase_trans} should be true. The strategy illustrated in this section was found after some discussions with A.\ Sportiello and E.\ Slivken.

Instead of studying almost square permutations, we restrict our analysis to permutations with a fixed proportion of left-to-right minima (and we think at points that are not left-to-right minima as internal points).

Let $\bm\sigma_n^k$ be a uniform permutation of size $m=n+k$, with exactly $n$ left-to right minima (and $k$ internal points). We explain here why the permuton limit of $\bm\sigma_n^k$ should be the following.

\begin{defn}\label{defn:circ_perm_corner}
	Given $\alpha\in \R_{>0}$, we denote by $\nu^{\alpha}_*$ the permuton corresponding to the red shape in \cref{fig:circ_perm_corner} where:
	\begin{itemize}
		\item The bottom-left boundary of the red shape is parametrized by the curve 
		\begin{equation}
			\Gamma(t)=\left(\frac{\gamma^t-1}{\gamma-1},\frac{{\gamma}^{1-t}-1}{\gamma-1}\right),\qquad t\in[0,1],
		\end{equation}	
		where $\gamma=\gamma(\alpha)$ solves $\frac{1}{1+\alpha}=\frac{1}{\gamma-1}\log(\gamma)$. 
		\item There is Lebesgue measure of total mass $\frac{\alpha}{\alpha+1}$ in the interior of the red shape.
		\item There is the unique measure $\rho$ supported on $\Gamma$ of total mass $\frac{1}{\alpha+1}$ satisfying
		\begin{equation}
			\rho(\Gamma([0,\log_\gamma(t(\gamma-1)+1)]))=t-\frac{\alpha}{\alpha+1}A_t,\quad\text{for all}\quad t\in[0,1],
		\end{equation}
	    where $A_t$ denotes the area of the points $(x,y)$ above the curve $\Gamma$ and with $x\leq t$.
	\end{itemize}
\end{defn}

\begin{figure}[htbp]
	\begin{minipage}[c]{0.5\textwidth}
		\centering
		\includegraphics[scale=0.4]{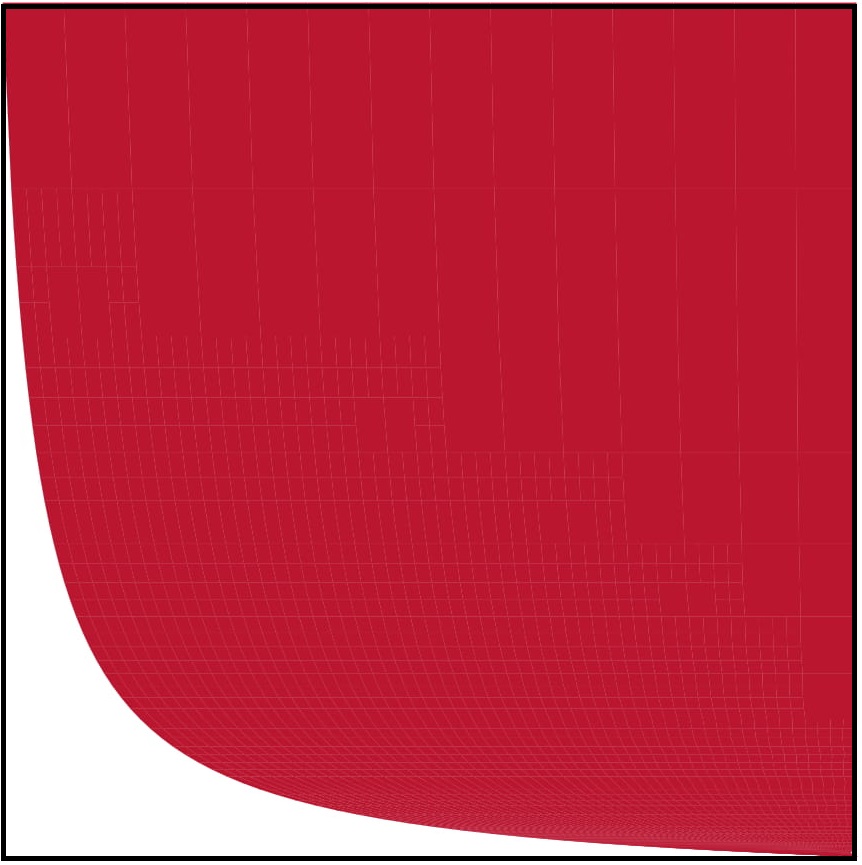}
	\end{minipage}
	\begin{minipage}[c]{0.49\textwidth}
		\caption{A schema for \cref{defn:circ_perm_corner}.\label{fig:circ_perm_corner}}
	\end{minipage}
\end{figure}

Note that the permuton $\nu^{\alpha}_*$ is equal to the the bottom-left corner of the permuton $\nu^{\alpha}$ defined in \cref{defn:circ_perm}. \cref{conj:phase_trans} is motivated by the following simplified version.

\begin{conj}\label{conj:hfbwefniwenfew}
		Let $\bm\sigma_n^k$ be a uniform permutation of size $m=n+k$, with exactly $n$ left-to-right minima, and assume that $k$ and $n$ both tend to infinity with $k/n\to \alpha,$ for some $\alpha\in \R_{>0}$. Then as $n\to \infty,$ 
		$$\mu_{\bm{\sigma}_n^k} \stackrel{d}{\longrightarrow} \nu^{\alpha}_*.$$
\end{conj}

We set $p^h_{n,k}(x,y)\coloneqq\P(\text{The $h$-\emph{th} left-to-right minimum of $\bm\sigma_n^k$ is at position } (x,y))$.

\begin{lem}\label{lem:decomp_counting} 
	It holds that 
	\begin{equation}\label{eq;wfevbiuywebfouw}
		p^h_{n,k}(x,y)=\frac{S_{x-1,h-1}S_{y-1,n-h}}{S_{m,n}}\binom{m-x}{m-x-y+1}\binom{m-y}{m-x-y+1}(m-x-y+1)!\;,
	\end{equation}
	where $S_{p,q}$ denotes the Stirling numbers of the first kind.
\end{lem}

\begin{proof}
	Recall that the Stirling numbers of the first kind $S_{p,q}$ counts the number of permutations of size $p$ with exactly $q$ left-to-right minima. In order to prove \cref{eq;wfevbiuywebfouw}, it is enough to decompose the diagram of $\bm\sigma_n^k$ into four boxes w.r.t.\ the $h$-\emph{th} left-to-right minimum (see \cref{fig:schema_decomp}) and note that:
	\begin{itemize}
		\item $S_{x-1,h-1}$ counts the possible configurations in the top-left box;
		\item $S_{y-1,n-h}$ counts the possible configurations in the bottom-right box;
		\item $\binom{m-x}{m-x-y+1}\binom{m-y}{m-x-y+1}$ counts the possible choices of rows and columns with a point in the top-right box;
		\item $(m-x-y+1)!$ counts the possible permutations in the top-right box.
	\end{itemize}	
	This is enough to conclude the proof.
\end{proof}

\begin{figure}[htbp]
	\begin{minipage}[c]{0.57\textwidth}
		\centering
		\includegraphics[scale=0.8]{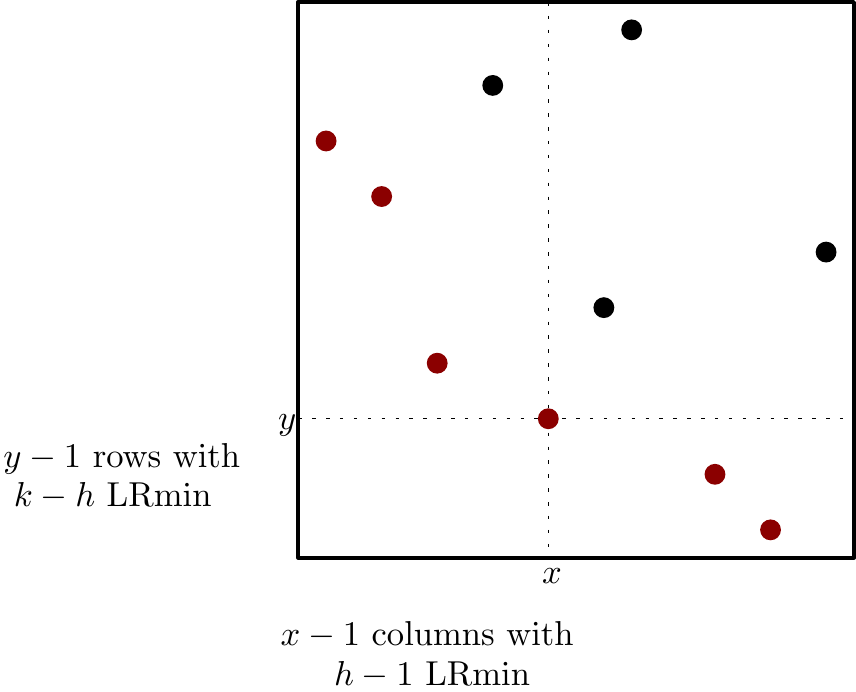}
	\end{minipage}
	\begin{minipage}[c]{0.42\textwidth}
		\caption{A schema for the proof of \cref{lem:decomp_counting}. The $n=6$ left-to right minima are highlighted in red. In this example $k=4, h=4, x=5,$ and $y=3$. \label{fig:schema_decomp}}
	\end{minipage}
\end{figure}

From \cite{MR1223774}, we have that when $n/m\to \ell$ and $a,b=O(1)$ then as $m\to\infty$,
\begin{equation}\label{eq:fweibwifbp;eiwq}
	\frac{S_{m+a,m+b}}{S_{m,n}}\sim m^{a-b}\frac{(1+Z)^a}{Z^b},
\end{equation}
where $Z=Z(\ell)$ is a fixed parameter that satisfies $\ell=Z\log(1+1/Z)$.

Heuristically, we expect that the distribution $p^h_{n,k}(x,y)$ is log-concave and that it concentrates around it maximum. The latter should coincide with the minimum of $-\log(p^h_{n,k}(x,y))$, i.e.\ it should be at $(x,y)$ satisfying $\nabla\log(p^h_{n,k}(x,y))=0$. Writing down the discrete partial derivatives we want to solve for $(x,y)$ the following system of equations:
\begin{equation}\label{eq:wejvbfewbfoewnbif}
	\begin{cases}
		\log(p^h_{n,k}(x+1,y))-\log(p^h_{n,k}(x,y))=0\\
		\log(p^h_{n,k}(x,y+1))-\log(p^h_{n,k}(x,y))=0
	\end{cases}.
\end{equation}
Setting $\beta\coloneqq\frac{1}{1+\alpha}$, we now want to solve the system above for 
\begin{equation}
	n\approx\beta m\qquad \text{and} \qquad h\approx\theta m,\quad \text{for some} \quad \theta\in(0,1).
\end{equation}
We expect that the solutions are of order
\begin{equation}
	x\approx\zeta m\qquad \text{and} \qquad y\approx\eta m,\quad \text{for some} \quad \zeta,\eta\in(0,1).
\end{equation}

Note that from \cref{lem:decomp_counting}  we have that
\begin{equation}
	\frac{p^h_{n,k}(x,y+1)}{p^h_{n,k}(x,y)}=\frac{S_{y,n-h}}{S_{y-1,n-h}}\frac{(m-y-1)!(y-1)!(m-x-y+1)!}{(m-y)!y!(m-x-y)!}\stackrel{\eqref{eq:fweibwifbp;eiwq}}{\approx}\frac{1-\eta-\zeta}{1-\eta}\left(1+Z(\frac{\beta-\theta}{\eta})\right),
\end{equation}
and similarly
\begin{equation}
	\frac{p^h_{n,k}(x+1,y)}{p^h_{n,k}(x,y)}\approx\frac{1-\eta-\zeta}{1-\zeta}\left(1+Z(\frac{\theta}{\zeta})\right).
\end{equation}
Therefore \cref{eq:wejvbfewbfoewnbif} rewrites as
\begin{equation}
	\begin{cases}
		\frac{1-\eta-\zeta}{1-\zeta}\left(1+Z(\frac{\theta}{\zeta})\right)=1\\
		\frac{1-\eta-\zeta}{1-\eta}\left(1+Z(\frac{\beta-\theta}{\eta})\right)=1
	\end{cases},
\end{equation}
and so
\begin{equation}
	\begin{cases}
		Z(\frac{\theta}{\zeta})=\frac{\eta}{1-\eta-\zeta}\\
		Z(\frac{\beta-\theta}{\eta})=\frac{\zeta}{1-\eta-\zeta}
	\end{cases}.
\end{equation}
Recalling that $Z=Z(\ell)$ solves $\ell=Z\log(1+1/Z)$, we obtain that
\begin{equation}
	\begin{cases}
		\frac{\theta}{\zeta}=\frac{\eta}{1-\eta-\zeta}\log\left(\frac{1-\zeta}{\eta}\right)\\
		\frac{\beta-\theta}{\eta}=\frac{\zeta}{1-\eta-\zeta}\log\left(\frac{1-\eta}{\zeta}\right)
	\end{cases}.
\end{equation}
Hence
\begin{equation}
	\beta=\frac{\eta\zeta}{1-\eta-\zeta}\log\left(\frac{(1-\eta)(1-\zeta)}{\eta\zeta}\right)=\frac{\eta\zeta}{(1-\eta)(1-\zeta)-\eta\zeta}\log\left(\frac{(1-\eta)(1-\zeta)}{\eta\zeta}\right).
\end{equation}
Setting $X\coloneqq\frac{\zeta}{1-\zeta}$ and $Y\coloneqq\frac{\eta}{1-\eta}$, we get $\beta=\frac{1}{XY-1}\log(XY)=\frac{1}{\gamma-1}\log(\gamma)$ for $\gamma\coloneqq XY$. Summarizing we have that
\begin{equation}
	(1-\eta)(1-\zeta)=\gamma \eta\zeta,\quad \text{where $\gamma$ solves} \quad  \beta=\frac{1}{\gamma-1}\log(\gamma).
\end{equation}
Finally, setting $t\coloneqq\frac{\theta}{\beta}=\log\left(\frac{1-\zeta}{\eta}\right)\frac{1}{\log(\gamma)}$, we have that
\begin{equation}
	\begin{cases}
		t=\frac{\log\left(\frac{1-\zeta}{\eta}\right)}{\log(\gamma)}\\
		(1-\eta)(1-\zeta)=\gamma \eta\zeta
	\end{cases},
\end{equation}
and we can conclude that
\begin{equation}
	\begin{cases}
		\zeta(t)=\frac{\gamma^t-1}{\gamma-1}\\
		\eta(t)=\frac{\gamma^{1-t}-1}{\gamma-1}\\
	\end{cases}, \quad\text{where $\gamma$ solves}\quad \beta=\frac{1}{\gamma-1}\log(\gamma)\quad \text{and}\quad t=\frac{\theta}{\beta}.
\end{equation}
In simple words, the $h\approx \theta m$-$th$ record "concentrates" at position $(\zeta(t),\eta(t))$ where $t=\frac{\theta}{\beta}$. 
Recalling that $\beta=\frac{1}{1+\alpha}$ we have that the curve $(\zeta(t),\eta(t))$ coincides with the curve described in \cref{defn:circ_perm_corner}, giving an informal proof of \cref{conj:hfbwefniwenfew}. With similar computations one can also show that the distribution of the internal points of $\bm\sigma_n^k$ is asymptotically uniform. Finally the expression for the boundary measure $\rho$ can be obtained using that a permuton has uniform marginals.

\section{321-avoiding permutations with internal points}
\label{sect:321av}

\subsection{Asymptotic enumeration}\label{sect:321av_1}

Permutations in $\avn(321)$ are in bijection with Dyck paths of size $2n$.  A Dyck path of size $2n$ is a path with two types of steps: $(1,1)$ or $(1,-1)$, that is conditioned to start at $(0,0),$ end at $(2n,0)$, and remain non-negative in between.  There are many possible bijections to choose from between 321-avoiding permutations and Dyck paths.  One particular bijection comes from \cite{MR1241505}, which we refer to as the Billey--Jockusch--Stanley (or BJS) bijection.  For a Dyck path, $\gamma_n,$ of size $2n$, we let $\tau_n = \tau_{\gamma_n}$ denote the corresponding permutation in $\avn(321)$ under the BJS-bijection.  In the other direction, for a permutation $\tau_n \in \avn(321)$ we let $\gamma_n = \gamma_{\tau_n}$ denote the corresponding Dyck path under the inverse bijection.  
This bijection is used in \cite{hoffman2017pattern} to show that the points of a 321-avoiding permutation converge to the Brownian excursion when properly scaled. 

Specifically, extending the definition of the permutation $\tau_n$ so that $\tau_n(0)=0$ and for $t\in [0,1]$, let 
$$F_{\tau_n}(t) \coloneqq \frac{1}{\sqrt{2n}}\big |\tau_n(\lfloor nt \rfloor) - \lfloor nt \rfloor \big|.$$  

\begin{thm}[Theorem 1.2 in \cite{hoffman2017pattern}]\label{chacha}
	Let $\bm{\tau}_n$ be a uniform random permutation in $\avn(321)$.  Then
	$$\left(F_{\bm{\tau}_n}(t)\right)_{t\in [0,1]} \stackrel{d}{\longrightarrow} \left(\bm{e}_t\right)_{t\in [0,1]},$$
	where $\bm{e}_t$ is the Brownian excursion on $[0,1]$ and the convergence holds in the space of right-continuous functions $D([0,1],\mathbb{R})$.
\end{thm}

The main step in the proof of Theorem \ref{chacha} is to show that the function $F_{\bm{\tau}_n}(t)$ is often close to the corresponding scaled Dyck path $\gamma_{\bm{\tau}_n}$, which converges in distribution to the Brownian excursion \cite{Kaigh1976excursion}.  The proof uses an alternative version of the Petrov conditions stated in terms of Dyck paths. We denote the Petrov conditions for Dyck paths with $PC'$ and the Petrov conditions for permutations used in this manuscript with $PC$ (see \cref{defn:petrov}). We say $\tau_n\in \avn(321)$ satisfies $PC'$ if, using the BJS bijection, the corresponding Dyck path, $\gamma_n$, satisfies $PC'$ (we remark that this translation of $PC'$ on permutations gives a slightly modified version of $PC$.).

In what follow we say that $\tau_n\in \avn(321)$ satisfies the Petrov conditions if it satisfies both $PC$ and $PC'$ (and we do the same for Dyck paths). The exact version of $PC'$ is not important for our results, though we point out that a uniform random permutation in $\avn(321)$ has exponentially small probability to fail either $PC$ or $PC'$. 
This, together with Corollary 5.5 and Proposition 5.6 and 5.7 of \cite{hoffman2017pattern}, implies that there exist positive constants $C$ and $\delta$ such that the probability that the Petrov conditions are not satisfied for a uniform random permutation $\bm \tau_n$ in $\avn(321)$ is bounded above by $Ce^{-n^\delta}$. 

In order to prove Theorem \ref{thin red line}, we need the following technical result.

\begin{lem}[Lemma 2.7 in \cite{hoffman2017pattern}]\label{flux}
	Let $\gamma_n$ be a Dyck path of size $2n$ that satisfies the Petrov conditions, and let $\tau_n$ be the corresponding permutation in $\avn(321)$.  If $(j,\tau_n(j))$ is a left-to-right maximum, then 
	$$|\tau_n(j) - j - \gamma_n(2j)| \leq  10n^{.4},$$
	and if $(j,\tau_n(j))$ is a right-to-left minimum, then 
	$$| \tau_n(j) - j  + \gamma_n(2j)| \leq 10n^{.4}.$$
	Therefore for all $j\leq n$, 
	$$\gamma_n(2j)- 10n^{.4} \leq | \tau_n(j) - j| \leq \gamma_n(2j) + 10n^{.4}.$$
\end{lem}

Let $M(\gamma_n)  = \max_{1\leq j\leq n}\gamma_n(2j)$ be the maximum of $\gamma_n$ at even positions and $$D(\tau_n) = \max_{1\leq j \leq n} | \tau_n(j) - j|$$ be the maximum absolute displacement.

\begin{cor}\label{dyck_max}
	Let $\tau_n$ be a permutation in $\avn(321)$ and let $\gamma_n = \gamma_{\tau_n}$ be the corresponding Dyck path of size $2n$.    If $\tau_n$ (and thus $\gamma_n$) satisfies the Petrov conditions, then $D(\tau_n) \leq M(\gamma_n) + 10n^{.4}.$
\end{cor}

We also need the following estimate, obtained in \cite{borga2019almost} after a careful analysis of the possible different ways to insert an internal point in a permutation avoiding the pattern 321.

\begin{lem}[{\cite[Lemma 7.6]{borga2019almost}}] \label{city maps}
	Fix $k>0$.  Let $\tau_n \in \avn(321)$ satisfy the Petrov conditions. Then 
	$$\frac{k!}{(2n)^{3k/2}}|\asq(\tau_n,k)| = \left(\int_0^1 F_{\tau_n}(t)dt \right)^k + O(n^{-.1}).$$	
\end{lem}

Now we can give the proof of Theorem \ref{thin red line}, that states that
$$|\avnk|\sim \frac{(2n)^{3k/2}}{k!}\cdot c_n\cdot  \mathbb{E}\left[\left(\int_0^1\bm e_t dt\right)^k\right].$$

\begin{proof}[Proof of Theorem \ref{thin red line}]

	Partition $\avn(321)$ into two sets $A_n$ and $B_n$, where permutations in $A_n$ satisfy the Petrov conditions and permutations in $B_n$ do not.  Let $c_n$ denote the $n$-th Catalan number $\frac{1}{n+1}{2n \choose n} \sim \frac{4^n}{\sqrt{2\pi n^3}}.$  Let $a_n$ and $b_n$ denote the size of $A_n$ and $B_n$ respectively.  For a uniform permutation in  $\avn(321)$, the Petrov conditions fail with probability at most $Ce^{-n^\delta}$ for some $C,\delta > 0$, thus we have that $b_n \leq  Ce^{-n^\delta}c_n$.    
	
	For any $\tau_n \in \avn(321)$ we always have the trivial upper bound $\asq(\tau_n,k) \leq (n+k)^{2k}$. Thus the contribution to $|\avnk|$ from permutations with external points in $B_n$, i.e.\ permutations in $\asq(B_n,k),$ is at most $c_n(n+k)^{2k} Ce^{-n^\delta}\leq c_n (2n)^{2k}Ce^{-n^\delta} = o(c_n)$.  Using Lemma \ref{city maps}, we obtain
	\begin{align}
		|\avnk| &=  \sum_{\tau_n \in \avn(321)} |\asq(\tau_n,k)|\label{carlos} \\
		 &=  \sum_{\tau_n \in A_n} |\asq(\tau_n,k)|+\sum_{\tau_n \in B_n} |\asq(\tau_n,k)| \\
		&= c_n \cdot \E\left[ |\asq(\bm{\tau}_n,k)|  \mathds{1}_{\left\{\bm \tau_n \in A_n\right\}} \right] + o\left(c_n n^{3k/2-.1}\right)\nonumber\\
		&= \frac{(2n)^{3k/2}c_n}{k!}\E\left [\left ( \int_0^1 F_{\bm{\tau}_n}(t) dt \right)^k \Bigg | \bm \tau_n \in A_n \right]\P(\bm\tau_n \in A_n) + o\left(c_n n^{3k/2-.1}\right).\nonumber
	\end{align}
	Using that $\P(\bm\tau_n\in B_n ) \leq Ce^{-n^{\delta}}$ and $F_{\bm\tau_n}(t) \leq n^{1/2}$, we have 
	\begin{equation}
		\label{eq:bound1}
		\E\left [\left ( \int_0^1 F_{\bm{\tau}_n}(t) dt \right)^k \Bigg | \bm \tau_n \in B_n \right]\P(\bm\tau_n\in B_n )\leq  Cn^{k/2}e^{-n^\delta}.
	\end{equation}
	Rewriting the expectation $\E\left [\left ( \int_0^1 F_{\bm{\tau}_n}(t) dt \right)^k \right]$ as
	\begin{equation}
		\label{eq:bound2}
		\E\left [\left ( \int_0^1 F_{\bm{\tau}_n}(t) dt \right)^k \Bigg | \bm \tau_n \in A_n \right] \P(\bm\tau_n\in A_n )+\E\left [\left ( \int_0^1 F_{\bm{\tau}_n}(t) dt \right)^k \Bigg | \bm \tau_n \in B_n \right] \P(\bm\tau_n\in B_n )
	\end{equation}
	we have that convergence of the $k$-th moment of $(\int_0^1 F_{\bm\tau_n}(t)dt \big | \bm\tau_n\in A_n)$ is equivalent to the convergence of the $k$-th moment of $\int_0^1 F_{\bm\tau_n}(t)dt$. Moreover, if the limits exist, they must agree. Suppose this is the case, then \cref{carlos} becomes
	
	\begin{equation}\label{dealio}
		|\avnk| = \frac{(2n)^{3k/2}c_n}{k!}\E\left [ \left ( \int_0^1 F_{\bm\tau_n}(t) dt \right)^k\right] + o\left(c_n n^{3k/2-.1}\right).	
	\end{equation}

	It remains to show the existence of the limit of the $k$-th moment of the area $\int_0^1 F_{\bm\tau_n}(t) dt$. We have the simple upper bound
	\begin{equation}\label{upside down}
		\int_0^1 F_{\bm{\tau}_n}(t) dt \leq \sup_{t\in [0,1]} F_{\bm\tau_n}(t) = \frac{1}{\sqrt{2n}} D( \bm \tau_n ).
	\end{equation}
	For each $k>0$ and for $n$ large enough, from Corollary \ref{dyck_max}
	\begin{align}
		\E\left [ \left( \int_0^1 F_{\bm\tau_n}(t) dt \right)^k  \Bigg | \bm\tau_n \in A_n \right] &\leq \E\left[ \left(\frac{1}{\sqrt{2n}}D( \bm\tau_n) \right)^k \Bigg | \bm\tau_n \in A_n \right] \nonumber \\
		& \leq  \E \left [ \left( \frac{1}{\sqrt{2n}}(M(\bm\gamma_n) + 10n^{.4})\right)^k \Bigg | \bm\tau_n \in A_n \right ] \nonumber \\
		& \leq \E\left[ \left( \frac{1}{\sqrt{2n}} M(\bm\gamma_n) \right)^k \right ] \frac{(1 + O(n^{-.1}))}{\P(\bm\tau_n \in A_n)}\nonumber \\
		& \leq \frac{2}{\P(\bm\tau_n \in A_n)}\E\left[ \left( \frac{1}{\sqrt{2n}} M(\bm\gamma_n) \right)^k \right ].\nonumber
	\end{align}
	Therefore, from \cref{eq:bound1} and \cref{eq:bound2} we obtain the following bound
	\begin{equation}
		\E\left [\left ( \int_0^1 F_{\bm{\tau}_n}(t) dt \right)^k \right]\leq 2\cdot \E\left[ \left( \frac{1}{\sqrt{2n}} M(\bm\gamma_n) \right)^k \right ]+Cn^{k/2}e^{-n^\delta}.
	\end{equation}
	By \cite[Theorem 1]{MR2535080} the exponential moment of $(2n)^{-1/2} M(\bm\gamma_n)$ is uniformly bounded in $n$,  thus for any $k>0$, the $k$-th moment of $\int_0^1 F_{\bm\tau_n}(t)dt$ is uniformly bounded in $n$.  This along with the convergence in distribution of $\int_0^1 F_{\bm\tau_n}(t) dt$ to $\int_0^1 \bm e_tdt$ implies convergence of the $k$-th moments:
	
	\begin{equation}\label{final countdown}
		\E\left [\left(\int_0^1 F_{\bm{\tau}_n}(t) dt \right)^k\right] \longrightarrow \E\left[ \left( \int_0^1 \bm{e}_t dt\right) ^k\right ]
	\end{equation}
	(see \cite[Theorem~4.5.2]{Chung2001course}, for instance). 
	Dividing both sides of \cref{dealio} by $(2n)^{3k/2}c_n/k!$ gives
	\begin{equation}
		\frac{|\avnk|}{(2n)^{3k/2}c_n/k!} 	 = \E\left [\left(\int_0^1 F_{\bm{\tau}_n}(t) dt \right)^k\right]  + o(1),
	\end{equation}
	and letting $n$ tend to infinity finishes the proof.
\end{proof}

\subsection{Fluctuations}\label{sect:321_fluctu}

We conclude this chapter proving Theorem \ref{thm:fluctuations_321}. We recall that for a permutation $\tau_n\in Av_n(321)$ (with the convention that $\tau_n(0)=0$) we defined
\begin{equation}\label{eq:first_def}
	F_{\tau_n}(t) \coloneqq \frac{1}{\sqrt{2n}}\big |\tau_n(\lfloor nt \rfloor) - \lfloor nt \rfloor \big|, \quad t\in[0,1].
\end{equation}
We also generalized this definition, by setting, for a permutation $\tau^k_n\in \avnk$ (with the convention that $\tau^k_n(0)=0$), 
\begin{equation}\label{eq:second_def}
	F_{\tau^k_n}(t) \coloneqq \frac{1}{\sqrt{2(n+k)}}\big |\tau^k_n(s(t)) - s(t) \big|,\quad  t\in[0,1],
\end{equation}
where $s(t)=\max\left\{m\leq \lfloor (n+k)t \rfloor|\tau^k_n(m)\text{ is an external point}\right\}$. Note that, for permutations in $ Av_n(321)$, the definition given in \cref{eq:second_def} coincides with the definition given in \cref{eq:first_def}.
We need the following technical result.

\begin{lem}[{\cite[Lemma 7.7]{borga2019almost}}]\label{lem:fluctuations}Let $Reg_n^k$ be the set of permutations in $\avnk$ such that the exterior satisfies the Petrov conditions. As $n\to\infty,$
	$$\sup_{\tau^k_n\in Reg^k_n}\|F_{\tau^k_n}(t)-F_{\ext(\tau^k_n)}(t)\|_{\infty}\to 0,$$
	where, for a function $f:[0,1]\to\R$, we denote $\|f\|_{\infty}=\sup_{t\in[0,1]}|f(t)|$.
\end{lem}

We can finally prove Theorem \ref{thm:fluctuations_321}, namely, that for a uniform random permutation $\bm{\tau}^k_n$ in $\avnk$, the following convergence in distribution holds
$$\left(F_{\bm{\tau}^k_n}(t)\right)_{t\in [0,1]} \stackrel{d}{\longrightarrow} \left(\bm{e}^k_t\right)_{t\in [0,1]},$$
where we recall that $\bm{e}^k_t$ denotes the $k$-biased Brownian excursion on $[0,1]$ (see \cref{def:kbiasedex}).

\begin{proof}[Proof of Theorem \ref{thm:fluctuations_321}]
	It is enough to show that for every bounded continuous functional \break
	$G:D([0,1],\mathbb{R})\to\R,$
	\begin{equation}
		\E\left[G\left(F_{\bm{\tau}^k_n}(t)\right)\right]\to\E\left[G\left(\bm{e}^k_t\right)\right].
	\end{equation}
	Note that
	\begin{multline}
		\left|\E\left[G\left(F_{\bm{\tau}^k_n}(t)\right)\right]-\E\left[G\left(\bm{e}^k_t\right)\right]\right|\\
		\leq\E\left[\left|G\left(F_{\bm{\tau}^k_n}(t)\right)-G\left(F_{\ext(\bm{\tau}^k_n)}(t)\right)\right|\right]+\left|\E\left[G\left(F_{\ext(\bm{\tau}^k_n)}(t)\right)\right]-\E\left[G\left(\bm{e}^k_t\right)\right]\right|.
	\end{multline}
	We first show that
	\begin{equation}\label{eq:gooal1}
		\E\left[\left|G\left(F_{\bm{\tau}^k_n}(t)\right)-G\left(F_{\ext(\bm{\tau}^k_n)}(t)\right)\right|\right]\to 0.
	\end{equation}
	We have that
	\begin{equation}
		\E\left[\left|G\left(F_{\bm{\tau}^k_n}(t)\right)-G\left(F_{\ext(\bm{\tau}^k_n)}(t)\right)\right|\right]=\sum_{\tau^k_n\in\avnk}\left|G\left(F_{\tau^k_n}(t)\right)-G\left(F_{\ext(\tau^k_n)}(t)\right)\right|\P\left(\bm{\tau}^k_n=\tau^k_n\right).
	\end{equation}
	The continuity of $G$ and Lemma \ref{lem:fluctuations} show that the contribution to the sum vanishes as $n\to\infty$ for $\tau_n^k\in Reg_n^k$.  Since $G$ is bounded and $\P(\bm{\tau}^k_n\notin Reg_n^k)\to 0$, we can conclude that the contribution to the sum for $\tau_n^k\notin Reg_n^k$ also vanishes as $n\to\infty$, and thus (\ref{eq:gooal1}) holds.
	
	It remains to prove that 
	\begin{equation}\label{eq:goaal2}
		\left|\E\left[G\left(F_{\ext(\bm{\tau}^k_n)}(t)\right)\right]-\E\left[G\left(\bm{e}^k_t\right)\right]\right|\to 0.
	\end{equation}
	Note that 
	\begin{equation}
		\E\left[G\left(F_{\ext(\bm{\tau}^k_n)}(t)\right)\right]=\sum_{\tau^k_n\in\avnk}G\left(F_{\ext(\tau^k_n)}(t)\right)\cdot\P\left(\bm{\tau}^k_n=\tau^k_n\right).
	\end{equation}
	From the asymptotic enumeration of $\avnk$ obtained in \cref{thin red line} we have that, uniformly for every $\tau^k_n\in\avnk$,
	\begin{equation}
		\P\left(\bm{\tau}^k_n=\tau^k_n\right)\sim \frac{k!}{(2n)^{3k/2}}\cdot\frac{1}{c_n}\cdot\mathbb{E}\left[\left(\int_0^1\bm e_t dt\right)^k\right]^{-1},
	\end{equation}
	and so, setting $Ar_k=\mathbb{E}\left[\left(\int_0^1\bm e_t dt\right)^k\right]^{-1}$, we obtain
	\begin{align}
		\E\left[G\left(F_{\ext(\bm{\tau}^k_n)}(t)\right)\right]\sim& Ar_k\cdot\sum_{\tau^k_n\in\avnk}G\left(F_{\ext(\tau^k_n)}(t)\right)\cdot\frac{k!}{(2n)^{3k/2}}\cdot\frac{1}{c_n}\\
		&=Ar_k\cdot\sum_{\sigma_n\in Av_n(321)}G\left(F_{\sigma_n}(t)\right)\cdot\left|\asq(\sigma_n,k)\right|\cdot\frac{k!}{(2n)^{3k/2}}\cdot\frac{1}{c_n}.
	\end{align}
	From Lemma \ref{city maps}, for every $\sigma_n\in Av_n(321)$ that satisfies the Petrov conditions, it holds that
	$$\left|\asq(\sigma_n,k)\right| = \frac{(2n)^{3k/2}}{k!}\left(\left(\int_0^1 F_{\sigma_n}(t)dt \right)^k + O(n^{-.1})\right).$$
	Therefore, using the asymptotic result above and recalling that the number of 321-avoiding permutations of size $n$ that do not satisfy the Petrov conditions is bounded by $Ce^{-n^\delta}c_n$, we obtain
	\begin{align}
		\E\left[G\left(F_{\ext(\bm{\tau}^k_n)}(t)\right)\right]\sim& Ar_k\cdot\sum_{\sigma_n\in Av_n(321)}G\left(F_{\sigma_n}(t)\right)\cdot \left(\int_0^1 F_{\sigma_n}(t)dt \right)^k\cdot\frac{1}{c_n}\\
		&=Ar_k\cdot\E\left[G\left(F_{\bm\sigma_n}(t)\right)\cdot \left(\int_0^1 F_{\bm\sigma_n}(t)dt \right)^k\right],
	\end{align}
	where $\bm\sigma_n$ is a uniform permutation in $Av_n(321)$. Using \cref{chacha} and similar arguments to the ones used for proving the result in (\ref{final countdown}), we have that
	\begin{equation}
		\E\left[G\left(F_{\bm\sigma_n}(t)\right)\cdot \left(\int_0^1 F_{\bm\sigma_n}(t)dt \right)^k\right]\to\E\left[G\left(\bm{e}_t\right)\cdot \left(\int_0^1 \bm{e}_tdt \right)^k\right].
	\end{equation}
	Finally, recalling the definition of $k$-biased excursion given in Definition \ref{def:kbiasedex}, we can conclude that (\ref{eq:goaal2}) holds, finishing the proof.
\end{proof}

\section{Open problems}

\begin{itemize}
	\item The main open problem of this section is undoubtedly to give a formal proof of \cref{conj:phase_trans}. We believe that the strategy outlined in \cref{sect:strat_conj} can be formalized without big problems. Nevertheless, in order to prove the full statement of \cref{conj:phase_trans} it will be also needed a concentration result for the four boundary points of the diagram and a sort of independence of the points in the four parts of the diagram (recall that we just discussed the limit shape of the bottom-left part of the diagram). 
	
	\item In Section \ref{sect:321av} we investigate $321$-avoiding permutations with $k$ additional internal points for $k$ fixed. What about the case when $k\to\infty$? We believe that in this regime the fluctuations of the points for a uniform permutation are not of order $\sqrt n$ any more, drastically changing the behavior of the limiting shape of the permutation. 
	
	\item The $k$-th moment of the Brownian excursion area appearing in Theorem~\ref{thin red line} is known to be the continuous limit of the normalized $k$-th moment of the area under large Dyck paths. It would be interesting to establish a bijection between $\avnk$ and some specific set of Dyck paths covering $k$ marked points. This would also probably help in the resolution of the problem mentioned in the previous item.

  	\item The approach of assigning labels to the points of a permutation and then projecting these labels both horizontally and vertically seems to be quite generalizable.  The following are some ideas for future work. 

		\begin{itemize}      
			\item Monotone grid classes (introduced in \cite{huczynska2006grid} and then studied in many others works \cite{albert2013geometric,atkinson1999restricted,vatter2011partial,waton2007permutation}) seem to be a family of models where our approach fits well. We point out that some initial results on the shape of such permutations were given by Bevan in his Ph.D. thesis \cite{bevan2015growth}.  Our approach suggests a way to give a description of the fluctuations and local limits of monotone grid classes.
			\item We also believe that our technique can give an alternative approach to the probabilistic study of $X$-permutations first considered in \cite{waton2007permutation,elizalde2011thexclass} and recently in \cite{bassino2019scaling}. In particular our technique could give nice additional results about the fluctuations similar to the ones explored in this chapter. We finally recall that $X$-permutations are a particular instance of geometric grid classes (see \cite{albert2013geometric} for a nice introduction), and so also these families might be investigated in future projects.
		\end{itemize}
\end{itemize}

    \chapter{Permuton limits: A path towards a new universality class}\label{chp:perm_lim}
\chaptermark{Permuton limits}

\begin{adjustwidth}{8em}{0pt}
	\emph{In which we introduce the biased Brownian separable permuton (left picture) and the Baxter permuton (middle picture). We show that the first one is the universal limit --in the permuton sense -- of uniform random permutations in many substitution-closed classes and the second one is the permuton limit of uniform Baxter permutations. We also prove joint convergence of the four trees characterizing a uniform bipolar orientation and its dual to a coupling of four Brownian CRTs. We conclude by describing a new (conjectural) universal limiting object for random permutations, called skew Brownian permuton (right picture). The latter should include both the biased Brownian separable permuton and the Baxter permuton as particular cases.}
\end{adjustwidth}

\vspace{1cm}

\begin{figure}[b]
	\centering
	\includegraphics[scale=0.095]{Separable_1_1335}
	\includegraphics[scale=0.0865]{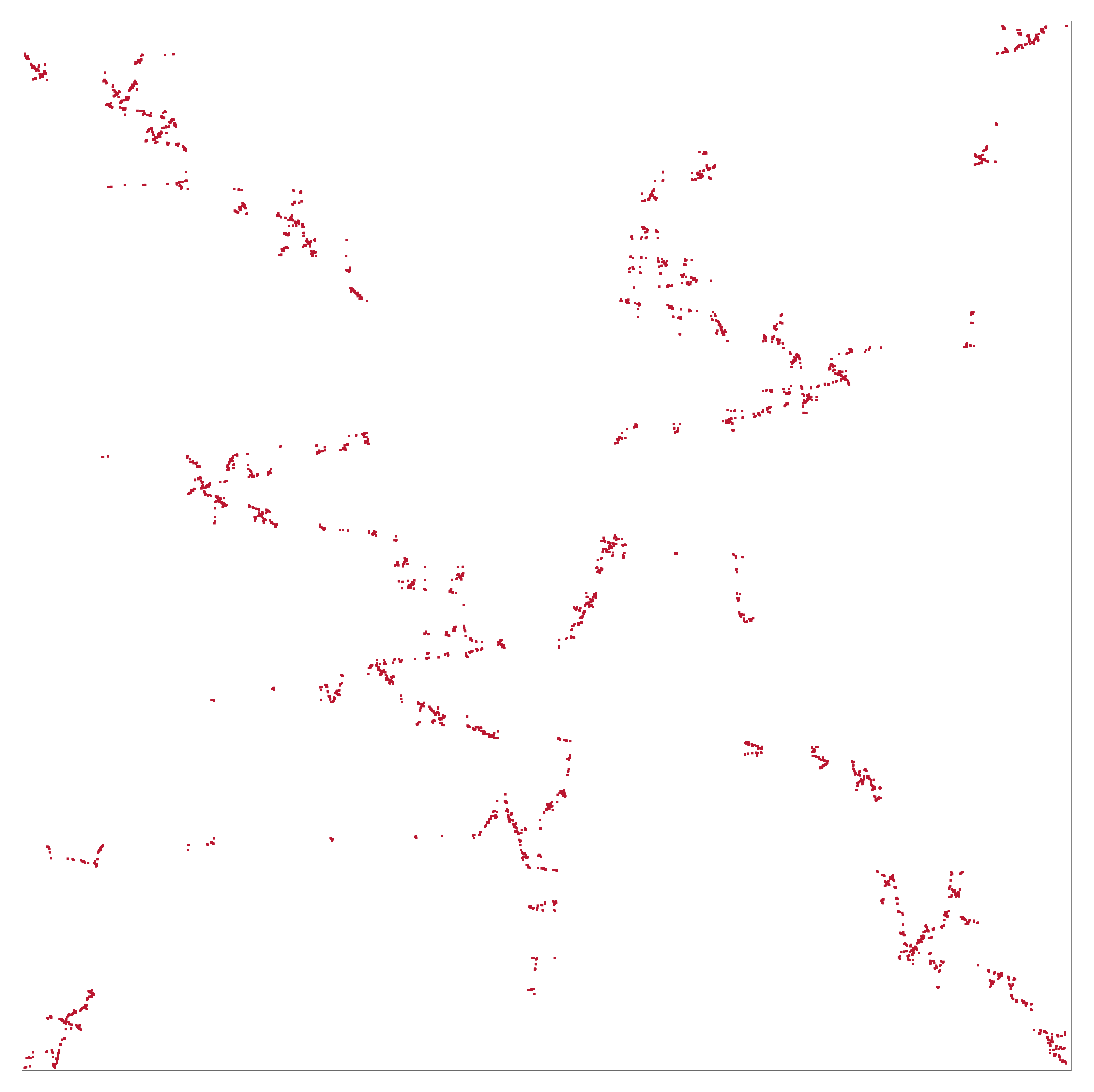}
	\includegraphics[scale=0.087]{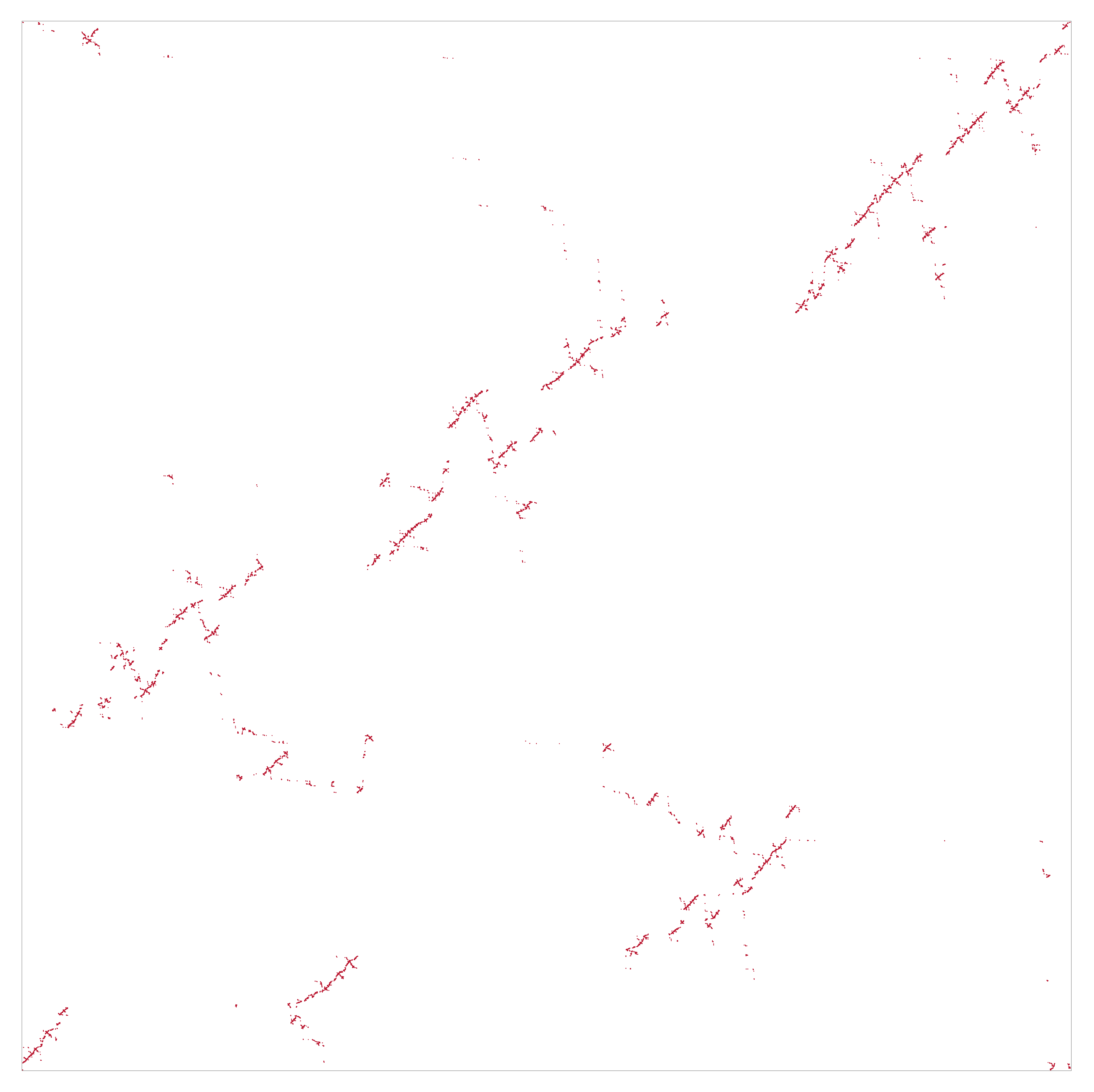}
	\caption{\textbf{Left}: The biased Brownian separable permuton. \textbf{Middle}: The Baxter permuton. \textbf{Right}: The skew Brownian permuton.\label{fig:sym}}
\end{figure}

\section{Overview of the main results}

In this chapter we deal with permuton limits, using the theory introduced in \cref{sect:perm_conv}. We focus on uniform random permutations in substitution-closed classes (\cref{sect:sub_close_cls}) and on uniform Baxter permutations (\cref{sect:intro_scaling_limit_results}), highlighting another  instance of a universal phenomenon for random constrained permutations.
Finally, we also investigate the scaling limit (in a sense made precise later) of uniform bipolar orientations. Both the proofs of the results for Baxter permutations and bipolar orientations builds on the existence of a scaling limit for the corresponding coalescent-walk processes, introduced in \cref{sect:coal_walk_intro}.

\medskip

At the end of the chapter (\cref{sect:skew_perm}), we finally discuss the definition of a new universal limiting permuton, called the \emph{skew Brownian permuton}, that is deeply related with the two permutons mentioned above. We will present several conjectures about this new random limiting object that we would like to investigate in depth in future research projects.

\subsection{Substitution-closed classes}\label{sect:sub_close_cls}
The {\em biased Brownian separable permuton} $\bm{\mu}^{(p)}$ of parameter $p\in[0,1]$
is a random permuton (see \cref{fig:perm_sub}) constructed from a Brownian excursion and independent signs associated
with its local minima. In \cref{sect:sep_perm_info} we will give more background on this Brownian object, discussing some of its fundamental properties.

The Brownian separable permuton $\bm{\mu}^{(1/2)}$ was first discovered in \cite{bassino2018separable} studying the permuton limit of separable permutations.
Then it was proved in \cite{bassino2017universal}  that the biased Brownian separable permuton $\bm{\mu}^{(p)}$ is a universal limiting object for substitution-closed classes,
in the sense that uniform random permutations in many substitution-closed classes $\mathcal C$
converge to $\bm{\mu}^{(p)}$, for some $p=p(\mathcal C)\in[0,1]$.

\medskip

Here, we give a new (more probabilistic\footnote{The proof in \cite{bassino2017universal} mainly uses analytic combinatorics techniques.}) proof of this theorem building on an extension of Aldous' skeleton decomposition~\cite{MR1207226}.

\begin{figure}[htbp]
	\begin{minipage}[c]{0.70\textwidth}
		\centering
		\includegraphics[scale=0.1]{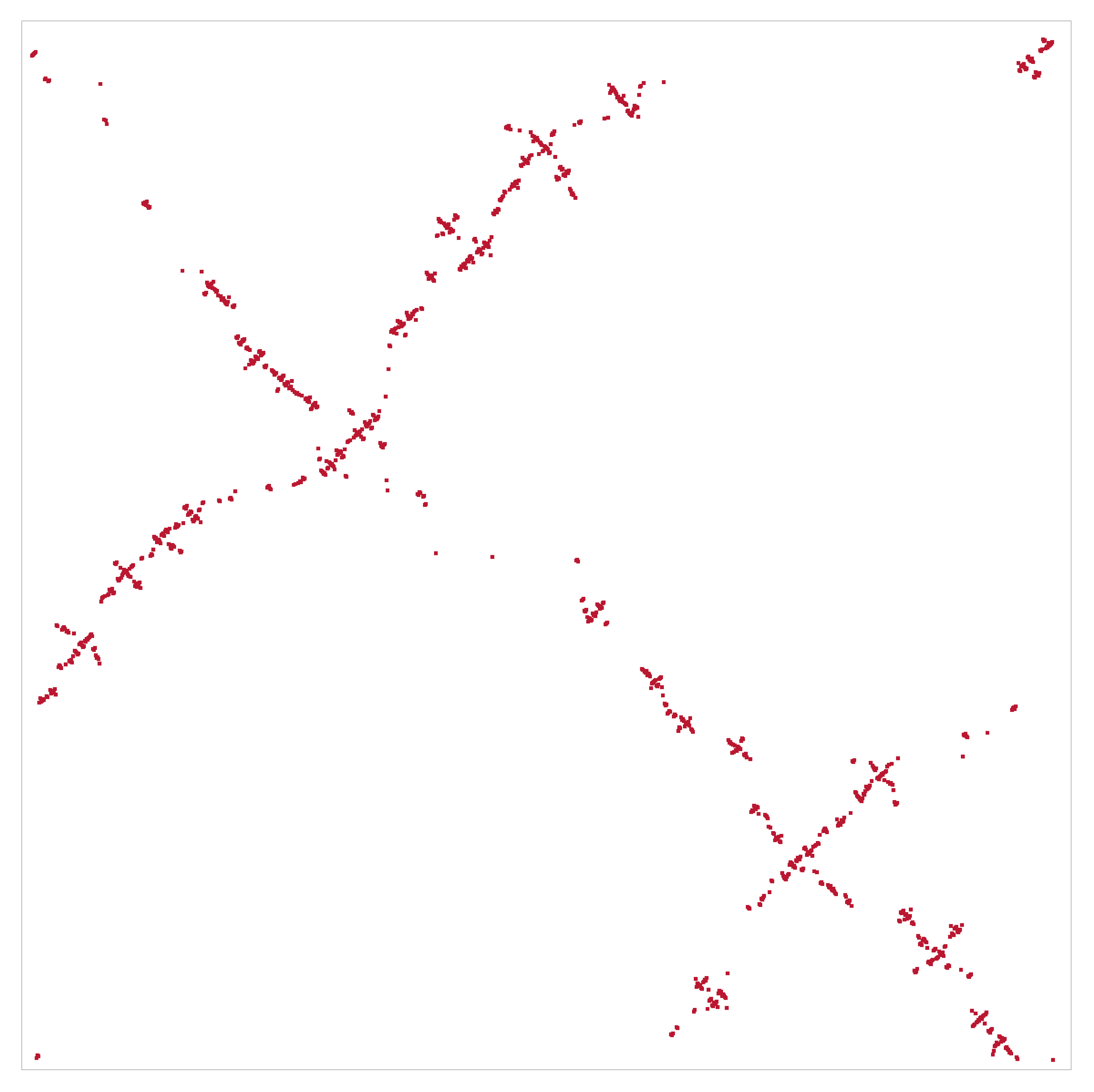}
		\includegraphics[scale=0.1]{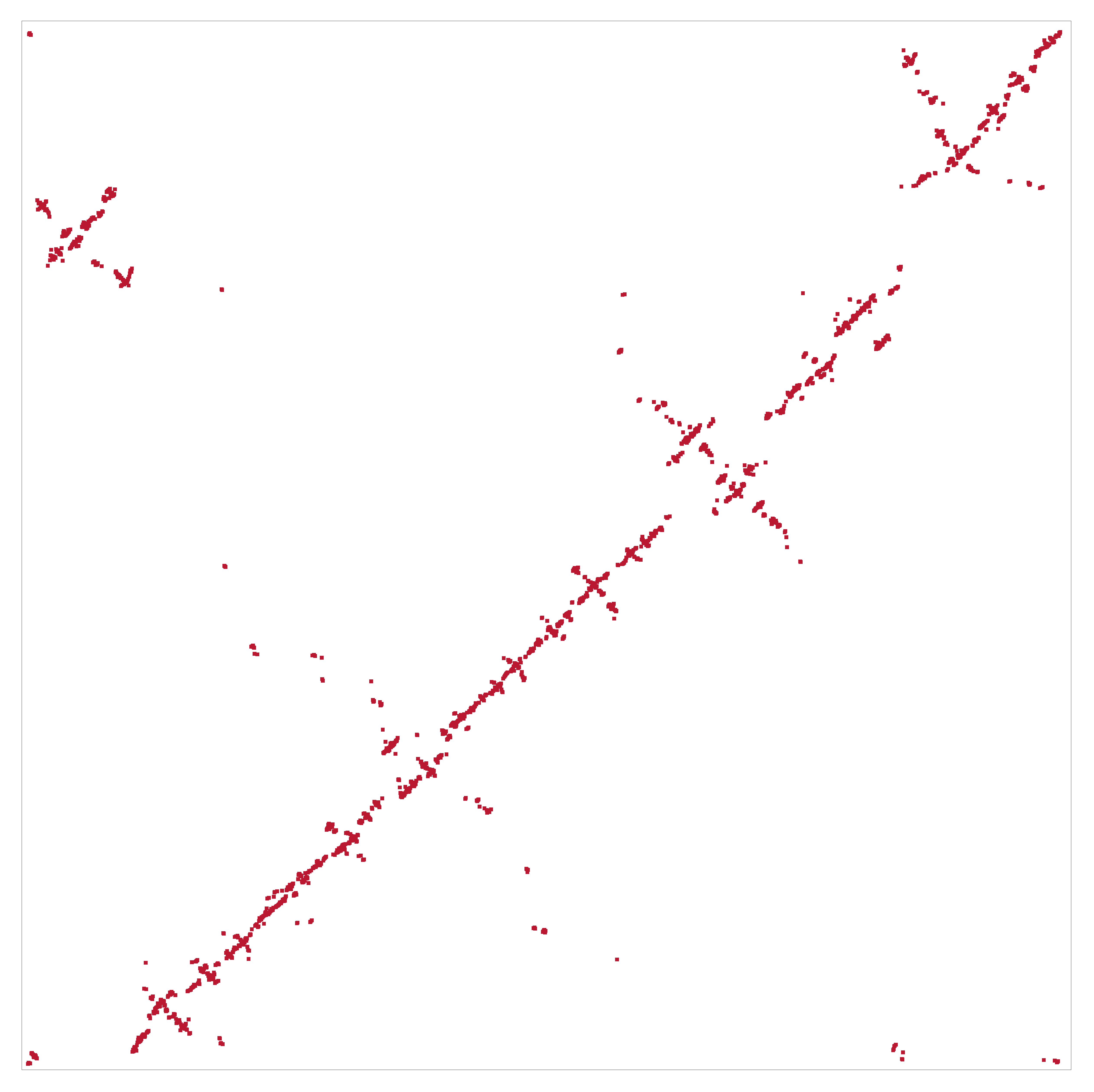}
	\end{minipage}
	\begin{minipage}[c]{0.29\textwidth}
		\caption{Two simulations of the biased Brownian separable permuton $\bm{\mu}^{(p)}$ for $p=1/2$ (left) and $p=0.6$ (right). 
			\label{fig:perm_sub}
		}
	\end{minipage}
\end{figure}

Recall the notation from \cref{subsec:PackedTrees_GW2} page \pageref{subsec:PackedTrees_GW2} and \cref{prop: offspring_distr_charact}.

\begin{thm}
	\label{thm:scaling_intro}
	Let $\bm{\sigma}_n$ be a uniform permutation of size $n$ from a proper substitution-closed class of permutations $\cC$, whose set of simple permutations is $\mathfrak S$. 
	Let $\xi$ be the offspring distribution of the Galton--Watson tree model associated with $\cC$. 
	Suppose that $\E[\xi]=1$ and $\V[\xi] < \infty$. That is, either
		\begin{equation}
		\label{eq:S_ExpMoments}
		\cS'(\rho_\cS) > \frac{2}{(1 +\rho_\cS)^2} -1,
		\end{equation}
	or
	\begin{equation}
		\label{eq:S_FiniteVariance}
		\cS'(\rho_\cS)  = \frac{2}{(1 +\rho_\cS)^2} -1 \qquad \text{and} \qquad 	\cS''(\rho_\cS)  < \infty.
	\end{equation}
	Then
	\[
	\bm{\mu}_{\bm{\sigma}_n} \xrightarrow{d} \bm{\mu}^{(p)},
	\]
	with $\bm{\mu}^{(p)}$ denoting the biased Brownian separable permuton with parameter
	\begin{equation}
		\label{eq:parm_p}
		p = \frac{2}{\sigma^2}\big(\kappa(1+\kappa)^3\Occ_{12}(\kappa)+\kappa\big),
	\end{equation}
	where $\kappa$ and $\sigma^2=\V[\xi]$ are defined in \cref{prop: offspring_distr_charact} and $\Occ_{12}(z)=\sum_{\alpha\in\mathfrak S}\occ(12,\alpha)z^{|\alpha|-2}$.
	
	This includes the case when $\cC$ is the class of separable permutations, for which $\mathfrak S=\emptyset$ and $p=1/2$.
\end{thm}

Specifically, the result~\cite[Thm. 1.2]{bassino2018separable} corresponds to the special case where $\mathfrak S= \emptyset$, and \cite[Thm. 1.10]{bassino2017universal} corresponds to the special case where \cref{eq:S_ExpMoments} is satisfied. The result~\cite[Thm. 7.8]{bassino2017universal} corresponds to the case where \cref{eq:S_FiniteVariance} is satisfied with the additional assumption (removed in \cref{thm:scaling_intro}) that $\cS(z)$ is amendable to singularity analysis.

\subsection{Baxter permutations and bipolar orientations}
\label{sect:intro_scaling_limit_results}

In this section we present a joint scaling limit result for Baxter permutations (see \cref{fig:Baxter_sub}) and the tandem walks describing the four trees characterizing bipolar orientations and their dual maps.

This result will lead us to the discovery of the Baxter permuton, which is a \emph{new} limiting \emph{random} permuton not included in the family of biased Brownian separable permutons.

\begin{figure}[htbp]
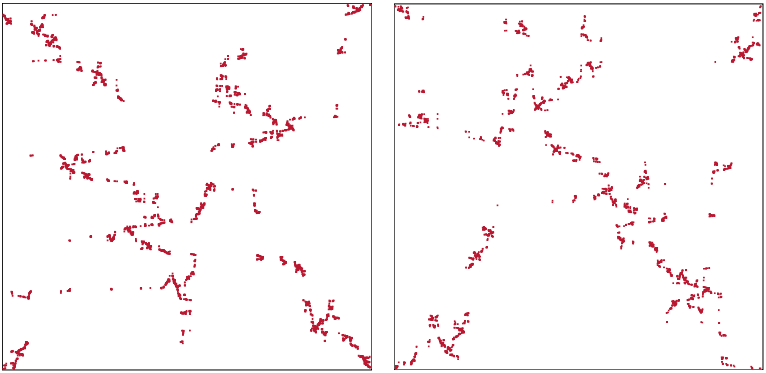

	\begin{minipage}[c]{0.70\textwidth}
		\centering
		\includegraphics[scale=0.09]{Baxter_1_4520}
		\includegraphics[scale=0.079]{Baxter_2_3253}
	\end{minipage}
	\begin{minipage}[c]{0.29\textwidth}
		\caption{Two large uniform Baxter permutations. 
			\label{fig:Baxter_sub}
		}
	\end{minipage}
\end{figure}

Recall the notation from \cref{sec:discrete}, see in particular \cref{thm:diagram_commutes} page \pageref{thm:diagram_commutes}. For $n\geq 1$, let $\bm \sigma_n$ be a uniform Baxter permutation of size $n$ and $\bm m_n=\bobp^{-1}(\bm \sigma_n)$ the corresponding uniform bipolar orientation with $n$ edges. Let $\bm W_n = \bow(\bm m_n)$ and $\bm W_n^* = \bow(\bm m_n^*)$ be the tandem walks associated with $\bm m_n$ and its dual $\bm m_n^*$.	
Let ${\conti W}_n$ and ${\conti W}_n^*$ be the two continuous functions from $[0,1]$ to $\R_{\geq 0}^2$ that linearly interpolate the points ${\conti W}_n^\theta\left(\frac kn\right) = \frac 1 {\sqrt {2n}} {\bm W}_{n}^\theta(k)$ for $1\leq k \leq n$ and $\theta\in\{\emptyset,*\}$.

Let $\conti W = (\conti X(t),\conti Y(t))_{t\geq 0}$ be a \textit{standard two-dimensional Brownian motion of correlation -1/2}, that is a continuous two-dimensional Gaussian process such that the components $\conti X$ and $\conti Y $ are standard one-dimensional Brownian motions, and $\mathrm{Cov}(\conti X(t),\conti Y(s)) = -1/2 \cdot \min\{t,s\}$. Let $\conti W_e$ be a \textit{two-dimensional Brownian excursion of correlation -1/2 in the non-negative quadrant}, that is the process $(\conti W(t))_{0\leq t\leq 1}$ conditioned on $\conti W(1) = (0,0)$ and on staying in the non-negative quadrant $\R_{\geq 0}^2$. A rigorous definition is given in \cref{sec:appendix}.

Consider the time-reversal and coordinate-swapping mapping
$s:\mathcal C([0,1],\R^2) \to \mathcal C([0,1],\R^2)$ defined by $s(f,g) = (g(1-\cdot), f(1-\cdot))$. Consider also the mapping $R: \mathcal M \to \mathcal M$ that rotates a permuton by an angle $-\pi/2$.

\begin{thm}\label{thm:joint_intro}
	There exist two measurable mappings 
	$$r:\mathcal C([0,1], \R_{\geq 0}^2) \to \mathcal C([0,1],\R_{\geq 0}^2)\quad\text{and}\quad \phi:\mathcal C([0,1], \R_{\geq 0}^2) \to \mathcal M$$ such that the following convergence in distribution holds
	\begin{equation}
	\label{eq:scal_lim_comp}
	({\conti W}_n,{\conti W}_n^*, \mu_{\bm \sigma_n}) \xrightarrow{d} (\conti W_e, \conti W_e^*, \bm \mu_B),
	\end{equation}	
	where $\conti W_e^* = r(\conti W_e)$, and $\bm \mu_B  = \phi(\conti W_e)$.
	In particular, we have $r(\conti W_e) \stackrel d= \conti W_e$. Moreover, we have the following equalities that hold at $\P_{\conti W_e}$-almost every point of $\mathcal C([0,1], \R_{\geq 0}^2)$,
	\begin{gather}
	r^2 = s, \quad
	r^4 = \Id,\quad
	\phi \circ r = R\circ \phi.
	\end{gather} 
\end{thm}

We give a few remarks on this result:
\begin{itemize}
	\item The convergence of the first marginal was obtained in \cite{MR3945746} as an immediate application of the results of \cite{duraj2015invariance} on walks in cones.
	\item Our strategy of proof is based on  coalescent-walk processes, which describe the relation between $\conti W_n$, $\conti W_n^*$ and $\bm \sigma_n$ in a way that allows itself to take limits. In the remainder of this section we explain what is the scaling limit of coalescent-walk processes, providing the reader with some insights on how the coupling of the right-hand side of \cref{eq:scal_lim_comp} is constructed. Precise statements, including explicit constructions of the mappings $r$ and $\phi$, are given in \cref{sec:final} (see in particular Theorems \ref{thm:permuton} and \ref{thm:joint_scaling_limits}).
	\item The limiting permuton $\bm \mu_B$, called the \emph{Baxter permuton}, is a new random measure on the unit square (see \cref{defn:Baxter_perm} for a precise definition and below for an informal one). 
	\item Recall that each coordinate of $\bm W_n$ or $\bm W_n^*$ records the height function of a tree which can be drawn on $\bm m_n$ or its dual (see in particular \cref{rem:height_process}). So this statement can be interpreted as joint convergence of four trees to a coupling of four Brownian CRTs. 
	\item For a discussion on the relation of our results and the Liouville quantum gravity literature (in particular with Conjecture 4.4 of \cite{MR3945746}, the main result of \cite{gwynne2016joint}, and other related works) we refer the reader to \cite[Section 1.6]{borga2020scaling}.
\end{itemize}

\medskip

As mentioned above, the proof of \cref{thm:joint_intro} is based on a scaling limit result for the coalescent-walk processes $\bm Z_n=\wcp (\bm W_n)$ , which appears to be of independent interest.
We give here some brief explanations and
we refer the reader to \cref{sec:coalescent} for more precise results.
The definition of the coalescent-walk processes $Z=\wcp(W)$ associated with a two-dimensional tandem walk $W$ (recall \cref{defn:distrib_incr_coal}) is a sort of "discretized" version of the following family of stochastic differential equations (SDEs) driven by \textit{the same} two-dimensional process $\conti W = (\conti X,\conti Y)$ and defined for $u\in\mathbb R$ by
\begin{align} 
\begin{cases}d\conti Z^{(u)}(t) = \idf_{\{\conti Z^{(u)}(t)> 0\}} d\conti Y(t) - \idf_{\{\conti Z^{(u)}(t)\leq 0\}} d\conti X(t), &t\geq u,\\
\conti Z^{(u)}(t)=0, &t\leq u.
\end{cases} \label{eq:flow_SDE_intro}
\end{align}
We emphasize that we have a different equation for every $u\in\R$, but all depend on \textit{the same} process $\conti W = (\conti X,\conti Y)$.
For a fixed $u\in\R$, this equation, that goes under the name of \textit{perturbed Tanaka's SDE}, has already been studied in the literature \cite{MR3098074,MR3882190} in the case where $\conti W$ is a two-dimensional Brownian motion of correlation $\rho$ with $\rho\in(-1,1)$, and more generally when the correlation coefficient varies with time. In particular, pathwise uniqueness and existence of a strong solution are known. Since the scaling limit of $\bm W_n$ (that is conditioned to start on the $x$-axis and end on the $y$-axis) is a two-dimensional Brownian \emph{excursion} $\conti W_e$ of correlation $-1/2$, one can expect that the scaling limit for the coalescent-walk process $\bm Z_n = \wcp(\bm W_n)$ is a sort of flow of solutions $\{\conti Z_e^{(u)}(t)\}_{u\in[0,1]}$ of the SDEs in \cref{eq:flow_SDE_intro} driven by $\conti W_e$ (instead of $\conti W$). This intuition is made precise in \cref{thm:discret_coal_conv_to_continuous} and it is the key-step for proving \cref{thm:joint_intro}.

\medskip

The study of flows of solutions driven by the \emph{same noise} is the subject of the theory of coalescing flows of Le Jan and Raimond, specifically that of \textit{flows of mappings}. See \cite{MR2060298,MR4112725} and the references therein. We point out that we do not need to make use of this theory. Indeed, in our proof of \cref{thm:joint_intro} we consider solutions of \cref{eq:flow_SDE_intro} for only a countable number of distinct $u$ at a time for a specific equation which admits strong solutions.
In particular, \cref{thm:discret_coal_conv_to_continuous} gives convergence of a countable number of trajectories in the product topology. Stronger convergence results, such as the ones obtained for the \textit{Brownian web} (see \cite{MR3644280} for a comprehensive survey) would be desirable, but are not a purpose of this thesis.

\medskip

We conclude this section by briefly describing the Baxter permuton $\bm \mu_B$. The collection of solutions $\{\conti Z_e^{(u)}\}_{u\in[0,1]}$ of the SDEs \eqref{eq:flow_SDE_intro}  driven by $\conti W_e$ (which cannot be defined for all pairs $(\omega,u)$ simultaneously, see \cref{rem:no_flow}) determines a random total order $\leq_{\conti Z_e}$ on (a subset of measure 1 of) $[0,1]$, defined for every $0\leq t<s \leq 1$ by:
$t\leq_{\conti Z_e} s$ if and only if $\conti Z_e^{(t)}(s)<0.$
This random total order $\leq_{\conti Z_e}$ identifies (see \cref{eq:level_function}) a Lebesgue measurable map $\varphi_{\conti Z_e}(t):[0,1]\to[0,1]$ satisfying $\varphi_{\conti Z_e}(t)\leq \varphi_{\conti Z_e}(s)$ if and only if $t\leq_{\conti Z_e} s$. The permuton $\bm \mu_B$ is, in some sense made precise in \cref{defn:Baxter_perm}, the graph of this function $\varphi_{\conti Z_e}$.

\section{Substitution-closed classes}\label{sect:subclosedperm}

In this section we show that the biased Brownian separable permuton is a universal limiting permuton, proving \cref{thm:scaling_intro}.

\subsection{Semi-local convergence of the skeleton decomposition}
\label{sec:skeleton}

\cref{subsec:PackedTrees_GW2} page \pageref{subsec:PackedTrees_GW2} establishes a connection between uniform permutations in substitution-closed classes
and conditioned Galton-Watson trees.
In this section, we provide a convergence result for skeletons induced by marked vertices in such trees.
The application to permutations will be discussed in further sections.\medskip

Aldous~\cite[Eq.\ (49)]{MR1207226} showed that 
the subtree spanned by a fixed number of random marked vertices 
in a large critical Galton--Watson tree admits a limit distribution. 
Here, we extend this \emph{skeleton decomposition} 
so that it additionally describes the asymptotic  local structure in $o(\sqrt{n})$-neighborhoods around the marked vertices
and their pairwise closest common ancestors.
Note also that Aldous works with Galton--Watson trees conditioned
on having $n$ vertices,
while we more generally consider
Galton--Watson trees conditioned
on having $n$ vertices with out-degree in a given set $\Omega$
(see \cite{MR2946438,MR3335013} for scaling limit results under such conditioning). Indeed, for our applications to the study of substitution-closed classes, we need to study Galton--Watson trees conditioned
on having $n$ leaves.

\subsubsection{Extracting the skeleton with a local structure}

Let $k \in\Z_{>0}$ denote a fixed integer and $T$ a rooted plane tree. Fix $\Omega\subseteq\Z_{\geq 0}$.
We choose an ordered sequence $\myvec{v} = (v_1, \ldots, v_k)$ of vertices in $T$
(possibly with repetitions) with out-degree in $\Omega$ that we call \emph{marked vertices}.
The goal of this section is to associate to this data
an object recording:
\begin{itemize}
	\item the genealogy between the marked vertices;
	\item the local structure around the \emph{essential vertices} of $T$, 
	which we define as the root of $T$, the marked vertices $v_1, \ldots, v_k$ and their pairwise closest common ancestors;
	\item the distances in the original tree between these vertices.
\end{itemize}
The reader can look at \cref{fig:skeleton} to see the different steps of the following construction.

\begin{itemize}
	\item  The first step is to consider the subtree $R(T, \myvec{v})$ consisting of the vertices $\myvec{v}$ and all their ancestors. For each $1 \le j \le k$ the vertex $v_j$ in $R(T, \myvec{v})$ receives the label $j$. 
	Note that the tree $T$ may be constructed from the skeleton $R(T, \myvec{v})$ by attaching an ordered sequence of branches (rooted plane trees) at each corner of $R(T, \myvec{v})$.  Here we have to consider the corner below the root-vertex twice, since branches at this corner may either be added to the left or to the right of $R(T, \myvec{v})$. 

	\item The second step is to remove the vertices of $T$ 
	which lie outside of the skeleton $R(T, \myvec{v})$ 
	and are "far" from the essential vertices. 
	For convenience, we call distance of any branch $B$ (grafted on $R(T, \myvec{v})$) 
	from a vertex  $w \in R(T, \myvec{v})$ the distance in $R(T, \myvec{v})$ 
	from $w$ to the corner where $B$ is attached. 
	For any integer $t \ge 0,$ we let $R^{[t]}(T, \myvec{v})$ denote the subtree of $T$ that contains $R(T, \myvec{v})$ and all branches grafted on $R(T, \myvec{v})$ that have distance at most $t$ from at least one essential vertex. In particular, $R^{[t]}(T, \myvec{v})$ contains all vertices of  $T$ that lie at distance at most $t$ from the essential vertices.
	
	\item The final step of the construction is to shrink the paths of $R^{[t]}(T, \myvec{v})$ 
	consisting of the vertices whose attached branches have been removed in step 2. 
	Indeed, we are interested in a scenario where the distance between any two  essential vertices is much larger than $2t$. Consider two essential  points $x \ne y$ that are connected by a path not containing other essential vertices. Assume that $x$ lies on the path from the root to $y$. If the distance between $x$ and $y$ is larger than $2t$, then the path joining $x$ and $y$ consists of a starting segment of length $t$ that starts at $x$, a \emph{middle segment} of positive length, and an end segment of length $t$ that ends at $y$.
	By construction, the branches attached to inner vertices of the middle segment of $R(T,\myvec{v})$
	do not appear in $R^{[t]}(T,\myvec{v})$.
	For any real number $s>0$, we let $s.R^{[t]}(T, \myvec{v})$ 
	denote the result of contracting each middle segment 
	to a single edge that receives a label given 
	by the product of $s$ and the number of deleted vertices in this segment. %
\end{itemize}

\begin{figure}[htbp]
	\centering
	\includegraphics[width=11cm]{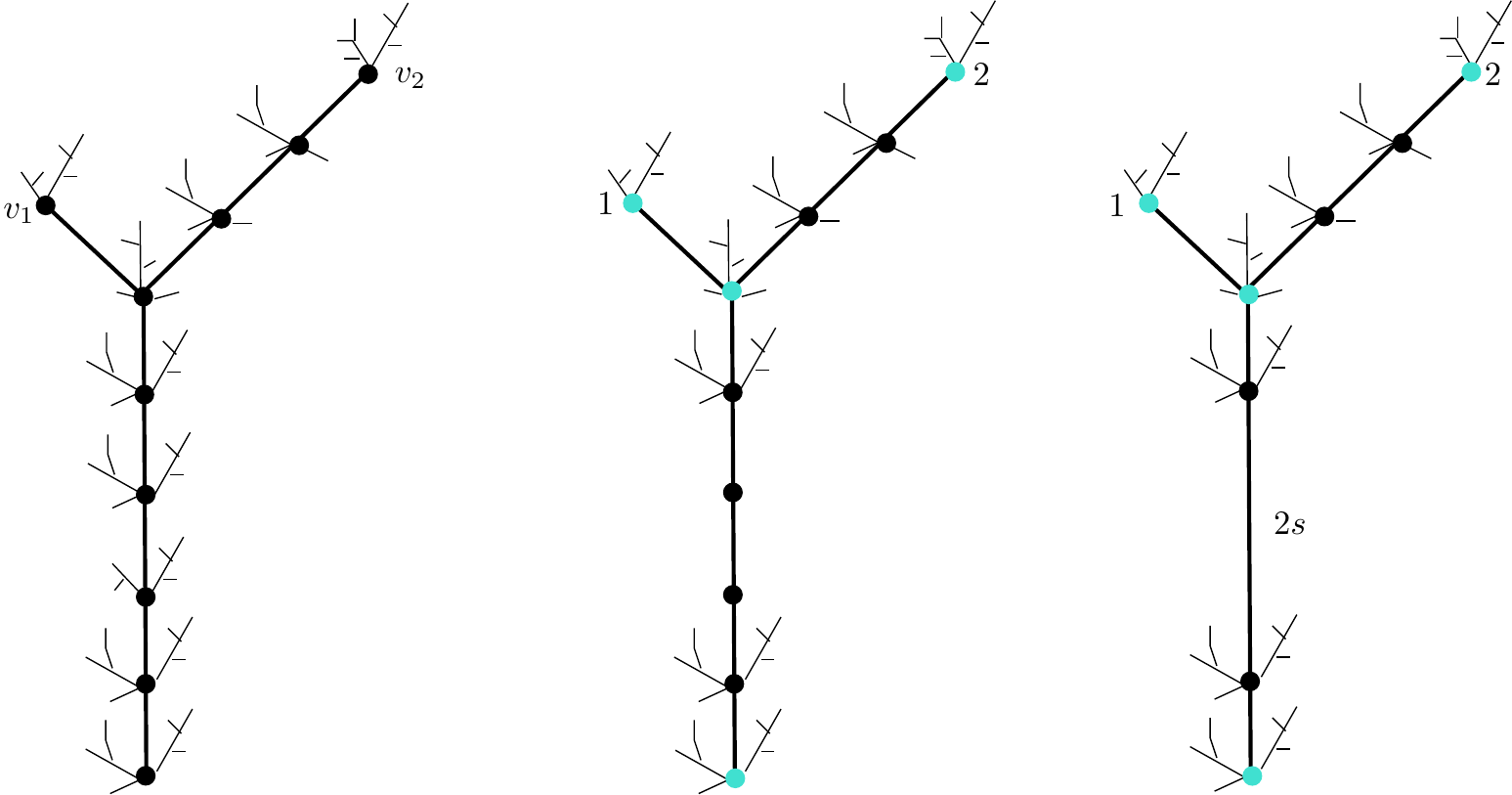}
	\caption{A tree $T$ with two marked vertices $v_1$ and $v_2$. 
		In the left-most picture, the subtree $R(T, \myvec{v})$ is represented in bold,
		while branches attached to its corner are drawn with thinner lines.
		The middle picture represent $R^{[1]}(T, \myvec{v})$:
		the essential vertices are in blue, and only two vertices of $R(T, \myvec{v})$
		are at distance more than 1 from the closest essential vertex.
		The branches attached to these vertices do not belong to $R^{[1]}(T, \myvec{v})$.
		The right-most picture represent $s.R^{[1]}(T, \myvec{v})$.
		In particular, observe that the three middle edges of the path 
		between the root and the branching vertex
		have been contracted into a single edge with label $2s$.}
	\label{fig:skeleton}
\end{figure}

\subsubsection{The space of skeletons with a local structure}
\label{ssec:skeleton_space}

In the following, we need to be more precise about 
the space in which $s.R^{[t]}(T, \myvec{v})$ lives and the topology we consider on it.
In the construction above, $s.R^{[t]}(T, \myvec{v})$ is a tree with $k$ distinguished vertices
with out-degree in $\Omega$,
where at most $2k-1$ edges have a (length-)label. 
Moreover, the distances between successive essential vertices
are at most $2t+1$ (we say that two essential vertices are successive if the path going from one to the other
does not contain any other essential vertex). 
The set of trees (without edge-labels) with $k$ marked distinguished vertices with out-degree in $\Omega$
such that the above distance condition holds is denoted $\setTkt$.
Moreover, we say that $G$ in $\setTkt$ is {\em generic} if:
\begin{itemize}
	\item there are $2k$ distinct essential vertices (the root, the $k$ distinguished vertices and $k-1$ distinct closest
	common ancestors of pairs of distinguished vertices);
	\item the distances between successive essential vertices 
	are exactly $2t+1$.
\end{itemize}

We note that the edges with (length-)label
are the middle edges of the paths of length $2t+1$ between essential vertices,
and hence depend only on the shape of the tree.
We can therefore encode these labels as a vector in $\RR^{2k-1}_{\geq 0}$, that has entries equal to $0$ whenever the corresponding essential vertices are at distance $2t$ or less.
Finally, $s.R^{[t]}(T, \myvec{v})$ can be seen as an element of
\[ \setTkt \times \RR_{\geq 0}^{2k-1}.\]

Using the discrete topology on $\setTkt$ and the usual one on $\RR^{2k-1}_{\geq 0}$,
this gives a topology on $\setTkt \times \RR_{\geq 0}^{2k-1}$, 
and then it makes sense to speak of convergence in distribution in this space.
We can also speak of {\em density}, taking as reference measure
the product of the counting measure on $\setTkt$ and the Lebesgue measure on $\RR_{\geq 0}^{2k-1}$.
Finally we denote by $\Sh$ and $\Lab$ the natural projections
from $\setTkt \times \RR_{\geq 0}^{2k-1}$ to $\setTkt$ and $\RR_{\geq 0}^{2k-1}$, respectively.
In words $\Sh$ erases the labels and outputs the {\em shape} of the tree,
while $\Lab$ outputs the vector of \emph{labels}.

\subsubsection{The limit tree}
\label{sec:limit_tree}

Throughout this and the following sections we let $\bm T$ denote a (non-degenerate) critical Galton--Watson tree having an aperiodic offspring distribution $\xi$.
We also assume that $\xi$ has finite variance $\sigma^2$.
We fix a subset $\Omega \subseteq \Z_{\geq 0}$ satisfying
\begin{align}
\P(\xi \in \Omega) > 0.
\end{align}
Given a rooted plane tree $T$, we let $|T|_\Omega$ denote the number of vertices $v \in T$ that have out-degree $d_T^+(v) \in \Omega$. 
For any value $n\in\Z_{>0}$ such that $\P(|\bm T|_\Omega=n)>0$,
we let $\bm T_n^\Omega$ denote the result of conditioning the tree $\bm T$ on $|\bm T|_\Omega = n$.
The goal is to describe the limit of $R^{[t]}(\bm T_n^\Omega,\myvec{\rv{v}})$,
where $\myvec{\rv{v}}=(\bm v_1,\dots,\bm v_k)$ are independently and uniformly chosen vertices of $\bm T_n^\Omega$,
{\em conditioned} to have out-degree in $\Omega$.
\medskip

We first recall the definition of simply and doubly size-biased versions of $\xi$,
namely the random variables $\hat{\xi}$ and $\xi^*$ with distributions
\begin{align}
\label{eq:xihat}
\P(\hat{\xi} = i) &= i\cdot \P(\xi = i), \\
\label{eq:xistar}
\P(\xi^* = i) &= i(i-1)\P(\xi = i) / \sigma^2.
\end{align}
Furthermore, for any fixed integer $k \in \Z_{>0}$ we say that a \emph{proper $k$-tree} is a rooted plane tree that has precisely $k$ leaves, labeled from $1$ to $k$, such that the root has out-degree $1$ and all other internal vertices have out-degree $2$.  Note that each such tree has $2k -1$ edges
and that there are $k! \mathsf{Cat}_{k-1}= 2^{k-1} \prod_{i=1}^{k-1}(2i-1)$ such trees.
Indeed, up to the single edge attached to the root, these trees are complete binary trees with $k$ leaves and a labeling of these leaves.
In the following, we order the edges of proper $k$-tree in some canonical order (e.g.\ depth first search order),
so that we can speak of the $i$-th edge of the tree; the chosen order is not relevant though.

\begin{figure}[htbp]
	\begin{minipage}[c]{0.6\textwidth}
		\centering
		\includegraphics[height=6.5cm]{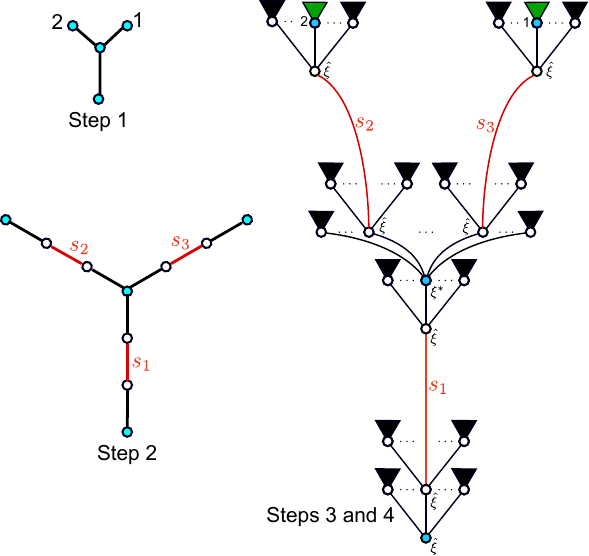}
	\end{minipage}
	\begin{minipage}[c]{0.39\textwidth}
		\caption{The construction of the limit tree ${\bm T}^{k,t}_\Omega$ for $k=2$ and $t=1$. The essential vertices are colored blue, and the middle edges are colored red. Each occurrence of $\hat{\xi}$ or $\xi^*$ at the side of a vertex represents that this vertex receives offspring according to an independent copy of the corresponding random variable (step 3). Each black triangle represents an independent copy of the Galton--Watson tree $\bm T$ (step 4). The green triangles represent independent copies of $\bm T$ conditioned on having root degree in $\Omega$ (step 4).\label{fig:ktree} }
	\end{minipage}
\end{figure}

For each integer $t \ge 1$ we can now construct a random rooted plane tree
${\bm T}^{k,t}_\Omega$ with $k$ distinguished vertices labeled from $1$ to $k$ having out-degree in $\Omega$, 
and $2k-1$ edges carrying length-labels.
We will prove later that this tree is the limit of $R^{[t]}(\bm T_n^\Omega,\myvec{\rv{v}})$.
A special case of this construction is illustrated in \cref{fig:ktree}. 
The general procedure goes as follows:
\begin{itemize}[]
	\item \emph{(Pick a skeleton)} Draw a proper $k$-tree uniformly at random. 
	Its leaves will correspond to the distinguished labeled vertices of ${\bm T}^{k,t}_\Omega$.
	Each possible outcome of this step is attained with probability
	\[
	\frac{1}{2^{k-1} \prod_{i=1}^{k-1}(2i-1)}.
	\]
	\item \emph{(Stretch it)} Select a vector $\myvec{\rv{s}} = (s_i)_i \in \mathbb{R}_{>0}^{2k-1}$ at random with density 
	\begin{align}
	\label{eq:density}
	\textstyle  (3 \cdot 5 \cdots (2k-3)) \, (\sum_i s_i) \, \exp \left( - \tfrac{(\sum_i s_i)^2}{2} \right).
	\end{align}
	It is simple to check that this defines a probability distribution,
	using classical expressions for absolute moments of the Gaussian distribution.
	For each $1 \le i \le 2k-1$, we replace the $i$-th edge of the $k$-tree by a path of length $2t+1$
	and assign label $s_i$ to the central edge of this path. 
	\item \emph{(Thicken it)} Each internal vertex receives additional offspring, independently from the rest. Here vertices with out-degree $1$ receive additional offspring according to an independent copy of $\hat{\xi} -1$, 
	while vertices with out-degree $2$ receive additional offspring according to an independent copy of $\xi^*-2$. An ordering of the total offspring that respects the ordering of the pre-existing offspring is chosen uniformly at random.
	\item \emph{(Graft branches)}
	Each distinguished vertex (i.e.\ each leaf of the original $k$-tree)
	becomes the root of an independent copy of $\bm T$ conditioned on having root-degree in $\Omega$.
	Other leaves of the tree resulting from step 3 
	become the roots of independent copies of Galton--Watson trees $\bm T$, 
	without conditioning.
\end{itemize}

The distribution of the random tree ${\bm T}^{k,t}_\Omega$, seen as an element of $\setTkt \times \RR_{\geq 0}^{2k-1}$, was determined in \cite[Lemma 4.1]{borga2020decorated}. This was a key step in the proof of the \cref{le:semilocal} below. Nevertheless, since we opt for skipping the technical details of the proof of \cref{le:semilocal}, we also avoid to show the quite complicated expression of this distribution.

\subsubsection{Convergence}
\label{sec:semiconv}

The following lemma extends Aldous' skeleton decomposition~\cite[Eq.\ (49)]{MR1207226} by keeping track of $o(\sqrt{n})$-neighborhoods 
near the essential vertices of the skeleton. 
The $o(\sqrt{n})$-threshold is sharp 
(for the applications in this manuscript,
the convergence of $t_n$-neighborhoods for any sequence $t_n$ 
tending to infinity would suffice).
We note that $o(\sqrt{n})$-neighborhoods of the root
have been previously considered in the literature,
{\em e.g.}\ by Aldous \cite{MR1085326,MR1166406}
and Kersting \cite{kersting2011height};
see also \cite[Theorem 5.2]{stufler2019offspring}
for a result on the $o(\sqrt{n})$-neighborhood
of a uniform random vertex in the tree.
Besides,
Lemma~\ref{le:semilocal} is also related to scaling limits obtained by Kortchemski~\cite{MR2946438}
and Rizzolo~\cite{MR3335013}, that imply  convergence of $R(\bm T_n^\Omega, \myvec{\rv{v}})$.

\medskip

We recall that we see trees of the form $s.R^{[t]}(T,\myvec{u})$ and ${\bm T}^{k, t}_\Omega$
as elements of $\setTkt \times \RR_{\geq 0}^{2k-1}$ as explained in \cref{ssec:skeleton_space}.

\begin{lem}[{\cite[Lemma 4.2]{borga2020decorated}}]
	\label{le:semilocal}
	Suppose that the offspring distribution $\xi$ is critical, aperiodic, and has finite variance $\sigma^2$. Let $\myvec{\rv{v}}$ be a vector of $k \ge 1$ independently and uniformly selected vertices with out-degree in $\Omega$ of the conditioned tree $\bm T_n^\Omega$. Then for each positive integer $t$ it holds that
	\begin{align}
	\label{eq:weaker}
	c_\Omega \sigma n^{-1/2}. R^{[t]}(\bm T_n^\Omega, \myvec{\rv{v}})	\xrightarrow{d} {\bm T}^{k, t}_\Omega,
	\end{align}
	with $c_\Omega= \sqrt{\P(\xi \in \Omega)}$. Even stronger, for each sequence $t_n = o(\sqrt{n})$ of positive integers it holds that
	\begin{align}
	\label{eq:toshow}
	\sup_{A,B} \left| \P\big[ c_\Omega \sigma n^{-1/2}.R^{[t_n]}(\bm T_n^\Omega, \myvec{\rv{v}})
	\in A \times B\big] - \P\big[{\bm T}^{k, t_n}_\Omega \in A \times B\big] \right| \to 0,
	\end{align}
	with $A$ ranging over all subsets of $\mathcal T_{k,\Omega}^{[t_n]}$,
	and $B$ over open intervals of $\RR_{\geq 0}^{2k-1}$.
\end{lem}

We skip in this manuscript the (slightly technical and quite long) proof of this lemma. The following statement will be useful (with $\Omega = \{0\}$, i.e.\ marking leaves) in the special case of separable permutations. It was deduced from some arguments used in the proof of \cref{le:semilocal}.

\begin{lem}[{\cite[Corollary 4.4]{borga2020decorated}}]
	Let the offspring distribution $\xi$ be critical, aperiodic, and have a finite variance. 
	Let $\myvec{\rv{v}}$ be a vector of $k \ge 1$ independently and
	uniformly selected vertices with out-degree in $\Omega$  of the conditioned tree $\bm T_n^\Omega$. 
	Then, for any fixed $t$, asymptotically as $n \to \infty$, the parities of the heights of the essential vertices induced by $\myvec{\rv{v}}$ (except the root of $\bm T_n^\Omega$) converge to  
	Bernoulli random variables of parameter $1/2$, independent among themselves, 
	and from the tree $\Sh(1.R^{[t]}(\bm T_n^\Omega, \myvec{\rv{v}}))$. 
	\label{cor:semilocal3}
\end{lem}

\subsection{Universality of the biased Brownian separable permuton}
\label{sec:scaling}

We start by exploring some properties of the biased Brownian separable permuton.

\subsubsection{Background on the biased Brownian separable permuton}\label{sect:sep_perm_info}

Recall (see \cref{thm:randompermutonthm} page \pageref{thm:randompermutonthm}) that if $\bm{\sigma}_n$ has size $n$, then these assertions are equivalent:
\begin{itemize}
	\item There exists a permuton $\bm{\mu}$ such that $\bm{\mu}_{\bm{\sigma}_n} \xrightarrow{d} \bm{\mu}$.
	\item For any integer $k \ge 1$ the pattern $\pat_{\rv{I}_{n,k}}(\bm{\sigma}_n)$ induced by a uniform random $k$-element subset $\rv{I}_{n,k} \subseteq [n]$ admits a distributional limit $\bm{\rho}_k$.
\end{itemize} 
In this case, the limiting family $(\bm{\rho}_k)_{k}$ is consistent (recall \cref{Def:consistency}, page~\pageref{Def:consistency}) and uniquely determine the distribution of the permuton $\bm{\mu}$. Indeed, $\bm{\mu}$ is the permuton limit of the sequence of permutations  $(\bm{\rho}_k)_{k}$ (recall \cref{Prop:existence_permuton}, page~\pageref{Prop:existence_permuton}).

\medskip

We can now define the biased Brownian separable permuton. Recall the bijection $\CanTree^{-1}$, defined at page \pageref{def:cantree}: it sends every canonical tree (see \cref{defintro:CanonicalTree}, page \pageref{defintro:CanonicalTree}) to the corresponding permutation in $\mathcal C$.
The following permutons were introduced in~\cite{bassino2018separable,bassino2017universal} 
where they were proved to be the limit of some substitution-closed classes:
\begin{itemize}
	\item The \emph{Brownian separable permuton} corresponds to the case where $\bm{\rho}_k$ is 
	the image by $\CanTree^{-1}$ of a uniform binary plane tree with $k$ leaves
	with uniform independent decorations from $\{\oplus, \ominus\}$ on its internal vertices (recall from \cref{rk:CT-1} that $\CanTree^{-1}$ can be applied to $\{\oplus, \ominus\}$-decorated trees,
	where neighbors may have the same sign). 
	\item Let $p\in[0,1]$ be fixed. The \emph{biased Brownian separable permuton} of parameter $p$ is constructed in the same way, but instead of assigning the $\{\oplus,\ominus\}$-decorations via fair coin flips, we toss a biased coin that shows $\oplus$ with probability $p$.
\end{itemize}

\begin{rem}\label{rem:maazoun_sep}
	Maazoun~\cite{maazoun17BrownianPermuton} gave a construction of the biased Brownian separable permutons in terms of decorated Brownian excursions. We do not present this point of view in this thesis, but we will discuss a (conjectural) equivalent construction of these permutons from Brownian excursions in \cref{sect:skew_perm} (see \cref{conj:Baxt_brow_same}, page~\pageref{conj:Baxt_brow_same}).
	
	We also mention that Maazoun~\cite[Theorem 1.5]{maazoun17BrownianPermuton} showed that, almost surely, the support of $\bm \mu^{(p)}$ is totally disconnected, and its Hausdorff dimension is 1 (with one-dimensional Hausdorff measure bounded above by $\sqrt{2}$). In addition, he showed that $\bm \mu^{(p)}$ inherits the self-similarity properties of the associated Brownian excursion, that is, $\bm \mu^{(p)}$ contains a lot of rescaled distributional copies of itself (see \cite[Theorem 1.6]{maazoun17BrownianPermuton} for a precise statement). Finally, in \cite[Theorem 1.7]{maazoun17BrownianPermuton}, he was able to calculate an expression for the intensity measure $\E[\bm \mu^{(p)}]$.
\end{rem}

Putting together the pattern characterization 
of permuton convergence (recalled above), 
the description of $\bm{\rho}_k$, 
and the connection between patterns and subtrees (explained in \cref{ssec:patterns_subtrees}),
we get a sufficient condition for convergence to a
biased Brownian separable permuton.

To state it, we recall that, if $\myvec\ell$ is an ordered sequence of marked leaves in a tree $T$,
then $R(T,\myvec\ell)$ denotes the subtree consisting of these marked leaves
and all their ancestors.
In addition, we denote by $R^\star(T,\myvec\ell)$ the tree obtained from $R(T,\myvec\ell)$
by successively removing all non-root vertices of out-degree $1$, merging their two adjacent edges.
\begin{lem}[{\cite[Lemma 5.1]{borga2020decorated}}]
	\label{lem:convergence_via_induced_subtrees}
	Let $p$ be fixed in $[0,1]$ and, for each $n \ge 1$, $\bm{\sigma}_n$ be a random permutation of size $n$.
	For each fixed $k \ge 1$, we take a uniform random sequence $\myvec{\bm{\ell}}= (\bm{\ell}_1, \ldots, \bm{\ell}_k)$ of $k$ leaves
	in the canonical tree $\bm{T}_n$ of $\bm{\sigma}_n$.
	We make the following assumptions.
	\begin{itemize}
		\item The tree $R^\star(\bm{T}_n,\myvec{\bm{\ell}})$ converges
		in distribution to a proper $k$-tree.
		\item For each non-root internal vertex $\bm u$ of $R^\star(\bm{T}_n,\myvec{\bm{\ell}})$, 
		we choose arbitrarily two leaves from $\myvec{\bm{\ell}}$, say $\bm{\ell}_{i_u}$ and $\bm{\ell}_{j_u}$,
		whose common ancestor is $\bm u$.
		We then assume that $\bm{\ell}_{i_u}$ and $\bm{\ell}_{j_u}$
		form a non-inversion asymptotically
		with probability $p$, and that,
		when $\bm u$ runs over non-root internal vertices of $R^\star(\bm{T}_n,\myvec{\bm{\ell}})$,
		these events are asymptotically
		independent from each other and from the shape $R^\star(\bm{T}_n,\myvec{\bm{\ell}})$.
	\end{itemize}
	Then $\bm{\sigma}_n$ converges to the biased Brownian separable permuton of parameter $p$.
\end{lem}

The arbitrary choices made in the second item above are irrelevant.
Indeed, when $\bm u$ has out-degree $2$ in $R^\star(\bm{T}_n,\myvec{\bm{\ell}})$
(which is the case with high probability under the first assumption),
the fact that $\bm{\ell}_{i_u}$ and $\bm{\ell}_{j_u}$ form an inversion
or not does not depend on the choice of $\bm{\ell}_{i_u}$ and $\bm{\ell}_{j_u}$
(this a simple consequence of the discussion from \cref{ssec:patterns_subtrees}).

\subsubsection{Permuton convergence of random permutations in substitution-closed classes}

We now prove our main theorem, Theorem~\ref{thm:scaling_intro}.

\begin{proof}[Proof of \cref{thm:scaling_intro}]
	By \cref{prop:giant_comp_perm} page \pageref{prop:giant_comp_perm}, it suffices to show that the uniform $n$-sized permutation $\bm{\sigma}_n$ from $\cC_{\nonp}$ satisfies
	\[
	\bm{\mu}_{\bm{\sigma}_n} \xrightarrow{d} \bm{\mu}^{(p)}.
	\]

	We first consider the separable case $\mathfrak S= \emptyset$. 
	Let $\bm T_n$ be the canonical tree of $\bm \sigma_n$.
	Here a vertex of $\bm T_n$ is decorated with $\ominus$ if and only if it has even height. 
	Without its decorations, $\bm T_n$ has the law of a critical Galton-Watson tree with finite variance
	conditioned on having $n$ leaves
	(see \cref{sec:indecomposable_GW}; 
	for the separable case, packed trees and canonical trees only differ by their decorations).
	
	Let $k \ge 1$ be given and $\myvec{\bm{\ell}} = (\bm{\ell}_1, \ldots, \bm{\ell}_k)$ be a uniform random sequence of leaves in $\bm T_n$. 
	It follows from Lemma~\ref{le:semilocal} 
	that $R^\star(\bm T_n, \myvec{\bm{\ell}})$ is asymptotically a uniform random proper $k$-tree.
	Corollary~\ref{cor:semilocal3} yields the additional information
	that the parities of the lengths of the $2k-1$ paths in $\bm T_n$ 
	corresponding to the edges of $R^\star(\bm T_n, \myvec{\bm{\ell}})$ 
	converge jointly to $2k-1$ independent fair coin flips,
	independently of the shape $R^\star(\bm T_n, \myvec{\bm{\ell}})$. 
	Hence in the limit as $n \to \infty$ each non-root internal vertex of $R^\star(\bm T_n, \myvec{\bm{\ell}})$
	receives a sign $\oplus$ or $\ominus$ with probability $1/2$
	(meaning that the corresponding leaves
	form an inversion with probability $1/2$, in the sense of the second item of \cref{lem:convergence_via_induced_subtrees}).
	Moreover, these events are asymptotically independent from each other
	and from the shape $R^\star(\bm T_n, \myvec{\bm{\ell}})$. 
	As this holds for all $k\ge 1$, thanks to \cref{lem:convergence_via_induced_subtrees},
	it follows that $\bm{\sigma}_n$ converges in distribution to the Brownian separable permuton $\bm{\mu}^{(1/2)}$.\vspace{3 mm}
	
	Let us now consider the case $\mathfrak S \neq \emptyset$. 
	In this case, it is more convenient to work with packed trees rather than canonical trees
	(note however that both trees have the same set of leaves).
	In particular, the random packed tree $\bm{P}_n=(\bm{T}_n,\bm{\lambda}_{\bm{T}_n})$
	associated with the uniform permutation $\bm \sigma_n$ in $\cC_{\nonp}$ 
	is a Galton-Watson tree with a specific offspring distribution $\xi$ (defined in \cref{eq:offspring_distribution_packed_tree} page \pageref{eq:offspring_distribution_packed_tree})
	conditioned on having $n$ leaves, with independent random decorations on each vertex
	(see \cref{sec:indecomposable_GW}). 
	As before, we fix $k \ge 1$ and consider a uniformly selected $k$-tuple of distinct leaves 
	$\myvec{\bm{\ell}} = (\bm{\ell}_1, \ldots, \bm{\ell}_k)$ in $\bm T_n$.
	
	By \cref{le:semilocal}, we know that
	$R^\star(\bm{T}_n,\myvec{\bm{\ell}})$ converges
	in distribution to a uniform proper $k$-tree
	(recall that $\xi$ is always aperiodic and that
	it has expectation $1$ and finite variance by assumption, as needed to apply \cref{le:semilocal}).
	In particular, the tree $R^\star(\bm{T}_n,\myvec{\bm{\ell}})$ is a proper $k$-tree 
	(with a root of out-degree $1$ and other internal vertices of out-degree $2$) 
	with high probability, as $n \to \infty$.
	When this is the case, since the packing construction only merges internal vertices,
	$R^\star(\bm{T}_n,\myvec{\bm{\ell}})$ coincide with 
	$R^\star(\tilde{\bm{T}}_n,\myvec{\bm{\ell}})$,
	where $\tilde{\bm{T}}_n$ is the {\em canonical} tree associated with $\bm{\sigma}_n$.
	Therefore, although \cref{lem:convergence_via_induced_subtrees} is stated with the canonical tree $\tilde{\bm{T}}_n$,
	we can use it here with the packed tree $\bm T_n$ instead.
	
	Using the notation of \cref{lem:convergence_via_induced_subtrees},
	it remains to analyse whether $\bm\ell_{i_u}$ and $\bm\ell_{j_u}$ form an inversion
	or not (for non-root internal vertices $\bm u$ of $R^\star(\bm{T}_n,\myvec{\bm{\ell}})$).
	
	We recall (see \cref{ssec:patterns_subtrees})
	that if $\bm u$ is decorated with an $\mathfrak S$-gadget,
	then whether $\bm \ell_{i_u}$ and $\bm \ell_{j_u}$ form an inversion or not
	is determined by the decoration of $\bm u$ and by which branches attached to $\bm u$
	contain $\bm \ell_{i_u}$ and $\bm \ell_{j_u}$.
	This information is contained in $\Sh(s.R^{[0]}(\bm{T}_n,\myvec{\bm{\ell}}))$ for any $s>0$.
	
	On the other hand, if $\bm u$ is decorated with $\circledast$, 
	then in order to determine whether $\bm \ell_{i_u}$ and $\bm \ell_{j_u}$ form an inversion or not, 
	we have to recover the parity of the distance of $\bm u$
	to its first ancestor decorated with an $\mathfrak S$-gadget 
	(if it exists, otherwise to the root of $\bm T_n$). 
	
	Take $t_n$ tending to infinity with $t_n=o(\sqrt{n})$.
	By Lemma~\ref{le:semilocal}, $\bm u$ has asymptotically $t_n$ ancestors 
	with out-degrees $\hat{\xi}_1, \hat{\xi}_2, \ldots, \hat{\xi}_{t_n}$ 
	being independent copies of $\hat{\xi}$ defined in \cref{eq:xihat}. 
	Moreover the vertex $\bm u$ and each of its ancestors receive a decoration
	that gets drawn independently and uniformly at random 
	among all $\widehat{\GGG(\mathfrak{S})}$-decorations with size equal to the out-degree of the vertex. 
	In this setting, with high probability, one of the $t_n$ ancestors will receive an $\mathfrak{S}$-gadget as decoration. 
	Therefore, with high probability, whether $\bm \ell_{i_u}$ and $\bm \ell_{j_u}$ form an inversion
	is determined by $\Sh(s.R^{[t_n]}(\bm{T}_n,\myvec{\bm{\ell}}))$ for any $s>0$.
	
	We say that two families (indexed by $\Z_{>0}$) of probability distributions are \emph{close} 
	when their total variation distance tends to $0$ as $n$ tends to infinity. 
	By Lemma~\ref{le:semilocal}, the distributions of the random trees $\Sh(s_n.R^{[t_n]}(\bm{T}_n,\myvec{\bm{\ell}}))$
	and $\Sh(\bm{T}^{k,t_n}_{\{0\}})$ are close, 
	for a well-chosen sequence $s_n$. 
	From the above discussion, this implies that
	the joint distributions of 
	$R^\star(\bm{T}_n,\myvec{\bm{\ell}})$ and
	\begin{equation}
		\left(\mathds{1}_{\left\{\text{$\bm \ell_{i_{\bm u}}$ and $\bm \ell_{j_{\bm u}}$ form an inversion
				in $(\bm{T}_n,\myvec{\bm{\ell}})$}\right\}}\right)_{\bm u}
	\end{equation}
	are close to the distributions of the same variables in the limiting tree $\bm{T}^{k,t_n}_{\{0\}}$.
	When $n$ tends to infinity, these tend a.s.\ (with the obvious coupling between the $\bm{T}^{k,t_n}_{\{0\}}$)
	to the same variables in $\bm{T}^{k,\bm{t}^*}_{\{0\}}$,
	where $\bm{t}^*$ denotes the minimal radius such that each internal 
	$\circledast$-decorated essential vertex (different from the root) has an ancestor decorated by an $\mathfrak{S}$-gadget.
	
	In the limiting tree $\bm{T}^{k,\bm{t}^*}_{\{0\}}$,
	the neighborhoods of the essential vertices $\bm u$ 
	are independent from each other
	and all have the same distribution (which does not depend on $k$,
	nor on the shape $R^\star(\bm{T}^{k,\bm{t}^*}_{\{0\}},\myvec{\bm{\ell}})$,
	the latter being the proper $k$-tree taken at step 1 of the construction).
	Therefore the probability that $\bm \ell_{i_u}$ and $\bm \ell_{j_u}$ form a non-inversion
	tend to some parameter $p$ in $[0,1]$,
	which depends only on the permutation class $\mathcal C$ we are working with.
	Moreover these events are asymptotically independent from each other 
	and from the shape $R^\star(\bm{T}_n,\myvec{\bm{\ell}})$.
	From \cref{lem:convergence_via_induced_subtrees}, this implies that $\bm{\mu}_{\bm{\sigma}_n} \xrightarrow{d} \bm{\mu}^{(p)}$.\medskip

	It remains to calculate the explicit expression for the limiting probability $p$ given in \cref{eq:parm_p}. For this we refer to the final part of the proof of \cite[Theorem 5.2]{borga2020decorated}.
\end{proof}

\section{Baxter permutations and related objects}\label{sect:baxperm}

In the previous section we saw that the biased Brownian separable permuton is a universal limiting object. Nevertheless, here we show that there exists a model of constrained random permutations that does not converge to the biased Brownian separable permuton but that exhibit a new interesting random limiting permuton.
In particular, this section is devoted to the proof of (more precise versions of) \cref{thm:joint_intro}. As explained before, this proof is based on some results on scaling limits of coalescent-walk processes.

\subsection{Scaling limits of coalescent-walk processes}\label{sec:coalescent}

Here we deal with scaling limits of discrete coalescent-walk processes both in the finite and infinite-volume case. The results culminate in \cref{thm:discret_coal_conv_to_continuous}, upon which the proofs of the two main theorems for Baxter permutations and bipolar orientations rely, namely \cref{thm:permuton} and \cref{thm:joint_scaling_limits}. Nevertheless, we believe that our intermediate results, \cref{thm:coal_con_uncond} and \cref{prop:coal_con_cond}, are of independent interest and might be useful also for some future projects.

All the spaces of continuous functions considered in this section are implicitly endowed with the topology of uniform convergence on every compact set.

\subsubsection{The continuous coalescent-walk process}

We start by defining the candidate continuous limiting object for discrete coalescent-walk process: it is formed by the solutions of the following family of stochastic differential equations (SDEs) indexed by $u\in \R$ and driven by a two-dimensional process $\conti W = (\conti X,\conti Y)$:
\begin{equation}\label{eq:flow_SDE}
\begin{cases}
d\conti Z^{(u)}(t) = \idf_{\{\conti Z^{(u)}(t)> 0\}} d\conti Y(t) - \idf_{\{\conti Z^{(u)}(t)\leq 0\}} d \conti X(t),& t\geq u,\\
\conti Z^{(u)}(t)=0,&  t\leq u.
\end{cases} 
\end{equation}
Existence and uniqueness of a solution of the SDE above (for $u\in \R$ fixed) were already studied in the literature in the case where the driving process $\conti W$ is a Brownian motion, in particular with the following result.

\begin{thm}[Theorem 2 of \cite{MR3098074}, Proposition 2.2 of \cite{MR3882190}]\label{thm:ext_and_uni}
	\label{thm:uniqueness}
	Fix $\rho \in (-1,1)$. Let $T\in (0,\infty]$ and let $\conti W = (\conti X,\conti Y)$ be a standard two-dimensional Brownian motion of correlation $\rho$ and time-interval $[0,T)$.
	We have pathwise uniqueness and existence of a strong solution for the SDE:
	\begin{equation} \label{eq:SDE}
	\begin{cases}
	d \conti Z(t) &= \idf_{\{\conti Z(t)> 0\}} d \conti Y(t) - \idf_{\{\conti Z(t)\leq 0\}} d \conti X(t), \quad  0\leq t < T,\\
	\conti Z(0)&=0.
	\end{cases} 
	\end{equation}
	Namely, letting $(\Omega, \mathcal F, (\mathcal F_t)_{0\leq t < T}, \P)$ be a filtered probability space satisfying the usual conditions, and assuming that $\conti W$ is an $(\mathcal F_t)_t$-Brownian motion, then: 
	\begin{itemize}
		\item If $\conti Z,\widetilde{\conti Z}$ are two $(\mathcal F_t)_t$-adapted continuous processes that verify \cref{eq:SDE} almost surely, then $\conti Z=\widetilde{\conti Z}$ almost surely.
		\item There exists an  $(\mathcal F_t)_t$-adapted continuous processes $\conti Z$ which verifies \cref{eq:SDE} almost surely.
	\end{itemize}
	In particular, there exists for every $t\in (0,T]$ a measurable mapping $\solution_t : \mathcal C([0,t)) \to \mathcal C([0,t))$, called the \emph{solution mapping}, such that:
	\begin{itemize}
		\item Setting $\conti Z = \solution_t(\conti W|_{[0,t)})$, then $\conti Z$ verifies \cref{eq:SDE} almost surely on the interval $[0,t)$.
		\item For every $0\leq s \leq t \leq T$, then $F_t(\conti W|_{[0,t)})|_{[0,s)} = F_s (\conti W|_{[0,s)})$ almost surely.
	\end{itemize} 
\end{thm}

Assume from now on that $\conti W = (\conti X, \conti Y)$ is a standard two-dimensional Brownian motion of correlation $-1/2$ and time-interval $\R$. Let $(\mathcal F_t)_{t\in \R}$ be the canonical filtration of $\conti W$. For every $u\in \R$ we define
\[\conti Z^{(u)}(t) = \begin{cases}
F_\infty\Big((\conti W(u + s) - \conti W(u))_{0\leq s <\infty}\Big)(t-u),& t\geq u,\\
0, &t<u.
\end{cases}\] 

It is clear that $\conti Z^{(u)}$ is $(\mathcal F_t)_t$-adapted. Because \cref{eq:flow_SDE} is invariant by addition of a constant to $\conti W$, and because $\conti W(u + s) - \conti W(u)$ is a Brownian motion with time-interval $\RR_{\geq 0}$, we see that for every fixed $u\in\R$, $\conti Z^{(u)}$ verifies \cref{eq:flow_SDE} almost surely.

Our construction makes the mapping $(\omega,u)\mapsto \conti Z^{(u)}$ jointly measurable. Hence by Tonelli's theorem, for almost every $\omega$, $\conti Z^{(u)}$ is a solution for almost every $u$. 

\begin{rem}\label{rem:no_flow}
	Given $\omega$ (even restricted to a set of probability one), we cannot say that $\{\conti Z^{(u)}\}_{u\in \R}$ forms a whole field of solutions to \cref{eq:flow_SDE} driven by $\conti W$. Indeed, we cannot guarantee that the SDE holds for all $u\in\R$ simultaneously. In fact, we expect thanks to intuition coming from Liouville Quantum Gravity, that there exist exceptional times $u\in\R$ where uniqueness fails, with two or three distinct solutions. This phenomenon is also observed in another coalescing flow of an SDE \cite{MR2094439} and in the Brownian web \cite{MR3644280}.
\end{rem}

\begin{rem}\label{rem:solution_of_SDE_are_BM}
	By Lévy's characterization theorem  \cite[Theorem 3.6]{MR1725357}, for every fixed $u\in\mathbb R$, the process $\conti Z^{(u)}$ is a standard one-dimensional Brownian motion on $[u,\infty)$ with $\conti Z^{(u)}(u) = 0$. Note however that the coupling of $\conti Z^{(u)}$ for different values of $u\in\R$ is highly nontrivial.
\end{rem}

\begin{defn}
	We call \emph{continuous coalescent-walk process} (driven by $\conti W$) the collection of stochastic processes $\left\{\conti Z^{(u)}\right\}_{u\in \R}$. 
\end{defn}

Since for all $u\in\R$, $(\conti Z^{(u)}(t))_{t\geq u}$ is a Brownian motion, one can define almost surely its local time process at zero  $\conti L^{(u)}$ (see \cite[Chapter VI]{MR1725357}): namely for $t\geq u$, $\conti L^{(u)}(t)$ is the limit in probability of
\[\frac 1 {2\eps} \Leb\left(\left\{s\in [u,t]: |\conti Z^{(u)}(s)| <\eps \right\}\right).\]
By convention we set $\conti L^{(u)}(t) = 0$ for $t<u$ so that  $\conti L^{(u)}$ is a continuous process on $\R$.

In the next section we show that the \emph{continuous coalescent-walk process} is the scaling limit of the discrete infinite-volume coalescent-walk processes defined in \cref{sect: bij_walk_coal}.

\subsubsection{The unconditioned scaling limit result}\label{sect:scal_coal_proc}

Let $\overline{\bm W} = (\overline{\bm X},\overline{\bm Y}) =(\overline{\bm X}_k,\overline{\bm Y}_k)_{k\in \Z}$ be the random two-dimensional walk (with step distribution $\nu$) defined below \cref{eq:step_distribution_walk} page~\pageref{eq:step_distribution_walk}, and let $\overline{\bm Z} = \wcp(\overline{\bm W})$ be the corresponding discrete coalescent-walk process. Let also $(\overline {\bm L}^{(i)}(j))_{-\infty<i\leq j <\infty} = L_{\overline{\bm Z}}$ be the local time process of $\overline{\bm Z}$ as defined in \cref{eq:local_time_process} page~\pageref{eq:local_time_process}. By convention, from now on we extend trajectories of $\overline{\bm Z}$ and $\overline {\bm L}$ to the whole $\Z$ by setting $\overline{\bm Z}^{(j)}(i) = \overline {\bm L}^{(j)}(i) = 0$ for $i,j \in \Z$, $i<j$.
We define rescaled versions: for all $n\geq 1, u\in \R$, let $\overline {\conti W}_n:\R\to \R^2$, $\overline {\conti Z}^{(u)}_n:\R\to\R$ and $\overline {\conti L}^{(u)}_n:\R\to\R$ be the continuous functions defined by linearly interpolating the following points:
\begin{equation}\label{eq:rescaled_version}
\overline{\conti W}_n\left(\frac kn\right) = \frac 1 {\sqrt {2n}} \overline{\bm W}_{k}, \quad
\overline{\conti Z}^{(u)}_{n}\left(\frac kn\right) = \frac 1 {\sqrt {2n}} \overline{\bm Z}^{(\lceil nu\rceil)}_{k}, \quad
\overline{\conti L}^{(u)}_{n}\left(\frac kn\right) = \frac 1 {\sqrt {2n}} \overline{\bm L}^{(\lceil nu\rceil)}(k),\quad k\in \Z.
\end{equation}
Our first scaling limit result for infinite-volume coalescent-walk processes is the following result that deals with a single trajectory.
\begin{thm}\label{thm:coal_con_uncond}
	Fix $u\in \R$. 
	We have the following joint convergence in $\mathcal C(\R,\R)^{4}$: 
	\begin{equation}
	\label{eq:coal_con_uncond}
	\left(\overline{\conti W}_n,\overline{\conti Z}^{( u)}_n,\overline{\conti L}^{( u)}_n\right) 
	\xrightarrow[n\to\infty]{d}
	\left(\conti W,\conti Z^{(u)},\conti L^{(u)}\right).
	\end{equation}
\end{thm}
	
\begin{proof} The proof is spitted in two main parts.

\paragraph*{Convergence of $\overline{\conti W}_n$ and of $(\overline{\conti Z}^{( u)}_n,\overline{\conti L}^{( u)}_n)$.}
The first step in the proof is to establish component-wise convergence in 
\cref{eq:coal_con_uncond}.
By Donsker's theorem, upon computing $\Var(\nu) = \begin{psmallmatrix}2 &-1 \\ -1 &2 \end{psmallmatrix}$,
we get that the rescaled random walk $\overline{\conti W}_n=(\overline{\conti X}_n,\overline{\conti Y}_n)$ converges to $\conti W=(\conti X, \conti Y)$ in distribution. Using \cref{prop:trajectories_are_rw}, we know that a single trajectory of the discrete coalescent-walk process has the distribution of a random walk, and applying again an invariance principle such as \cite[Theorem 1.1]{MR749918}, we get that $(\overline{\conti Z}^{(u)}_n({u+t}), \overline{\conti L}^{(u)}_n({u+t}))_{t\geq 0}$ converges to a one-dimensional Brownian motion and its local time, which is indeed distributed like $(\conti Z^{(u)},\conti L^{(u)})$ thanks to \cref{rem:solution_of_SDE_are_BM}.

\paragraph{Joint convergence.}
The second step in the proof is to establish joint convergence in distribution. By Prokhorov's theorem, both $\overline{\conti W}_n$ and $(\overline{\conti Z}^{( u)}_n,\overline{\conti L}^{( u)}_n)$ are tight sequences of random variables. Since the product of two compact sets is compact, then the left-hand side of \cref{eq:coal_con_uncond} forms a tight sequence. Therefore, again by Prokhorov's theorem, it is enough to identify the distribution of all joint subsequential limits in order to show the convergence in \cref{eq:coal_con_uncond}. Assume that along a subsequence, we have 
\begin{equation}\label{eq:joint_conve_to_prove}
\left(\overline{\conti W}_n,\overline{\conti Z}^{( u)}_n,\overline{\conti L}^{( u)}_n\right) 
\xrightarrow[n\to\infty]{d}
\left(\conti W,\widetilde {\conti Z},\widetilde{\conti L}\right),
\end{equation}
where $\widetilde {\conti Z}$ is a Brownian motion started at time $u$, and $\widetilde {\conti L}$ is its local time process at zero. The joint distribution of the right-hand side is unknown for now, but we will show that $\widetilde {\conti Z} = \conti Z^{(u)}$ almost surely, which will complete the proof. 
Using Skorokhod's theorem, we may define all involved variables on the same probability space and assume that the convergence in \cref{eq:joint_conve_to_prove} is in fact almost sure.

Let $(\mathcal G_t)_t$ be the smallest complete filtration to which $\conti W$ and $\widetilde{\conti Z}$ are adapted, i.e.\ $\mathcal G_t$ is the completion of $\sigma(\conti W(s),\widetilde{\conti Z}(s), s\leq t)$ by the negligible events of $\P$.
We shall show that $\conti W$ is a $(\mathcal G_t)_t$-Brownian motion, that is for $t\in \R, s\in \R_{\geq 0}$, 
$$\left(\conti W(t+s)-\conti W(t)\right) \indep \mathcal G_t.$$ 
For fixed $n$, by definition of random walk, $\left(\overline{\conti W}_n({t+s}) - \overline{\conti W}_n(t)\right) \indep \sigma\left(\overline{\bm W}_k, k\leq \lfloor n t \rfloor\right)$. Therefore, 
\begin{equation}
\left(\overline{\conti W}_n({t+s})-\overline{\conti W}_n(t)\right) \ \indep\ \left(\overline{\conti W}_n(r),\overline{\conti Z}^{(u)}_n(r)\right)_{r\leq n^{-1}\lfloor nt \rfloor}.
\end{equation}
By convergence, we obtain that $\left(\conti W(t+s)-\conti W(t)\right)$ is independent of $\left(\conti W(r),\widetilde{\conti Z}(r)\right)_{r\leq t}$, completing the claim that $\conti W$ is a $(\mathcal G_t)_t$-Brownian motion.

Now fix a rational $\eps>0$ and an open interval with rational endpoints $(\bm a,\bm b)$ on which $\widetilde{\conti Z}(t)>\eps$. Note that the interval $(\bm a,\bm b)$ depends on $\widetilde{\conti Z}(t)$. By almost sure convergence, there is $N_0$ such that for $n\geq N_0$, $\overline{\conti Z}^{(u)}_n>\eps/2$ on $(\bm a,\bm b)$. 
On the interval $(\bm a+1/n,\bm b)$, the process $\overline{\conti Z}^{(u)}_n - \overline{\conti Y}_n$ is constant  by construction of the coalescent-walk process (see \cref{sect: bij_walk_coal}). 
Hence the limit $\widetilde{\conti Z} - \conti Y$ is constant too on $(\bm a,\bm b)$ almost surely. We have shown that almost surely $\widetilde{\conti Z} - \conti Y$ is locally constant on $\{t>u:\widetilde{\conti Z}(t)>\eps\}$. This translates into the following almost sure equality:
\[\int_{u}^t \idf_{\{\widetilde{\conti Z}(r)>\eps\}} d\widetilde{\conti Z}(r) = \int_u^t \idf_{\{\widetilde{\conti Z}(r)>\eps\}} d\conti Y(r), \quad t\geq u. \]
The stochastic integrals are well-defined: on the left-hand side by considering the canonical filtration of $\widetilde {\conti Z}$, on the right-hand side by considering $(\mathcal G_t)_t$.	
The same can be done for negative values, leading to 
\[\int_u^t \idf_{\{|\widetilde{\conti Z}(r)|>\eps\}} d\widetilde{\conti Z}(r)  =  \int_{u}^t \idf_{\{\widetilde{\conti Z}(r)>\eps\}} d\conti Y(r) -  \int_{u}^t \idf_{\{\widetilde{\conti Z}(r)<-\eps\}} d\conti X(r). \]
By stochastic dominated convergence theorem, one can take the limit as $\eps\to 0$, \cite[Thm. IV.2.12]{MR1725357}, and obtain 
\[\int_u^t \idf_{\{\widetilde{\conti Z}(r)\neq0\}}d\widetilde{\conti Z}(r)  =  \int_{u}^t \idf_{\{\widetilde{\conti Z}(r)>0\}} d\conti Y(r) -  \int_{u}^t \idf_{\{\widetilde{\conti Z}(r)<0\}} d\conti X(r). \]
Thanks to the fact that $\widetilde{\conti Z}$ is a Brownian motion, $\int_u^t \idf_{\{\widetilde{\conti Z}(r)=0\}}d\widetilde{\conti Z}(r) = 0$, so that the left-hand side of the equation above equals $\widetilde {\conti Z}(t)$.
As a result $\widetilde {\conti Z}$ verifies \cref{eq:SDE} almost surely and we can apply pathwise uniqueness (\cref{thm:uniqueness}, item 1) to complete the proof that $\widetilde {\conti Z} = \conti Z^{(u)}$ almost surely. 

\end{proof}

\subsubsection{The conditioned scaling limit result}
\label{sec:cond_conv}

In the previous section we saw a scaling limit result for infinite-volume coalescent-walk processes. We now deal with the finite-volume case, the one that we need for our results.

For all $n\geq 1$, let $\bm W_n$ be a uniform element of the space of tandem walks $\mathcal W_n$ and $\bm Z_n = \wcp(\bm W_n)$ be the corresponding uniform coalescent-walk process. Let also $\bm L_n = ( {\bm L}^{(i)}_{n}(j))_{1\leq i\leq j \leq n} = L_{{\bm Z_n}}$ be the local time process of $\bm Z_n$ as defined in \cref{eq:local_time_process} page~\pageref{eq:local_time_process}. For all $n\geq 1, u\in (0,1)$, let ${\conti W}_n:[0,1]\to \R^2$, ${\conti Z}^{(u)}_n:[0,1]\to\R$ and ${\conti L}^{(u)}_n:[0,1]\to\R$ be the continuous functions defined by linearly interpolating the following points defined for all $k\in [n]$,
\begin{equation}
{\conti W}_n\left(\frac kn\right) = \frac 1 {\sqrt {2n}} \bm W_{n}(k), \quad
{\conti Z}^{(u)}_{n}\left(\frac kn\right) = \frac 1 {\sqrt {2n}} \bm Z^{(\lceil nu\rceil)}_{n}(k), \quad
{\conti L}^{(u)}_{n}\left(\frac kn\right) = \frac 1 {\sqrt {2n}} \bm L^{(\lceil nu\rceil)}_{n}(k).
\end{equation}Our goal is to obtain a scaling limit result for these processes in the fashion of \cref{thm:coal_con_uncond}.

Let $\conti W_e$ denote a two-dimensional Brownian excursion of correlation $-1/2$ in the non-negative quadrant and denote by $(\Omega, \mathcal F, (\mathcal F_t)_{0\leq t\leq 1}, \P_\exc)$ the completed canonical probability
space of $\conti W_e$. From now on we work in this space. The law of the process $\conti W_e$ is characterized (for instance) by \cref{prop:brown_ex} in \cref{sec:appendix}. 
Using \cref{prop:unif_law} and \cref{prop:DW}, we have that $\conti W_e$ is the scaling limit of $\conti W_n$. Then, the scaling limit of $\conti Z_n$ should be the continuous coalescent-walk process driven by $\conti W_e$, i.e.\ the collection (indexed by $u\in[0,1]$) of solutions of \cref{eq:flow_SDE} driven by $\conti W_e$. 

Let us remark that since Brownian excursions are semimartingales \cite[Exercise 3.11]{MR1725357}, it makes sense to consider stochastic integrals against such processes, so the SDE in \cref{eq:flow_SDE} driven by $\conti W_e$ is well-defined. We can also transport existence and uniqueness of strong solutions from \cref{thm:ext_and_uni} using absolute continuity arguments, obtaining the result in \cref{thm:ext_and_uni_excursion}, whose proof is omitted in this manuscript.

Denote by $\mathcal F^{(u)}_t$ the sigma-algebra generated by $\conti W_e(s) - \conti W_e(u)$  for $ u\leq s \leq t$ and completed by negligible events of $\P_\exc$.
\begin{thm}[{\cite[Theorem 4.6]{borga2020scaling}}]\label{thm:ext_and_uni_excursion}
	For every $u\in(0,1)$, there exists a continuous $\mathcal F^{(u)}_t$-adapted stochastic process $\conti Z_e^{(u)}$ on $[u,1)$, such that
	\begin{enumerate}
		\item  the mapping $(\omega,u)\mapsto \conti Z_e^{(u)}$ is jointly measurable. 
		\item For every $0<u<r<1$, $\conti Z_e^{(u)}$ verifies  \cref{eq:flow_SDE} with driving motion $\conti W_e$, restricted to the interval $[u,r]$, almost surely.
		\item If $0<u<r<1$ and $\widetilde {\conti Z}$ is an $\mathcal F^{(u)}_t$-adapted stochastic process that verifies \cref{eq:flow_SDE}  with driving motion $\conti W_e$ on interval $[u,r]$, then $\widetilde  {\conti Z} =  \conti Z_e^{(u)}$ on $[u,r]$ almost surely. 
	\end{enumerate}
\end{thm}

Now for $u\in (0,1)$ denote by $\conti Z_e^{(u)}$ the strong solution of \cref{eq:flow_SDE} driven by $\conti W_e$ provided by the previous theorem. Note that the process $\conti Z_e^{(u)}$ is a continuous process on the interval $[u,1)$. Since $\conti Z_e^{(u)}(u) = 0$, we extend continuously $\conti Z_e^{(u)}$ on $[0,1)$ by setting $\conti Z_e^{(u)}(t) = 0$ for $0\leq t\leq u$. It will turn out (see \cref{prop:coal_con_cond}) that $\conti Z_e^{(u)}$ can also be extended continuously at time $1$.

A remark similar to \cref{rem:no_flow} holds for the family $\{\conti Z_e^{(u)}\}_{u\in(0,1)}$, that is, we can only guarantee that for almost every $\omega$, $\conti Z_e^{(u)}$ is a solution of \cref{eq:flow_SDE} for almost every $u\in(0,1)$. Denote by $(\conti L_e^{(u)})_{u\leq t < 1}$ the local time process at zero of the semimartingale $\conti Z_e^{(u)}$ on $[u,1)$. By convention, set $\conti L_e^{(u)}(t) = 0$ for $0\leq t <u$.

\begin{defn}\label{defn:cont_coal_proc}
	We call \emph{continuous coalescent-walk process} (driven by $\conti W_e$) the collection of stochastic processes $\left\{\conti Z^{(u)}_e\right\}_{u\in(0,1)}$. 
\end{defn}

We can now state a scaling limit result for finite-volume coalescent-walk processes. We first deal with the case of a single trajectory $\conti Z_n^{(u)}$. Then we consider a more general case in \cref{thm:discret_coal_conv_to_continuous}. The key tools for the proof (here omitted) of the following result are the absolute continuity results in \cref{sec:appendix} used together with \cref{thm:coal_con_uncond}.

\begin{prop}[{\cite[Proposition 4.8]{borga2020scaling}}]\label{prop:coal_con_cond}
	Fix $u\in (0,1)$.
	The stochastic process ${\conti Z}_e^{( u)}$ can be extended to a continuous function on $[0,1]$ by setting ${\conti Z}_e^{( u)}(1) = 0$, and we
	have the following joint convergence in the product space of continuous functions $\mathcal C([0,1],\R^2)\times \mathcal C([0,1],\R)\times \mathcal C([0,1),\R)$: 
	\begin{equation}
	\label{eq:coal_con_cond}
	\left({\conti W}_n,{\conti Z}^{( u)}_n,{\conti L}^{( u)}_n \right)
	\xrightarrow[n\to\infty]{d}
	\left(\conti W_e,\conti Z_e^{(u)},\conti L_e^{(u)}\right).
	\end{equation}
\end{prop}

\begin{rem}
	Note that the convergence of local times does not go up to time 1. This will be corrected later in \cref{lem:local_time_does_not_disappear} for a uniformly random starting point.
\end{rem}

We finish by stating a simple generalization of the previous result, where we consider several uniform starting points. This is the foundation upon which the next section is built.
\begin{thm}[{\cite[Theorem 4.10]{borga2020scaling}}]
	\label{thm:discret_coal_conv_to_continuous}
	Let $(\bm u_i)_{i\in\Z_{>0}}$ be a sequence of i.i.d.\ uniform random variables on $[0,1]$, independent of all other variables. We have the following joint convergence in the product space of continuous functions $\mathcal C([0,1],\R^2)\times (\mathcal C([0,1],\R)\times \mathcal C([0,1),\R))^{\mathbb Z_{>0}}$:
	\begin{equation}
	\left(
	{\conti W}_n,\left({\conti Z}^{(\bm u_i)}_n, {\conti L}^{(\bm u_i)}_n\right)_{i\in\Z_{>0}}
	\right)
	\xrightarrow[n\to\infty]{d}
	\left( \conti {W}_e,\left({\conti Z}^{(\bm u_i)}_e, {\conti L}^{(\bm u_i)}_e\right)_{i\in\Z_{>0}}
	\right).
	\end{equation}
\end{thm}

\subsection{Scaling limits of Baxter permutations and bipolar orientations}
\label{sec:final}

This section is split in two parts: in the first one, we construct the Baxter permuton (see \cref{defn:Baxter_perm}) from the continuous coalescent-walk process $\conti Z_e= \{\conti Z_e^{(u)}\}_{u\in [0,1]}$ introduced in \cref{defn:cont_coal_proc}, and we show that it is the limit of uniform Baxter permutations (see \cref{thm:permuton}). We also show that this convergence holds jointly with the one for the coalescent-walk process (proved in \cref{thm:discret_coal_conv_to_continuous}). Building on these results, in the second part, we prove a joint scaling limit convergence result for all the objects considered in the commutative diagram in \cref{thm:diagram_commutes}, i.e.\ tandem walks, Baxter permutations, bipolar orientations and coalescent-walk processes (see \cref{thm:joint_scaling_limits}). In both cases, a key ingredient is the convergence of the discrete coalescent-walk process to its continuous counterpart (\cref{thm:discret_coal_conv_to_continuous}).

\subsubsection{The permuton limit of Baxter permutations}

We now introduce the candidate limiting permuton for Baxter permutations. Its definition is rather straightforward by analogy with the discrete case (see \cref{sect:from_coal_to_perm}). We consider the continuous coalescent-walk process ${\conti Z_e} = \{\conti Z_e^{(t)}\}_{t\in [0,1]}$. Actually $\conti Z_e^{(t)}$ was not defined for $t\in \{0,1\}$ (see \cref{defn:cont_coal_proc}). As what happens on a negligible subset of $[0,1]$ is irrelevant to the arguments to come, this causes no problems.

We first define a random binary relation $\leq_{\conti Z_e}$ on $[0,1]^2$ as follows (this is an analogue of the definition given in \cref{eq:coal_to_perm} page~\pageref{eq:coal_to_perm} in the discrete case):
\begin{equation}\label{eq:cont_coal_to_perm}
\begin{cases}t\leq_{\conti Z_e} t &\text{ for every }t\in [0,1],\\
t\leq_{\conti Z_e} s &\text{ for every }0\leq t<s \leq 1\text{ such that }\ \conti Z_e^{(t)}(s)<0,\\
s\leq_{\conti Z_e} t,&\text{ for every }0\leq t<s \leq 1\text{ such that }\ {\conti Z}_e^{(t)}(s)\geq0.\end{cases}
\end{equation}
Note that the map $(\omega, t,s)\mapsto \idf_{t\leq_{\conti Z_e} s}$ is measurable.

\begin{prop}[{\cite[Proposition 5.1]{borga2020scaling}}]\label{prop:total_order}
	The relation $\leq_{\conti Z_e}$ is antisymmetric and reflexive. Moreover, there exists a random set $\bm A \subset [0,1]^2$ of a.s.\ zero Lebesgue measure, i.e.\ $\P(\Leb(\bm A)=0)=1$, such that the restriction of $\leq_{\conti Z_e}$ to $[0,1]^2\setminus \bm A$ is transitive almost surely.
\end{prop}

We now define a random function that encodes the total order $\leq_{\conti Z_e}$:
\begin{multline}\label{eq:level_function}
\varphi_{\conti Z_e}(t)\coloneqq\Leb\left( \big\{x\in[0,1]|x \leq_{\conti Z_e} t\big\}\right)\\
=\Leb\left( \big\{x\in[0,t)|\conti Z_e^{(x)}(t)<0\big\} \cup \big\{x\in[t,1]|\conti Z_e^{(t)}(x)\geq0\big\} \right),
\end{multline}
where here $\Leb(\cdot)$ denotes the one-dimensional Lebesgue measure. Note that since the mapping $(\omega, t,s)\mapsto \idf_{t\leq_{\conti Z_e} s}$ is measurable, the mapping $(\omega, t)\mapsto \varphi_{\conti Z_e}(t)$ is measurable too.

\begin{obs}
	Note that the function defined in \cref{eq:level_function} is inspired by the following: if $\sigma$ is the Baxter permutation associated with a coalescent-walk process $Z=\{Z^{(t)}\}_{t\in[n]}\in\mathscr{C}$, then
	$\sigma(i)=\#\{j\in[n]|j\leq_Z i\}$.	
\end{obs}

\begin{defn}\label{defn:Baxter_perm}
	The \emph{Baxter permuton} $\bm \mu_B$ is the push-forward of the Lebesgue measure on $[0,1]$ via the mapping $(\Id,\varphi_{\conti Z_e})$, that is
	\begin{equation}
	\label{eq:def_permuton}
	\bm \mu_B(\cdot)\coloneqq(\Id,\varphi_{\conti Z_e})_{*}\Leb (\cdot)= \Leb\left(\{t\in[0,1]|(t,\varphi_{\conti Z_e}(t))\in \cdot \,\}\right).
	\end{equation} 
\end{defn}

The Baxter permuton $\bm \mu_B$ is a random measure on the unit square $[0,1]^2$ and the terminology is justified by the following lemma, that also states some results useful for the proof of \cref{thm:permuton}. 

\begin{lem}[{\cite[Lemma 5.5]{borga2020scaling}}]\label{lem:boundary} The following claims hold:
	\begin{enumerate}
		\item The random measure $\bm \mu_B$ is a.s.\ a permuton.
		\item Almost surely, for almost every $t<s\in[0,1]$, 
		either
		$\conti Z_e^{(t)}(s)>0$ and $\varphi_{\conti Z_e}(s)<\varphi_{\conti Z_e}(t)$, or  $\conti Z_e^{(t)}(s)<0$ and $\varphi_{\conti Z_e}(s)>\varphi_{\conti Z_e}(t)$.
	\end{enumerate}
\end{lem}

We can now prove that Baxter permutations converge in distribution to the Baxter permuton. Since it will be useful in the next section, we also show that this convergence is joint with the convergence of the corresponding tandem walk and the corresponding coalescent-walk process.

We reuse the notation of \cref{sec:cond_conv}. In particular, 
${\bm W}_n$ is a uniform element of the space of tandem walks $\mathcal W_n$, ${\bm Z}_n=\wcp({\bm W}_n)$ is the associated uniform coalescent-walk process, and ${\bm \sigma}_n = \cpbp({\bm Z}_n)$ is the associated uniform Baxter permutation.

\begin{thm}\label{thm:permuton}
	Jointly with the convergences in \cref{thm:discret_coal_conv_to_continuous}, we have that $\mu_{\bm \sigma_n} \xrightarrow{d} \bm \mu_B$.
\end{thm}

\begin{proof}
	Assume there is a subsequence along which $\mu_{\bm \sigma_n}$ converges in distribution to some random permuton $\widetilde{\bm \mu}$ jointly with the convergences in \cref{thm:discret_coal_conv_to_continuous}. We shall show that $\widetilde{\bm \mu}$ has the same distribution as $\bm \mu_B$. By virtue of Prokhorov's theorem and compactness of the space of permutons $\mathcal M$, this is enough to prove the joint convergence claim. To simplify things, we assume that the subsequential convergence is almost sure using Skorokhod's theorem. In particular, almost surely as $n\to\infty$, $\mu_{\bm \sigma_n} \to \widetilde{\bm \mu}$ in the space of permutons, and for every $i\geq 1$, $\conti Z_{n}^{(\bm u_i)} \to \conti Z_{e}^{(\bm u_i )}$ uniformly on $[0,1]$, where $(\bm u_i)_{i\geq 1}$ are i.i.d.\ uniform random variables on $[0,1]$.
	
	Fix $k\in\Z_{>0}$. We denote by $\bm \rho^k_n$ the pattern induced by $\bm \sigma_n$ on the indices $\bm \lceil n \bm u_1 \rceil,\ldots,\lceil n \bm u_k \rceil$ ($\bm \rho^k_n$ is undefined if two indices are equal). From the uniform convergence above, and recalling that ${\conti Z}^{(u)}_{n}\left(\frac kn\right) = \frac 1 {\sqrt {2n}} \bm Z^{(\lceil nu\rceil)}_{n}(k)$, we have for all $1\leq i<j\leq k$ that
	\begin{equation}\label{eq:weifguwefibghwe}
	\sgn\left({\bm Z}_n^{(\bm \lceil n \bm u_i \rceil \wedge \lceil n \bm u_j \rceil)}(\lceil n \bm u_i \rceil \vee \lceil n \bm u_j \rceil)\right)\xrightarrow[n\to\infty]{}\sgn(\conti Z_e^{(\bm u_i \wedge \bm u_j)}(\bm u_j \vee \bm u_i))\quad \text{ a.s.},
	\end{equation}
	where we used the conventions $x\wedge y=\min\{x,y\}$ and $x\vee y=\max\{x,y\}$.
	Note that the function $\sgn$ is not continuous, but by the second claim of \cref{lem:boundary}, the random variable $\conti Z_e^{(\bm u_i \wedge \bm u_j)}(\bm u_j \vee \bm u_i)$ is almost surely nonzero, hence a continuity point of $\sgn$.
	
	By \cref{prop:patterns} page \pageref{prop:patterns}, 
	\begin{equation}\label{eq:wefvbweibfwen}
		{\bm Z}_n^{(\bm \lceil n \bm u_i \rceil \wedge \lceil n \bm u_j \rceil)}(\lceil n \bm u_i \rceil \vee \lceil n \bm u_j \rceil) \geq 0   \iff  \bm \rho^k_n(\bm M_{i,j})<\bm \rho^k_n(\bm m_{i,j}),
	\end{equation}
	where $\bm M_{i,j}$ (resp.\ $\bm m_{i,j}$) denotes the ranking of $\lceil n \bm u_i \rceil \vee \lceil n \bm u_j \rceil$ (resp.\ $\lceil n \bm u_i \rceil \wedge \lceil n \bm u_j \rceil$) in the sequence $(\bm \lceil n \bm u_1 \rceil,\ldots,\lceil n \bm u_k \rceil)$.
	Using the second claim of \cref{lem:boundary}, \cref{eq:weifguwefibghwe,eq:wefvbweibfwen} mean that 
	$$\bm \rho_n^k \xrightarrow[n\to\infty]{} \bm \rho^k,$$ 
	where $\bm \rho^k$ denotes the permutation $\Perm_k(\bm \mu_B)$ induced by $(\bm u_i, \varphi_{\conti Z_e}(\bm u_i))_{i \in [k]}$.
	From \cref{lem:subpermapproxrandompermuton}, we have for $k$ large enough that
	\begin{equation}\label{eq:first_perm_bound}
	\P\left[
	d_{\square}(\mu_{\bm \rho_n^k},\mu_{\bm \sigma_n})
	> 16k^{-1/4}\right]
	\leq \frac12 e^{-\sqrt{k}} + O(n^{-1}),
	\end{equation}
	where the error term $O(n^{-1})$ comes from the fact that $\bm \rho^k_n$ might be undefined.
	Since $\bm \rho_n^k \xrightarrow[n\to\infty]{} \bm \rho^k$ and $\mu_{\bm \sigma_n} \to \widetilde{\bm \mu}$, then
	taking the limit as $n\to \infty$, we obtain that
	\[\P\left[
	d_{\square}(\mu_{\bm \rho^k},\widetilde {\bm \mu})
	> 16k^{-1/4}\right]
	\leq \frac12 e^{-\sqrt{k}},\]
	and so $\mu_{\bm \rho^k}\xrightarrow[k\to\infty]{P}\widetilde {\bm \mu}$. Another application of \cref{lem:subpermapproxrandompermuton} gives that $\mu_{\bm \rho^k} \xrightarrow[k\to\infty]{P}\bm \mu_B$. The last two limits yield $\widetilde {\bm \mu} = \bm \mu_B$ almost surely. This concludes the proof.
\end{proof}

\subsubsection{Joint convergence of the four trees of bipolar orientations}
Fix $n\geq 1$. Let $\bm m_n$ be a uniform bipolar orientation of size $n$, and consider its iterates $\bm m_n^*, \bm m_n^{**}, \bm m_n^{***}$ by the dual operation. Denote $\bigstar = \{\emptyset,*,**,{**}*\}$ the group of dual operations, that is isomorphic to $\Z/4\Z$. For $\theta \in \bigstar$, let $\bm W_n^\theta =(\bm X_n^\theta, \bm Y_n^\theta)$, $\bm Z_n^\theta$ and $\bm \sigma_n^\theta$ be the uniform objects corresponding to $\bm m_n^\theta$ via the commutative diagram in \cref{eq:comm_diagram} page~\pageref{eq:comm_diagram}. We also denote by $\bm L^\theta_n$  the discrete local time process of $\bm Z^\theta_n$ (see \cref{eq:local_time_process} page~\pageref{eq:local_time_process} for a definition).
We define rescaled versions as usual: for $u\in [0,1]$, let ${\conti W}_n^\theta:[0,1]\to \R^2$, ${\conti Z}^{\theta,(u)}_{n}:[0,1]\to\R$ and ${\conti L}^{\theta,(u)}_{n}:[0,1]\to\R$ be the continuous functions obtained by linearly interpolating the following families of points defined for all $k\in [n]$:
\begin{equation}
{\conti W}_n^\theta\left(\frac kn\right) = \frac 1 {\sqrt {2n}} {\bm W}_{n}^\theta (k), 
\quad 
{\conti Z}^{\theta,(u)}_{n}\left(\frac kn\right) = \frac 1 {\sqrt {2n}} {\bm Z}^{\theta,(\lceil nu\rceil)}_{n}(k),
\quad 
{\conti L}^{\theta,(u)}_{n}\left(\frac kn\right) = \frac 1 {\sqrt {2n}} {\bm L}^{\theta,(\lceil nu\rceil)}_{n}(k).
\end{equation}
Finally, for each $n\in\Z_{>0}$, let $((\bm u_{n,i}, \bm u_{n,i}^{*}))_{i\geq 1}$ be an i.i.d.\ sequence of distribution $\mu_{\bm \sigma_n}$ conditionally on $\bm m_n$. Let also $\bm u_{n,i}^{**} = 1-\bm u_{n,i}^{}$ and $\bm u_{n,i}^{***} = 1-\bm u_{n,i}^{*}$ for $n,i\in\Z_{>0}$. The first and second marginals of a permuton are uniform irregardless of the permuton, which implies that for all $n\in\Z_{>0}$ and $\theta \in \bigstar$, $(\bm u_{n,i}^\theta)_{i\geq 1}$ is an i.i.d.\ sequence of uniform random variables on $[0,1]$ independent of $\bm m_n^\theta$ (but for every fixed $n\in\Z_{>0}$, the joint distribution of $\left((\bm u_{n,i}^\theta)_{i\geq 1}\right)_{\theta\in \bigstar}$ depends on $(\bm m^\theta_n)_{\theta\in \bigstar}$).

\medskip

We can now state the main theorem of this section which is in some sense a joint scaling limit convergence result for all these objects. 
Recall the time-reversal and coordinates-swapping mapping $s:\mathcal C([0,1],\R^2) \to \mathcal C([0,1],\R^2)$ defined by $s(f,g) = (g(1-\cdot), f(1-\cdot))$.

\begin{thm}\label{thm:joint_scaling_limits}
	Let $\conti W_e$ be a two-dimensional Brownian excursion of correlation $-1/2$ in the non-negative quadrant. Let $\conti Z_e$ be the associated continuous coalescent-walk process and $\conti L_e$ be its local-time process. Let $\bm u$ denote a uniform random variable in $[0,1]$ independent of $\conti W_e$. Then
	\begin{enumerate}
		\item almost surely, $\conti L_e^{(\bm u)} \in \mathcal C([0,1),\R)$ has a limit at $1$, and we still denote by $\conti L_e^{(\bm u)} \in \mathcal C([0,1],\R)$ its extension.
		\item 
		There exists a measurable mapping $r:\mathcal C([0,1],\R^2) \to \mathcal C([0,1],\R^2)$ such that almost surely, denoting $(\widetilde{\conti X}, \widetilde{\conti Y}) = r(\conti W_e)$,
		\begin{equation}\label{eq:definition_map_r_1}
		\widetilde{\conti X}(\varphi_{\conti Z_e}(\bm u)) =  \conti L_e^{(\bm u)}(1)\qquad\text{and}\qquad r(s(\conti W_e)) = s(r(\conti W_e)).
		\end{equation}
		These properties uniquely determine the mapping $r$ $\P_{\conti W_e}$-almost everywhere. Moreover,
		\begin{equation}\label{eq:definition_map_r_2}
		r(\conti W_e) \stackrel d= \conti W_e, \qquad r^2=s \qquad\text{and}\qquad r^4 = \Id, \qquad\P_{\conti W_e}-\text{a.e.}
		\end{equation}
		\item Let $(\bm u_{i})_{i\geq 1}$ be an auxiliary i.i.d.\ sequence of uniform random variables on $[0,1]$, independent of $\conti W_e$. For each $\theta \in \{\emptyset, *,**\}$, let $\conti W_e^{\theta*} = r(\conti W_e^{\theta})$ and  $\bm u_i^{\theta *} = \varphi_{\conti Z_e}(\bm u_i^\theta)$ for $i\geq 1$. Let also $\conti Z_e^\theta$ be the associated continuous coalescent-walk process, $\conti L_e^\theta$ be its local-time process and $\mu_{\conti Z^\theta_e}$ be the associated Baxter permuton. 
		Then we have the joint convergence in distribution
		\begin{multline}\label{eq:joint_scaling_limits}
		\left(
		\conti W_n^\theta, 
		\left(\bm u_{n,i}^\theta,\conti Z^{\theta,(\bm u_{n,i}^\theta)}_n, \conti L^{\theta,(\bm u_{n,i}^\theta)}_n\right)_{i\in\Z_{>0}}, 
		\mu_{\bm \sigma_n^\theta}
		\right)_{\theta\in \bigstar}\\
		\xrightarrow[n\to\infty]{d}
		\left(
		\conti W_e^\theta, 
		\left(\bm u_i^\theta,\conti Z^{\theta,(\bm u_i^\theta)}_e, \conti L^{\theta,(\bm u_i^\theta)}_e\right)_{i\in\Z_{>0}}, 
		\mu_{\conti Z^\theta_e}
		\right)_{\theta\in \bigstar}
		\end{multline}
		in the space
		$$\left(\mathcal C([0,1], \R^2)\times ([0,1] \times \mathcal C([0,1], \R) \times \mathcal C([0,1], \R))^{\Z_{>0}} \times \mathcal M\right)^4.$$
		\item In this coupling, we almost surely have, for $\theta\in \bigstar$,
		\begin{equation}
		\varphi_{\conti Z_e^{\theta*}}\circ \varphi_{\conti Z_e^\theta} = 1 -\Id,\qquad\P_{\conti W_e}-\text{a.e.}
		\end{equation}
	\end{enumerate}
\end{thm}

\begin{rem}
	As in the discrete case, we point out that even though the joint distribution of $\left((\bm u^{\theta}_{i})_{i\geq 1}\right)_{\theta\in\bigstar}$ depends on $(\conti W_e^{\theta})_{\theta\in \bigstar}$, we have that $(\bm u^{\theta}_{i})_{i\geq 1}$ is independent of $\conti W_e^{\theta}$ for every fixed $\theta\in\bigstar$.
\end{rem}

\begin{rem}
	We highlight that the results presented in the theorem above (in particular in \cref{eq:definition_map_r_1,eq:definition_map_r_2}) are continuous analogs of the results obtained in \cref{sect:anti-invo} for discrete objects. The specific connections between the results for continuous and discrete objects are made clear in the proof of the theorem.
\end{rem}

\begin{proof}[Proof of \cref{thm:joint_scaling_limits}]\label{proof_of_thm}
	We start by showing that the left-hand side of \cref{eq:joint_scaling_limits} is tight. \cref{thm:discret_coal_conv_to_continuous} and \cref{thm:permuton} give us  tightness of all involved random variables, with the caveat that $(\conti L_n^{\theta,(\bm u_i^\theta)})_n$ is a priori only tight in the space $\mathcal C([0,1),\R)$. Tightness in $\mathcal C([0,1],\R)$ follows from the following admitted result\footnote{We remark that the proof of \cref{lem:local_time_does_not_disappear} is highly technical and builds on several combinatorial constructions and asymptotic estimates.}, which proves in passing item 1.
	
	\begin{lem}[{\cite[Lemma 5.11]{borga2020scaling}}] \label{lem:local_time_does_not_disappear}
		Let $\bm u$ be a uniform random variable on $[0,1]$, independent of ${\bm W}_n$. The sequence $(\conti L_n^ {(\bm u)} (1))_n$ is tight, and for every $\eps, \delta >0$, there exists $x\in (0,1)$ and $n_0\geq 1$ such that 
		\begin{equation}\label{eq:continuity_local_time_at_1}
		\P\Big(\conti L_n^ {(\bm u)} (1) -\conti L_n^ {(\bm u)} (1-x) \geq \delta\Big) \leq \eps,\quad \text{for all} \quad n\geq n_0.
		\end{equation}
		Therefore $(\conti L_n^{(\bm u)})_n$ is tight in the space $\mathcal C([0,1],\R)$.
	\end{lem}
	
	We now consider a subsequence of 
	\[\left(
	\conti W_n^\theta, 
	\left(\bm u_i^\theta,\conti Z^{\theta,(\bm u_{n,i}^\theta)}_n, \conti L^{\theta,(\bm u_{n,1}^\theta)}_n\right)_{i\geq 1}, 
	\mu_{\bm \sigma_n^\theta}
	\right)_{\theta\in \bigstar}
	\]
	converging in distribution. For fixed $\theta \in \bigstar$, we know the distribution of the limit thanks to \cref{thm:discret_coal_conv_to_continuous} and \cref{thm:permuton} (the limit of $\conti L^{\theta,(\bm u_i^\theta)}_n$, being a random continuous function on $[0,1]$, is determined by its restriction to $[0,1)$). Henceforth, it is legitimate to denote by
	\begin{equation}\label{eq:limiting_vector}
	\left(
	\conti W_e^\theta, 
	\left(\bm u_i^\theta,\conti Z^{\theta,(\bm u_i^\theta)}_e, \conti L^{\theta,(\bm u_i^\theta)}_e\right)_{i\geq 1}, 
	\mu_{\conti Z^\theta_e}
	\right)_{\theta\in \bigstar}
	\end{equation}
	the limit, keeping in mind that the coupling for varying $\theta$ is undetermined at the moment. We shall determine it to complete the proof of items 2 and 3. We start by proving the following identities (we use the convention $\text{\small****}=\emptyset$): 
	\begin{align}
	\conti W_e^{**} = s(\conti W_e),\quad \conti W_e^{***} = s(\conti W_e^*), \label{eq:main_reversal}\\
	\bm u_i^{\theta *} = \varphi_{\conti Z_e}(\bm u_i^\theta),&\quad  i\geq 1, \theta \in {\bigstar}, \label{eq:main_phi}\\
	\conti X_e^{\theta*}(\bm u_i^{\theta *}) = \conti L_e^{\theta,(\bm u_i^\theta)}(1),&\quad   i\geq 1,\theta \in {\bigstar}. \label{eq:main_L}
	\end{align}
	
	\medskip
	
	The claim in \cref{eq:main_reversal} is the easiest. Thanks to \cref{prop:rev_coal_prop} page \pageref{prop:rev_coal_prop}, we have that $\conti W_n^{**} = s(\conti W_n)$ and $\conti W_n^{***} = s(\conti W_n^*)$, for every $n\in\Z_{>0}$.
	Since $s$ is continuous on $\mathcal C([0,1],\R^2)$, the same result holds in the limit, proving \cref{eq:main_reversal}.
	
	\medskip
	
	To prove \cref{eq:main_phi}, we use the following lemma, whose proof is skipped. It follows rather directly from the definition of weak convergence of measures.
	\begin{lem}\label{lem:cv_iid_seq}
		Suppose that for $n\in \Z_{>0} \cup \{\infty\}$, $\bm \mu_n$ is a random measure on a Polish space and $(\bm X^n_i)_{i\geq 1}$ an i.i.d.\ sequence of elements with distribution $\bm \mu_n$ conditionally on $\bm \mu_n$.
		Assume that $\bm \mu_n \to \bm \mu_\infty$ in distribution for the weak topology. Then we have the joint convergence in distribution 
		\[
		(\bm \mu_n,(\bm X^n_i)_{i\geq 1}) \xrightarrow[n\to\infty]{d} (\bm \mu_\infty,(\bm X^\infty_i)_{i\geq 1})\; .
		\]
	\end{lem}
	In view of the construction of $\left(\mu_{\bm \sigma_n^\theta},(\bm u_{n,i}^\theta, \bm u_{n,i}^{\theta*})_{i\geq 1}\right)$,
	it implies that the joint distribution of $\left(\mu_{\conti Z^\theta_e},(\bm u_i^\theta, \bm u_i^{\theta*})_{i\geq 1}\right)$
	is that of $\mu_{\conti Z^\theta_e}$ together with an i.i.d.\ sequence of elements with distribution $\mu_{\conti Z^\theta_e}$ conditionally on $\mu_{\conti Z^\theta_e}$.
	In particular, we must have $\bm u_i^{\theta *} = \varphi_{\conti Z_e^{\theta}}(\bm u_i^{\theta})$ almost surely.
	This proves \cref{eq:main_phi}.
	
	\medskip
	
	Finally, thanks to \cref{cor:local_time} page \pageref{cor:local_time}, we have the discrete identity 
	\begin{equation}
		\conti X_n^{\theta*}(n^{-1}\lceil n\bm u_i^{\theta *}\rceil) = \conti L_n^{\theta,(\bm u_i^\theta)}(1) - \frac 1 {\sqrt{2n}},\quad\text{for every}\quad n\geq 1.
	\end{equation}
	By convergence in distribution, we obtain  \cref{eq:main_L}.
	
	\medskip
	
	The continuous stochastic process $\conti X_e^{\theta *}$ is almost surely determined by its values on the dense sequence $(\bm u_i^{\theta *})_{i\geq i_0}$. By \cref{eq:main_L} and 0-1 law, we have that $\conti X_e^{\theta *} \in \sigma(\conti W_e^\theta)$. This together with \cref{eq:main_reversal} implies that 
	$$\sigma(\conti Y_e^{\theta *}) = \sigma(\conti X_e^{\theta ***}) \subset \sigma(\conti W_e^{\theta**}) = \sigma(\conti W_e^{\theta}).$$ 
	As a result $\conti W_e^{\theta*} \in \sigma(\conti W_e^{\theta})$ and so there exists a measurable map $r:\mathcal C([0,1],\R^2) \to \mathcal C([0,1],\R^2)$ such that 
	\begin{equation}\label{eq:main_r}
	r(\conti W_e^{\theta}) = \conti W_e^{\theta*}.
	\end{equation} Then the claims in \cref{eq:definition_map_r_1,eq:definition_map_r_2} are an immediate consequence of \cref{eq:main_L,eq:main_reversal}. The fact that \cref{eq:definition_map_r_1} uniquely determines $r$ $\P_{\conti W_e}$-almost everywhere also results from the fact that a continuous function is uniquely determined by its values on a set of full Lebesgue measure. This completes the proof of item 2.
	
	Additionally, \cref{eq:main_phi,eq:main_r} show that the coupling in \cref{eq:limiting_vector} is the one announced in the statement of item 3, and in particular is independent of the subsequence. Together with tightness, this proves item 3.
	
	For item 4, we observe that $\bm u_{n,1}^{\theta} = 1- \bm u_{n,1}^{\theta**} + 1/n$, so that taking the limit, $\bm u_{1}^{\theta**} = 1- \bm u_{1}^{\theta}$. Then item 4 follows from \cref{eq:main_phi}.
\end{proof}

\section{Better than open problems: The skew Brownian permuton}\label{sect:skew_perm}

\emph{\textbf{Note}. In this section we present some possible future projects and several conjectures. There are no rigorous mathematical results here.} 

\medskip

We believe that a generalized version of the Baxter permuton, called the \emph{skew Brownian permuton}, describes the limit of various models of random permutations, identifying a new universality class for random permutations. We now introduce this new limiting permuton and then we comment on why it should be a universal object and on how it should be also related to the biased Brownian separable permuton.

\subsection{Definition}

The definition of the skew Brownian permuton follows the same lines as the definition of the Baxter permuton. Let $(\conti E_{\rho}(t))_{t\in [0,1]}$ denote a two-dimensional Brownian excursion of correlation\footnote{Note that $\conti E_{-1/2}$ coincides with $\conti W_e$ from the previous section. We adopt this change of notation here for typographical convenience.} $\rho\in[-1,1]$ in the non-negative quadrant and  let $q\in[0,1]$ be a further parameter.
Consider the solutions (see below for a discussion on existence and uniqueness) of the following family of stochastic differential equations (SDEs) indexed by $u\in [0,1]$ and driven by $\conti E_{\rho} = (\conti X_{\rho},\conti Y_{\rho})$:
\begin{equation}\label{eq:flow_SDE_gen}
\begin{cases}
d\conti Z_{\rho,q}^{(u)}(t) = \idf_{\{\conti Z_{\rho,q}^{(u)}(t)> 0\}} d\conti Y_{\rho}(t) - \idf_{\{\conti Z_{\rho,q}^{(u)}(t)\leq 0\}} d \conti X_{\rho}(t)+(2q-1)\cdot d\conti L^{\conti Z_{\rho,q}^{(u)}}(t),& t\geq u,\\
\conti Z_{\rho,q}^{(u)}(t)=0,&  t\leq u,
\end{cases} 
\end{equation}
where $\conti L^{\conti Z_{\rho,q}^{(u)}}(t)$ is the local-time process at zero of $\conti Z_{\rho,q}^{(u)}$.
Like before, we call \emph{continuous coalescent-walk process} driven by $(\conti E_{\rho},q)$ the collection of stochastic processes $\left\{\conti Z^{(u)}_{\rho,q}\right\}_{u\in[0,1]}$ and we consider the following random mapping: 
\begin{equation}\label{defn:varphi_gener}
\varphi_{\conti Z_{\rho,q}}(t)\coloneqq
\Leb\left( \big\{x\in[0,t)|\conti Z_{\rho,q}^{(x)}(t)<0\big\} \cup \big\{x\in[t,1]|\conti Z_{\rho,q}^{(t)}(x)\geq0\big\} \right), \quad t\in[0,1].
\end{equation} 

\begin{rem}
	We point out that a single trajectory $\conti Z^{(u)}_{\rho,q}$, for a fixed $u\in[0,1]$, is a skew Brownian motion (for an introduction to skew Brownian motions we refer to \cite{MR2280299}). This fact inspired our terminology in the following definition.
\end{rem}

\begin{defn}\label{defn:Baxter_perm_gen}
	Fix $\rho\in[-1,1]$ and $q\in[0,1]$. The \emph{skew Brownian permuton} with parameters $\rho, q$, denoted $\bm \mu_{\rho,q}$, is the push-forward of the Lebesgue measure on $[0,1]$ via the mapping $(\Id,\varphi_{\conti Z_{\rho,q}})$. 
\end{defn}

Note that the Baxter permuton coincides with the skew Brownian permuton with parameters $\rho=-1/2$ and $q=1/2$.
We also conjecture\footnote{In Appendix A \& B of \cite{borga2020scaling} we carefully explain why we strongly believe in this conjecture.} that the biased separable permuton $\bm \mu^{(p)}$ introduced in \cref{sect:sub_close_cls} is a sort of extremal case of the skew Brownian permuton.
	\begin{conj}\label{conj:Baxt_brow_same}
		For all $p\in[0,1]$, the biased Brownian separable permuton $\bm \mu^{(p)}$ has the same distribution as the skew Brownian permuton $\bm \mu_{1,1-p}$.
	\end{conj} 

Note that when $\rho=1$, $\conti X_{\rho}(t)=\conti Y_{\rho}(t)=\bm e(t)$, where $(\bm e(t))_{t\in [0,1]}$ is a one-dimensional Brownian excursion on $[0,1]$. When in addition $q=1-p$, the SDEs in \cref{eq:flow_SDE_gen} rewrite as
\begin{equation}\label{eq:flow_SDE_gen_Tanaka}
\begin{cases}
d\conti Z_{1,1-p}^{(u)}(t) = \sgn(\conti Z_{1,1-p}^{(u)}(t)) d \bm e(t)+(1-2p)\cdot d\conti L^{\conti Z_{1,1-p}^{(u)}}(t),& t\geq u,\\
\conti Z_{1,1-p}^{(u)}(t)=0,&  t\leq u.
\end{cases} 
\end{equation}

When $p=1/2$, this is the coalescing flow of the well-known \textit{Tanaka SDE}, driven by an excursion instead of a Brownian motion. The characteristic feature of the Tanaka equation is the absence of pathwise uniqueness: solutions are not measurable functions of the driving process $\bm e$, but rather incorporate additional randomness, which takes in this instance (see \cite[Section 4.4.3]{MR2060298}) the form of independent uniform signs $\bm s(\ell) \in\{\oplus,\ominus\}$ for every $\ell\in (0,1)$ that is a local minimum\footnote{
	For the technicalities involved in indexing an i.i.d.\ sequence by this random countable set, see \cite{maazoun17BrownianPermuton}.} of $\bm e$.
The solutions are then constructed explicitly as follows. 

For $0\leq u \leq t \leq 1$, set $\bm m(u,t) \coloneqq \inf_{[u,t]} \bm e$, and $\bm \mu(u,t) =\inf\{s\geq t: \bm e(s) = \bm m(u,t)\}$. Then
$$
\conti Z^{(u)}_{1,1/2}(t) \coloneqq (\bm e(t) - \bm m(u,t))\bm s( \bm \mu(u,t)).
$$
By analogy, we believe in the following.

	\begin{conj}\label{conj:conj_rho_1}
		For $\rho=1$ and for all $p\in[0,1]$, the solutions $\conti Z_{1,1-p}^{(u)}$ of \cref{eq:flow_SDE_gen_Tanaka} can be constructed as above by taking independent \emph{biased} signs $\bm s(\ell) \in\{\oplus,\ominus\}$ in such a way that $\P(\bm s(\ell)=\oplus)=p$.
	\end{conj}

The other extremal case $\rho = -1$ gives the following SDEs defined for all $u\in[0,1]$,
\begin{equation}\label{eq:flow_SDE_gen_anti_Tanaka}
\begin{cases}
d\conti Z_{-1,q}^{(u)}(t) = d \bm e(t)+(2q-1)\cdot d\conti L^{\conti Z_{-1,q}^{(u)}}(t),& t\geq u,\\
\conti Z_{-1,q}^{(u)}(t)=0,&  t\leq u.
\end{cases} 
\end{equation}
This equation (when $\bm e$ is replaced by a Brownian motion) is Harrison and Shepp's equation (\cite{MR606993}) defining the skew Brownian motion, whose solutions are pathwise unique \cite{MR2280299}, and whose coalescing flow was studied by Burdzy and Kaspi (see \cite{MR2094439} and the references therein). 

All the discussions above make us conjecture the following.

	\begin{conj}\label{conj:conj_rho_gen}
		Existence and pathwise uniqueness for the solutions of the SDEs in \cref{eq:flow_SDE_gen} holds for all $\rho \in [-1,1)$ and $q\in [0,1]$.
	\end{conj}

Note that a corollary of Conjectures \ref{conj:conj_rho_1} and \ref{conj:conj_rho_gen} is the following.

	\begin{conj}
		The skew Brownian permuton $\bm \mu_{\rho,q}$ is well-defined for all $(\rho,q)\in[-1,1]\times[0,1]$.
	\end{conj}

Simulations of the skew Brownian permuton for various values of $\rho$ and $q$ can be found in \cref{fig:uyievievbeee} or in \href{https://drive.google.com/drive/folders/1bqDI6fjKdoRuMAoq7qm-0hkyKIMWey8s?usp=sharing}{this link}.
\begin{figure}[htbp]
	\begin{minipage}[c]{0.06\textwidth}
		\centering
		\footnotesize$\rho=$\\
		\footnotesize$-0.995$
	\end{minipage}
	\begin{minipage}[c]{0.15\textwidth}
		\centering
		\footnotesize{$q\approx 0.1$}
		\includegraphics[scale=0.18]{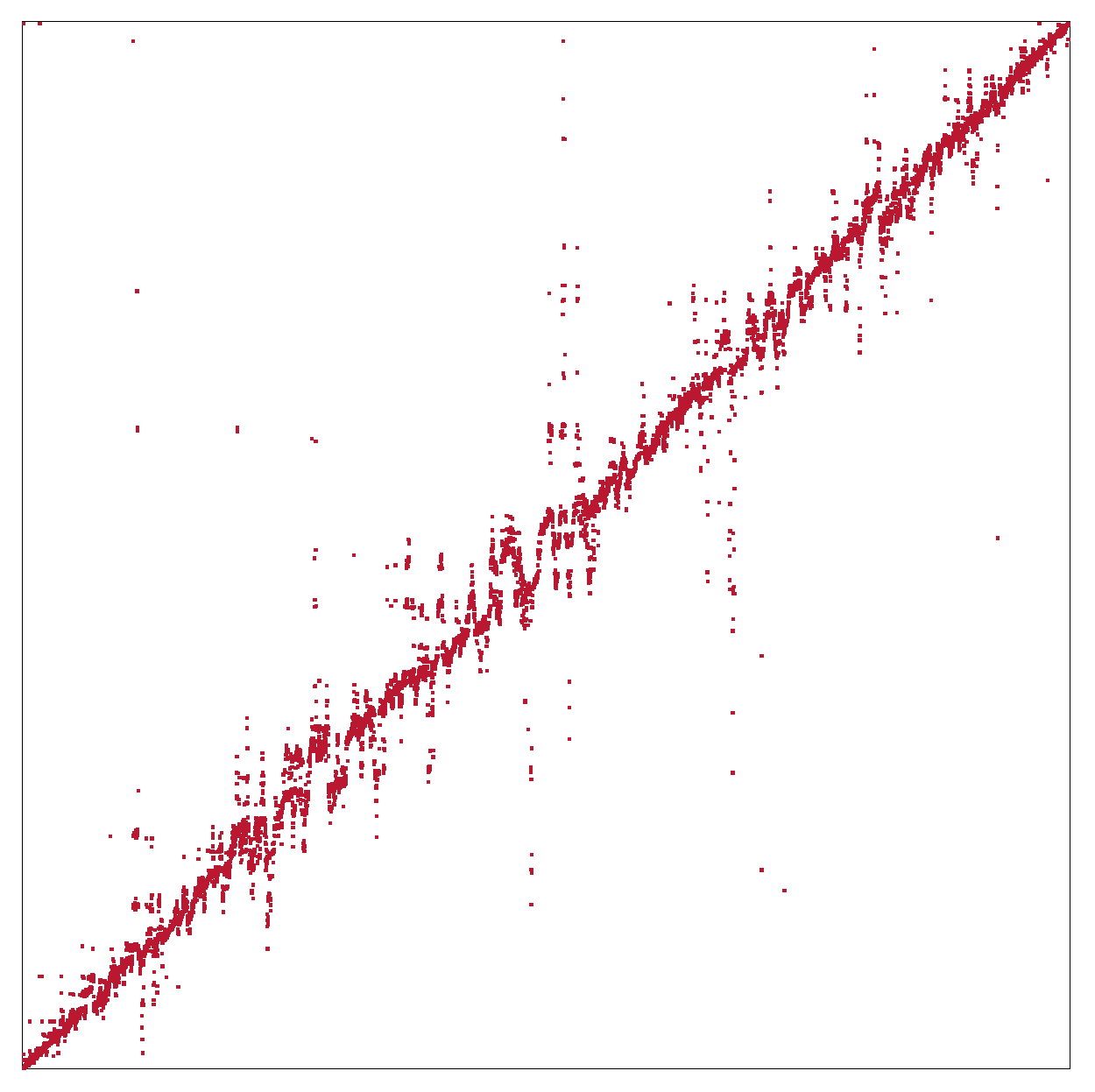}
	\end{minipage}
	\begin{minipage}[c]{0.15\textwidth}
		\centering
		\footnotesize{$q\approx 0.4$}
		\includegraphics[scale=0.18]{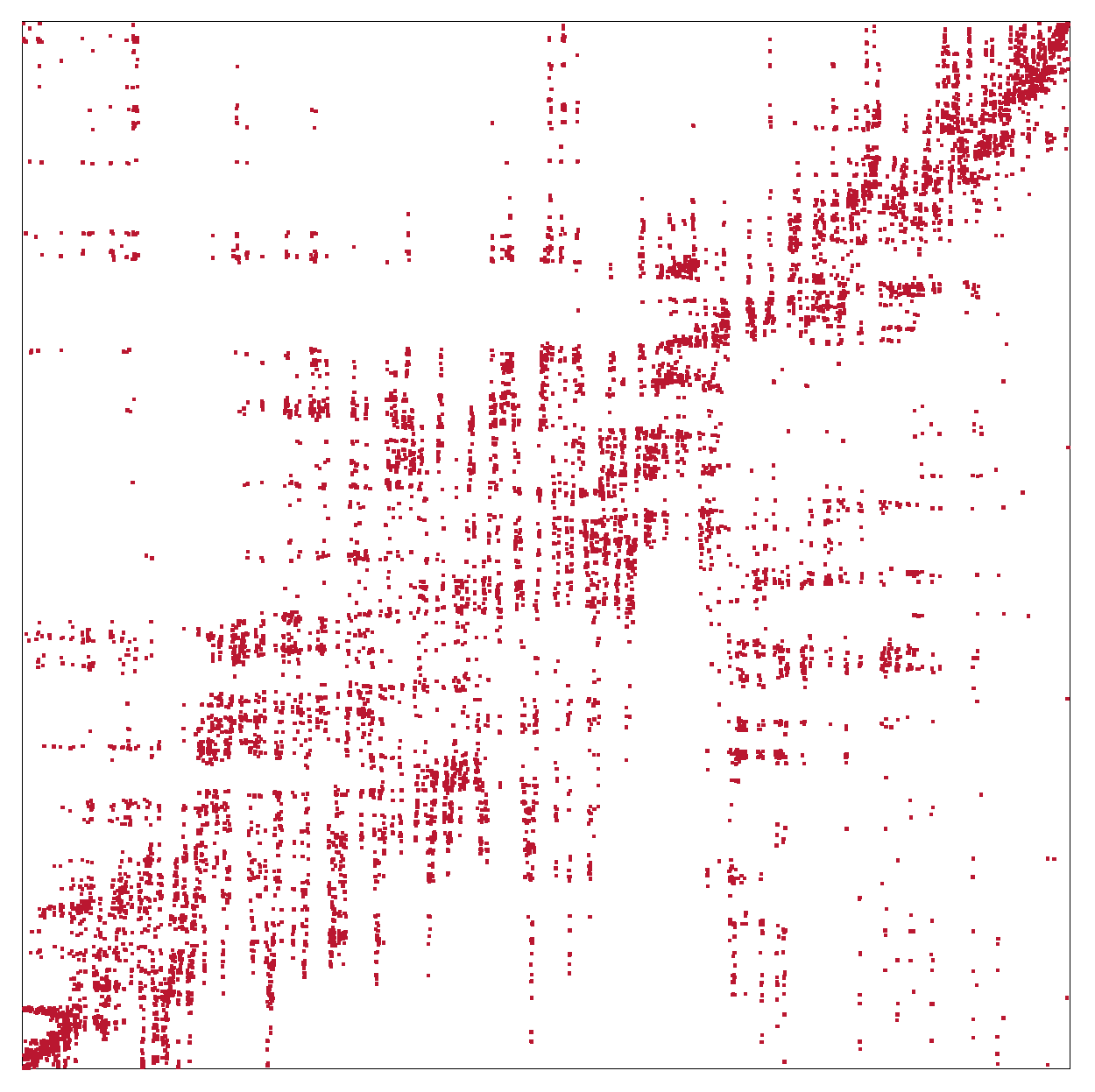}
	\end{minipage}
	\begin{minipage}[c]{0.15\textwidth}
		\centering
		\footnotesize{$q= 0.5$}
		\includegraphics[scale=0.18]{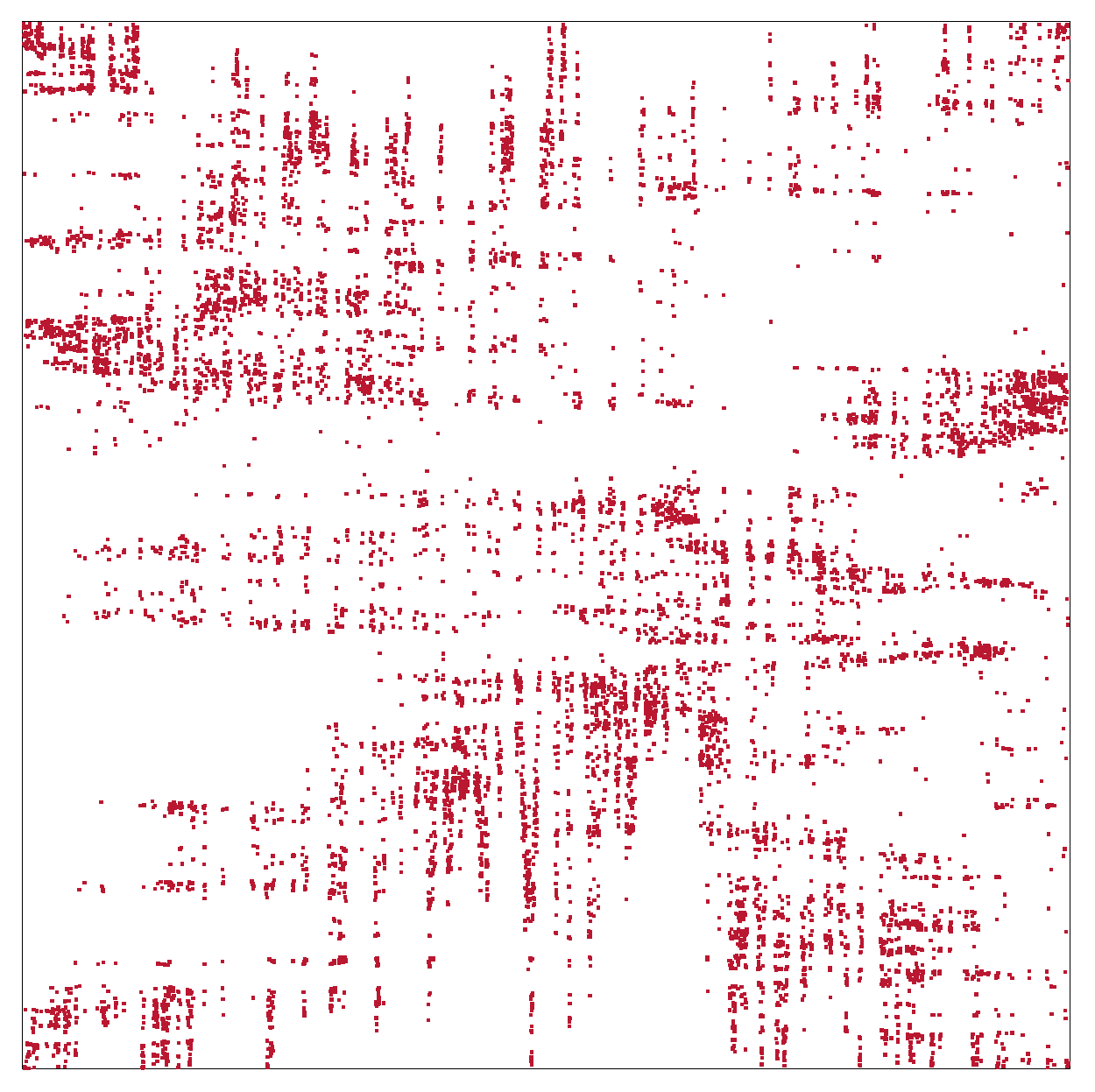}
	\end{minipage}
	\begin{minipage}[c]{0.15\textwidth}
		\centering
		\footnotesize{$q\approx 0.6$}
		\includegraphics[scale=0.18]{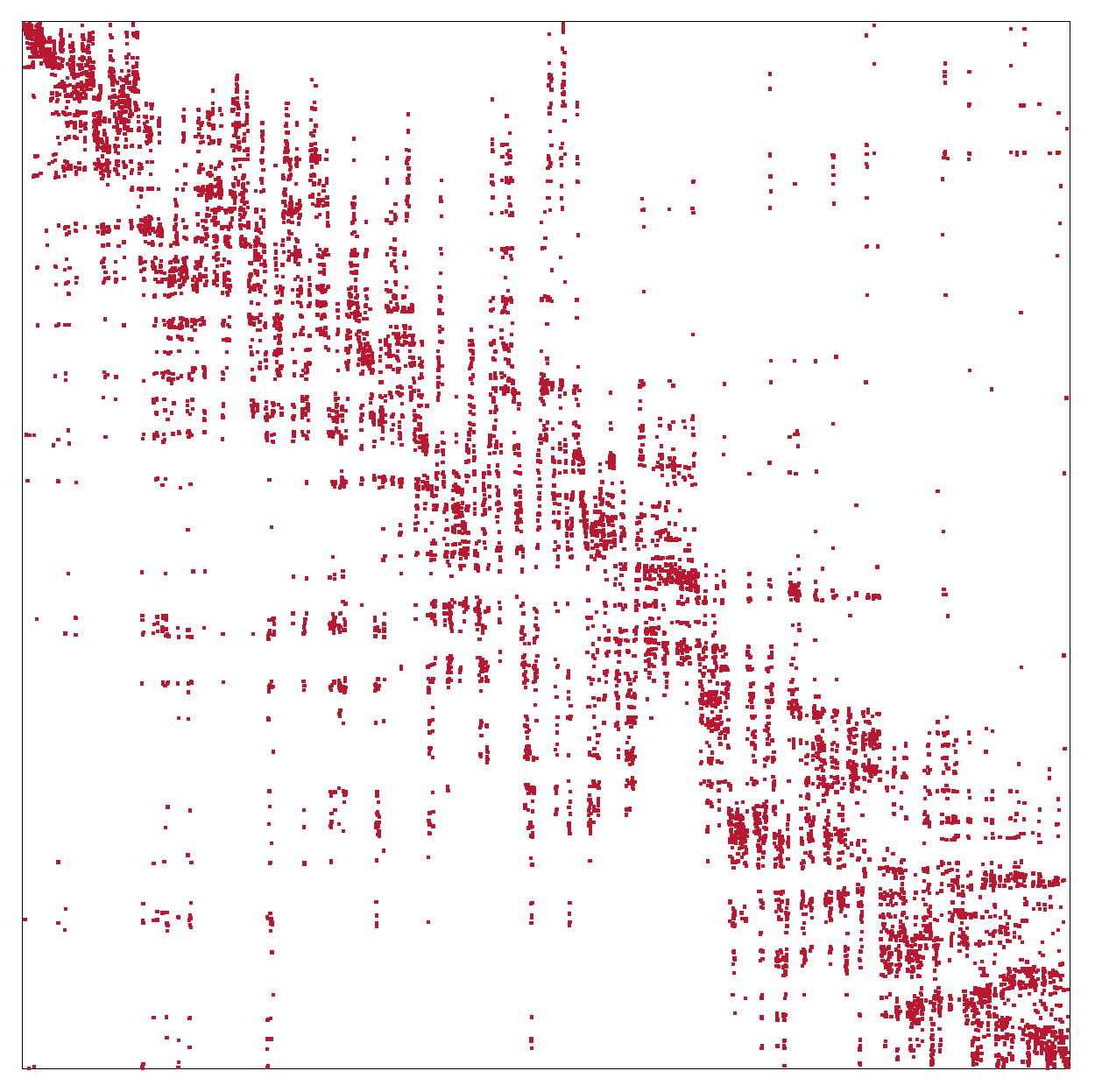}
	\end{minipage}
	\begin{minipage}[c]{0.15\textwidth}
		\centering
		\footnotesize{$q\approx 0.9$}
		\includegraphics[scale=0.18]{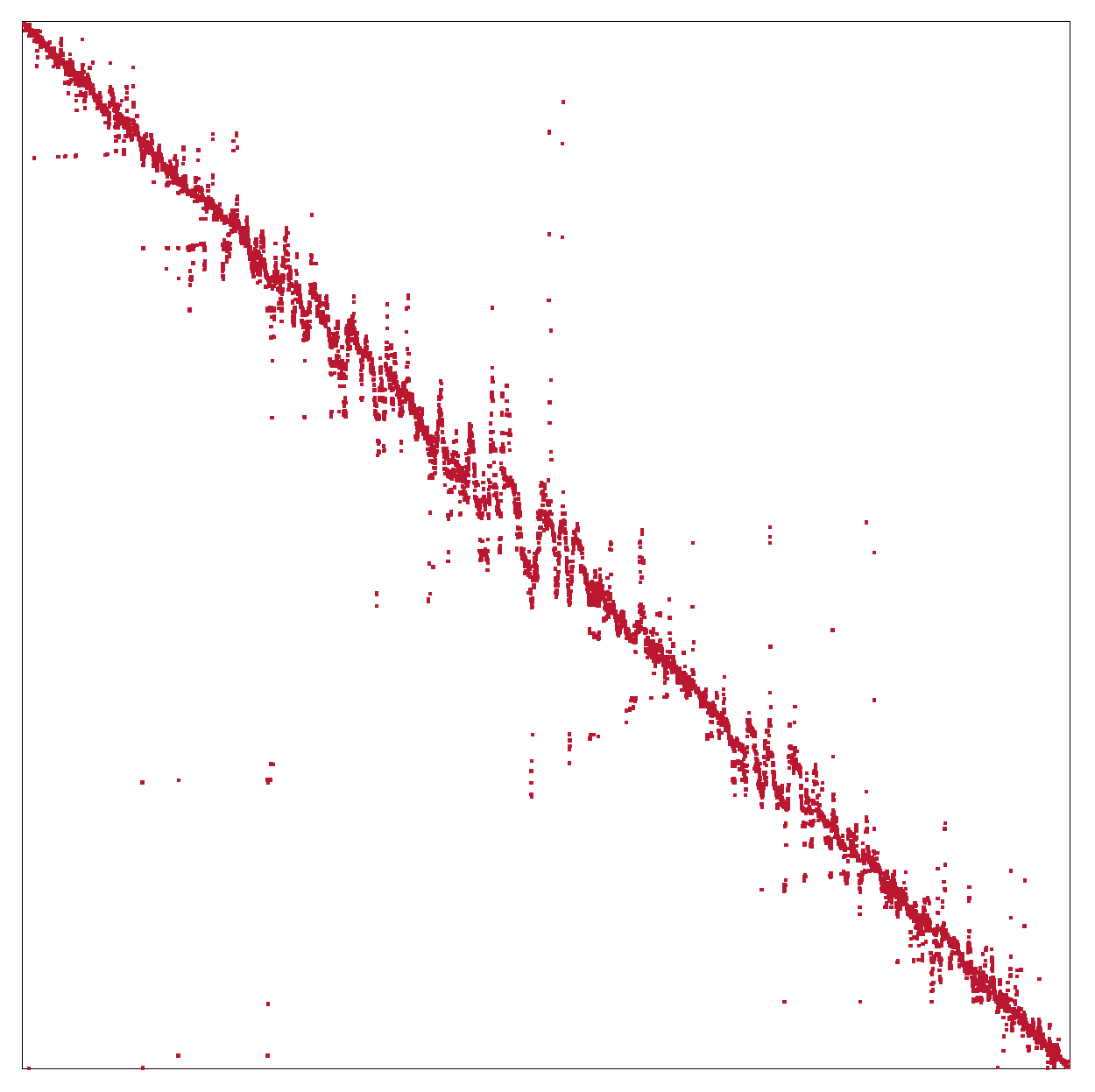}
	\end{minipage}
	\hspace{0.1 cm}
	\begin{minipage}[c]{0.15\textwidth}
		\centering
		\footnotesize{$\conti E_{\rho} = (\conti X_{\rho},\conti Y_{\rho})$}
		\includegraphics[scale=0.18]{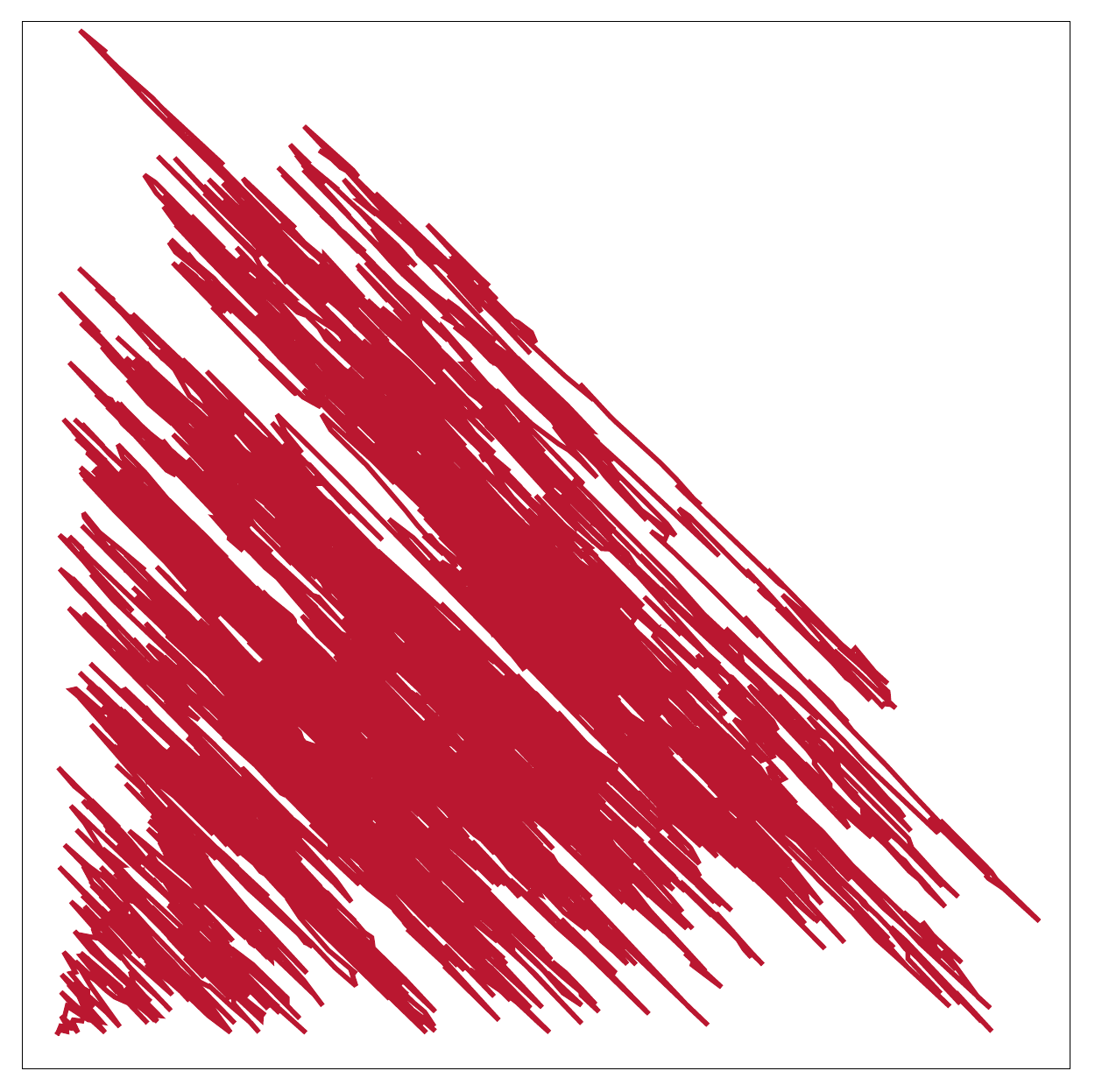}
	\end{minipage}
	
	\begin{minipage}[c]{0.06\textwidth}
		\centering
		\footnotesize$\rho=$\\
		\footnotesize$-0.5$
	\end{minipage}
	\begin{minipage}[c]{0.15\textwidth}
		\centering
		\includegraphics[scale=0.18]{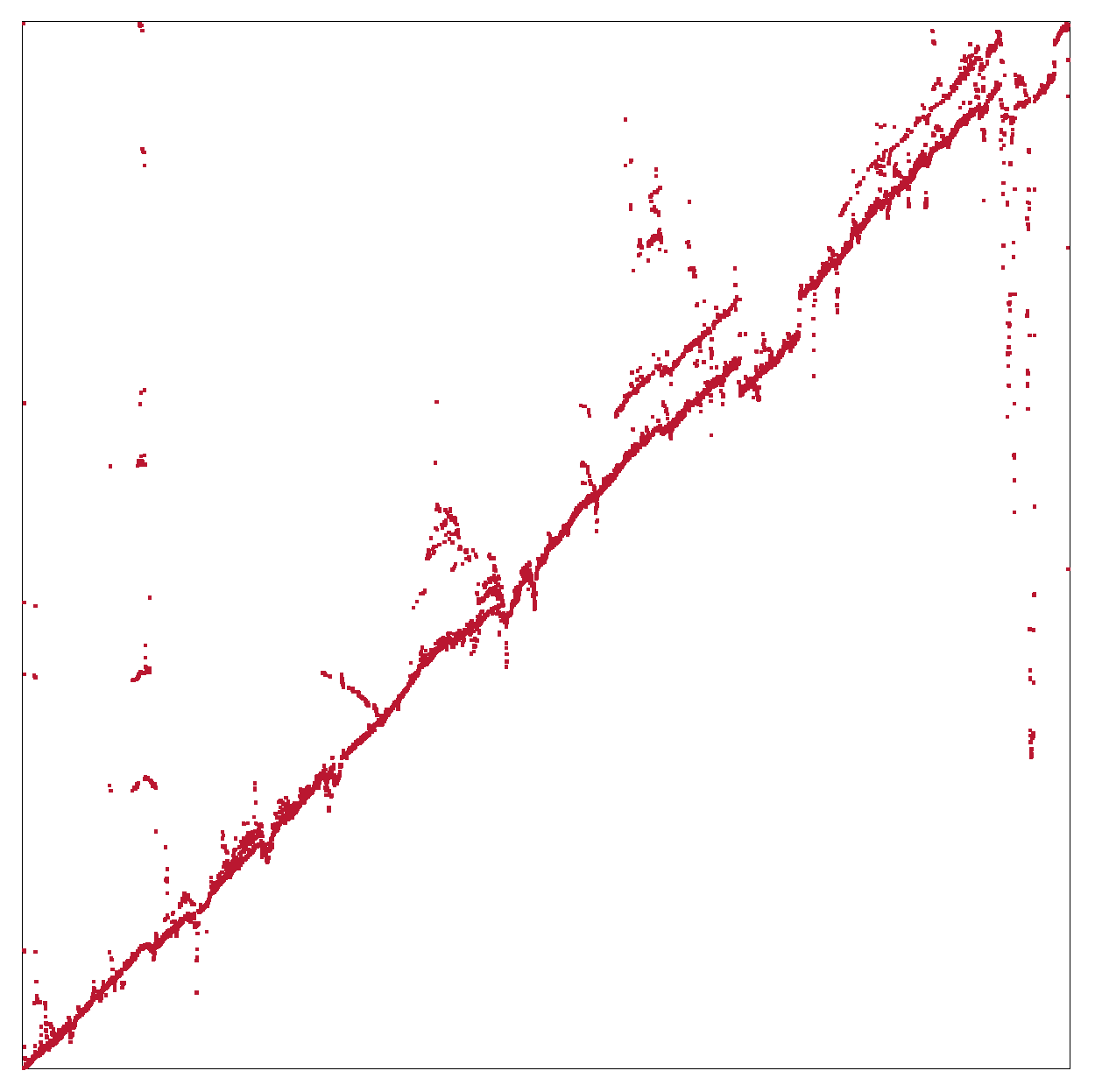}
	\end{minipage}
	\begin{minipage}[c]{0.15\textwidth}
		\centering
		\includegraphics[scale=0.18]{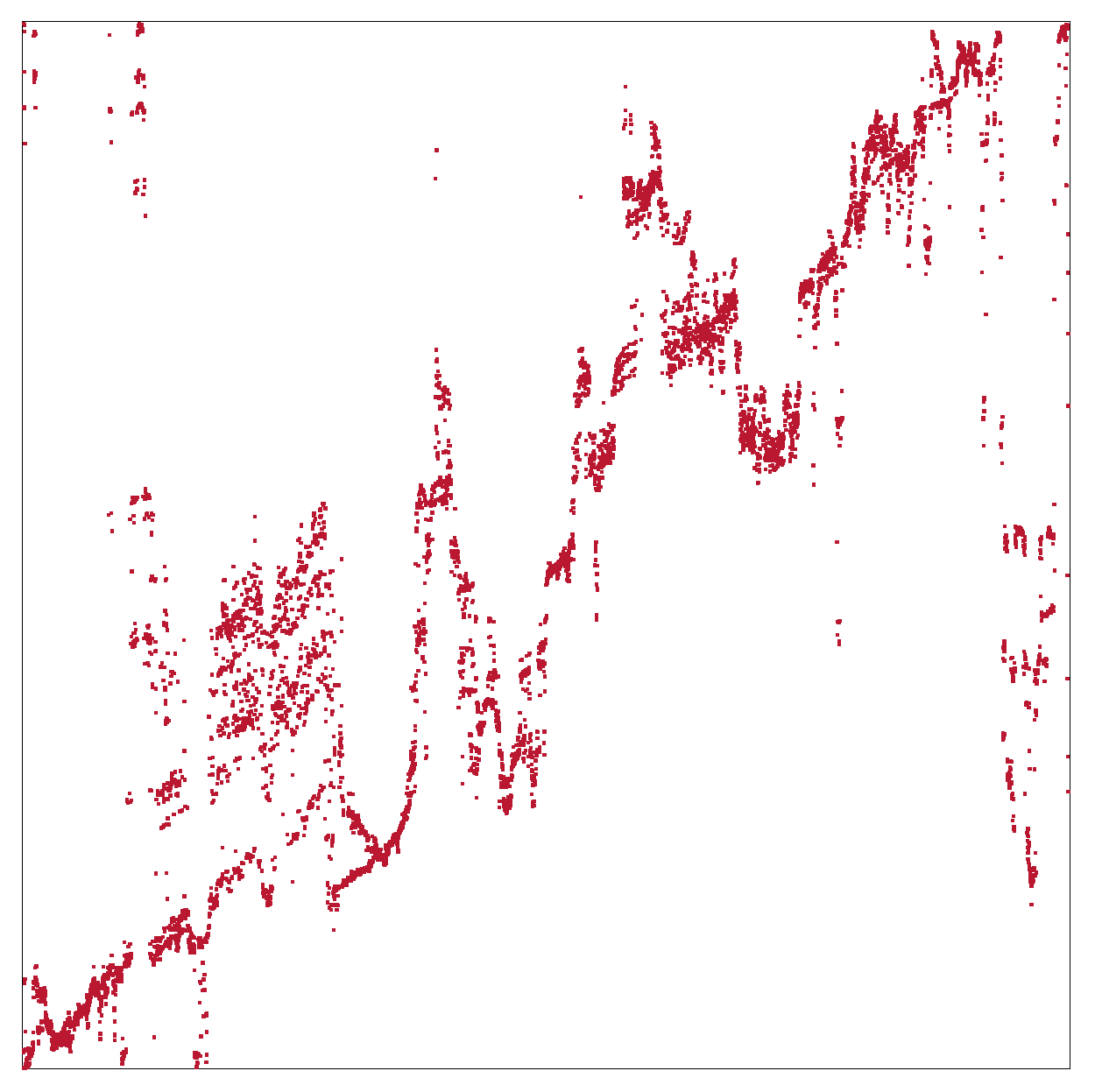}
	\end{minipage}
	\begin{minipage}[c]{0.15\textwidth}
		\centering
		\includegraphics[scale=0.18]{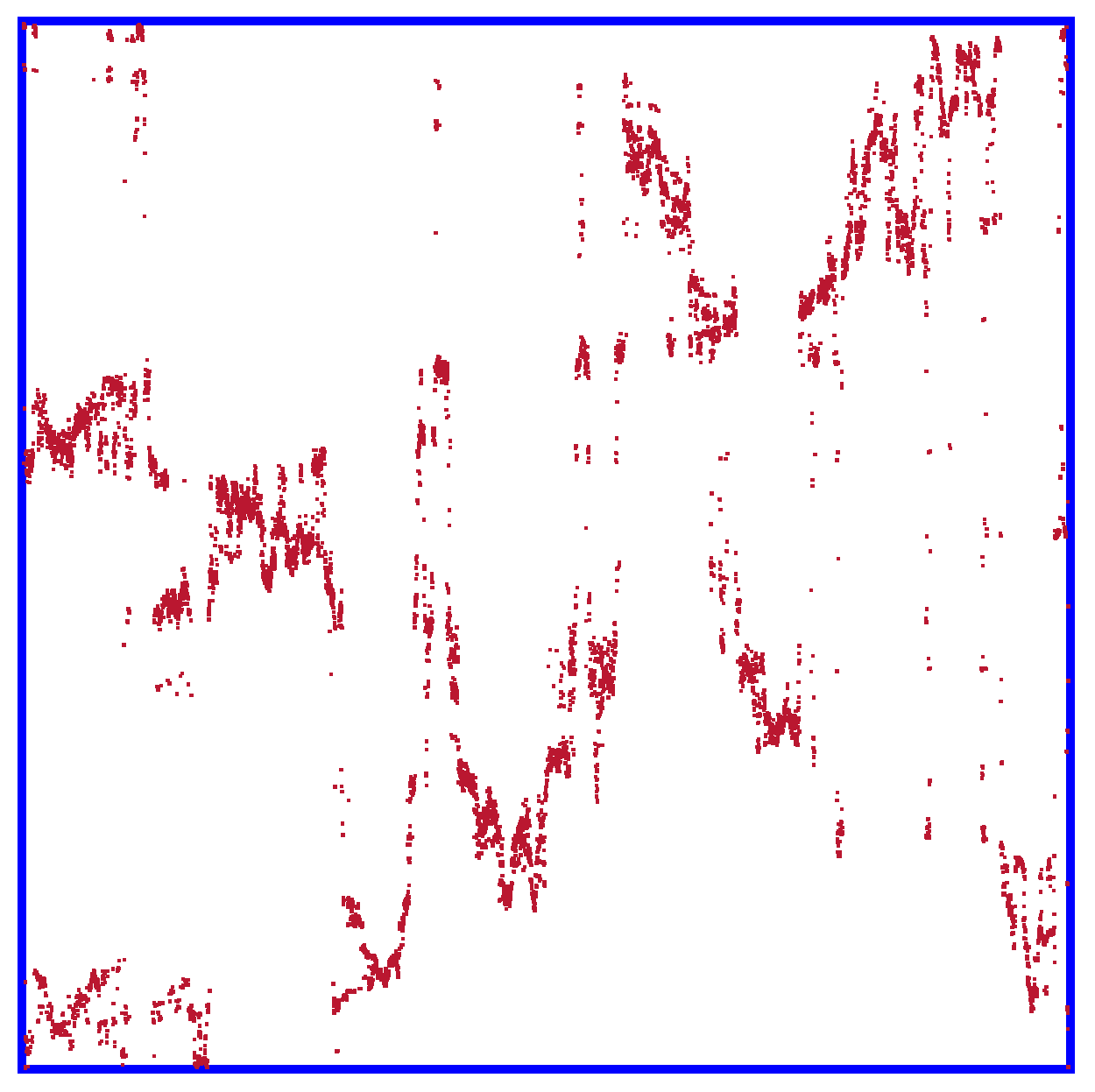}
	\end{minipage}
	\begin{minipage}[c]{0.15\textwidth}
		\centering
		\includegraphics[scale=0.18]{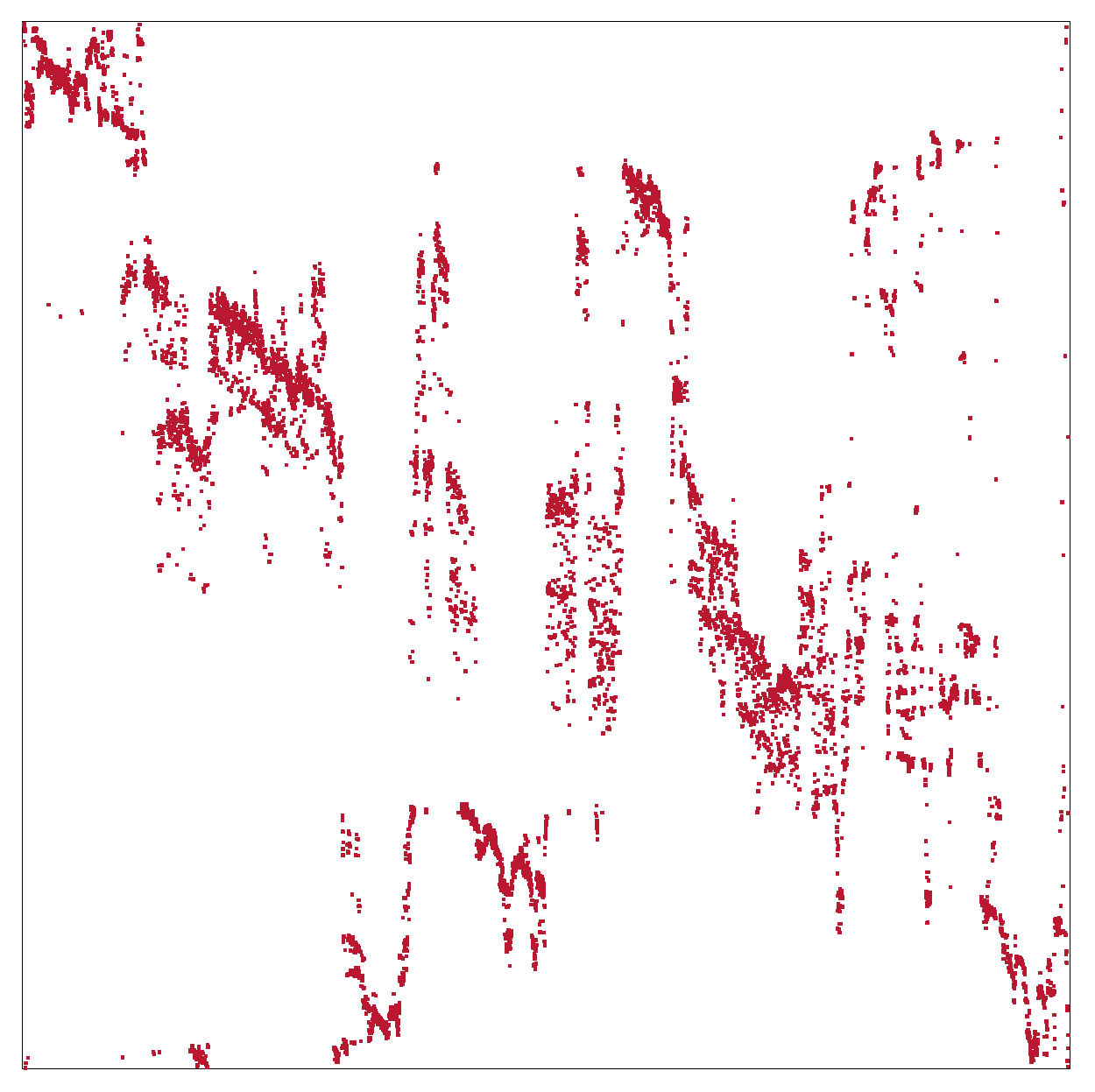}
	\end{minipage}
	\begin{minipage}[c]{0.15\textwidth}
		\centering
		\includegraphics[scale=0.18]{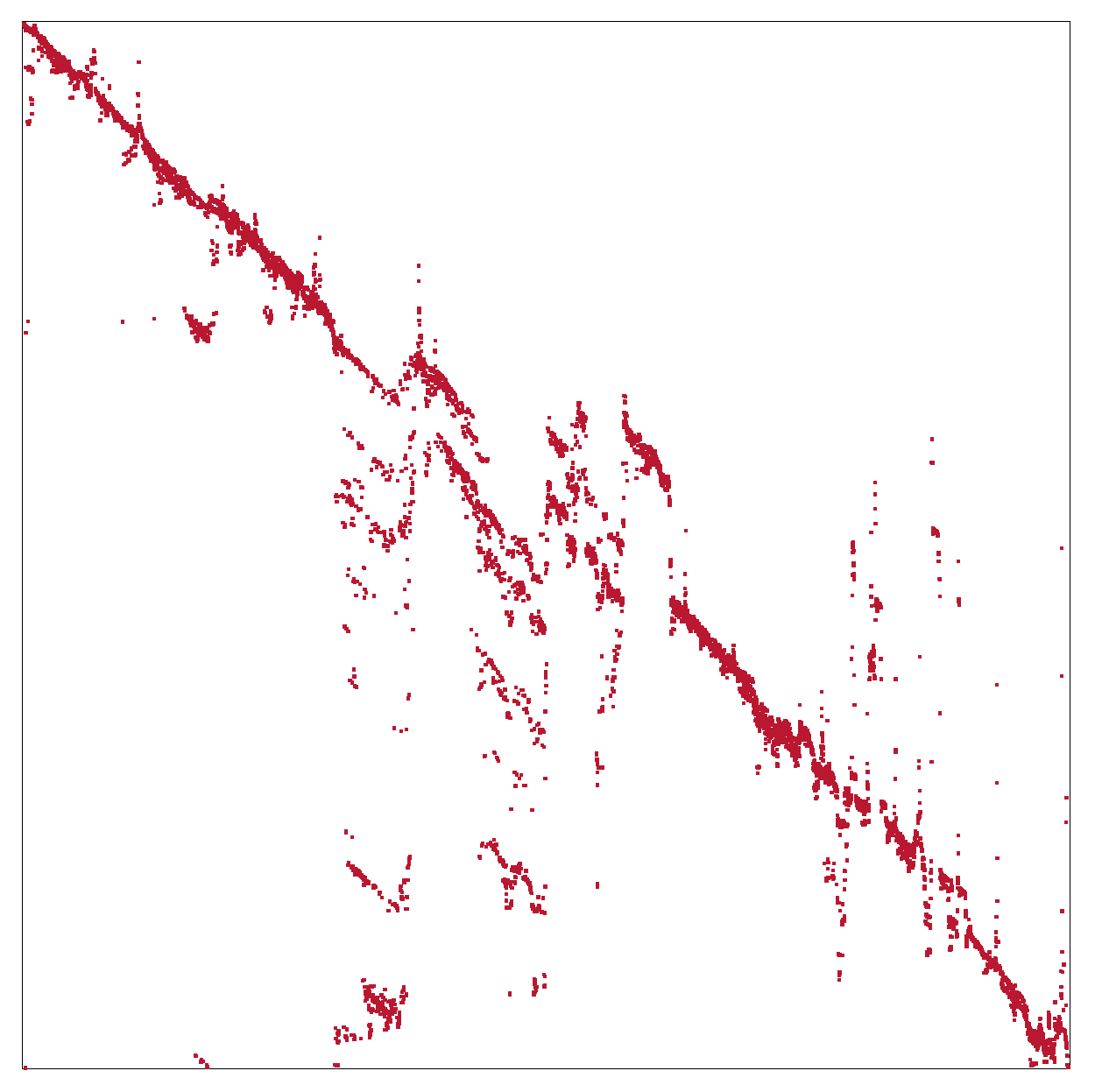}
	\end{minipage}
	\hspace{0.1 cm}
	\begin{minipage}[c]{0.15\textwidth}
		\centering
		\includegraphics[scale=0.18]{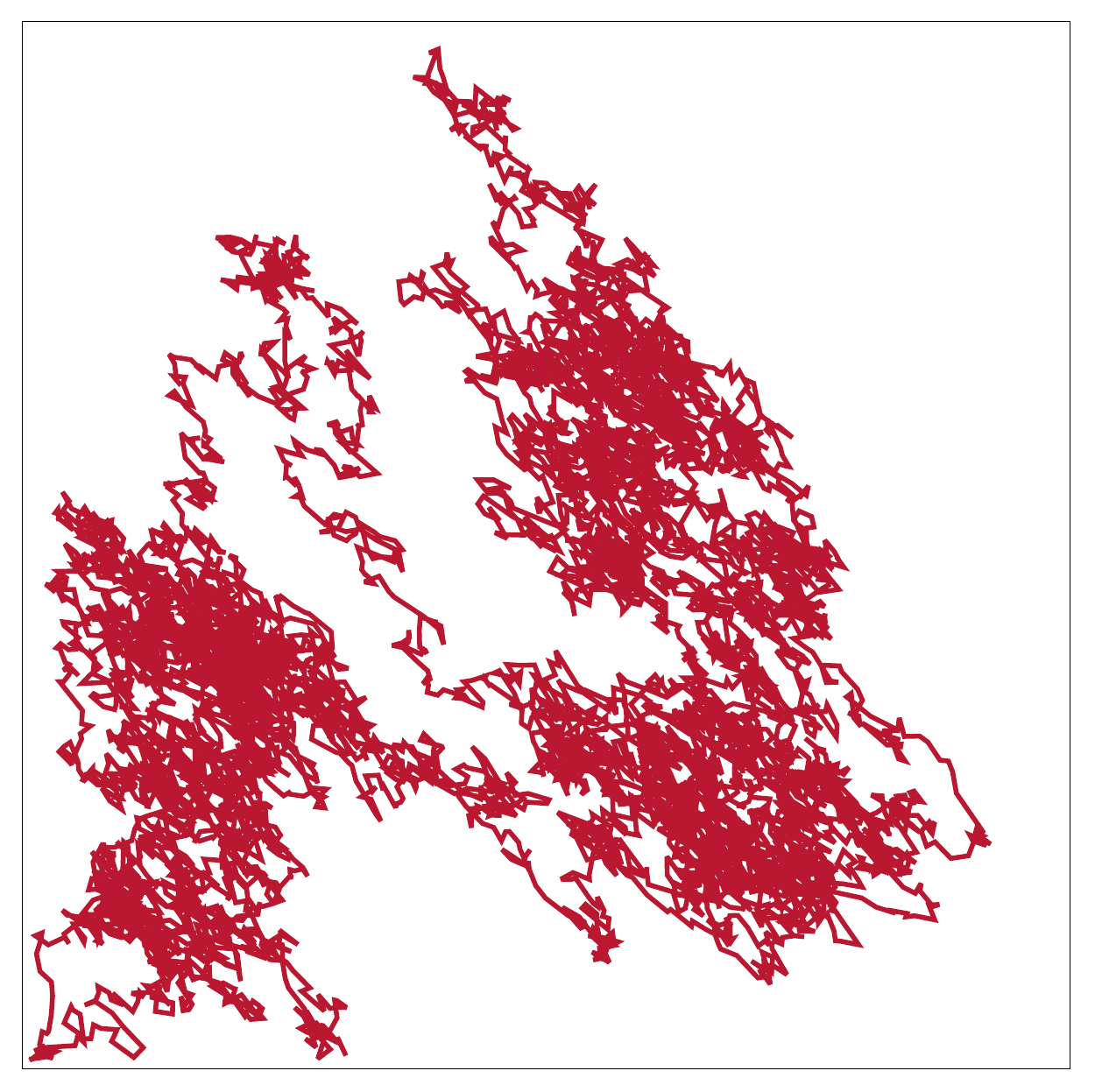}
	\end{minipage}
	
	\begin{minipage}[c]{0.06\textwidth}
		\centering
		\footnotesize$\rho=0$
	\end{minipage}
	\begin{minipage}[c]{0.15\textwidth}
		\centering
		\includegraphics[scale=0.18]{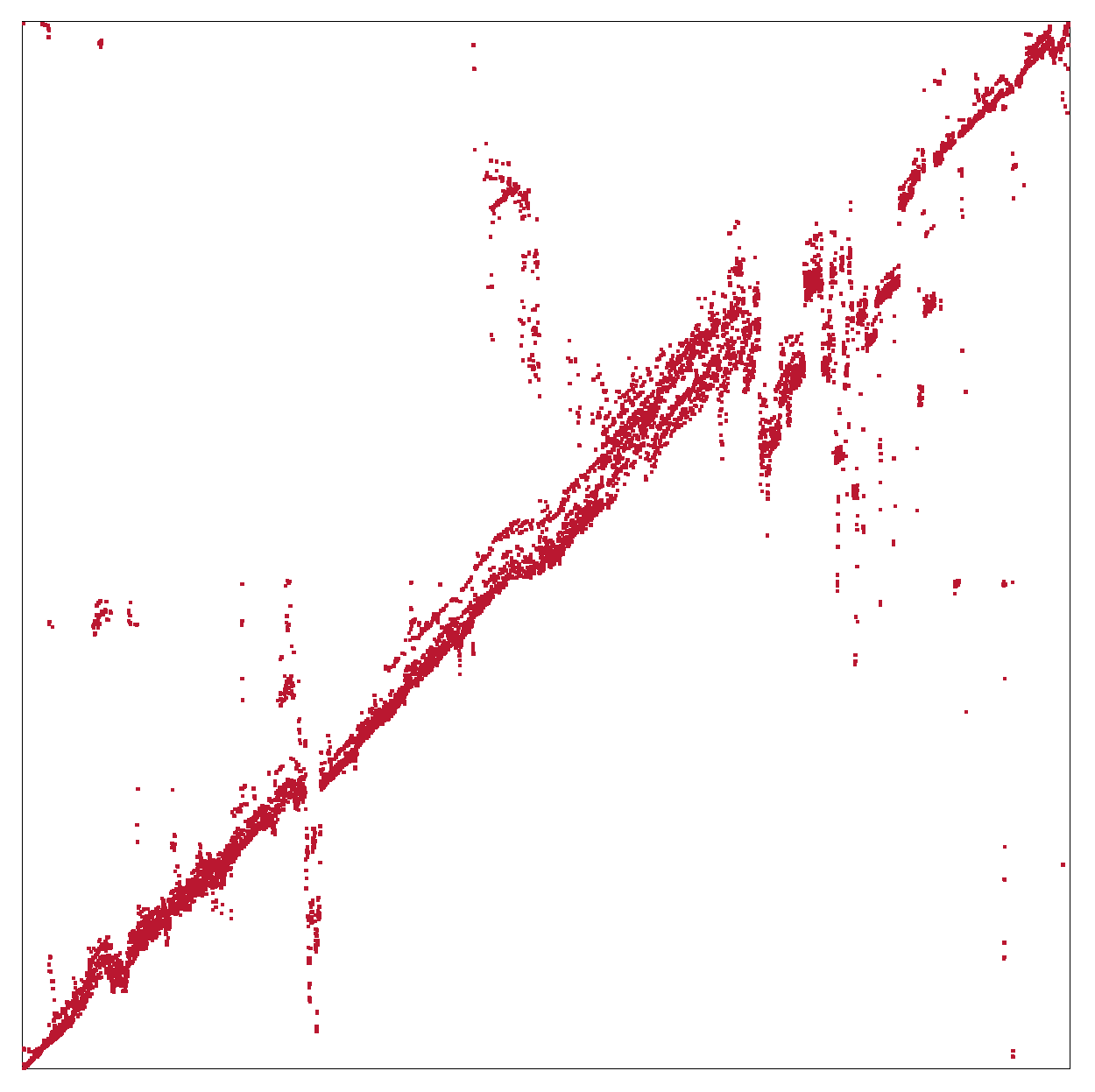}
	\end{minipage}
	\begin{minipage}[c]{0.15\textwidth}
		\centering
		\includegraphics[scale=0.18]{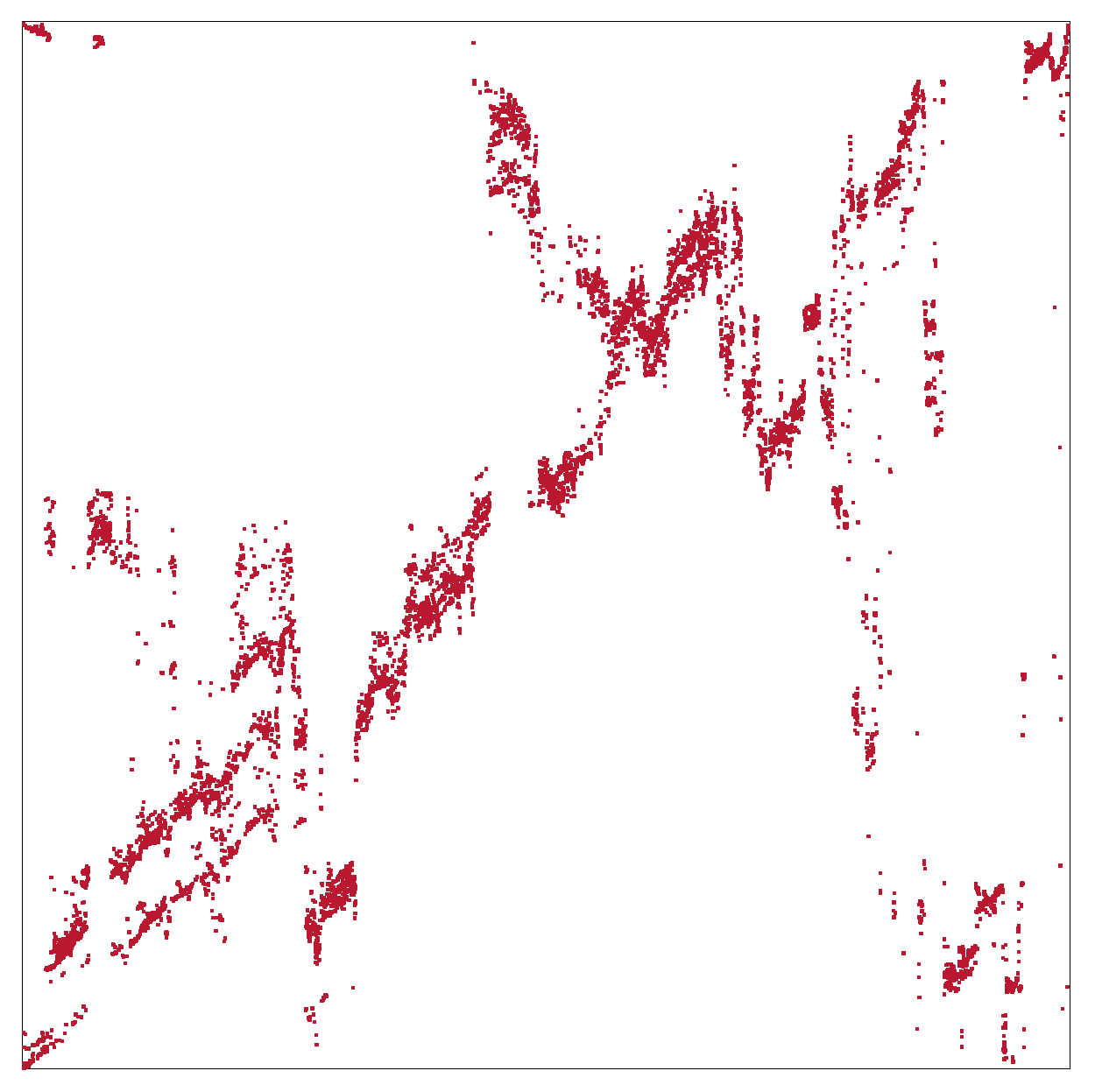}
	\end{minipage}
	\begin{minipage}[c]{0.15\textwidth}
		\centering
		\includegraphics[scale=0.18]{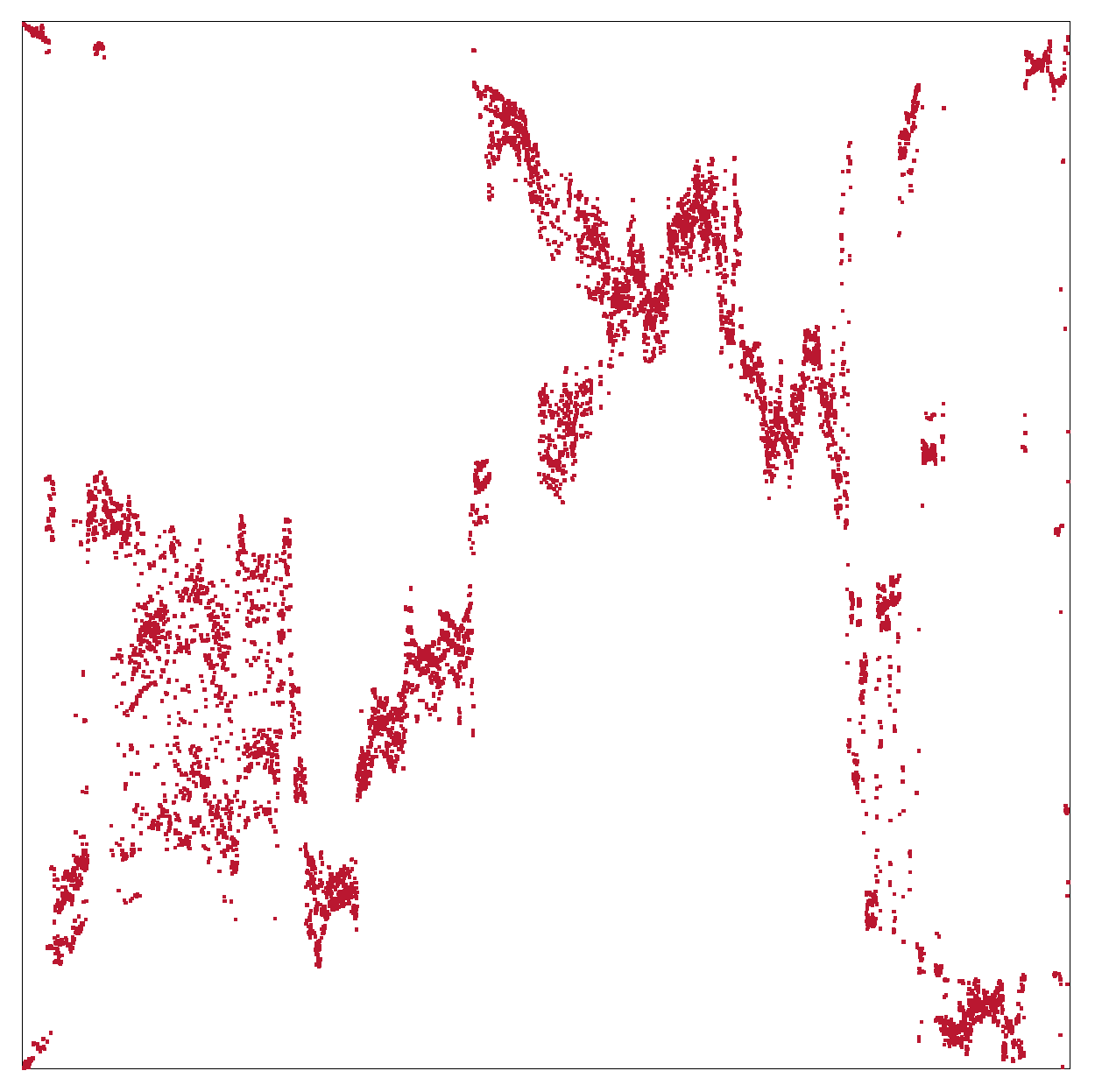}
	\end{minipage}
	\begin{minipage}[c]{0.15\textwidth}
		\centering
		\includegraphics[scale=0.18]{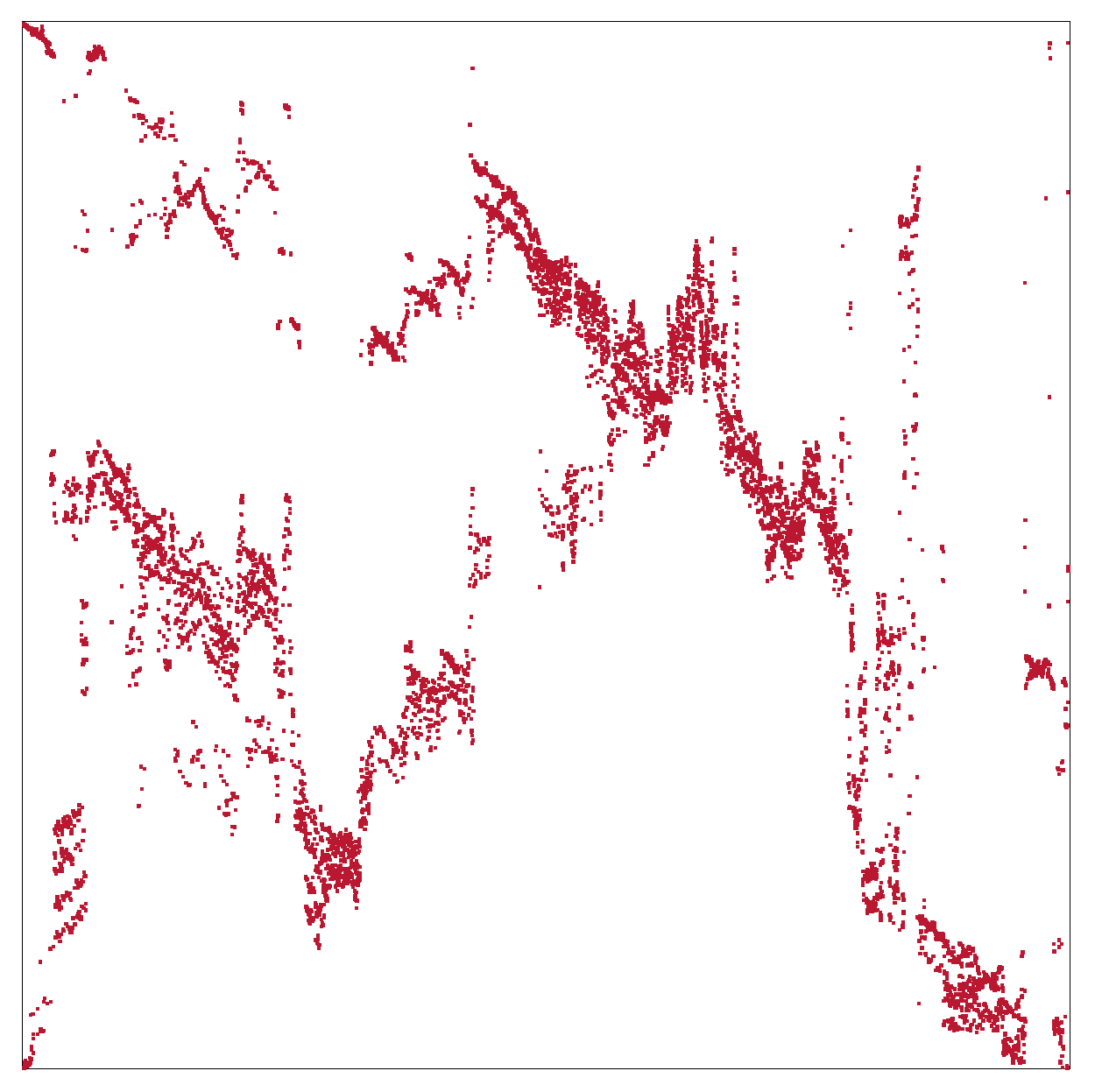}
	\end{minipage}
	\begin{minipage}[c]{0.15\textwidth}
		\centering
		\includegraphics[scale=0.18]{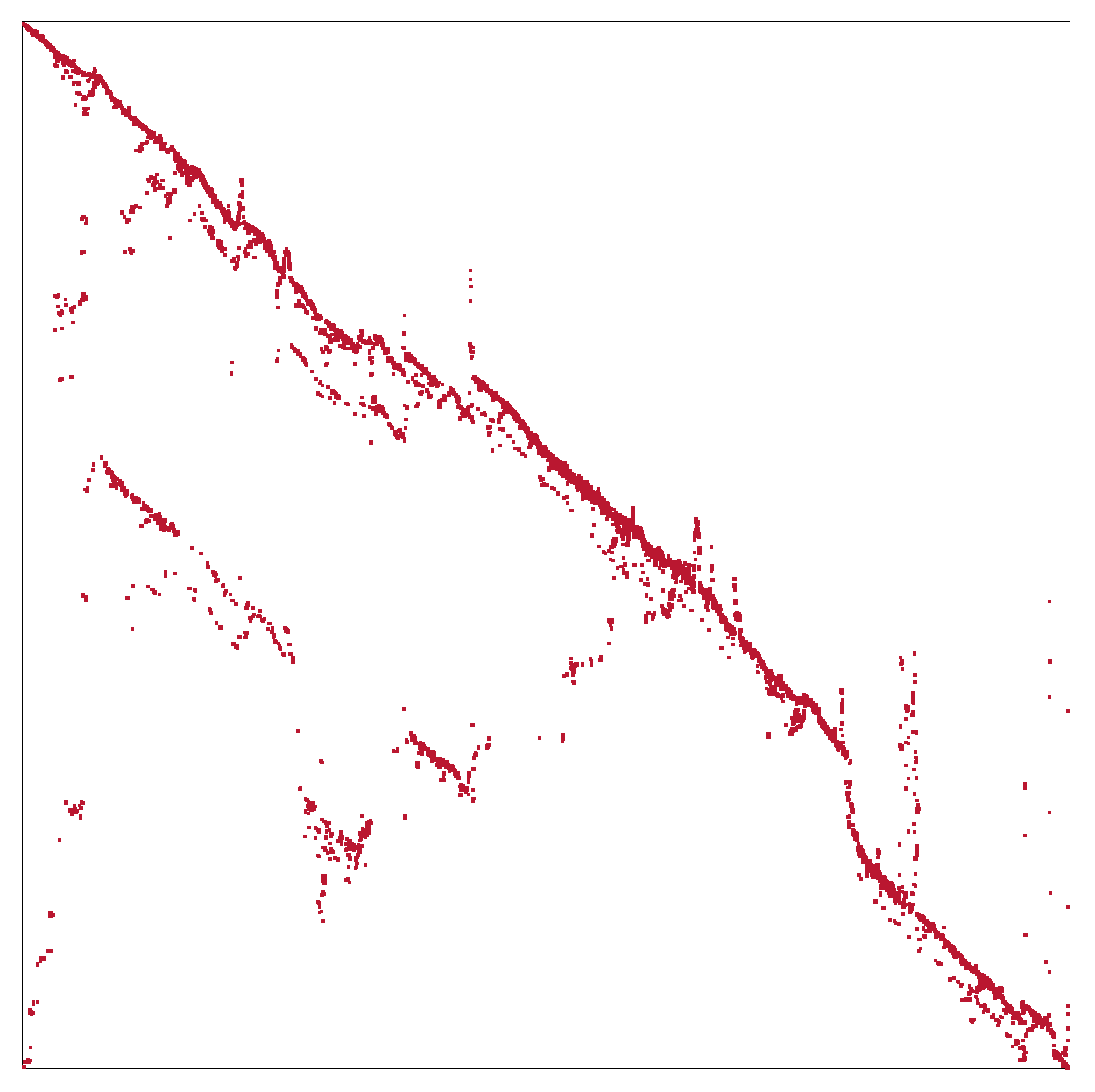}
	\end{minipage}
	\hspace{0.1 cm}
	\begin{minipage}[c]{0.15\textwidth}
		\centering
		\includegraphics[scale=0.18]{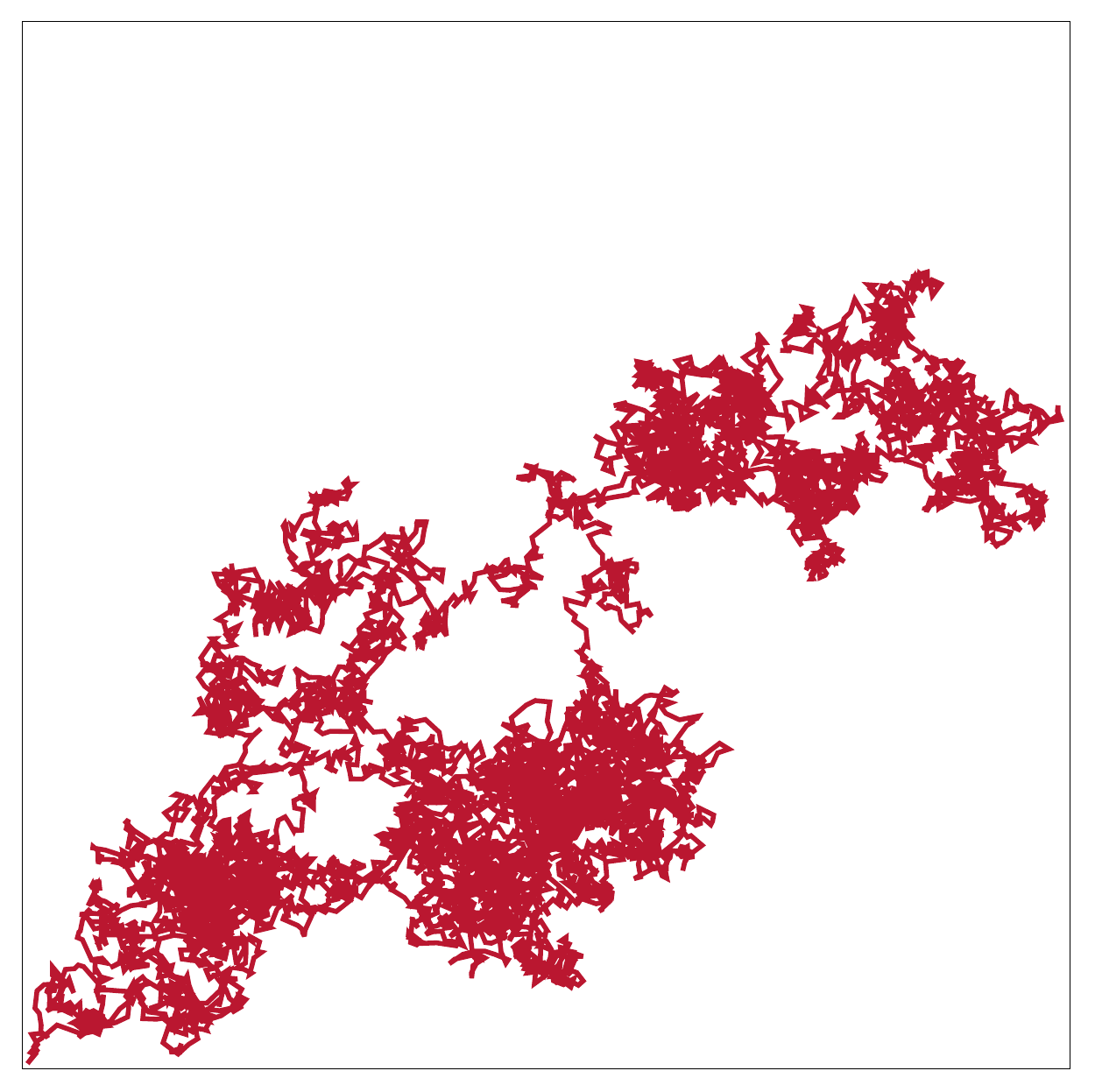}
	\end{minipage}
	
	\begin{minipage}[c]{0.06\textwidth}
		\centering
		\footnotesize$\rho=$\\
		\footnotesize$0.5$
	\end{minipage}
	\begin{minipage}[c]{0.15\textwidth}
		\centering
		\includegraphics[scale=0.18]{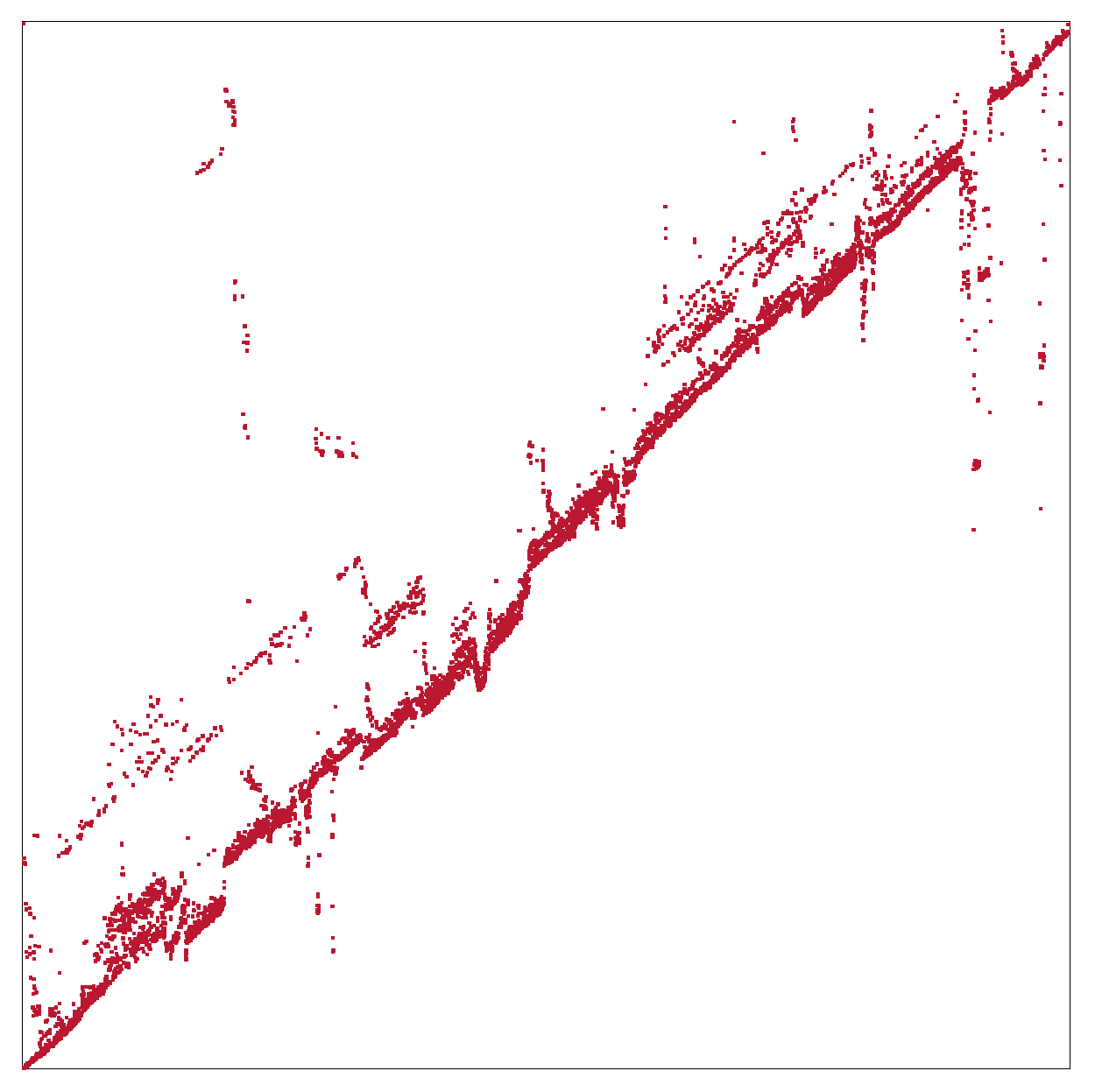}
	\end{minipage}
	\begin{minipage}[c]{0.15\textwidth}
		\centering
		\includegraphics[scale=0.18]{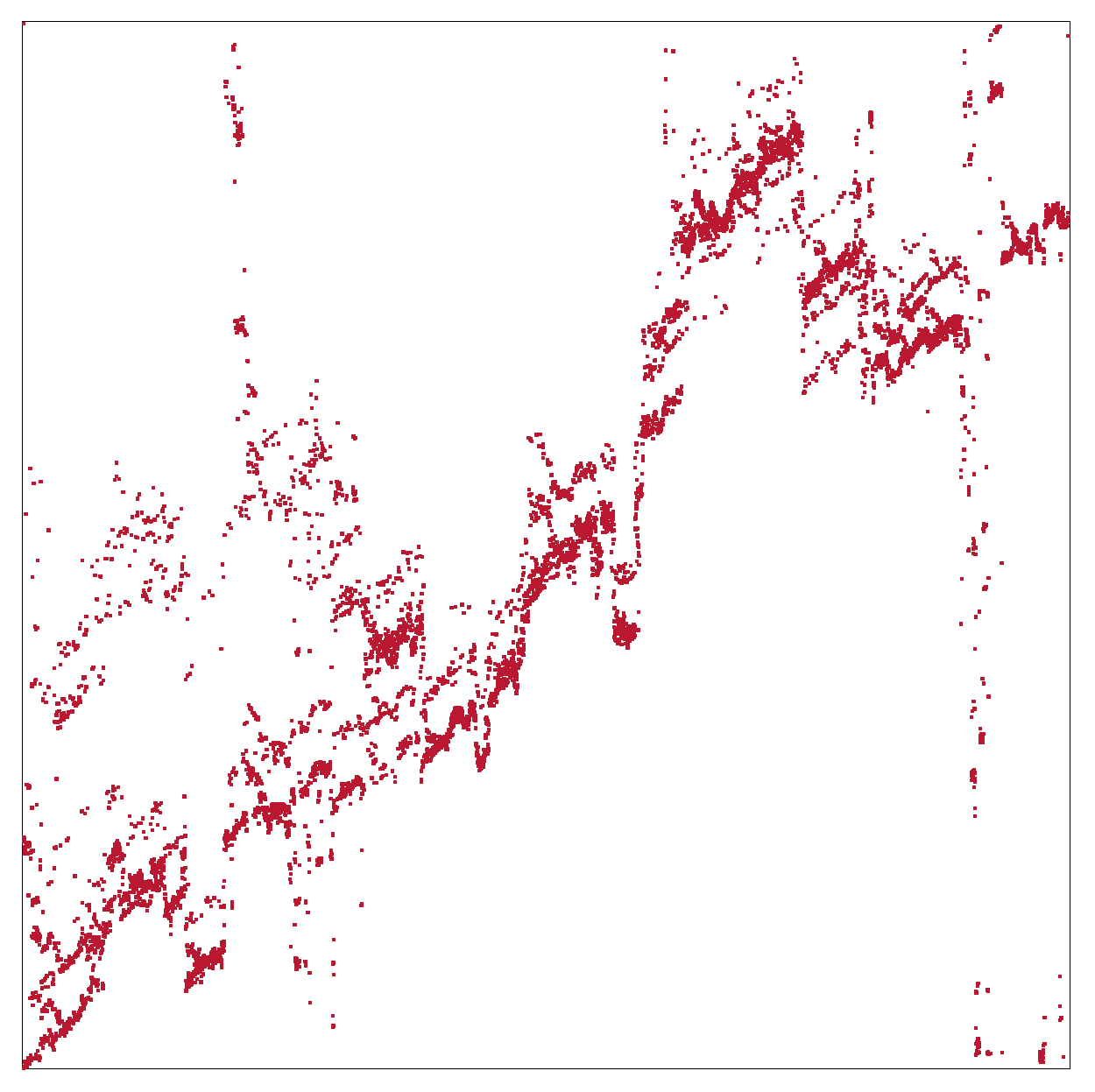}
	\end{minipage}
	\begin{minipage}[c]{0.15\textwidth}
		\centering
		\includegraphics[scale=0.18]{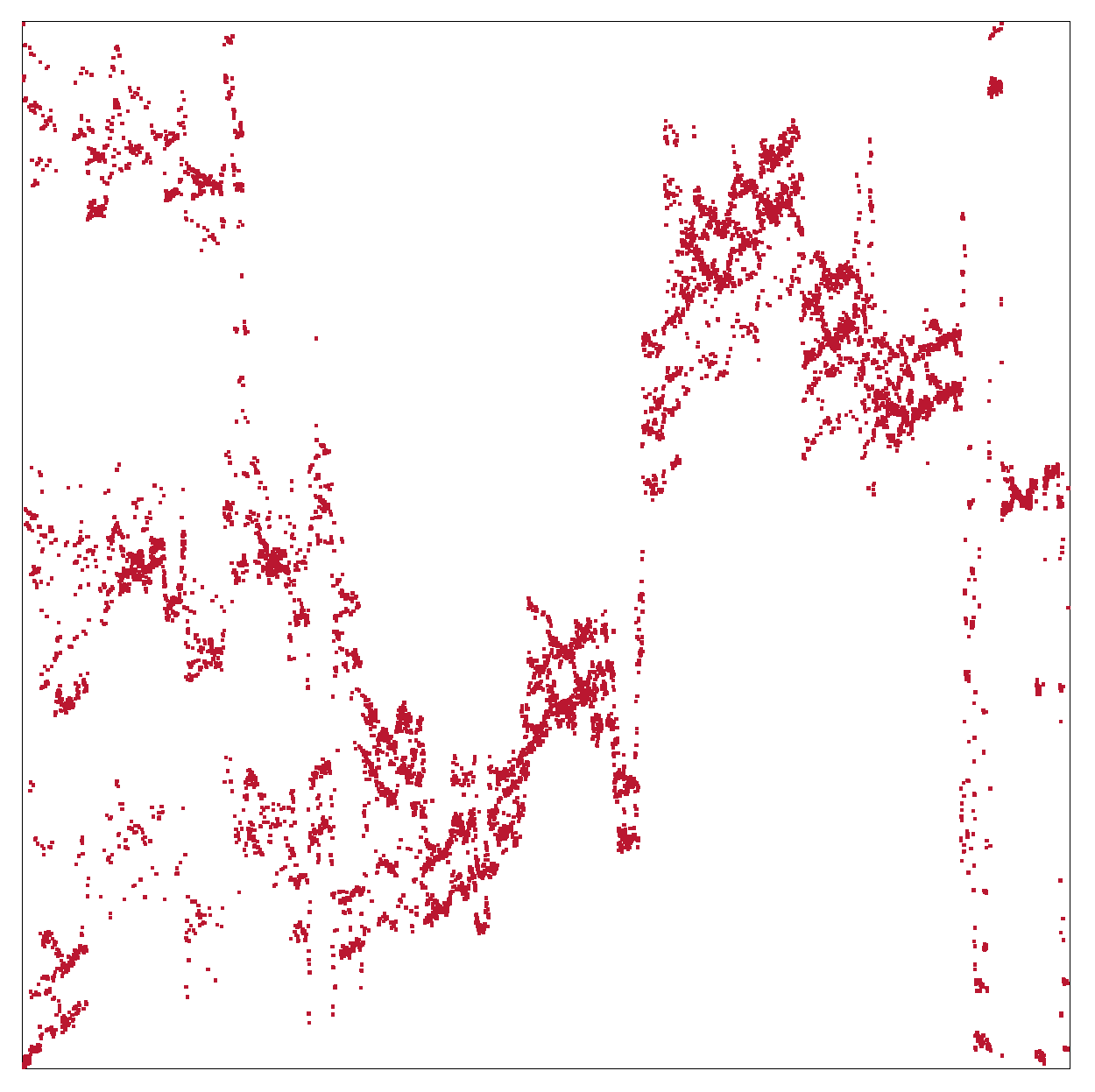}
	\end{minipage}
	\begin{minipage}[c]{0.15\textwidth}
		\centering
		\includegraphics[scale=0.18]{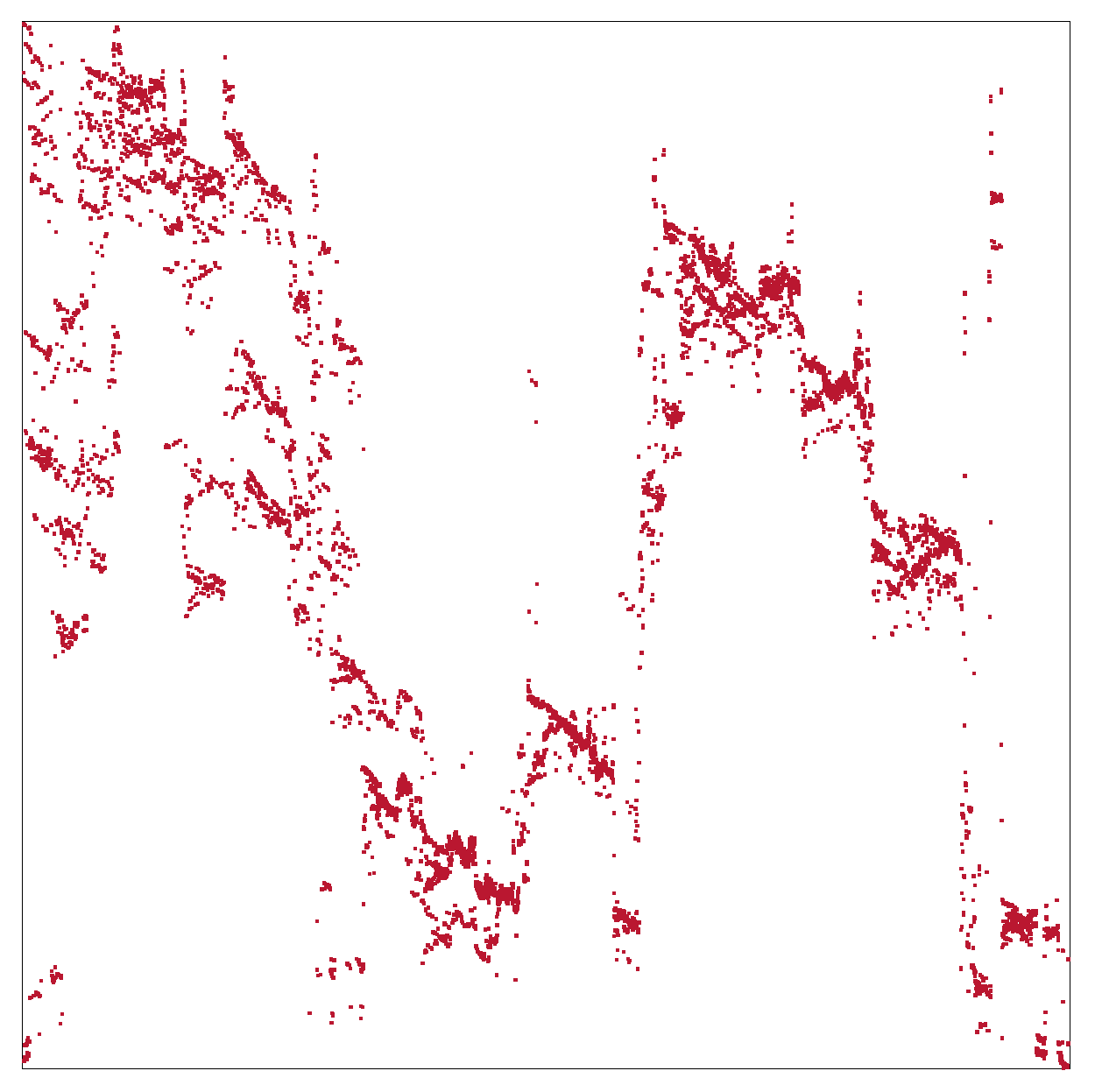}
	\end{minipage}
	\begin{minipage}[c]{0.15\textwidth}
		\centering
		\includegraphics[scale=0.18]{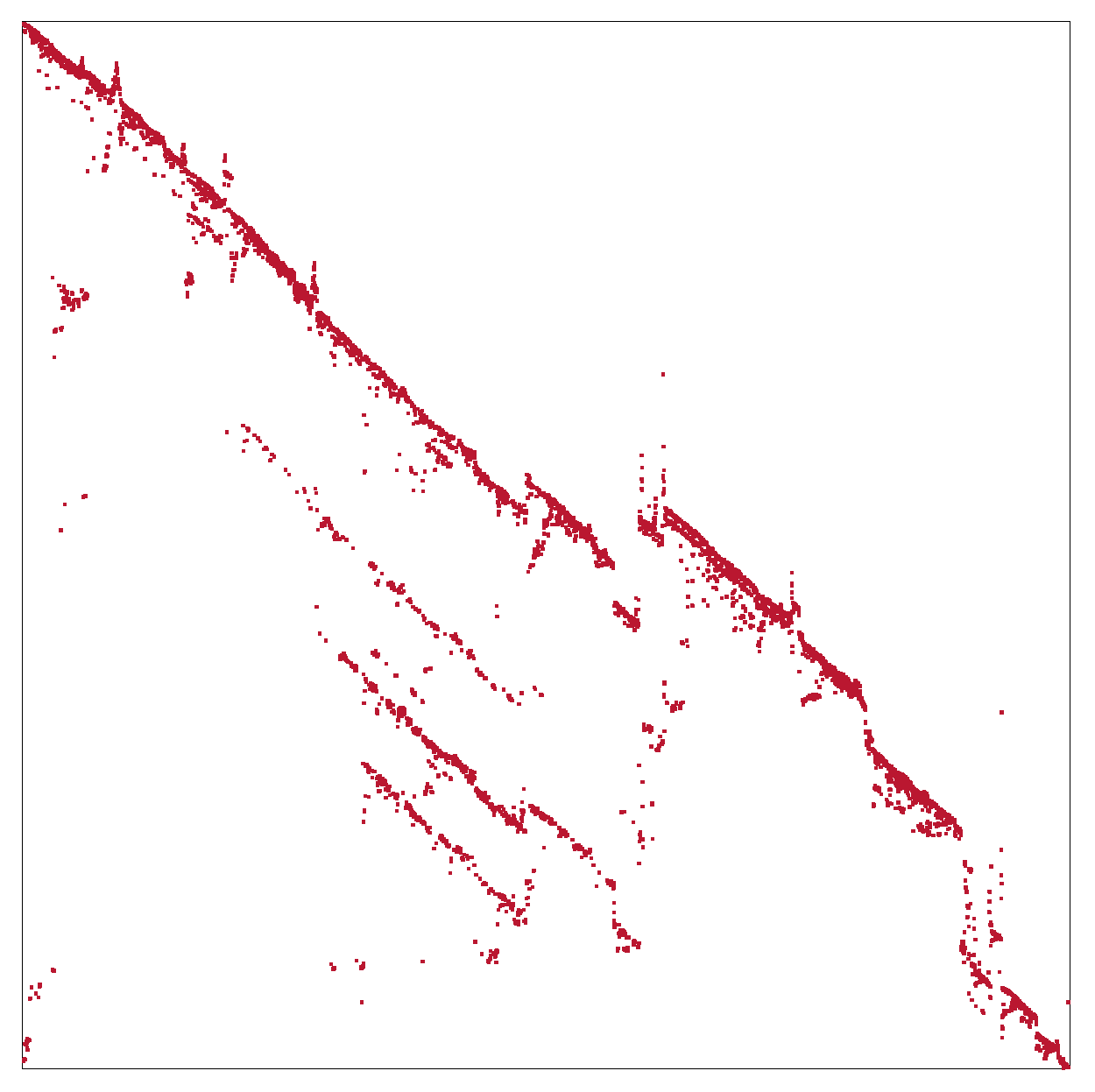}
	\end{minipage}
	\hspace{0.1 cm}
	\begin{minipage}[c]{0.15\textwidth}
		\centering
		\includegraphics[scale=0.18]{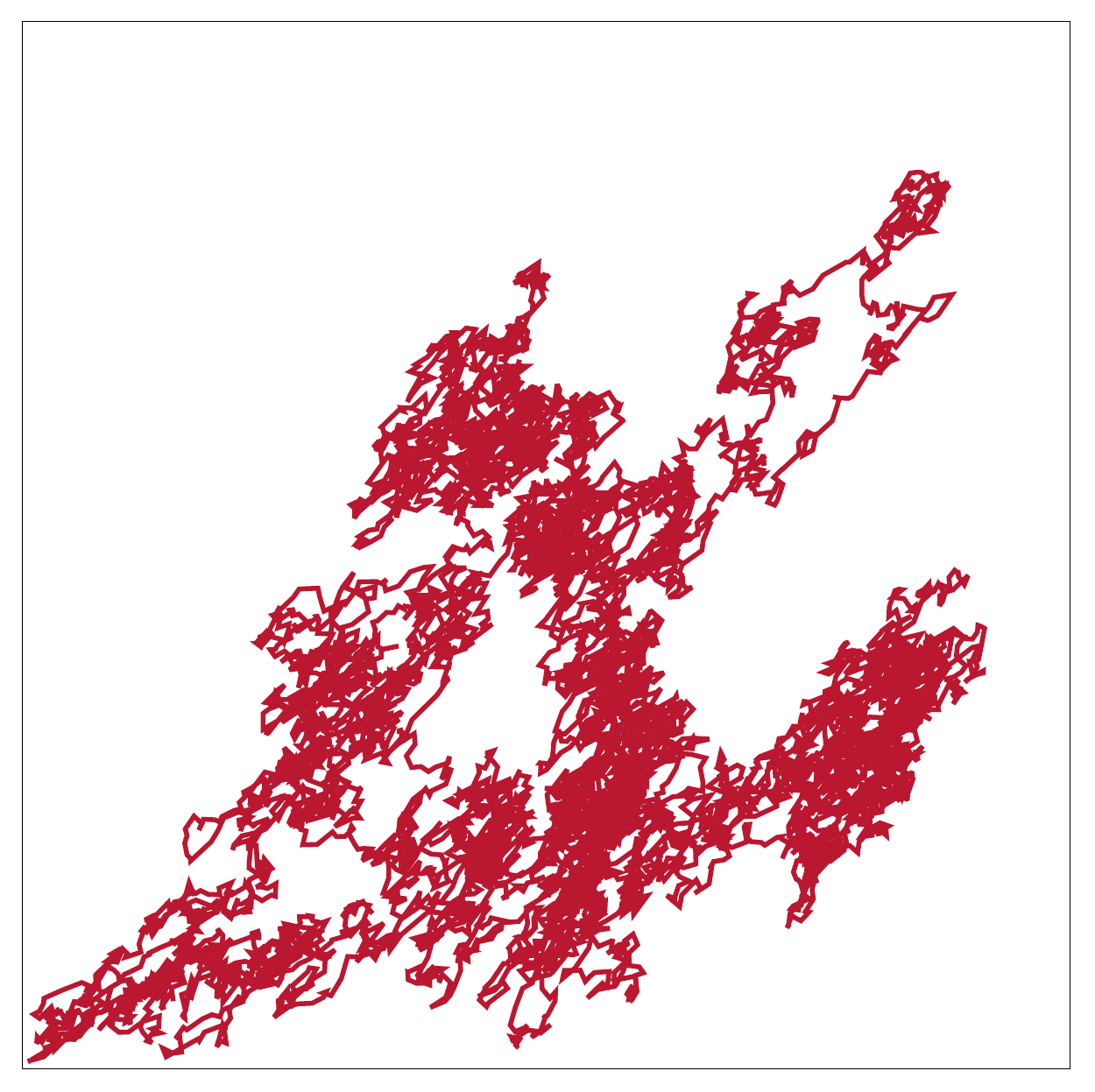}
	\end{minipage}
	
	\begin{minipage}[c]{0.06\textwidth}
		\centering
		\footnotesize$\rho=$\\
		\footnotesize$0.995$
	\end{minipage}
	\begin{minipage}[c]{0.15\textwidth}
		\centering
		\includegraphics[scale=0.18]{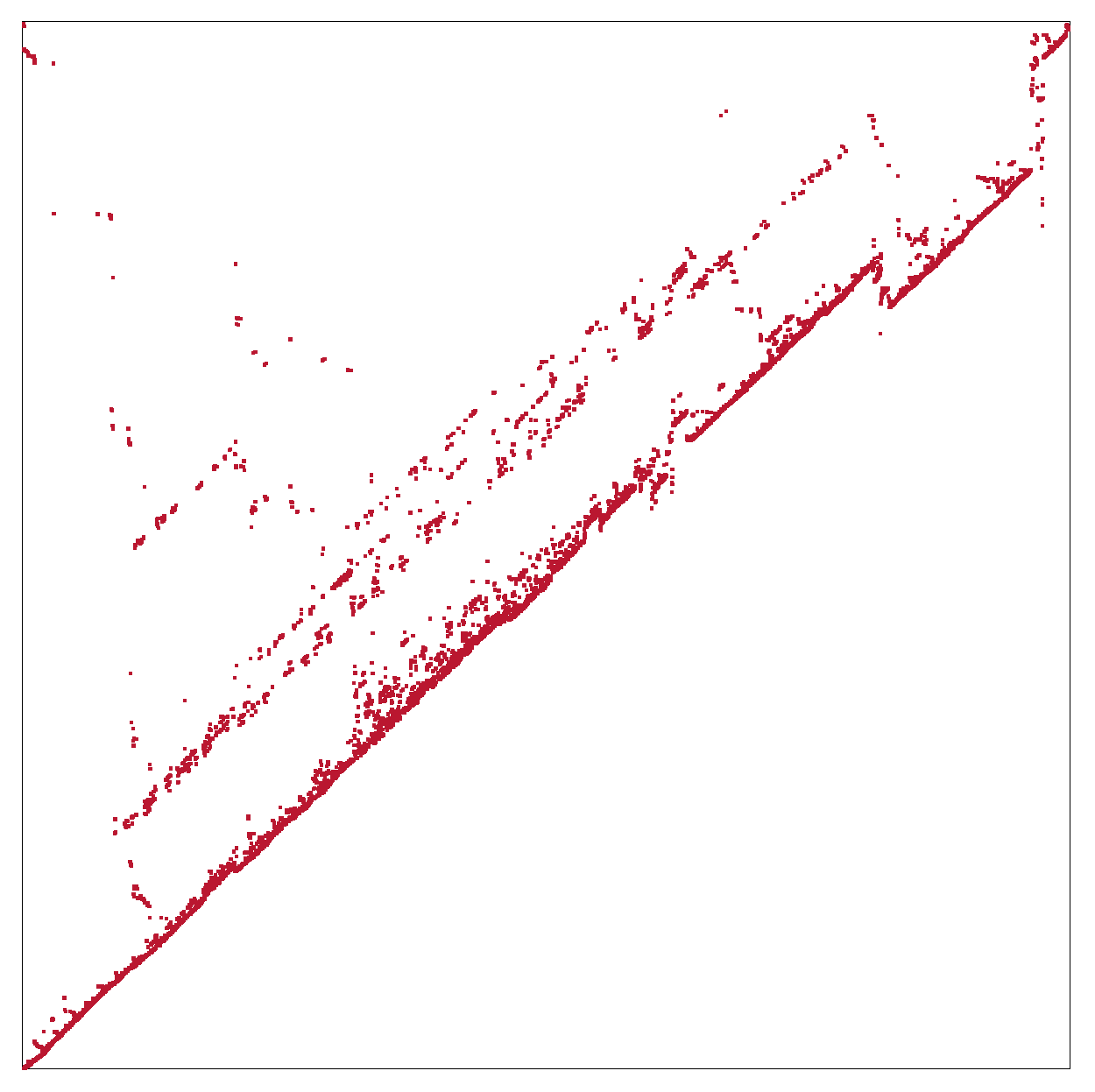}
	\end{minipage}
	\begin{minipage}[c]{0.15\textwidth}
		\centering
		\includegraphics[scale=0.18]{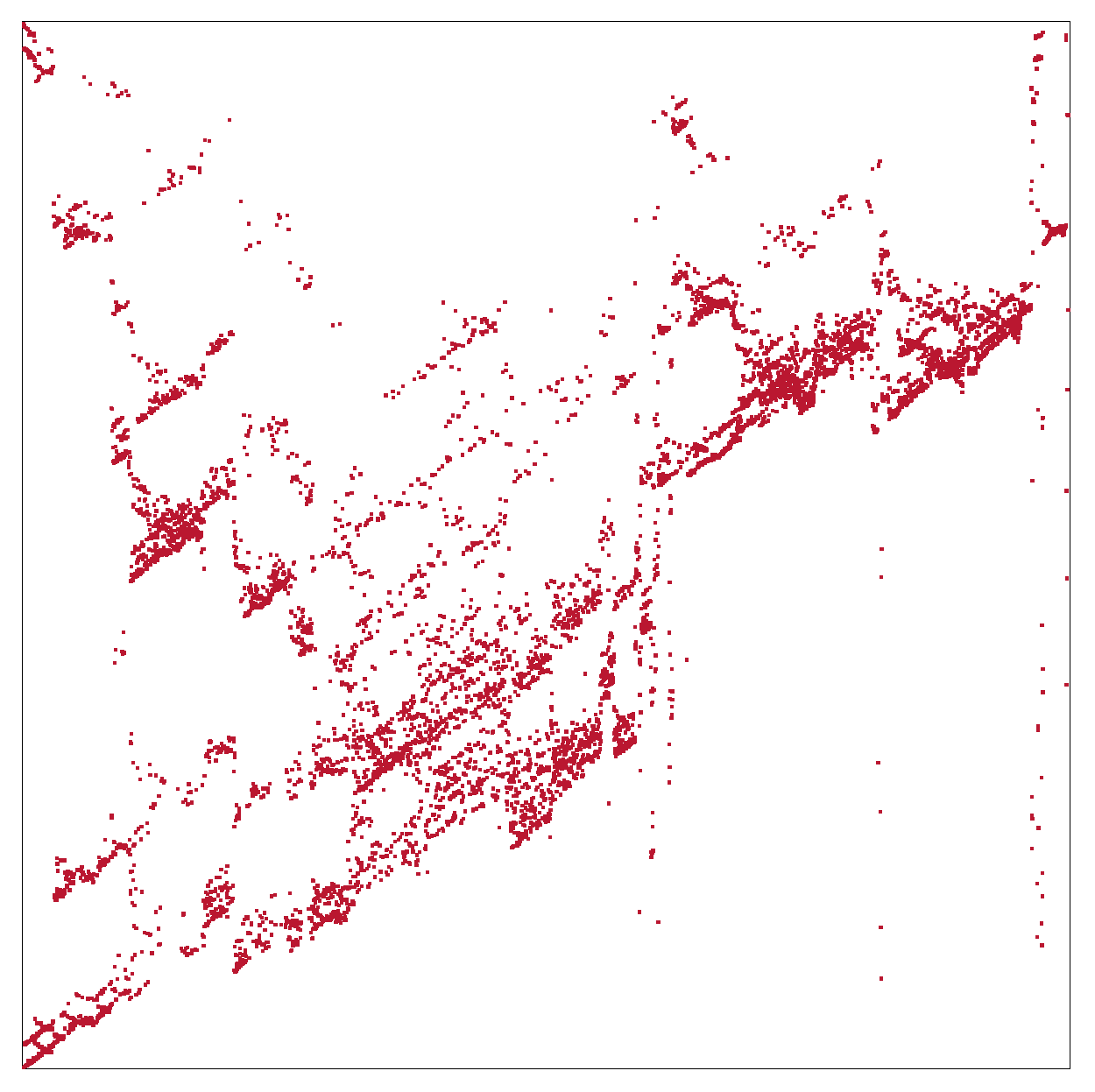}
	\end{minipage}
	\begin{minipage}[c]{0.15\textwidth}
		\centering
		\includegraphics[scale=0.18]{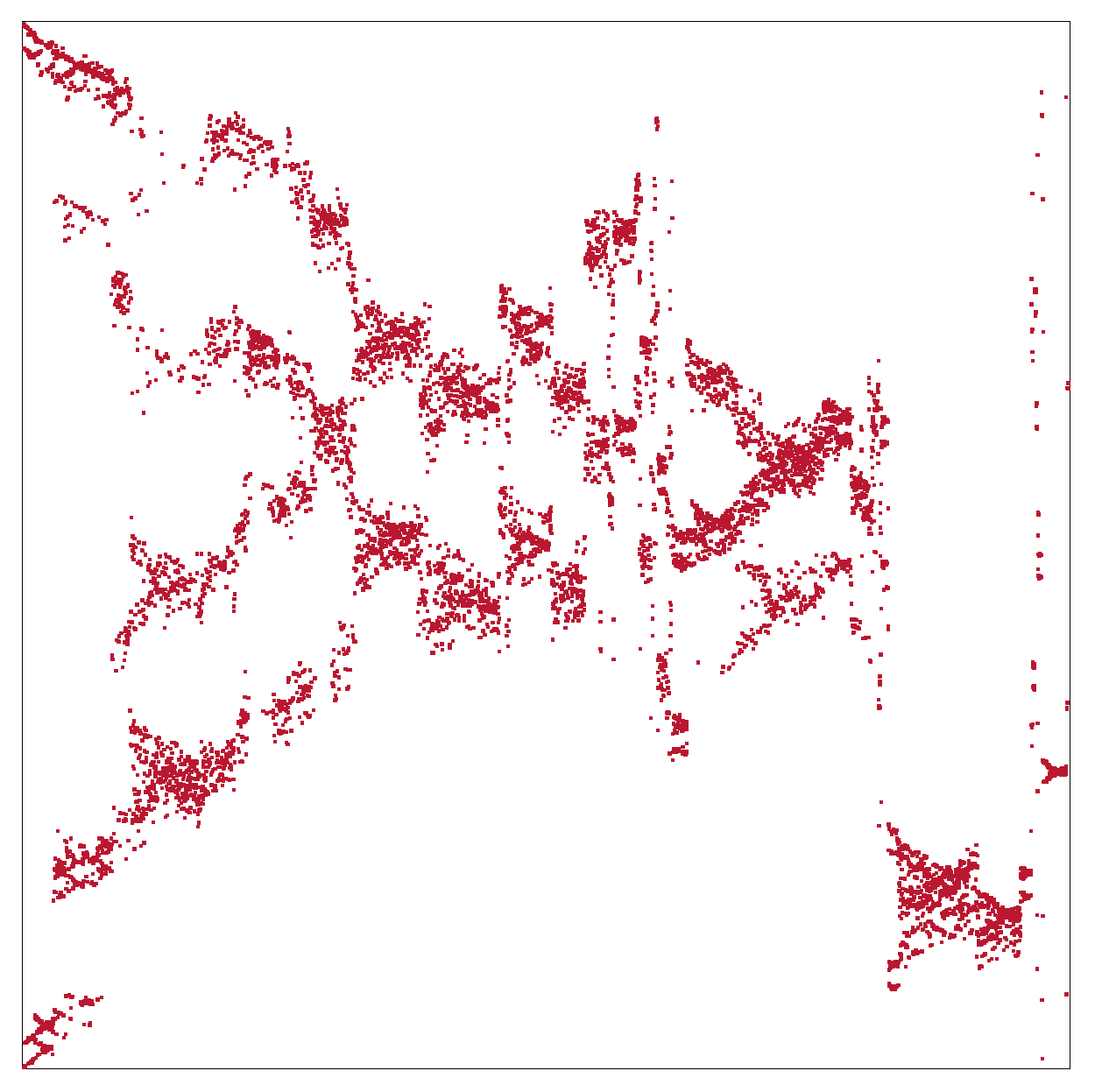}
	\end{minipage}
	\begin{minipage}[c]{0.15\textwidth}
		\centering
		\includegraphics[scale=0.18]{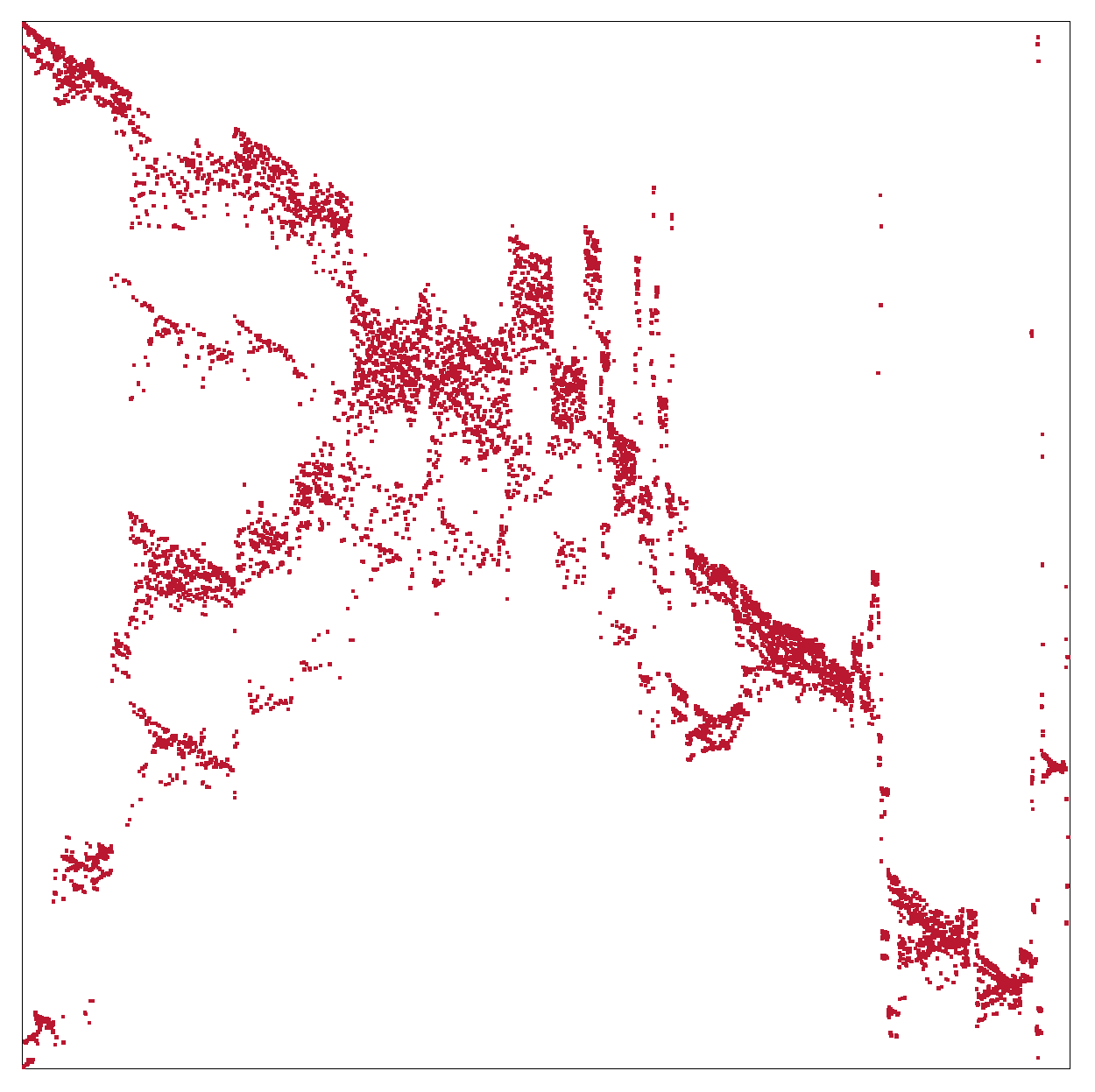}
	\end{minipage}
	\begin{minipage}[c]{0.15\textwidth}
		\centering
		\includegraphics[scale=0.18]{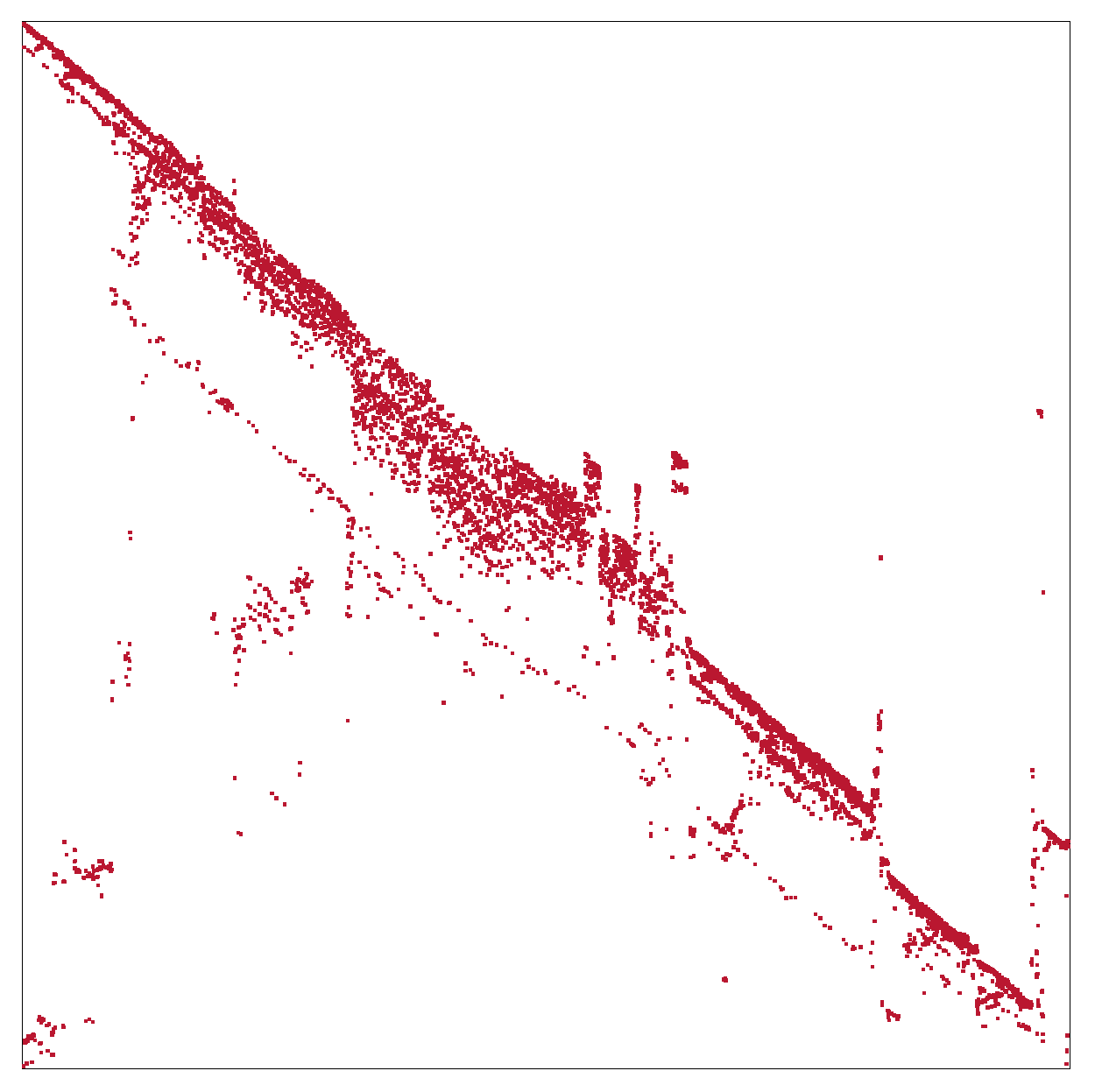}
	\end{minipage}
	\hspace{0.1 cm}
	\begin{minipage}[c]{0.15\textwidth}
		\centering
		\includegraphics[scale=0.18]{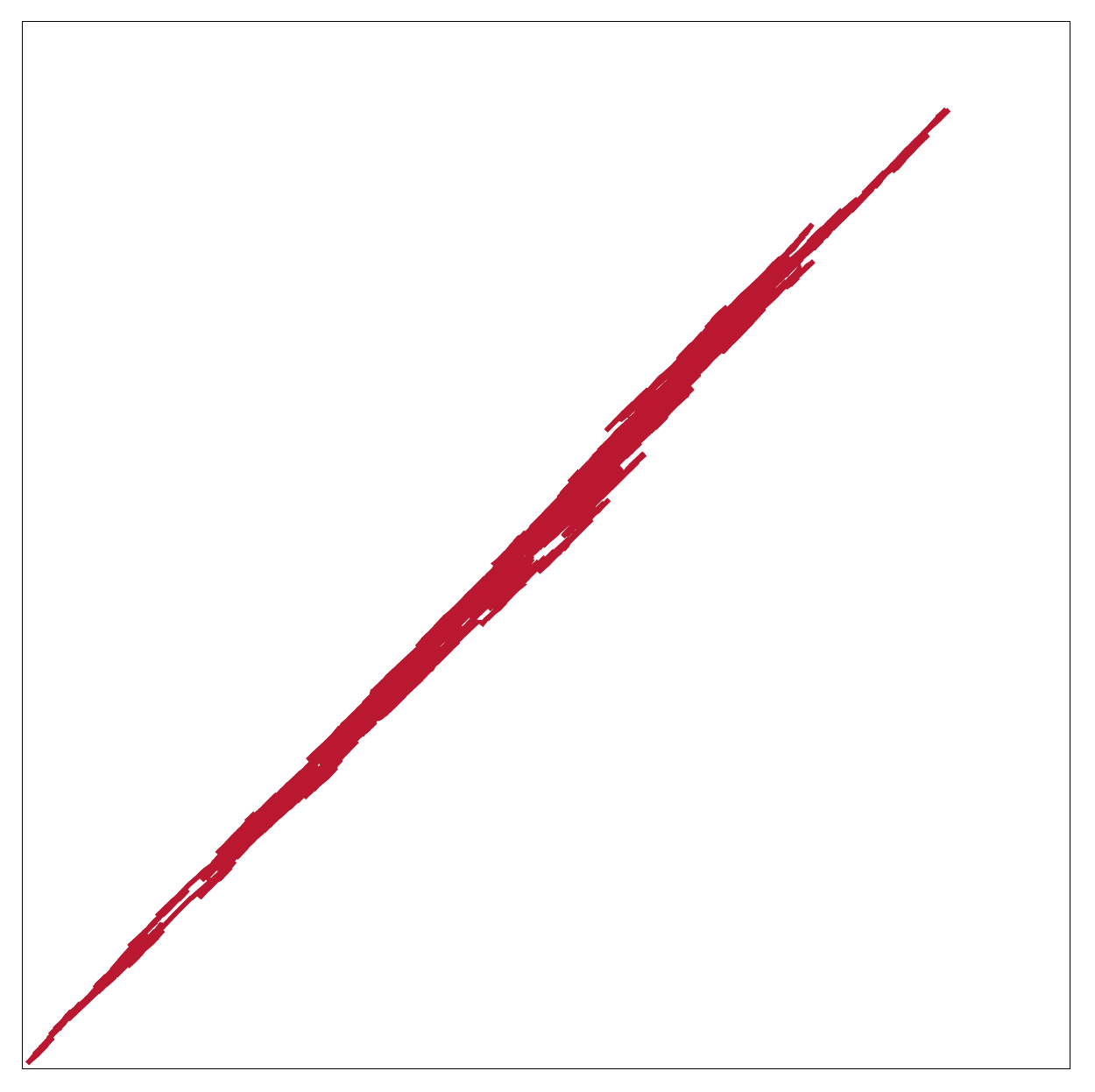}
	\end{minipage}
	
	\begin{minipage}[c]{0.06\textwidth}
		\centering
		\footnotesize$\rho=1$
	\end{minipage}
	\begin{minipage}[c]{0.15\textwidth}
		\centering
		\includegraphics[scale=0.18]{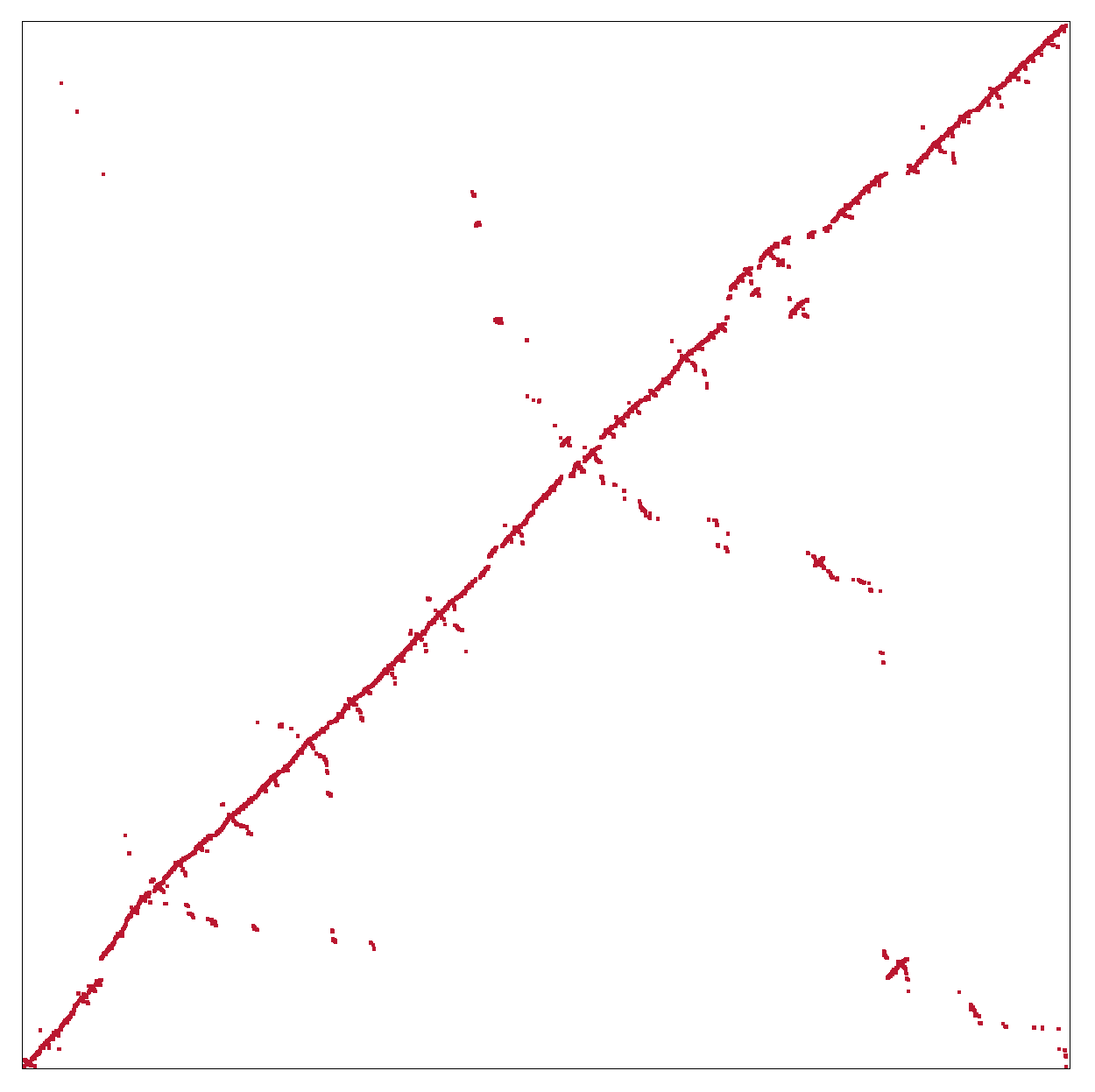}
	\end{minipage}
	\begin{minipage}[c]{0.15\textwidth}
		\centering
		\includegraphics[scale=0.18]{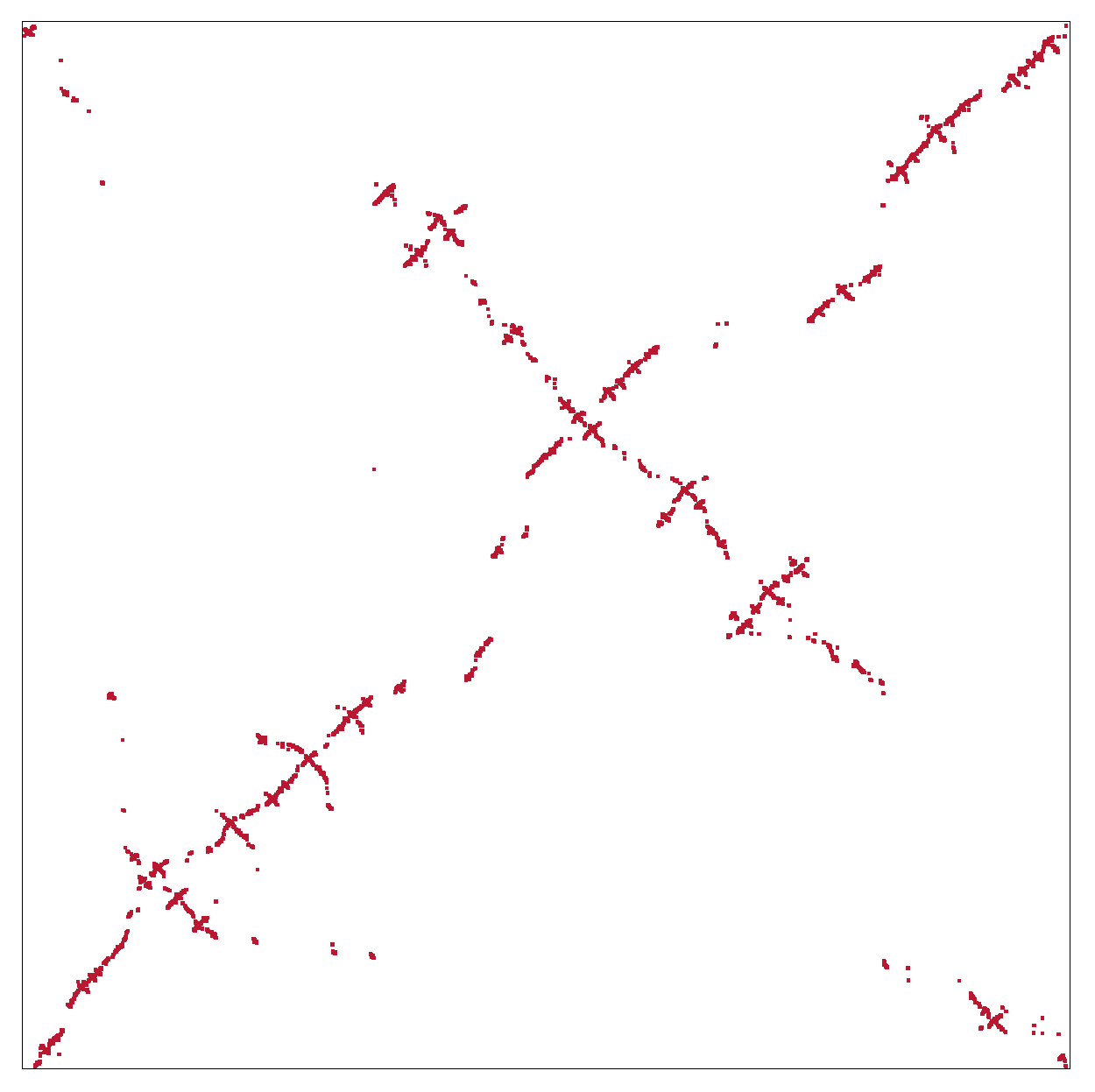}
	\end{minipage}
	\begin{minipage}[c]{0.15\textwidth}
		\centering
		\includegraphics[scale=0.18]{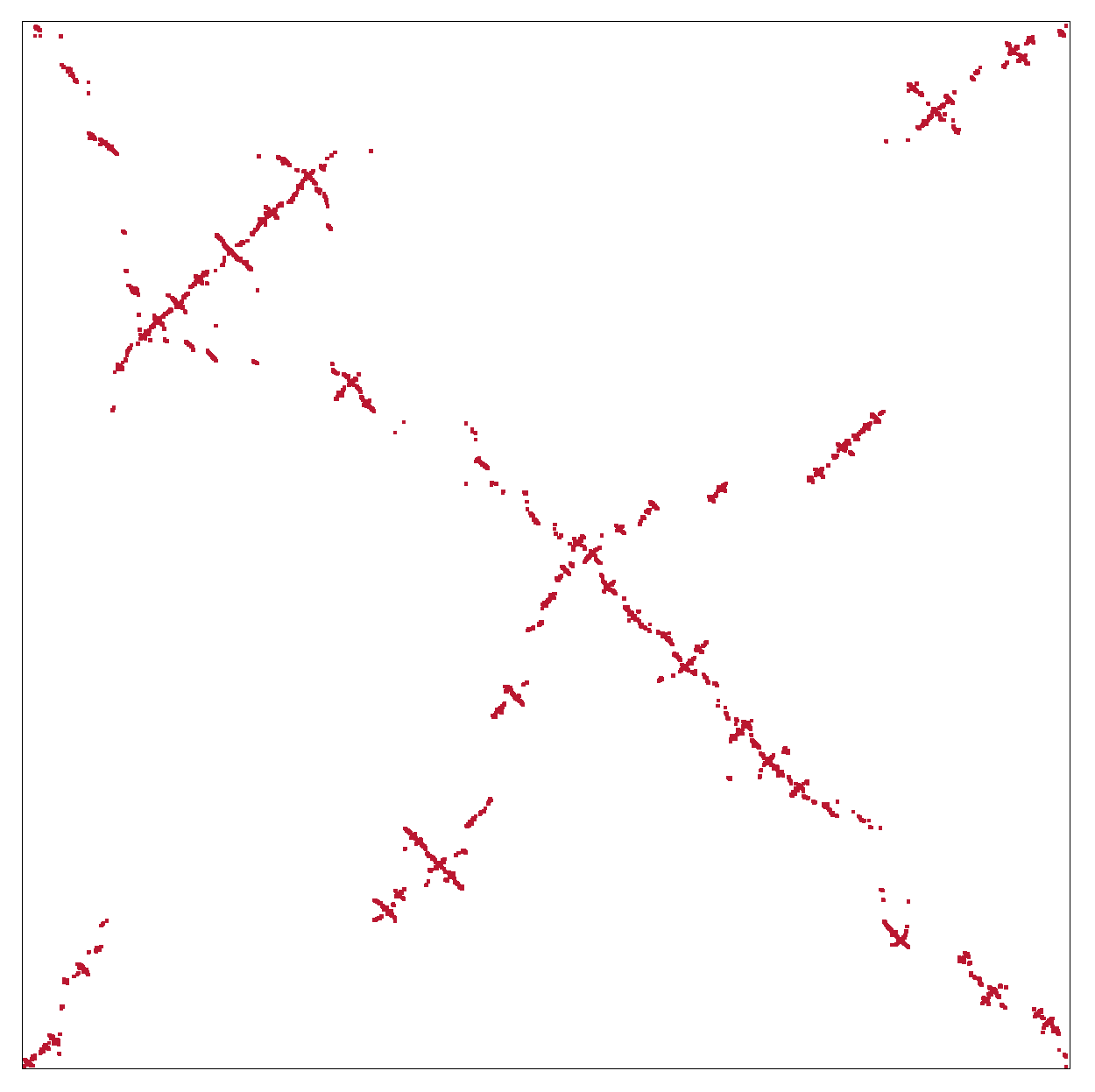}
	\end{minipage}
	\begin{minipage}[c]{0.15\textwidth}
		\centering
		\includegraphics[scale=0.18]{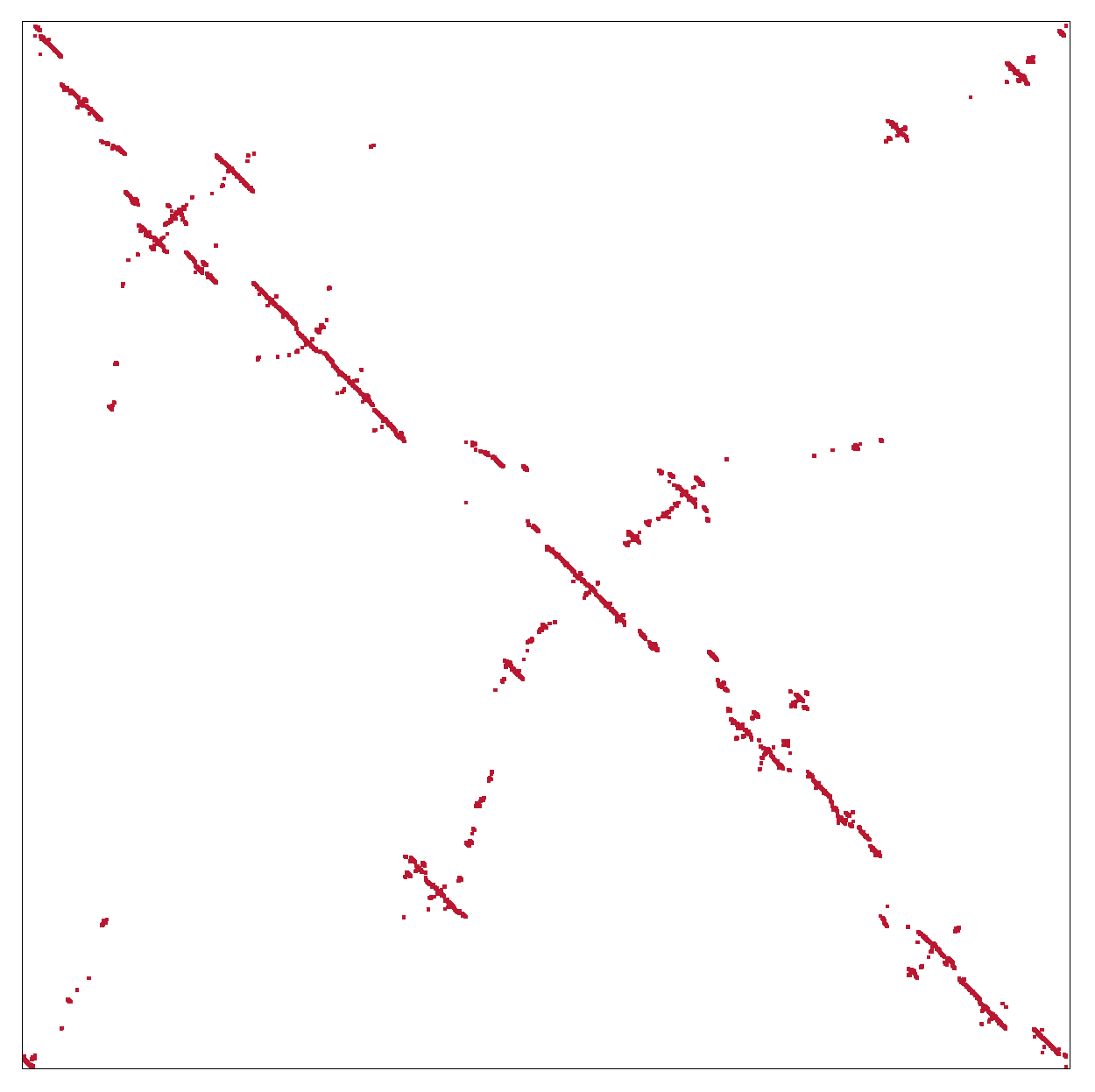}
	\end{minipage}
	\begin{minipage}[c]{0.15\textwidth}
		\centering
		\includegraphics[scale=0.18]{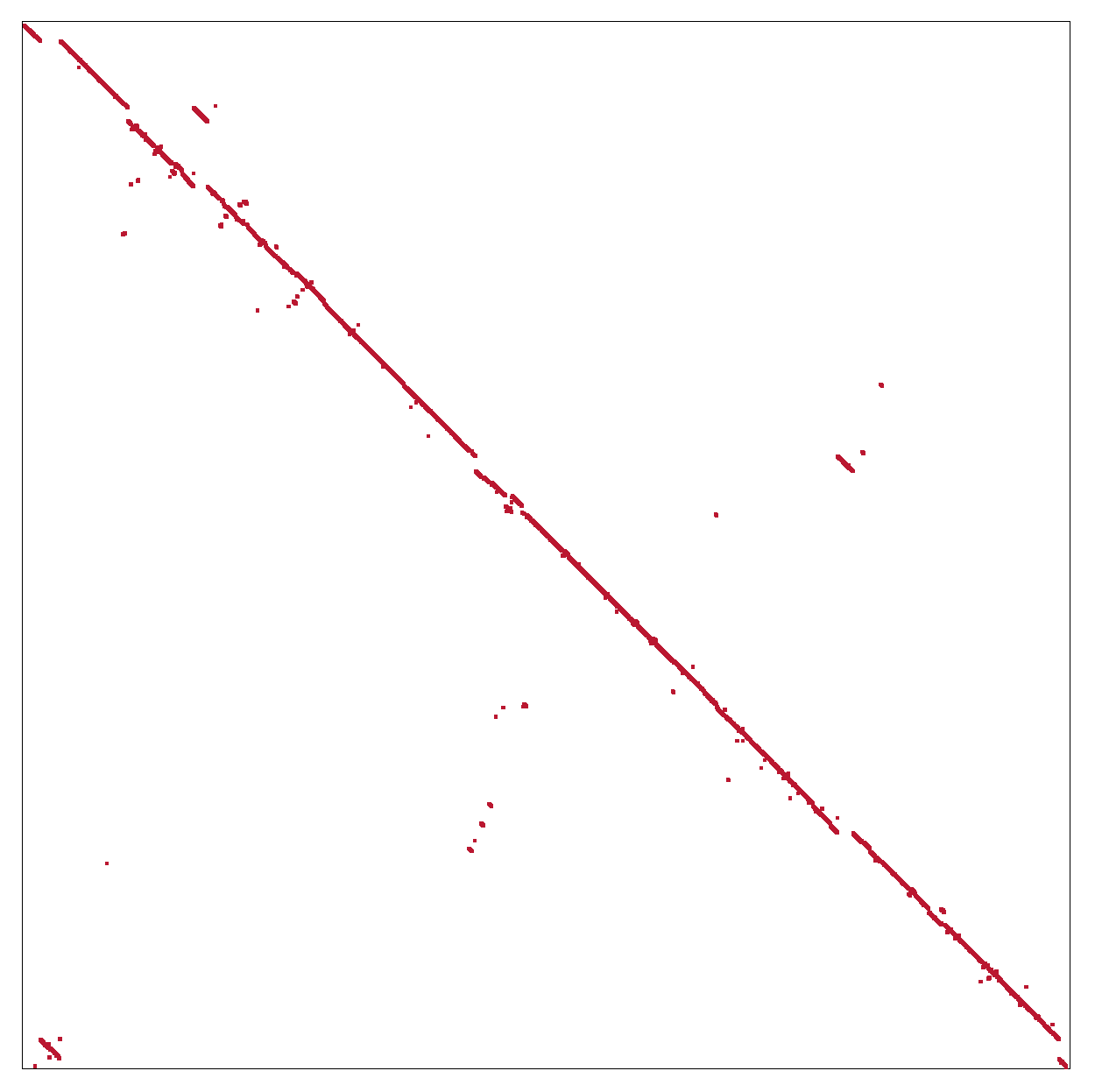}
	\end{minipage}
	\hspace{0.1 cm}
	\begin{minipage}[c]{0.15\textwidth}
		\centering
		\includegraphics[scale=0.18]{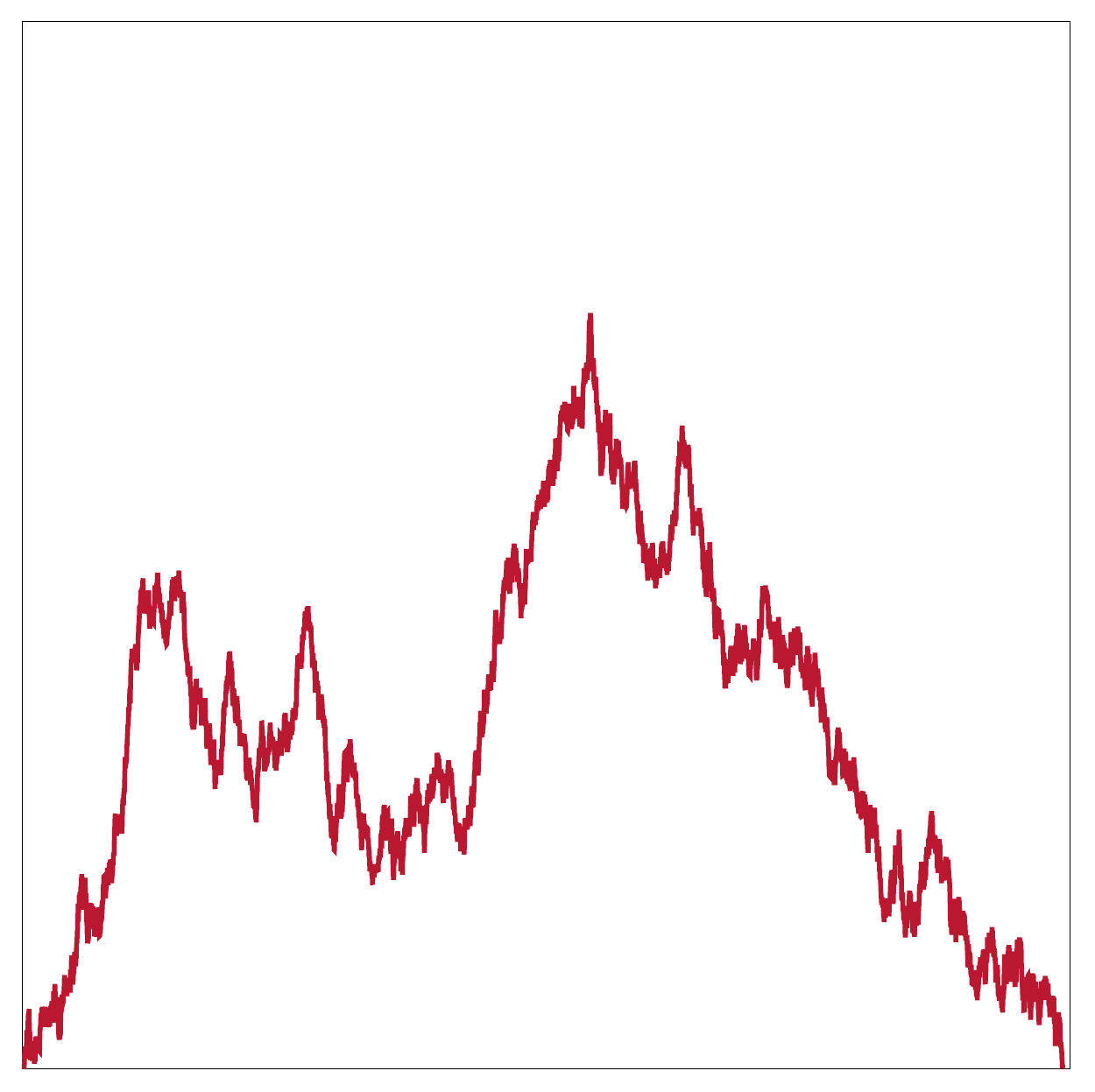}
	\end{minipage}
	\caption{Simulations of the skew Brownian permuton. In each row, five samplings of the skew Brownian permuton and the corresponding two-dimensional Brownian excursions in the non-negative quadrant of correlation $\rho$ (the specific value of $\rho$ is indicated at the beginng of each row). Every column corresponds to different values of the parameter $q$ (the specific value of $q$ is indicated at the top of each column).
		We remark that the five permutons in the same row are driven by the same Brownian excursion plotted at the end of the row. Note that when $\rho=1$ the corresponding two-dimensional Brownian excursion is actually a one-dimensional Brownian excursion.\newline
		The Baxter permuton (corresponding to the parameters $\rho=-1/2$ and $q=1/2$) is highlighted with a blue thicker boundary. The permutons in the bottom row (corresponding to the parameter $\rho=1$) should be skew Brownian permutons as stated in \cref{conj:Baxt_brow_same}.
		}\label{fig:uyievievbeee}
\end{figure}
Once this new random permuton has been defined, many natural questions arise. We mention a few of them.

	\begin{prbl}
		What is the Hausdorff dimension of the support of the skew Brownian permuton $\bm \mu_{\rho,q}$? 
		What is its intensity measure $\E[\bm \mu_{\rho,q}]$? 
	\end{prbl}

We recall (see \cref{rem:maazoun_sep} page~\pageref{rem:maazoun_sep}) that Maazoun~\cite{maazoun17BrownianPermuton} provided some answers to similar questions for the biased Brownian separable permuton (which should correspond, thanks to \cref{conj:Baxt_brow_same}, to the case $\rho=1$). 

Another interesting question is related with the already mentioned (see \cref{sect:op_probl_2} page \pageref{sect:op_probl_2}) recent work of  Bassino, Bouvel, Drmota, F{\'e}ray, Gerin, Maazoun, and Pierrot~\cite{bassino2021increasing}: They showed that the length of the longest increasing subsequence in permutations converging to the Brownian separable permuton has sublinear size (building on self-similarity properties of  the Brownian separable permuton).

\begin{prbl}
	Let $\bm \sigma_n$ be a sequence of random permutations converging to the skew Brownian permuton with parameters $\rho, q$. What is the the length of the longest increasing subsequence in $\bm \sigma_n$? We expect, under some regularity conditions on the sequence $\bm \sigma_n$, a formula that only depends on $\rho, q$.
\end{prbl}

We conclude this section with an additional discussion on the solutions of \cref{eq:flow_SDE_gen}: in the same spirit of \cref{rem:no_flow}, given $\omega$ (even restricted to a set of probability one), we cannot say that $\{\conti Z_{\rho,q}^{(u)}(\omega)\}_{u\in [0,1]}$ forms a whole field of solutions to \cref{eq:flow_SDE_gen}.
Also here, we expect that there exist exceptional times $u\in[0,1]$ where uniqueness fails, with two or three distinct solutions. 
	\begin{prbl}
		What are the properties of the stochastic flow determined by \cref{eq:flow_SDE_gen}? In particular, is it possible to describe the coalescence points and/or non-uniqueness points? What is the interaction between the stochastic flow and its backwards version? 
	\end{prbl}
We believe that answers to these questions will be fundamental in order to deepen the study of the geometric properties of the skew Brownian permuton $\bm \mu_{\rho,q}$.

\subsection{Universality}\label{sect:univ}

We believe that many models of uniform constrained permutations converge to the skew Brownian permuton with various values of the two parameters.

In \cref{thm:scaling_intro} we showed that the biased Brownian separable permuton has a universal property: it is the permuton limit of uniform random permutations in proper substitution-closed classes.  Assuming that \cref{conj:Baxt_brow_same} holds, about which we are very confident as explained before, \cref{thm:scaling_intro} is a first evidence of the universality of the skew Brownian permuton. A second evidence is \cref{thm:permuton} about Baxter permutations. 
The unsatisfactory feature of all these results is that either $\rho=1$ or $q=1/2$, and in both cases
the SDEs in \cref{eq:flow_SDE_gen} take a simplified form: either the driving process is a one-dimensional Brownian excursion or the local time term cancels. The goal of a future project is the following.

	\begin{prbl}\label{prbl:find_fam}
		Exhibit a natural model of random permutations converging to the skew Brownian permuton with parameters $|\rho|<1$ and $q\neq1/2$. 
	\end{prbl}

Recall that a key point in the proof of \cref{thm:permuton} is 
the possibility to encode a Baxter permutation with a specific two-dimensional walk and then to construct from it a discrete coalescent-walk process.
A good candidate for solving \cref{prbl:find_fam} are the families of semi-Baxter and strong-Baxter permutations (a superset and a subset of Baxter permutations, respectively). These families are encoded by two-dimensional labeled generating trees (see \cite{MR3882946}): using the machinery developed in \cref{sect:bij_color_walk_perm} we can easily sample a uniform semi-Baxter or strong-Baxter permutation as a two-dimensional random walk conditioned to stay in a cone. Moreover, from these walks it is possible to define discrete coalescent-walk processes for uniform semi-Baxter or strong-Baxter permutations.

	\begin{conj}\label{conj:semi-bax}
		The permuton limit of uniform semi-Baxter or strong-Baxter permutations (see \cref{fig:wekijfvuiyewbf}) is the skew Brownian permuton. In particular, in the case of strong-Baxter permutations, the limiting skew Brownian permuton has parameters $|\rho|<1$ and $q\neq1/2$.
	\end{conj}

\begin{figure}[htbp]
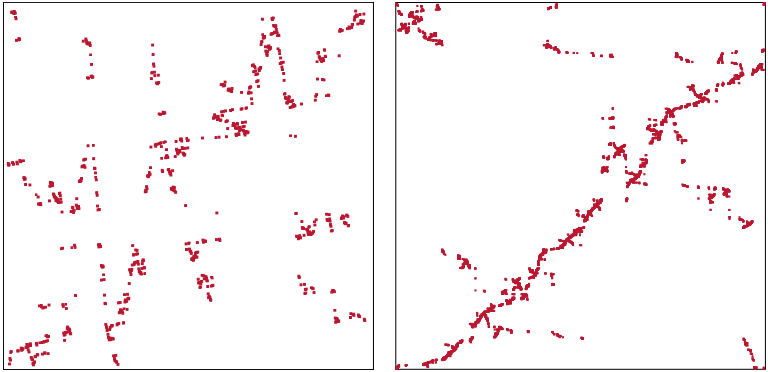

	\begin{minipage}[c]{0.70\textwidth}
		\centering
		\includegraphics[scale=0.099]{Semi_baxter_1_823}
		\includegraphics[scale=0.4]{Strong_baxter_1_11009}
	\end{minipage}
	\begin{minipage}[c]{0.29\textwidth}
		\caption{A large uniform semi-Baxter permutation (left) and a large uniform strong-Baxter permutation (right).
		\label{fig:wekijfvuiyewbf}}
	\end{minipage}
\end{figure}

There are various reasons to believe in this conjecture. First of all, it is easy to predict that $|\rho|<1$ simply computing the correlation between the coordinates of the two-dimensional random walk that describes a uniform random semi-Baxter or strong-Baxter permutation.

In addition, we have done both numerical simulations and informal computations to estimate the parameter $q$ in both cases, obtaining that $q\neq1/2$ for strong-Baxter permutations, while $q=1/2$ for semi-Baxter permutations. It is somehow surprising that the parameter $q$ for  semi-Baxter permutations is exactly $1/2$. Indeed, we can show that in both cases the trajectories of the corresponding discrete coalescent-walk processes are not martingales, and so we would expect a local time term appearing in the corresponding SDEs in both cases. This is another aspect that deserves a more accurate investigation in future projects.

Semi-Baxter and strong-Baxter permutations are just two possible suitable models of random permutations for solving \cref{prbl:find_fam}. Thanks to the machinery developed in \cref{sect:bij_color_walk_perm}, we could investigate many other families of permutations encoded by a two-dimensional labeled generating tree (for instance see \cite{bousquet2003four,MR2376115,MR3206150,MR3882946}).

\subsection{Connections with the LQG}

In this final section we focus on the connection between the skew Brownian permuton and the Liouville quantum gravity (LQG), with a particular focus on the work of Gwynne, Holden and Sun \cite{gwynne2016joint}. This connection was already mentioned in this manuscript\footnote{For more details see also \cite[Section 1.6 and Appendix B.2]{borga2020scaling}.}: indeed bipolar orientations\footnote{This is a corollary of our \cref{thm:joint_scaling_limits} that generalizes the results of \cite{MR3945746,gwynne2016joint}.} converge (in the Peano-sphere topology) jointly with its dual to a $\sqrt{4/3}$-LQG decorated with two orthogonal SLE${}_{12}$ (in the imaginary geometry sense). 

We believe that the SDEs in \cref{eq:flow_SDE_gen} defining the skew Brownian permuton also encode some quantities related to the LQG. We now explain this point more precisely.

\medskip

The first trajectory of the discrete coalescent-walk process associated with a bipolar orientation (in the terminology of \cite{borga2020scaling}) corresponds to the process $\conti Z_n$ defined\footnote{In \cite{gwynne2016joint} the process $\conti Z_n$ is denoted by $\conti X_n$.} in \cite[Equation (2.1)]{gwynne2016joint}. The latter describes (\cite[Lemma 2.1]{gwynne2016joint}) the points of intersection between two discrete curves in the bipolar orientation and its dual (i.e.\ the interface paths defined in \cref{sect:coal_walk_intro}).
In \cite[Proposition 4.1]{gwynne2016joint} it was shown that $\conti Z_n$ converges to a limiting process $\conti Z$ which describes (see \cite[Proposition 3.2]{gwynne2016joint}) the points of intersection between the two already mentioned orthogonal SLE${}_{12}$ (these are the limits of the two interface paths in the bipolar orientation and its dual).

The limiting process $\conti Z$ above is constructed in \cite{gwynne2016joint} in a much more general framework:
Let $\kappa' \in (4,\infty)$ and $\theta\in [0,2\pi)$. Let $\bm \mu$ be a $\sqrt{16/\kappa'}$-LQG quantum plane, 
$\bm h$ be a Gaussian free field independent of $\bm \mu$, and $\bm \eta$ be the space-filling SLE${}_{\kappa'}$ curve of angle zero generated (in the sense of imaginary geometry) by $\bm h$. It is possible to associate with these objects a process $(\conti B^x_{\kappa'},\conti B^y_{\kappa'},\conti Z_{\kappa',\theta})$ defined as follows:
\begin{itemize}
	\item The process $(\conti B^x_{\kappa'},\conti B^y_{\kappa'})$ is a two-dimensional Brownian motion with $\rho=-\cos(4\pi/\kappa')$. This Brownian motion is given by the \emph{mating-of-tree} encoding of $(\bm \mu,\bm \eta)$ (see \cite{gwynne2019mating}).
	\item The process $\conti Z_{\kappa',\theta}$ tracks, in some sense, the interaction between $\bm \eta$ and another SLE${}_{16/\kappa'}$ curve of angle $\theta$ also generated by $\bm h$.
\end{itemize}
Gwynne, Holden and Sun prove that there exists a constant\footnote{The explicit expression of $p(\kappa',\theta)$ is not known.} $p = p(\kappa',\theta)$, with $p(\kappa',\pi/2) \equiv 1/2$, such that $\conti Z_{\kappa',\theta}$ is a skew Brownian motion with parameter $p$. 
They also note that $\conti Z_{\kappa',\theta}$ is a measurable functional of $(\bm \mu,\bm h)$, which turns out to be completely determined by $(\conti B^x_{\kappa'},\conti B^y_{\kappa'})$. Nevertheless, they do not explicitly describe the measurable mapping $(\conti B^x_{\kappa'},\conti B^y_{\kappa'})\mapsto \conti Z_{\kappa',\theta}$.

\medskip

In the case $\kappa' = 12$ and $\theta = \pi/2$, our approach in \cite{borga2020scaling} provides the explicit mapping $(\conti B^x_{12},\conti B^y_{12})\mapsto \conti Z_{12,\pi/2}$: it is enough to solve the SDE in \cref{eq:flow_SDE} driven by $(\conti B^x_{12},\conti B^y_{12})$ for $u=0$. For general $\kappa' \in (4,\infty)$ and $\theta\in [0,2\pi)$, we conjecture the following.

	\begin{conj}\label{conj:LCQ_SKEW}
		The process $ \conti Z_{\kappa',\theta}$ is explicitly described in terms of  $(\conti B^x_{\kappa'},\conti B^y_{\kappa'})$ by the SDE
		\begin{equation}\label{eq:generalized_LQG}
		\begin{cases}
		d\conti Z_{\kappa',\theta}(t) = \idf_{\{\conti Z_{\kappa',\theta}(t)> 0\}} d\conti B^y_{\kappa'}(t) - \idf_{\{\conti Z_{\kappa',\theta}(t)\leq 0\}} d\conti B^x_{\kappa'}(t) +(2p-1)d\conti L^{\conti Z_{\kappa',\theta}}(t), \quad &t>0,\\
		\conti Z_{\kappa',\theta}(t) = 0, \quad &t= 0,
		\end{cases}
		\end{equation}
		where  $p = p(\kappa',\theta)$ is a constant to be determined, and $\conti L^{ \conti Z_{\kappa',\theta}}(t)$ is the local time process at zero of $\conti Z_{\kappa',\theta}$ accumulated during the time interval $[0,t]$.
	\end{conj}

Let us briefly discuss the relations between the parameters $(\rho, q)$ defining the skew Brownian permuton and the parameters $(\kappa',\theta)$ above. It is natural to expect that $\rho=-\cos(4\pi/\kappa')$ and $q=p(\kappa',\theta)$, where  $p(\kappa',\theta)$ is the unknown constant mentioned above. We point out that at the moment we do not have an interpretation of the parameter $\theta$ in terms of permutations.
\cref{conj:LCQ_SKEW} makes us believe that the skew Brownian permuton can be directly constructed from the LQG (on a bounded domain\footnote{We prefer to not specify the precise notion of finite-volume LQG needed, since, at the moment, we are not sure what is the correct one.}).

	\begin{conj}
		Let $\kappa' \in (4,\infty)$ and $\theta\in [0,2\pi)$.
		Consider a $\sqrt{16/\kappa'}$-LQG (on a bounded domain) decorated with two SLEs curves $(\bm\eta(t),\bm\eta'(t))_{t\in [0,1]}$ with parameters $\kappa'$ and $16/\kappa'$ respectively, and of angle $\theta$. Consider the function $\bm{\psi}_{\kappa',\theta}(t)$ defined for all $t\in[0,1]$ by  $\bm\eta(t)=\bm\eta'(\bm{\psi}_{\kappa',\theta}(t)).$
		Then $\bm{\psi}_{\kappa',\theta}(t)=\varphi_{\conti Z_{\kappa',\theta}}(t)$ a.s.,
		where $\varphi_{\conti Z_{\kappa',\theta}}(t)$ is defined as in \cref{defn:varphi_gener} for $\{\conti Z^{(u)}_{\kappa',\theta}\}_{u\in[0,1]}$ solving \cref{eq:generalized_LQG} in the finite-volume case and for different starting times $u\in[0,1]$.
	\end{conj}

\vspace{-0.5cm}

\subsection{Connections with families of planar maps}

We conclude with a quick discussion on some possible connections between families of permutations and families of (decorated) planar maps. We saw that Baxter permutations and bipolar orientations are in bijection (and this bijection played a fundamental role in guessing many of the connections described above between the skew Brownian permuton and the LQG). This is not the only instance of a bijection between families of (decorated) planar maps and families of permutations: separable permutations are in bijection with rooted series-parallel non-separable maps\footnote{Rooted series-parallel non-separable maps are rooted non-separable maps that do not contain the complete graph $K_4$.} (see \cite[Proposition 6]{MR2734180}). Moreover, Baxter permutations avoiding the pattern $2413$ -- which form a subset of Baxter permutations, and a superset of separable permutations -- are in bijection with rooted non-separable planar maps (see \cite[Proposition 5]{MR2734180}). Other interesting maps are Schnyder woods (studied in \cite{li2017schnyder}). The corresponding model of permutations is a weighted model of Baxter permutations. It would be also interesting to study the generalizations of Schnyder woods described in \cite{MR2871142}.

	\begin{prbl}
		Can the bijections above (or additional ones) between families of permutations and families of planar maps help in investigating the connections between the skew Brownian permuton and the LQG?
	\end{prbl} 

We believe that, as it was in the case of Baxter permutations and bipolar orientations, these additional bijections will be helpful in improving our understanding of several connections between the skew Brownian permuton and the LQG, leading to the solutions of some of the problems/conjectures mentioned above.

    \appendix
    \chapter{Two-dimensional walks in the non-negative quadrant}
\label{sec:appendix}

Let $\bm W=(\bm W_k)_{k\in \Z_{\geq 0}}$ be a two-dimensional random walk with step distribution $\nu$ (defined in \cref{eq:step_distribution_walk} page~\pageref{eq:step_distribution_walk}), started at a point $x\in\Z^2$. We denote this measure by $\P_x$. Let $\conti W = (\conti X,\conti Y)$ be a standard two-dimensional Brownian motion of correlation $-1/2$.
After the simple computation $\Var(\nu) = \begin{psmallmatrix}2 &-1 \\ -1 &2 \end{psmallmatrix}$, the classical Donsker's theorem implies that the process $\left(\frac 1 {\sqrt {2n}} \bm W_{\lfloor nt \rfloor}\right)_{t\in[0,1]}$ converges in distribution to the process $(\conti W_t)_{t\in[0,1]}$.
In this section we are interested in the behavior of $\bm W$ under the conditioning of starting and ending close to the origin, and staying in the non-negative quadrant $Q = \Z_{\geq 0}^2$.
This has been treated in much wider generality in \cite{MR3342657} and \cite{duraj2015invariance}, and specialized in \cite{bousquet2019plane} to families of walks with steps in $\Steps$ (defined in \cref{eq:admis_steps} page~\pageref{eq:admis_steps}).
The following convergence in distribution can be found in \cite{MR3945746}, as an immediate consequence of \cite[Theorem 4]{duraj2015invariance}.
\begin{prop}\label{prop:DW}
	Let $x,y\in Q$. 
	Then
	\[\P_x\left(\left(\tfrac 1 {\sqrt{2 n}} \bm W_{\lfloor nt \rfloor}\right)_{0\leq t \leq 1} \in \cdot \;
	\middle| \; \bm W_{[0,n]}\subset Q, \bm W_n = y\right)
	\xrightarrow[n\to\infty]{} \P(\conti W_e \in \cdot),\]
	where $\conti W_e$ is some process that we call the two-dimensional Brownian excursion of correlation $-1/2$ in the non-negative quadrant.
\end{prop}

We will now go through the initial steps of a slightly different proof of this result, one that highlights an absolute continuity phenomenon between a conditioned walk away of its starting and ending points and an unconditioned one. The two lemmas that we prove here (absolute continuity of the walk and local limit estimate of the density factor) are needed in this manuscript to show convergence of a coalescent-walk process driven by a conditioned random walk. 

In what follows, if $W = (X,Y)$ is a two-dimensional walk, then $\inf W = (\inf X, \inf Y)$.
We also use the \emph{hat} to denote reversal of coordinates, so that $\widehat{(i,j)} = (j,i)$.

\begin{lem}\label{lem:AbsCont}
	Let $h:\Z^2\to\R$ be a bounded measurable function. Let $x,y\in Q$ and $1\leq m<n/2$. Then
	\begin{multline}
	\E_x[h((\bm W_{i+m}-\bm W_m)_{0\leq i \leq n-2m}) \mid \bm W_{[0,n]}\subset Q, \bm W_n = y]\\
	= \E_0\left[
	h(\bm W_i)_{0\leq i \leq n-2m}\cdot
	\alpha_{n,m}^{x,y}\left(\inf_{0\leq i \leq n-2m} \bm W_i\; ,\;\bm W_{n-2m}\right)
	\right],
	\end{multline}
	where
	\begin{equation}\label{eq:alpha_tilting_function}
	\alpha_{n,m}^{x,y}(a,b) = \sum_{z\in Q \colon z+a \in Q} \frac{
		\P_x(\bm W_m = z, \bm{W}_{[0,m]}\subset Q)
		\P_{\widehat y}(\bm W_m = \widehat z+ \widehat b, \bm{W}_{[0,m]}\subset Q)
	}{\P_x(\bm W_n = y, \bm{W}_{[0,n]}\subset Q)}.
	\end{equation}
\end{lem}

\begin{lem}\label{lem:LLT}
	Fix $x,y\in Q$. For all $1/2>\eps>0$,
	\begin{equation}
	\lim_{n\to\infty} \sup_{a\geq 0,b\in \Z} \left\lvert\alpha^{x,y}_{n,\lfloor n\eps\rfloor}(a,b) - \alpha_\eps\left(\tfrac a{\sqrt {2  n}},\tfrac b{\sqrt {2 n}}\right)\right\rvert = 0,
	\end{equation}
	where $\alpha_\eps$ is a bounded continuous function on $(\R_+)^2\times \R^2$  defined by
	\begin{equation}\label{eq:fuction_alpha}
	\alpha_\eps(a,b) =\frac {\sqrt 3}{8\eps^5} \int_{x: x+a \in \R^2_+} g(x)g(x+b)dx
	\end{equation}
	and 
	\begin{equation}\label{eq:fuction_g}
	g(x_1, x_2)=\frac{1}{\sqrt{3 \pi}} x_1 x_2(x_1+x_2) \exp \left(-\frac{1}{3}\left(x_1^{2}+x_2^{2}+x_1 x_2\right)\right).
	\end{equation}
\end{lem}

A byproduct of this approach is a different characterization of the law of $\conti W_e$, which is immediate from Proposition \ref{prop:DW} and Lemmas \ref{lem:AbsCont} and \ref{lem:LLT}.
\begin{prop}\label{prop:brown_ex}
	For every $\eps>0$, the distribution of $(\conti W_e (\eps+t) - \conti W_e(\eps))_{0\leq t\leq 1-2\eps}$ is absolutely continuous with regards to the distribution of $\conti W_{|[0,1-2\eps]}$. The density function is the map
	\begin{equation}
	\mathcal{C}([0,1-2\eps],\R^2)\to \R,\qquad
	f\mapsto
	\alpha_\eps\big(\inf_{[0,1-2\eps]}f\; ,\; f(1-2\eps)\big).
	\end{equation}
	In particular, for every $\eps>0$ and for every integrable function  $h:\mathcal{C}([0,1-2\eps],\R^2)\to \R$,
	\begin{equation}
	\E\Big[h\big((\conti W_e (\eps+t) - \conti W_e(\eps))_{0\leq t\leq 1-2\eps}\big)\Big]
	= \E\Big[
	h\big(\conti W_{|[0,1-2\eps]}\big)
	\alpha_\eps\big(\inf_{[0,1-2\eps]} \conti W\; ,\; \conti W(1-2\eps)\big)\Big].
	\end{equation}
\end{prop}

The proofs of the lemmas above can be found in \cite[Appendix A]{borga2020scaling}.

	\printnoidxglossary[sort=def,title={Notation}]\label{not:notation}

\bibliographystyle{alpha} 
\bibliography{my_bib.bib}

\newpage\thispagestyle{empty}

{\begingroup
	\AlCentroPaginaMod{Back_0_4}
\topskip0pt
\vspace*{\fill}
\begin{center}
	    \hspace{1cm}
		\textcolor{UMRed}{\textbf{\Huge{For the Love of Maths}}}
\end{center}
\vspace*{\fill}
\endgroup}

\end{document}